\documentclass[10pt,twoside,openright,reqno]{amsbook}
\usepackage[english]{babel}
\usepackage{indentfirst}
\usepackage[T1]{fontenc}
\usepackage[utf8]{inputenc}
\usepackage{paralist}
\usepackage{amssymb}
\usepackage{mathabx}
\usepackage{amsthm}
\usepackage{amsmath}
\usepackage{amscd}
\usepackage{graphicx}
\usepackage[margin=1.2in]{geometry}
\usepackage{cite}
\usepackage{lineno} %% adds numbers to the lines
\usepackage{pdfpages}
\usepackage[OT2,OT1]{fontenc}

\newcommand\cyr
{
\renewcommand\rmdefault{wncyr}
\renewcommand\sfdefault{wncyss}
\renewcommand\encodingdefault{OT2}
\normalfont
\selectfont
}
\DeclareTextFontCommand{\textcyr}{\cyr}

\makeatletter
\def\@settitle
{
\begin{center}
\baselineskip14\p@\relax
\LARGE\@title
\end{center}
}
\makeatother

\theoremstyle{plain} 
\newtheorem{theorem}{Theorem}[section] 
\newtheorem{corollary}[theorem]{Corollary} 
\newtheorem{lemma}[theorem]{Lemma}
\newtheorem{prop}[theorem]{Proposition}
\newtheorem{defn}[theorem]{Definition}
\newtheorem*{defn*}{Definition}
\newtheorem*{thm*}{Theorem}
\newtheorem*{prop*}{Proposition}
\theoremstyle{remark}
\newtheorem{remark}[theorem]{Remark}
\numberwithin{equation}{chapter} 
\numberwithin{section}{chapter} 
\numberwithin{figure}{chapter} 
\theoremstyle{definition}
\newtheorem{example}{Example}[section]

\newcommand{\Z}{\ensuremath{\mathbb{Z}}}
\newcommand{\R}{\ensuremath{\mathbb{R}}}
\newcommand{\N}{\ensuremath{\mathbb{N}}}
\newcommand{\Rn}{\ensuremath{\mathbb{R}^n}}
\newcommand{\obal}{\ensuremath{\mathbb{R}^n} \setminus B_r}

\newcommand{\frlap}{\ensuremath{(-\Delta)^s}}
\newcommand{\eee}{\ensuremath{\varepsilon}}
\newcommand{\E}{\ensuremath{\mathcal{E}}}
\newcommand{\C}{\ensuremath{\mathcal{C}}}
\newcommand{\Co}{\ensuremath{\mathbb{C}}}
\newcommand{\K}{\ensuremath{\mathcal{K}}}
\newcommand{\bgs}[1]{\begin{equation*} \begin{aligned} #1 \end{aligned} \end{equation*}}
\newcommand{\eqlab}[1]{\begin{equation}  \begin{aligned}#1 \end{aligned}\end{equation}}
\newcommand{\sys}[2][]{\begin{equation*}#1  \left\{\begin{aligned}#2\end{aligned}\right.\end{equation*}}
 \newcommand{\al}{\ensuremath{&\;}}
\newcommand{\Lm}{\ensuremath{\mathcal{L}}}

\newcommand{\lr}[1]{\left( #1 \right)}
\newcommand{\lrq}[1]{\left[#1 \right]}
\newcommand{\ig}{\ensuremath{ \frac{\sin \pi s}{ \pi}}}

\newcommand{\F}{\ensuremath{\mathcal{F}}}
\newcommand{\Sa}{\ensuremath{\mathcal{S}}}
\newcommand{\wck}{\ensuremath{\widecheck}}
\newcommand{\Fl}{\ensuremath{\mathcal{F}_{\varepsilon}}} %The energy functional 
\newcommand{\Lst}{\ensuremath{L^{2_s^*}(\Rn)}} % The space with the critical exponent
\newcommand{\eps}{\ensuremath{\varepsilon}} % A shorter name for varepsilon
\newcommand{\ukp}{\ensuremath{(U_k)_+}} %The positive part of U_k$
\newcommand{\cx}{\ensuremath{2_s^{\star}}} %The critical exponent 2_s^*
 \newcommand{\Hsa}{\ensuremath{\dot  H_a^s(\R^{n+1}_+)}}  %The functional space we use
 \newcommand{\Rp}{\ensuremath{\mathbb{R}^{n+1}_+}} %The n+1 dimensional real numbers on the positive half-space
  \newcommand{\syslab}[2] []  {\begin{equation}#1  \left\{\begin{aligned}#2\end{aligned}\right.\end{equation}} %System with label and the possibility of using the graph parenthesis and well aligned ex: \syslab[f(x)=]{1, if x=0 \\ 2, if x=1.}
   \newcommand{\scp}[1]{\langle #1 \rangle} %The scalar product

\newcommand{\I}{\mathcal I}

\newcommand{\Ha}{\mathcal H}
\newcommand{\Ll}{\mathcal L}
\newcommand{\alig}[1] {\left\{\begin{aligned}#1 \end{aligned}\right.}

\makeatletter
\def\l@subsection{\@tocline{2}{0pt}{2.5pc}{5pc}{}}
\makeatother

\begin{document}

\includepdf[pages={1}]{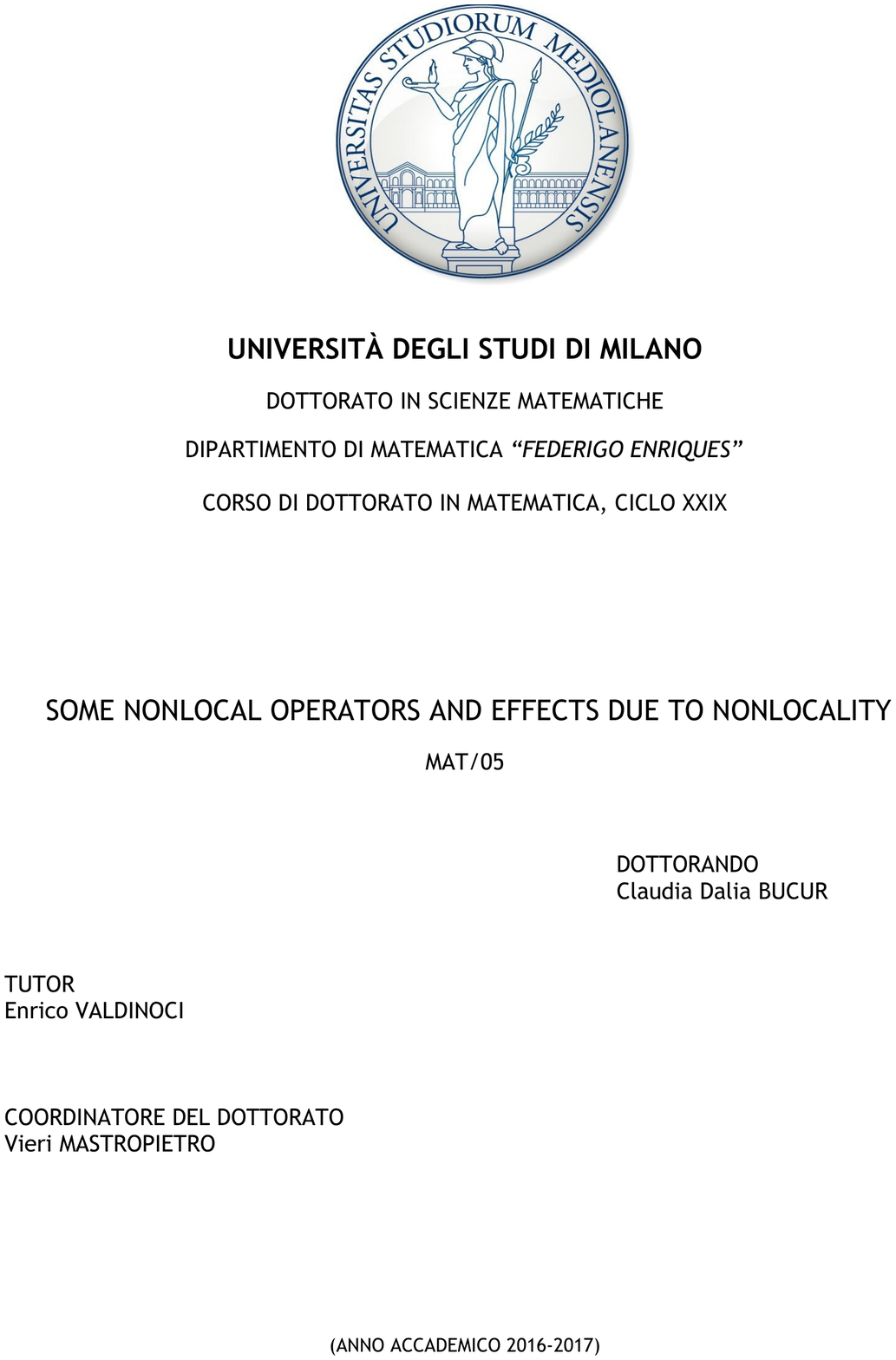}

%\nocite*
\date{}
\author{Claudia Bucur}

\address{Claudia Bucur: Dipartimento di Matematica\\ Universit\`a degli Studi di Milano \\ Via Cesare Saldini, 50 \\ 20100, Milano-Italy}

\email{claudia.bucur@unimi.it}

\email{}
\thanks{}

\title[]{Some nonlocal operators and effects due to nonlocality}
\maketitle
\begin{abstract}In this thesis, we deal with problems related to nonlocal operators, in particular to the fractional Laplacian and to some other types of fractional derivatives (the Caputo and the Marchaud derivatives). We make an extensive introduction to the fractional Laplacian, we present some related contemporary research results and we add some original material. Indeed, we study the potential theory of this operator, introduce a new proof of Schauder estimates using the potential theory approach, we study a fractional elliptic problem in $\Rn$ with convex nonlinearities and critical growth  and we present a stickiness property of  nonlocal minimal surfaces for small values of the fractional parameter. Also, we point out that the (nonlocal) character of the fractional Laplacian gives rise to some surprising nonlocal effects. We 
 % focusing our attention on some particular traits of the fractional Laplacian, we
 prove that other fractional operators have a similar behavior: in particular, Caputo-stationary functions are dense  in the space of smooth functions; moreover, we introduce an extension operator for Marchaud-stationary functions. 
\end{abstract}

\tableofcontents
%--------------------------------------
% SECTION
%--------------------------------------
\chapter*{Introduction}
The interest in nonlocal operators has increased in the last decades given their numerous applications in many branches of physics, engineering, biology and so on. Just to name a few, several models involving nonlocal operators are being used to describe anomalous diffusion processes, viscoelasticity, signal processing, geomorphology, materials sciences, fractals, and many others. 

 Nonlocal operators have the peculiarity of capturing long-range interactions, i.e. events that happen far away, may that be in time or in space. In our setting, we study some aspects of nonlocal behavior, introduced by the following integral operators: the fractional Laplacian, the Caputo and the Marchaud fractional derivatives.
 
 Fractional calculus is a classical argument, studied since the end of the seventeenth century by many great mathematicians (see \cite{MillerRoss} for an interesting time-line history). Fractional operators generalize classical (integer) ones, in the sense that if the order of the fractional operator is given by the parameter $s\in (0,1)$, then letting $s\to 0^+$ one obtains the identity, and letting $s\to 1^-$, one gets the classical (integer order) operator. In the literature, there are several definitions of fractional operators, like the Riemann-Liouville, the Caputo, the Riesz, the Marchaud fractional derivative, or the generalization given by the Erdélyi-Kober operator (see \cite{KST06,MillerRoss,samkokilbas} for more details on fractional integrals, derivatives and applications). 
 
   The fractional Laplacian well describes nonlocal diffusion phenomena. For instance, we can use it to describe what happens to a sheet of metal (that has  a crystalline configuration), since  the behavior at a given point (for instance, its deformation when an external force is applied) depends on a large scale on the whole object. On the other hand, if we think of a function depending on time, the Caputo and the Marchaud derivatives exhibit a ``memory effect'', that is they ``see past events'', providing a model in which the state of a system at a given time depends on the past. They describe, hence, a causal system, also called a non-anticipative system. \\
%We deal in this thesis with three fractional operators: the fractional Laplacian (which is related to the fractional Riesz derivative), the Caputo and the Marchaud derivatives. 
We dedicate most of the thesis to the fractional Laplacian. Moreover, we introduce along the way the two other fractional derivatives (Caputo and Marchaud) and show that their nonlocal character induces some properties similar to those of the fractional Laplacian. 

\smallskip

Starting from the basics of the nonlocal equations and in particular of the fractional Laplace operator, in this thesis we will discuss in
detail some recent developments in some very interesting topics of research, presenting:
\begin{itemize}
\item a problem arising in crystal dislocation (which is related to a
classical model introduced by Peierls and Nabarro),
\item a problem arising in phase transitions
(which is related to
a nonlocal version of the classical
Allen–Cahn equation),
\item a nonlocal version of the Schr\"odinger equation
for standing waves (as introduced by Laskin), and
\item the limit interfaces arising in the above
nonlocal phase transitions
(which turn out
to be nonlocal minimal surfaces,
as introduced by
Caffarelli, Roquejoffre and Savin).
\end{itemize}
In particular,  we focus our 
attention on the following original contributions: 
\begin{itemize}
\item a Schauder estimate for the fractional Laplacian using the potential theory approach,
\item a fractional equation in $\Rn$ in the convex, critical case,
\item a stickiness phenomenon of nonlocal minimal surfaces, when the fractional parameter is small.
\end{itemize}

Moreover,  we prove some original results related to the Caputo derivative (see \cite{CAPUTO}) and the Marchaud derivative (see \cite{Marchaud}). In particular, we state a density property that  the Caputo derivative shares with the fractional Laplacian. Indeed, Caputo-stationary functions are locally dense in the space of smooth functions (just like $s$-harmonic functions are).

Furthermore, we introduce the extension operator of the Marchaud derivative. The extension is a local operator defined in one dimension more, whose trace is the original nonlocal operator itself. In this way, the nonlocal behavior of the Marchaud derivative can be seen as the effect of local events that occur in a space with an extra dimension (similarly to the extension operator for the fractional Laplacian, check \cite{CAFSIL}).  The advantage of working with the extension is that one can overcome the difficulty induced by the nonlocality, and use tools that are somehow classical. In our case, we prove a Harnack inequality for Marchaud-stationary functions using the Harnack inequality in the local case. 

\bigskip

%%%%%%%%%%%%%%%%%%%%%%%%%%%%%%%%%%%
\section*{Overview of the thesis and original results}

This thesis gathers some recent research on the fractional Laplacian. Starting from the basics of the theory for this operator, we will collect examples, recent results and observations, and enrich the material with some original contributions. Also, we will see that some nonlocal effects registered by the fractional Laplacian find correspondence for other fractional operators. In this sense, we will introduce and work with two types of fractional derivative (the Caputo and the Marchaud definitions). Furthermore, we will present some known recent results on   nonlocal minimal surfaces, and discuss in detail some new results on the behavior of  nonlocal minimal surfaces for a small value of the fractional parameter.
\bigskip

This thesis is organized in seven chapters, each of which focuses on a particular research theme. We consequently present the content of each chapter.

\bigskip
To start with, in the first chapter, we will give a motivation
for the fractional Laplacian, that
originates from probabilistic considerations.
For $s\in (0,1)$ and for regular enough functions the fractional Laplacian is defined as
\begin{equation}\label{frlapintr1}
\frlap u(x)= \frac{C(n,s) }2  \int_{\Rn} \frac{u(x) -u(x+y)-u(x-y)} {|y|^{n+2s} } \, dy,\end{equation} 
where $C(n,s)$ is a positive constant.
 We present here two probabilistic models in which this operator naturally arises: a random walk that allows long jumps and a payoff model. Indeed, we show that the fractional heat equation, i.e. 
\[\partial_t u+ (-\Delta)^s u=0,\]  naturally arises from a probabilistic process in which a particle moves randomly in the space, subject to a probability that allows long jumps. Using the same probabilistic process and supposing that exiting the domain for the first time by jumping to an outside point means earning a certain (known) quantity of money,  the payoff function will be $s$-harmonic in the domain, that is, inside the domain it will satisfy $\frlap u =0 .$	\\
As a matter of fact, no advanced knowledge of probability theory
is assumed from the reader, and the topic is dealt with
at an elementary level. 
%
%A brief physical motivation will also be introduced for the Caputo fractional derivative. 

\bigskip

In Chapter 2, we will recall some basic properties of
the fractional Laplacian, 
discuss some explicit examples in detail
and point out some structural inequalities, that are due to a
comparison principle.  
At first, we introduce some equivalent representations for the fractional Laplacian, as a principal value integral (in the sense of Cauchy),
\[\frlap u(x)=C(n,s) P.V. \int_{\Rn} \frac{u(x)-u(y)}{|x-y|^{n+s}}\, dy\]
and as a pseudo-differential operator
\[\frlap u(x) =\mathcal F^{-1}\left(|\xi|^{2s}\widehat u(\xi) \right).\] We also explicitly compute the constant $C(n,s) $ that  was introduced in definition \eqref{frlapintr1}. \\ Fractional Sobolev spaces enjoy quite a number of important functional inequalities. We present here two important results and give some simple and nice proofs, namely the fractional Sobolev inequality and the generalized co-area formula.\\
 Moreover, we present an explicit example of an $s$-harmonic function on the positive half-line, that is, we prove that
 \[\frlap (x_+)^s=0 \quad \mbox{ on } \; \R_+,\] and an example of a function with constant Laplacian on the ball, that is, we prove that up to constants
 \[ \frlap (1-|x|^2)_+^s = 1 \quad \mbox{in } \; B_1.\] We also discuss some maximum principles and a Harnack inequality, and present a quite surprising result, which states that every smooth function can be locally approximated
by functions with vanishing fractional Laplacian
(in sharp contrast with the rigidity of the classical harmonic functions).  

In analogy to this result on the Laplacian, \textbf{\textcolor{black}{as an original contribution}}, we prove that Caputo-stationary functions are dense in the space of smooth functions. Indeed, the nonlocal character of the Caputo derivative gives rise to this peculiar behavior: on a bounded interval, say $[0,1]$, one can find a Caputo-stationary function ``close enough'' to any smooth function, without any geometrical constraints. This again is surprising, since classical derivatives are rigid in tis sense (for instance, the functions with null first derivative are constant functions and  the functions with null second derivatives are affine functions). \\
 We notice that this behavior seems to be a typical nonlocal feature, and is shared by solutions of other nonlocal equations (see for this \cite{approxforall}, where the same type of result is proved for solutions of the fractional heat equation).\\
In particular, we introduce the Caputo derivative of a (good enough) function $u$ to be
\[ D_a^s u(x)= c_s\int_a^x u'(t) (x-t)^{-s}\, dt,\]
where $c_s$ is a positive constant, and prove the following result.
\begin{thm*}
Let $k\in \N_0$ and $s\in (0,1)$ be two arbitrary parameters. Then for any $f \in C^k\big([0,1]\big)$ and any $\varepsilon>0$ there exists an initial point $a<0$ and a function $u\in C^{1,s}_a $ such that 
	\[ D_a^s u(x)=0 \text{ in } [0,\infty) \]and 
	\[\| u-f\|_{C^k\lr{[0,1]}} < \varepsilon.\]	
	\end{thm*}
  
\noindent To prove this theorem, we follow the steps of \cite{DSV14}. The main difficulties are given by the structure of the Caputo derivative and the lack of symmetry of the exterior conditions.  In order to get the result, we do the following: we reduce the problem to finding a Caputo stationary function close to any monomial, and this comes to finding a Caputo stationary function with an arbitrarily large number of derivatives prescribed. By providing the ``right'' prescribed (exterior) data, we build a sequence of Caputo-stationary functions that tends uniformly to the function $x_+^s$. This allows us to obtain a Caputo-stationary function with an arbitrarily large number of derivatives prescribed and to conclude the proof.

\bigskip

In Chapter 3, we introduce the potential theory related to the fractional Laplacian. We underline here that the fractional Laplacian is closely related to the Riesz kernel, that in our context is the fundamental solution of the fractional Laplacian. As a matter of fact
\[ \frlap \Phi=\delta_0,\]
where $\Phi$ is the Riesz kernel and $\delta_0$ is the Dirac Delta at zero. The convolution operation with this singular integral kernel (which is also called, in fractional calculus, the Riesz fractional integral) is  the inverse operator (in a distributional sense) of the fractional Laplacian.  Indeed, for $u\in C_c^{0,\eps}(\Rn)$ we have that
\[ \frlap (u*\Phi)(x)= u(x) \quad \mbox{ in } \; \Rn,\]
both pointwise and in a distributional sense. We introduce also the Poisson kernel, that gives an $s$-harmonic function inside the ball when the exterior data is known, by convolution with the known term,
that is
\bgs{ \frlap \bigg(\int_{\Rn \setminus B_1} P_1(y,x) u(y) \bigg)= 0  \quad \mbox{ in } \; B_1 ,}
where $u$ is fixed outside of $B_1$ (and is continuous and integrable at infinity with respect to the weight). Here, $P_1(y,x)$ is the Poisson kernel (outside of the ball of radius $1$).\\
When a function is zero outside the ball, the Green function gives the solution of the identity problem inside the ball,  more precisely in $B_1$ we have that
\[ \frlap \left(\int_{B_1} u(y)G(x,y)\, dy\right) = u(x)  ,\]
for $u\in C_c^{0,\eps}(B_1)\cap C(\overline B_1) $ and $u=0$ in $\Rn \setminus B_1$. We also prove a formula for the Green function, that is more suitable for applications.
 The main results in this section are inspired from \cite{Landkof,conto,stableprocess}, but the proofs we give are elementary and easy to follow. 
 
 Furthermore, using the potential theory, we give \textbf{\textcolor{black}{an original proof}} of the Schauder estimates for a fractional Laplacian equation, using a dyadic ball argument. In particular, we take $f$ to be a H\"{o}lder continuous function in $B_1$  and $u$ solving 
$ (-\Delta)^s u=f $  in $B_1.$ Then we prove
%  that given  $f\in C^{0,\alpha}(B_1)\cap C(\overline B_1)$, 
  that on the half ball, $u$ has the regularity of $f$ increased by $2s$.  
  More precisely
 \begin{thm*} Let $s\in(0,1)$,  $\alpha <1$ and $f\in C^{0,\alpha}(B_1)\cap C(\overline B_1)$ be a given function with modulus of continuity 
 \[\omega(r):=\sup_{|x-y|<r} |f(x)-f(y)|.\]
 Let  $u\in  L^{\infty}(\Rn)\cap C^1(B_1)$ be a pointwise solution of 
 \[ \frlap u=f \quad \mbox{ in }  \; B_1.\]
  Then for any $x,y \in B_{1/2}$ and denoting $\delta:=|x-y|$ we have for  $ s\leq 1/2$ that
  %\eqlab{\label{mm1}  
  \bgs{  |u(x)-u(y)|\leq &\; C_{n,s}  \bigg( \delta  \|u\|_{L^{\infty}(\Rn\setminus B_1)} +\delta \sup_{\overline B_1}|f| + \int_0^{c \delta} \omega(t)t^{2s-1}\, dt + \delta \int_{\delta}^1  \omega(t)t^{2s-2}\, dt\bigg) ,}
  and  for $s>1/2$ that
 % \eqlab{\label{mm} 
 \bgs{ |Du(x)-Du(y)|\leq &\; C_{n,s}  \bigg(\delta  \|u\|_{L^{\infty}(\Rn\setminus B_1)} +\delta \sup_{\overline B_1}|f| + \int_0^{c \delta} \omega(t)t^{2s-2}\, dt+ \delta \int_{\delta}^1  \omega(t)t^{2s-3}\, dt\bigg) ,}
 where $C_{n,s}$ and $ c$ are positive dimensional constants.
\end{thm*}
\noindent In particular, 
%by substituting that $\omega(r)\leq C r^\alpha$, 
we have for $s\leq 1/2$ that
%\[ |u(x)- u(y) |\leq  C_{n,s}\delta \lr{  \|u\|_{L^{\infty}(\Rn\setminus B_1)} + \sup_{\overline B_1}|f|   + \delta^{\alpha+2s-1}},\] 
 $u\in C^{0,2s+\alpha}(B_{1/2})$ as long as $\alpha <1-2s$ and that $u$ is Lipschitz if $\alpha>1-2s$. For $s>1/2$ we have that
	%\[ |Du(x)-D u(y) |\leq  C_{n,s}\delta \lr{  \|u\|_{L^{\infty}(\Rn\setminus B_1)} + \sup_{\overline B_1}|f|   + \delta^{\alpha+2s-2}}.\] 
	 $u\in C^{1,\alpha+2s-1}(B_{1/2})$ if  $\alpha \leq 2-2s$, while for  $2-2s\leq \alpha <1$ the derivative $Du$ is Lipschitz in $B_{1/2}.$ 
	
	\noindent In order to prove these bounds, we rely on the very nice method used in \cite{wang} for the classical Laplacian,  which is based only on the 
 higher order derivative estimates and a dyadic ball argument. 

\bigskip

In Chapter 4 we deal with extended problems.
It is a quite remarkable fact that
in some occasions nonlocal operators can be equivalently
represented as local (though possibly degenerate or singular)
operators in one dimension more. Moreover, as a counterpart,
several models arising in a local framework
give rise to nonlocal equations, due to boundary effects.
So, to introduce the extension problem and give
a concrete intuition of it, we will present some models in physics
that are naturally set on an extended space
to start with, and we will show their relation with the fractional Laplacian
on a trace space. We will also give a detailed justification
of this extension procedure by means of the Fourier transform. \\
As a special example of problems arising in physics
that produce a nonlocal equation, we consider a problem related to crystal dislocation, present some mathematical results
that have been recently obtained on this topic, and
discuss the relation between these results and
the observable phenomena.

We end this chapter by introducing the Marchaud fractional derivative and, as \textbf{\textcolor{black}{an original contribution}}, the extension operator related to it. The Marchaud (left) fractional derivative is defined (up to constants) for a bounded, locally H{\"o}lder continuous function in $\mathbb{R}$, as   
\begin{equation*} \label{frader}
	\begin{split} \bold{D}^s f (t):=  \int_{0}^\infty \frac{f(t)-f(t-\tau)}{\tau^{s+1}}\, d\tau.
	%= \int_{-\infty}^t \frac{f(t)-f(\tau)}{(t-\tau)^{s+1}}\, d\tau.
	 \end{split}
\end{equation*}
This derivative naturally arises when dealing with a family of singular/degenerate parabolic problems (which, for $s={1}/{2}$, reduces to the heat conduction problem) on the positive half-plane, with a positive space variable and  for all times, namely for $(x,t)\in [0,\infty)\times \mathbb{R} $.  

\noindent Considering the function $\varphi$ of one variable, formally representing the time variable, our approach relies on constructing a parabolic local operator by adding an extra variable, say the space variable, on the positive half-line, and working on the extended plane $ [0,\infty)\times \mathbb{R} $. Namely, we prove that
\begin{thm*} 
Let $s\in (0,1) $ and $\bar{\gamma}\in (s,1]$ be fixed. Let $\varphi \in C^{\bar{\gamma}}(\mathbb{R})$ 
be a bounded function 
and let $U\colon [0,\infty)\times \mathbb{R}\to \mathbb{R}$ be a solution of the problem 
\begin{equation}\label{prob1}
\left\{
%\begin{array}{ll} 
\begin{split}
	&\frac{\partial U}{\partial t} (x,t)= \frac{1-2s}{x} \frac{\partial U }{\partial x}(x,t)+ \frac{\partial^2 U }{\partial x^2}(x,t), & &    (x,t)\in(0,\infty)\times \mathbb{R}\\
	 & U(0,t)=\varphi(t),  &&t\in\mathbb{R}\\
	  & \lim_{x \to +\infty} U(x,t)=0, &&t\in\mathbb{R}.
%\end{array}
\end{split}
\right.
\end{equation}
Then $U$ defines the extension operator for $\varphi$, such that
\begin{equation*}
\bold{D}^s \varphi(t)=-\lim_{x\to 0^+} c_s x^{-2s}(U(x,t)- \varphi(t)),
\end{equation*}
where $c_s$ is a positive constant.
 \end{thm*} 

\noindent An interesting application that follows from this extension procedure is a Harnack inequality for Marchaud-stationary functions in an interval $J \subseteq {\mathbb{R}}$ (namely for functions that satisfy  $\bold{D}^s \varphi=0$  in  $J$). 
%This fact is not obvious, indeed the set of functions determined by fractional-stationary functions (on an interval) is nontrivial, see  for instance Section \ref{capdens}.  Indeed
\begin{thm*} 
Let $s\in (0,1)$. There exists a positive constant $\gamma$ such that, if $\bold{D}^s\varphi=0$ in an interval $J\subseteq \mathbb{R}$ and $\varphi\geq 0$ in $\mathbb{R}$, then  
\begin{equation*}\label{Harnack_intro}
\sup_{[t_0-\frac{3}{4}\delta,t_0-\frac{1}{4}\delta]}\varphi\leq \gamma \inf_{[t_0+\frac{3}{4}\delta,t_0+\delta]}\varphi
\end{equation*}
for every $t_0\in \mathbb{R}$ and for every  $\delta >0$ such that $[t_0-\delta,t_0+\delta]\subset J$.
\end{thm*} 
\noindent This result is obtained from the Harnack inequality for some degenerate parabolic operators by ``looking at it'' on the trace. Indeed, using the extension operator, it is quite easy to obtain this type of result.  

\bigskip

In Chapter 5, we look at some nonlocal equations related to the fractional Laplacian. We first discuss a stationary Schr\"{o}dinger type equation arising in quantum mechanics, given by
\begin{equation*}
		\begin{cases}
		\varepsilon^{2s} \frlap u +u = u^p &\text{ in } \Omega \subset \Rn\\ 
		u=0 &\text{ in } \Rn \setminus \Omega, 
		\end{cases}
	\end{equation*}
in the subcritical case $p\in (1, 2_s^*-1)$,  where for $n>2s$, 
$2_s^*:=  \frac{2n}{n-2s} $ is the critical fractional Sobolev exponent. We give a sketch of the proof of the existence of a solution that concentrates at interior points of the domain for sufficiently small values of $\varepsilon$. This concentration phenomena is written
in terms of the ground state solution~$w$ (i.e. $w$ solves $\frlap w +w=w^p \text{ in } \Rn$).
Namely, the first approximation for the solution
is exactly the ground state~$w$,
scaled and concentrated at an appropriate point of
the domain. Moreover, we discuss a connection between the uncertainty principle and a fractional weighted inequality.

In the last section of this chapter, we prove as \textbf{\textcolor{black}{an original result}}, the existence of a positive solution of the nonlinear and nonlocal elliptic equation in $\Rn$ 
\eqlab{\label{introdproblem} (-\Delta)^s u =\varepsilon h u^q+u^{\cx-1} }
in the convex case $1\leq q<\cx-1$,  where $\varepsilon$ is a small parameter and $h$ is a given bounded, integrable function.  		
	The problem has a variational structure and we prove the existence of a solution by using the classical Mountain-Pass Theorem. We work here with the harmonic extension of the fractional Laplacian, which allows us to deal with a weighted  degenerate local operator, rather than with a nonlocal energy. In order to overcome the loss of compactness induced by the critical power we use a Concentration-Compactness principle. 
  The main result of this section goes as follows.
	\begin{thm*}
	Let $\varepsilon>0$ be a small parameter, let $q\in [1, \cx-1)$ and $h$ be such that
	\bgs{& h\in L^1(\Rn)\cap L^{\infty}(\Rn)\quad \mbox{and}\\
	& \mbox{there exists a ball }B\subset \Rn \mbox{ such that } \inf_B h>0.\\
	& \mbox{If }n\in (2s,6s), \mbox{suppose in addition }h\geq 0.}
Then problem \eqref{introdproblem}
 admits a positive (mountain-pass) solution, provided that
	 $n>{\frac{2s(q+3)}{q+1} }$.
	\end{thm*}
	Notice that in our problem the two nonlinearities are convex, and the geometry of the functional suggests the existence of one solution. In order to prove the existence of a solution we use, roughly speaking, the following strategy:\\
(i) we consider the energy functional associated to \eqref{introdproblem} and we prove that it satisfies some compactness condition (Palais-Smale condition) below a certain energy level,\\
(ii) we build a sequence of functions with an appropriate geometry (of Mountain Pass type) whose energy lies below the critical level found in (i), and\\
(iii) we apply the Mountain Pass Lemma to pass to the limit, getting a solution.\\
\noindent The proof follows the strategy of \cite{maria} for the concave-convex (fractional) case, and is based on two fundamental points: to identify the energy level, and to find the appropriate sequence. We point out that, in the concave-convex (fractional) problem, the geometry derived from the concave term (the functional has a minimum of negative energy) helps to prove that the sequence stays below the critical level. However, here both nonlinearities are convex, and the proof gets more involved.
%Indeed, if one adapts straightforwardly the compactness result in \cite[Proposition 4.2.1]{maria} and builds the sequence in the standard way (by considering the path along the Sobolev minimizers), then the arguments to prove that the energy of the sequence is small enough do not work.  
\\
\noindent Thus, the study of \eqref{introdproblem} will first require a finer analysis of the compactness properties of the functional. More precisely, we will have to improve the estimates of the functional in order to get a slightly higher critical level.
Accordingly, once we have found this new critical level, we perform a more careful analysis of the energy of the sequence given by the minimizers. We will finally conclude by applying the Mountain Pass Lemma in the standard way.

\bigskip

Chapter 6 and 7
present topics of contemporary research related to the fractional Laplacian.
We will discuss in particular:
some phase transition
equations of nonlocal type and
their limit interfaces,
which (below a critical threshold of the fractional parameter)
are surfaces that minimize a nonlocal perimeter functional.
We will present a De Giorgi conjecture in the fractional setting, which wonders whether entire, smooth, monotone (in one direction), bounded solutions of the (fractional) Allen-Cahn equation are one-dimensional. This indeed is the case in the classical framework in dimension at most $3$ (and up to $8$, with an additional, quite natural assumption, see \cite{SavinGenius,S09}). The dimension $8$ seems to be suggested also  by a link with a problem of Bernstein. This problem asks if all 
minimal graphs (i.e.
surfaces that locally minimize the perimeter and that
are graphs in a given direction)
in $\R^n$ must be necessarily affine.
The link between this Bernstein problem and the conjecture of De Giorgi could be implied
by the fact that
minimizers approach minimal surfaces in the limit but
of course, 
much work is needed to deeply understand the connections
between the two problems.

In Chapter 6 we consider
a nonlocal phase transition model, in particular described by the Allen-Cahn equation
\[ \frlap u=u-u^3\] in a bounded domain $\Omega\subset \Rn$. The Allen-Cahn equation  in a  nonlocal 
setting has theoretical interest and concrete applications. 
Indeed, the study of
long range interactions naturally leads to the analysis of
phase transitions and interfaces of nonlocal type. A fractional analogue of
a conjecture of De Giorgi, that deals with possible one-dimensional symmetry of entire solutions, naturally arises from treating the fractional Allen-Cahn equation, and is then presented. We give an alternative proof to the De Giorgi conjecture in $\R^2$, using the Dirichlet energy associated to the fractional Allen-Cahn equation. 

\bigskip
Chapter 7 deals with nonlocal minimal surfaces, as introduced in \cite{nms} in 2010. In particular, following the approach of De Giorgi (for classical minimal surfaces), we introduce the fractional perimeter and look for minimizers in bounded open sets with respect to some fixed exterior data. The boundaries of such (nonlocal minimal) sets are called nonlocal minimal surfaces (and are indeed $(n-1)$-dimensional and smooth almost everywhere). We give some notions on this subject, outline some nice recent achievements and also present a new result about a stickiness phenomena when the fractional parameter is small.\\
The fractional perimeter is defined as
\bgs{ \text{Per}_s(E,\Omega) := \Ll_s(E\cap \Omega,\C E)
+\Ll_s(E\setminus \Omega,\Omega \setminus E), }
where the interaction $\Ll_s(A,B)$ between two disjoint subsets of $\Rn$ is
	\bgs{
		\Ll_s(A,B):=\int_A \int_B \frac{dx\, dy}{|x-y|^{n+s}}
			=\int_{\Rn} \int_{\Rn} \frac{ \chi_A(x) \chi_B (x) }{|x-y|^{n+s}}\, dx \, dy.
	}
	Moreover, taking~$\Omega$ an open set of~$\R^n$, we say that $E\subset \Rn$ is $s$-minimal in~$\Omega$
if $\text{Per}_s(E,\Omega)$ is finite and if, for any competitor (that is, for any set~$F$ such that $E\setminus\Omega =F\setminus \Omega$), we have that
	\[ \text{Per}_s(E,\Omega) \leq \text{Per}_s(F,\Omega).\]
%A measurable set is $s$-minimal in $\Rn$ if it is $s$-minimal in any ball $B_r$, where $r > 0$.
Furthermore, we introduce the $s$-fractional mean curvature of a set $E$ at a point $q\in\partial E$ (as the fractional counterpart of the classical mean curvature notion). It is defined as the principal value integral
\[\I_s[E](q):=P.V.\int_{\R^n}\frac{\chi_{\C E}(y)-\chi_E(y)}{|y-q|^{n+s}}\,dy\]
%that is
%\[\I_s[E](q):=\lim_{\rho\to0^+}\I_s^\rho[E](q),\qquad\textrm{where}\qquad
%\I_s^\rho[E](q)=\int_{
%\C B_\rho(q)}\frac{\chi_{\C E}(y)-\chi_E(y)}{|y-q|^{n+s}}\,dy.\] 
(for the main properties of the fractional mean curvature, we refer to \cite{Abaty}).\\
 In this Chapter we discuss some very nice known results such as
\begin{itemize}
\item  $s$-minimal graphs (i.e., $s$-minimal sets that are graphs in given direction)  in $\R^{n+1}$ are flat if no singular cones exist in dimension $n$ (and this is related to a known Bernstein problem),
\item minimizers with respect to the exterior data that is a subgraph, is a subgraph also inside the domain,
\item  nontrivial minimal cones in dimension two do not exist (which implies, according to the first point, that $s$-minimal graphs in $\R^3$ are flat). 
\end{itemize}
Also, we discuss some nice examples of boundary regularity and stickiness phenomena.

On the other hand, the asymptotic behavior of nonlocal minimal surfaces as $s$ reaches $0$ or $1$ is another interesting matter. 
As $s\to 1^+$, one would like to obtain the classical counterpart of the objects under study. And this is indeed the case, as the following known results show. 
For a set $E \subset \Rn$ with $C^{1,\gamma}$ boundary in $B_R$ for some $R>0$ and $\gamma\in(0,1)$, for almost any $r<R$ and up to constants one has indeed that
  \[ \lim_{s \to 1^-} (1-s)P_s(E,B_r)= P(E,B_r).\] 	 See for the proof \cite{uniform}.
  Not only, but also (see Theorem 12 in \cite{Abaty}, and \cite{uniform}) for a set $E\subset \Rn$ with $C^2$ boundary and any $x\in \partial E$, one has that
\[ \lim_{s \to 1} \mathcal (1-s)I_s[E] (x) = \omega_{n-1}H[E](x),\]  where $H$ is the classical mean curvature of $E$ at the point $x$ (with the convention that we take $H$ such that the curvature of the ball is a positive quantity).  
  
As $s\to 0^+$, the asymptotic behavior is a bit more involved and some surprising behavior may arise.  This is due to the fact that as $s$ gets smaller, the nonlocal contribution to the perimeter becomes more and more important, and the local behavior loses influence. Some very nice first results in this sense were achieved in \cite{asympt1}. There, in order to mathematically encode the behavior at infinity of a set, the authors introduce the following quantity:
   \eqlab{\label{alpha} \alpha(E)=\lim_{s\to 0^+} s\int_{\C B_1} \frac{\chi_E(y)}{|y|^{n+s}}\, dy,  }
  (see  \cite{asympt1}). The set function $\alpha(E)$ appears naturally when looking at the
behavior near $s=0$ of the fractional perimeter (see \cite{asympt1}). So, let $\Omega$ be a bounded open set with $C^{1,\gamma}$ boundary, for some $\gamma\in (0,1)$, and $E \subset \Rn$ be a set with have finite $s_0$-perimeter, for some $s_0\in (0,1)$. If $\alpha(E)$ exists then
		\bgs{  \lim_{s\to 0^+} sP_s(E,\Omega)= &\; \alpha(\C E) |E\cap \Omega| + \alpha(E) |\C E \cap \Omega| 	.}
On this argument, we introduce in the last section some other \textbf{\textcolor{black}{original achievements}} on the behavior of $s$-minimal surfaces for small values of the fractional perimeter.	
	Indeed, there we obtain the asymptotic behavior of the fractional mean curvature for $s\to 0^+$, noticing that the limit takes into account only the data at infinity. In essence we prove that 
  \begin{thm*}
Let $E\subset\Rn$ and let $p\in\partial E$ be such that $\partial E$ is $C^{1,\gamma}$ near $p$,
for some $\gamma\in(0,1]$. Then
\bgs{& \liminf_{s\to0^+} s\,\I_s[E](p) =\omega_n -2 \overline \alpha(E)\\
& \limsup_{s\to0^+} s\,\I_s[E](p) =\omega_n-2 \underline\alpha(E).}
\end{thm*} 
Furthermore, we prove the continuity of the fractional mean curvature in all variables for $s\in [0,1]$. As a matter of fact, the $s$-fractional mean curvature is continuous
with respect to $C^{1,\alpha}$ convergence of sets, for any $s\in (0,\alpha)$ and with respect to $C^2$ convergence of sets, for $s$ close to 1. Here, by $C^{1,\alpha}$ convergence of sets we mean that our sets locally converge in measure and can locally be described as the supergraphs
of functions which converge in $C^{1,\alpha}$.  Indeed, we have the following results:

\begin{thm*}
Let $E_k\xrightarrow{C^{1,\alpha}}E$ in a neighborhood of $q\in \partial E$. Let $q_k\in\partial E_k$ be such that $
q_k\longrightarrow q$ and let $s,s_k\in(0,\alpha)$ be such that $s_k\xrightarrow{k\to\infty} s$.
Then
\[\lim_{k\to\infty}\I_{s_k}[E_k](q_k)=\I_s[E](q).\]

Let $E_k\xrightarrow{C^2}E$ in a neighborhood of $q\in \partial E$. Let $q_k\in\partial E_k$ be such that $q_k\longrightarrow q$ and let $s_k\in(0,1)$ be such that $s_k\xrightarrow{k\to\infty} 1$. Then
\[\lim_{k\to\infty}(1-s_k)\I_{s_k}[E_k](q_k)=\omega_{n-1}H[E](q).\]
\end{thm*}
Using this, we see that as the parameter $s$ varies, the fractional mean curvature may change sign. 

When $s\to 0^+$ we do not need the $C^{1,\alpha}$ convergence of sets, but only the uniform boundedness  of the $C^{1,\alpha}$ norms of the functions defining the boundary of $E_k$ in a neighborhood of the boundary points. However, we have to require that the measure of the symmetric difference is uniformly bounded. More precisely:

\begin{prop*}\label{propsto0}
Let $ E\subset \Rn$
be such that $\alpha(E)$ exists.  Let 
$q \in \partial E$ be such that 
\bgs{ E\cap Q_{r,h}(q)=\{(x',x_n)\in\R^n\,|\, x'\in B'_{r}(q'),\,u(x')<x_n<h+q_n\},}
for some $r,h>0$ small enough and $u\in C^{1,\alpha}(\overline B'_{r}(q'))$ such that $u(q')=q_n$. 
Let $E_k\subset \Rn$ be such that
\[ |E_k\Delta E|<C_1  \] 
for some $C_1>0$. Let $q_k\in \partial E_k \cap B_d$, for some $d>0$, such that 
 \bgs{E_k\cap Q_{r,h}(q_k)=\{(x',x_n)\in\R^n\,|\,x'\in B'_{r}(q_k'),\,u_k(x')<x_n<h+q_{k,n}\} }
for some functions $u_k\in C^{1,\alpha}(\overline B_{r}'(q_k'))$ such that $u_k(q_k')=q_{k,n}$ and
\[ \|u_k\|_{C^{1,\alpha}(\overline B'_{r}(q'_k))} <C_2 \] 
for some $C_2>0$. Let $s_k\in(0,\alpha)$ be such that $s_k\xrightarrow{k\to\infty} 0$. Then
\bgs{\lim_{k\to\infty}s_k\I_{s_k}[E_k](q_k)=\omega_n-2\alpha(E).}
\end{prop*}

Finally, when $s\in (0,1)$ is small we classify the behavior of $s$-minimal surfaces, in dependence of the exterior data at infinity.
We prove that when the fractional parameter is small and the exterior data at infinity occupies (in measure, with respect to the weight) less than half the space, then $s$-minimal sets completely stick at the boundary (that is, they are empty inside the domain), or become ``topologically dense'' in their domain.
Indeed, denoting 
  \bgs{\label {baralpha1} \overline \alpha (E):= \limsup_{s\to 0^+} s\int_{\C B_1} \frac{\chi_E(y)}{|y|^{n+s}}\, dy ,} we give the next definition.
  \begin{defn*}
   Let $\Omega\subset \Rn$ be a bounded open set.
   We say that a set $E$ is $\delta$-{dense} in $\Omega$ for some fixed $\delta>0$ if $|B_\delta(x)\cap E|>0$ for any $x\in \Omega$ for which $B_\delta(x)\subset\subset\Omega$.
  \end{defn*}
  \noindent This notion of $\delta$-density is a ``topological'' notion,
rather than a measure theoretic one. With this definition and denoting
\[ \delta_s=-\frac{c}s,\quad \mbox{ where} \quad c:=c\big(E_0)= \log\frac{3\omega_n-4\overline \alpha(E_0)}{5\omega_n -2 \overline\alpha (E_0)},\] we obtain the following classifications:
\begin{thm*}
  Let $\Omega$ be a  bounded and  connected open set with $C^2$ boundary. Let $E_0\subset \C \Omega$ be such that\[\overline \alpha(E_0)<\frac{\omega_n}{2}.\]  
 Then the following two results hold.\\
  A)  There exists
  $s_1=s_1(E_0,\Omega)\in (0,1/2)$ such that if $s<s_1$ and $E$ is an $s$-minimal set in $\Omega$ with exterior data $E_0$, then either
     \bgs{(A.1) \;  E\cap \Omega=\emptyset \quad  \mbox{ or} \quad\; (A.2)\;  E \mbox{ is } \delta_s-\mbox{dense}.}
    %     Let $\Omega$ be an  bounded and  connected open set with $C^2$ boundary. Let $E_0\subset \C \Omega$ such that\[\overline \alpha(E_0)<\frac{\omega_n}{2}.\] %and let $s_0$ and $\delta_s$ be as in Theorem \ref{positivecurvature}.  
%Then,
 \noindent
 B) Either \\
(B.1) there exists
  $\tilde s=\tilde s(E_0,\Omega)\in (0,1)$ such that if $E$ is an $s$-minimal set in $\Omega$ with exterior data $E_0$ and $s\in(0,\tilde s)$, then
     \bgs{  E\cap \Omega=\emptyset,}
     or \\
    (B.2)    there exist  $\delta_k \searrow 0$, $s_k \searrow 0$ and a sequence of sets  $E_k$ such that each $E_k$ is $s_k$-minimal in $\Omega$ with exterior data $E_0$ and for every $k$
     \bgs{ \partial E_k \cap B_{\delta_k}(x) \neq \emptyset \quad \forall \; B_{\delta_k}(x)\subset\subset \Omega.}
     \end{thm*}
 An analogue result, that is that $s$-minimal sets fill the domain or their complementaries become dense, is obtained when the exterior data occupies in the appropriate sense more than half the space (so this threshold is  optimal). We point out that in this way, when  $\alpha(E_0)\neq  {\omega_n}/{2} $ we have a complete classification of $s$-minimal sets when $s$ is small.

 \section*{Scientific production}
 This thesis collects the results obtained by myself and others in the following papers.
 \begin{enumerate}
 \item  Claudia Bucur, \emph{Some observations on the Green function for the ball in the fractional Laplace framework}, Communications on Pure and Applied Analysis \textbf{15} (2016), pp. 657-699, DOI 10.3934/cpaa.2016.15.657. MR3461641 
 \item  Claudia Bucur and Enrico Valdinoci, Nonlocal diffusion and applications, \emph{Lecture Notes of the Unione Matematica Italiana} \textbf{20} (2016), Springer, Unione Matematica Italiana, Bologna, pp. xii+155, DOI 10.1007/978-3-319-28739-3. MR3469920
\item Claudia Bucur and Fausto Ferrari, \emph{An extension problem for the  fractional derivative defined by Marchaud} , Fractional Calculus and Applied Analysis \textbf{19} (2016), pp. 867-887, DOI 10.1515/fca-2016-0047. MR3543684 
\item Claudia Bucur and Aram L. Karakhanyan, \emph{Potential theory approach to Schauder estimates for the Fractional {L}aplacian}, Proc. Amer. Math. Soc. \textbf{145} (2017), pp. 637-651, DOI https://doi.org/10.1090/proc/13227.
\item  Claudia Bucur, \emph{Local density of Caputo-stationary functions in the space of smooth functions}, Accepted for pubblication in ESAIM: Control, Optimisation and Calculus of Variations, DOI http://dx.doi.org/10.1051/cocv/2016056 
\item Claudia Bucur and Maria Medina, \emph{A fractional elliptic problem in $\mathbb{R}^n$ with critical growth and convex nonlinearities}, preprint
\item Claudia Bucur, Luca Lombardini and Enrico Valdinoci, \emph{Stickiness at the boundary of nonlocal minimal surfaces}, preprint
 \end{enumerate}
 
 \newpage
 
 \section*{Notations}
 \begin{itemize}
 \item We consider $n\in \N$ to be the dimension of the space of reference, that will be $\R^n$. We usually denote elements in the reference spaces as $x\in \Rn$ and $X\in \R^{n+1}$.
 \item  We write $|E|=\mathcal L^n(E)$ for the $n$-dimensional Lebesgue measure of a set $E\subset\Rn$ and $\mathcal H^d$ for the $d$-dimensional Hausdorff measure for any $d\geq 0$. 
  \item  We denote by $\C E= \Rn \setminus E$ the complementary of any $E\subset \Rn$. 
 \item We denote the $n$-dimensional open ball of radius $r$ and center $x_0\in\Rn$ as
 \[ B_r(x_0)=\{ x\in \Rn \; \big| \; |x-x_0|<r\}\]  and write $B_r$ whenever $x_0=0$. Also,
 we use the notation
 \[{S}^{n-1} =\partial B_1\]
 for the $(n-1)$-dimensional sphere.
 \item We define the area of the surface of the $(n-1)$-dimensional sphere as the constant
 \[ \omega_n=\mathcal H^{n-1}(S^{n-1}) = \frac{2\pi^{\frac{n}2}}{\Gamma(\frac{n}2)},\]
 where $\Gamma$ is the Gamma function defined in \eqref{ABRAMOWITZ}. The volume of the $n$-dimensional unit ball is then \[\mathcal L^{n}(B_1)=\frac{\omega_n}n.\]
 \item We denote by $\mathcal S(\Rn)$ the Schwartz space of smooth functions rapidly decaying at infinity (see Section \ref{Four} in the Appendix for the definition and some other details).
 \item We will use the following notation for the class of H\"{o}lder continuous functions. Let $\alpha\in (0,1]$, let $S\subset\R^n$ and let $v:S\longrightarrow\R^m$. The $\alpha$-H\"{o}lder semi-norm of $v$ in $S$
is defined as
\[[v]_{C^{0,\alpha}(S,\R^m)}:=\sup_{x\neq y\in S}\frac{|v(x)-v(y)|}{|x-y|^\alpha}.\]
With a slight abuse of notation, we will omit the $\R^m$ in the formulas.
We also define
\[\|v\|_{C^0(S)}:=\sup_{x\in S}|v(x)|\quad\textrm{and}\quad\|v\|_{C^{0,\alpha}(S)}
:=\|v\|_{C^0(S)}+[v]_{C^{0,\alpha}(S)}.\]
Given an open set $\Omega\subset\R^n$, we define the space of uniformly H\"{o}lder continuous functions
$C^{0,\alpha}(\overline \Omega,\R^m)$ as
\[C^{0,\alpha}(\overline \Omega,\R^m):=\{v\in C^0(\overline{\Omega},\R^m)\,|\,
\|v\|_{C^{0,\alpha}(\overline{\Omega})}<\infty\}.\] 
Recall that $C^1(\overline{\Omega})$ is the space of those functions $u:\overline{\Omega}\longrightarrow\R$ such that
$u\in C^0(\overline{\Omega})\cap C^1(\Omega)$ and
such that $\nabla u$ can be continuously extended to $\overline{\Omega}$.
For every $S\subset\overline{\Omega}$ we write
\[\|u\|_{C^{1,\alpha}(S)}:=\|u\|_{C^0(S)}+\|\nabla u\|_{C^{0,\alpha}(S)},\]
and we define
\[C^{1,\alpha}(\overline \Omega):=\{u\in C^1(\overline{\Omega})\,|\,
\|u\|_{C^{1,\alpha}(\overline{\Omega})}<\infty\}.\]
We will usually consider the local versions of the above spaces. Given an open set $\Omega\subset\R^n$,
the space of locally H\"{o}lder continuous functions $C^{k,\alpha}(\Omega)$, with $k\in\{0,1\}$, is defined as
\[C^{k,\alpha}(\Omega):=\{u\in C^k(\Omega)\,|\,\|u\|_{C^{k,\alpha}(\mathcal O)}<\infty \mbox{ for every } \mathcal O\subset \subset \Omega\}.\] 

 \end{itemize}

\chapter{A probabilistic motivation for the fractional Laplacian} \label{chap2}
\begin{abstract} The goal of this chapter is to show that nonlocal operators well describe nonlocal phenomena. We introduce briefly the fractional Laplacian and then we present two probabilistic models in which such operator naturally arises. Indeed, we show that the fractional heat equation  arises from a probabilistic process in which a particle moves randomly in the space subject to a probability that allows long jumps. Using the same probabilistic process and supposing that exiting the domain for the first time by jumping to an outside point means earning a certain (known) quantity of money,  the payoff function will be $s$-harmonic in the domain. This models are treated in an elementary way, and little knowledge on probability theory is required from the reader. 
%In the last section, we give a short physical interpretation of the Caputo derivative. 
\end{abstract}

%	\section{A probabilistic motivation for the fractional Laplacian} 
%\label{S:1}

\bigskip 
\bigskip

We consider a function~$u\colon \Rn\to \R$ (which is supposed\footnote{To
write~\eqref{frlap2def} it is sufficient, for simplicity, to take $u$ in the Schwartz space $\mathcal{S}(\Rn)$  
of smooth and rapidly decaying functions (see \eqref{schsp}), or in~$ C^2(\Rn)\cap
L^{\infty}(\Rn)$.}
to be regular enough) and a fractional parameter~$s\in (0,1)$. Then, the 
fractional Laplacian of~$u$ is given by
\begin{equation}\label{frlap2def}
\frlap u(x)= \frac{C(n,s)}{2} \int_{\Rn} \frac{2 u(x) -u(x+y)-u(x-y)} {|y|^{n+2s} } \, dy,\end{equation} where $C(n,s)$ is a dimensional\footnote{The explicit
value of~$C(n,s)$ is usually unimportant. Nevertheless,
we will compute its value explicitly in formulas~\eqref{cnsgalattica}
and~\eqref{EXPL}.
The reason for which it is convenient to divide~$C(n,s)$ by a factor~$2$
in~\eqref{frlap2def} will be clear later on, in formula~\eqref{frlapdef}.}
constant.

One sees from~\eqref{frlap2def}
that $(-\Delta)^s$
is an operator of order~$2s$, namely, it arises from
a differential quotient of order~$2s$ weighted in the whole
space. 
%Different fractional operators have been
%considered in literature
%(see e.g. \cite{CAPUTO, SeV14, MUSINA-NAZAROV}), and all of them
%come from interesting problems in pure or/and applied
%mathematics. We will focus here
%on the operator in~\eqref{frlap2def}
%and we will motivate it by probabilistic considerations
%(as a matter of fact, many other motivations are possible).

The probabilistic model under consideration is a random process
that allows long jumps
(in further generality, it is known that the fractional Laplacian 
is an infinitesimal generator of L\`evy processes,
see e.g.~\cite{B96,Applebaum} for further details). 
A more detailed
mathematical introduction
to the fractional Laplacian
is presented in the subsequent
Section \ref{sspn}.

\section{The random walk with arbitrarily long jumps}\label{srw}

We will show here that the fractional heat equation (i.e. the ``typical'' equation that drives the fractional diffusion and that can be written, up to dimensional constants, as~$\partial_t u+ (-\Delta)^s u=0$) naturally arises from a probabilistic process in which a particle moves randomly in the space subject to a probability that allows long jumps with a polynomial tail.

For this scope, we introduce a probability distribution on the natural numbers~$\N^*:=\{1,2,3,\cdots\}$ as follows. If $I\subseteq \N^*$, then the probability of $I$ is defined to be 
	\[ P(I):= c_s \, \sum_{k\in I} \frac{1}{|k|^{1+2s}}.\]
The constant $c_s$ is taken in order to normalize $P$ to be a probability measure. Namely, we take 
	\[ c_s:=\left( \sum_{k\in \N^*} \frac{1}{|k|^{1+2s}}\right)^{-1},\]
so that we have~$P(\N^*)=1$.

Now we consider a particle that moves in $\R^n$ according to a probabilistic process. The process will be discrete both in time and space
(in the end, we will formally take the limit when these time and space steps are small).
We denote by $\tau$ the discrete time step, and by $h$ the discrete space step. We will take the scaling $\tau=h^{2s}$ and we denote by $u(x,t)$ the probability of finding the particle at the point $x$ at time $t$.

The particle in $\R^n$ is supposed to move according to the following probabilistic law: at each time step~$\tau$, the particle selects randomly both a direction $v\in \partial B_1$, according to the uniform distribution on $\partial B_1$, and a natural number~$k\in\N^*$,
according to the probability law $P$, and it moves by a discrete space step $khv$. Notice that long jumps are allowed with small probability. Then, if the particle is at time~$t$ at the point~$x_0$ and, following the probability law, it picks up a direction~$v\in \partial B_1$
and a natural number~$k\in\N^*$, then the particle at time~$t+\tau$ will lie at~$x_0+khv$.

Now, the probability~$u(x,t+\tau)$ of finding the particle at~$x$ at time~$t+\tau$ is the sum of the probabilities of finding the
particle somewhere else, say at~$x+khv$, for some direction~$v\in \partial B_1$ and some natural number~$k\in\N^*$, times the probability of having selected such a direction and such a natural number.
%
%\begin{center}
\begin{figure}[htpb]
%	\hspace{0.5cm}
	\begin{minipage}[b]{0.70\linewidth}
%	\centering[width=0.7\textwidth]
	\includegraphics[width=0.7\textwidth]{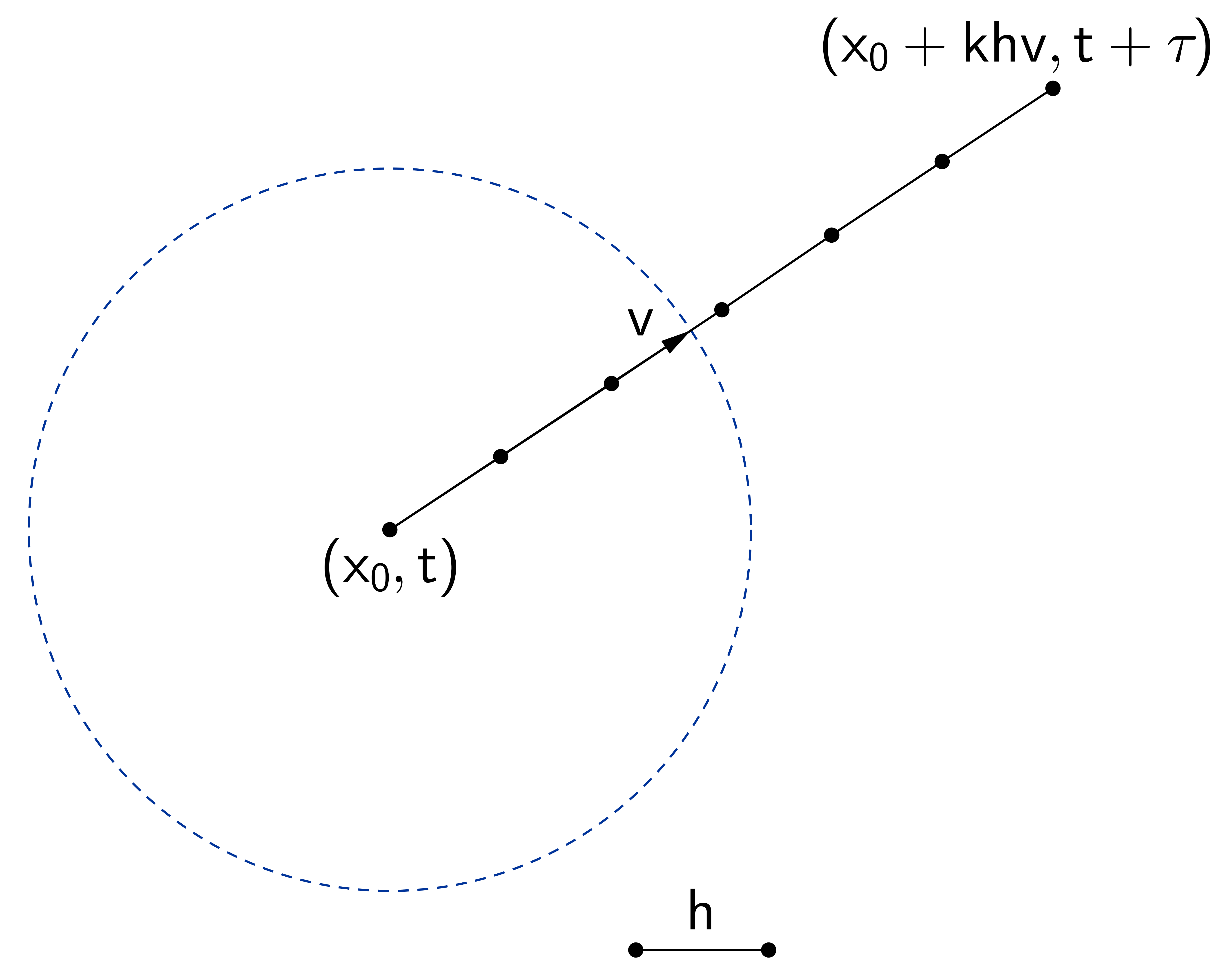}
	\caption{The random walk with jumps}   
	\label{fign:rndw}
	\end{minipage}
\end{figure}
%\end{center}

\noindent This translates into 	
	\[ u(x,t+\tau) =\frac{c_s}{|\partial B_1|} \sum_{k\in\N^*} \int_{\partial B_1}
\frac{u(x+khv,t)}{|k|^{1+2s}} \,d\mathcal{H}^{n-1}(v).\]
Notice that the factor $c_s/|\partial B_1|$ is a normalizing probability
constant, hence we subtract~$u(x,t)$ and we obtain
	\bgs{
	u(x,t+\tau) -u(x,t)=\al
	\frac{c_s}{|\partial B_1|} \sum_{k\in\N^*} \int_{\partial B_1}
	 \frac{u(x+khv,t)}{|k|^{1+2s}} \,d\mathcal{H}^{n-1}(v) - u(x,t)\\
	=\al
	\frac{c_s}{|\partial B_1|} \sum_{k\in\N^*} \int_{\partial B_1}
	\frac{u(x+khv,t)-u(x,t)}{|k|^{1+2s}} \,d\mathcal{H}^{n-1}(v).
	}
As a matter of fact, by symmetry, we can change~$v$ to~$-v$ in the integral above, so we find that
	\[ u(x,t+\tau) - u(x,t) = \frac{c_s}{|\partial B_1|}  \sum_{k\in\N^*} \int_{\partial B_1} \frac{u(x-khv,t)-u(x,t)}{|k|^{1+2s}} \,d\mathcal{H}^{n-1}(v).\]
Then we can sum up these two expressions (and divide by $2$) and obtain that
	\bgs { u(&x,t+ \tau)  -u(x,t) \\
	\al = \frac{c_s}{2\,|\partial B_1|}  \sum_{k\in\N^*} \int_{\partial B_1} \frac{u(x+khv,t)+u(x-khv,t)-2u(x,t)}{|k|^{1+2s}}\,d\mathcal{H}^{n-1}(v) .} 
Now we divide by~$\tau=h^{2s}$, we recognize a Riemann sum, we take a formal limit and we use polar coordinates, thus obtaining:
	\bgs{
		\partial_t u(x,t)\simeq & \frac{ u(x,t+\tau)-u(x,t)}{\tau} 
		\\ =&
		\frac{c_s\,h}{2\,|\partial B_1|} \sum_{k\in\N^*} \int_{\partial B_1}
		\frac{u(x+khv,t)+u(x-khv,t)-2u(x,t)}{|hk|^{1+2s}}d\mathcal{H}^{n-1}(v) 
		\\ \simeq&  \frac{c_s}{2\,|\partial B_1|} \int_0^{+\infty}
		\int_{\partial B_1}
		\frac{u(x+rv,t)+u(x-rv,t)-2u(x,t)}{|r|^{1+2s}}d\mathcal{H}^{n-1}(v) \,dr \\
		=& \frac{c_s}{2\,|\partial B_1|} \int_{\R^n}
		\frac{u(x+y,t)+u(x-y,t)-2u(x,t)}{|y|^{n+2s}} \,dy\\
		=& - c_{n,s} \,(-\Delta)^s u(x,t)	}
for a suitable $c_{n,s}>0$.

This shows that, at least formally, for small time and space steps, the above probabilistic process approaches a fractional heat equation.\bigskip

We observe that processes of this type occur in nature quite
often, see in particular the biological observations
in~\cite{MARINE}, 
%ALBA}
other interesting observations in \cite{reynolds_levy,schouten,woycz} and the mathematical discussions
in~\cite{KAO, FRIEDMAN, MONTEFUSCO-PELLACCI-VERZINI, MASSACCESI,metzler}.

Roughly speaking, let us say that it is not unreasonable that
a predator may decide to use a nonlocal dispersive strategy
to hunt its preys more efficiently
(or, equivalently, that the natural selection may favor some kind of
nonlocal diffusion):
small fishes will not wait to be eaten by a big fish once they
have seen it, so it may be more convenient for the big fish
just to pick up a random direction, move rapidly
in that direction, stop quickly and eat the small fishes
there (if any) and then go on with the hunt. And this ``hit-and-run''
hunting procedure
seems quite related to that described in Figure~\ref{fign:rndw}.

\section{A payoff model}\label{pym}

Another probabilistic motivation for the fractional Laplacian arises from a payoff approach. Suppose to move in a domain~$\Omega$ according to a random walk with jumps as discussed in Section~\ref{srw}. Suppose also that exiting the domain~$\Omega$ for the first time by jumping to an outside point~$y\in\Rn \setminus\Omega$, means earning $u_0(y)$ Monopoly money. A relevant question is, of course, how rich we expect to become
in this way. That is, if we start at a given point~$x\in\Omega$ and we denote by~$u(x)$ the amount of Monopoly money that we expect to gain, is there a way to obtain information on~$u$?

The answer is that 
(in the right scale limit of the random walk with jumps presented in Section~\ref{srw}) the expected payoff~$u$ is determined by the equation
	\begin{equation}\label{HAS}
			\begin{cases}
				(-\Delta)^s u =0 & {\mbox{ in }}\Omega,\\
				u=u_0 &{\mbox{ in }}\R^n\setminus\Omega.
			\end{cases}
	\end{equation}
To better explain this,
let us fix a point~$x\in\Omega$. The expected value of the payoff at~$x$ is the average of all the payoffs at the points~$\tilde x$ from which one can reach~$x$, weighted by the probability of the jumps. That is, by writing~$\tilde x=x+khv$, with~$v\in\partial B_1$, $k\in\N^*$ and~$h>0$, as in the previous Section~\ref {srw}, we have that the probability of jump is~$\displaystyle \frac{c_s}{|\partial B_1|\,|k|^{1+2s}}$. This leads to the formula 
	\[ u(x)= \frac{c_s}{|\partial B_1|} \sum_{k\in\N^*} \int_{\partial B_1} \frac{u(x+khv)}{|k|^{1+2s}}\,d{\mathcal{H}}^{n-1}(v).\]
By changing~$v$ into~$-v$ we obtain 
  \[ u(x)= \frac{c_s}{|\partial B_1|} \sum_{k\in\N^*}\int_{\partial B_1} \frac{u(x-khv)}{|k|^{1+2s}}\,d{\mathcal{H}}^{n-1}(v)\]
 so, by summing up,
	\[  2u(x)= \frac{c_s}{|\partial B_1|} \sum_{k\in\N^*}\int_{\partial B_1} \frac{u(x+khv)+u(x-khv)}{|k|^{1+2s}}\,d{\mathcal{H}}^{n-1}(v).\]
Since the total probability is~$1$, we can subtract~$2u(x)$ to both sides and obtain that
 \[ 0 = \frac{c_s}{|\partial B_1|} \sum_{k\in\N^*}\int_{\partial B_1} \frac{u(x+khv)+u(x-khv)-2u(x) }{|k|^{1+2s}}\,d{\mathcal{H}}^{n-1}(v). \]
We can now divide by~$h^{1+2s}$ and recognize a Riemann sum, which, after passing to the limit as~$h\searrow0$, gives~$0=-(-\Delta)^s
u(x)$,
that is~\eqref{HAS}.

\chapter{The fractional Laplacian and the Caputo derivative} \label{chap3}
\begin{abstract}We introduce here some preliminary notions on the fractional Laplacian and on fractional Sobolev spaces. The definition and equivalent representations for the fractional Laplacian are introduced and the constant that appears in this definition is explicitly computed. Fractional Sobolev spaces enjoy quite a number of important functional inequalities.
We will present here two important inequalities which have a simple and nice proof, namely the fractional Sobolev Inequality and the Generalized Coarea Formula. Moreover, we present an explicit example of an $s$-harmonic function on the positive half-line, i.e. $\frlap (x_+)^s=0$ on $\R_+$ and an example of a function with constant Laplacian on the ball. We also discuss some maximum principles and a Harnack inequality, and present a quite surprising local density property of $s$-harmonic functions into the space of smooth functions. In the last section, we prove that Caputo-stationary functions enjoy the same property, that is they locally approximate any given smooth function.\end{abstract}

\section{The fractional Laplacian}\label{sspn}
We introduce here the fractional Laplace operator, the fractional 
Sobolev spaces and give some useful pieces of notation.
We also refer to~\cite{galattica} for further details related to the topic.

%We consider the  Schwartz space of rapidly decaying functions defined as 
%	\eqlab{\label{schsp}\mathcal{S} ({\Rn}): = \left \{  f \in C^\infty(\Rn)\;  \big|   \; \forall \alpha, \, \beta \in \mathbf{N}^n_0, \,  \sup_{x\in {\Rn}} |x^{\alpha} \partial_{\beta} f(x)|<  \infty \right \}.  }
%	In other words, the Schwartz space consists of smooth functions whose derivatives (including the function itself) decay at infinity faster than any power of $x$. Endowed with the family of seminorms
%	\begin{equation} \label{seminormss} [f]_{\Sa(\Rn)}^{\alpha,N} =\sup_{x\in {\Rn}} (1+|x|)^N  \sum_{|\alpha| \leq N}| D^{\alpha} f(x)|, \end{equation}
%	the Schwartz space is a locally convex topological space. We denote by $\Sa'(\Rn)$ the space of tempered distributions, the topological dual of $\Sa(\Rn)$.
%	
%For any $f \in \mathcal{S}(\Rn)$, denoting the space variable $x \in \Rn$ and the frequency variable $\xi \in \Rn$, the Fourier transform
%and the inverse Fourier transform are defined,
%respectively, as
%		\begin{equation*} 
%		\widehat f(\xi) := \mathcal{F} f(\xi):= 
%\int_{{\Rn}} f(x) e^{-  2\pi i x \cdot \xi} \, dx
%		\label{transF} 
%	\end{equation*}
%and
%	\begin{equation*}	\label{invF}
%		f(x)=  \mathcal{F}^{-1} \widehat f(x)  = \int_{\Rn} \widehat f (\xi) e^{2\pi i x \cdot \xi} \, d\xi.
%	\end{equation*} 
%	We recall that the pointwise product is taken into the convolution product and vice versa, namely for all $ f,\,g \in \Sa(\Rn)$
%	\eqlab{\label{Fconv}  \mathcal{F} (f*g)=  \mathcal{F}(f) \, \mathcal{F}(g) . }
Another useful notion for the fractional Laplacian (other than the definition \eqref{frlap2def}) is the one of principal value, namely we consider
the definition
\begin{equation}\label{PV-1} \mbox{P.V.}
\int_{\Rn}\frac{u(x) - u(y)}{|x- y|^{n+2s}}\, dy :=
\lim_{\eee \to 0} \int_{{\Rn}\setminus
B_{\eee}(x)} \frac{u(x) - u(y)}{|x- y|^{n+2s}}\, dy.\end{equation}
Notice indeed that the integrand above
is singular when~$y$ is in a neighborhood
of~$x$, and this singularity is, in general,
not integrable (in the sense of Lebesgue). As a matter of fact, near~$x$
we have that~$u(x) - u(y)$ behaves at the first order
like~$\nabla u(x)\cdot(x-y)$,
hence the integral above behaves at the first order like
\begin{equation}\label{PV-2}
\frac{ \nabla u(x)\cdot(x-y) }{|x- y|^{n+2s}}\end{equation}
whose absolute value gives an infinite integral near~$x$
(unless either~$\nabla u(x)=0$
or~$s<1/2$).

The idea of the definition in~\eqref{PV-1}
is that the term in~\eqref{PV-2} averages out
in a neighborhood of~$x$
by symmetry, since the term is odd with respect to~$x$,
and so it does not contribute to the integral if we perform it
in a symmetrical way. In a sense,
the principal value in~\eqref{PV-1}
kills the first order of the function at the numerator,
which produces a linear growth, and focuses
on the second order remainders.

The notation in~\eqref{PV-1} allows us to write~\eqref{frlap2def}
in the following more compact form:
\begin{equation*}\begin{split}
	\frlap u(x)\,& =\,\frac{C(n,s)}{2}  \int_{\Rn} \frac{2 u(x) -u(x+y)-u(x-y)} {|y|^{n+2s}}\, dy\\
	&=\, \,\frac{C(n,s)}{2}\lim_{\eee \to 0} \int_{{\Rn}\setminus B_{\eee}}\frac{2 u(x) -u(x+y)-u(x-y)} {|y|^{n+2s}}\, dy \\
	&=\,\frac{C(n,s)}{2}	\lim_{\eee \to 0} \left[\int_{{\Rn}\setminus B_{\eee}}\frac{u(x) -u(x+y)} {|y|^{n+2s}}\, dy+\int_{{\Rn}\setminus B_{\eee}}\frac{u(x) -u(x-y)} {|y|^{n+2s}}\, dy\right]\\
	&=\,\frac{C(n,s)}{2}\lim_{\eee \to 0}\left[\int_{{\Rn}\setminus B_{\eee}(x)}\frac{u(x) -u(\eta)} {|x-\eta|^{n+2s}}\, d\eta+\int_{{\Rn}\setminus B_{\eee}(x)}\frac{u(x) -u(\zeta)} {|x-\zeta|^{n+2s}}\, d\zeta\right]\\
	&=\,C(n,s)\,\lim_{\eee \to 0}\int_{{\Rn}\setminus B_{\eee}(x)}\frac{u(x) -u(\eta)} {|x-\eta|^{n+2s}}\, d\eta,
	\end{split}\end{equation*}
where the changes of variable~$\eta:=x+y$ and~$\zeta:=x-y$
were used, i.e.
\begin{equation}\label{frlapdef}
	\frlap u(x) =C(n,s)\, \mbox{P.V.} \int_{\Rn}\frac{u(x) - u(y)}{|x- y|^{n+2s}}\, dy .\end{equation}
The simplification
above also explains why it was convenient to
write~\eqref{frlap2def} with the factor~$2$
dividing~$C(n,s)$.
Notice that the
expression in~\eqref{frlap2def}
does not require the P.V. formulation since, for instance, taking $u\in L^{\infty}(\Rn)$ and locally $C^2$, using a Taylor expansion of $u$ in $B_1$, one observes that 
	\[ \begin{split}
		 \int_{\Rn} \al \frac{|2 u(x) -u(x+y)-u(x-y)|} {|y|^{n+2s} } \, dy \\
		\leq \al   \|u\|_{L^{\infty}(\Rn)} \int_{\Rn \setminus B_1} |y|^{-n-2s} \, dy + \int_{B_1} \frac{|D^2u(x)| |y|^2}{|y|^{n+2s}}\, dy \\
	\leq \al \|u\|_{L^{\infty}(\Rn)}\int_{\Rn \setminus B_1}   |y|^{-n-2s}\, dy + \|D^2 u\|_{L^{\infty}(\Rn)}\int_{B_1} |y|^{-n-2s +2 }\, dy ,\\
\end{split} \] 
and the integrals above provide a finite quantity.

 As a further remark, definition \eqref{frlapdef} is well posed for $u\in L^1_s(\Rn)$ and locally $C^{2s+\eee}$, where the space
	\eqlab{ \label{ls1space} L^1_s (\Rn):=\Big\{ u \in L^1_{\text{loc}}(\Rn)\; \mbox{ s.t. } \; \int_{\Rn}\frac{ |u(x)|}{1+|x|^{n+2s}} \, dx <\infty\Big\}}
	is endowed naturally with the norm
		\[ \| u\|_{L^1_s({\Rn})} := \int_{\Rn}\frac{ |u(x)|}{1+|x|^{n+2s}} \, dx.\]
Moreover, for $\eee>0$ a small fixed quantity, we write $C^{2s+\eee}$ to denote both $C^{0,2s+\eee}$ for $s<1/2$  and $C^{1,2s+\eee-1}$ for $s\geq 1/2$.

\smallskip
Formula~\eqref{frlapdef}
has also a stimulating analogy with the classical Laplacian.
Namely, the classical Laplacian (up to normalizing constants)
is the measure of the infinitesimal displacement of
a function in average
(this is the ``elastic'' property
of harmonic functions, whose value at a given point
tends to revert to the average in a ball). Indeed,
by canceling the odd contributions, and using that
\begin{eqnarray*}&&\int_{B_r(x)} |x-y|^2 \,dy =
\sum_{k=1}^n \int_{B_r(x)} (x_k-y_k)^2 \,dy
= n \int_{B_r(x)} (x_i-y_i)^2 \,dy,
\\&&\qquad{\mbox{ for any }}i\in\{1,\dots,n\},\end{eqnarray*}
we see that
\begin{equation}\label{KJ896}
\begin{split}
 \lim_{r\to0} \frac{1}{r^2}&
\left( u(x)-\frac{1}{|B_r(x)|}\, \int_{B_r(x)} u(y)\,dy\right)\\
=\;& 
\lim_{r\to0} -\frac{1}{r^2 |B_r(x)| } \int_{B_r(x)} \big(u(y)-u(x) \big) dy
\\
=\;& 
\lim_{r\to0} -\frac{1}{r^{n+2}\,|B_1|}\, \int_{B_r(x)} \nabla u(x)\cdot(x-y)
+\frac{1}{2} D^2u(x) (x-y)\cdot(x-y)
\\ &\qquad+{\mathcal{O}} (|x-y|^3)\,dy
\\
=\;&
\lim_{r\to0} -\frac{1}{2r^{n+2}\,|B_1|}\, \sum_{i,j=1}^n \int_{B_r(x)} 
\partial^2_{i,j} u(x)\,(x_i-y_i)(x_j-y_j) \,dy \\
=\;&
\lim_{r\to0} -\frac{1}{2r^{n+2}\,|B_1|}\, \sum_{i=1}^n \int_{B_r(x)}
\partial^2_{i,i} u(x)\,(x_i-y_i)^2 \,dy \\
=\;&
\lim_{r\to0} -\frac{1}{2n \,r^{n+2}\,|B_1|}\, \sum_{i=1}^n 
\partial^2_{i,i} u(x)\,
\int_{B_r(x)} |x-y|^2 \,dy \\
=\;& -C_n \Delta u(x),
\end{split}\end{equation}
for some~$C_n>0$. In this spirit,
when we compare this formula with~\eqref{frlapdef},
we can think that the fractional Laplacian corresponds to
a weighted average of the function's oscillation.
While the average in~\eqref{KJ896}
is localized in the vicinity of a point~$x$, the one in~\eqref{frlapdef}
occurs in the whole space (though it decays at infinity).
Also, the spacial homogeneity of the average in~\eqref{KJ896}
has an extra factor that is proportional to the space variables to
the power~$-2$, while the corresponding power in the average 
in~\eqref{frlapdef} is~$-2s$ (and this is consistent for~$s\to1$).
\bigskip

We use the usual notations for the Fourier and inverse Fourier transform (see Appendix \ref{Four}). For $u \in \mathcal{S}(\Rn)$, the fractional Laplace
operator can be 
expressed in Fourier frequency variables multiplied
by~$(2\pi|\xi|)^{2s}$,
as stated in the following lemma. 
\begin{lemma}\label{frlaphdeflem}
We have that
	\begin{equation} 
		\frlap u(x)=\mathcal{F}^{-1} \big((2\pi|\xi|)^{2s} \widehat u (\xi)\big) .
		\label{frlaphdef} 
	\end{equation}
\end{lemma}

Roughly speaking, formula~\eqref{frlaphdef}
characterizes the fractional Laplace operator in the Fourier
space, by taking the $s$-power of the multiplier associated
to the classical Laplacian operator.
Indeed, by using the inverse Fourier transform, one has that
\begin{eqnarray*}
&& -\Delta u(x) =-\Delta ({\mathcal{F}}^{-1}(\widehat u))(x)
=-\Delta
\int_{{\Rn}} \widehat u(\xi) e^{2\pi i x \cdot \xi} \, d\xi
\\ &&\qquad = 
\int_{{\Rn}} (2\pi |\xi|)^2
\widehat u(\xi) e^{2\pi i x \cdot \xi} \, d\xi
= {\mathcal{F}}^{-1} \big( (2\pi |\xi|)^2
\widehat u(\xi)\big),\end{eqnarray*}
which gives that the classical Laplacian acts in a Fourier space
as a multiplier of~$(2\pi |\xi|)^2$. {F}rom this and
Lemma~\ref{frlaphdeflem} it also follows
that the classical Laplacian is the limit case of the 
fractional one, namely for any $u \in \mathcal{S}(\Rn)$
	\[\lim_{s\to 1} \frlap u =-\Delta u  \quad \mbox{ and also }\quad \lim_{s\to 0}\frlap  u = u.\]

Let us now
prove that indeed the two formulations \eqref{frlap2def} and \eqref{frlaphdef} are equivalent.
\begin{proof}[Proof of Lemma \ref{frlaphdeflem}]
Consider identity \eqref{frlap2def} and apply the Fourier transform to obtain 
	\begin{equation}\label{3.7bis} \begin{split}
		  \mathcal{F} \Big(\frlap u(x)\Big)  = &\; \frac{C(n,s)}{2} \int_{\Rn} \frac{\mathcal{F}\Big(2u(x) -u(x+y)-u(x-y)\Big)}{|y|^{n+2s}} \, dy \\ 
		= &\; \frac{C(n,s)}{2} \int_{\Rn} \widehat{u}(\xi)   \frac{2- e^{2\pi
i\xi \cdot  y} - e^{-2\pi
i\xi \cdot  y} }  {|y|^{n+2s}}\, dy   \\
		= &\; C(n,s) \, \widehat{u}(\xi)  \int_{\Rn}   \frac{1-\cos(2\pi\xi \cdot y) }  {|y|^{n+2s}}\, dy.   
	\end{split}\end{equation} 
Now, we use the change of variable $z= |\xi| y$ and get that
		\[ \begin{split}
		J (\xi)  : =  \int_{\Rn} \frac{1-\cos(2\pi\xi \cdot  y)}{|y|^{n+2s}} \, dy
		= |\xi|^{2s} \int_{\Rn} \frac{1-\cos\frac{2\pi\xi}{|\xi|}\cdot z} {|z|^{n+2s}} \, dz.	 
		\end{split}\] 
Since~$J$ is rotationally invariant, we consider a rotation $R$ that sends~$e_1=(1,0,\dots,0)$ into~$\xi/|\xi|$,
that is~$Re_1= \xi/|\xi|$,
and we
call $R^T$ its transpose. Then, by using the change of variables $\omega=R^T z$ we have that
	\begin{equation*}
\begin{split}
		 J(\xi) = &\;	|\xi|^{2s} \int_{\Rn} \frac{1-\cos (2\pi
Re_1 \cdot z)}{|z|^{n+2s}}\, dz
				= 	|\xi|^{2s} \int_{\Rn} \frac{1-\cos(2\pi R^T z \cdot e_1)}{|R^Tz|^{n+2s}}\, dz \\
				= &\;|\xi|^{2s}  \int_{\Rn} \frac{1-\cos (2\pi
\omega_1)}{|\omega|^{n+2s}}\, d\omega.
	\end{split}\end{equation*}
Changing variables $\tilde \omega = 2\pi\omega$ (we still write $\omega $ as a variable of integration), we obtain that
	\eqlab{ \label{37BB}
	 	J(\xi)  = \left(2\pi|\xi|\right)^{2s}  \int_{\Rn} \frac{1-\cos \omega_1}{|\omega|^{n+2s}}\, d\omega.}
Notice that this integral is finite.
Indeed, integrating outside the ball $B_1$ we have that 
	\[ 	 \int_{{\Rn}\setminus B_1} \frac{|1-\cos\omega_1|}{|\omega|^{n+2s}}  \, d\omega  \leq \int_{{\Rn}\setminus B_1} \frac{2}{|\omega|^{n+2s}} < \infty, 	\]
while inside the ball we can use the Taylor expansion of the cosine function and observe that
	\[ \int_{B_1} \frac{|1-\cos\omega_1|}{|\omega|^{n+2s}}  \, d\omega \leq \int_{B_1} \frac{|\omega|^2}{|\omega|^{n+2s}}  \, d\omega \leq \int_{B_1} \frac{ d\omega}{|\omega|^{n+2s-2}} < \infty. \]
Therefore, by taking
	\begin{equation}
		 C(n, s) := \bigg(\int_{\Rn} \frac{1 - \cos\omega_1}{|\omega|^{n+2s}} \, d\omega\bigg)^{-1} 
		\label{cnsgalattica}
	\end{equation}
it follows from~\eqref{37BB} that
	\[J(\xi) = \frac{\left(2\pi |\xi|\right)^{2s}}{C(n,s)}.\]
By inserting this into~\eqref{3.7bis}, we obtain that
$$ \mathcal{F} \Big(\frlap u(x)\Big)= C(n,s)\,\widehat u(\xi)\,J(\xi)=
(2\pi|\xi|)^{2s}\widehat u(\xi),$$
which concludes the proof.
\end{proof}

Notice that the renormalization constant $C(n,s)$ introduced in~\eqref{frlap2def} is now computed in~\eqref{cnsgalattica}.
\bigskip

Another approach to the fractional Laplacian comes
from the theory of semigroups (or, equivalently from the fractional calculus
arising in subordination identities). This technique
is classical (see~\cite{YOSIDA}),
but it has also been efficiently used in recent research papers
(see for
instance~\cite{LLAVE-VALDINOCI-AIHP, stingatorrea2, STINGA-CAFFARELLI}).
Roughly speaking, the main idea underneath
the semigroup approach comes from the explicit
formulas for the Euler's function (check the Appendix \ref{special}): for any~$\lambda>0$, one
uses an integration by parts and the substitution~$\tau=\lambda t$
to see that
\bgs{  -s\Gamma(-s) =\al \Gamma(1-s) = \int_0^{+\infty} \tau^{-s} e^{-\tau}\,d\tau  
	= - \int_0^{+\infty} \tau^{-s} \frac{d}{d\tau} (e^{-\tau}-1)\,d\tau
\\=\al 
-s\int_0^{+\infty} \tau^{-s-1} (e^{-\tau}-1)\,d\tau 
 = -s\lambda^{-s} \int_0^{+\infty} t^{-s-1} (e^{-\lambda t}-1)\,dt,}
that is
\begin{equation}\label{L-EQ} \lambda^s =
\frac{1}{\Gamma(-s)} \int_0^{+\infty} t^{-s-1} (e^{-\lambda t}-1)\,dt.\end{equation}
Then one applies formally this identity to~$\lambda:=-\Delta$. Of course, this formal step needs to be justified, but if
things go well one obtains
$$ (-\Delta)^s =
\frac{1}{\Gamma(-s)} \int_0^{+\infty} t^{-s-1} (e^{t\Delta}-1)\,dt,$$
that is (interpreting~$1$ as the identity operator)
\begin{equation}\label{FP}
(-\Delta)^s u(x) =
\frac{1}{\Gamma(-s)} \int_0^{+\infty} t^{-s-1} (e^{t\Delta} u(x)-u(x))\,dt.
\end{equation}
Formally, if~$U(x,t):=e^{t\Delta}u(x)$, we have that~$U(x,0)=u(x)$ and
$$ \partial_t U=\frac{\partial}{\partial t} ( e^{t\Delta} u(x) ) = \Delta
e^{t\Delta} u(x) =\Delta U,$$
that is~$U(x,t)=e^{t\Delta}u(x)$ can be interpreted as the solution
of the heat equation with initial datum~$u$.
We indeed point out that these formal computations can be justified:

\begin{lemma}\label{L32}
Formula~\eqref{FP} holds true. That is, if~$u\in{\mathcal{S}}(\R^n)$
and~$U=U(x,t)$ is
the solution of the heat equation
$$ \left\{
\begin{matrix}
\partial_t U=\Delta U & {\mbox{ for }}t>0,\\
U\big|_{t=0} = u,
\end{matrix}
\right. $$
then
\begin{equation}\label{EX-CO} (-\Delta)^s u(x) = 
\frac{1}{\Gamma(-s)} \int_0^{+\infty} t^{-s-1} (U(x,t)-u(x))\,dt.
\end{equation}
\end{lemma}

\begin{proof} {F}rom Theorem~1 on page~47
in~\cite{EVANS}
we know that~$U$ is obtained by Gaussian convolution
with unit mass, i.e.
\eqlab{\label{G-EQ} 
& U(x,t)=\int_{\R^n} G(x-y,t)\,u(y)\,dy=\int_{\R^n} G(y,t)\,u(x-y)\,dy, \qquad{\mbox{ where }}\\
&\;
G(x,t):= (4\pi t)^{-n/2} e^{-\frac{|x|^2}{4t}}.}
As a consequence, using the substitution~$\tau:=|y|^2/(4t)$,
\bgs{  \int_0^{+\infty} & t^{-s-1} (U(x,t)-u(x))\,dt \\ =\al  \int_0^{+\infty}\left[ \int_{\R^n} t^{-s-1}  G(y,t)\,\big( u(x-y)-u(x)\big) \,dy \right]\,dt \\
	=\al \int_0^{+\infty}\left[ \int_{\R^n} (4\pi t)^{-n/2} t^{-s-1}  e^{-|y|^2/(4t)} \,\big( u(x-y)-u(x)\big) \,dy \right]\,dt\\
	=\al \int_0^{+\infty}\left[\int_{\R^n} \tau^{n/2} (\pi |y|^2)^{-n/2} |y|^{-2s} (4\tau)^{s+1}  e^{-\tau} \,\big( u(x-y) - u(x)\big) \,dy \right]\,\frac{ d\tau }{4\tau^2 } \\
	=\al 2^{2s-1}\pi^{-n/2} \int_0^{+\infty}\left[ \int_{\R^n}  \tau^{\frac{n}{2}+s-1} e^{-\tau}\frac{ u(x+y)+u(x-y)-2u(x) }{ |y|^{n+2s} }\,dy \right]\,d\tau.
}
Now we notice that
$$ \int_0^{+\infty}\tau^{\frac{n}{2}+s-1} e^{-\tau}\,d\tau=\Gamma\left(
\frac{n}{2}+s\right),$$
so we obtain that
\begin{equation*}\begin{split} \int_0^{+\infty} &t^{-s-1} (U(x,t)-u(x))\,dt
\\=\al 2^{2s-1}\pi^{-n/2}\Gamma\left(\frac{n}{2}+s\right)
\int_{\R^n} \frac{ u(x+y)+u(x-y)-2u(x) }{ |y|^{n+2s} } \,dy
\\=\al -\frac{ 2^{2s}\,\pi^{-n/2}\,\Gamma\left(\frac{n}{2}+s\right) }{ C(n,s) }(-\Delta)^s u(x).\end{split}\end{equation*}
This proves~\eqref{EX-CO}, by choosing $C(n,s)$ appropriately.
And, as a matter of fact, gives the explicit value of the constant~$C(n,s)$
as
\begin{equation}\label{EXPL}
C(n,s)=-\frac{2^{2s}\,\Gamma\left( \frac{n}{2}+s\right)}{\pi^{n/2} 
\Gamma(-s)}
=
\frac{2^{2s}\,s\,\Gamma\left( \frac{n}{2}+s\right)}{\pi^{n/2}
\Gamma(1-s)},
\end{equation}
where we have used again that~$\Gamma(1-s)=-s\Gamma(-s)$, for any~$s\in(0,1)$.
\end{proof}

It is worth pointing out that the renormalization constant~$C(n,s)$
introduced in~\eqref{frlap2def}
has now been explicitly computed in~\eqref{EXPL}.
Notice that the choices of~$C(n,s)$
in~\eqref{cnsgalattica}
and~\eqref{EXPL} must agree (since we have computed
the fractional Laplacian in two different ways). We give below a direct proof
that the settings in~\eqref{cnsgalattica}
and~\eqref{EXPL} are the same, by using
Fourier methods and~\eqref{L-EQ}. 

\begin{lemma}\label{LM-COST}
For any~$n\in\N$, $n\ge1$, and~$s\in(0,1)$, we have that
\begin{equation}\label{EQW-IO-1}
\int_{\R^n} \frac{1-\cos(2\pi\omega_1)}{|\omega|^{n+2s}}\,d\omega
=\frac{\pi^{\frac{n}{2}+2s}\,\Gamma(1-s)}{
s\,\Gamma\left( \frac{n}{2}+s\right)}.\end{equation}
Equivalently, we have that
\begin{equation}\label{EQW-IO-2}
\int_{\R^n} \frac{1-\cos\omega_1}{|\omega|^{n+2s}}\,d\omega
=\frac{\pi^{\frac{n}{2}}\,\Gamma(1-s)}{
	2^{2s}s\,\Gamma\left( \frac{n}{2}+s\right)}. \end{equation}
\end{lemma}

\begin{proof} Of course, formula~\eqref{EQW-IO-1}
is equivalent to~\eqref{EQW-IO-2} (after the substitution~$\tilde\omega:=
2\pi\omega$).
Strictly speaking, in Lemma~\ref{frlaphdeflem}
(compare \eqref{frlap2def}, \eqref{frlaphdef}, and~\eqref{cnsgalattica})
we have proved that 
\eqlab{ \label{PoP-0} \frac{1} {2\displaystyle \int_{\Rn}  \displaystyle \frac{1 -\cos\omega_1}{|\omega|^{n+2s}} \, d\omega } \, \int_{\Rn} \frac{2 u(x) -u(x+y)-u(x-y)} {|y|^{n+2s} } \, dy 
=  {\mathcal{F}}^{-1} \big((2\pi|\xi|)^{2s} \widehat u(\xi)\big).}
Similarly, by means of Lemma~\ref{L32}
(compare \eqref{frlap2def}, 
\eqref{EX-CO} and~\eqref{EXPL}) we know that
\begin{equation}\label{PoP-1}\begin{split}&
\frac{2^{2s-1}\,s\,\Gamma\left( \frac{n}{2}+s\right)}{\pi^{n/2}
\Gamma(1-s)} \int_{\Rn} \frac{2 u(x) -u(x+y)-u(x-y)} {|y|^{n+2s} } \, dy\\&\qquad=
\frac{1}{\Gamma(-s)} \int_0^{+\infty} t^{-s-1} (U(x,t)-u(x))\,dt.\end{split}\end{equation}
Moreover (see~\eqref{G-EQ}), we have that~$U(x,t):=G(\cdot,t) * u(x)$.
%, where $$G(x,t)=(4\pi t)^{-n/2} e^{-|x|^2/(4t)}.$$
We recall that the Fourier transform of
a Gaussian is a Gaussian itself, namely
$$ {\mathcal{F}}(e^{-\pi|x|^2}) =e^{-\pi|\xi|^2},$$
therefore, for any fixed~$t>0$, using the substitution~$y:=x/\sqrt{4\pi t}$,
\begin{eqnarray*}
{\mathcal{F}}\,G(\xi,t) &=& 
\frac{1}{(4\pi t)^{n/2}} \int_{\R^n} 
e^{-|x|^2/(4t)} e^{-2\pi ix\cdot \xi}\,dx
\\ &=& 
\int_{\R^n}
e^{-\pi|y|^2} e^{-2\pi iy\cdot (\sqrt{4\pi t}\xi)}\,dy = e^{-4\pi^2 t|\xi|^2}.
\end{eqnarray*}
As a consequence
\begin{eqnarray*}
&& {\mathcal{F}} \big(U(x,t)-u(x)\big)
={\mathcal{F}}\big( G(\cdot,t)*u(x)-u(x)\big)
\\ &&\qquad ={\mathcal{F}}(G(\cdot,t)*u)(\xi)-\widehat u(\xi)
=\big({\mathcal{F}}\,G(\xi,t))-1\big)\widehat u(\xi)
\\ &&\qquad = (e^{-4\pi^2 t|\xi|^2}-1)\widehat u(\xi).
\end{eqnarray*}
We multiply by~$ t^{-s-1}$ and integrate over~$t>0$,
and we obtain
\begin{eqnarray*}
{\mathcal{F}} \int_0^{+\infty} t^{-s-1}\big(U(x,t)-u(x)\big)\,dt
&=& \int_0^{+\infty} t^{-s-1}(e^{-4\pi^2 t|\xi|^2}-1)\,dt\,
\widehat u(\xi)
\\&=&\Gamma(-s)\, (4\pi^2 |\xi|^2)^s\,\widehat u(\xi),\end{eqnarray*}
thanks to~\eqref{L-EQ} (used here with~$\lambda:=4\pi^2 |\xi|^2$).
By taking the inverse Fourier transform, we have
$$ \int_0^{+\infty} t^{-s-1}\big(U(x,t)-u(x)\big)\,dt =
\Gamma(-s)\, (2 \pi)^{2s} {\mathcal{F}}^{-1}
\big(|\xi|^{2s}\,\widehat u(\xi)\big).$$
We insert this information into~\eqref{PoP-1}
and we get
$$ \frac{2^{2s-1}\,s\,\Gamma\left( \frac{n}{2}+s\right)}{\pi^{n/2}
\Gamma(1-s)} \int_{\Rn} \frac{2 u(x) -u(x+y)-u(x-y)} {|y|^{n+2s} } \, dy=
(2 \pi)^{2s} {\mathcal{F}}^{-1}
\big(|\xi|^{2s}\,\widehat u(\xi)\big).$$
Hence, recalling~\eqref{PoP-0},
\bgs{ &\frac{2^{2s-1}\,s\,\Gamma\left( \frac{n}{2}+s\right)}{\pi^{n/2}
\Gamma(1-s)} \int_{\Rn} \frac{2 u(x) -u(x+y)-u(x-y)} {|y|^{n+2s} } \, dy
\\ &\qquad=
\frac{1}{2\displaystyle\int_{\Rn} \displaystyle \frac{1 -
\cos\omega_1}{|\omega|^{n+2s}} \, d\omega}\,
\int_{\Rn} \frac{2 u(x) -u(x+y)-u(x-y)} {|y|^{n+2s} } \, dy
,}
which gives the desired result.
\end{proof}

An alternative proof of Lemma \ref{LM-COST} is given in the subsequent Theorem \ref{thm:Cc} in Chapter \ref{chap4}, by using the potential theory approach. 
%For the sake of completeness,
%a different proof will be given in
%Subsection~\ref{1.1.A}. There, to prove Lemma~\ref{LM-COST},
%we will use the theory of special functions rather than
%the fractional Laplacian. 
For other approaches to the proof of
Lemma~\ref{LM-COST} see also the recent PhD dissertations~\cite{Felsinger} (and related \cite{Felsinger2})
and~\cite{Jarohs}.\bigskip

\subsection{Fractional Sobolev Inequality and Generalized Coarea Formula}\label{sobineq}
\noindent Fractional Sobo-\\lev spaces enjoy quite
a number of important functional inequalities.
It is almost impossible to list here all the results
and the possible applications, therefore we will
only present two important inequalities which have a simple and nice proof,
namely the fractional Sobolev Inequality and the Generalized Coarea Formula.

The fractional Sobolev Inequality can be written as follows:

\begin{theorem}
For any~$s\in(0,1)$, $p\in\left(1,\frac{n}{s}\right)$ and~$u\in C^\infty_0(\R^n)$,
\begin{equation}\label{OK:sob:p}
\| u \|_{L^{\frac{np}{n-sp}}(\R^n)}\le C\left(
\int_{\R^n}\int_{\R^n}\frac{|u(x)-u(y)|^p}{|x-y|^{n+sp}}
\,dx\,dy\right)^{\frac{1}{p}},\end{equation}
for some~$C>0$, depending only on~$n$ and~$p$.
\end{theorem}

\begin{proof} We follow here the very nice proof given in~\cite{PONCE} (where, in turn, the proof is attributed to Ha\"{\i}m Brezis).
We have that
	\[ |u(x)|\leq |u(x)-u(y)|+ |u(y)|.\]
For a fixed $R$ (that will be given later on), we integrate over the ball $B_R(x)$ and have that
	\eqlab{ \label{sobin1}  |B_R(x)||u(x)|\leq \int_{B_R(x)} |u(x)-u(y)|\, dy + \int_{B_R(x)}|u(y)|\, dy =I_1+I_2.}
	We apply the H{\"o}lder inequality for the exponents $p$ and $p/(p-1)$ in the first integral and obtain that
	\bgs{ I_1=\al \int_{B_R(x)}\frac{|u(x)-u(y)|}{|x-y|^{\frac{n+sp}{p}}} |x-y|^{\frac{n+sp}{p}} \, dy \\
	\leq \al R^{\frac{n+sp}{p}} \lr{\int_{B_R(x)} \frac{|u(x)-u(y)|^p}{|x-y|^{n+sp}}\, dy }^{\frac{1}p}\lr{\int_{B_R(x)} dy}^{\frac{p-1}p}\\
	= \al C R^{n+s} \lr{\int_{\Rn} \frac{|u(x)-u(y)|^p}{|x-y|^{n+sp}}\, dy }^{\frac{1}p} .}
The H{\"o}lder inequality for $\frac{np}{n-sp}$ and $\frac{np}{n(p-1)+sp}$ gives in the second integral
\bgs{ I_2 \leq\al  \lr{\int_{B_R(x) }|u(y)|^{\frac{np}{n-sp}} \, dy}^{\frac{n-sp}{np}}\lr{ \int_{B_R(x)} dy}^ {\frac{n(p-1)+sp}{np}} \\
	\leq \al R^{\frac{n(p-1) +sp}{p}} \lr{ \int_{\Rn}|u(y)|^{\frac{np}{n-sp}}\, dy}^{\frac{n-sp}{np}}.} 
Dividing by $R^{n}$ in \eqref{sobin1} and renaiming the constants, it follows that
	\[ |u(x)|\leq C R^s \lrq{ \lr{\int_{\Rn}  \frac{|u(x)-u(y)|^p}{|x-y|^{n+sp}}\, dy }^{\frac{1}p}  + R^{-\frac{n}p} \lr{ \int_{\Rn}|u(y)|^{\frac{np}{n-sp}}\, dy}^{\frac{n-sp}{np}}   },\] where $C=C(n,p)>0$.
	We take now $R$ such that	
	\[  \lr{ \int_{\Rn} \frac{|u(x)-u(y)|^p}{|x-y|^{n+sp}}\, dy }^{\frac{1}p}  = R^{-\frac{n}p} \lr{ \int_{\Rn}|u(y)|^{\frac{np}{n-sp}}\, dy}^{\frac{n-sp}{np}} \] and we obtain
	\[ |u(x)|\leq C  \lr{\int_{\Rn}  \frac{|u(x)-u(y)|^p}{|x-y|^{n+sp}}\, dy }^{\frac{n-sp}{np}} \lr{ \int_{\Rn}|u(y)|^{\frac{np}{n-sp}}\, dy}^{\frac{s(n-sp)}{n^2}}.\] 
	Raising to the power $\frac{np} {n-sp}$ and integrating over $\Rn$, we get that
		\[ \int_{\Rn} |u(x)|^{\frac{np} {n-sp}} \, dx \leq C \lr{\iint_{\R^{n}\times \Rn} \frac{|u(x)-u(y)|^p}{|x-y|^{n+sp}}\, dx\, dy }\lr{\int_{\Rn} |u(y)|^{\frac{np}{n-sp}} \, dy }^{\frac{ps}n}.\] After a simplification, we obtain that
		\[\lr{ \int_{\Rn} |u(x)|^{\frac{np}{n-sp}} \, dx} ^{\frac{n-sp}{np}} \leq C \lr{ \iint_{\R^{n}\times \Rn} \frac{|u(x)-u(y)|^p}{|x-y|^{n+sp}}\, dx\, dy}^{\frac{1}p}. \]
which is~\eqref{OK:sob:p}. This concludes the proof of the Theorem.
\end{proof}

What follows is the Generalized Co-area Formula of~\cite{VISENTIN}
(the link with the classical Co-area Formula will be indeed
more evident in terms of the fractional perimeter functional
discussed in Chapter~\ref{nlms}).

\begin{theorem}
For any~$s\in(0,1)$ and any measurable function~$u:\Omega\to[0,1]$,
$$ \frac12 \int_\Omega\int_\Omega \frac{|u(x)-u(y)|}{|x-y|^{n+s}}\,dx\,dy
=\int_0^1 \left(\int_{\{x\in\Omega,\;
u(x)>t\}}\int_{\{y\in\Omega,\;u(y)\le t\}}
\frac{dx\,dy}{|x-y|^{n+s}}\right)\,dt.$$
\end{theorem}

\begin{proof} We claim that for any~$x$, $y\in\Omega$
\begin{equation}\label{CV-OA:1}
|u(x)-u(y)|=\int_0^1 \Big(
\chi_{\{u>t\}}(x) \,\chi_{\{u\le t\}}(y)+
\chi_{\{u\le t\}}(x)\,\chi_{\{u>t\}}(y) \Big)\,dt.
\end{equation}
To prove this, we fix $x $ and $y$ in $\Omega$, and by possibly exchanging them, we can suppose that~$
u(x)\ge u(y)$. Then, we define
$$\varphi(t):=
\chi_{\{u>t\}}(x) \,\chi_{\{u\le t\}}(y)+
\chi_{\{u\le t\}}(x)\,\chi_{\{u>t\}}(y) .$$
By construction
\bgs{ \varphi(t)=\left\{\begin{matrix}
		0 & {\mbox{ if }} & &t<u(y) {\mbox{ and }} t\ge u(x),\\
		1 & {\mbox{ if }} && u(y) \leq t<u(x) ,\\
		\end{matrix}
\right.}
therefore
$$ \int_0^1 \varphi(t)\,dt =
\int_{u(y)}^{u(x)} \,dt=u(x)-u(y),$$
which proves \eqref{CV-OA:1}.

So, multiplying by the singular kernel
and integrating~\eqref{CV-OA:1} over~$\Omega\times\Omega$, we obtain
that
\begin{eqnarray*}&&
\int_\Omega\int_\Omega \frac{|u(x)-u(y)|}{|x-y|^{n+s}}\,dx\,dy\\&=&
\int_0^1 \left(
\int_\Omega\int_\Omega
\frac{
\chi_{\{u>t\}}(x) \,\chi_{\{u\le t\}}(y)+
\chi_{\{u\le t\}}(x)\,\chi_{\{u>t\}}(y) }{|x-y|^{n+s}}\,dx\,dy\right)\,dt
\\ &=&
\int_0^1 \left(
\int_{\{u>t\}}\int_{\{u\le t\}}
\frac{dx\,dy}{|x-y|^{n+s}}
+\int_{\{u\le t\}}
\int_{\{u>t\}}
\frac{dx\,dy}{|x-y|^{n+s}}
\right)\,dt
\\ &=& 2\int_0^1\left( \int_{\{u>t\}}\int_{\{u\le t\}}
\frac{dx\,dy}{|x-y|^{n+s}}\right)\,dt,
\end{eqnarray*}
as desired.
\end{proof}

\subsection{Maximum Principle and Harnack Inequality}\label{S:PGRAG}

The Harnack Inequality and \label{PGRAG} the Maximum Principle for harmonic functions are classical topics in elliptic regularity theory. Namely, in the classical case, if a non-negative function is harmonic in~$B_1$ and~$r\in(0,1)$, then its minimum and maximum in~$B_{r}$ must always be comparable (in particular, the function cannot touch the level zero
in~$B_r$).

It is worth pointing out that the fractional counterpart of these facts is, in general, false, as this next simple
result shows (see \cite{K15}):

\begin{theorem} \label{NH}
There exists a bounded function~$u$ which is~$s$-harmonic in~$B_1$, non-negative in~$B_1$, but such that~$\displaystyle \inf_{B_1} u =0$.
\end{theorem}
\begin{proof}[Sketch of the proof]
The main idea is that we are able to take the datum of~$u$ outside~$B_1$ in a suitable way as to ``bend down'' the function inside~$B_1$ until it reaches the level zero. Namely, let~$M\ge 0$ and we take~$u_M$ to be the function satisfying
	\begin{equation}\label{e56}
		\begin{cases}
			(-\Delta)^s u_M =0 & {\mbox{ in }} B_1,\\
			u_M = 1-M & {\mbox{ in }} B_3\setminus B_2,\\
			u_M =1 & {\mbox{ in }} \R^n\setminus \big((B_3\setminus B_2)\cup B_1\big).
	\end{cases}
	\end{equation}
When~$M=0$, the function~$u_M$ is identically~$1$.
When~$M>0$, we expect~$u_M$ to bend down, since the fact that the fractional Laplacian vanishes in $B_1$ forces the second order quotient to vanish 
in average (recall~\eqref{frlap2def},
or the equivalent formulation in~\eqref{frlapdef}).
Indeed, we claim that there exists~$M_\star>0$ such that~$u_{M_\star}\ge0$ in~$B_1$ with~$\displaystyle \inf_{B_1} u_{M_\star}=0$. Then, the result of
Theorem~\ref{NH} would be reached by taking~$u:=u_{M_\star}$.

To check the existence of such~$M_\star$,
we show that~$\displaystyle \inf_{B_1} u_M\to -\infty$ as~$M\to+\infty$.
Indeed, we argue by contradiction and suppose this cannot happen. Then, for any~$M\ge0$, we would have that
\begin{equation}\label{e1234}
\inf_{B_1} u_M\ge -a,\end{equation} for some fixed~$a\in\R$.
We set
	\[  v_M:= \frac{u_M+M-1}{M}.\]
Then, by~\eqref{e56}, 
	\[  \begin{cases}
			(-\Delta)^s v_M =0 & {\mbox{ in }} B_1,\\
			v_M = 0 & {\mbox{ in }} B_3\setminus B_2,\\
			v_M =1 & {\mbox{ in }} \R^n\setminus \big((B_3\setminus B_2)\cup B_1\big).
	\end{cases}\]
Also, by~\eqref{e1234}, for any~$x\in B_1$,
	\[  v_M(x)\ge \frac{-a+M-1}{M}.\]
By taking limits, one obtains that~$v_M$ approaches a function~$v_\infty$ that satisfies
 	\[ \begin{cases}
(-\Delta)^s v_\infty =0 & {\mbox{ in }} B_1,\\
v_\infty= 0& {\mbox{ in }} B_3\setminus B_2,\\
v_\infty = 1 & {\mbox{ in }} \R^n\setminus \big((B_3\setminus B_2)\cup 
B_1\big)
\end{cases} \]
and, for any~$x\in B_1$,	
		\[ v_\infty (x)\ge 1.\] 
In particular the maximum of~$v_\infty$ is attained at some point~$x_\star\in B_1$, with~$v_\infty(x_\star)\ge 1$. Accordingly, 
	\begin{eqnarray*}
&& 0 = P.V. \int_{\R^n} \frac{v_\infty(x_\star)-v_\infty(y)}{|x_\star-y|^{n+2s}}\,dy
\ge P.V. \int_{B_3\setminus B_2} 
\frac{v_\infty(x_\star)-v_\infty(y)}{|x_\star-y|^{n+2s}}\,dy
\\ &&\qquad\ge
P.V. \int_{B_3\setminus B_2} \frac{1-0}{|x_\star-y|^{n+2s}}\,dy>0,
\end{eqnarray*}
which is a contradiction.
\end{proof}

The example provided by Theorem~\ref{NH} is not the end
of the story concerning the Harnack Inequality
in the fractional setting.
On the one hand, Theorem~\ref{NH} is just a particular case of the very dramatic effect that the datum at infinity may have on the fractional Laplacian (a striking example of this phenomenon will be given in Section \ref{afsh}).
On the other hand, the Harnack Inequality and the Maximum Principle hold true if, for instance, the sign of the function~$u$ is controlled in the whole of~$\Rn$.

We refer to \cite{BASS-KASSMANN, SILV-HOELDER, K15,FerFra} and to the references therein for a detailed introduction to the fractional Harnack Inequality, and to \cite{DCKP14} for general estimates of this type.

Just to point out
the validity of a global Maximum Principle, 
we state in detail the following simple result:

\begin{theorem}\label{THM-MA-1}
If~$(-\Delta)^s u\ge0$ in~$B_1$ and~$u\ge0$ in~$\R^n\setminus B_1$,
then~$u\ge 0$ in~$B_1$.\end{theorem}

\begin{proof} Suppose by contradiction that the minimal point~$x_\star\in B_1$ satisfies~$u(x_\star)<0$. Then $u(x_*)$ is a minimum in $\Rn$ (since $u$ is positive outside $B_1$), if $y \in B_2$ we have that $2u(x_\star) -u(x_\star+y)-u(x_\star-y)  \leq 0$. On the other hand, in $\Rn \setminus B_2$ we have that $x_\star\pm y \in \Rn \setminus B_1$, hence $u(x_\star\pm y)\geq 0$. We thus have	\[\begin{split}  0\leq \al \int_{\R^n} \frac{2u(x_\star)-u(x_\star+y)-u(x_\star-y)}{|y|^{n+2s}}\,dy\\
		\leq \al  \int_{\R^n\setminus B_2} \frac{2u(x_\star)-u(x_\star+y)-u(x_\star-y)}{|y|^{n+2s}}\,dy \\ 
		\leq \al \int_{\R^n\setminus B_2}  \frac{2u(x_\star)}{|y|^{n+2s}}\,dy<0. \end{split}\]
This leads to a contradiction.\end{proof}

Similarly to Theorem~\ref{THM-MA-1}, one can prove a Strong Maximum Principle,
such as:

\begin{theorem}\label{THM-MA-1-STRONG}
If~$(-\Delta)^s u\ge0$
in~$B_1$ and~$u\ge0$ in~$\R^n\setminus B_1$,
then~$u>0$ in~$B_1$, unless~$u$ vanishes identically.\end{theorem}

\begin{proof}
We observe that we already know that~$u\ge0$
in the whole of~$\R^n$, thanks to Theorem~\ref{THM-MA-1}.
Hence, if $u$ is not strictly positive, there
exists~$x_0\in B_1$
such that~$u(x_0)=0$. This gives that
\[  0\le \int_{\R^n} \frac{2u(x_0)-u(x_0+y)-u(x_0-y)}{|y|^{n+2s}}\,dy
=- \int_{\Rn}\frac{u(x_0+y)+u(x_0-y)}{|y|^{n+2s}}\,dy.\]
Now both $u(x_0+y)$ and  $u(x_0-y)$ 
are non-negative, hence the latter
integral is less than or equal to zero,
and so it must vanish identically, proving that~$u$ also
vanishes identically.
\end{proof}

A simple version of a Harnack-type
inequality in the fractional setting can be also obtained as follows:

\begin{prop}
Assume that~$(-\Delta)^s u\ge0$ in~$B_2$, with~$u\ge0$ in
the whole of~$\R^n$. Then
$$ u(0)\ge c\int_{B_1} u(x)\,dx,$$
for a suitable~$c>0$.
\end{prop}

\begin{proof} Let~$\Gamma\in C^\infty_0(B_{1/2})$,
with~$\Gamma(x)\in[0,1]$ for any~$x\in\R^n$, and~$\Gamma(0)=1$.
We fix~$\epsilon>0$, to be taken arbitrarily small at the end of this
proof and set
\begin{equation}\label{78HJP}
\eta:=u(0)+\epsilon>0.\end{equation}
We define~$\Gamma_a(x):= 2\eta\,\Gamma(x)-a$.
Notice that if~$a>2\eta$, then~$\Gamma_a (x) \le2\eta-a<0\leq u(x)$
in the whole of~$\R^n$, hence the set $\{ \Gamma_a<u {\mbox{ in $\R^n$}}\}$ is not empty, and we can define
$$ a_* := \inf_{a\in \R} \{ \Gamma_a<u {\mbox{ in $\R^n$}}\}.$$
By construction
\begin{equation*}\label{89U}
a_* \le 2\eta.
\end{equation*}
If~$a<\eta$ then~$\Gamma_a(0)=2\eta-a > \eta >u(0)$, therefore
\begin{equation}\label{U7-a-be}
a_* \ge \eta.
\end{equation}
Notice that
\begin{equation}\label{x0EX0}
{\mbox{$\Gamma_{a_*}\le u$ in the whole of $\R^n$.}}\end{equation}
We claim that 
\begin{equation}\label{x0EX}
{\mbox{there exists~$x_0\in \overline{B_{1/2}}$ such that $\Gamma_{a_*}(x_0)=u(x_0)$.}}
\end{equation}
To prove this, we suppose by contradiction that~$u>\Gamma_{a_*} $
in~$\overline{B_{1/2}}$, i.e.
$$ \mu:=\min_{ \overline{B_{1/2}} } (u-\Gamma_{a_*} )>0.$$
Also, if~$x\in \R^n\setminus \overline{B_{1/2}}$, we have that
$$ u(x)-\Gamma_{a_*}(x) = u(x)
-2\eta\,\Gamma(x)+a_* = u(x)+a_*\ge a_*\ge\eta,$$
thanks to~\eqref{U7-a-be}.
As a consequence, for any~$x\in\R^n$,
$$ u(x)-\Gamma_{a_*}(x) \ge\min\{\mu,\eta\}=:\mu_*>0.$$
So, if we define~$a_\sharp:= a_*-(\mu_*/2)$,
we have that~$a_\sharp<a_*$ and
$$ u(x)-\Gamma_{a_\sharp}(x) = u(x)-\Gamma_{a_*}(x) -\frac{\mu_*}{2}
\ge \frac{\mu_*}{2} >0.$$
This is in contradiction with the definition of~$a_*$
and so it proves~\eqref{x0EX}.

{F}rom~\eqref{x0EX} we have that $x_0\in \overline{ B_{1/2}}$, hence ~$(-\Delta)^s u(x_0)\ge0$.
Also~$|(-\Delta)^s \Gamma_{a_*}(x)|=
2\eta\,|(-\Delta)^s\Gamma(x)|\le C\eta$, for any~$x\in\R^n$, therefore,
recalling~\eqref{x0EX0} and~\eqref{x0EX},
\begin{eqnarray*}
C\eta &\ge& (-\Delta)^s \Gamma_{a_*}(x_0) -(-\Delta)^s u(x_0)\\
&=& C(n,s)\,\mbox{P.V.} \int_{\R^n} 
\frac{ \big[ \Gamma_{a_*}(x_0)-\Gamma_{a_*}(x_0+y) \big]- 
\big[ u(x_0)-u(x_0+y)\big] }{|y|^{n+2s}}\,dy \\
&=& C(n,s)\,\mbox{P.V.} \int_{\R^n} 
\frac{ u(x_0+y)-\Gamma_{a_*}(x_0+y)}{|y|^{n+2s}}\,dy
\\ &\ge& C(n,s)\,\mbox{P.V.} \int_{B_1(-x_0)}
\frac{ u(x_0+y)-\Gamma_{a_*}(x_0+y)}{|y|^{n+2s}}\,dy
.\end{eqnarray*}
Notice now that if~$y\in B_1(-x_0)$, then~$|y|\le |x_0|+1<2$,
thus we obtain
$$ C\eta \ge \frac{C(n,s)}{2^{n+2s}}\,\int_{B_1(-x_0)}
\big[ u(x_0+y)-\Gamma_{a_*}(x_0+y) \big]\,dy.$$
Notice now that $\Gamma_{a_*}(x)= 2\eta\Gamma(x) -a_* \le \eta$, due to~\eqref{U7-a-be},
therefore we conclude that
$$ C\eta \ge \frac{C(n,s)}{2^{n+2s}}\,\left( \int_{B_1(-x_0)}
u(x_0+y)\,dy-\eta|B_1|\right).$$
That is,
using the change of variable~$x:=x_0+y$, recalling~\eqref{78HJP}
and renaming the constants, we have
$$ C\big(u(0)+\epsilon\big)=
C\eta\ge \int_{B_1} u(x)\,dx,$$
hence the desired result follows by sending~$\epsilon\to0$.
\end{proof}

\subsection{An $s$-harmonic function} \label{shf1}

We provide here an explicit
example of a function that is $s$-harmonic on the positive line $\R_+:=
(0,+\infty)$. Namely, we prove the following result: 
\begin{theorem}\label{G1}
For any $x\in\R$, let $w_s(x):=x_+^s=\max\{x,0\}^s$. Then
	\[ (-\Delta)^s w_s (x) = \left\{
	\begin{matrix}
		-c_s |x|^{-s} & \mbox{ if  } x<0,\\
		0 &\mbox{ if } x>0,
	\end{matrix}
	\right.\]
for a suitable constant~$c_s>0$.
\end{theorem}
%\begin{center}
\begin{figure}[htpb]
%\sidecaption
%	\hspace{0.8cm}
	\begin{minipage}[b]{0.90\linewidth}
%	\centering
	\includegraphics[width=0.85\textwidth]{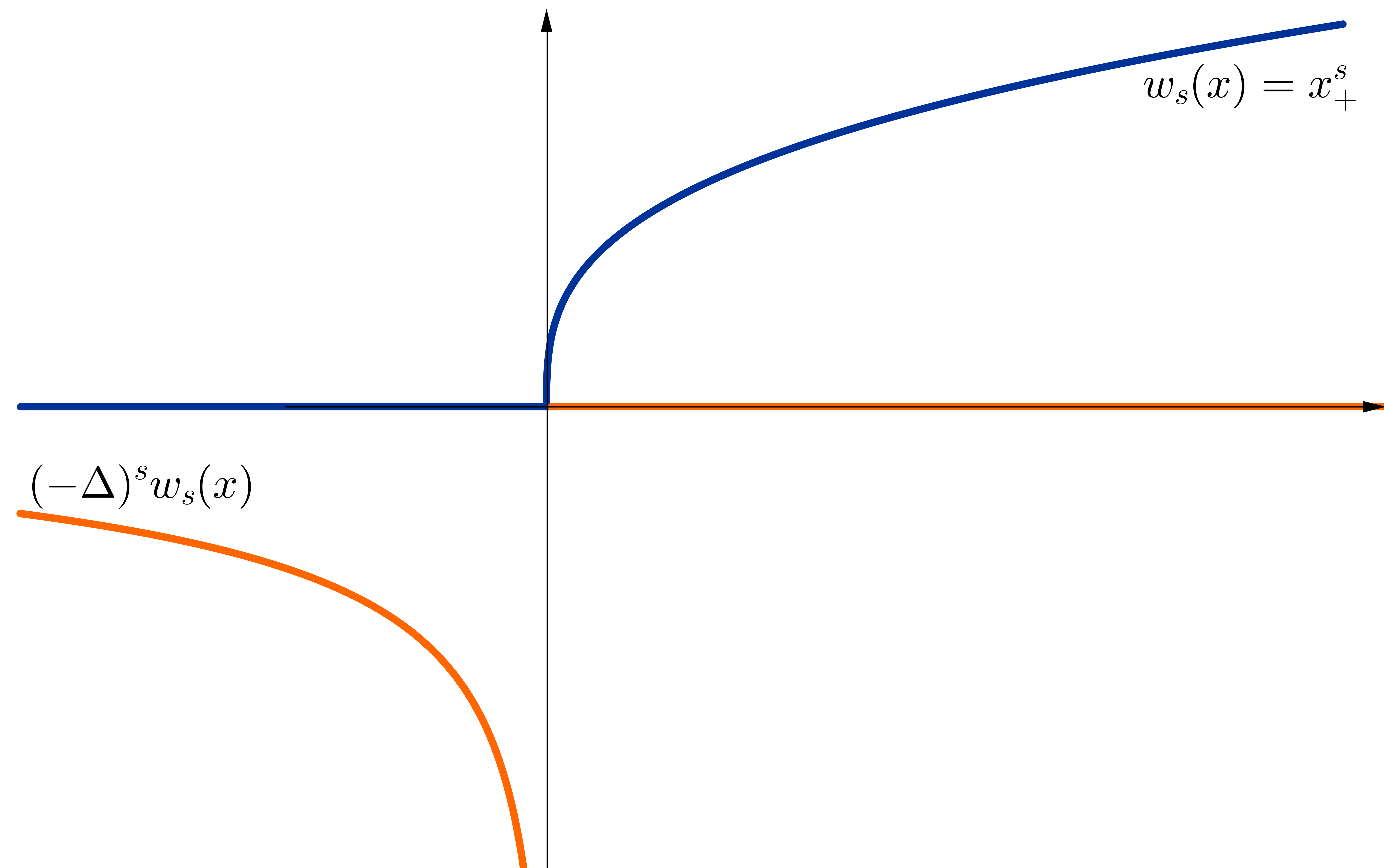}
	\caption{An $s$-harmonic function}   
	\label{fign:sh}
	\end{minipage}
\end{figure}
%\end{center}
At a first glance, it may be quite surprising that the function~$x_+^s$ is $s$-harmonic in~$(0,+\infty)$, since such function is not smooth (but only continuous) uniformly up to the boundary, so let us try to give some heuristic explanations for it. 

We try to understand why the function~$x_+^s$ is $s$-harmonic in, say, the interval~$(0,1)$ when~$s\in(0,1]$. From the discussion in Section~\ref{pym}, we know that the $s$-harmonic function in~$(0,1)$ that agrees with~$x_+^s$ outside~$(0,1)$ coincides with the expected value of a payoff, subject to a random walk (the random walk is classical when~$s=1$ and it presents jumps when~$s\in(0,1)$). If~$s=1$ and we start from the middle of the interval, we have the same probability of being moved randomly to the left and to the right. This means that we have the same probability of exiting the interval~$(0,1)$ to the right (and so ending the process at~$x=1$, which gives~$1$ as payoff) or to the left~(and so ending the process at~$x=0$, which gives~$0$ as payoff). Therefore the expected value starting at~$x=1/2$ is exactly the average between~$0$ and~$1$, which is~$1/2$. Similarly, if we start the process at the point~$x=1/4$, we have the same probability of reaching the point~$0$ on the left and the point~$1/2$ to the right. Since we know that the payoff at~$x=0$ is~$0$ and the expected value of the payoff at~$x=1/2$ is~$1/2$, we deduce in this case that the expected value for the process starting at~$1/4$ is the average between~$0$ and~$1/2$, that is exactly~$1/4$. We can repeat this argument over and over, and obtain the (rather obvious)
fact that the linear function is indeed harmonic in the classical sense.

The argument above, which seems either
trivial or unnecessarily complicated in the classical case,
can be adapted when~$s\in(0,1)$ and it can give a qualitative picture of the corresponding~$s$-harmonic function. Let us start again the random walk, this time with jumps, at~$x=1/2$: in presence of jumps, we have the same probability of reaching the left interval~$(-\infty,0]$ and the right interval~$[1,+\infty)$. Now, the payoff at~$(-\infty,0]$ is~$0$, while the payoff at~$[1,+\infty)$ is {\em bigger} than~$1$. This implies that the expected value at~$x=1/2$ is the average between~$0$ and something bigger than~$1$, which produces a value larger than~$1/2$. When repeating this argument over and over, we obtain a concavity property enjoyed by the~$s$-harmonic functions in this case (the exact values prescribed in~$[1,+\infty)$ are not essential here, it would be enough that these values were monotone increasing and larger than~$1$).
%\begin{center}
\begin{figure}[htpb]
%\sidecaption
%	\hspace{0.8cm}
	\begin{minipage}[b]{0.95\linewidth}
%	\centering
	\includegraphics[width=0.90\textwidth]{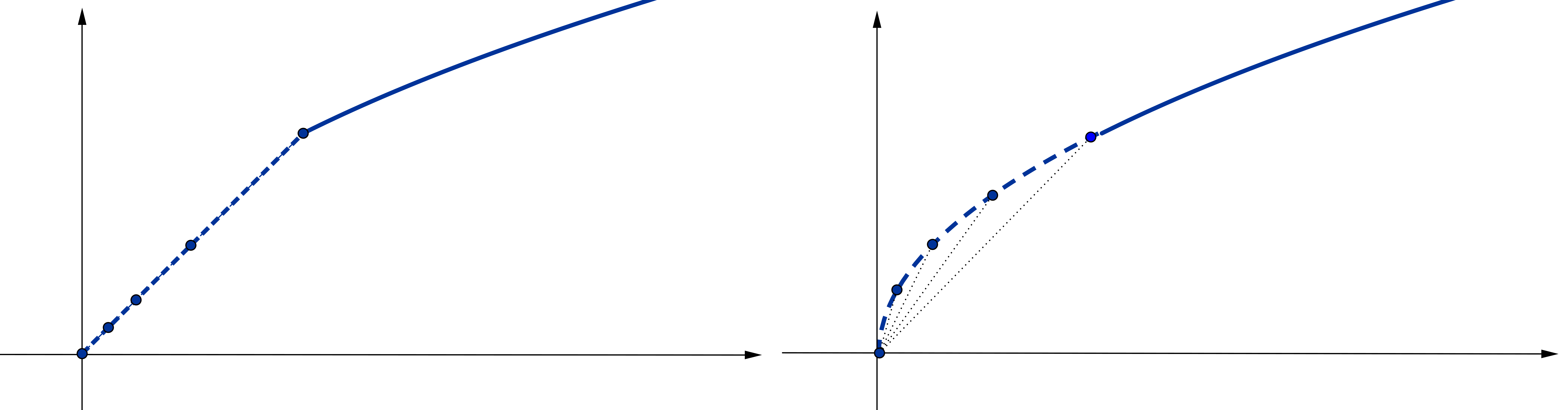}
	\caption{A payoff model: case $ s=1$ and $s \in (0,1)$}   
	\label{fign:py}
	\end{minipage}
\end{figure}
%\end{center}

In a sense, therefore, this concavity properties and loss of Lipschitz regularity near the minimal boundary values is typical of nonlocal diffusion and it is due to the possibility of ``reaching far away points'', which may increase the expected payoff.

\bigskip

Now we present a complete, elementary proof
of Theorem \ref{G1},
%The proof originated from some pleasant discussions
%with Fernando Soria and it is based on some rather surprising
%integral cancellations.
%The reader who wishes to skip this proof
%can go directly to Section~\ref{S:PGRAG}
%on page~\pageref{PGRAG}. 
%Moreover, a shorter, but
%technically more advanced proof,
%is presented in Subsection~\ref{1.1.A}.
starting with
some preliminary computations.  

\begin{lemma}\label{L1}
For any $s\in(0,1)$
\[ \int_0^{1} \frac{(1+t)^s+(1-t)^s-2}{t^{1+2s}}\,dt+ \int_1^{+\infty}\frac{(1+t)^s}{t^{1+2s}}\, dt = \frac{1}{s}.\]
\end{lemma}

\begin{proof} Fixed $\eee>0$, we integrate by parts:
	\eqlab{ \label{6f7gg}		
	\int_\eee^{1} & \frac{(1+t)^s+(1-t)^s-2}{t^{1+2s}}\,dt \;\\\ =\al  -\frac{1}{2s} \int_\eee^{1} \Big[ (1+t)^s+(1-t)^s-2\Big] \frac{d}{dt} t^{-2s}\,dt \\
	 =&\;\frac{1}{2s} \bigg[\frac{ (1+\eee)^s+(1-\eee)^s-2 }{ \eee^{2s} } -2^s+2\bigg]
	+\frac{1}{2} \int_\eee^{1} \frac{ (1+t)^{s-1}-(1-t)^{s-1}}{ t^{2s}} \,dt \\ 
	 =&\;\frac{1}{2s} \left[ o(1) -2^s+2\right]+\frac{1}{2}\bigg( \int_\eee^{1} (1+t)^{s-1} t^{-2s}\,dt
 	-\int_\eee^{1}(1-t)^{s-1} t^{-2s}\,dt\bigg), }
		with~$o(1)$ infinitesimal as~$\eee\searrow0$.
Moreover, by changing variable~$\tilde t:=t/(1-t)$, that is~$t:=\tilde t/(1+\tilde t)$,
we have that
\begin{equation*}
	\int_\eee^{1} (1-t)^{s-1} t^{-2s}\,dt
= \int_{\eee/(1-\eee)}^{+\infty} (1+\tilde t)^{s-1}\tilde t^{-2s}\,d\tilde t.
\end{equation*}
Inserting this into \eqref{6f7gg} (and writing $t$ instead of~$\tilde t$ as variable of integration), we obtain
	\begin{equation}\label{98-1}
		\begin{split}
		&\int_\eee^{1} \frac{(1+t)^s+(1-t)^s-2}{t^{1+2s}}\,dt \\
		&\quad=\frac{1}{2s} \big[o(1)	-2^s+2\big]+\frac{1}{2}\bigg[\int_\eee^{1} (1+t)^{s-1} t^{-2s}\,dt	- \int_{\eee/(1-\eee)}^{+\infty}(1+t)^{s-1}t^{-2s}\,dt\bigg]\\
		&\quad=	\frac{1}{2s} \big[ o(1)	-2^s+2\big]+\frac{1}{2}\bigg[	\int_\eee^{\eee/(1-\eee)} (1+t)^{s-1} t^{-2s}\,dt- \int_1^{+\infty}(1+t)^{s-1}t^{-2s}\,dt\bigg] .\end{split}\end{equation}
Now we remark that
	\[ \int_\eee^{\eee/(1-\eee)} (1+t)^{s-1} t^{-2s}\,dt\le  \int_\eee^{\eee/(1-\eee)} (1+\eee)^{s-1} \eee^{-2s}\,dt
	=\eee^{2-2s}(1-\eee)^{-1}(1+\eee)^{s-1} ,\]
therefore
	\[\lim_{\eee\searrow0} \int_\eee^{\eee/(1-\eee)} (1+t)^{s-1} t^{-2s}\,dt =0.\]
So, by passing to the limit in~\eqref{98-1}, we get
	\begin{equation}\label{98c}
		\int_0^{1} \frac{(1+t)^s+(1-t)^s-2}{t^{1+2s}}\,dt=
		\frac{-2^s+2}{2s} -\frac{1}{2}
		\int_1^{+\infty} (1+t)^{s-1}t^{-2s}\,dt.
	\end{equation}
Now, integrating by parts we see that
	\begin{eqnarray*}
		 \frac{1}{2}\int_1^{+\infty} (1+t)^{s-1}t^{-2s}\,dt
	&=& \frac{1}{2s}\int_1^{+\infty} t^{-2s} \frac{d}{dt}(1+t)^s\,dt
		\\ &=&-\frac{2^s}{2s}+
		\int_1^{+\infty} t^{-1-2s} (1+t)^s\,dt.
	\end{eqnarray*}
By plugging this into~\eqref{98c} we obtain that
\[\int_0^{1} \frac{(1+t)^s+(1-t)^s-2}{t^{1+2s}}\,dt=
\frac{-2^s+2}{2s}+
\frac{2^s}{2s}-
\int_1^{+\infty} t^{-1-2s} (1+t)^s\,dt , \]
which gives the desired result.
\end{proof}

 {F}rom Lemma~\ref{L1} we deduce the following (somehow unexpected)
cancellation property.

\begin{corollary}\label{L2}
Let~$w_s$ be as in the statement of Theorem~\ref{G1}. Then
$$ (-\Delta)^s w_s (1) = 0.$$
\end{corollary}

\begin{proof} The function~$t\mapsto (1+t)^s+(1-t)^s-2$
is even, therefore
$$ \int_{-1}^1 \frac{(1+t)^s+(1-t)^s-2}{|t|^{1+2s}}\,dt=2
\int_{0}^1 \frac{(1+t)^s+(1-t)^s-2}{t^{1+2s}}\,dt.$$
Moreover, by changing variable~$\tilde t:=-t$, we have that
$$ \int_{-\infty}^{-1}\frac{(1-t)^s-2}{|t|^{1+2s}}\,dt
=\int_1^{+\infty}\frac{(1+\tilde t)^s-2}{\tilde t^{1+2s}}\,d\tilde t.$$
Therefore
\begin{equation*}\begin{split}
\int_{-\infty}^{+\infty}& \frac{w_s(1+t)+w_s(1-t)-2w_s(1)}{|t|^{1+2s}}\,dt\\
=&
\int_{-\infty}^{-1}\frac{(1-t)^s-2}{|t|^{1+2s}}\,dt
+
\int_{-1}^{1}\frac{(1+t)^s+(1-t)^s-2}{|t|^{1+2s}}\,dt
+
\int_{1}^{+\infty}\frac{(1+t)^s-2}{|t|^{1+2s}}\,dt
\\=&
2\int_{0}^{1}\frac{(1+t)^s+(1-t)^s-2}{t^{1+2s}}\,dt
+2\int_{1}^{+\infty}\frac{(1+t)^s-2}{t^{1+2s}}\,dt
\\=&
2\left[
\int_{0}^{1}\frac{(1+t)^s+(1-t)^s-2}{t^{1+2s}}\,dt
+\int_{1}^{+\infty}\frac{(1+t)^s}{t^{1+2s}}\,dt
-2\int_{1}^{+\infty}\frac{dt}{t^{1+2s}}
\right]
\\ =&
2\left[\frac1s-2\int_{1}^{+\infty}\frac{dt}{t^{1+2s}}\right],
\end{split}\end{equation*}
where Lemma~\ref{L1} was used in the last line.
Since
$$ \int_{1}^{+\infty}\frac{dt}{t^{1+2s}} = \frac1{2s},$$
we obtain that
$$ \int_{-\infty}^{+\infty}\frac{w_s(1+t)+w_s(1-t)-2w_s(1)}{|t|^{1+2s}}\,dt=0,$$
that proves the desired claim.
\end{proof}

The counterpart of Corollary~\ref{L2} is given by the following
simple observation:

\begin{lemma}\label{L3}
Let~$w_s$ be as in the statement of Theorem~\ref{G1}. Then
$$ -(-\Delta)^s w_s (-1) > 0.$$
\end{lemma}

\begin{proof} We have that
$$ w_s(-1+t)+w_s(-1-t)-2w_s(-1)=(-1+t)_+^s+(-1-t)_+^s\ge0$$
and not identically zero, which implies the desired result.
\end{proof}

We have now all the elements to proceed to the proof of Theorem~\ref{G1}.
\begin{proof} [Proof of Theorem~\ref{G1}]

We let~$\sigma\in\{+1,-1\}$ denote the sign of a fixed $x\in\R\setminus\{0\}$.
We claim that
\begin{equation}\label{Si}\begin{split}
\int_{-\infty}^{+\infty}&
\frac{w_s(\sigma(1+t))+w_s(\sigma(1-t))-2w_s(\sigma)}{ |t|^{1+2s}}\,dt
\\=\;&\int_{-\infty}^{+\infty}
\frac{w_s(\sigma+t)+w_s(\sigma-t)-2w_s(\sigma)}{ |t|^{1+2s}}\,dt.
\end{split}\end{equation}
Indeed, the formula above is obvious when~$x>0$ (i.e. $\sigma=1$),
so we suppose~$x<0$ (i.e. $\sigma=-1$) and we change variable~$\tau:=-t$,
to see that, in this case,
\begin{equation*}
	\begin{split}
\int_{-\infty}^{+\infty}&
\frac{w_s(\sigma(1+t))+w_s(\sigma(1-t))-2w_s(\sigma)}{ |t|^{1+2s}}\,dt\\ =\; &
\int_{-\infty}^{+\infty}
\frac{w_s(-1-t)+w_s(-1+t)-2w_s(\sigma)}{ |t|^{1+2s}}\,dt \\
 =\; &\int_{-\infty}^{+\infty}
\frac{w_s(-1+\tau)+w_s(-1-\tau)-2w_s(\sigma)}{ |\tau|^{1+2s}}\,d\tau
\\=\; &\int_{-\infty}^{+\infty}
\frac{w_s(\sigma+\tau)+w_s(\sigma-\tau)-2w_s(\sigma)}{ |\tau|^{1+2s}}\,d\tau
,\end{split}
	\end{equation*}
thus checking~\eqref{Si}.

Now we observe that, for any~$r\in\R$,
$$ w_s(|x|r)=(|x|r)_+^s =|x|^s r_+^s =|x|^s w_s(r).$$
That is
$$ w_s(xr)=w_s(\sigma |x| r)=|x|^sw_s(\sigma r).$$ 
So we change variable~$y=tx$ and we obtain that
\begin{equation*}
	\begin{split}
 \int_{-\infty}^{+\infty}&
\frac{w_s(x+y)+w_s(x-y)-2w_s(x)}{|y|^{1+2s}}\,dy\\=\; &
\int_{-\infty}^{+\infty}
\frac{w_s(x(1+t))+w_s(x(1-t))-2w_s(x)}{|x|^{2s} |t|^{1+2s}}\,dt
\\=\; &|x|^{-s} \int_{-\infty}^{+\infty}
\frac{w_s(\sigma(1+t))+w_s(\sigma(1-t))-2w_s(\sigma)}{ |t|^{1+2s}}\,dt
\\=\; & |x|^{-s} \int_{-\infty}^{+\infty}
\frac{w_s(\sigma+t)+w_s(\sigma-t)-2w_s(\sigma)}{ |t|^{1+2s}}\,dt,
\end{split}
	\end{equation*}
where~\eqref{Si} was used in the last line.
This says that
$$ (-\Delta)^s w_s(x)
= \left\{\begin{matrix}
|x|^{-s} \,(-\Delta)^s w_s(-1)& {\mbox{ if $x<0$,}}\\
|x|^{-s}\,(-\Delta)^s w_s(1)&{\mbox{ if $x>0$,}}\end{matrix}
\right.$$
hence the result in Theorem~\ref{G1} follows from Corollary~\ref{L2}
and Lemma~\ref{L3}.
\end{proof}

\subsection {A function with constant fractional Laplacian on the ball}\label{ctfrlap}
In this subsection, we explicitly compute the fractional Laplacian of the function $ \mathcal U(x) = (1-|x|^2)_+^s$ in $B_1$. We have that
	\[ \frlap \mathcal U(x) = C(n,s)\frac{\omega_n}{2} \beta(s, 1-s) \quad \mbox{ any } x\in B_1,\]
	where $C(n,s)$ is given by \eqref{cnsgalattica} and $\beta$ is the special Beta function (see Appendix \ref{special} and references therein).
	%6.2 in \cite{ABRAMOWITZ}). 
	Just to give an idea of how such computation can be obtained, with small modifications respect to \cite{DYDA,DYDA1} we go through the proof of this result. The reader can find the more general result, i.e. for $\mathcal U(x)= (1-|x|^2)_+^p$ for $p>-1$, in the above mentioned \cite{DYDA,DYDA1}. 	
	  
 Let us take $u \colon \R \to \R$ as $u(x) = (1-|x|^2)^s_+$. We consider the regional  fractional Laplacian restricted to $(-1,1)$ 
	\[\Lm u(x) := P.V. \int_{-1}^1 \frac{u(x)-u(y)}{|x-y|^{1+2s}}\, dy  \] and we compute its value at zero. By symmetry we have that
	\[	 		\Lm u(0) = 2 \lim_{\eee\to 0}\int_{\eee}^1 \frac{1-(1-y^2)^s}{y^{1+2s}}\, dy . 	\]
 Changing the variable $\omega= y^2$ and integrating by parts we get that
	\[
		\begin{split}
		\Lm u(0) =& \; 2 \lim_{\eee\to 0} \bigg( \int_{\eee}^1 y^{-1-2s}\, dy - \int_{\eee}^1 (1-y^2)^s y^{-1-2s}\, dy \bigg) \\
		 = & \; -\frac{1}{s} + \lim_{\eee\to 0} \bigg( \frac{\eee^{-2s}}{s}  - \int_{\eee^2}^1 (1-\omega)^s \omega^{-s-1}  \, d\omega\bigg) \\
	= & \; -\frac{1}{s} + \lim_{\eee\to 0} \bigg( \frac{\eee^{-2s}-\eee^{-2s} (1-\eee^2)^s}{s} + \int_{\eee^2}^1 \omega^{-s}(1-\omega)^{s-1}\, d\omega \bigg).
		\end{split} \] 
Using the integral representation of the Beta function (see Appendix \ref{special}, formula \eqref{betazerouno}) %\cite{ABRAMOWITZ}, formula 6.2.1), i.e. 
%\[ \beta(x, y)= \int_0^1 t^{x-1} (1-t)^{y-1} \, dt \quad \mbox{ for } x,y>0\] 
it yields that
	\[
		\Lm u(0) =\beta(1-s,s) -\frac{1}{s}.
	\]	
	For $x\in B_1$ we use the change of variables $ \omega=\frac{x-y}{1-xy}$. We obtain that
\begin{equation}
		\begin{split}\label{llpux}
		\Lm u(x) =  &\; P.V. \int_{-1}^1 \frac{ (1-x^2)^s -(1-y^2)^s}{|x-y|^{1+2s}} \, dy \\
			= &\; (1-x^2)^{-s} P.V. \int_{-1}^1 \frac{(1-\omega x)^{2s-1} -(1-\omega^2)^s(1-\omega x)^{-1}} {|\omega|^{2s+1}}\, d\omega\\
			= &\; (1-x^2)^{-s}P.V. \Bigg( \int_{-1}^1 \frac{1-(1-\omega^2)^s}{|\omega|^{2s+1}} \, d\omega + \int_{-1}^1 \frac{(1-\omega x)^{2s-1}-1}{|\omega|^{2s+1}} \, d\omega    \\
			&\;  +\int_{-1}^1 \frac{(1-\omega^2)^s \Big ( 1- (1-\omega x)^{-1} \Big)}{|\omega|^{2s+1}} \, d\omega  \Bigg) \\
		=&\; (1-x^2)^{-s}\left(  \Lm u(0) + J(x) +I(x) \right),
		\end{split}
	\end{equation}
where we have recognized the regional fractional Laplacian and denoted
	\[ \begin{split} J(x) := &\;P.V. \int_{-1}^1 \frac{(1-\omega x)^{2s-1}-1}{|\omega|^{2s+1}} \, d\omega \qquad \qquad \mbox{ and } \\
	 I(x) :=&\; \text{P.V.} \int_{-1}^1 \frac{1-(1-\omega x)^{-1} } {|\omega|^{2s+1}} (1-\omega^2)^s \, d\omega. \end{split} \]
In $J(x)$ we have that 
	\[ \begin{split}
		J(x)  =&\; P.V.\left( \int_{-1}^1 \frac{(1-\omega x)^{2s-1}}{|\omega|^{2s+1}}  \, d\omega - \int_{-1}^1 |\omega|^{-1-2s} \, d\omega \right)\\
	=&\;  \lim_{\eee \to 0}\bigg( \int_{\eee}^1\frac{(1+\omega x)^{2s-1} +(1-\omega x)^{2s-1} }{|\omega|^{2s+1}} \, d\omega -  2 \int_{\eee}^1 \omega^{-1-2s} \, d\omega \bigg) .\end{split}  \] With the change of variable $t=\displaystyle \frac{1}{\omega}$ we get that
	\begin{equation}\label{jxah} \begin{split}
		J(x)  =&\; \frac{1}{s} + \lim_{\eee \to 0}   \bigg(\int_1^{1/\eee} \Big[ (t+x)^{2s-1} + (t-x)^{2s-1} \Big] \, dt - \frac{\eee^{-2s} }{s}\bigg)\\
		=&\; \frac{1}{s}  - \frac{(1+x)^{2s} +(1-x)^{2s}}{2s} + \frac{1}{2s} \lim_{\eee \to 0}\frac{ (1+\eee x)^{2s} +(1-\eee x)^{2s}   -2}{\eee^{2s}}\\
		=&\; \frac{1}{s}  - \frac{(1+x)^{2s} +(1-x)^{2s}}{2s}.
		\end{split} \end{equation}
To compute $I(x)$, with a Taylor expansion of $ (1-\omega x)^{-1}$ at $0$ we have that  
	\[	I(x) =  \text{P.V.} \int_{-1}^1 \displaystyle \frac{-\sum_{k=1}^\infty (x\omega)^k } {|\omega|^{2s+1}} (1-\omega^2)^s \, d\omega.\]
The odd part of the sum vanishes by symmetry, and so
	\begin{equation*}
		\begin{split}
		I(x)=\;&  -2 \lim_{\eee\to 0} 	\int_{\eee}^1 \displaystyle \frac{\sum_{k=1}^{\infty} (x\omega)^{2k} } {\omega^{2s+1}} (1-\omega^2)^s \, d\omega\\
			=\;&- 2 \lim_{\eee\to 0} \displaystyle  \sum_{k=1}^{\infty} x^{2k}  \int_{\eee}^1 \omega^{2k-2s-1} (1-\omega^2)^s\, d\omega. \\
				\end{split}	
	\end{equation*}
We change the variable $t =\omega^2$ and integrate by parts to obtain
	\begin{equation*}
		\begin{split}
	  I(x)=	\;&  - \lim_{\eee\to 0}   \sum_{k=1}^{\infty} x^{2k}   \int_{\eee^2 }^1 t^{k-s-1} (1-t)^s\, dt,\\
	     = \;&  \sum_{k=1}^{\infty} x^{2k} \lim_{\eee\to 0}   \bigg[ \frac{ \eee^{2k-2s}(1-\eee^2)^s}{k-s} - \frac{s}{k-s}  \int_{\eee^2}^1 t^{k-s}(1-t)^{s-1} \, dt \bigg].
		\end{split}	
	\end{equation*}
For $k\geq 1$, the limit for $\eee$ that goes to zero is null, and using the integral representation of the Beta function, we have that
	\begin{equation*}	
			I(x) = \sum_{k=1}^{\infty} x^{2k} \frac{-s}{k-s} \beta(k+1-s,s) .
		\end{equation*}
We use the Pochhammer symbol defined as
		\begin{equation}\label{Pochh}
		(q)_k = \begin{cases}   
			1   &\mbox{ for } k = 0 ,\\
  		q(q+1) \cdots (q+k-1) &\mbox{ for } k > 0
		 \end{cases}
		\end{equation} and with some manipulations, we get
	\[	\begin{split}
			&\; \frac{-s}{k-s}  \beta(k+1-s,s) =   \frac{(-s) \Gamma(k+1-s) \Gamma(s)}{(k-s) \Gamma(k+1) }\\
			 = &\; \frac{(-s) \Gamma(k-s)\Gamma(s)}{k!}
			= \beta(1-s,s) \frac{(-s)_k}{k!}. 
		\end{split}\]	
And so
\[ I(x) =  \beta(1-s,s)  \sum_{k=1}^{\infty} x^{2k} \frac{(-s)_k}{k!}.\] By the definition of the hypergeometric function (see Appendix \ref{special}) we obtain 
 \[ \begin{split}
 I(x)  =&\; -\beta(1-s,s) + \beta(1-s,s) \sum_{k=0}^{\infty}  { (-s)_k} \frac{x^{2k}}{k!} \\
  = &\; \beta(1-s,s) \bigg(  F\Big(-s,\frac{1}{2},\frac{1}{2}, x^2\Big)- 1\bigg).\end{split}\]
Now, by %15.1.8 in \cite{ABRAMOWITZ} 
\eqref{hypelc1} in the Appendix, we have that
\[ F\Big(-s,\frac{1}{2},\frac{1}{2}, x^2\Big) = (1-x^2)^s\]
and therefore
	\begin{equation*}  I(x)  =  \beta(1-s,s) \Big( (1-x^2)^s-1\Big).\end{equation*}
Inserting this and \eqref{jxah} into \eqref{llpux} we obtain
	\begin{equation} \label{onedimfin} \Lm u(x) = \beta(1-s,s) -(1-x^2)^{-s}   \frac{(1+x)^{2s} +(1-x)^{2s}}{2s}.\end{equation}
We write the fractional Laplacian of $u$ as
	\[\begin{split}
	\frac{\frlap u(x) }{C(1,s)} =&\; \Lm u(x) + \int_{-\infty}^{-1}\frac {u(x)}{|x-y|^{1+2s}} \, dy + \int_1^{\infty}\frac {u(x)}{|x-y|^{1+2s}} \, dy \\
							 =&\;\Lm u(x)+ (1-x^2)^s \bigg(  \int_{-\infty}^{-1} |x-y|^{-1-2s} \, dy +\int_1^{\infty}|x-y|^{-1-2s}  \, dy \bigg) \\
							 =&\; \Lm u(x) +  (1-x^2)^s \frac {(1+x)^{-2s} +(1-x)^{-2s}}{2s}.
	\end{split}\]		 	
Inserting \eqref{onedimfin} into the computation, we obtain		 	
\begin{equation} \label{infrl122}	\frlap u(x) =C(1,s) \beta(1-s,s).\end{equation}
To pass to the $n$-dimensional case, without loss of generality and up to rotations, we consider $x=(0,0,\dots, x_n)$ with $x_n\geq 0$. We change into polar coordinates $x-y=th$, with $h\in \partial B_1$ and $t\geq0$. We have that
	\begin{equation}\label{lapcaln1} \begin{split}
			\frac{\frlap \mathcal U (x)}{C(n,s)} =&\; P.V. \int_{\Rn} \frac{(1-|x|^2)^s - (1-|y|^2)^s}{|x-y|^{n+2s}} \, dy\\
					=&\;\frac{1}{2}\int_{\partial B_1} \bigg( P.V. \int_{\R}  \frac{(1-|x|^2)^s - (1-|x+ht|^2)^s}{|t|^{1+2s}} \,dt\bigg) \, d\mathcal{H}^{n-1} (h).	
	\end{split}
		\end{equation}
Changing the variable $t= -|x|h_n +\tau \sqrt{|h_nx|^2-|x|^2+1}$, we notice that \[1-|x+ht|^2 = (1-\tau^2)(1-|x|^2 +|h_nx|^2)\] and so 
	 \[\begin{split}
	 P.V. \int_{\R}  &\frac{(1-|x|^2)^s - (1-|x+ht|^2)^s}{|t|^{1+2s}}  \, dt\\
	  = &P.V. \int_{\R} \frac{(1-|x|^2)^s - (1-\tau^2)^s(|h_nx|^2-|x|^2+1)^s } {\Big| -|x|h_n +\tau \sqrt{|h_nx|^2-|x|^2+1}\Big|^{1+2s}}\sqrt{|h_nx|^2-|x|^2+1} \, d\tau \\
	 	=& P.V. \int_{\R}  \frac{ \bigg(1-\displaystyle \frac{|x|^2h_n^2}{|h_nx|^2-|x|^2+1} \bigg)^s- (1-\tau^2)^s}{\bigg|\tau -\displaystyle\frac{|x|h_n}{ \sqrt{|h_nx|^2-|x|^2+1}}\bigg| ^{1+2s}}\, d\tau \\
	 		=& \frac{\frlap u\bigg( \displaystyle\frac{|x|h_n}{ \sqrt{|h_nx|^2-|x|^2+1}}\bigg)}{C(1,s)}\\
	 		= &  \beta(1-s,s),
	 \end{split}		\]
where the last equality follows from identity \eqref{infrl122}. Hence from \eqref{lapcaln1} we have that \[  \frlap \mathcal U(x) = C(n,s) \beta(1-s,s)\frac{\omega_n}{2}.\] This concludes the proof of the result.
\subsection[All functions are locally $s$-harmonic up to a small error]{All functions are locally $s$-harmonic up to a small error } \label{afsh}

Here give a sketch of the proof that $s$-harmonic functions can locally approximate any given function, without any geometric constraints (the reader can see the paper \cite{DSV14} for further details and a complete proof).
This fact is rather surprising and it
is a purely nonlocal feature, in the sense that it has no classical
counterpart.
Indeed, in the classical setting, harmonic functions are quite rigid, for instance they cannot have a strict local maximum, and therefore cannot approximate a function with a strict local maximum. The nonlocal picture is, conversely, completely different, as the oscillation of a function ``from far'' can make the function locally harmonic, almost independently from its local behavior.

We want to give here some
hints on the proof of this approximation result:

\begin{theorem}\label{ALL FUNCTIONS}
Let $k\in \N$ be fixed. Then for any $f\in C^k(\overline{B_1})$ and any $\eee>0$ there exists $R>0$ and $u\in H^s(\Rn)\cap C^s(\Rn)$ such that
	\begin{equation}\label{cc1} \begin{cases} \frlap u(x)= 0 \quad &\mbox{ in }  B_1\\
					u=0 \quad &\mbox{ in }  \Rn \setminus  B_R	
	\end{cases}\end{equation}
and \[ \|f-u\|_{C^k(\overline{B_1})} \leq \eee.\]					
\end{theorem}

\begin{proof}[Sketch of the proof]
For the sake of convenience, we divide
the proof into three steps. Also, for simplicity, we give the sketch of the proof in the one-dimensional case. See \cite{DSV14} for the entire and more general proof.\\
\textbf{Step 1. Reducing to monomials}\\
Let $k\in \N$ be fixed. We use first of all the Stone-Weierstrass Theorem and we have that for any $\eee>0$ and any $f \in C^k \big ([0,1]\big)$ there exists a polynomial $P$ such that
\[ \|f-P\|_{C^k{(\overline{B_1})}} \leq \eee.\]
Hence it is enough to prove Theorem~\ref{ALL FUNCTIONS}
for polynomials. Then, by linearity, it is enough to prove it
for monomials. Indeed, if~$P(x)= \displaystyle \sum_{m=0}^N c_m x^{m} $
and one finds
an $s$-harmonic function  $u_m$ such that  
	\[  \|u_m - x^{m} \|_{C^k{(\overline{B_1})}} \leq \frac{\eee}{|c_m|(N+1)}, \]
then by taking $u:=\displaystyle\sum_{m=0}^N c_m u_m $ we have that 
	\[ \| u-P\|_{C^k{(\overline{B_1})}} 
\leq \sum_{m=0}^N |c_m| \|u_m -x^{m} \|_{C^k{(\overline{B_1})}} \leq {\eee}.\] 
Notice that the function $u$ is still $s$-harmonic, since the fractional Laplacian is a linear operator. 
  \bigskip

\noindent \textbf{Step 2. Spanning the derivatives}\\ We prove
the existence of an $s$-harmonic function in $B_1$, vanishing outside a compact set and with arbitrarily large number of derivatives prescribed.
That is, we show that for any~$m\in\N$
there exist $R > r > 0$, a point~$x\in\R$ and a function~$u$
such that
\eqlab{ \label{SPA}
& (-\Delta)^s u=0 &{\mbox{ in }} &(x-r,x+r),\\
& u=0 &{\mbox{ in }} &\R \setminus (x-R,x+R),}
and 
\eqlab{ \label{spnd} & D^j u(x)=0 {\mbox{ for any }} j\in\{0,\dots,m-1\},\\
\;&
D^m u(x)=1.} 
To prove this,
we argue by contradiction.

We consider $\mathcal Z$ to be the set of all pairs $(u,x)$ of $s$-harmonic functions in a neighborhood of~$x$,
and points $x\in \R$ satisfying \eqref{SPA}.
 To any pair, we associate the vector \[\big(u(x), Du(x), \dots, D^m u(x) \big) \in \R^{m+1}\] and take $V$ to be the vector space 
spanned by this construction, i.e.
$$ V:= \Big\{
\big(u(x), Du(x), \dots, D^m u(x) \big),\;{\mbox{ for }}(u,x)\in {\mathcal Z}
\Big\}.$$
Notice indeed that
\begin{equation}\label{LIy}
{\mbox{$V$ is a linear space.}}\end{equation}
Indeed, let~$V_1$, $V_2\in V$ and~$a_1$, $a_2\in\R$.
Then, for any~$i\in\{1,2\}$,
we have that\[ V_i = \big(u_i(x_i), Du_i(x_i), \dots, D^m u_i(x_i) \big) \quad \mbox{ for some } \; (u_i,x_i)\in {\mathcal Z}, \]
 i.e.~$u_i$ is $s$-harmonic
in~$(x_i-r_i,x_i+r_i)$ and vanishes outside~$(x_i-R_i,x_i+R_i)$, for some~$R_i\ge r_i>0$.
We set
$$ u_3(x):= a_1 u_1(x+x_1)+a_2u_2(x+x_2).$$
By construction, $u_3$ is $s$-harmonic in~$(-r_3,r_3)$,
and it vanishes outside~$(-R_3,R_3)$,
with~$r_3:=\min\{r_1,r_2\}$
and~$R_3:=\max\{R_1,R_2\}$, therefore~$(u_3,0)\in {\mathcal Z}$.
Moreover
$$ D^j u_3(x)= a_1 D^j u_1(x+x_1)+a_2D^ju_2(x+x_2)$$
and thus
\begin{eqnarray*}&&
a_1 V_1+a_2 V_2 \\&=& a_1\big(u_1(x_1), Du_1(x_1), \dots, D^m u_1(x_1) \big)
+a_2 \big(u_2(x_2), Du_2(x_2), \dots, D^m u_2(x_2) \big)
\\ &=& \big(u_3(0), Du_3(0), \dots, D^m u_3(0) \big).
\end{eqnarray*}
This establishes~\eqref{LIy}.

Now, to complete the proof of Step 2, it is enough to show that
\begin{equation}\label{PaV}
V =\R^{m+1}.\end{equation}
Indeed, if~\eqref{PaV}
holds true, then in particular~$(0,\dots, 0,1)\in V$,
which is the desired claim in Step~2.

To prove~\eqref{PaV}, we argue by contradiction: if
not, by~\eqref{LIy}, we have that~$V$ is a proper subspace
of~$\R^{m+1}$ and so it lies in a hyperplane.

Hence there exists a vector $c=(c_0, \dots, c_{m}) 
\in\R^{m+1}\setminus \{0\}
$ such that
$$ V\subseteq \left\{ \zeta\in\R^{m+1} \; \big| \; c\cdot\zeta=0\right\}.$$
That is, taking a pair $(u,x)\in \mathcal Z$, we have that 
\[ \sum_{j \leq m } c_j D^{j} u(x) =0,\]
i.e. the vector~$c=(c_0, \dots, c_{m})$ is
orthogonal to any vector $\big(u(x), Du(x), \dots, D^m u(x) \big)$.
To find a contradiction, we now choose an appropriate
$s$-harmonic function~$u$
and we evaluate it at an appropriate point~$x$.
As a matter of fact,
a good candidate for the $s$-harmonic function is $x^s_+$,
as we know from Theorem~\ref{G1}:
strictly speaking, this function is not allowed
here, since it is not compactly supported,
but let us say that one can construct a compactly
supported $s$-harmonic function with the same 
behavior near the origin. With this slight caveat set aside,
we compute for a (possibly small)~$x$ in~$(0,1)$:
\[ D^{j} x^s=
s(s-1)\dots (s-j+1) x^{s-j} \] and multiplying the sum with $x^{m-s}$ (for $x\neq 0$) we have that 
	\[ \displaystyle \sum_{j\leq m} c_j  s(s-1)\dots (s-j+1) x^{m-j}  =0.\] But since $s\in (0,1)$ the product $ s(s-1)\dots (s-j+1)$ never vanishes. Hence the polynomial is identically null if and only if $c_j=0$ for any $j$, and we reach a contradiction.
This completes the proof of the existence of a function~$u$
that satisfies~\eqref{SPA} and~\eqref{spnd}.

\noindent \textbf{Step 3. Rescaling argument and completion of the proof}\\
By Step 2,
for any $m\in\N$ 
we are able to construct
a locally $s$-harmonic
function $u$ such that $u(x)=x^m +\mathcal{O}(x^{m+1})$ near the origin
(up to a translation).
By considering the blow-up 
\[ u_\lambda (x) = \frac{u(\lambda x) }{\lambda^m} = x^m +\lambda \mathcal{O}(x^{m+1})\]
 we have that for $\lambda$ small, $u_\lambda$ is arbitrarily close to the 
monomial $x^m$. As stated in Step 1, this concludes the proof
of Theorem~\ref{ALL FUNCTIONS}.
\end{proof}

It is worth pointing out that the flexibility
of $s$-harmonic functions given by Theorem~\ref{ALL FUNCTIONS}
may have concrete consequences.
For instance, as a byproduct of Theorem~\ref{ALL FUNCTIONS},
one has that a biological population with nonlocal dispersive attitudes
can better locally adapt to a given distribution of
resources (see e.g. Theorem~1.2 in~\cite{MASSACCESI}).
Namely, nonlocal biological species may efficiently use
distant resources and they can fit to the resources available nearby
by consuming them (almost) completely,
thus making more difficult for a different
competing species to come into place.

\section{Density of Caputo stationary functions in the space of smooth functions}\label{capdens}

The Caputo fractional derivative is a so-called nonlocal operator, that models long-range interactions. For instance, if we think of a function depending on time, the Caputo fractional derivative would represent a memory effect, pointing out that the state of a system at a given time depends on past events.  In other words, the Caputo derivative describes a causal system (also known as a non-anticipative system).   

This nonlocal character of the Caputo derivative gives rise to a peculiar behavior: on a bounded interval, say $[0,1]$, one can find a Caputo-stationary function ``close enough'' to any smooth function, without any geometrical constraints. This is a surprising result when one thinks of the rigidity of the classical derivatives. For instance, the functions with null first derivative are constant functions,  the functions with null second derivatives are affine functions. Such functions cannot approximate locally any given $C^k$ function, for any fixed $k\in \N_0$. We remark that this property of Caputo-stationary functions is in analogy to $s$-harmonic functions, as proved in Subsection \ref{afsh}.

Let $a \in \R$ and $s \in (0,1)$ be two arbitrary parameters. We define the functional space 
\eqlab{	 \label{ca1s} C_a^{1,s}  := \Big\{ f   \colon \R \to \R \: \big| \;\; \mbox{for any } x>a,  \;  f \in AC\big([a,x]\big) 
	\mbox{ and } \displaystyle & f'(\cdot){(x-\cdot)^{-s}} \in L^1\big( (a, x)\big)  \Big\} .}
We denote here by $AC(I)$ the space of absolutely continuous functions on $I$ and define the Caputo derivative.
\begin{defn}
The Caputo derivative of $u\in C_a^{1,s}$ with initial point $a\in \R$ at the point $x>a$ is given by
	\begin{equation} \label{caputo}
		D^s_a u(x):= \displaystyle \frac{1}{\Gamma(1-s)}\int_a^x u'(t)(x-t)^{-s}\, dt .
		\end{equation}
\end{defn}
 With this definition, we have that:
\begin{defn}
We say that $u\in C_a^{1,s}$  is Caputo-stationary with initial point $a\in \R$ at the point $x>a$   if
	\bgs{  \label{caph}
	 	D^s_a u(x)=0.} 
Let $I$ be an interval such that $a\leq \inf I$. We say that $u$ is Caputo-stationary with initial point $a$ in $I$ if $D_a^su(x)=0$ holds for any $x \in I$.
\end{defn}
For $k\in \N_{0}$, we consider $C^k\big([0,1]\big)$ to be the space of the $k$-times continuous differentiable functions on $[0,1]$, endowed with the $C^k$-norm 
\[ \|f\|_{C^k\lr{[0,1]}} =\sum_{i=0}^k \sup_{x\in [0,1]}|f^{(i)}(x)|.\] The main result that we prove here is that for any fixed $k \in \N_0$, given any $C^k\big([0,1]\big)$ function, there exists an initial point $a<0$ and a Caputo-stationary function with initial point $a$, that in $[0,1]$ is arbitrarily close (in the $C^k$ norm) to the given function. More precisely:

\begin{theorem}\label{thm:thm1}
Let $k\in \N_0$ and $s\in (0,1)$ be two arbitrary parameters. Then for any $f \in C^k\big([0,1]\big)$ and any $\varepsilon>0$ there exists an initial point $a<0$ and a function $u\in C^{1,s}_a $ such that 
	\[ D_a^s u(x)=0 \text{ in } [0,\infty) \]and 
	\[\| u-f\|_{C^k\lr{[0,1]}} < \varepsilon.\]	
\end{theorem}

\bigskip
In the next lines we recall some notions and make some preliminary remarks on the Caputo derivative.
\smallskip

The reader can see Chapter 7.5 in \cite{zygmund} for the definition of absolutely continuous functions. In particular, we use the following characterization, given in Theorem 7.29 in \cite{zygmund}, that we recall in the next Theorem.
\begin{theorem}\label{acrep} A function $f$ is absolutely continuous in $ [a,b]$ if and only if $f'$ exists almost everywhere in $[a,b]$, $f'$ is integrable on $[a,b]$ and
\bgs{ f(x)-f(a)=\int_a^x f'(t)\, dt, \quad a\leq x\leq b.}
\end{theorem} 
% Moreover, from Theorem 7.28 in \cite{zygmund} we have that if $f'$ vanishes almost everywhere in $[a,b]$, then $f $ is a constant function.}
  
By convention, when we take the Caputo derivative $D_a^s$ of a function, we assume that the function is ``causal'', i.e. that it is constant on $(-\infty,a)$. In particular, we take $u(x)=u(a)$ for any $x<a$ and this, by the definition \eqref{caputo}, implies that $D_a^s u(x) =0$ for $x<a$. 

%Lastly, we recall the Beta function (see Chapter 6.2 in the book \cite{ABRAMOWITZ} for other details) defined for $x,y >0$ as
%	\eqlab {\beta(x,y) :=\int_0^1 t^{x-1} (1-t)^{y-1} \, dt \label{b01}.} 
%	We also have that 
%	\[ \beta(x,y)=\frac{\Gamma(x)\Gamma(y)}{\Gamma(x+y)}.\]  In particular, the next explicit result holds
%	\eqlab { \label{b02} \beta(s,1-s)=\Gamma(s)\Gamma(1-s)=\frac{\pi}{\sin\pi s}.} 

Moreover, we notice that if, for instance $u\in C^2(\R)$, then
\[ \lim_{s\to 0^+} D_a^s u(x)=u(x)-u(a), \qquad \lim_{s\to 1^{-}}D_a^s u(x)=u'(x).\] 
Indeed, 
\[ |u'(t) (x-t)^{-s} |\leq  |u'(t)|\chi_{[a,x-1]}(t) + |u'(t)|(x-t)^{-1/2}\chi_{[x-1,x]}(t)\in L^1\left([a,x]\right),\] and using the Dominated Convergence Theorem  we get that
\[  \lim_{s\to 0^+} D_a^s u(x)= \lim_{s\to 0^+} \frac{1}{\Gamma(1-s)}\int_a^x u'(t)(x-t)^{-s}\, dt = u(x)-u(a).\]
Also, integrating by parts and using \eqref{gamxx1} we get that
\[ D_a^s u(x) = \frac{u'(a)(x-a)^{1-s}}{\Gamma(2-s)} + \frac{1}{\Gamma(2-s)} \int_a^x u''(t)(x-t)^{1-s}\, dt.\] 
Since
\[ |u''(t) (x-t)^{1-s} |\leq  |u''(t)|\chi_{[a,x-1]}(t)(x-t) + |u''(t)|\chi_{[x-1,x]}(t)\in L^1\left([a,x]\right),\] using the Dominated Convergence Theorem  we obtain that
\[ \lim_{s \to 1^-} D_a^su(x)=u'(x).\]

The proof of Theorem \ref{thm:thm1} follows the steps of the sketch of the proof introduced in Subsection \ref{afsh} for the fractional Laplacian. Here, we give a complete proof of the statement, taking into account the structure of the Caputo derivative. As a matter of fact, the main idea of the proof is (as for the fractional Laplacian) that one can build a Caputo-stationary function in say $I = [0,1]$ by choosing a ``good'' given function as exterior datum. But while the fractional Laplacian takes into account the entire space and the exterior datum is $ \C I$,  the Caputo derivative considers only the left-side complement and this reflects in the lack of symmetry of these exterior conditions. Namely, the exterior datum is $(-\infty, 0]$, adding the convention that events start at a given point, say $t_0<0$ and $f$ is constant before time $t_0$. This structure has to be accounted for when proving Theorem \ref{thm:thm1}. 
 
We reduce the proof of Theorem \ref{thm:thm1} to finding a Caputo-stationary function close to any monomial. 
%\begin{theorem}\label{thm:SW}
% For any $f \in C^k\big([0,1]\big)$ and any positive $\varepsilon$ there exists a polynomial $P$ such that \[ \|f-P\|_{C^k\lr{[0,1]}} < \varepsilon.\]
%\end{theorem}
 For this, we follow Step 1 of the sketch of the proof of Theorem \ref{ALL FUNCTIONS}, using the Stone Weierstrass Theorem and the linearity of the Caputo derivative.
% claim that it suffices to prove that for any monomial 
%	\[ q_m(x)= x^m \mbox{, } m\in \N\] 
%and for any $\varepsilon_m >0$ there exists a function $u_m$ that is Caputo-stationary in $[0,1]$, such that  
%	\begin{equation}\label{mapp1} \|u_m-q_m\|_{C^k\lr{[0,1]}}  < \varepsilon_m.\end{equation}
%Indeed, consider an arbitrary $n \in \N$ and the polynomial $\displaystyle P(x)= \sum_{m=0}^n c_m q_m(x)$. Then the function 
%$\displaystyle u(x):=\sum_{m=0}^n c_m u_m(x)$ would satisfy 
%	\[ \|u  - P \|_{C^k\lr{[0,1]}}  \leq \sum_{m=0}^n |c_m|\, \|u_m- q_m\|_{C^k\lr{[0,1]}}   < \sum_{m=0}^n |c_m| \varepsilon_m =\varepsilon,\]
%where one considers for any $m$ the small quantity  $\varepsilon_m=\displaystyle \frac{\varepsilon}{|c_m|(n+1)}$.
%Also, the function $u$ is Caputo-stationary, since the Caputo derivative is linear. Hence, the function $u$ is Caputo-stationary and is ``close'' to any polynomial.  By the Stone-Weierstrass Theorem, we have that for $k\in \N_{0}$ a fixed arbitrary number, any  $f \in C^k\big([0,1]\big)$ and any positive $\varepsilon$ there exists a polynomial $P$ such that \[ \|f-P\|_{C^k\lr{[0,1]}} < \varepsilon.\]Therefore
%\[ \begin{split}\|u -f \|_{C^k\lr{[0,1]}}    \leq \al \|u-P\|_{C^k\lr{[0,1]}}    +\|f-P\|_{C^k\lr{[0,1]}}   <  2\varepsilon.
%	\end{split}\] 
% This proves the claim.
In the rest of the section,
%paper, we prove that we can find a Caputo-stationary function close to any given monomial. To do this, 
we proceed as follows: 
%\begin{itemize}
%\item 
In Subsection \ref{sectrfcp}, we obtain a representation formula for $u$, when $D_a^s u(x)= 0 $ in $(b,\infty)$ for a given $b>a$ and having prescribed $u$ on $(-\infty,b]$.
% To do this, we prove that solving $D_a^s u(x)= 0 $ is equivalent to solving a particular integro-differential equation. We then obtain a representation formula for the integro-differential equation, hence for our initial equation. 
%\item  
In Subsection \ref{sectaps}, we prove that there exists a sequence $(v_j)_{j\in \N}$ of Caputo-stationary functions in $(0,\infty)$ such that, uniformly on bounded  subintervals of $(0,\infty)$, we have that $\lim_{j\to \infty} v_j(x) = \kappa x^s$, for a suitable constant $\kappa>0$.
%\item
Then, in Subsection \ref{sectma} we prove that there exists a Caputo-stationary function with an arbitrarily large number of derivatives prescribed. 
%We do this by taking advantage of the particular structure of the function $x^s$. If we take any derivative of such a function, say $(x^s)^{(i)}= s(s-1)\dots(s-i+1) x^{s-i},$ for $x> 0$ this derivative never vanishes.
%\item
and the last Subsection \ref{sectthm1} deals with the proof of Theorem \ref{thm:thm1}. 
%Prescribing the derivatives of $u$ such that, for $m\in \N$, they vanish at $0$ until the order $m-1$, and are equal to $1$ at order $m$, using a Taylor expansion and performing a blow-up argument, we can conclude the proof of the main theorem.
%\end{itemize}

\subsection{A representation formula}\label{sectrfcp}
We deduce here a Poisson-like representation formula for a function $u$  that is Caputo-stationary with initial point $a$ in the interval $(b,\infty)$ for $b>a$, and fixed outside. To do this, we prove that if $u\in  C_a^{1,s}$, the two following problems are equivalent

\[ \begin{matrix}  
	\displaystyle D_a^s u(x) =  0 &\text{ in } & (b,\infty),  & & & &    &\displaystyle \int_b^x u'(t)(x-t)^{-s}\, dt = g(x)  &\text{ in }     & (b,\infty), \\ 
	\mbox{  prescribed data }      &\text{ in } & (-\infty,b], & & & &    &\mbox{  prescribed data }  					 &\text{ in }   & (-\infty,b]. 
\end{matrix}\]
%	\bgs{ &D_a^s u(x) =  0 &\text{ in } & (b,\infty),\\ 
%	&\mbox{  prescribed data }  &\text{ in } & (-\infty,b].  }
%
%	\bgs{&\int_b^x u'(t)(x-t)^{-s}\, dt = g(x) &\text{ in } & (b,\infty),\\
%	&\mbox{  prescribed data }  &\text{ in } & (-\infty,b], }
%for a given function $g$ (that depends on the prescribed data of the initial problem). Then, we 
%introduce in Theorem \ref{thm:probc} a representation formula for this integro-differential equation. With these two results in hand, we obtain a representation for the solution of the initial problem. 
\noindent Moreover, we present here an interior regularity result.
	\begin{center}
\begin{figure}[htpb]
	\hspace{0.6cm}
	\begin{minipage}[b]{0.85\linewidth}
	\centering
	\includegraphics[width=0.90\textwidth]{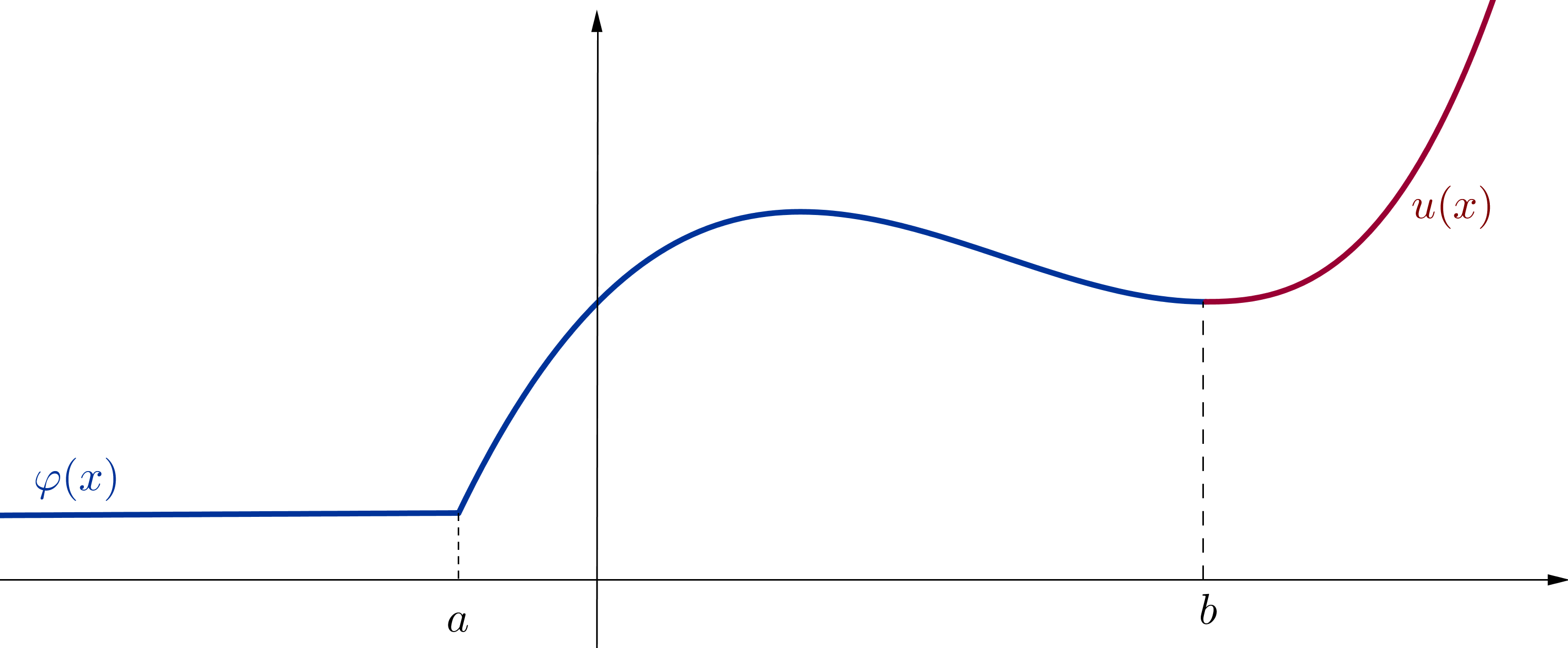}
	\caption{A Caputo-stationary function in $(b,\infty)$ prescribed on $(-\infty,b]$}   
	\label{fign:Lem31}
	\end{minipage}
\end{figure} 
\end{center} 

In this subsection, we fix the arbitrary parameters $a,b \in \R$ with $b>a$ and $s\in(0,1)$. 
\begin{lemma}\label{lem:int11}
Let $\varphi \in C\big((-\infty,b] \big)\cap C^1\big([a,b]\big)$ such that $\varphi(x) = \varphi(a)$ in $(-\infty,a]$. Then $u\in  C_a^{1,s}$  satisfies the equation
	\bgs{ \label{intr1}
			D_a^s u(x)& =0 &\text{ in } & (b,\infty),\\
			u(x)&=\varphi(x) &\text{ in }  & (-\infty,b]}
if and only if it satisfies 
	\bgs{ \label{intr2}
		 \int_b^x u'(t)(x-t)^{-s} \, dt   &=- \int_{a}^b \varphi'(t)(x-t)^{-s}\, dt & \text{ in }& (b,\infty),\\
		u(x)&=\varphi(x)&\text{ in } & (-\infty,b]. }
\end{lemma}
The reader can see a qualitative graphic of a function described by Lemma \ref{lem:int11} in Figure \ref{fign:Lem31}. An explicit example of such a function is build in Example \ref{exam}, in Figure \ref{fign:es1}.
\begin{proof} Since $\varphi \in C^1\big([a,b]\big)$ we have for any $x\geq b$  
	\bgs{\bigg |\int_a^{b} \varphi'(t)(x-t)^{-s} \, dt \bigg |\leq  \sup_{t\in [a,b]} |\varphi'(t)| \frac{ (x-a)^{1-s}-(x-b)^{1-s}}{1-s}<\infty.}
Hence the map $x\mapsto\displaystyle  \int_a^{b} \varphi'(t)(x-t)^{-s} \, dt $ is well defined in $[b,\infty)$. Using the definition \eqref{caputo} for $x >b$ we have that
	\begin{equation*}
		\begin{split}
		\Gamma(1-s) D_a^s u(x)
		%=\al    \int_a^x u'(t)(x-t)^{-s} \, dt \\ 
		= \al  \int_b^x u'(t)(x-t)^{-s} \, dt + \int_a^{b}  u'(t)(x-t)^{-s} \, dt \\
		=\al  \int_b^x u'(t)(x-t)^{-s} \, dt + \int_a^{b} \varphi'(t)(x-t)^{-s} \, dt  .
		\end{split}
	\end{equation*}
It follows that $D_a^s u(x)=0$ on $(b,\infty)$ is equivalent to 
		\bgs{ 
		 \int_b^x u'(t)(x-t)^{-s} \, dt   =- \int_a^{b} \varphi'(t)(x-t)^{-s} \, dt  \quad \text{ in } (b,\infty).	}
 This concludes the proof of the Lemma. 
\end{proof}

In the following Theorem we introduce a representation formula for an integro-differential equation.
\begin{theorem}
\label{thm:probc}
Let $g \in C_b^{1,1-s}$. The problem 
\eqlab{ \label{probc1}
			\int_b^x u'(t)(x-t)^{-s} \,dt  & = g(x) \quad \mbox{ in } (b,\infty),\\
			 u(b)  &=  0 } admits on $[b,\infty)$ a unique solution $u\in C_b^{1,s}$. Moreover, for any $x>b$,
			\begin{equation} \label{solc1} 
	u(x)= \ig \int_b^x g(t)(x-t)^{s-1} \, dt .
	\end{equation}
\end{theorem}

\begin{proof}
We prove this theorem by showing that $u$ given in \eqref{solc1} is well defined, belongs to the space $C_b^{1,s}$ and is the unique solution of the problem \eqref{probc1}.
 
Since $g$ belongs to $C_b^{1,1-s}$ (recall \eqref{ca1s}), for any $x> b$ we have that
	\bgs{|u(x)| \leq\ig  \int_b^x |g(t)| (x-t)^{s-1} \, dt  \leq c_s\sup_{ t \in [b,x]}	|g(t)|  (x-b)^{s} <\infty,} where $c_s$ is a positive constant. Hence the definition \eqref{solc1} is well posed.  
%It also follows that 	\eqlab{\label{ucontb}  \lim_{x \to b^+} u(x)=0,} therefore $u$ is continuous in $b$.

\bigskip 
We prove that $u$ belongs to $ C_b^{1,s}$.
%and we claim that
%\eqlab{ \label{cbsu} &\mbox{for } g\in C_b^{1,1-s} \mbox{  we have that a.e. in  } [b,\infty)\\ & \frac{\pi} {\sin \pi s} \,u'(y)=\lr{ \int_b^y g'(\tau)(y-\tau)^{s-1}\, d\tau + g(b)(y-b)^{s-1}}. }
We claim that 
 \eqlab{ \label{cbsu1}  g\in C_b^{1,1-s} & \mbox{ and } u \mbox{ as in } \eqref{solc1} \Rightarrow  \\
 & u \in AC\big([b,\infty)\big) \mbox{ and} \\
  & u'(y)=\frac{\sin \pi s} {\pi}  \lr{ \int_b^y g'(\tau)(y-\tau)^{s-1}\, d\tau + g(b)(y-b)^{s-1}} \quad  \mbox{   a.e. in  } [b,\infty). }
\noindent We fix an arbitrary $x>b$. According to definition \eqref{ca1s},  we have $g \in AC\big( [b,x]\big) $ and thanks to Theorem \ref{acrep} it follows that
%$g'$ exists almost everywhere in $[b,x]$, the derivative $g' $ is in $L^1\big((b,x)\big)$ and 
for any $t\in [b,x]$
	\[ g(t)=\int_b^t g'(\tau)\, d\tau + g(b).\]
And so in \eqref{solc1} we have that
	\eqlab{ \label{bla2}   \frac{\pi}{\sin \pi s} \, u(x)
	% =\al  \int_b^x g(t) (x-t)^{s-1} dt \\ 
		%= \al \int_b^x \lr{ \int_b^t g'(\tau)\, d\tau +g(b)} (x-t)^{s-1}  dt\\
		= \al \int_b^x \lr{ \int_b^t g'(\tau)\, d\tau } (x-t)^{s-1}\, dt + g(b) \int_b^x (x-t)^{s-1}\,dt .}
We compute
	\eqlab {\label{bla1}  \int_b^x (x-t)^{s-1} \, dt = \frac{(x-b)^s}{s} = \int_b^x (y-b)^{s-1} \, dy.}
Tonelli theorem applied to the positive measurable function $|g'(\tau)|(x-t)^{s-1}$ on the domain 
\eqlab{ \label{rev2}D_{b,x}:=\big\{ (t,\tau) \: \big| \;\;  b\leq t\leq x, \;b\leq \tau\leq t\big\}} 
with the product measure $d(t,\tau)$ gives 
	\eqlab{\label{rev1}  
		\al \iint_{D_{b,x}} |g'(\tau)|\,   (x-t)^{s-1} \, d (t, \tau) =    \int_b^x |g'(\tau)| \lr{  \int_\tau^x (x-t)^{s-1} \, dt }\,d\tau  \\ 
	 	=\al \frac{1}{s}\int_b^x |g'(\tau)| (x-\tau)^s \, d\tau \leq  \frac{(x-b)^s}{s}  \|g'\|_{L^1\lr{(b,x)}}, }
which is a finite quantity. Hence $|g'(\tau)| (x-\tau)^{s-1} \in L^1\big(D_{b,x}, d(t,\tau)\big)$ and by Fubini theorem \textcolor{black}{and using \eqref{bla1}} it follows that
	\bgs{ \al \int_b^x \lr{\int_b^t g'(\tau)\, d\tau } (x-t)^{s-1}\, dt  = \int_b^x g'(\tau) \lr{ \int_{\tau}^x  (x-t)^{s-1} \, dt} \, d\tau\\
			= \al \int_b^x  g'(\tau) \lr{\int_\tau^x (y-\tau)^{s-1} \, dy} \, d\tau = \int_b^x \lr{ \int_b^y g'(\tau) (y-\tau)^{s-1} \, d\tau }\, dy.}
Inserting this and identity \eqref{bla1} into \eqref{bla2}, we obtain that
	\[  \frac{\pi}{\sin \pi s}\,  u(x) = \int_b^x \lr{  \int_b^y g'(\tau) (y-\tau)^{s-1} \, d\tau + g(b) (y-b)^{s-1} } \, dy.\] 
Hence $u$ is the integral function of a $L^1\big((b,x)\big)$ function (thanks to \eqref{rev1}) and recalling that $u(b)=0$, according to Theorem \ref{acrep} we have that $u\in AC\big([b,x]\big)$.
%$G\colon (b,\infty)\to \R$ be defined as 
%\[G(y) := \int_b^y g'(\tau) (y-\tau)^{s-1} \, d\tau + g(b) (y-b)^{s-1},\] 
%we have that $G \in L^1 \big((b,x)\big)$ } and 
% since
%	\bgs{ \int_b^x |G(y)| dy \leq \al \int_b^x \lr{ \int_b^y |g'(\tau)| (y-\tau)^{s-1}\, d\tau + |g(b) |(y-b)^{s-1}} \, dy\\
%			= \al \int_b^x |g'(\tau)| \frac{(x-\tau)^{s}}{s} \,d\tau  + |g(b)| \frac{ (x-b)^s}{s}\\
%			\leq  \al \frac{(x-b)^s}{s} \lr{ \|g'\|_{L^1\big((b,x)\big)} + |g(b)| }  .}
%Thus $\displaystyle  
% \[ \beta(s,1-s) u(x)= \int_b^x G(y)\, dy$ with $G\in L^1\big((b,x)\big)$. R
Moreover, almost everywhere in $[b,x]$ 
%(hence in $[b, \infty)$ given the arbitrary choice of $x$)
% we have that $u'(y)=\displaystyle \ig G(y)$. Therefore, inserting the definition of $G$ we conclude that almost everywhere in $[b,\infty)$
	 \bgs {\label{u1d}
	\frac{\pi} {\sin \pi s} \,u'(y)=\int_b^y g'(\tau)(y-\tau)^{s-1}\, d\tau + g(b)(y-b)^{s-1}. }
\textcolor{black}{With this, given the arbitrary choice of $x$, we have proved the claim \eqref{cbsu1}.}

We claim now that $ u'(\cdot) (x-\cdot)^{-s} \in L^1\big( (b,x)\big)$. Using the second identity in \eqref{cbsu1}, we obtain that
	\eqlab{ \label{fbca} &\frac{\pi}{\sin \pi s}  \int_b^x  | u'(y)| (x-y)^{-s} \, dy \\
	\leq 
	%\al \int_b^x \lr{  \int_b^y |g'(\tau )| (y-\tau)^{s-1}\, d\tau +| g(b) | (y-b)^{s-1} } (x-y)^{-s}dy\\
	 \al \int_b^x  \lr{\int_b^y |g'(\tau)| (y-\tau)^{s-1}  \, d\tau} (x-y)^{-s} \, dy + | g(b) |  \int_b^x (y-b)^{s-1} (x-y)^{-s}dy . }
Tonelli theorem applied to the positive function $|g'(\tau)| (y-\tau)^{s-1} (x-y)^{-s} $ on the domain $D_{b,x}$ \textcolor{black}{given in \eqref{rev2}} with the product measure $d(y,\tau)$ gives
	\bgs { \label{tt31} \iint_{D_{b,x}}  |g'(\tau)| (y-\tau)^{s-1} (x-y)^{-s} \, d(y, \tau) = \al \int_b^x  |g'(\tau)| \lr{ \int_\tau^x (y-\tau)^{s-1} (x-y)^{-s} \, dy } \, d\tau.} 
By using the change of variables $\displaystyle t = \frac{y-\tau}{x-\tau}$, thanks to the definition of the Beta function \eqref{betazerouno} and identity \eqref{betas} we have that
 \eqlab{ \label{bfcomp} 
\int_\tau^x (y-\tau)^{s-1} (x-y)^{-s} \, dy  = \int_0^1 t^{s-1}(1-t)^{-s} \, dt = \frac{\pi}{\sin\pi s}.}
Hence we obtain that
	\eqlab{  \label{Fubsto}  \iint_{D_{b,x}}  |g'(\tau)| (y-\tau)^{s-1} (x-y)^{-s} \, d(y ,\tau) 
	%= \al \beta(s,1-s) \int_b^x  |g'(\tau)| \, d\tau\\
 = \al  \frac{\pi}{\sin\pi s} \| g'\|_{L^1\left((b,x)\right)}.} 
\textcolor{black}{From this and using again \eqref{bfcomp} with $b=\tau$,} we obtain in \eqref{fbca} that 
	\bgs{ \al  \int_b^x  | u'(y)| (x-y)^{-s} \, dy \leq   \|g'\|_{L^1\lr{(b,x)}} + |g(b)|.}
Hence  $ u'(\cdot) (x-\cdot)^{-s} \in L^1\big( (b,x))$, as claimed. \textcolor{black}{From this and \eqref{cbsu1}, recalling definition \eqref{ca1s} it follows that} $u$ 
%defined in \eqref{solc1} 
belongs to the space $C_b^{1,s}$.

\bigskip 
We prove now that $u$ is a solution of the problem \eqref{probc1}. Using the second identity in \eqref{cbsu1} we have that 
	\eqlab{ \label{blaq1}  \frac{\pi}{\sin \pi s} \int_b^x u' (y)(x-y)^{-s}\,  dy  
						%	=\al   \int_b^x   \lr{  \int_b^y g'(\tau)(y-\tau)^{s-1}\, d\tau + {g(b)(y-b)^{s-1}}} (x-y)^{-s}\, dy\\
						=\al \int_b^x \lr { \int_b^y g'(\tau)(y-\tau)^{s-1} \, d\tau } (x-y)^{-s} \, dy  \\ \al +  g(b) \int_b^x (y-b)^{s-1} (x-y)^{-s}\, dy  .}
Thanks to \eqref{Fubsto}, we have that $|g'(\tau)| (y-\tau)^{s-1} (x-y)^{-s} \in L^1 \big(D_{b,x}, d(y,\tau)\big) $. We apply Fubini theorem and using \eqref{bfcomp} we get that  
	\bgs{ \int_b^x \lr{\int_b^y g'(\tau)(y-\tau)^{s-1} (x-y)^{-s}\, d\tau }\, dy =\al  \int_b^x g'(\tau) \lr{ \int_\tau^x (y-\tau)^{s-1} (x-y)^{-s} \, dy }  \, d\tau, \\
		%=\al \beta(s,1-s) \int_b^x g'(\tau)\, d\tau\\ 
		=\al \frac{\pi}{\sin \pi s}  \lr{ g(x)-g(b) } .} 
%where we have used again \eqref{bfcomp}. 
Thanks again to \eqref{bfcomp}, in \eqref{blaq1} it follows that 	
\bgs{ \int_b^x \al u'(y)(x-y)^{-s}\, dy  = 	g(x),}
therefore $u$ is a solution of the problem \eqref{probc1}.
\bigskip

The solution is unique. We prove this by taking two different solutions $u_1,u_2\in C_b^{1,s}$ of the problem \eqref{probc1}. Let $u:=u_1-u_2$, then $u$ satisfies
	\bgs{ \int_b^x u'(t)(x-t)^{-s}\, dt &=0 &\text{in} \quad &(b,\infty),	\\
		u(b)&=0 .&&}
We take any $y>x$, we multiply both terms by the positive quantity $(y-x)^{s-1}$, integrate from $b$ to $y$ and obtain that
	\eqlab {\label{bla3}  \int_b^y \lr{ \int_b^x u'(t) (x-t)^{-s}\, dt } (y-x)^{s-1} \, dx =0.}
Since $u \in  C_b^{1,s}$, \textcolor{black}{we use Tonelli theorem on $D_{b,y}$ (we recall definition \eqref{rev2}) and by \eqref{bfcomp} we obtain that}
 	\bgs{  \iint_{D_{b,y}}  |u'(t)|  (x-t)^{-s}(y-x)^{s-1} \, d(x,t) 
	=\al  \int_b^y |u'(t)| \lr{ \int_t^y  (x-t)^{-s}(y-x)^{s-1} \, dx} \, dt\\
	=\al  \frac{\pi}{\sin \pi s} \|u'\|_{L^1\lr{(b,y)}}, } which is a finite quantity. Fubini theorem then allows us to compute 
\bgs{  \int_b^y \lr{ \int_b^x u'(t) (x-t)^{-s}\, dt } (y-x)^{s-1} \, dx =&\;  \int_b^y  u'(t) \lr{  \int_t^y (x-t)^{-s}(y-x)^{s-1} \, dx }\, dt \\
		= \al  \frac{\pi}{\sin \pi s}u(y)  	  .} 
It follows from \eqref{bla3} and from the initial condition $u(b)=0$ that $u_1(x)= u_2(x)$ on $[b,\infty)$. Therefore $u$ given in \eqref{solc1} is the unique solution of the problem \eqref{probc1} and this concludes the proof of the Theorem.
\end{proof}

We introduce an interior regularity result.
\begin{lemma} \label{intreg}
Let $g \in C^{\infty}\big([b,\infty)\big)$ and $u$ be defined as in \eqref{solc1}. Then $u\in C^{\infty}\big((b,\infty)\big)$. 
\end{lemma} 
\begin{proof}
We prove by induction that the next statement, which we call $P(n)$, holds for any $n\in \N$: 
\textcolor{black}{  \[ u\in C^n\big((b,\infty)) \]} 
and
 \eqlab{ \label{undiff1} u^{(n)}(y) = \ig\lr{  \int_b^y g^{(n)}(\tau) (y-\tau)^{s-1} \, d\tau +  \sum_{i=0}^{n-1} \tilde c_{s,i} g^{(i)}(b) (y-b)^{s-n+i}  } \\\mbox{for any } y\in (b,\infty) ,}
where 
	\begin{equation} \label{ctcsi1} \tilde c_{s,i} = \begin{cases}    (s-1)\dots (s-n+i+2) (s-n+i+1) \quad &\text{ for }  i\neq n-1\\
										  1  \quad \quad\quad\quad &\text{ for }  i= n-1. \end{cases}
	\end{equation}
		
%From identity \eqref{cbsu2} we have that almost everywhere in $[b,\infty)$ 
%\[ u'(y) = \ig \lr{\int_b^y g'(\tau)(y-\tau)^{s-1} \, d\tau + g(b)(y-b)^{s-1} }.\]
%The function $(y-b)^{s-1}$ is continuous on $(b,\infty)$. 
We denote by
	\[ v(y):=\int_b^y g'(\tau)(y-\tau)^{s-1}\, d\tau \] 
\textcolor{black}{and from \eqref{cbsu1} we have that almost anywhere in $[b,\infty)$
	\eqlab{\label{uprim} u'(y) = \ig \lr{ v(y) +   g(b)(y-b)^{s-1}}.} Since $g\in C^{\infty}\big([b,\infty)\big)$, we have in particular that $g'\in C_b^{1,1-s}$ hence from the definition of $v$ and \eqref{cbsu1} we get that $v\in AC\big([b,\infty)\big)$. It follows that $u'\in C\big((b,\infty)\big)$, since it is a sum of continuous functions. Therefore $u\in C^1\big((b,\infty)\big)$ and \eqref{uprim} holds pointwise in $(b,\infty)$}.
%	\[ u'(y) = \ig \lr{\int_b^y g'(\tau)(y-\tau)^{s-1} \, d\tau + g(b)(y-b)^{s-1} }.\]
And so $P(1)$ is true. 

In order to prove the inductive step, we suppose that $P(n)$ holds and prove $P(n+1)$.
Let now
	\[ v(y):=\int_b^y  g^{(n)}(\tau)(y-\tau)^{s-1}\, d\tau.\] From \eqref{undiff1} we have that for any $y\in (b,\infty)$
	\eqlab{\label{unv1} u^{(n)}(y) =\ig\lr{ v(y)+  \sum_{i=0}^{n-1} \tilde c_{s,i} g^{(i)}(b) (y-b)^{s-n+i} }  .} 
\textcolor{black}{Since $g\in C^{\infty}\big([b,\infty)\big)$, in particular we have that $g^{(n)}\in C_b^{1,1-s}$, hence from the definition of $v$ and thanks to \eqref{cbsu1} we get that $v\in AC\big([b,\infty)\big)$ and 
%n formula \eqref{} By the considerations made in the proof of statement $P(1)$, we have that $v\in C^1\big((b,x]\big)$ and also (see equality 
almost everywhere on $[b,\infty) $
	\[ v'(y) = \int_b^y g^{(n+1)}(\tau) (y-\tau)^{s-1}\, d\tau + g^{(n)}(b) (y-b)^{s-1}.\] Now, also $g^{(n+1)}\in C_b^{1,1-s}$ and so, thanks to \eqref{cbsu1}, the map 
	\eqlab{\label{yg1} y \mapsto \displaystyle \int_b^y g^{(n+1)}(\tau) (y-\tau)^{s-1}\, d\tau \quad \in AC\big([b,\infty)\big) .} It yields that $v\in C^1\big((b,\infty)\big)$ and so from \eqref{unv1} we get that $u^{(n+1)}\in C\big((b,\infty)\big)$. Taking the derivative of \eqref{unv1} we have that pointwise in $(b,\infty)$}
\bgs{ \frac{\pi}{\sin\pi s}  u^{(n+1)} (y)=  \al \int_b^y g^{(n+1)}(\tau) (y-\tau)^{s-1}\, d\tau  +  	
%	
%Since the statement $P(n)$ is true, one has from \eqref{undiff1} that
%\[ u^{(n)}(y) = \ig v(y) + \ig \sum_{i=0}^{n-1} \tilde c_{s,i} g^{(i)}(b) (y-b)^{s-n+i}. \]
	%\beta(s,1-s)  u^{(n+1)} (y)  = \al   v'(y)+ 
	g^{(n)}(b)(y-b)^{s-1} \\ &+ 
	\sum_{i=0}^{n-1} \tilde c_{s,i} g^{(i)}(b) (s-n+i) (y-b)^{s-n+i-1} \\
%=\al   \int_b^y g^{(n+1)}(\tau)(y-\tau)^{s-1}\, d\tau +   g^{(n)}(b) (y-b)^{s-1} \\
%			\al +   \sum_{i=0}^{n-1} \tilde c_{s,i} g^{(i)}(b) (s-n+i) (y-b)^{s-n+i-1}\\
		=\al \int_b^y g^{(n+1)}(\tau)(y-\tau)^{s-1}\, d\tau  +  \sum_{i=0}^{n} \tilde c_{s,i} g^{(i)}(b) (y-b)^{s-n+i},}
		 where we have used \eqref{ctcsi1} in the last line. 
Therefore the statement $P(n+1)$ is true and the proof by induction is concluded. 

It finally yields that $u\in C^{\infty}\big((b,\infty)\big)$ and this concludes the proof of the Lemma.
\end{proof}
\smallskip 

\subsection{Building a sequence of Caputo-stationary functions}\label{sectaps}

In this subsection we build a sequence of functions that are Caputo-stationary  in $(0,\infty)$ and that tends uniformly on bounded subintervals of $(0,\infty)$ to the function $x^s$. We do this by building a Caputo-stationary function in $(1,\infty)$, that at the point $1+\varepsilon$ is asymptotic to $\varepsilon^s$ and then we use a blow-up argument.

\bigskip

We fix the arbitrary parameter $s\in (0,1)$. We introduce the first Lemma of this subsection.

\begin{lemma}\label{caplem1}
Let $\psi_0 \in C^1\big( [0,1]\big) \cap C\big((-\infty,1]\big)$ be such that 
	\eqlab {	\label{fifi41} &\psi_0(x)=\psi_0(0) & \text{ for any } &x\in (-\infty, 0],\\
			&\psi_0(x) =0 &\text{ for any } &x\in \bigg[\frac{3}{4}, 1\bigg],\\
	&\psi'_0 (x)<0 &\text{ for any } &x\in \bigg[0,\frac{3}{4}\bigg). }
Let $\psi\in C_0^{1,s}$ be the solution of the problem
	\eqlab {\label{capprob1}
			D_0^s\psi(x) &=0 &\text{ in }& (1,\infty),\\
			\psi(x)&=\psi_0(x)&\text{ in } &(-\infty,1].} \textcolor{black}{Then $\psi \in C^{\infty}\big((1,\infty)\big)$ and} if $x=1+\varepsilon$, we have that 
	\eqlab{\label{psieee} \psi(1+\varepsilon) =\kappa \varepsilon^s + \mathcal O(\varepsilon^{s+1})  }
as $\varepsilon \to 0$, for some $\kappa>0$. 
\end{lemma}
%We prove this Lemma in the following way:
%\begin{itemize}
%	\item We use at first Lemma \ref{lem:int11} to affirm that $\psi \in C_0^{1,s}$ is solution of the problem \eqref{prob1} if and only if 
%	\bgs{ \int_1^x \psi'(t)(x-t)^{-s}\, dt &= -\int_0^{3/4} \psi_0'(t) (x-t)^{-s} \, dt &\mbox{in } &(1,\infty), \\
%										\psi(x) &=\psi_0(x) &\mbox{in } &(-\infty,1];}	
%	\item We prove that the function $g\colon [1,\infty)\to \R$ defined as
%	\bgs{g(x):=- \int_0^{3/4} \psi_0'(t) (x-t)^{-s} \, dt } belongs to the space $C^{\infty}\big([1,\infty)\big)$. This says that in particular $g \in C_1^{1,1-s}$;\\
%	\item Noticing that $\psi(1)=0$, Theorem \ref{thm:probc} implies that we have a unique solution of the problem \eqref{prob1}, given by
%	\bgs{ &\psi(x)=\ig  \int_1^x g(t) (x-t)^{s-1} \, dt &\mbox{ in } &(1,\infty),\\
%		&\psi(x)=\psi_0(x) &\mbox{ in } & (-\infty,1].}  
%	\item By substituting $x=1+\varepsilon$ into \eqref{solll}, we prove the asymptotic behavior in \eqref{psieee};
%	\item Since $g\in C^{\infty}\big([1,\infty)\big)$, Lemma \ref{intreg} assures us that $\psi \in C^{\infty}\big((1,\infty)\big)$. 
%\end{itemize}
%\begin{center}
% \begin{figure}[htpb]
%	\hspace{0.6cm}
%	\begin{minipage}[b]{0.95\linewidth}
%	\centering
%	\includegraphics[width=0.95\textwidth]{Lem41.png}
%	\caption{A Caputo-stationary function in $(1,\infty)$, prescribed on $(-\infty, 1]$ as in \eqref{fifi41}}
%	\label{fign:Lem41}
%	\end{minipage}
%\end{figure} 
%\end{center}

An explicit example of a function described in Lemma \ref{caplem1} is depicted in Figure \ref{fign:es2} in Example \ref{exam2}.
\begin{proof}[Proof of Lemma \ref{caplem1}]

Thanks to Lemma \ref{lem:int11} we have that $\psi \in C_0^{1,s}$ is solution of the problem \eqref{capprob1} if and only if 
	\bgs{ \int_1^x \psi'(t)(x-t)^{-s}\, dt &= -\int_0^{3/4} \psi_0'(t) (x-t)^{-s} \, dt &\mbox{in } &(1,\infty), \\
										\psi(x) &=\psi_0(x) &\mbox{in } &(-\infty,1].}	
On $[1,\infty)$ we define the function
\eqlab{\label{gbla1} g(x):=- \int_0^{3/4} \psi_0'(t) (x-t)^{-s} \, dt, }
\textcolor{black}{hence our problem is now
	\eqlab{ \label{psil4} \int_1^x \psi'(t)(x-t)^{-s}\, dt &=g(x) &\mbox{in } &(1,\infty), \\
										\psi(x) &=\psi_0(x) &\mbox{in } &(-\infty,1].}	
We claim that $g\in C^{\infty}\big([1,\infty)\big)$.}
For that, let $F\colon [1,\infty)\times [0,3/4]\to \R$ be defined as $F(x,t):=\psi_0'(t)(x-t)^{-s}$. 
\textcolor{black}{Now, for any $h>0$ arbitrarily small we have that
\[ \bigg| \frac{F(x+h,t)-F(x,t)}h \bigg| \leq \sup_{t\in [0,3/4]} |\psi_0'(t)| \bigg|\frac{(x+h-t)^{-s}-(x-t)^{-s}}h\bigg|. \]
Since the map $[1,\infty)\ni x \mapsto (x-t)^{-s}$ is differentiable for any $t\in [0,3/4]$, by the Mean Value Theorem we have that for $\theta \in (0,h)$
\[ \bigg|\frac{(x+h-t)^{-s}-(x-t)^{-s}}h\bigg| \leq s (x+\theta-t)^{-s-1} \leq s(x-t)^{-s-1}.\]
Then
\[\bigg|\frac{F(x+h,t)-F(x,t)}h\bigg|\leq s \sup_{t\in [0,3/4]} |\psi_0'(t)| (x-t)^{-s-1} \in L^1\big([0,3/4],dt\big),\]
hence by the Dominated Convergence Theorem, we can pass the limit inside the integral and obtain that
\[ g'(x)= -\int_0^{3/4} \partial_x F (x,t)\, dt =  s\int_0^{3/4} \psi'_0(t)  (x-t)^{-s-1}\, dt.\]}
%Then, since $x\geq1$ and $t\in [0,3/4]$ we have that 
%	\[ \int_0^{3/4} |F(x,t)|\,dt \leq 3 \cdot 4^{s-1} \sup_{t\in [0,3/4]}|\psi_0'(t)|,\]
%hence $F(x,\cdot)\in L^1\big([0,3/4],\,dt\big)$. Furthermore, the partial derivative with respect to $x$ exists everywhere in $(1,\infty)$, we have that $F_x(x,t) = (-s) \psi'_0(t)(x-t)^{-s-1}$ and
%	\[ |F_x(x,t)|\leq s 4^{s+1} |\psi_0'(t)| \in L^1\big([0,3/4],\, dt\big).\] 
%This considerations allow us to take the derivative under the integral sign, and 
We can now take for any $n \in \N$ the function $F_n\colon [1,\infty)\times [0,3/4]\to \R$ to be $F_n(x,t):=\psi_0'(t)(x-t)^{-s-n}$ and repeat the above argument. We obtain that $g$ is $C^\infty\big([1,\infty)\big)$, as claimed and
% differentiation as many times as we want, hence for any $n \in \N$, the derivative $g^{(n)}$ exists, is obviously continuous  and we have that
moreover for any $n\in \N_0$ we have that
	\eqlab{ \label{gn1} g^{(n)}(x) =  -\bar c_{s,n}\int_0^{3/4} \psi_0'(t) (x-t)^{-s-n}\,dt, }
where
	\begin{equation} \label{ctcns2} \bar c_{s,n} = \begin{cases} (-s)(-s-1)\dots (-s-n+1) &\mbox{ for } n\neq 0\\
																1 &\mbox{ for } n=0.\end{cases}\end{equation}
																
Since $\psi(1)=0$ and $g\in C^\infty\big([1,\infty)\big)$ (hence in particular $g\in C_1^{1,1-s}$), thanks to Theorem \ref{thm:probc} we get that the problem \eqref{psil4} admits a unique solution $\psi \in C_1^{1,s}$ given by
\eqlab{ \label{solll} &\psi(x)=\ig  \int_1^x g(t) (x-t)^{s-1} \, dt &\mbox{ in } &(1,\infty),\\
		&\psi(x)=\psi_0(x) &\mbox{ in } & (-\infty,1].} 
 %We notice that $\psi$ was taken such that $\psi(1)=0$. Theorem \ref{thm:probc} implies that we have a unique solution 
%$\psi \in C_1^{1,s}$ of the problem \eqref{psil4} given by  \eqref{solc1}
%	\eqlab{ \label{solll} &\psi(x)=\ig  \int_1^x g(t) (x-t)^{s-1} \, dt &\mbox{ in } &(1,\infty),\\
%		&\psi(x)=\psi_0(x) &\mbox{ in } & (-\infty,1].} 
 \textcolor{black}{Moreover, we claim that $\psi \in C_0^{1,s}$. Indeed, from Lemma \ref{intreg} we get that $\psi\in C^{\infty}\big( (1,\infty)\big)$. Also 
$ \lim_{x\to 1^+} \psi(x) =0 =\psi(1)$ and so from this and the hypothesis we have that $\psi \in C^{\infty}\big( (1,\infty)\big)  \cap  C^1\big([0,1]\big) \cap C(\R)$, hence $\psi \in AC\big([0,\infty)\big)$. Also for any $x>0$
\[ \int_0^x |\psi'(t)(x-t)^{-s}| \, dt \leq c_s  \|\psi'\|_{L^{\infty}\lr{(0,x)}} x^{1-s} <\infty, \]
and so the claim follows from definition \eqref{ca1s}.}
%$\psi \in C\lr{(-\infty,1] }$, and since $\psi \in C^{\infty}\big( (1,\infty)\big)$  we have that
Therefore, $\psi \in C_0^{1,s}$ is the unique solution of problem \eqref{psil4} and from Lemma \ref{lem:int11} it follows that \eqref{solll} is also the unique solution of problem the \eqref{capprob1}.  
\bigskip

We prove now the claim \eqref{psieee}. Let $x=1+\varepsilon$. Then from \eqref{solll} we have that 
\bgs{  \frac{\pi}{\sin \pi s} \psi(1+\varepsilon) =  \int_1^{1+\varepsilon} g(\tau) (1+\varepsilon-\tau)^{s-1} \, d\tau.} 
The change of variables $z= (\tau-1)/\varepsilon$ gives
	\bgs{ \frac{\pi}{\sin \pi s} \psi(1+\varepsilon) = \varepsilon^s  \int_0^1 g(\varepsilon z+ 1) (1-z)^{s-1}\, dz.}
Using definition \eqref{gbla1} we have that 
	\[ g(\varepsilon z+ 1) = -\int_0^{3/4} \psi_0'(t)(\varepsilon z+1-t)^{-s} \, dt ,\] 
hence
	\bgs{ \frac{\pi}{\sin \pi s} \psi(1+\varepsilon) = -\varepsilon^s \int_0^1 \lr{ \int_0^{3/4} \psi_0'(t)(\varepsilon z+1-t)^{-s}\, dt} (1-z)^{s-1}\,dz.} 
Tonelli theorem on $[0,1]\times [0,3/4]$ applied to the function $|\psi_0'(t)|(\varepsilon z+1-t)^{-s} (1-z)^{s-1}$ yields
	\bgs{  \iint_{ [0,1]\times [0,3/4]} \al |\psi_0'(t)|(\varepsilon z+1-t)^{-s} (1-z)^{s-1} d(t,z)\\
=\al  \int_0^{3/4} |\psi_0'(t)| \lr{ \int_0^1 (1-z)^{s-1}(\varepsilon z+1-t)^{-s}\, dz}\, dt.}
We have that $(\varepsilon z+1-t)^{-s}\leq (1-t)^{-s}\leq 4^s$, hence
	 \bgs{  \int_0^{3/4} |\psi_0'(t)| \lr{ \int_0^1 (1-z)^{s-1}(\varepsilon z+1-t)^{-s}\, dz}\, dt \leq \al 4^s  \int_0^{3/4} |\psi_0'(t)| \lr{\int_0^1 (1-z)^{s-1}\, dz}\, dt  \\
			\leq \al \frac{3 \cdot 4^{s-1}}s \sup_{t\in[0,3/4]}|\psi_0'(t)|  ,}
which is finite. Therefore $|\psi_0'(t)|(\varepsilon z+1-t)^{-s} (1-z)^{s-1}\in L^1\big([0,1]\times [0,3/4], d(t,z)\big)$ and by Fubini theorem we have that
	\eqlab{ \label{psibla2}
			 \frac{\pi}{\sin \pi s}  \psi(1+\varepsilon)  =\al  -\varepsilon^s \int_0^{3/4}\psi_0'(t) \lr{ \int_0^1 (\varepsilon z+1-t)^{-s} (1-z)^{s-1}  \, dz}\, dt \\
			=\al -\varepsilon^s \int_0^{3/4} \psi_0'(t)  I_s(\varepsilon,t)\, dt.}
We consider the function $f(z)=(\varepsilon z+1-t)^{-s}$ and make a Taylor expansion with a Lagrange reminder in $0$. Namely, one has that there exists $c \in (0,z)$ such that
%	\[ f(z)=\sum_{i=0}^n f^{(i)}(0) \frac{z^i}{i!} + \frac{f^{(n+1)}(c)}{(n+1)!}z^{n+1}.\]
%We have that for some $c\in (0,z)$ 
	\[ (\varepsilon z+ 1-t)^{-s} = \sum_{i=0}^n  \frac{\bar c_{s,i}} {i!}\varepsilon^i  (1-t)^{-s-i}  z^i + \frac{\bar c_{s,n+1}}{(n+1)!} \varepsilon^{n+1} (\varepsilon c+1-t)^{-s-n-1} z^{n+1},\]
where $\bar c_{s,i}$ is given in \eqref{ctcns2}. Using this, we have that
	\bgs{ I_s(\varepsilon,t)= \al \sum_{i=0}^n  \frac{\bar c_{s,i}} {i!}\varepsilon^i  (1-t)^{-s-i}\int_0^1 (1-z)^{s-1} z^i\, dz \\ \al +  \frac{\bar c_{s,n+1}}{(n+1)!} \varepsilon^{n+1} (\varepsilon c+1-t)^{-s-n-1} \int_0^1 (1-z)^{s-1} z^{n+1}\, dz.}
We use the definition \eqref{betazerouno} of the Beta function and continue
	\bgs{  I_s(\varepsilon,t) = \sum_{i=0}^n  \frac{\bar c_{s,i}\beta(i+1,s) } {i!}\varepsilon^i  (1-t)^{-s-i} + \frac{\bar c_{s,n+1}\beta(n+2,s)}{(n+1)!} \varepsilon^{n+1} (\varepsilon c+1-t)^{-s-n-1}.}
	In \eqref{psibla2} we obtain that
	\eqlab { \label{psibla3}  \frac{\pi}{\sin \pi s}  \psi(1+\varepsilon) =\al -\varepsilon^{s} \sum_{i=0}^n \frac{\bar c_{s,i}\beta(i+1,s) } {i!} \varepsilon^i \int_0^{3/4}\psi_0'(t) (1-t)^{-s-i}\, dt \\ \al- \varepsilon^{s+n+1} \frac{\bar c_{s,n+1}\beta(n+2,s)}{(n+1)!}  \int_0^{3/4} \psi_0'(t)  (\varepsilon c+1-t)^{-s-n-1} \, dt.}
We notice that $(\varepsilon c+1-t)^{-s-n-1}\leq 4^{s+n+1}$ and it follows that
	\[ \bigg| \int_0^{3/4} \psi_0'(t)  (\varepsilon c+1-t)^{-s-n-1} \, dt\bigg| \leq 3\cdot  4^{s+n} \sup_{t \in [0,3/4]} |\psi_0'(t)|,\] which is finite.
 We define then the finite quantities
	\bgs{ C_{s,\psi_0,i}:=\al  -\frac{\bar c_{s,i} \beta(i+1,s) }{i!} \int_0^{3/4} \psi_0'(t)(1-t)^{-s-i}\, dt\\ =\al  \frac{\beta(i+1,s) } {i!}   g^{(i)}(1) \quad  \mbox{for} \quad  i=0,\dots,n } and
	 \bgs{ C_{s,\psi_0,n+1}: = \al -\frac{\bar c_{s,n+1} \beta(n+2,s) }{(n+1)!} \int_0^{3/4} \psi_0'(t)(\varepsilon c +1-t)^{-s-n-1}\, dt \\
=\al  \frac{\beta(n+2,s) } {(n+1)!}   g^{(n+1)}(\varepsilon c +1), }
where we have used \eqref{gn1}.  \\
It follows in \eqref{psibla3} that
	\[ \frac{\pi}{\sin \pi s}\psi(1+\varepsilon) = \sum_{i=0}^{n+1} C_{s,\psi_0,i}  \textcolor{black}{\varepsilon^{s+i}} . \]
 This gives for $\varepsilon \to 0$ that
	\[\psi(1+\varepsilon)= \kappa\varepsilon^s + \mathcal O (\varepsilon^{s+1}),\]
where 
	\bgs{ \kappa =C_{s,\psi,0}= \beta(1,s) g(1) \textcolor{black}{=} -   \beta(1,s) \int_0^{3/4} \psi_0'(t)(1-t)^{-s} \, dt. }
Since $-\psi_0'(x)>0$ in $[0,3/4)$ by hypothesis (see \eqref{fifi41}), we have that 
\[-\int_0^{3/4} \psi_0'(t)(1-t)^{-s} \, dt>0.\] 
This implies that $\kappa$  is strictly positive and it concludes the proof of the Lemma.
\end{proof}
 \bigskip

Blowing up the function built in Lemma \ref{caplem1}, we obtain a sequence of Caputo-stationary functions in $(0,\infty)$ that on $(0,\infty)$ tends to the function $x^s$.  

\begin{lemma}\label{ls1}
There exists a sequence $(v_j)_ {j \in \N}$ of functions $v_j \in  C^{1,s}_{-j}\cap C^{\infty}\big((0,\infty)\big)$ such that  for any $j \in \N$
	\begin{equation} \label{pbvj1}
		\begin{aligned}
			D^s_{-j} v_j(x)&= 0 &\text{ in } &(0,\infty),	\\
			v_j(x) &= 0 &\text{ in } & \Big[-\frac{j}4,0\Big]
		\end{aligned} 
	\end{equation}  
and for any $x>0$
	 \eqlab{ \label{limvj} \lim_{j\to \infty} v_j(x)=\textcolor{black}{\kappa} x^s,  } for some $\textcolor{black}{\kappa}>0$. 
Moreover, on any bounded subinterval $I\subseteq (0,\infty)$ the convergence is uniform.
\end{lemma}	
A qualitative example of a sequence described in Lemma \ref{ls1} is depicted in Figure \ref{fign:Lem42}.

\begin{center}
 \begin{figure}[htpb]
	\hspace{0.6cm}
	\begin{minipage}[b]{0.85\linewidth}
	\centering
	\includegraphics[width=0.95\textwidth]{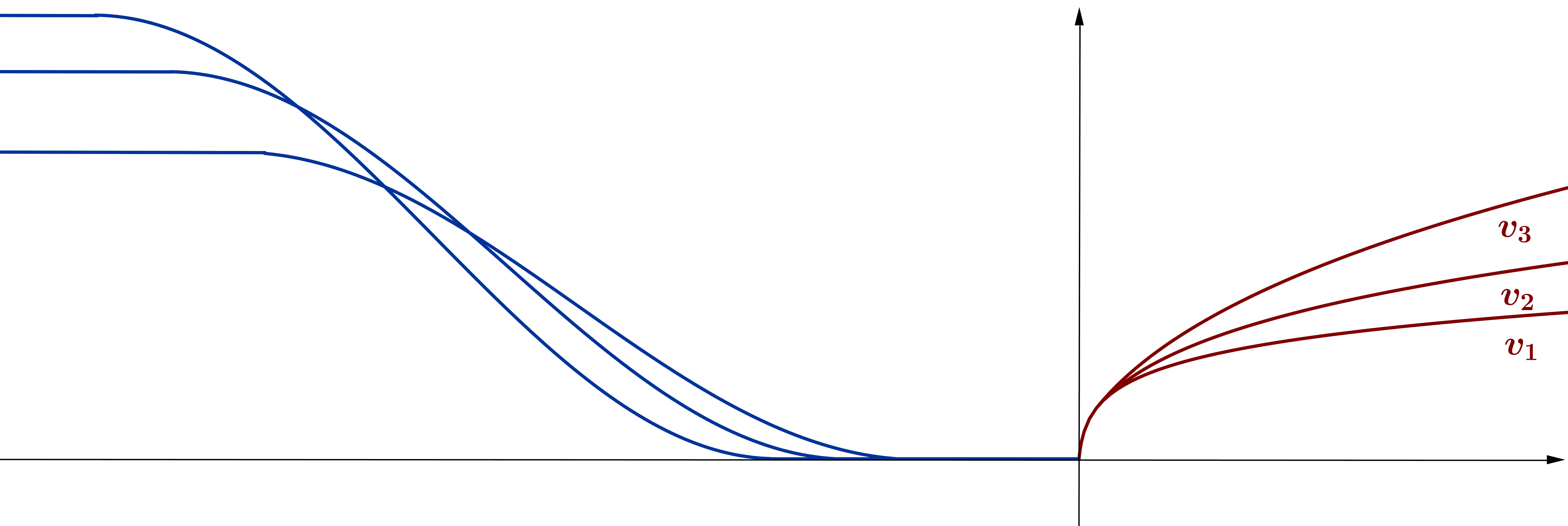}
	\caption{A sequence of Caputo-stationary functions in $(0,\infty)$}   
	\label{fign:Lem42}
	\end{minipage}
\end{figure}
\end{center} 
\begin{proof}
We consider the function $\psi$ solution of the problem \eqref{capprob1} as introduced in Lemma \ref{caplem1}, and define for any $j\in \N$
\[ v_j(x) := j^s \psi\bigg(\frac{x}{j} +1\bigg).\] 
% We prove that for any $j\in \N$ the function $v_j$ is solution of the problem \eqref{pbvj1}.
%Recalling Lemma \ref{caplem1}, we easily verify that
%\bgs {&v_j(x)=j^s \psi_0(x) & \text{ for any } &x\in (-\infty,0],
%	&v_j(x)=j^s \psi_0(0) & \text{ for any } &x\in (-\infty, -j],\\
%			&v_j(x)  =0 &\text{ for any } &x\in \bigg[-\frac{j}{4}, 0\bigg]. }
%%
%have that $\psi(x)=\psi_0(x)$ in $(-\infty,1]$, hence $v_j(x)=j^s \psi_0 \displaystyle \bigg(\frac{x}{j} +1\bigg)$ when $\displaystyle \frac{x}{j} +1 \leq 1$, i.e. when $x\leq 0$. Moreover, from conditions \eqref{fifi41} we have that $v_j(x)=j^s \psi_0(0)$ when $\displaystyle \frac{x}{j} +1 \leq 0$, hence when $x\leq -j$ and $v_j(x)=0$ when $\displaystyle\frac{3}{4}\leq \displaystyle \frac{x}{j} +1 \leq 1$, hence for $x\in \displaystyle \left[-\frac{j}4,0\right]$.  
Since $\psi \in C_0^{1,s}\cap C^{\infty}\big((1,\infty)\big)$, then $v_j\in C_{-j}^{1,s} \cap C^{\infty}\big((0,\infty)\big)$.  Also, since $\psi$ is solution of the problem \eqref{capprob1}, we have that
	\bgs{	D_{-j}^s v_j(x) =\al  \frac{1}{\Gamma(1-s)}  \int^x_{-j} v_j'(t)(x-t)^{-s} \, dt\\
					= \al  \frac{ j^{s-1}}{\Gamma(1-s)} \int^x_{-j}\psi'\Big(\frac{t}{j}+1\Big) (x-t)^{-s} \, dt.}
We use the change of variables $y= t/j +1$ and obtain
	\bgs{	D_{-j}^s v_j(x) = \frac{1}{\Gamma(1-s)} \int_0^{x/j+1} \psi'(y)\Big(\frac{x}{j}+1-y\Big)^{-s}\, dy = D_0^s \psi \Big(\frac{x}{j} +1\Big).}
This implies that $D_{-j}^s v_j(x) =0 $ (using \eqref{capprob1}) when $x>0$. So, using Lemma \ref{caplem1} and the definition in $v_j$, we easily verify
%\textcolor{black}{And so in conclusion} we 
that for any $j\in \N$  the functions $v_j \in C_{-j}^{1,s} \cap C^{\infty}\big((0,\infty)\big)$ satisfy
	\bgs{ D_{-j}^s v_j(x) &= 0 &\mbox{ in } &  (0,\infty), \\ 
			v_j(x) &= 0 &\mbox{ in } & \left[-\frac{j}4,0\right] }
and
	\bgs{ v_j(x)&=j^s\psi_0\lr{\frac{x}j +1} 	&\mbox{ in } &  (-\infty,0],\\
				v_j(x)&=j^s\psi_0(0) 	&\mbox{ in } &  (-\infty,-j].}
In particular, $v_j$ is solution of the problem \eqref{pbvj1} for any $j \geq 1$.

Now, using \eqref{psieee}, for $x>0$ and for a large $j$ we have that
	\[ v_j(x) = j^s \psi \left(\frac{x}{j}+1 \right) = j^s \left( \kappa \frac{x^s}{j^s} + \mathcal O \left(\frac{x^{s+1}}{j^{s+1}}\right)\right) = \kappa x^s + \mathcal O \left(\frac{x^{s+1}}{j}\right).\]
By sending $j$ to infinity we obtain that 
	\[ \lim_{j\to \infty} v_j(x)=\kappa x^s.\]
On any bounded subinterval $I\subseteq (0,\infty)$, we have that
	\bgs{ \lim_{j \to \infty} \sup_{x\in I} |v_j(x) - \kappa x^s| = 0. } It follows also that on any bounded subinterval $I\subseteq (0,\infty)$ the sequence $v_j$ is uniformly bounded.
This concludes the proof of the Lemma. 
\end{proof}

\subsection{A Caputo-stationary function with  derivatives prescribed}\label{sectma}

Using Lemma \ref{ls1} we prove that there exists a Caputo-stationary function with arbitrarily large number of derivatives prescribed.
%Namely, for any $m\in \N$ we want to prove that we can find a Caputo-stationary function $v$ and a point $p$, such that the derivatives of 
%$v$ in $p$ vanish until the order $m-1$. 
More precisely:

\begin{theorem}\label{thm4}
 For any $m \in \N$ there exist a point $p>0$, a constant $R>0$ and a function $v \in C_{-R}^{1,s} \cap C^{\infty} \big((0,\infty)\big)$ such that
	\eqlab{ \label{cc1} 
				D^s_{-R} v(x)&=0 &\text{ in } &(0, \infty), \\
		v(x)&=0 &\text{ in } & \Big [-\frac{R}4,0\Big]}	
and 
	\eqlab{ \label{cc2}
			&v^{(l)}(p)=0 & & \text{ for any } \quad  l< m\\
		&v^{(m)}(p)=1.&&}
	\end{theorem}

%The proof is done by contradiction. Suppose that for any Caputo-stationary function with properties \eqref{cc1}, a point $p>0$ such that the properties \eqref{cc2} are fulfilled does not exist. Then one observes that all the functions $v_j$ introduced in Lemma \eqref{ls1} satisfy \eqref{cc1}. We are able to reach a contradiction by using the asymptotic behavior in \eqref{limvj}. 
\begin{proof}

We consider $\mathcal Z$ to be the set of the pairs $(v,x)$ of all functions $v\in C_{-R}^{1,s}\cap C^{\infty} \big((0,\infty)\big)$ satisfying conditions \eqref{cc1} for some $R>0$, and $x\in (0,\infty)$. So let
	\bgs{\mathcal Z: = \Big\{ (v,x) \; \big| \;\;  x\in (0,\infty) \mbox{ and } \exists\, R>0 \text{ s.t. } & v\in C_{-R}^{1,s} \cap C^{\infty} \big((0,\infty)\big), D^s_{-R} v=0 \text{ in } (0, \infty), \\ 
	&v =0 \text{ in } \Big [-\frac{R}4,0\Big]  \Big\}. }
We fix $m\in \N$. To each pair $(v,x)\in \mathcal Z$ we associate the vector $ \big(v(x), v'(x), \dots, v^{(m)}(x)\big) \in \R^{m+1}$ and consider $\mathcal V$ to be the vector space spanned by this construction. We claim that this vector space exhausts $\R^{m+1}$. 
Suppose by contradiction that this is not so and $\mathcal V$ lays in a hyperplane. Then there exists a vector $(c_0,c_1,\dots,c_m)\in \R^{m+1}\setminus \{0\}$ orthogonal to any vector $  \big(v(x), v'(x), \dots, v^{(m)}(x)\big) $ with $(v,x) \in \mathcal Z$, hence 
	\[ \sum_{i=0}^m c_i v^{(i)}(x) = 0.\] 
We notice that for any $j\geq 1$ the pairs $(v_j,x)$ with $v_j$ satisfying problem \eqref{pbvj1} and $x\in (0,\infty)$ belong to the set $\mathcal Z$. It follows that for any $j\geq 1$ we have that
	\eqlab{  \label{bla10} \sum_{i=0}^m c_i v_j^{(i)} (x) =0.}

Let $\varphi \in  C^{\infty}_c\big((0,\infty)\big)$ be a smooth compactly supported function. Integrating by parts we have that for every $i\in \N_{0}$
	\bgs{ \int_\R v_j^{(i)}(x)\varphi(x)\, dx= (-1)^i \int_\R v_j(x) \varphi^{(i)}(x)\, dx.} 
Thanks to Lemma \ref{ls1}, the sequence $v_j$ is uniformly convergent to $\kappa x^s$ on any bounded subinterval $I\subseteq (0,\infty)$, for some $\kappa>0$. By the  Dominated Convergence Theorem we have that
	\[ \lim_{j \to \infty}  \int_\R v_j^{(i)}(x)\varphi(x)\, dx = (-1)^i\lim_{j \to \infty}  \int_{\R} v_j(x) \varphi^{(i)}(x)\, dx =   (-1)^i \int_{\R} \kappa x^s \varphi^{(i)}(x) \, dx.\] We integrate by parts one more time and obtain that
	\[  (-1)^i \int_\R \kappa x^s \varphi^{(i)}(x)\, dx = \int_\R \kappa (x^s)^{(i)} \varphi(x)\, dx.\]
It follows that
	\[ \lim_{j \to \infty}  \int_\R v_j^{(i)}(x)\varphi(x)\, dx =   \int_\R  \kappa ( x^s)^{(i)} \varphi(x)\, dx .\] Multiplying by $c_i$ and summing up, we obtain that
\bgs{  \lim_{j\to \infty}  \int_\R \sum_{i=0}^m c_i v_j^{(i)}(x) \varphi(x)\, dx  =  \int_\R  \sum_{i=0}^m c_i \kappa ( x^s)^{(i)} \varphi(x)\, dx .}
From this and equality \eqref{bla10} we finally obtain that 
	\[  0 = \int_{\R}\sum_{i=0}^m c_i \kappa (x^s)^{(i)} \varphi(x)\, dx \] for any $\varphi \in C_c^{\infty}\big((0,\infty)\big)$. 
This implies that on $(0,\infty)$
	\[ 0=  \kappa \sum_{i=0}^m c_i(x^s)^{(i)} =\kappa \sum_{i=0}^m c_i s(s-1)\dots(s-i+1)x^{s-i}.   \]
We divide this relation by $\kappa$ (that is strictly positive), multiply by $x^{m-s}$ and obtain that for any $x \in (0,\infty) $
	\[ \sum_{i=0}^m c_i s(s-1)\dots(s-i+1) x^{m-i}  =0.\]
We have here a polynomial that vanishes for any positive $x$. Thanks to the fact the $s \in(0,1)$ the product $s(s-1) \dots (s-i+1)$ is never zero, therefore one must have $c_i= 0 $ for every $i\in\N_0$. This is a contradiction since the vector $(c_0,\dots,c_m)$ was assumed not null. Hence the vector space $\mathcal V$ exhausts $\R^{m+1}$ and there exists $(v,p) \in \mathcal Z $ such that $\big( v(p), v'(p),\dots, v^{(m)}(p)\big)=(0,0,\dots,1)$. This concludes the proof of Theorem \ref{thm4}.
\end{proof}

\subsection{Proof of the density result}\label{sectthm1}
This subsection is dedicated to the proof of Theorem \ref{thm:thm1}. 
%We translate and rescale the function $v$ as given in Theorem \ref{thm4}. The derivatives of the rescaled function vanish in $0$ until the order $m-1$, and the $\mbox{m}^{th}$ derivative equals $1$. Using a Taylor expansion, we obtain that this rescaled function well approximates the monomial $q(x)=c_m x^m$.

\begin{proof}[Proof of Theorem \ref{thm:thm1}]	
We prove that for any $m\in \N$ and any monomial $q_m(x)=x^m$ there exists a Caputo-stationary function $u$ such that 
	\[ \|u-q_m\|_{C^k\lr{[0,1]}} < \varepsilon.\] 
For an arbitrary $m\in \N$, we take for convenience the monomial \[ q_m(x)=\frac{x^{m} }{m!}. \] \textcolor{black}{Also, we consider $p,R>0$ and the function $v$ as introduced in Theorem \ref{thm4}} and 
%From Theorem \ref{thm4} we know that, for an arbitrary given $m$, there exist $p>0$, $R>0$ and a function $v \in C_{-R}^{1,s} \cap C^{\infty}\big((0,\infty)\big)$ that satisfies conditions \eqref{cc1} and \eqref{cc2}. 
we translate and rescale $v$. Let $\delta $ be a positive quantity (to be taken conveniently small  in the sequel) and let $u$ be the function
		 \[ u(x):= \frac{v(\delta x +p)}{\delta^m} .\]
Since $v\in C_{-R}^{1,s} \cap C^{\infty}\big((0,\infty)\big)$ we have that $u \in  C_{ \frac{-p-R}{\delta}}^{1,s} \cap C^{\infty} \lr{\Big(-\displaystyle \frac{p}{\delta}, \infty\Big)}$ and 
	\bgs{  \Gamma(1-s) D_{ \frac{-p-R}{\delta} }^s u(x) =\al    \int_{  \frac{-p-R}{\delta} }^x u'(t)(x-t)^{-s}\, dt \\
		= \al  \delta^{1-m} \int_{  \frac{-p-R}{\delta} }^x v'(\delta t+ p) (x-t)^{-s}\, dt.}
We change the variable $y=\delta t +p$ and obtain that
	\bgs{ \Gamma(1-s)  D^s_{\frac{-p-R}{\delta} } u(x)  = \al  \delta^{s-m} \int_{-R}^{\delta x +p} v'(y) (\delta x+ p-y)^{-s}\, dy\\=\al 
\Gamma(1-s)  D^s_{-R} v(\delta x +p).}
Let  $a:=\displaystyle \frac{-p-R}{\delta}$. Using the properties \eqref{cc1} of $v$ we obtain that
	\bgs{ D_a ^su(x) =0 \text{ in } \Big(-\frac{p}{\delta}, \infty\Big). }  
With this notation, we have that $u\in C_a^{1,s}$ and since $\displaystyle -\frac{p}{\delta}<0$, that $D_a ^su(x) =0 \text{ in } [0, \infty). $\\ 
Furthermore, from the conditions \eqref{cc2} and the definition of $u$ we get that
	\bgs{ u^{(l)}(0)& = \delta^{l-m}v^{(l)}(p)= 0 & & \text{ for any } \quad  l< m\\
		u^{(m)}(0)& = v^{(m)}(p)= 1.&&}
Let for any $x>-p/\delta$ \[ g(x):= u(x) -q_m(x).\] We have that
	\eqlab{ \label{gr1} g^{(l)}(0)&= 0  &\mbox{ for any } &l\leq m\quad  \mbox{and}\\
		g^{(m+l)} (x)&= u^{(m+l)}(x) & \mbox{ for any } &l\geq 1 .} 
Moreover for $l\geq 1 $ we have that $u^{(m+l)}(x)= \delta^l v^{(m+l)} (\delta x +p)$ and it follows that
	\[| g^{(m+l)}(x) | = \delta^{l} |v^{(m+l)}(\delta x+p) | .\]
Hence for $x \in [0,1]$ we have the bound
	\eqlab{ \label{bg1} | g^{(m+l)}(x) |\leq \delta^l \sup_{y \in [p, p+\delta] } |v^{(m+l)} (y)| = \tilde C \delta^l,} where $\tilde C$ is a positive constant.
We consider the derivative of order $k$ of $g$ and take its Taylor expansion with the Lagrange reminder. \textcolor{black}{Thanks to \eqref{gr1}, for some $c\in (0,x)$ we have that
	\[ g^{(k)}(x)= \sum_{i=\max\{ k,m+1\}} ^{k+m+1} g^{(i)}(0) \frac{x^{i-k}}{(i-k)!} +g^{(m+k+2)} (c)  \frac{x^{m+2}}{(m+2)!} .\] }
Using \eqref{bg1} for any $x\in [0,1]$, eventually renaming the constants we have that 
\[ |  g^{(k)}(x)| \leq C \sum_{i={\max\{1,k-m\}}}^{k+2} \delta^i, \]
therefore for $k\in \N_{0}$
\[ |  g^{(k)}(x)| = |q_m^{(k)}(x) -u^{(k)}(x)|= \mathcal O(\delta) .  \]
If we let $\delta\to 0$ we have that $u^{(k)}$ approximates $q_m^{(k)}$. Finally, for any small $\varepsilon(\delta)>0$
	\[ \|u-q_m\|_{C^k\lr{[0,1]}} < \varepsilon\] and this concludes the proof of Theorem \ref{thm:thm1}.
\end{proof}
%
%\subsection{Some explicit examples}\label{exam}
%In this subsection, we want to give some explicit examples related to some Lemmas that were introduced in this paper.

We give here some explicit examples related to some Lemmas that were introduced in this section.\\
\begin{example}\label{exam}
To give an example of Lemma \ref{lem:int11}, we take $a=0, b=1, s=1/2$ and the function $\varphi(x)=x$ in $[0,1]$ and $\varphi(x)=0$ in $(-\infty,0)$. We built the function $u\in C_0^{1,1/2}$ that satisfies
	\eqlab{\label{es1} D_0^{\frac{1}2} u(x)& =0 &\text{ in } & (1,\infty),\\
			u(x)&=x &\text{ in }  & [0,1],\\
			u(x)&=0&\text{ in }  & (-\infty,0) .}
Let \[ g(x):= -\int_0^1 \frac{\varphi'(t)}{\sqrt{x-t}} \, dt=-\int_0^1 (x-t)^{\frac{1}2}\, dt = 2\sqrt{x-1}-2\sqrt{x}.\] 
According to Lemma \ref{lem:int11} and to Theorem \ref{thm:probc}, the unique solution of the problem \eqref{es1} is given by 
\[ u(x) =u(1)+  \frac{1}{\pi}\int_1^x \frac{g(t)}{\sqrt{x-t}}\, dt,\]
and computing, this gives
\[ u(x)=\frac{2}{\pi} \lr{x\arcsin \frac{1}{\sqrt x}-\sqrt{x-1}}. \]
We depict this function in the following Figure \ref{fign:es1}.
\begin{center}
\begin{figure}[htpb]
	\hspace{0.6cm}
	\begin{minipage}[b]{0.85\linewidth}
	\centering
	\includegraphics[width=0.90\textwidth]{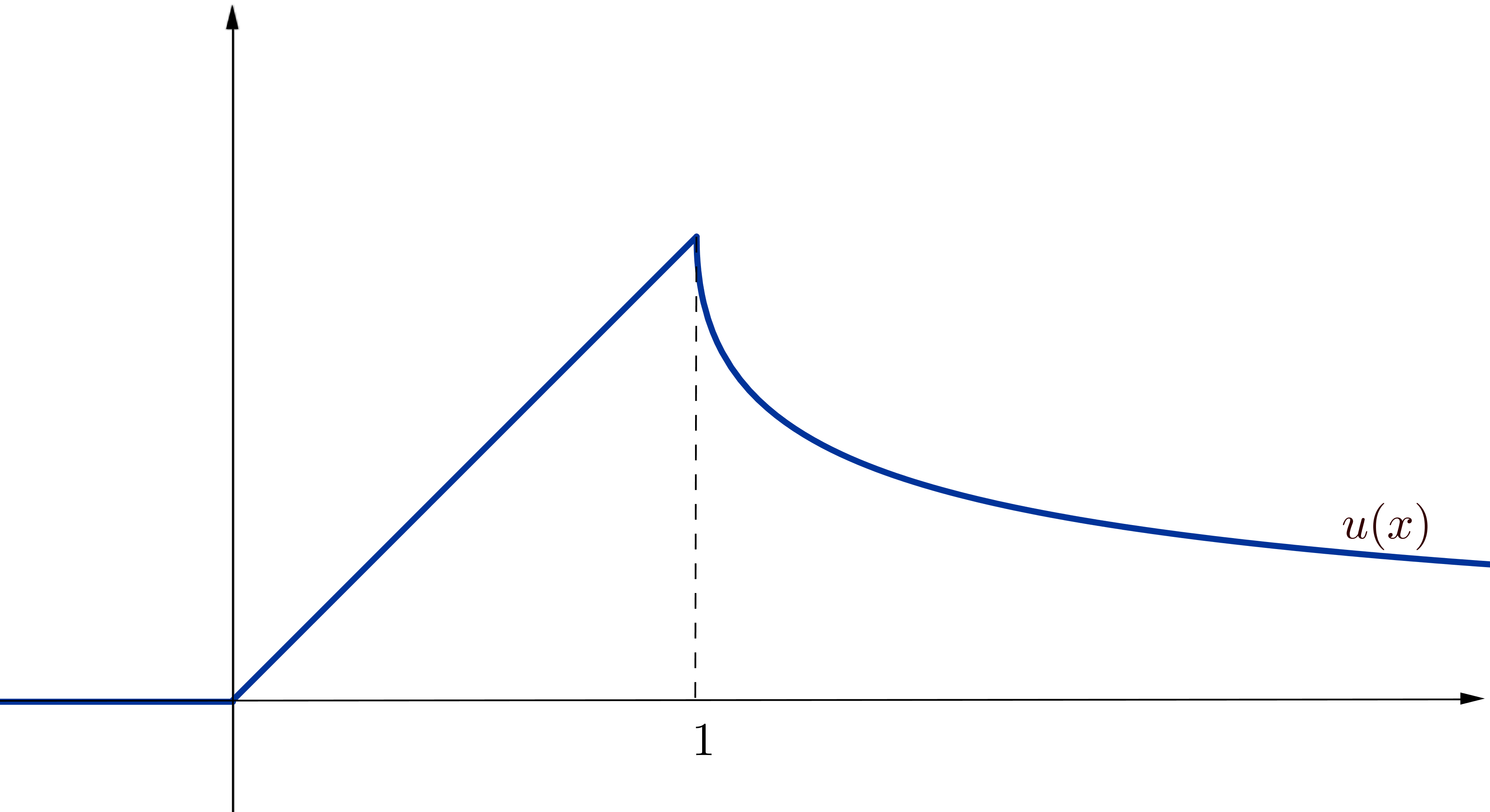}
	\caption{A Caputo-stationary function in $(1,\infty)$ prescribed on $(-\infty,1]$}   
	\label{fign:es1}
	\end{minipage}
\end{figure} 
\end{center}
\end{example}

\begin{example}\label{exam2}
In Lemma \ref{caplem1}, we take $a=0, b=1, s=1/2$ and the quadratic function
\begin{equation*}   \psi_0(x) =\left\{\begin{aligned} &\frac{16}9 \lr{x-\frac{3}4}^2 &\mbox{ in } &\lrq{0,\frac{3}4},\\
                       	& 0  & \mbox{ in } &\lrq{\frac{3}4,1}.\end{aligned}\right.\end{equation*}
                So we are looking for a function $\psi  \in C_0^{1,1/2}$ that satisfies   
                \eqlab{\label{es2} D_0^{\frac{1}2} \psi(x)& =0 &\text{ in } & (1,\infty),\\
			\psi(x)&=\psi_0(x) &\text{ in }   & (-\infty,1] .}
			The solution, according again to Lemma \ref{lem:int11} and to Theorem \ref{thm:probc} is given by
			\[\psi(x) = \frac{1}{\pi}\int_1^{x}g(t)(x-t)^{-\frac{1}2} \, dt, \quad \mbox{ where } \quad g(t)= -\int_0^{\frac{3}4}  \psi_0'(t) (x-t)^{-\frac{1}2}\, dt.\]
Computing this, we have that
\[ g(t)= -\frac{16}{27} \lr{ 8t^{\frac{3}2}-9t^{\frac{1}2} -(4t-3)^{\frac{3}2}}\] 
and  
\bgs{  27 \pi \psi  (x) = &\; 27 \pi + \sqrt{x-1} (-48x+52) +\arcsin \frac{1}{\sqrt{x}} (96x^2-144x)   \\
	&\; -\arcsin \frac{1}{\sqrt{4x-3}} (96x^2-144x+54).}
We depict this function in the following Figure \ref{fign:es2}.	
\begin{center}
\begin{figure}[htpb]
	\hspace{0.6cm}
	\begin{minipage}[b]{0.85\linewidth}
	\centering
	\includegraphics[width=0.90\textwidth]{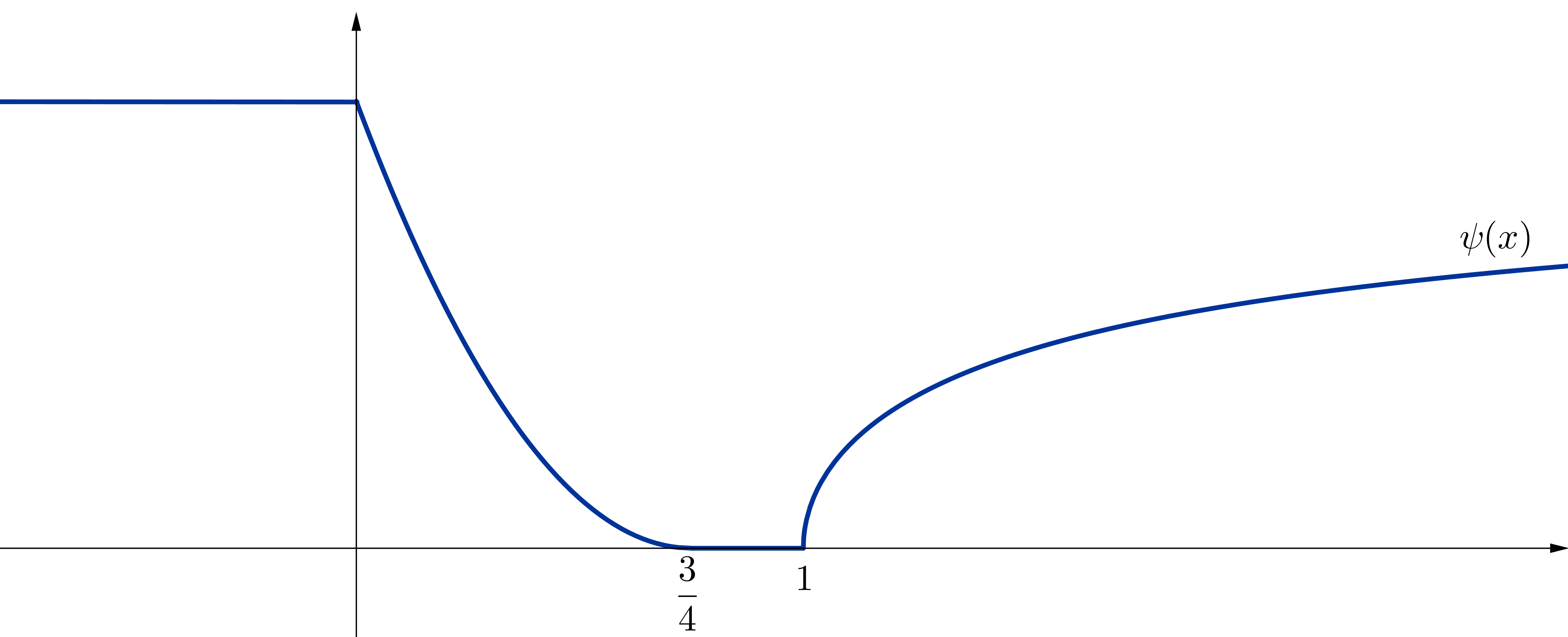}
	\caption{A Caputo-stationary function in $(1,\infty)$ prescribed on $(-\infty,1]$}   
	\label{fign:es2}
	\end{minipage}
\end{figure} 
\end{center}	
\end{example}

\chapter{Potential theory approach to the fractional Laplacian}\label{chap4}
\begin{abstract}
In this chapter, we give a self-contained elementary exposition of the representation formula for the Green function on the ball. In this exposition, only elementary calculus techniques will be used, in particular, no probabilistic methods or computer assisted algebraic manipulations are needed.  The main result of the first section in itself is not new (see for instance \cite{conto,stableprocess}), however we believe that the exposition is original and easy to follow. In the last section of the Chapter we present an elementary approach for the proof of the Schauder estimates for the equation 
$(-\Delta)^s u(x)=f(x)$, with $f$ having a modulus of continuity $\omega_f$. This is 
based on the Poisson representation formula and 
dyadic ball approximation argument. We give the explicit modulus of continuity of 
$u$ in balls $B_r(x)\subset \R^n$ in terms of $\omega_f$.
\end{abstract}

\section{Some observations on the Green function on the ball}
The Green function for the ball in a fractional Laplace framework naturally arises in the study of the representation formulas for the fractional Laplace equations. In particular, in analogy to the classical case of the Laplacian, given an equation with a known forcing term on the ball and vanishing Dirichlet data outside the ball, the representation formula for the solution is precisely the convolution of the Green function with the forcing term. As in the classical case, the Green function is introduced in terms of the Poisson kernel. For this, we will provide both the representation formulas for the problems
	\begin{equation}\label{LaplaceeqD}
		\begin{cases}
			\frlap u= 0   \qquad  &\mbox{ in }  {B_r},    \\
			 u= g \qquad  &\mbox{ in } {\obal}
		\end{cases}
	\end{equation}
and
	\begin{equation}\label{PoissoneqD}
		\begin{cases}
   			\frlap u= g \qquad  &\mbox{ in }    {B_r},    \\
			 u= 0   \qquad  &\mbox{ in }  {\obal}
	\end{cases}
	\end{equation}
in terms of the fractional Poisson kernel and respectively the Green function. Moreover, we will prove an explicit formula for the Green function on the ball.

\bigskip
Here follow some notations and a few preliminary notions.

Check the Appendix \ref{Four} for a brief introduction to the Fourier transform. We add here that $\widehat{f}$ is the Fourier transform of $f$ in a distributional sense, for $f$ that satisfies \[  \int_{\Rn} \frac{|f(x)|}{1+|x|^{p}}\, dx <\infty \quad \mbox{ for some }p \in \mathbb{N}\]  if for any $\varphi \in \mathcal{S}(\Rn)$ we have that
	\begin{equation} \label{dists1} \int_{\Rn} \widehat f(x) \varphi(x) \, dx= \int_{\Rn} f(x) \widehat \varphi (x) \, dx.\end{equation} We remark that the integral notation is used in a formal manner whenever the arguments are not integrable.
	
  We introduce the notion of distributional solution. Following the approach in \cite{Silvestre} (see Definition 2.1.3),  we introduce a suitable functional space where distributional solutions can be defined.
 Let \[ \Sa_s(\Rn) := \Big\{ f \in C^{\infty}(\Rn) \; \big| \; \forall \alpha \in \mathbf{N}^n_0, \; \sup_{x\in \Rn} \big(1+|x|^{n+2s}\big) |D^{\alpha} f(x) | <+\infty \Big\}. \]   The linear space $\Sa_s(\Rn)$ endowed with the family of seminorms
 	\begin{equation*} \label{seminormss1}[f]^{\alpha}_{\Sa_s(\Rn)}:=\sup_{x\in \Rn} \big(1+|x|^{n+2s}\big) |D^{\alpha} f (x)|\end{equation*} is a locally convex topological space. We denote with $\Sa_s'(\Rn)$ the topological dual of $\Sa_s(\Rn)$.
 	
We notice that if $\varphi \in \Sa(\Rn)$ then $\frlap \varphi \in \Sa_s(\Rn)$, which makes this framework appropriate for the distributional formulation. In order to prove this, we observe that for any $x\in \Rn \setminus B_1$ the bound
 	\begin{equation} |\frlap \varphi(x)| \leq c_{n,s} |x|^{-n-2s}\label{frb1} \end{equation}
 	follows from the upcoming computation and the fact that $\varphi \in S(\Rn)$
 	\[\begin{split}
 	  |\frlap& \varphi (x) |\\ \leq \;&  \int_ {B_{\frac{|x|}{2}}} \frac{\big|2 \varphi (x) -\varphi(x-y) - \varphi(x+y) \big| }{|y|^{n+2s}}\, dy
 	 + 2 \int_ {\Rn \setminus B_{\frac{|x|}{2}}} \frac{\big| \varphi (x)-  \varphi(x+y) \big|}{|y|^{n+2s}}\, dy  \\
 	 \leq &\; c_{n,s} |x|^{-n-2s} \bigg( \sup_{z\in \Rn} (1+|z|)^{n+2}|D^2\varphi (z)|   +  \sup_{z\in \Rn} (1+|z|)^{n}|\varphi (z)|
 	 + \|\varphi \|_{L^1(\Rn)}\bigg)  .
 	\end{split}\]
 	Moreover, we observe that, up to constants,
 		\[ \begin{split}   \partial_{x_i} \frlap \varphi(x) &\;=\partial_{x_i} \F^{-1}\Big(|\xi|^{2s}\widehat \varphi(\xi)\Big)(x)  		
 		= \F^{-1}\Big( i\xi_i|\xi|^{2s}\widehat \varphi(\xi) \Big)(x)\\
 		&\;= \F^{-1}\Big(|\xi|^{2s}\widehat{\partial_{x_i} \varphi} (\xi) \Big) (x)= \frlap \partial_{x_i} \varphi(x).
 			\end{split}\]
Hence, by iterating the presented argument, one proves that $\frlap \varphi \in \Sa_s(\Rn)$, which gives the claim. And so:
\begin{defn}
Let $f \in \Sa'(\Rn)$, we say that $u \in \Sa_s'(\Rn)$ is a distributional solution of
		\[ \frlap u=f \; \mbox{in } \; \Rn\] if
		\begin{equation}\label{disf1} \langle u,\frlap \varphi \rangle_s= \int_{\Rn} f(x)\varphi(x) \,dx\quad \mbox{for any} \; \varphi \in \Sa(\Rn), \end{equation}
where $\langle\cdot, \cdot \rangle_s$ denotes the duality pairing of $\Sa_s'(\Rn)$ and $\Sa_s(\Rn)$ and the latter (formal) integral notation designates the pairing $\Sa(\Rn)$ and $\Sa'(\Rn)$.
\end{defn}
\noindent We use the integral notation in \eqref{disf1} in a formal manner whenever the arguments are not integrable. Notice that the inclusion $L_s^1(\Rn) \subset \Sa_s'(\Rn)$ holds, in particular for any $u\in L_s^1(\Rn)$ and $\psi \in \Sa_s(\Rn)$ we have that
	\begin{equation}\begin{split}
	\Big| \langle u ,\psi\rangle_s  \Big| \leq &\;  \int_{\Rn} |u(x)|\, |\psi(x)|\, dx \leq \int_{\Rn} \frac{|u(x)|}{1+|x|^{n+2s}}  (1+|x|^{n+2s})|\psi(x)|\, dx \\
	\leq&\;  [\psi]^0_{\Sa_s(\Rn)} \| u\|_{L_s^1(\Rn)}.
	\label{db1} \end{split}
	\end{equation}

%In order obtain the solution to the problems \eqref{LaplaceeqD} and \eqref{PoissoneqD} we will solve also the equations on the whole space $\Rn$
%\begin{equation}
%   			 \frlap u  = 0  \label{Laplaceeq}
%\end{equation}
%and
%\begin{equation}
%	\frlap u =f,  \label{Poissoneq}
%\end{equation}
%using the $s$-mean kernel and the fundamental solution.
We introduce now the four functions $A_r$, $\Phi$, $P_r$ and $G$, namely the $s$-mean kernel, the fundamental solution, the Poisson kernel and the Green function. The reader can see Section 2.2 in \cite{EVANS} for the theory in the classical case.

\begin{defn}
 Let  $r>0$ be fixed. Then
%  function $A_r$ is defined by
	\begin{equation} \label{smeandefn}
	A_r(y) :=  \begin{cases}
		c(n,s)  \displaystyle \frac{r^{2s}}{(|y|^2-r^2)^s|y|^n} \quad &y \in \Rn \setminus \overline{B_r},\\
		0 \quad &y \in \overline B_r,
		\end{cases}
	\end{equation}
where $ c(n,s)>0$.
\end{defn}

\begin{defn} For any $x\in \Rn \setminus \{0\}$ 
%the function $\Phi$ is defined  by
   	\begin{equation}
		\Phi(x) :=
			\begin{cases}
			 \displaystyle a(n,s){|x|^{-n+2s}} \quad &\text {if } n \neq 2s, \\
			\displaystyle a\Big(1,\frac{1}{2}\Big) \log |x| \quad & \text {if } n = 2s  ,
			\end{cases}
	\label{fundsolution}
	\end{equation}
where $a(n,s)>0$.
\end{defn}

\begin{defn}
Let $r>0$ be fixed. For any $ x \in B_r$ and any $ y \in \Rn \setminus \overline{B}_r$
% the Poisson kernel $P_r$ is defined by
	\begin{equation}
	 P_r(y,x) :=  c(n,s) \Bigg (\frac {r^2-|x|^2}{|y|^2-r^2}\Bigg)^s \frac {1}{|x-y|^n}. \label{poissondefn}
	\end{equation}
\end{defn}
\noindent The Poisson kernel $P_r$ gives a function which is known outside the ball and $s$-harmonic inside (i.e., a solution for the problem \eqref{LaplaceeqD}), by convolution with the known exterior data. Indeed:

 \begin{theorem} \label{theorem:DPL}
 Let $r>0$, $g \in L^1_s(\Rn) \cap C({\Rn})$ and let
\begin{equation}
		 u_g(x) : =
			\begin{cases}	
				\displaystyle  \int_{{\Rn}\setminus B_r} P_r(y,x) g(y)\, dy &\quad  \, \text{if } x\in B_r, \\
				g(x) &\quad \, \text{if } x \in {\obal}.
			\end{cases} \label{solD}
	\end{equation}
Then $u_g$ is the unique pointwise continuous solution of the problem \eqref{LaplaceeqD}
	\begin{equation*}
	\begin{cases}
	\frlap u= 0 \qquad  &\mbox{ in }  {B_r},
\\	u= g \qquad  &\mbox{ in }  {\obal}.
		\end{cases}
	\end{equation*}
\end{theorem}

\begin{defn} Let $r>0$ be fixed. For any $x, z \in B_r$ and $x\neq z$,
\begin{equation} G(x,z) := \Phi(x-z) -\int_{\obal} \Phi(z-y) P_r(y,x) \, dy . \label{greendefn}\end{equation}
\end{defn}
\noindent A formula for the Green function $G$ that is more suitable for applications is introduced in the following result. 
% Indeed, one of the main goals is to prove this simpler, explicit formula for the Green function on the ball, by means of elementary calculus techniques. At this purpose, Theorem \ref{theorem:thm1} establishes a symmetrical expression for~$G$.
\begin{theorem} \label{theorem:thm1} 
Let $r>0$ be fixed and let $G$ be the function defined in \eqref{greendefn}. Then if $n \neq 2s$\begin{equation} \label{forgkns}   G(x,z) =  \kappa(n,s) |z-x|^{2s-n}  \int_0^{r_0(x,z)}  \frac{t^{s-1}} {(t+1)^\frac{n}{2}} \, dt ,\end{equation}
where
 	\begin{equation} \displaystyle r_0(x,z) = \frac{(r^2-|x|^2)(r^2-|z|^2)}{r^2|x-z|^2}\label{ro} \end{equation}
and $\kappa(n,s)>0$.\\
For $n=2s$, the following holds
\begin{equation} G(x,z)= \kappa\Big(1,\frac{1}{2}\Big) \log\bigg( \frac{r^2-xz+\sqrt{(r^2-x^2)(r^2-z^2)}}{r|z-x|}\bigg).\label{formn1s12} \end{equation}
\end{theorem}
This result is not new (see \cite{conto,stableprocess}), however, the proof we provide uses only calculus techniques, therefore we hope it will be accessible to a wide audience. It makes elementary use of special functions like the Euler-Gamma function, the Beta and the hypergeometric function, that are introduced in the Appendix \ref{special} (see also references therein).
% Moreover, the point inversion transformations and some basic calculus facts, that are also outlined in the Appendix, are used in the course of this proof.

The main property of the Green function is stated in the next theorem, as it gives the solution of an equation with a known forcing term in a ball and vanishing Dirichlet data outside the ball, by convolution with the forcing term. While this convolution property in itself may be easily guessed from the superposition effect induced by the linear character of the equation, the main property that we point out is that the convolution kernel is explicitly given by the function $G$.

\begin{theorem}\label{theorem:thm2}
Let $r>0$, $h \in {C^{0,\eee}(  B_r)  \cap C(\overline B_r)}$ and let
	 \begin{equation*}
		u(x) : =
			\begin{cases}
		\displaystyle \int_{B_r} h(y) G(x,y) \, dy \quad & \text{ if } x\in B_r, \\
		0 \quad \quad & \text{ if } x \in {\obal}.
		\end{cases}
	\end{equation*}
Then $u$ is the unique pointwise continuous  solution of the problem \eqref{PoissoneqD}
	\begin{equation*}
		\begin{cases}
		    \frlap u= h   \qquad &\mbox{ in } {B_r} ,
 		\\  u= 0   \qquad &\mbox{ in } {\obal}.
	\end{cases}
	\end{equation*}
\end{theorem}	
\noindent The proof is classical, and makes use of the properties and representation formulas involving the two functions $\Phi$ and $P_r$.

We are also interested in the values of the normalization constants that appear in the definitions of the $s$-mean kernel (and the Poisson kernel) and of the fundamental solution. {We will deal separately with the two cases $n\neq2s$ and $n=2s$. We have the following definition}:

\begin{defn}
\label{definition:ctn} The constant $a(n,s)$ introduced in definition \eqref{fundsolution} is
\begin{align}
        a(n,s) : &= \displaystyle { \frac{ \Gamma (\frac{n}{2}-s)} { 2^{2s}\pi ^ {\frac{n}{2} }  \Gamma(s)}} &  \text{ for } &n\neq 2s \label{ctans1},\\
	a\Big(1,\frac{1}{2}\Big) :&= { -\frac{1}{\pi}} & \text{ for } &n=2s. \label{ctans3}
	\end{align}

	The constant $c(n,s)$ introduced in definition \eqref{smeandefn} is
\begin{equation}
	c(n,s) :=\frac {\Gamma (\frac{n}{2}) \sin \pi s }  {\pi^{\frac{n}{2}+1} }.\label{ctcns}
\end{equation}

\end{defn}

\noindent These constants are used for normalization purposes, and we explicitly clarify how their values arise. However, these values are only needed to compute the constant $\kappa(n,s)$ from Theorem \ref{theorem:thm1}, and have no role for the rest of our discussion. Indeed, we explicitly compute:

\begin{theorem}\label{theorem:kns} The constant $\kappa(n,s)$ introduced in identity \eqref{forgkns} is
	\begin{equation*} \begin{aligned}
	 \kappa(n,s) &= \displaystyle {\frac{\Gamma(\frac{n}{2}) } {2^{2s}\pi^{ \frac{n}2}   \Gamma^2(s) } }& \text{ for } &n \neq 2s, \\
		\kappa\Big(1,\frac{1}{2}\Big) &={\frac{1}{\pi}}
		  & \text{ for } &n=2s.
		  \end{aligned}
		  \end{equation*}
\end{theorem}

\noindent One interesting thing that we want to point out here is related to the two constants $C(n,s)$ and $c(n,s)$. The constant $C(n,s)$ is given in \cite{galattica} in the definition of the fractional Laplacian, is consistent with the Fourier expression of the fractional Laplacian, and was explicitly computed in \eqref{EXPL}. The constant $c(n,s)$ is introduced in \cite{Landkof} in the definition of the $s$-mean kernel and the Poisson-kernel, and is here given in \eqref{ctcns}. It is used to normalize the Poisson kernel (and the $s$-mean kernel), and is consistent with the constants used for the fundamental solution and the Green function. Hence, the two constants are used for different normalization purposes, and they have similar asymptotic properties. 
In the following proposition we give another proof of the explicit expression obtained in \eqref{EXPL}.

\begin{theorem}\label{thm:Cc}
The constant $C(n,s)$ is given by
	\begin{equation} \label{cnscomputed} C(n,s) = \frac{2^{2s} s \Gamma\left(\frac{n}2 +s\right)} {\pi^{\frac{n}2} \Gamma(1-s)}.\end{equation}
\end{theorem}
	
This section is structured as follows: in Subsection \ref{smean} we define the $s$-mean value property by means of the $s$-mean kernel and prove that if a function has the $s$-mean value property then it is $s$-harmonic. Subsection \ref{fundsol} deals with the study of the function $\Phi$ as the fundamental solution of the fractional Laplacian.
%By means of the fundamental solution, we obtain the representation formula for the equation \eqref{Poissoneq}.
The fractional Poisson kernel is introduced in Subsection \ref{thepoissonkernel}, and the representation formula for equation \eqref{LaplaceeqD} is obtained. Subsection \ref{green} focuses on the Green function, and there we prove Theorems \ref{theorem:thm1} and \ref{theorem:thm2}.
%%%
%presenting two main theorems: Theorem \ref{theorem:thm1} gives a more basic formula of the function $G$ for the ball and Theorem \ref{theorem:thm2} illustrates how the solution to equation \eqref{PoissoneqD} is built by means of the function $G$.  
The computation of the normalization constants introduced at the beginning of this section is done at the end of Subsection \ref{green}.
In Subsection \ref{appy} we recall the point inversion transformations and present some calculus identities that we use in this section.

%\subsection[Preliminaries]{Preliminaries} \label{preliminaries}
%In this section, we deal with the $s$-mean kernel, the fundamental solution and the Poisson kernel.
%%We introduce the representation formulas for equations \eqref{Laplaceeq}, \eqref{Poissoneq} and \eqref{LaplaceeqD}.
%
%Let $s\in (0,1)$.

Throughout this section, we fix the fractional parameter $s\in (0,1)$.

\subsection{The $s$-mean value property}\label{smean}
We give here some properties of the $s$-mean kernel.  The $s$-mean value property of the function $u$ is an average property defined by convolution of $u$ with the $s$-mean kernel. We recall the definition \eqref{ls1} of the weighted $L^1$ space.

\begin{defn}[$s$-mean value property]
Let $x\in \Rn$. We say that $u\in L_s^1(\Rn)$, continuous in a neighborhood of $x$, has the $s$-mean value property at $ x$ if, for any $r>0$ arbitrarily small,
	\begin{equation}  u(x)= A_r* u (x) .\label{smvp}\end{equation}
We say that $u$ has the $s$-mean value property in $\Omega \subseteq {\Rn}$ if for any $r>0$ arbitrarily small, identity \eqref{smvp} is satisfied at any point $x \in \Omega$.
\end{defn}

The above definition makes it reasonable to say that $A_r$ plays the role of the $s$-mean kernel. The main result that we state here is that if a function has the $s$-mean value property, then it is $s$-harmonic  (i.e. it satisfies the classical relation $\frlap u =0$).

\begin{theorem}
\label{theorem:arm}
Let $u \in L_s^1(\Rn)$ be $C^{2s+\eee}$ in a neighborhood of $x\in \Rn$.
If $u$ has the $s$-mean value property at $x$, then $u$ is $s$-harmonic at $x$.
\end{theorem}

\begin{proof}[Proof]
The function $u$ has the $s$-mean value property for any $r>0$ arbitrarily small, namely
	 \begin{equation*}
           u(x)=  A_r* u (x) =  \int_{  \obal} A_r(y)u(x-y) \, dy.
	\end{equation*}
Using identity \eqref{Ir} we obtain that
	\begin{equation*}
		\begin{split}
	0= u(x)  - \int_{\obal}   A_r(y)  u(x-y) \,  dy  = c(n,s) r^{2s}  \int_{\obal}   \frac{u(x) -u(x-y)}{{(|y|^2-r^2)}^s|y|^n}  \, dy,
		\end{split}
	\end{equation*}
thus, since $r>0$
 	\begin{equation}  \int_{\obal}  \frac {u(x)-u(x-y)  }{(|y|^2-r^2)^s |y|^n}\,  dy =0.\label{zeroeq} \end{equation}
 Hence, in order to obtain $\frlap u(x)=0$ we prove that
 	\begin{equation} \lim_{r \to 0} \int_{\Rn\setminus B_r} \frac{u(x)-u(x-y)}{|y|^{n+2s}} \, dy = \lim_{r \to 0} \int_{\obal}  \frac {u(x)-u(x-y)  }{(|y|^2-r^2)^s |y|^n}\,  dy \label{clsmfrl}. \end{equation}
Let $R>r\sqrt{2}$. We write the integral in \eqref{zeroeq} as
	\begin{equation}\label{eq1112}
		\begin{split}
		 \int_{\obal}   &\frac{u(x)-u(x-y)}{(|y|^2-r^2)^s |y|^n}\, dy \\ = \; &\int_{{\Rn}\setminus B_{R}}  \frac { u(x)-u(x-y) }{(|y|^2-r^2)^s |y|^n}\, dy
	 + \int_{B_{R}\setminus B_r} \frac { u(x)-u(x-y) }{(|y|^2-r^2)^s |y|^n}\, dy \\
	=\;  &I_1(r,R)+I_2(r,R).
		\end{split}
	\end{equation}
In $I_1(r,R)$ we see that $ \frac{|y|^2 }{|y|^2-r^2} <2$ and obtain 
	\[ \frac{|u(x)-u(x-y)|} {(|y|^2-r^2)^s |y|^n}   \leq 2^{s}\frac{ |u(x)-u(x-y) |} {|y|^{n+2s}}  \in L^1(\Rn \setminus B_R, \,dy),\]
	%=  \frac{|u(x)-u(x-y)|}{|y|^{n+2s}} \frac{|y|^{2s}} {(|y|^2-r^2)^s}
	as  $u\in L_s^1(\Rn)$. We can use the Dominated Convergence Theorem, send $r \to 0$  and conclude that \begin{equation}\label{eq1111} \lim_{r\to 0} I_1(r,R)=\int_{\Rn\setminus B_{R}} \frac{u(x)-u(x-y)}{|y|^{n+2s}} \, dy .\end{equation}
Now, for $r<|y|<R$ and $u\in C^{2s+\eee}$ (for $s<1/2$) in a neighborhood of $x$  we have the bound
	\[ \begin{split}
		\big| u(x)-u(x-y) \big|\leq c |y|^{2s+\eee},  \end{split}\]
while for $s\geq1/2$ and $u \in C^{1,2s+\eee-1}$ we use that
	\[\begin{split}
		 |u(x) -u(x-y) - y \cdot \nabla u(x)|\; =&\; \Big| \int_0^1 y \big(\nabla u(x-ty) -\nabla u(x) \big) \, dt\Big|\\
		 	 \leq &\; |y| \int_0^1 \Big| \nabla u(x-ty) -\nabla u(x) \Big| \,  \, dt
		 	 \leq c(s,\eee) |y|^{2s+\eee}. \end{split} \]
Notice that $ \frac{ y \cdot \nabla u(x)}{(|y|^2-r^2)^s |y|^n} $ and $\frac{ y \cdot \nabla u(x)}{|y|^{2s+n}} $ are even functions, hence they vanish when integrated on the symmetrical domain $B_{R} \setminus B_r$. Therefore, by setting
	\begin{equation}\label{eq1113}  J(r,R):=I_2(r,R) - \int_{B_R \setminus B_r} \frac{u(x)-u(x-y) } {|y|^{2s+n}}  \,dy\end{equation}
we have that
	\[ \begin{split}    J(r,R) = \int_{B_R \setminus B_r} \Bigg( \frac{u(x)\!-\!u(x\!-\!y)\!-\!y \cdot \nabla u(x)}{(|y|^2\!-\!r^2)^s |y|^n}  -\frac{u(x)\!-\!u(x\!-\!y) \!-\!y \cdot \nabla u(x)} {|y|^{2s+n}} \Bigg)\,dy
	\end{split}\]
and by passing to polar coordinates and afterwards making the change of variables $\rho =rt$ we get
\[ \begin{split}
		|J(r,R)| \leq &\; c (s,\eee) \int_{B_R \setminus B_r} |y|^{2s+\eee}  \Big((|y|^2-r^2)^{-s} |y|^{-n} - |y|^{-n-2s}\Big)\, dy \\
				= &\; {c}(n,s,\eee) \int_r^R \rho ^{\eee-1} \bigg( \frac{\rho^{2s}}{(\rho^2-r^2)^{s}} -1\bigg) \, d\rho 
%			= &\; {c}(s,\eee) r^{\eee} \int_1^{\frac{R}{r}} t^{\eee-1}  \bigg( \frac{t^{2s}}{(t^2-1)^s} -1\bigg) \, dt\\
				< 	 {c}(n,s,\eee) r^{\eee} \int_1^{\frac{R}{r}} t^{\eee-1}  \bigg( \frac{t^{s}}{(t-1)^s} -1\bigg) \, dt	
\end{split}\]
since $t/(t+1)>1$. Now for $ t \in (1, \sqrt 2)$ we have that
	\[ \begin{split}   \int_1^{\sqrt 2} t^{\eee-1} \bigg( \frac{t^s}{(t-1)^s} -1\bigg) \, dt  \leq c(s) \int_1^{\sqrt 2} \bigg( (t-1)^{-s} -t^{-s}\bigg) \, dt = \tilde c(s).
	\end{split} \]
On the other hand, for $t\geq\sqrt 2$
\[\Big(1- \frac{1}{t}\Big)^{-s} - 1 \leq \frac{s}{t} \Big(1-\frac{1}{\sqrt 2} \Big)^{-s-1} \] and we have that
	 \[ \begin{split}   \lim_{r \to 0} \int_{\sqrt 2}^{\frac{R}{r}} t^{\eee-1} \bigg( \frac{t^s}{(t-1)^s} -1\bigg) \, dt   \leq &\;  \int_{\sqrt 2}^{\infty}  t^{\eee-1}  \Bigg( \Big( 1- \frac{1}{t}\Big)^{-s} -1 \Bigg) \, dt\\
\leq &\;c(s) \int_{\sqrt 2} ^{\infty} t^{\eee -2 }\, dt = \bar c(s,\eee).
	\end{split} \]
Thus by sending $r \to 0 $ we obtain that
	\[ \begin{split}   \lim_{ r \to 0} J(r,R) =0
	\end{split} \]
and therefore in \eqref{eq1113} \[ \lim_{r \to 0}I_2(r,R) =  \lim_{r \to 0} \int_{B_R \setminus B_r} \frac{u(x)-u(x-y) } {|y|^{2s+n}}  \,dy.\]
Using this and \eqref{eq1111} and passing to the limit in \eqref{eq1112}, claim \eqref{clsmfrl} follows and hence  the conclusion that $\frlap u(x)=0$.
\end{proof}

\subsection{The fundamental solution}\label{fundsol}

We claim that the function $\Phi$ plays the role of the fundamental solution of the fractional Laplacian, namely the fractional Laplacian of $\Phi$ is equal in the distributional sense to the Dirac Delta function evaluated at zero. The following theorem provides the motivation for this claim.

\begin{theorem}
\label{theorem:thm3} In the distributional sense (given by definition \eqref{disf1})
	 \[\frlap \Phi =\delta_0 .\]
\end{theorem}

The computation of the Fourier transform of the fundamental solution is required in order to prove Theorem \ref{theorem:thm3}.
\begin{prop}
\label{proposition:ansss}
a) For $n>2s$, let $f \in L^1(\Rn)\cap C(\Rn)$ with $\wck f \in \Sa_s(\Rn)$,	\\
b) for $n\leq 2s$, let $f\in L^1(\R)\cap C(\R) \cap C^1\big( (-\infty, 0) \cup (0,+\infty)\big)$ with $\wck f\in \Sa_s(\R)$ such that
	 \begin{equation}\label{condffff}\begin{aligned}
	 &|f(x)|\leq c_1|x|^{2s} \quad &\text{ for } & x \in \R\\
	   & |f(x)|\leq \frac{c_2}{|x|} \quad &\text{ for } &|x|>1 \\
	  	 &|f'(x)|\leq c'_1|x|^{2s-1} \quad &\text{ for } &0 <|x|\leq 1\\
	  &  |f'(x)|\leq \frac{c'_2}{|x|} \quad &\text{ for } &|x|>1. \end{aligned} \end{equation}
 Then in both cases
\[\int_{\Rn} \Phi(x) \wck f(x) \, dx =\int_{\Rn} ({2\pi} |x|)^{-2s} f(x) \, dx.\]
\end{prop}

\begin{proof}
We notice that the hypothesis in the proposition % of Proposition \ref{proposition:ansss} \eqref{condffff} 
assure that both integrals above are well defined. Indeed, since $\Phi \in L_s^1(\Rn) \subset \Sa'_s(\Rn)$ the left hand side is finite thanks to \eqref{db1}. The right hand side is also finite since, for $n>2s$,
	\[ \begin{split} \int_{\Rn} |f(x)||x|^{-2s} \, dx  \leq &\; c_n\sup_{x\in B_1} |f(x)| \int_0^1 \rho^{n-2s-1}\, d\rho + \int_{\Rn \setminus B_1}  |f(x)| |x|^{-2s}\, dx\\
		\leq &\; c_n\sup_{x \in B_1} |f(x)| + \|f\|_{L^1(\Rn)} \end{split}\]
		and for $n\leq 2s$  we have that
	\[ \begin{split}\int_{\R} |f(x)||x|^{-2s} \, dx  \leq &\; \int_{\R\setminus B_1} |f(x)| |x|^{-2s} \, dx+ c_1 \int_{B_1}  dx \\
							\leq &\; \|f\|_{L^1(\R)} + 2c_1.\end{split}\]

\noindent a) For $n>2s$ we prove that
	\begin{equation} \label{firstcl111} a(n,s)\int_{\Rn} |x|^{-n+2s} \wck {f}(x)\, dx=\int_{\Rn} ({2\pi}|x|)^{-2s} f(x)\, dx.\end{equation}
We use the Fourier transform of the Gaussian distribution as the starting point of the proof. For any $ \delta > 0$ we have that
	\[    \mathcal{F} (e^{-\pi \delta |x|^2})  = \delta^{- \frac{n}{2}} e^{-\pi\frac{|x|^2}{\delta}}.\]
In particular for any $ f \in L^1(\Rn)$ and $\wck {f}\in \Sa_s(\Rn)$ (which is a subspace of $L^2(\Rn)$), by Parseval identity we obtain
	\[ \int_{\Rn}  e^{-\pi \delta |x|^2} \wck {f}(x) \, dx= \int_{\Rn} \delta^{-\frac{n}{2}} e^{-\pi\frac{|x|^2}{\delta}}  f(x)\, dx. \]
We multiply both sides by $\delta^{\frac{n}{2} -s-1}$, integrate in $\delta$ from $0$ to $\infty$. We use the notations
	\[I_1=\int_0^{\infty} \Bigg( \int_{\Rn} \delta^{\frac{n}{2} - s -1} e^{-\pi\delta |x|^2}\wck {f} (x)  \, dx\Bigg) \, d\delta \]
and
	\[I_2= \int_0^{\infty} \Bigg( \int_{\Rn} \delta^{- s -1} e^{-\pi \frac {|x|^2}{\delta}}   f(x) \, dx\Bigg) \, d\delta,\] having $I_1=I_2$.
With the change of variable $\alpha =\delta|x|^2$ we obtain that
	\[I_1 =  \int_{\Rn}|x|^{-n+2s} \wck{f} (x)  \Bigg(\int_0^{\infty} \alpha^{\frac{n}{2} -s-1}  e^{-\pi\alpha}\, d\alpha \Bigg)\, dx. \]
We set	\begin{equation} \label{c1} c_1 := \int_0^{\infty} \alpha^{\frac{n}{2} -s-1}  e^{-\pi\alpha}\, d\alpha, \end{equation} which is a finite quantity since $ \frac{n}{2} -s-1 > -1$.
%, the integral in $\alpha$ is finite and write
%	\[ I_1= c_1  \int_{\Rn} |x|^{-n+2s} \wck{f}(x)  \, dx.\]
On the other hand in $I_2$ we change the variable $ \displaystyle \alpha = {|x|^2}/{\delta} $ and obtain that
	\[I_2 =  \int_{\Rn}   |x|^{-2s}  f(x) \Bigg(\int_0^{\infty} \alpha^{s-1}  e^{-\pi\alpha}\, d\alpha \Bigg)\, dx. \]
We then set
	\begin{equation} c_2 :=  \int_0^{\infty} \alpha^{s-1}  e^{-\pi\alpha}\, d\alpha\label{c2},  \end{equation} which is finite since $ s-1 > -1$.
%	 Hence we write that
%	\[ I_2= c_2  \int_{\Rn}  |x|^{-2s} f(x)  \, dx.\]
As $I_1=I_2$ it yields that
	\begin{equation*}  \frac{c_1}{c_2 ({2\pi})^{2s}} \int_{\Rn} |x|^{-n+2s}  \wck{f} (x)\, dx =  \int_{\Rn} ({2\pi}|x|)^{-2s}   f(x)\, dx . \end{equation*}
We take
	\begin{equation}\label{aaargh1} a(n,s) =  \frac{c_1}{c_2 ({2\pi})^{2s}}\end{equation}and the claim \eqref{firstcl111} follows. This concludes the proof for $n>2s$.

b) For $n<2s$ (hence $n=1$ and $s>1/2$), let $R>0$ be as large as we wish (we will make $R$ go to $\infty$ in sequel). Then
	\[ \begin{split}
		\int_{B_R}&\; |x|^{2s-1} \wck f (x) \, dx  = \int_0^R x^{2s-1} \Big( \wck f(x) + \wck f(-x)\Big) \, dx \\
		= &\;2\int_0^R x^{2s-1} \int_{\R} f(\xi) \cos(2\pi  \xi x) \, d\xi \, dx 
		= 2 \int_{\R} f(\xi) \left( \int_0^R x^{2s-1} \cos(2\pi \xi x) \, dx\right) \, d\xi.\end{split} \]
{We use the change of variables $\bar x = 2 \pi x$ (but still write $x$ as the variable of integration for simplicity), and let $\bar R = 2 \pi R$}. Then
 \[\int_0^R x^{2s-1} \cos({2\pi} \xi x) \, dx = {(2\pi)^{-2s}} \int_0^{\bar R} x^{2s-1} \cos(\xi x) \, dx.\]
 We have that
 	\[ 	\int_{B_R} |x|^{2s-1} \wck f (x) \, dx  ={\frac{2^{1-2s}}{\pi^{2s}} }\int_{\R} f(\xi) \left( \int_0^{\bar R} x^{2s-1} \cos(\xi x)\, dx  \right)\, d\xi.\]
Integrating by parts and changing the variable $|\xi| x=t$ we get that
	\[ \begin{split}
				\int_0^{\bar R} x^{2s-1} \cos(\xi x) \, dx = &\;x^{2s-1} \frac{\sin(\xi x)}{\xi} \bigg|_0^{\bar R} -(2s-1) \int_0^{\bar R} x^{2s-2} \frac{\sin(\xi x)}{\xi} \, dx \\
				=&\; {\bar R}^{2s-1} \frac{\sin(\xi {\bar R})}{\xi} -(2s-1)|\xi|^{-2s} \int_0^{{\bar R}|\xi|} t^{2s-2} \sin t \, dt.
	\end{split} \]
	Therefore
	\begin{equation} \begin{split}\label{wckn1}
		\int_{B_R} |x|^{2s-1} \wck f(x)\, dx = &\;\frac{2^{1-2s}}{\pi^{2s} } {\bar R}^{2s-1} \int_{\R} f(\xi)\frac{\sin(\xi {\bar R})}{\xi} \, d\xi \\
			&\;- \frac{2^{1-2s}(2s-1) }{\pi^{2s} }  \int_{\R} f(\xi) |\xi |^{-2s} \bigg( \int_0^{{\bar R}|\xi|} t^{2s-2} \sin t \, dt \bigg)\, d\xi. \end{split}\end{equation}
We claim that
	\begin{equation}\label{n1c1} \lim_{R \to \infty} {\bar R}^{2s-1}\int_{\R} f(\xi)\frac{\sin(\xi {\bar R})}{\xi} \, d\xi=0 . \end{equation}
	We integrate by parts and obtain that
	\[ \begin{split}
		\bigg|\int_0^{\infty} f(\xi) \frac{\sin(\xi {\bar R})}{\xi} \, d\xi \bigg| \leq &\;  \frac{|f(\xi)|}{\xi} \frac{|\cos(\xi {\bar R})|}{{\bar R}} \bigg|_{0}^{\infty} + \frac{1}{{\bar R}}   \bigg( \int_{0}^{\infty}  |\cos(\xi {\bar R})| \frac{|f(\xi)|}{\xi^2} \, d\xi   \\
		&\; +\int_{0}^{\infty} |\cos(\xi {\bar R})| \frac{|f'(\xi)|}{\xi} \, d\xi\bigg).
	\end{split}\]
By \eqref{condffff}, for $\xi $ large we have that
	 \[ \frac{|f(\xi)|}{\xi}|\cos(\xi {\bar R})|  \leq c_2\frac{|\cos(\xi {\bar R})| }{\xi^2}, \quad \mbox{hence} \quad \lim_{\xi \to  \infty}  \frac{|f(\xi)|}{\xi} \frac{|\cos(\xi {\bar R})|}{{\bar R}} =0.\]
	  For $\xi $ small we have that \[ \frac{|f(\xi)|}{\xi} \leq c _1\xi^{2s-1},\quad \mbox{hence} \quad  \lim_{\xi \to 0} \frac{|f(\xi)|}{\xi} \frac{|\cos(\xi {\bar R})|}{{\bar R}} =0.\]
Furthermore, by changing the variable $t=\xi {\bar R}$ (and noticing that the constants may change value from line to line) we have that
\[\begin{split}
	\int_0^{\infty} \frac{|f(\xi)|}{\xi^2} |\cos (\xi {\bar R}) |\, d\xi \leq &\;c_1 \int_0^1   \xi^{2s-2}  |\cos (\xi {\bar R}) | \, d\xi  +c_2  \int_1^{\infty}  \xi^{-3}  |\cos (\xi {\bar R}) |\, d\xi \\
	\leq &\;c_1 {\bar R}^{1-2s} \int_0^{\bar R} t^{2s-2}  |\cos t |\, dt +c_2 {\bar R}^2  \int_{\bar R}^{\infty} t^{-3}|\cos t| dt
	\leq   \frac{c}{2}
	\end{split}\]
and \[\begin{split}
	\int_0^{\infty} \frac{|f'(\xi)|}{\xi} |\cos (\xi {\bar R}) |\, d\xi \leq &\;c_1' \int_0^1 	  \xi^{2s-2} |\cos (\xi {\bar R}) |  \, d\xi  +c_2' \int_1^{\infty}  \xi^{-2} |\cos (\xi {\bar R}) |  \, d\xi \\
	\leq &\; c_1' {\bar R}^{1-2s} \int_0^{\bar R} t^{2s-2}|\cos t| \, dt +c_2'{\bar R} \int_{\bar R}^{\infty}t^{-2} |\cos t|\, dt
	 \leq  \frac{c}{2}.
	\end{split}\]
Hence \[\bigg| \int_0^\infty  f(\xi) \frac{\sin(\xi{\bar R})}{\xi} \, d\xi \bigg| \leq \frac{c}{R},\]
and in the same way we obtain
	\[\bigg|\int_{-\infty}^0  f(\xi) \frac{\sin(\xi {\bar R})}{\xi} \, d\xi  \bigg|= \bigg|\int_0^\infty  f(-\xi) \frac{\sin(\xi {\bar R})}{\xi} \, d\xi  \bigg| \leq \frac{c}{R}.\] Therefore
	\[ \lim_{R\to \infty} {\bar R}^{2s-1} \int_{\R}  f(\xi) \frac{\sin(\xi {\bar R})}{\xi} \, d\xi   = 0 \] and we have proved the claim \eqref{n1c1}.
	Now we claim that (and this holds also for $n= 2s$)
\begin{equation}\label{n1c2}   \lim_{R \to \infty} \int_\R f(\xi)|\xi|^{-2s}  \Big(\int_0^{{\bar R}|\xi|} t^{2s-2} \sin t \, dt\Big) \, d\xi = - \cos (\pi s) \Gamma(2s-1) \int_\R f(\xi)|\xi|^{-2s}   \, d\xi .\end{equation} In order to prove this, we estimate the difference
	\[ \begin{split}
		\bigg| \int_0^\infty t^{2s-2} \sin t \, dt &\; -\int_0^{{\bar R}|\xi|} t^{2s-2} \sin t \, dt \bigg| \leq  \bigg| \int_{{\bar R}|\xi|}^\infty t^{2s-2} \sin t \, dt\bigg|\\
		\leq &\; |t ^{2s-2} \cos t |\bigg|_{{\bar R}|\xi|}^\infty + (2s-2) \int_{{\bar R}|\xi|}^\infty |t|^{2s-3} | \cos t | \, dt
		\leq  c({\bar R}|\xi|)^{2s-2}.
	\end{split}\]
We then have that
	\[ \begin{split}
	\bigg|	\int_\R & f(\xi)|\xi|^{-2s}  \Big( \int_0^\infty t^{2s-2} \sin t \, dt  -\int_0^{{\bar R}|\xi|} t^{2s-2} \sin t \, dt \Big)  \, d\xi \bigg| \\
	\leq&\; c {\bar R}^{2s-2} \int_\R |f(\xi)| |\xi|^{-2}\, d\xi 
	\leq  c {\bar R}^{2s-2}\bigg( c_1\int_0^1  \xi^{2s-2} \, d\xi + c_2\int_1^\infty \xi^{-3}\, d\xi \bigg)
	= R^{2s-2} \overline c. \end{split}\]
	Hence we obtain
		\[  \lim_{R \to \infty} \int_\R f(\xi)|\xi|^{-2s}  \Big(\int_0^{{\bar R}|\xi|}t^{2s-2} \sin t \, dt\Big) \, d\xi = \int_\R f(\xi)|\xi|^{-2s} \left(\int_0^{\infty}t^{2s-2} \sin t \, dt \right)  \, d\xi \]
and the claim \eqref{n1c2} follows from the identity \eqref{ctcomp1111} at the end of this Section (in Subsection \ref{appy}). \\
By sending $R$ to infinity in \eqref{wckn1} we finally obtain  that
		\begin{equation} \begin{split} \label{oba1s}
			\int_{\R} |x|^{2s-1} \wck f(x)\, dx =&\; \frac{2^{1-2s}}{\pi^{2s}} (2s-1) \cos(\pi s) \Gamma(2s-1) \int_\R |\xi|^{-2s} f(\xi)\, d\xi \\
				=&\;   2\cos(\pi s) \Gamma(2s) \int_\R (2\pi|\xi|)^{-2s} f(\xi)\, d\xi.\end{split}\end{equation}
Therefore {taking $ a(1,s)  =({2\cos(\pi s)\Gamma(2s)})^{-1}$ we get that}
	\[ a(1,s) \int_{\R} |x|^{2s-1} \wck f(x)\, dx = \int_\R ({2\pi} |x|)^{-2s} f(x)\, dx ,\]
%	It follows that
%	\begin{equation*} \begin{split} %\label{oba1s}
%			\int_{\R} \Phi(x) \wck f(x)\, dx = \int_\R (2\pi|x|)^{-2s} f(x)\, dx,\end{split}\end{equation*}
			hence the result for $n<2s$.
			
On the other hand, for $n=2s$ we have that
			\[ \int_{B_R} \log |x| \wck f(x)\, dx = 2\int_{\R} f(\xi) \left(\int_0^R \log x \cos({2 \pi} \xi x) \, dx \right)\, d\xi .\]
		We change the  variable $\bar x = 2 \pi x$ (but still write $x$ as the variable of integration for simplicity) and let $\bar R = 2 \pi R$. Then we have that
		\[ \begin{split} \int_0^R \log x\cos(2\pi \xi x)\, dx = &\; \int_0^{\bar R } \Big(\log x -\log (2\pi) \Big) \cos(\xi x)\, \frac{dx}{2\pi}\\
			= &\; \frac{1}{2\pi}  \int_0^{\bar R}  \log x\cos(\xi x) \, dx - \frac{\log(2\pi)}{2\pi} \frac{\sin(\xi \bar R)}{\xi}.
			\end{split}\]
	We integrate by parts and obtain that
			\[\int_0^{\bar R} \log x\cos(\xi x) \, dx = \log {\bar R} \frac{ \sin (\xi {\bar R})}{\xi} -\frac{1}{|\xi|} \int_0^{{\bar R}|\xi|} \frac{\sin t }{t}\,dt.\]
	We thus have that
			\[ \int_{B_R} \log |x| \wck f(x)\, dx = \frac{1}{\pi} \log R \int_{\R} f(\xi) \frac{\sin(\xi \bar R)}{\xi} \, d\xi - \frac{1}{\pi} \int_{\R} f(\xi) |\xi|^{-1} \left (\int_0^{\bar R |\xi|} \frac{\sin t}{t} \, dt \right)\, d\xi.\]
			We claim that\begin{equation} \label{n2c1} \lim_{R\to \infty} \log R \int_{\R} f(\xi) \frac{ \sin (\xi {\bar R} )}{\xi} \, d\xi =0  .\end{equation}
	Indeed we have  \[ \bigg|\int_0^{1/{\bar R}} f(\xi) \frac{\sin (\xi {\bar R})}{\xi} \, d\xi \bigg| \leq  \int_0^{1/{\bar R}} |f(\xi)| \frac{\xi {\bar R}}{\xi} \, d\xi \leq c_1 {\bar R}\int_0^{1/{\bar R}} \xi \, d\xi = \frac{  c_1}{\bar R}.\]
	  Moreover  integrating by parts
	 \[ \begin{split}
	 	\bigg| \int_{1/{\bar R}}^{\infty} f(\xi) \frac{\sin (\xi {\bar R})}{\xi} \, d\xi  \bigg| \leq &\;  \frac{|f(\xi)|}{\xi } \frac{| \cos(\xi {\bar R})| } {{\bar R}} \Bigg|_{1/{\bar R}}^{\infty}  +\frac{1}{{\bar R}} \Bigg( \int_{1/{\bar R}}^{\infty} \frac{|f(\xi)|}{\xi^2} |\cos(\xi {\bar R})|\, d\xi\\
	 	&\; + \int_{1/{\bar R}}^{\infty} \frac{|f'(\xi)|}{\xi}| \cos(\xi {\bar R})| \, d\xi \Bigg).
		\end{split}\]
	We have  for $\xi$ large that
			 	\[  \frac{|f(\xi)|}{\xi } \frac{|\cos(\xi {\bar R})|}{{\bar R}} \leq \frac{c_2}{\bar R} \frac{|\cos(\xi {\bar R})|}{\xi^2} \] hence
			 	\[ \lim_{\xi\to \infty} \frac{|f(\xi)|}{\xi } \frac{|\cos(\xi {\bar R})|} {{\bar R}}  =0.\]
%			 	Meanwhile, $f(1/\bar R) \leq c \bar R^{-2s}$.
%			 \[  {\bar R} |f(1/{\bar R})| \frac{\cos 1}{{\bar R}} \leq \frac{2 \pi c_1 \cos 1}{{ R}} .\]
	On the other hand by using the change of variables $t=\xi {\bar R}$
		\[ \begin{split} \int_{1/{\bar R}}^{\infty} \frac{|f(\xi)|}{\xi^2}|\cos(\xi {\bar R})| \, d\xi \leq &\;c_1 \int_{1/{\bar R}}^{1} \frac{ |\cos(\xi {\bar R})| }{\xi}\, d\xi +c_2 \int_1^{\infty} \frac{|\cos(\xi {\bar R})|}{\xi^3} \, d\xi\\
		  =&\; c_1 \int_1^{\bar R} \frac{|\cos t|}{t} \, dt + c_2 \int_{\bar R}^{\infty} \frac{{\bar R}^2}{t^3} |\cos t|\, dt
		  \leq \bar c_1 \log R +\bar c_2.
		\end{split}\]
			 Moreover we have that
			 \[ \begin{split} \int_{1/{\bar R}}^{\infty} \frac{|f'(\xi)|}{\xi}|\cos(\xi {\bar R})| \, d\xi \leq &\;c'_1 \int_{1/{\bar R}}^{1} \frac{|\cos(\xi {\bar R})|}{\xi} \, d\xi +c'_2 \int_1^{\infty} \frac{|\cos(\xi {\bar R})| }{\xi^2} \, d\xi\\
		  =&\; c'_1 \int_1^{\bar R} \frac{|\cos t|}{t} \, dt + c'_2 \int_{\bar R}^{\infty} \frac{{\bar R}}{t^2} |\cos t|\, dt
		  \leq\bar c'_1 \log { R} +\bar c'_2.
		\end{split}\]
		Hence \[ \lim_{R \to \infty} \log R \int_0^{\infty} f(\xi) \frac{\sin (\xi {\bar R})}{\xi} \, d\xi =0\] and since the same bounds hold for $ \int_{-\infty}^0  f(\xi) \frac{\sin (\xi {\bar R})}{\xi} \, d\xi ,$  the claim \eqref{n2c1} follows.
	Also, the proof of claim \eqref{n1c2} gives that  \[ \begin{split} \lim_{R \to \infty} \int_0^{\infty} f(\xi) |\xi|^{-1}  \bigg(\int_0^{{\bar R}|\xi|} \frac{\sin t }{t}\,dt \bigg) \, d\xi  = &\; \int_0^{\infty} f(\xi) |\xi|^{-1} \bigg(\int_0^{\infty} \frac{\sin t }{t}\,dt \bigg) \, d\xi   \\
			  =&\; \frac{\pi}{2} \int_0^{\infty} f(\xi) |\xi|^{-1} \, d \xi .\end{split}\] It follows that
			\begin{equation*} \int_{\R} \log |x| \wck f(x)\, dx = -\frac{1}{2}  \int_\R |\xi|^{-1} f(\xi)\, d\xi,  \end{equation*} 	
				hence 				
				\begin{equation} \label{obans3} -\frac{1}\pi  \int_{\R} \log |x| \wck f(x)\, dx = \int_\R ({2\pi} |\xi|)^{-1} f(\xi)\, d\xi  \end{equation}
				and the result holds for $n=2s$.				This concludes the proof of the Proposition.
	\end{proof}

\begin{remark}
It is now clear that we have chosen $a(n,s)$ in Definition \ref{definition:ctn} in order to normalize the Fourier transform of the fundamental solution. Indeed, for $n>2s$, we perform the change of variable $\pi \alpha =t$ in \eqref{c1} and by the definition of the Gamma function (see \eqref{ABRAMOWITZ}) we obtain that
	\[ c_1= \pi^{ s-\frac{n}{2}} \int_0^{\infty} t^{\frac{n}{2}-s-1} e^{-t} \, dt = \pi^{ s-\frac{n}{2}}  \Gamma \bigg( \frac{n}{2} -s\bigg).\]
Also in \eqref{c2} we change the variable $\pi \alpha =t$ and get that
	\[c_2= \pi^{ -s}  \int_0^{\infty}t^{s-1} e^{-t}\, dt= \pi^{ -s} \Gamma(s).\]
Therefore
	 \begin{equation*} \frac{c_1}{c_2}= \frac{ \pi ^ {2s-\frac{n}{2} }  \Gamma (\frac{n}{2}-s)} {\Gamma(s)}, \quad  \mbox{hence by \eqref{aaargh1}}\quad   a(n,s)= {\frac{ \Gamma (\frac{n}{2}-s)} {2^{2s}\pi^{\frac{n}2} \Gamma(s)}}.\end{equation*}
The value $a(1,s)$ is computed in \eqref{oba1s}. We point out that we can rewrite this value using \eqref{gam1} and \eqref{gam2}, as follows
	\[\begin{aligned}  a(1,s) = \frac{1}{2\cos(\pi s)\Gamma(2s) }= \frac{\Gamma(1/2-s) }{2^{2s} \sqrt{\pi} \Gamma(s)} .\end{aligned}\]
	Moreover, we observe that identity \eqref{obans3} says that
				\[ a\left(1,\frac{1}{2}\right) = -\frac{1}{\pi}.  \]
		\end{remark}		

By applying this latter Proposition \ref{proposition:ansss}, we prove Theorem \ref{theorem:thm3}.
\begin{proof}[Proof of Theorem \ref{theorem:thm3}]
For any $f\in \mathcal{S}(\Rn)$ we have that $\F^{-1} \Big(|\xi|^{2s} \widehat f(\xi) \Big) \in \Sa_s(\Rn) $ (according to definition \eqref{frlaphdef} and to \eqref{frb1}). Notice that $|\xi|^{2s} \widehat f(\xi) \in L^1(\Rn)\cap C(\Rn)$, since
	\[ \int_{\Rn} |x|^{2s}|\widehat f(x)|\, dx\leq [\widehat f]_{\Sa(\Rn)}^{0,n+2} \int_{\Rn \setminus B_1} |x|^{2s-n-2} \, dx + \sup_{x\in B_1} |\widehat f(x)| \leq c(f), \]where we use the seminorm defined in \eqref{seminormss}.
Moreover, for $n\leq 2s$ we have that
	 	\[\begin{aligned}
	 		&  |\xi|^{2s} |\widehat f(\xi)| \leq \|\widehat f\|_{L^{\infty} (\R)} |\xi|^{2s} = c_1 |\xi|^{2s} \quad &\text{ for } &\xi\in \R,\\
	 		&  |\xi|^{2s} |\widehat f(\xi)| \leq [\widehat f(\xi)]_{\Sa(\R)}^{0,3} |\xi|^{2s-3} \leq \frac{c_2}{|\xi|} \quad &\text{ for } &|\xi| > 1.\end{aligned}\]  Also for $0 \neq|\xi|\leq 1$
	 		\[\begin{split} \Big| \frac{d}{d\xi} \big( |\xi|^{2s}\widehat f(\xi)\big) \Big| \leq &\; 2s |\xi|^{2s-1} |\widehat f(\xi)| + |\xi|^{2s}\Big|\frac{d}{d\xi} \widehat f (\xi)\Big| \\
	 		\leq&\; |\xi|^{2s-1} \bigg( 2s\|\widehat f \|_{L^{\infty}(\R)} +\Big\|\frac{d\widehat f(\xi)}{d\xi}  \Big \|_{L^{\infty}(\R)} \bigg)
	 		= c_1' |\xi|^{2s-1} \end{split}\]
	 		and for $|\xi|>1$ 	 		
	 \[\begin{split} \Big| \frac{d}{d\xi} \big( |\xi|^{2s}\widehat f(\xi)\big) \Big| \leq &\; 2s |\xi|^{2s-1} |\widehat f(\xi)| + |\xi|^{2s}\Big|\frac{d}{d\xi} \widehat f (\xi)\Big| \\	
	 			\leq &\; 2s [\widehat f]_{\Sa(\R)}^{0,2} |\xi|^{2s-3} + |\xi|^{2s -3} [\widehat f]_{\Sa(\R)}^{1,3}
	 			\leq c_2' |\xi|^{-1}, \\
	 		\end{split}\]
which proves that $f$ satisfies \eqref{condffff}. From Proposition \ref{proposition:ansss} it follows that
	\begin{equation*}
		\begin{split}
	\langle \Phi, \frlap f\rangle_s=   &\;\int_{\Rn} \Phi(x) \F^{-1}\Big( (2\pi|\xi|)^{2s} \widehat f(\xi) \Big)(x) \, dx \\
	= &\;\int_{\Rn} (2\pi|\xi|)^{-2s} (2\pi |\xi|)^{2s} {\widehat f (\xi)}  \, d\xi
	=\int_{\Rn} {\widehat f (\xi)}  \, d\xi
	= f(0).
			\end{split}
	 \end{equation*}
Therefore in the distributional sense
\begin{equation*} \frlap \Phi = \delta_0 .	\qedhere\end{equation*}
\end{proof}
We introduce now the following Lemmata, that will be the main ingredients in the proof of main result of the subsection.

\begin{lemma}\label{lem:uff2}
Let  $f \in  C^{\infty}_c(\Rn) $, let $\varphi $ be an arbitrary function such that we have $\wck \varphi \in \Sa_s(\Rn)$ and the following hold:\\
a) for $n>2s$, $\varphi \in L^1(\Rn)\cap C(\Rn)$, \\
b) for $n\leq 2s$, $\varphi \in L^1(\R)\cap C(\R) \cap C^1\big( (-\infty,0)\cup (0,\infty)\big))$ and
	 \begin{equation*} \begin{aligned}
	 &|\varphi(x)|\leq c_1|x|^{2s} \quad &\text{ for } & x \in \R\\
	   & |\varphi(x)|\leq \frac{c_2}{|x|} \quad &\text{ for } &|x|>1 \\
	    &|\varphi'(x)|\leq c'_1|x|^{2s-1} \quad &\text{ for } &0< |x|\leq 1\\
	  &  |\varphi'(x)|\leq \frac{c'_2}{|x|} \quad &\text{ for } &|x|>1. \end{aligned} \end{equation*}
	  Then in both cases
 \begin{equation} \label{disf2} \int_{\Rn} f*\Phi(x) \wck \varphi (x) \, dx = \int_{\Rn} ({2\pi}|x|)^{-2s} \wck f(x) \varphi(x) \, dx.\end{equation}
\end{lemma}
\begin{proof}
In order to prove identity \eqref{disf2} we notice that by the Fubini-Tonelli theorem we have that
	\[\begin{split}
		\int_{\Rn} f*\Phi(x) \wck \varphi(x) dx =&\; \int_{\Rn} \bigg( \int_{\Rn} \Phi(y) f(x-y) \, dy\bigg) \, \wck \varphi(x)\, dx \\
		= &\;\int_{\Rn} \Phi(y)\bigg( \int_{\Rn} f(x-y) \wck\varphi(x)\, dx \bigg) \, dy.  \end{split}\]
		We denote
		\[ f\bar*\wck \varphi (y):=\int_{\Rn} f(x-y)\wck \varphi(x) \, dx =\int_{\Rn} f(x)\wck \varphi(x+y)\, dx\]
		and write \begin{equation} \label{fbarwck1} \int_{\Rn} f*\Phi(x) \wck \varphi(x) dx  =\int_{\Rn} \Phi(y) f\bar*\wck\varphi(y) \, dy.\end{equation}
The operation $\bar* $ is well defined for $f \in C^{\infty}_c(\Rn)$ and $\wck \varphi \in \Sa_s(\Rn)$, furthermore it is easy to see that
	\[ \F(f\bar *\wck\varphi)(x) = \wck f (x)  \varphi (x).  \] We notice at first that since $\varphi$ and $\wck \varphi$ are continuous, $\F(\wck \varphi)=\varphi$ on $\Rn$. We define
		\begin{equation}\label{psifi2}  \psi(x) := \F(f\bar *\wck\varphi)(x) = \wck f (x)  \varphi (x)\end{equation} and we write \eqref{fbarwck1} as
		\begin{equation}\label{psifi1}  \int_{\Rn} f*\Phi(x) \wck \varphi(x) dx  =\int_{\Rn} \Phi(y) \wck \psi(y) \, dy .\end{equation}
		We check that $\psi$ verifies the hypothesis of Proposition \ref{proposition:ansss}.
Since \[ \int_{\Rn} |\psi(x)|\, dx = \int_{\Rn} |\wck f(x)| |\varphi(x)| \, dx\leq \|\wck f\|_{L^{\infty}(\Rn)} \|\varphi\|_{L^1(\Rn)},\] we have that $\psi \in L^1(\Rn)$. Also, $\psi \in C(\Rn)$ as a product of continuous functions.  We claim that $f\bar*\wck\varphi \in \Sa_s(\Rn)$. Indeed, suppose $\text{supp}\, f\subseteq B_R$ for $R>0$. We remark that in the next computations the constants may change from line to line. Then for $|x|\leq 2R$ we have that	
	\[ \begin{split}
			(1+|x|^{n+2s} ) |f\bar*\wck\varphi(x)| \leq &\; c_{n,s,R} \int_{B_{R}(x)} |f(y-x) \wck\varphi(y)|\, dy \\
			\leq &\;  c_{n,s,R} \|f\|_{L^\infty(B_R)} \|\wck \varphi\|_{L^{\infty}(B_{3R})}.
			\end{split}\]
For $|x|>2R$ we have that
	\[ \begin{split} |x|^{n+2s} |f\bar*\wck\varphi(x)|\leq\|f\|_{L^\infty(B_R)} [\wck \varphi]_{\Sa_s(\Rn)}^0 |x|^{n+2s} \int_{B_{R}  } |x+y|^{-n-2s} \, dy
	\end{split}\]
and we remark that $|y| \leq |x| /2$ (otherwise $y\notin \text{supp} \,f$). Then we use the bound $|x+y|\geq |x|-|y|\geq |x|/2$ and we have that
	\[
	 |x|^{n+2s} |f\bar*\wck\varphi(x)|\leq\|f\|_{\infty} [\wck \varphi]_{\Sa_s(\Rn)}^0  |x|^{n+2s}  \int_{ B_{R} }  |x|^{-n-2s}\, dy =c_{n,s,R}.
	\]
	We can iterate the same method to prove that $(1+|x|^{n+2s} ) |D^{\alpha} f\bar*\wck\varphi(x)|$ is bounded since $D^{\alpha} f\bar*\wck\varphi(x)=f\bar*D^{\alpha} \wck\varphi(x)$ and $D^{\alpha} \wck\varphi \in \Sa_s(\Rn)$.
	For $n\leq 2s$ we have that 	
\[\begin{aligned}
	 		&  |\psi(x)| \leq | \wck f(x)||\varphi(x)| \leq \|\wck f\|_{L^{\infty}(\R)} c_1|x|^{2s} \quad &\text{ for } &|x| \leq 1,\\
	 		&  |\psi(x)| \leq | \wck f(x)||\varphi(x)| \leq \|\wck f\|_{L^{\infty}(\R)} \frac{c_2}{|x|} \quad &\text{ for } &|x| > 1 .\end{aligned}\]
	 		Moreover, for $|x|>1$
	 		\[\begin{split}
	 		|\psi'(x)| \leq&\; | \wck f(x)||\varphi'(x)| + \Big|\frac{d}{dx} \wck f (x) \varphi(x)\Big|\\
	 		 \leq &\;\|\wck f\|_{L^{\infty}(\R)} \frac{c'_2}{|x|} + \Big| \int  f (\xi) (i\xi)e^{ix\xi} \, d\xi\Big| |\varphi(x)|\\
	 		 \leq&\; \|\wck f\|_{L^{\infty}(\R)} \frac{c'_2}{|x|} + \|\xi f(\xi)\|_{L^1(\R)} \frac{c_2}{|x|} 	 \leq\frac{C}{|x|}
	 		\end{split}\]
and for $|x|\leq 1$, since $f\in C_c^{\infty}(\R)$
 	\[\begin{split}
	 		|\psi'(x)| \leq&\; | \wck f(x)||\varphi'(x)| + \Big|\frac{d}{dx} \wck f (x)\Big| | \varphi(x)|\\
	 					\leq&\; \|\wck f\|_{L^{\infty}(\R)} c_1' |x|^{2s-1} + c_1|x|^{2s} \|\xi f(\xi)\|_{L^{\infty}(\R)} |x|^{-1}
	 					= C|x|^{2s-1}.
\end{split}\]
Hence we can apply Proposition \ref{proposition:ansss} and taking also into account \eqref{psifi1} we have that
			\[ \int_{\Rn} f*\Phi(x)\wck \varphi(x)\, dx = \int_{\Rn} \Phi (x) \wck \psi (x) \, dx =\int_{\Rn} ({2\pi}|x|)^{-2s} \psi (x)  \, dx,\]
			and from \eqref{psifi2} we conclude that
				\[  \int_{\Rn} f*\Phi(x)\wck \varphi(x)\, dx=\int_{\Rn} ({2\pi}|x|)^{-2s}  \wck f (x) \varphi(x)\, dx. \qedhere\]
\end{proof}

Also:
\begin{lemma}\label{lem:uff1}
Let  $f \in  C_c(\Rn) $, then $ f*\Phi \in  L_s^1(\Rn)$.
\end{lemma} 

\begin{proof}
To prove that $f*\Phi \in  L_s^1(\Rn)$, we suppose that $\mbox{supp}\, f\subseteq B_R$ and we compute
\[\begin{split}
\int_{\Rn} \frac{|f*\Phi(x)|}{1+|x|^{n+2s} }\, dx =&\;
% \int_{\Rn} \frac{\displaystyle \int_{B_R} |f(y)| \Phi(x-y) \, dy}{1+|x|^{n+2s}} \, dx \\
	\int_{B_R}| f(y)|\bigg(\int_{\Rn}\frac{ \Phi(x-y)}{1+|x|^{n+2s}} \, dx \bigg) \, dy\\
	\leq &\; \|f\|_{L^\infty(\Rn)} \int_{B_R} \bigg(\int_{\Rn}\frac{ \Phi(x-y)}{1+|x|^{n+2s}} \, dx \bigg) \, dy.
	\end{split} \]
We set
	\begin{equation} \label{intppr}{ c_{n,s,R} :=\int_{B_R} \bigg(\int_{\Rn}\frac{ \Phi(x-y)}{1+|x|^{n+2s}} \, dx \bigg) \, dy} \end{equation}
	and prove it is a finite quantity.
%	 in the three cases $n>2s$, $n<2s$ and $n=2s$, which leads to the bound
	We take for simplicity $R=1$ and remark that the constants in the next computations may change value from line to line. For $n > 2s$ we have that
   \[\begin{split}
   	 \int_{B_1} \bigg(\int_{\Rn}\frac{ \Phi(x-y)}{1+|x|^{n+2s}} \, dx \bigg) \, dy  = a(n,s) \int_{B_1} \bigg(\int_{\Rn}\frac{ |x-y|^{2s-n}}{1+|x|^{n+2s}} \, dx \bigg) \, dy.   	 	
   \end{split} \]
 For $x$ small we have that
      \[\begin{split}
       \int_{B_1}  &\; \bigg(\int_{B_2}\frac{ |x-y|^{2s-n}}{1+|x|^{n+2s}} \, dx \bigg) \, dy   \leq\int_{B_1} \left( \int_{B_{2}} |x-y|^{2s-n} \, dx\right)\, dy\\
       \leq &\;  c_n \int_{B_1} \bigg(\int_0^{2+|y|} t^{2s-1}  \, dt \bigg) \, dy 
%       =&\; c_{n,s} \int_{B_1} (2+|y|)^{2s} \, dy \\
       =c_{n,s} \int_0^1 (2+t)^{2s} t^{n-1}\, dt  = \bar c_{n,s}.
   \end{split} \]
For $x$ large, we use that $|x-y|\geq |x|-|y|$ and $1+|x|^{n+2s}>|x|^{n+2s}$, thus
  \[\begin{split}
    \int_{B_1} &\; \bigg(\int_{\Rn \setminus B_2}\frac{ |x-y|^{2s-n}}{1+|x|^{n+2s}} \, dx \bigg) \, dy \leq  \int_{B_1} \bigg(\int_{\Rn \setminus B_2} (|x|-|y|)^{2s-n} |x|^{-n-2s} \, dx \bigg) \, dy\\
    =&\; c_n \int_0^1 t^{n-1}\bigg(\int_2^{\infty} (\rho-t)^{2s-n} \rho^{-n-2s} \rho^{n-1} \, d\rho \bigg)\, dt 
    \leq c_n \int_2^{\infty}  (\rho-1)^{2s-n-1} \, d\rho = c_{n,s}.
   \end{split} \]
   Hence for $n>2s$ the quantity $c_{n,s,R}$ in \eqref{intppr} is finite.	
Meanwhile, for $n<2s$ for $x$ small the same bound as for $n>2s$ holds. For $x$ large, we have that
  \[\begin{split}
    \int_{B_1} \bigg(\int_{\R \setminus B_2}\frac{ |x-y|^{2s-1}}{1+|x|^{1+2s}} \, dx \bigg) \, dy \leq &\; \int_{B_1} \bigg(\int_{\R \setminus B_2} (|x|+|y|)^{2s-1} |x|^{-1-2s} \, dx \bigg) \, dy\\
    =&\; c  \int_0^1 \bigg(\int_2^{\infty} (\rho+t)^{2s-1} \rho^{-1-2s}  \, d\rho \bigg)\, dt  = c_s.\\
%    \leq &\;c \int_2^{\infty}  (\rho+1)^{2s-1} \rho^{-2s-1} \, d\rho
%    \leq &\;c \int_2^{\infty}  (\rho+1) \rho^{-2s-1} \, d\rho \\
     \end{split} \]
In the case $n=2s$ from the triangle inequality we have that
	\[ \int_{B_1} \bigg( \int_{B_2} \frac {\log |x-y|}{1+|x|^2} \, dx \bigg) \, dy  \leq  c  \int_0^2 \log(t+1) \, dt = \tilde c\]
	and
	\[\begin{split}
	\int_{B_1} \bigg(  \int_{\R \setminus B_2}  \frac {\log |x-y|}{1+|x|^2} \, dx \bigg) \, dy   \leq  \int_2^{\infty} \log(t+1) t^{-2}\, dt =\tilde c.
		\end{split}\]
Hence $c_{n,s,R}$ in \eqref{intppr} is finite and we have that
	\begin{equation} \int_{\Rn} \frac{|f*\Phi(x)(x)|}{1+|x|^{n+2s} }\, dx \leq  c_{n,s,R}\|f\|_{L^\infty(\Rn)} \label{db2} .\end{equation} It follows that $f*\Phi\in L_s^1(\Rn)$, as stated.
	\end{proof}

We introduce moreover the following regularity result.
\begin{lemma}  \label{grlem12} Let $s\in (0,1)$ be fixed. Let $f\in C_c^{0,\eee}(\Rn)$ be a given function (for a small $\eps>0$) and $u$ be defined as
 \eqlab{ u(x):= \int_{\Rn} \frac{f(y)}{|x-y|^{n-2s}} \, dy.}Then 
 $u\in C^{2s+\eee}(\Rn)$. 
\end{lemma}
\begin{proof}
Let $s<1/2$ and $\eps>0$ be such that $2s+\eps<1$ and we prove that $u\in C^{0,2s+\eps}(\Rn)$.
Let $R>0$ be such that $\mbox{supp } f\subseteq B_R$. Then taking $x_1, x_2 \in \Rn$ and denoting by 
\[\delta:=|x_1-x_2|\]
 we have that
\bgs{ |u(x_1)&-u(x_2)| \\\leq \al \int_{B_R \cap \{ |x_1-y|\leq 2\delta\} }\frac{|f(y)-f(x_1)|}{|x_1-y|^{n-2s}} \, dy+  \int_{B_R \cap \{ |x_1-y|\leq 2\delta\}} \frac{|f(y)-f(x_1)|}{|x_2-y|^{n-2s}} \, dy\\ \al 
+ 
 \int_{B_R \setminus \{ |x_1-y|\leq 2\delta\}} |f(y)-f(x_1)| \bigg| \frac{1}{|x_1-y|^{n-2s} }-\frac{1}{|x_2-y|^{n-2s}}\bigg|\, dy  \\ \al+ |f(x_1)|\bigg|\int_{B_R} \lr{ \frac{1}{|x_1-y|^{n-2s} }-\frac{1}{|x_2-y|^{n-2s}}} dy \bigg|\\
 =:\al I_1 +I_2+I_3 +I_4.  }
 Since $f$ is H{\"o}lder continuous, we have that for $C>0$
 \[ |f(y)-f(x_1)|\leq C |y-x_1|^{\eee}.\] Noticing that in the next computations the constants may change value form line to line, we obtain that
 \bgs{ & I_1 \leq C  \int_{B_R \cap \{ |x_1-y|\leq 2\delta\} } |x_1-y|^{-n+2s+\eee}\, dy = C_{n,s} \delta^{2s+\eee},\\
 &I_2 \leq  C\int_{B_R \cap \{ |x_1-y|\leq 2\delta\} } |x_1-y|^\eee |x_2-y|^{-n+2s}\, dy \leq C_{n,s} \delta^{2s+\eee}}
 since $|x_2-y|\leq |x_2-x_1|+|x_1-y|\leq 3\delta$. \\
In $I_3$ we notice that since $|x_1-y|\geq 2\delta$ we have that $|x_2-y|\geq|x_1-y|- |x_1-x_2|\geq \delta$. The function $|\cdot\,-y|^{2s-n}$ is differentiable  in $\Rn \setminus B_{\delta}(y)$, hence at each point on the segment $x_1x_2$. Using the Mean Value Theorem we have that for some $x^\star$ on the segment $x_1x_2$ 
 \eqlab{\label{MVTgreen}\left| \frac{1}{|x_1-y|^{n-2s}} -\frac{1}{|x_2-y|^{n-2s}} \right| \leq C_n \frac{|x_1-x_2|}{|x^\star-y|^{n-2s+1}}    =C_n \delta \frac{1}{|x^\star-y|^{n-2s+1}} .}
 It follows that
   \bgs{  I_3 \leq  C  \delta \int_{B_R \setminus   \{ |x_1-y|\leq 2\delta\} } \frac{|x_1-y|^{\eee}}{|x^\star -y|^{n-2s+1}}  \, dy .}
  Since $|y-x^\star| \geq |y-x_1| -|x_1-x^\star| \geq \frac{1}2 |x_1-y|$, recalling that $2s+\eps<1$ we obtain
   \bgs{  I_3   \leq C_{n,s} \delta^{2s+\eee}. }
   
Now for $I_4$, if $x_1\notin B_R$ or $x_2\notin B_R$ (it is enough in this latter case to replace $x_1$ with $x_2$ in the above computations), then we are done. Else, for  $x_1,x_2 \in B_R$, suppose that $\mbox{dist}(x_1,\partial B_R) \geq  \mbox{dist}(x_2,\partial B_R)  $ and take $p\in \partial B_R$ (hence $f(p)=0$) such that $\mbox{dist} (x_1,\partial B_r)=|x_1-p|$. So
 \[|f(x_1)|=|f(x_1)-f(p)|\leq C  |x_1-p|^{\eee} \]  and we distinguish two cases. When
 \[  (1) \quad |x_1-x_2| \geq \frac{1}2|x_1-p|\]   
% we use the result in Lemma \ref{lem1}, observing that from \eqref{lem1eq1} and \eqref{lem1eq2}
  we have that
  \bgs{  \int_{B_R } & \bigg| \frac{1}{|x_1-y|^{n-2s}} -\frac{1}{|x_2-y|^{n-2s}} \bigg| \, dy \\
\leq \al  \Bigg[ \int_{B_R \cap \{ |x_1-y|\leq 2\delta\} }\frac{dy}{|x_1-y|^{n-2s}} +  \int_{B_R \cap \{ |x_1-y|\leq 2\delta\}} \frac{dy}{|x_2-y|^{n-2s}}  \\ \al 
+ 
 \int_{B_R \setminus \{ |x_1-y|\leq 2\delta\}} \bigg| \frac{1}{|x_1-y|^{n-2s} }-\frac{1}{|x_2-y|^{n-2s}} \bigg|\, dy \Bigg]\\
 =:\al  J_1 +J_2+J_3.  }
By passing to polar coordinates we have that 
 \[ J_1 \leq C_n  \int_0^{2\delta} \rho ^{2s-1} \, d\rho = C_{n,s} \delta^{2s}\]
and that
 \[ J_2  \leq C_n   \int_0^{3\delta}\rho^{2s-1}\, d\rho = C_{n,s} \delta^{2s},\] since $|x_2-y|\leq 3\delta$.
 For $J_3$, we use \eqref{MVTgreen} and get 
 \[ J_3\leq C_n \delta  \int_{B_R \setminus \{ |x_1-y|\leq 2\delta\}}  \frac{1}{|x^\star-y|^{n-2s+1}}\, dy.\]
% Since $|x^\star-x_1| \leq \delta \leq \frac{1}2 |x_1-y|$, we have that $|x^\star-y|\geq \frac{1}2 |x_1-y|$. 
 Passing to polar coordinates, since $2s<1$, we get that
 	\bgs{ J_3\leq &\;C_n\delta   \int_{2\delta}^{\infty}  \rho ^{2s-2}\, d\rho = C_{n,s} \delta^{2s}  .}
 	 	So we obtain that
\eqlab{\label{exressch}  \int_{B_R } \bigg| \frac{1}{|x_1-y|^{n-2s}} -\frac{1}{|x_2-y|^{n-2s}} \bigg| \, dy \leq C \delta^{2s},}where $C=C(n,s) $ is a positive constant.  
 %  
% \[ \int_{B_R }\bigg| \frac{1}{|x_1-y|^{n-2s}}-\frac{1}{|x_2-y|^{n-2s}}\bigg|  \, dy \leq C |x_1-x_2|^{2s} ,\]
 Given that $|x_1-p|<2|x_1-x_2|=2\delta$ we get that 
 \[ I_4\leq  C|x_1-p|^{\eee} \delta^{2s} \leq C \delta^{2s+\eee}.\] 
 This sets the bound for $J_4$ in the case (1).
 On the other hand, when   \[ (2)\quad |x_1-x_2| \leq \frac{1}2 |x_1-p|\]  we use the following bound (see Lemmas 2.1 and 3.5 in \cite{samko})
 \eqlab{\label{sdom} \bigg|\int_{B_R }|x_1-y|^{2s-n}-|x_2-y|^{2s-n} \, dy\bigg| \leq \frac{|x_1-x_2|}{\max\{\mbox{dist}(x_1,\partial B_R) ,\mbox{dist}(x_2,\partial B_R) \}^{1-2s} }.} Since $2s+\eee-1 <0$ we get that
\bgs{ I_4 \leq C \delta |x_1-p |^{2s-1+\eee} \leq C_{n,s} \delta^{2s+\eee}.}
 This concludes the proof of the Lemma for $s<1/2$. 
 In order to prove the result for $s\geq1/2$ (hence to prove that $u\in C^{1,2s+\eps-1}$), one can use Lemma 4.1 in \cite{trudy}
\bgs{ D u(x)=\int_{\Omega} D \Phi(x-y) f(y)\, dy=\int_{\Omega} \frac{f(y) }{|x-y|^{n-2s+1}}\, dy,} and iterate the computations of this proof.
   \end{proof}
The interested reader can see Theorem 4.6 in \cite{samko}, where the result given here in Lemma \ref{grlem12} is proved for $u$ defined as
\[ u(x)= \int_{\Omega} \frac{f(y)}{|x-y|^{n-2s}} \, dy,\] where $\Omega$ is a domain with the $s$-property (see Definition 3.3 therein). In particular, these domains are defined such that they satisfy a bound of the type given in \eqref{sdom}, while the ball is the typical example of this type.  
%
%\begin{theorem}
%\label{theorem:poissonsolution1}
% Let $f \in  {C^{2s+\eee}_c(\Rn)} $ and let $u$  be defined as
%	\[  u(x):=\Phi*f\, (x) .\]
%Then $ u\in L_s^1(\Rn)\cap C^{2s+\eee}$ and pointwise in $\Rn$
%	\[ \frlap u =f. \]
%\end{theorem}

We state now the main result of this subsection.
\begin{theorem}
\label{theorem:poissonsolution}
 Let $f \in  {C^{0,\eee}_c(\Rn)} $ and let $u$  be defined as
	\[  u(x):=\Phi*f\, (x) .\]
Then $ u\in L_s^1(\Rn)\cap C^{2s+\eee}(\Rn)$ and both in the distributional sense and pointwise
	\[ \frlap u =f. \]
\end{theorem}
\begin{proof}

From Lemma \ref{lem:uff1}, we have that $u\in L_s^1(\Rn)\cap C^{2s+\eee}(\Rn)$. We prove at first the statement for $f\in C^{\infty}_c(\Rn)$.
% By approximation, we prove that the results holds when the known term has  $C^{0,\eee}_c(\Rn) $ regularity.

We notice that for $\varphi \in \mathcal{S}(\Rn)$, the function $\mathcal F^{-1} ((2\pi |\xi)|^{2s} \widehat \varphi (\xi)) \in \mathcal{S}_s(\Rn)$ and $(2\pi |\xi)|^{2s} \widehat \varphi (\xi)$ satisfies the hypothesis of Lemma \ref{lem:uff2}. Hence, by \eqref{disf2}
	\begin{equation*}
		\begin{split}
	\langle u, \frlap \varphi\rangle_s =&\;\int_{\Rn}  f*\Phi(x) \mathcal{F}^{-1} \Big((2\pi |\xi|)^{2s}\widehat \varphi(\xi) \Big)(x) \, dx\\
%			= &\;\int_{\Rn} (2\pi |\xi|)^{-2s} \wck{f}(\xi) (2\pi |\xi|)^{2s} {\widehat \varphi(\xi)} \, d\xi	\\
		=&\; \int_{\Rn} \wck{f}(\xi) {\widehat \varphi(\xi)} \, d\xi=  \int_{\Rn}  f(x) \varphi(x) \, dx.
	\end{split}
	\end{equation*}
	The last equality follows since $\wck f \in L^1(\Rn)$, which is implied by the infinite differentiability of $f$.
%	Indeed, $|\wck f| (x)\leq (1+|x|)^{-n-1} \|\F^{-1} (f^{(n+1)})\|_{L^{\infty}(\Rn)} \leq (1+|x|)^{-n-1} \|f^{(n+1)}\|_{L^1(\Rn)}$.
	We conclude that $u$ is the distributional solution of \[\frlap u = f.\]

We consider now $f\in C^{0,\eee}_c(\Rn) $. We take a sequence of functions $(f_k)_k \in C^{\infty}_c (\Rn)$ such that $\|f_k-f\|_{L^{\infty}(\Rn)} \underset{k\to \infty}{\longrightarrow} 0$ and we consider $u_k=\Phi*f_k$. Then we have that for any $\varphi \in \Sa(\Rn)$
	\[ \langle u_k, \frlap \varphi\rangle_s= \int_{\Rn} f_k (x) \varphi(x)\, dx.\] By definition of $f_k$ \[ \lim_{k \to +\infty} \int_{\Rn} f_k(x)\varphi(x)\, dx = \int_{\Rn} f(x)\varphi(x)\, dx ,\]
	moreover, using \eqref{db1} and \eqref{db2} we have that
		\[\begin{split}
		 \langle u_k -u,\frlap \varphi\rangle_s \leq&\; [\frlap \varphi]^{0}_{\Sa_s(\Rn)}\|u_k-u\|_{L_s^1(\Rn)}\\
		  \leq&\;  c_{n,s,R}[\frlap \varphi]^{0}_{\Sa_s(\Rn)}\|f_k-f\|_{L^{\infty}(\Rn)} \underset{k \to \infty}{\longrightarrow} 0.
		 \end{split}\]
	We thus obtain that for any $\varphi \in \Sa(\Rn)$ \[\langle u,\frlap \varphi\rangle_s =\int_{\Rn} f(x)\varphi(x)\, dx.\] Hence in the distributional sense $\frlap u = f $ on $\Rn$ for any $f\in C^{0,\eee}_c(\Rn)$.
	 
	In order to obtain the pointwise solution, 
%	using Lemmata \ref{lem:uff1} and \ref{grlem12}, we have that the fractional Laplacian of $u$ is well defined. Moreover, 
	from the continuity of the mapping $ \Rn \ni x \mapsto \frlap u(x) $ (according to Proposition 2.1.7 from \cite{Silvestre}), we have that
		$  \int_{\Rn}\frlap u(x) \varphi(x)\, dx$ is well defined. Since for any $\varphi \in C^{\infty}_c(\Rn)$ we have that
		\[\int_{\Rn} u(x)\frlap \varphi(x)\, dx = \int_{\Rn} f(x)\varphi(x)\, dx ,\]  thanks to Fubini-Tonelli's Theorem and changing variables we obtain that 
		%for any $\varphi \in C^\infty_c(\Rn)$
			\[ \int_{\Rn} f(x)\varphi(x)\, dx=\int_{\Rn} u(x)\frlap \varphi(x)\, dx = \int_{\Rn} \frlap u(x)\varphi(x) \, dx  .\]  Since both $f$ and $\frlap u$ are continuous, we conclude that pointwise in $\Rn$ \[ \frlap  u(x) = f(x).\qedhere \]
	 \end{proof}

{As a corollary, we have a representation formula for a $C^{\infty}_c ({\Rn})$ function.}
\begin{corollary}
\label{lemma:unouno}
For any $f \in C^{\infty}_c ({\Rn})$ there exists a function $\varphi \in C^{\infty}({\Rn})$ such that 		
	\[f(x)=\varphi* \Phi(x), \]  and $\varphi (x)=\mathcal{O} (|x|^{-n-2s})$ as $|x| \rightarrow \infty$.
\end{corollary}

\begin{proof}[Proof]
For $f \in C^{\infty}_c(\Rn)$, we define $\varphi$ as
	\[ \varphi (x):=\frlap f(x).\]
The bound established in \eqref{frb1} gives the asymptotic behavior of~$\varphi$, while it is not hard to see that $\varphi \in C^{\infty}({\Rn})$.
Then by using Theorem \ref{theorem:poissonsolution} we have that pointwise in $\Rn$
	\[\varphi*\Phi(x)= \frlap f *\Phi(x)=  f(x).\qedhere\]

\end{proof}

\subsection{The Poisson kernel} \label{thepoissonkernel}

We claim that $P_r$ plays the role of the fractional Poisson kernel
%. Indeed, the function $P_r$ arises in the construction of the solution to Dirichlet problem with vanishing Laplacian inside the ball and a known forcing term outside the ball.
namely, we prove here Theorem \ref{theorem:DPL}.
%, as stated in the next Theorem~\ref{theorem:DPL}.
%In order to prove this Theorem, we rely on a representation formula for a $C^{\infty}_c ({\Rn})$ function, as stated in
%, any function $ C^{\infty}_c ({\Rn})$ can be represented as a convolution of the function $\Phi$  with a second function belonging to $C^{\infty}({\Rn})$ and vanishing at infinity like $|x|^{-n-2s}$. More precisely we have
%the following lemma.
%Having Lemma \ref{lemma:unouno} and making use of some results from the Appendix, we proceed with the proof of Theorem~\ref{theorem:DPL}.
%\begin{theorem}
%\label{theorem:DPL}
% Let $r>0$, $g \in L^1_s(\Rn) \cap C({\Rn})$ and let
%\begin{equation}
%		 u_g(x) : =
%			\begin{cases}	
%				\displaystyle  \int_{{\Rn}\setminus B_r} P_r(y,x) g(y)\, dy &\quad  \, \text{if } x\in B_r, \\
%				g(x) &\quad \, \text{if } x \in {\obal}.
%			\end{cases} \label{solD}
%	\end{equation}
%Then $u_g$ is the unique pointwise continuous solution of the problem \eqref{LaplaceeqD}
%	\begin{equation*}
%	\begin{cases}
%	\frlap u= 0 \qquad  &\mbox{ in }  {B_r},
%\\	u= g \qquad  &\mbox{ in }  {\obal}.
%		\end{cases}
%	\end{equation*}
%\end{theorem}
\begin{proof}[Proof of Theorem \ref{theorem:DPL}]
We see at first that $u_g \in L_s^1(\Rn) $. Take $R>2r$ and $x \in B_r$, then by using \eqref{Ip}, the inequality $|x-y| >|y|-r$ and for $|y|>R$ the bound
	\begin{equation} \label{bbly1}\frac{|y|^{n+2s}}{(|y|^2-r^2)^s|x-y|^n} \leq  2^{n+s}\end{equation}
we have that
	\[ \begin{split} |u_g(x)|\leq &\;  \int_{R>|y|>r} P_r(y,x) |g(y)|\, dy + \int_{|y|>R} P_r(y,x) |g(y)|\, dy\\
							\leq&\; c(n,s) \sup_{y\in \overline B_R\setminus B_r} |g(y)| + 2^{n+s} c(n,s) (r^2-|x|^2)^s \int_{|y|>R}  \frac{|g(y)|}{|y|^{n+2s}}\, dy\\
							\leq&\; c(n,s) \sup_{y\in \overline B_R\setminus B_r} |g(y)| + 2^{n+s} c(n,s) r^{2s}  \int_{|y|>R}  \frac{|g(y)|}{|y|^{n+2s}}\, dy.
							 \end{split} \]
Since $g \in L_s^1(\Rn)$, the last integral is bounded, and so $u_g$ is bounded in $B_r$. It follows that $u_g\in L_s^1(\Rn)$, as stated. Moreover, the local $C^\infty$ regularity of $u_g$ in $B_r$ follows from the regularity of the Poisson kernel.

Let us fix $x \in B_r$ and prove that $u_g$ has the $s$-mean value property in $x$. If this holds, indeed, Theorem \ref{theorem:arm} implies that $\frlap u(x)=0$, and given the arbitrary choice of $x$, the same is true in the whole $B_r$.
%We prove the continuity and the uniqueness of $u_g$ as the solution to the Dirichlet problem in the final part of the proof.

%We prove that $u_g$ has the $s$-mean value property in the following manner: at first, we prove the result in the particular case in which the forcing term $g$ is defined as $ g(y) := \displaystyle \int_{B_R\setminus B_r} \Phi(z-y) \, dz,$ for a fixed $R>r$. Then, in a second step, we prove the result for any function $g \in C^{\infty}_c({\Rn})$ by means of the representation formula provided in Lemma \ref{lemma:unouno}. At last, we prove the desired result for~$g$ in $L_s^1(\Rn)\cap C(\Rn)$, employing approximations by $C^{\infty}_c({\Rn})$ functions.

%Let $x\in B_r$ be fixed.
We claim that for any $\rho$ such that $0<\rho< r-|x|$ we have
\begin{equation} A_{\rho} * u_g (x) = u_g(x).\label{grclaim} \end{equation}
\noindent Let at first $g$ be in $C^{\infty}_c{(\Rn)} $. By Corollary \ref{lemma:unouno}, there exists a function $\varphi\in C^{\infty}(\Rn)$ such that
	  \[ g(y)= \int_{\Rn} \Phi(z-y) \varphi(z) \, dz \] and at infinity $\varphi (z) = \mathcal{O}(|z|^{-n-2s})$.
For $r>0$ fixed, we write $g$ as
	\begin{equation} \label{grbbbb} g(y)=\int_{\obal} \Phi(t-y) \varphi(t) \, dt + \int_{B_r} \Phi(z-y) \varphi(z) \, dz.\end{equation}
Using identity \eqref{Ipu} we have that
% may write the second term of the sum in identity \eqref{bbbb} as
	\begin{equation*}
		\begin{split}
	\int_{B_r} \Phi(z-y) \varphi(z) \,  dz  =\;&\int_{B_r} \bigg(\int_{\obal}P_r(t,z)\Phi(y-t) \, dt\bigg)\, \varphi(z) \, dz\\
	= \;&\int_{\obal}\Phi(y-t) \bigg(\int_{B_r}P_r(t,z)\varphi(z)\, dz\bigg)\, dt.
		\end{split}
	\end{equation*}
Therefore, in \eqref{grbbbb} it follows that
	 \begin{equation} g(y) = \int_{\obal} \Phi(y-t)\psi(t)\, dt, \label{gsecondcase}\end{equation}
where $\psi(t) = \varphi(t)+\int_{B_r}P_r(t,z)\varphi(z) \, dz$.
In particular, using \eqref{solD} and \eqref{gsecondcase} we have that {
%with computations identical to the ones we have done in the previous step, we obtain that for $x\in B_r$
	\begin{equation*} \begin{split}
		 u_g(x)=&\; \int_{|y|>r}P_r(y,x) \bigg(\int_{|t|>r} \Phi(y-t) \psi(t) \, dt\bigg) \, dy\\
		 	= &\; \int_{|t|>r} \psi(t)\bigg( \int_{|y|>r} P_r(y,x) \Phi(y-t) \, dy\bigg) \, dt = \int_{|t|>r} \psi(t) \Phi(x-t) \, dt\end{split}
		 \end{equation*}
thanks to \eqref{Ipu}. Furthermore, we compute
	\begin{equation*} \begin{split}
		A_{\rho}*u_g(x) =&\; \int_{|y|>\rho} A_\rho(y)  \bigg(\int_{|t|>r} \psi(t) \Phi(x-y-t) \, dt \bigg)\, dy  \\
					=&\;    \int_{|t|>r} \psi(t) \bigg(\int_{|y|>\rho} A_\rho(y) \Phi(x-y-t) \, dy\bigg) \, dt.					
	\end{split}
		 \end{equation*}}
Having chosen $\rho\leq r-|x|$ we have that $|x-t|\geq |t|-|x| \geq \rho$ and from \eqref{Ifu} we obtain
 \begin{equation*}
		A_{\rho}*u_g(x) = \int_{|t|>r} \psi(t)  \Phi(x-t)\, dt.
	 \end{equation*}
Consequently $A_{\rho} * u_g (x) =u_g(x),$ thus for $g\in C^{\infty}_c({\Rn})$ the claim \eqref{grclaim} is proved.\\

We now prove the claim \eqref{grclaim} for any forcing term $g \in L^1_s({\Rn})\cap C(\Rn)$. In particular, let $\eta_k \in C^{\infty}_c(\Rn)$ be such that $\eta_k(x) \in [0,1]$, $\eta_k=1$ in $B_k$ and $\eta_k=0$ in $B_{k+1}$. Then $g_k:=\eta_k g \in C^{\infty}_c(\Rn)$ and we have that $g_k \underset{k\to \infty}{\longrightarrow} g$ pointwise in $\Rn$, in norm $L_s^1(\Rn)$ and uniformly on compact sets. So, for any $k \geq 0$ the function $u_{g_k}(x)$ has the $s$-mean value property in $x$. Precisely, for any $\rho>0$ small independent of $k$,
	\begin{equation} \big(A_{\rho}*u_{g_k}\big)(x) =u_{g_k}(x). \label{primacosa} \end{equation}
We claim that
	\begin{equation}
		\lim_{k\to \infty} u_{g_k}(x) = u_g(x) \label{puno}
	\end{equation}
and that for any $\rho>0$ small
	\begin{equation}
		\begin{split}
		&\lim_{k\to \infty} \big(A_{\rho}*u_{g_k}\big)(x) = A_{\rho}*u_g (x)  \label{pdue}.
		\end{split}
	\end{equation}
%In order to prove claim \eqref{puno} we notice that, by definition \eqref{solD} of $u_g$, we have that
%	\[ u_{g_k} (x) = \int_{\obal} g_k(y)P_r(y,x) \, dy \]
%and that \[ u_g(x) = \int_{\obal} g(y)P_r(y,x) \, dy.\]
Let $\eee $ be any arbitrarily small quantity. For $k$ large and $R>2r,$ we take advantage of \eqref{bbly1} and obtain that for $x \in B_r$
\[ \begin{split}	
		|u_{g_k} (x) -u_g(x)| \leq &\; \int_{\Rn\setminus B_r} |g_k(y)-g(y)| P_r(y,x)\,dy \\
%				=&\; c(n,s) \int_{\Rn\setminus B_R} |g_k(y)-g(y)| \frac{(r^2-|x|^2)^s}{(|y|^2-r^2)^s|x-y|^n}\,dy \\
%					&\;+ \int_{B_R \setminus B_r} |g_k(y)-g(y)| P_r(y,x)\,dy\\
				\leq&\;  2^{n+s} c(n,s) (r^2-|x|^2)^s  \int_{\Rn \setminus B_R}  \frac{|g_k(y)-g(y)| }{|y|^{n+2s}} \, dy \\
				&\;+ \sup _{y \in \overline B_R\setminus B_r}|g_k(y)-g(y)| \int_{B_R\setminus B_r}  P_r(y,x)\, dy\\	
				\leq &\; c(n,s,r) \int_{\Rn \setminus B_R}  \frac{|g_k(y)-g(y)| }{|y|^{n+2s}} \, dy + \sup _{y \in \overline B_R\setminus B_r}|g_k(y)-g(y)|  \leq \eee \end{split}\]
by the convergence in $L_s^1(\Rn)$ norm, the uniform convergence on compact sets of $g_k$ to $g$ and integrability in $\Rn \setminus B_r$ of the Poisson kernel (by identity  \eqref{Ip}). Hence, claim \eqref{puno} is proved.			
In order to prove claim \eqref{pdue}, we notice that for any $\rho>0$ small we have that
	\begin{align} \label{illy}
		 | A_{\rho}* u_{g_k}(x)  - A_\rho * u_g(x) |\leq &  \int_{|y|>\rho} A_{\rho}(y)|   u_{g_k}  (x-y)  -u_g(x-y)| \, dy \notag\\
		\leq & \int_{\substack{ {|y|>\rho}\\{|x-y|\geq r}}} A_{\rho}(y) |g_k(x-y)-g(x-y)|  \, dy \notag\\
		&+  \int_{|z|>r}  | g_k(z)-g(z)|    \int_{\substack{ {|y|>\rho}\\ {|x-y|<r}}}    A_{\rho} (y) P_r(z,x-y)  dy \, dz\notag\\
		=&\; I_1+I_2.
	\end{align}
Let $R>2\rho$. Thanks to the bound \eqref{bbly1} for $|y|>R$, the convergence in $L_s^1(\Rn)$ norm, the uniform convergence on compact sets of $g_k$ to $g$ and the integrability in $\Rn \setminus B_{\rho}$ of the $s$-mean kernel (by identity \eqref{Ir}) we have that for $k$ large
\[	\begin{split}
		I_1 =&\; c(n,s) r^{2s} \int_{\substack{ {|y|>\rho}\\{|x-y|\geq r}}} \frac{|g_k(x-y)-g(x-y)|} {(|y|^2-\rho^2)^s |y|^n} \, dy\\
			\leq &\;2^{n+s} c(n,s,r) \int_{|y|>R} \frac{|g_k(x-y) - g(x-y)| }{|y|^{n+2s}}\, dy \\
						&\; + \sup_{y \in \overline B_R\setminus B_\rho} |g_k (x-y)-g(x-y)| \int_{R>|y|>\rho} A_\rho(y) \, dy
						\leq  \frac{\eee}{2}. \end{split} \]
Once more, for $R>2r$ and $|z|>R$ we use the bound \eqref{bbly1} and we have that
\[	\begin{split}
		I_2 = &\; \int_{|z|>R} |g_k(z)-g(z)| \int_{\substack{{|y|>\rho}\\ {|x-y|<r}}} A_\rho(y)P_r(z,x-y)  dy   \, dz \\\
		&\;  +  \int_{R>|z|>r} |g_k(z)-g(z)| \int_{\substack{{|y|>\rho}\\ {|x-y|<r}}} A_\rho(y)P_r(z,x-y)  \, dy \, dz \\
	\leq &\;c(n,s) \int_{\substack{{|y|>\rho}\\ {|x\!-\!y|<r}}} A_\rho(y) (r^2\!-\!|x\!-\!y|^2)^s \!   \int_{|z |>R} \frac{|g_k(z) \!-\! g(z)| }{(|z|^2\! -\!r^2)^s |z\!-\!x\!+\!y|^n}\, dz  \, dy \\
						&\;+  \sup_{z\in \overline B_R\setminus B_r} |g_k (z)-g(z)|  \int_{\substack{{|y|>\rho}\\ {|x-y|<r}}} A_\rho(y)   \int_{R>|z|>r} P_r(z,x-y)  \, dz  \, dy \\
				\leq &\;c(n,s,r)  \int_{|z|>R} \frac{|g_k(z) - g(z)| }{|z|^{n+2s}}dz   + \sup_{z \in \overline B_R\setminus B_r} |g_k (z)-g(z)| 
				\end{split} \]
since by identity \eqref{Ir} and \eqref{Ip}
	\[
	\int_{\substack{{|y|>\rho}\\ {|x-y|<r}}} A_\rho(y)   \int_{R>|z|>r} P_r(z,y-x)  \, dz  \, dy  	\leq 1.
	\]
 Therefore again by the convergence in $L_s^1(\Rn)$ norm, the uniform convergence on compact sets of $g_k$ we have that $I_2 \leq 	\displaystyle \frac{\eee}{2}$. In \eqref{illy} it follows that
\begin{equation*}
		\lim_{k\to \infty} \big(A_{\rho}* u_{g_k}\big)(x) =  A_{\rho}*u_g(x),
	\end{equation*}
thus the desired result \eqref{pdue}.

By \eqref{primacosa}, \eqref{puno} and \eqref{pdue} we have that
\[A_{\rho}*u_g (x) = \lim_{k \to \infty}A_{\rho}* u_{g_k}(x)   = \lim_{k \to \infty} u_{g_k} (x) =u_g(x),\]
thus $u_g$ has the $s$-mean value property at $x$.
This concludes the proof of the claim \eqref{grclaim}.

We now prove the continuity of $u_g$. Of course, $u_g$ is continuous in $B_r$ and in $\obal$. We need to check the continuity at the boundary of $B_r$.

Let $y_0 \in \partial B_r$ and $\eee>0$ arbitrarily small to be fixed,  $\delta_{\eee}>0$ be such that, if $y \in B_{\delta_{\eee}}(y_0)$ then $|g(y)-g(y_0)| < \eee$. We fix $\mu$ arbitrarily small such that $0 <\mu<  \frac{\delta_{\eee}}{2}$, $R> 2r$, and $x \in  B_r \cap \ B_{\mu} (y_0 )$. Notice that \[ r^2-|x|^2 = (r+|x|) (r-|x|) < 2r |y_0-x| < 2r \mu.\]
From \eqref{Ip} we have that
	\begin{equation}\label{illy1} |u_g(x)-g(y_0)| \leq \int_{\obal} \big|g(y)-g(y_0)\big|P_r(y,x) \, dy .\end{equation}
For $r<|y|<R$ and $|y-y_0|\geq \delta_\eee$ we have that
	$ |x-y| \geq \delta_{\eee} - \mu>  \frac{\delta_{\eee}}{2} .$
Meanwhile, for $|y|>R$ we use the bound \eqref{bbly1}. We have that
	\begin{equation*}
		\begin{split}
		& \int_{\substack{ {|y|>r}\\{|y-y_0|\geq  \delta_{\eee}}}}  |g(y)-g(y_0)|P_r(y,x) \, dy\\
%		= \; &  c(n,s) (r^2-|x|^2)^s \int_{\substack{ {|y|>r}\\{|y-y_0|\geq  \delta_{\eee}}}}\frac{|g(y)-g(y_0)|}{ (|y|^2-r^2)^s |x-y|^n}   \, dy  \\
		 \leq\; &   c(n,s,R) \mu^s \bigg( \frac{2^n}{ \delta_{\eee}^{n}} \int_{\substack{ {R>|y|>r}\\{|y-y_0|\geq  \delta_{\eee}}}} \frac{|g(y)|+|g(y_0)|}{ (|y|^2-r^2)^s }\, dy
		 + 2^{n+s} \int_{|y|>R} \frac{|g(y)|+ |g(y_0)|}{ |y|^{n+2s}}\, dy   \bigg) \\
	\leq \; &  c(n,s,R) \mu^s \bigg(\frac{2^n }{\delta_{\eee}^{n}} \overline c(r,R,s,g)
	+ 2^{n+s} \|g\|_{L_s^1(\Rn)}+ \tilde c (g,s, R)  \bigg) \\
		=  &\;C(n,s,R,r,g,\delta_\eee) \mu^s .
		\end{split}
	\end{equation*}
	From this and the fact that	\[\int_{\substack{  {|y|>r}\\{|y-y_0|< \delta_{\eee}}} } |g(y)-g(y_0)|P_r(y,x) \, dy \leq \eee \int_{\obal} P_r(y,x)\, dy =\eee\]
		by \eqref{Ip} and the continuity of $g$, we can pass to the limit in \eqref{illy1}. Sending first $\mu  \rightarrow 0$ and afterwards $\eee \rightarrow 0$ we obtain that
	\[  \lim_{x \rightarrow y_0 }\big(u_g(x)-g(y_0)\big) =0,\] thus the continuity of $u_g$.

{The uniqueness of the solution follows from the Maximum Principle. Indeed, if one takes $u_1$ and $u_2$ two different continuous solutions of the Dirichlet problem, then $u=u_1-u_2$ is a continuous solution to the problem
	 \begin{align*}
		\frlap u(x)  &=0, &  &\text{in } B_r \\
		u(x)& =0, & &\text{in } \Rn \setminus  B_r.
		\end{align*}
By Theorem \ref{THM-MA-1-STRONG}, the solution $u$ is constant, hence null since it is continuous in $\Rn$ and vanishing outside of $B_r$.} This concludes the proof of the Theorem.
\end{proof}

%%%%%%%%%%%%%%%%%%%%%%%%%%%%%%%%%%%%%%%%%%%%

\subsection[The Green function for the ball]{The Green function for the ball}\label{green}

The purpose of this subsection is to prove Theorems~\ref{theorem:thm1} and~\ref{theorem:thm2}.
%Theorem~\ref{theorem:thm1} introduces a formula for the Green function on the ball, that is more suitable for applications. In Theorem \ref{theorem:thm2} the solution to the Dirichlet problem with vanishing data outside the ball and a given term inside the ball is built in terms of the Green function. 
We also compute the normalization constants needed in the formula of the Green function on the ball.

We prove now Theorem \ref{theorem:thm1} in the three cases $n>2s$, $n<2s$ and $n=2s$ separately.
\begin{proof}[Proof of Theorem \ref{theorem:thm1} for~$n>2s$]

Let $x,\, z\in B_r$ be fixed.\\
We insert the explicit formula \eqref{fundsolution} into definition \eqref{greendefn} and obtain that
	\begin{equation}G(x,z) =  a(n,s) \big( |z-x|^{2s-n}-A(x,z)  \big), \label{axz}\end{equation}
where \[A(x,z) := \int_{|y|>r} \frac{P_r(y,x) }{|y-z|^{n-2s}} \, dy.\]
Inserting also definition \eqref{poissondefn} we have that
	\begin{equation*}
\begin{split}
 A(x,z)= c(n,s)\int_{|y|>r} \frac{ (r^2-|x|^2)^s} {|y-z|^{n-2s} (|y|^2-r^2)^s|y-x|^n} \, dy .\\
%= \;&  c(n,s)\int_{|y|>r}  (r^2-|x|^2)^s |y-x|^{-n} \frac{|y-z|^{2s} } { (|y|^2-r^2)^s} \, \frac{dy}{|y-z|^n}  .
		\end{split}
	\end{equation*}

We use the point inversion transformation that is detailed in the Subsection \ref{appy}. { Let $x^* \in \Rn \setminus \overline{B}_r$ and $y^* \in B_r$ be the inversion of $x$, respectively $y$ with center at $z$, defined by the relation \eqref{tr}. }
With this transformation, using formulas \eqref{firsttr}, \eqref{dxtr} and \eqref{sectr} we obtain that
	\begin{equation*}
		\begin{split}
		A(x, z)
%		 \;&  c(n,s)  \int_{B_r} \bigg(\frac{|x^*-z|}{r^2-|z|^2}\bigg)^{n-2s} \frac{  (|x^*|^2-r^2)^s }{  (r^2-|y^*|^2)^s |x^*-y^*|^n } \, dy^* \\
		=   c(n,s) |z-x|^{2s-n} \int_{B_r}  \frac{  (|x^*|^2-r^2)^s }{  (r^2-|y^*|^2)^s |x^*-y^*|^n } \, dy^*.
		\end{split}	
	\end{equation*}

{We continue the proof for $n>3$. However, the results hold for $n\leq 3$ and can be proved with similar computations. We use hyperspherical coordinates with $\rho>0$ and  $\theta, \theta_1, \dots, \theta_{n-3} \in [0, \pi], \theta_{n-2} \in [0,2\pi]$ (see \eqref{chofvarhypersph} in the Subsection \ref{appy} and observations therein).}
Without loss of generality and up to rotations, we assume that $x^* =|x^*|e_n$, so we have the identity $ |x^*-y^*|^2 = \rho^2+|x^*|^2-2|x^*| \rho\cos\theta$ (see Figure~\ref{fign:xycalc} in the Subsection \ref{appy} for clarity). With this change of coordinates, we obtain
\begin{equation*}
		\begin{split}
		  A(x,z)=\; &c(n,s)|z-x|^{2s-n} (|x^*|^2-r^2)^s 2\pi \prod_{k=1}^{n-3}\int_0^{\pi} \sin^k\theta\, d\theta  \\
		 & \;  \int_0^r \frac{\rho^{n-1}}{(r^2 -\rho^2)^s}   \bigg(\int_0^{\pi} \frac{ \sin^{n-2}\theta}{(\rho^2+|x^*|^2-2|x^*| \rho\cos\theta)^{n/2}} \, d\theta \bigg) \,d\rho.		
    \end{split}
	\end{equation*}
Let $\tau:=  \frac{|x^*|}{\rho}$ (notice that $\tau >1$). We have that
	\[	\int_0^{\pi} \frac{ \sin^{n-2}\theta}{(\rho^2+|x^*|^2-2|x^*| \rho\cos\theta)^{n/2}} \, d\theta =  \frac{1}{\rho^n}  \int_0^{\pi} \frac{ \sin^{n-2}\theta}{(\tau^2+1-2\tau \cos\theta)^{n/2}} \, d\theta.\]
Thanks to identity \eqref{prop1} we obtain that
	\begin{equation*}
		\begin{split}
\int_0^{\pi} \frac{ \sin^{n-2}\theta}{(\rho^2+|x^*|^2-2|x^*| \rho\cos\theta)^{n/2}} \, d\theta = \; &\frac{1}{\rho^n}  \frac{1}{\tau^{n-2}(\tau^2-1)} \int_0^{\pi} \sin^{n-2}\alpha \, d\alpha \\
=\; &  \frac{1}{|x^*|^{n-2}(|x^*|^2-\rho^2)}\int_0^{\pi}\sin^{n-2}\alpha \, d\alpha.
		\end{split}
	\end{equation*}
Then, using identity \eqref{prop2} and inserting the explicit value of $c(n,s)$ given by \eqref{ctcns}, we arrive at
	\begin{equation} \label{axzz}
		\begin{split}
	 A(x,z)
%	 \; & c(n,s)|z-x|^{2s-n} \frac{{\pi}^{n/2}}{\Gamma(n/2)} \frac{\big(|x^*|^2-r^2\big)^s} {|x^*|^{n-2}} \int_0^r\frac{2\rho^{n-1}}{(r^2 -\rho^2)^s (|x^*|^2-\rho^2)}\, d\rho\\
	= \; & \frac{\sin (\pi s) }{\pi}   |z-x|^{2s-n} \frac{\big(|x^*|^2-r^2\big)^s} {|x^*|^{n-2}}  \int_0^r  \frac{2\rho^{n-1}}{(r^2 -\rho^2)^s (|x^*|^2-\rho^2)}\, d\rho\\
	=  \; & \frac{\sin (\pi s) }{\pi}  |z-x|^{2s-n} \frac{\big(|x^*|^2-r^2\big)^s} {|x^*|^{n-2}} J(x^*),
		\end{split}
	\end{equation}
where \[ J(x^*)=  \int_0^r  \frac{2\rho^{n-1}}{(r^2 -\rho^2)^s (|x^*|^2-\rho^2)}\, d\rho.\]

Now we define the constant \begin{equation} k(n,s) :=  \frac{1}{2}  \bigg(\int_0^1  \tau^{n-2s-1} (1-\tau^2)^{s-1} \, d\tau\bigg)^{-1} \label{intkns}\end{equation}
(we compute its explicit value at the end of Subsection \ref{green}).
% Hence we have that \[ 2 k(n,s) \int_0^1 \tau^{n -2s-1}(1-\tau^2)^{s-1} \, d\tau = 1.\]
Then we have that
	\[J(x^*)= 2 k(n,s) \int_0^r\frac{2\rho^{n-1}}{(r^2-\rho^2)^s(|x^*|^2-\rho^2)} \bigg(\int_0^1 \tau^{n-2s-1} (1-\tau^2)^{s-1} \, d\tau\bigg)  \, d\rho.\]
We perform the change of variables $t=\tau \rho$ and apply the Fubini-Tonelli's theorem to obtain that
	\begin{equation*}
	\begin{split}
	J(x^*)=\;& 2 k(n,s) \int_0^r \frac{2\rho}{(r^2-\rho^2)^s(|x^*|^2-\rho^2)} \bigg(\int_0^{\rho} t^{n-2s-1} (\rho^2-t^2)^{s-1} \, dt  \bigg)  \, d\rho\\
	=\;& 2k(n,s) \int_0^r t^{n-2s-1} \bigg( \int_t^r \frac{ 2\rho (\rho^2-t^2)^{s-1} }{(r^2-\rho^2)^s(|x^*|^2-\rho^2) }\, d\rho \bigg) \, dt.
	\end{split}
	\end{equation*}
We change variables $\rho^2-t^2=\tau$ and $r^2-\tau-t^2=\rho$ to obtain
\begin{equation*}
	\begin{split}
J(x^*)=\;& 2   k(n,s) \int_0^r t^{n-2s-1} \bigg( \int_0^{r^2-t^2} \frac{\tau^{s-1}}{(r^2-\tau-t^2)^s(|x^*|^2- \tau -t^2)} \, d\tau \bigg) \, dt\\
		 = \;& 2k(n,s)\int_0^r t^{n-2s-1} \bigg( \int_0^{r^2-t^2} \frac{ (r^2-t^2-\rho)^{s-1}}{\rho^s (|x^*|^2+\rho -r^2)} \, d\rho \bigg) \, dt \\
	= \;&  2k(n,s) \int_0^r t^{n-2s-1}  I(t) \, dt,
		\end{split}
	\end{equation*}
where \[I(t) = \int_0^{r^2-t^2} \frac{ (r^2-t^2-\rho)^{s-1}}{\rho^s (|x^*|^2+\rho -r^2)} \, d\rho.\]
Using Proposition \ref{proposition:uss} for $\alpha=r^2-t^2$ and $\beta=|x^*|^2-r^2$
%noticing that \[ 0\leq \frac{r^2-t^2}{|x^*|^2-t^2} \leq 1,\]
we have that \[  I(t)= \frac{\pi}{\sin(\pi s)} \frac{(|x^*|^2-t^2)^{s-1}}{(|x^*|^2-r^2)^s}.\]
Hence in $J(x^*)$, with the changes of variables $ \frac{|x^*|}{t}= \tau$ and then $\tau^2-1=t$ we have that
	\begin{equation*}
		\begin{split}
		 J(x^*) = \;& 2k(n,s)  \frac{\pi}{\sin(\pi s)} (|x^*|^2-r^2)^{-s} \int_0^r t^{n-2s-1} (|x^*|^2 -t^2)^{s-1} \, dt.\\
		=\;& 2 k(n,s) \frac{\pi}{\sin(\pi s)} \frac{ |x^*|^{n-2}}{(|x^*|^2-r^2)^s}  \int_{\frac {|x^*|}{r}}^{\infty} \frac{(\tau^2-1)^{s-1} }{\tau^{n-1}} \, d\tau \\
 	=   \;&   k(n,s)\frac{\pi}{\sin(\pi s)} \frac{ |x^*|^{n-2}}{(|x^*|^2-r^2)^s}  \int_{\frac{|x^*|^2-r^2}{r^2}}^{\infty} \frac{t^{s-1}} {(t+1)^{n/2}} \, dt.
		\end{split}
	\end{equation*}
{	Using formula \eqref{firsttr} and definition \eqref{ro} we have the equalities
	\[\frac{|x^*|^2-r^2}{r^2}  =\frac{(r^2-|x|^2)(r^2-|z|^2)}  {r^2|x-z|^2}= r_0(x,z). \]
Therefore inserting $J(x^*)$ into \eqref{axzz} it follows that
	\[	A(x,z)=  k(n,s) |z-x|^{2s-n}  \int_{r_0(x,z)}^{\infty} \frac{t^{s-1}} {(t+1)^{n/2}} \, dt.\]}
By inserting this into \eqref{axz} we obtain that
%	\begin{equation*} G(x,z)=  a(n,s)|z-x|^{2s-n} \bigg(  1 - k(n,s)\int_{\frac{|x^*|^2-r^2}{r^2}}^{\infty} \frac{t^{s-1}} {(t+1)^{n/2}} \, dt   \bigg).	 \end{equation*}
% we have that $\displaystyle  r_0(x,z) = \frac{|x^*|^2-r^2}{r^2},$
%therefore
\begin{equation*} G(x,z)=  a(n,s)|z-x|^{2s-n} \bigg(  1 - k(n,s)\int_{r_0(x,z)}^{\infty} \frac{t^{s-1}} {(t+1)^{n/2}} \, dt   \bigg).
\end{equation*}
 Now we change the variable $t=1/\tau^2 -1$ in definition \eqref{intkns} and obtain that
	\begin{equation}  k(n,s) \int_0^{\infty}  \frac{t^{s-1}} {(t+1)^\frac{n}{2}} \, dt =1.\label{knsinfty}\end{equation}
It follows that
	\begin{equation*}
		\begin{split}
	 G(x,z)
%	 = \;& a(n,s)|z-x|^{2s-n} k(n,s)\bigg(  \int_0^{\infty}  \frac{t^{s-1}} {(t+1)^\frac{n}{2}} \, dt - \int_{r_0(x,z)}^{\infty} \frac{t^{s-1}} {(t+1)^{n/2}} \, dt   \bigg)\\
	=  a(n,s)k(n,s) |z-x|^{2s-n} \int_0^{r_0(x,z)}  \frac{t^{s-1}} {(t+1)^\frac{n}{2}} \, dt  .
		\end{split}
	\end{equation*}
We set
 \begin{equation}\kappa(n,s):=a(n,s)k(n,s)\label{kappa}\end{equation}
and conclude that
	\begin{equation*} G(x,z)= \kappa(n,s)|z-x|^{2s-n}  \int_0^{r_0 (x,z)}  \frac{t^{s-1}} {(t+1)^\frac{n}{2}} \, dt .    \end{equation*}  	
Hence the desired result in the case $n>2s$.
\end{proof}

\begin{proof}[Proof of Theorem \ref{theorem:thm1} for~$n<2s$]
We consider without loss of generality $r=1$ (by rescaling, the statement of the theorem is verified in the more general case). By \eqref{fundsolution} and definition \eqref{greendefn} we have that
\begin{equation}G(x,z) =  a(1,s) \big( |z-x|^{2s-1}-A(x,z)  \big), \label{axz11}\end{equation}
where \[A(x,z) := \int_{\mathbb{R}\setminus (-1,1)} \frac{P_1(y,x) }{|z-y|^{1-2s}} \, dy.\]
Using definition \eqref{poissondefn} we have that
\begin{equation*}
\begin{split}
A(x,z)= c(1,s)\int_{\mathbb{R}\setminus (-1,1)} \frac{ (1-x^2)^s} {|y-z|^{1-2s} (y^2-1)^s|y-x|} \, dy  .
		\end{split}
	\end{equation*}
We proceed exactly as in the case $n>2s$ performing the point inversion transformation and we arrive at
\begin{equation*}
		\begin{split}
		 A(x, z)=  c(1,s) |z-x|^{2s-1} \int_{-1}^1  \frac{  (x^{*^2}-1)^s }{  (1-y^{*^2})^s |x^*-y^*| } \, dy^*,
		\end{split}	
	\end{equation*}
	where $|x^*|>1$.
By symmetry we have that \begin{equation}
		 A(x, z)=   c(1,s) |z-x|^{2s-1} (x^{*^2}-1)^s |x^*|   J(x^*),\label{axzn1}
		\end{equation}
with \[J(x^*)=  \int_0^1  \frac{2 }{  (1-y^{*^2})^s (x^{*^2}-y^{*^2} )} \, dy^* .\]
We change the variable ${y^*}^2=t$ and obtain that
	\[J(x^*)= \Big(\frac{1}{x^*}\Big)^2 \int_0^1 t^{-1/2} (1-t)^{-s} \Big(1-\frac{t}{{x^*}^2} \Big)^{-1} \, dt.\]
By the integral representation \eqref{inthyp} of the hypergeometric function, it follows that
	\[ J(x^*)=   \Big(\frac{1}{x^*}\Big)^2 \, \frac{\Gamma(\frac{1}{2})\Gamma(1-s)}{\Gamma(\frac{3}{2}-s)} F\bigg(1, \frac{1}{2}, \frac{3}{2}-s, \frac{1}{{x^*}^2}\bigg).\]
We use the linear transformation \eqref{hyp4} (notice that $\big({1}/{x^*}\big)^2 <1$) and obtain that
	\begin{equation}
		\begin{split}
		 F\bigg(1, \frac{1}{2}, \frac{3}{2}-s, \frac{1}{{x^*}^2}\bigg)=\;& \frac{\Gamma(\frac{3}{2}-s) \Gamma(-s)}{\Gamma(\frac{1}{2}-s)\Gamma(1-s)} F\bigg(1,\frac{1}{2},s+1,\frac{{x^*}^2-1}{{x^*}^2}\bigg) \\
	+ \;&  \bigg(\frac{{x^*}^2\!-\!1}{{x^*}^2} \bigg)^{-s}  \frac{\Gamma(\frac{3}{2}\!-\!s) \Gamma(s)}{\Gamma(1)\Gamma(\frac{1}{2})} F\bigg(\frac{1}{2}\!-\!s, 1\!-\!s,1\!-\!s, \frac{{x^*}^2-1}{{x^*}^2}\bigg).
		\label{sumhyp}
		\end{split}
	\end{equation}
The first hypergeometric function obtained in the sum \eqref{sumhyp} is transformed according to \eqref{hyp3} as
	\[ F\bigg(1,\frac{1}{2},s+1,\frac{{x^*}^2-1}{{x^*}^2} \bigg)=  |x^*|F\bigg(\frac{1}{2},s,s+1,1-{x^*}^2 \bigg). \]
 For $s+1>s>0$ and $|1-{x^*}^2|<1 $ the convergence conditions are fulfilled for the integral representation \eqref{inthyp} of the hypergeometric function. Therefore we may write
	\[F\bigg(1,\frac{1}{2},s+1,\frac{{x^*}^2-1}{{x^*}^2} \bigg)= |x^*| \frac{\Gamma(s+1)}{\Gamma(s)} \int_0^1 \frac{t^{s-1}}{ \big(1+({x^*}^2-1) t\big)^{1/2}} \, dt .\]
On the other hand, for the second hypergeometric function obtained in identity \eqref{sumhyp}, we use transformations \eqref{hyp2} and \eqref{hyp3} and arrive at
	\begin{equation*}
		\begin{split}
			F\bigg(\frac{1}{2}-s, 1-s,1-s, \frac{{x^*}^2-1}{{x^*}^2}\bigg)= &\;{x^*}^{-2s} |x^*| F\bigg(\frac{1}{2}-s,0,1-s,1-{x^*}^2\bigg)\\
	=&\;{x^*}^{-2s} |x^*|  F\bigg(0,\frac{1}{2},1-s, \frac{{x^*}^2-1}{{x^*}^2}\bigg).
		\end{split}
	\end{equation*}
We use the Gauss expansion \eqref{gausshyp} with $a=0$, $b=  \frac{1}{2}$, $c=1-s$ and $w=  \frac{{x^*}^2-1}{{x^*}^2}$ (we notice that $0>c-a-b>-1$ for $s>1/2$ and $|w|<1$, thus the series is convergent). Since $a=0$, all the terms of the series vanish, except for $k=0$. Hence we obtain that
	\[F\bigg(0,\frac{1}{2},1-s, \frac{{x^*}^2-1}{{x^*}^2}\bigg)=1\]
and therefore
	\[F\bigg(\frac{1}{2}-s, 1-s,1-s, \frac{{x^*}^2-1}{{x^*}^2}\bigg)= {x^*}^{-2s} |x^*| .\]
Consequently
	\begin{equation*}
		\begin{split}
	J(x^*)=&\;\frac{1}{|x^*| } \bigg( \frac{ \Gamma(\frac{1}{2}) \Gamma(-s) \Gamma(s+1)}{\Gamma(\frac{1}{2}-s) \Gamma(s)}
	\int_0^1 \frac{t^{s-1}}{ \big(1+({x^*}^2-1) t\big)^{1/2}} \, dt
	+  \frac{ \Gamma(s)\Gamma(1-s)}{ ( {x^*}^2 -1)^{s}}\bigg).
		\end{split}
	\end{equation*}
{We recall that $c(1,s)= \big(\Gamma(s)\Gamma(1-s)\big)^{-1}$ and we define the constant
\begin{equation}k(1,s):= c(1,s) \frac{ \Gamma(\frac{1}{2}) \Gamma(-s) \Gamma(s+1)}{\Gamma(\frac{1}{2}-s) \Gamma(s)}.\label{k1s}\end{equation} We insert $J(x^*)$ into \eqref{axzn1} and have that
	\begin{equation*}
		\begin{split}
	   A(x,z)=  k(1,s) |z-x|^{2s-1}    \int_0^1 \frac{ ({x^*}^2-1)^s t^{s-1}}{ \big(1+({x^*}^2-1) t\big)^{1/2}} \, dt  +|z-x|^{2s-1} .			
		 \end{split}
	\end{equation*}
	With the change of variables $({x^*}^2 -1)t = \tau$ we obtain that
	\begin{equation*}
		\begin{split}
	   A(x,z)=  k(1,s) |z-x|^{2s-1}\int_0^{{x^*}^2-1} \frac{t^{s-1}} {(t+1)^{\frac{1}{2}}} \, dt  +|z-x|^{2s-1} .			
		 \end{split}
	\end{equation*}}
Inserting this into \eqref{axz11} and noticing that ${x^*}^2-1 =r_0(x,z)$ it follows that
		\[ G(x,z)= - a(1,s) k(1,s)  |z-x|^{2s-1}\int_0^{r_0(x,z)} \frac{t^{s-1}} {(t+1)^{\frac{1}{2}}} \, dt.\]
We call \begin{equation} \kappa(1,s) =- a(1,s) k(1,s) \label{kk1s}\end{equation} and conclude the proof of Theorem \ref{theorem:thm1} for $n<2s$.
\end{proof}

\begin{proof}[Proof of Theorem \ref{theorem:thm1} for~$n=2s$]
Without loss of generality, we assume $r=1$.  We insert the explicit formulas  \eqref{fundsolution} and \eqref{poissondefn} into definition \eqref{greendefn}. Moreover, we use the explicit values of the constant $  a\Big(1,\frac{1}{2}\Big)$ from \eqref{ctans3} and $ c\Big(1, \frac{1}{2}\Big)$ from \eqref{ctcns}. We obtain that
	 \begin{equation}\label{gxz111}
			G(x,z) =-\frac{1}\pi \log |x-z| + \frac{1}{\pi^2} \int_{|y|\geq 1} \log|y-z| \sqrt{\frac{1-x^2}{y^2-1}  } \frac{dy}{|x-y|}.
		\end{equation}
Let \[A(x,z):= \int_{|y|\geq 1} \log|y-z| \sqrt{\frac{1-x^2}{y^2-1}  } \frac{dy}{|x-y|}.\]
We perform the change of variables $v=  \frac{yx-1}{y-x}$. Since $1-v^2 \geq 0$, we have that $|v|\leq 1$. We set $w:=  \frac{xz-1}{z-x}$ and observe that $|w|\geq 1$.
%We have that
%	\begin{equation*}
%		\begin{split}
%		  &dy= \frac{1-x^2}{(v-x)^2} \, dv, 	\\
%		&y^2-1= \frac{(1-v^2)(1-x^2)}{(v-x)^2}, \\
%		&z-y= \frac{v-w}{v-x}(z-x),\\
%		&x-y= \frac{1-x^2}{v-x}
%		\end{split}
%	\end{equation*}
It follows that
	\[A(x,z) = \int_{|v|\leq 1}   \bigg(\log\frac{ |v-w|}{|v-x|}+ \log|z-x|\bigg) \, \frac{dv}{\sqrt{1-v^2}}. \]
We use identity \eqref{logid} and since $|w|\geq 1$ and $|x|\leq 1$ we obtain that
 \begin{equation*}
	\begin{split}
	  A(x,z) =&\; \pi \log \Big(|w|+(w^2-1)^{1/2}\Big)  + \pi \log |x-z| \\
	= &\; \pi \log (1-zx + \sqrt{(1-x^2)(1-z^2)} ).
	\end{split}
\end{equation*}
Inserting this into \eqref{gxz111} we obtain that
	\[G(x,z)=\frac{1}\pi \log \bigg(\frac{1-zx + \sqrt{(1-x^2)(1-z^2)}}{|x-z|} \bigg).\]
This completes the proof of Theorem \ref{theorem:thm1} for~$n=2s$.
\end{proof}

We prove here Theorem~\ref{theorem:thm2}, which gives the representation formula for the Poisson equation.

%\begin{theorem} 
%Let $r>0$, $h \in C^{2s+\eee}(B_r)\cap C( \overline B_r)$ and let
%	 \begin{equation*}
%		u(x) : =
%			\begin{cases}
%		\displaystyle \int_{B_r} h(y) G(x,y) \, dy \quad & \text{ if } x\in B_r, \\
%		0 \quad \quad & \text{ if } x \in {\obal}.
%		\end{cases}
%	\end{equation*}
%Then $u$ is the unique pointwise continuous solution of the problem \eqref{PoissoneqD}
%	\begin{equation*}
%		\begin{cases}
%		    \frlap u= h   \qquad &\mbox{ in } {B_r} ,
% 		\\  u= 0   \qquad &\mbox{ in } {\obal}.
%	\end{cases}
%	\end{equation*}
%\end{theorem}	

\begin{proof}[Proof of Theorem \ref{theorem:thm2}]
We identify $h$ with its $C_c^{0,\eee} (\Rn)$ extension, namely we consider $\tilde h \in C_c^{0,\eee}(\Rn)$ with $B_r \subset \text{supp} \,\tilde h$ such that $\tilde h=h$ on $B_r$.
Then, by definition \eqref{greendefn} we have that in $B_r$
	\begin{equation*}
		\begin{split}
		 u(x) = &\; \int_{B_r}h(z) G(x,z)\, dz \\
		 	= &\; \int_{B_r} h(z)\Phi(z-x)  dz -   \int_{B_r} h(z) \bigg(\int_{\obal} \Phi(y-z)P_r(y,x) dy \bigg)  dz \\
	=&\;  h*\Phi(x) - \int_{\obal} P_r(y,x) \big(h*\Phi\big)(y) \, dy.
		\end{split}
	\end{equation*}
Let  \[ g(x):= h*\Phi(x) \mbox{ for any  } x \in \Rn.\] From Theorem~\ref{theorem:poissonsolution}, we have that $g \in L_s^1(\Rn) \cap C^{2s+\eps}(\Rn)$.
Let for any $x \in \Rn$ \[ u(x)=v_0(x)-v_1(x),\] where
$v_0(x) = g(x) \mbox{ in } \Rn$
and
\[ v_1(x)=	
		\begin{cases}
	\displaystyle  \int_{{\Rn}\setminus B_r} P_r(y,x) g(y)\, dy &\quad  \, \text{if } x\in B_r, \\
				g(x) &\quad \, \text{if } x \in {\obal}.
		\end{cases}
\]
Then for $x\in B_r$, thanks to Theorems \ref{theorem:poissonsolution} and \ref{theorem:DPL}
	\[ \frlap u(x) = h(x) - 0=h(x),\]
hence $u$ is solution \eqref{PoissoneqD}. Also, from Theorems \ref{theorem:poissonsolution} and \ref{theorem:DPL}, it follows that  $u\in C(\Rn)$.

The uniqueness of the solution follows from the simple application of the Maximum Principle for the fractional Laplacian (see Theorem \ref{THM-MA-1-STRONG}).
\end{proof}
\bigskip
%
%\subsection{Computation of the normalization constants}\label{constants}
%
%This subsection is dedicated to the computation of 
We compute now the constant $\kappa(n,s)$ in Theorem \ref{theorem:thm1}. For this, we start with the next identity for  
%We first compute the constant
 $k(n,s)$ given in \eqref{intkns}, when   $n>2s$:
%We recall its definition,
%\[ k(n,s):=  \frac{1}{2} \bigg( \int_0^1  t^{n-2s-1} (1-t^2)^{s-1} \, dt \bigg)^{-1} .\]
%We claim that
 \begin{equation} k(n,s)=  \frac{\Gamma({\frac{n}{2}})} { \Gamma(\frac{n}{2}-s) \Gamma(s)}.\label{kns} \end{equation}

Indeed, using definition \eqref{intkns} and taking the change of variable $\tau^2=t$ we have that
\[ \frac{1}{k(n,s)} = 2 \int_0^1  \tau^{n-2s-1} (1-\tau^2)^{s-1} \, dt =  \int_0^1 t^{\frac{n}{2}-s -1} (1-t)^{s-1} \, dt .\]
We use identities \eqref{betazerouno} and \eqref{betagamma} to obtain that
\[ \int_0^1  t^{\frac{n}{2}-s-1} (1-t)^{s-1} \, dt =  \frac{\Gamma( \frac{n}{2}  -s)\Gamma(s)}{\Gamma(\frac{n}{2} )},\]
which is exactly the result.

We now prove Theorem \ref{theorem:kns}, namely we compute the constant $\kappa(n,s)$ encountered in the formula of the Green function $G$.
%\begin{theorem} \label{theorem:kns}The constant $\kappa(n,s)$ introduced in identity \eqref{forgkns} is
%	\begin{align*}
%		 \kappa(n,s) &= \displaystyle \frac{\Gamma(\frac{n}{2}) } {2^{2s}\pi^{ \frac{n}2}   \Gamma^2(s) } & \text{ for } &n \neq 2s, \\
%		\kappa\Big(1,\frac{1}{2}\Big) &=\frac{1}{\pi}
%		  & \text{ for } &n=2s.
%		\end{align*}
%\end{theorem}
\begin{proof}[Proof of Theorem \ref{theorem:kns}]
For $n>2s$, we insert the values of $a(n,s)$ from \eqref{ctans1} and of $k(n,s)$ from \eqref{kns} into  \eqref{kappa} and we obtain that
\begin{equation*} \begin{aligned}
\kappa(n,s) =a(n,s)k(n,s)
% = &\; \frac{ \Gamma (\frac{n}{2}-s)} { 2^{2s}\pi ^ {\frac{n}{2} }  \Gamma(s)}  \frac{\Gamma({\frac{n}{2}})} { \Gamma(\frac{n}{2}-s) \Gamma(s)} \\
 =  \frac{ \Gamma(\frac{n}{2}) }{2^{2s}\pi^{\frac{n}2} \Gamma^2(s)}.
\end{aligned} \end{equation*}
For $n<2s$, we recall definitions \eqref{k1s}, \eqref{kk1s} and \eqref{ctans1}, we use  identities \eqref{gam3}, \eqref{gam4} and \eqref{gamxx1} relative to the Gamma function and obtain that
	\[ \begin{aligned}
	\kappa(1,s)=   -a(1,s) k(1,s)
%		=&\;  - \frac{ \Gamma(1/2-s) }{2^{2s} \sqrt{\pi} \Gamma(s)} \frac{\sin(\pi s)} {\pi} \frac{ \Gamma(1/2) \Gamma(-s) \Gamma(s+1)}{\Gamma(1/2-s) \Gamma(s)}\\
		=  \frac{(-s)\Gamma(-s) }{2^{2s}\Gamma(s)} \frac{1}{\Gamma(1-s)\Gamma(s)}
		=\frac{1}{2^{2s}\Gamma^2(s)}. \end{aligned} \]
%It is easily verified that this is the general formula given in Theorem \ref{theorem:kns}, evaluated at $n=1$.
On the other hand, we recall that $\kappa\Big( 1,\frac{1}{2}\Big)=  \frac{1}\pi$, as we have seen in the proof of Theorem \ref{theorem:thm1}  for $n=2s$.
This concludes the proof of Theorem \ref{theorem:kns}.
\end{proof}

%We now introduce the direct computation of the fractional Laplacian of a particular function. Namely:
%
%\begin{lemma}
%Let $u(x) =(1-|x|^2)_{+}^s$. Then in $B_1$
%	\[ \frlap u(x) = C(n,s)\frac{\omega_n}{2} B(s, 1-s),\]
%where $B$ is the Beta function defined in \eqref{beta}.
%\label{dydares}
%\end{lemma}
%{The more general case (more precisely, for the function $u(x) =(1-|x|^2)_{+}^p$ for any $p>-1$)  was proved in \cite{Dyda,Dyda2}. With small modifications with respect to the general case, the proof of Lemma \ref{dydares} can be also found in Section 3.6 in \cite{nonlocal}.}

We prove now Theorem \ref{thm:Cc}, that gives the value of the constant $C(n,s)$ introduced in \eqref{frlap2def}.
%\begin{theorem}\label{thm:Cc}
%The constant $C(n,s)$ introduced in \eqref{cnsgalattica} is given by
%	\begin{equation} \label{cnscomputed} C(n,s) = \frac{2^{2s} s \Gamma\left(\frac{n}2 +s\right)} {\pi^{\frac{n}2} \Gamma(1-s)}.\end{equation}
%\end{theorem}

\begin{proof}[Proof of Theorem \ref{thm:Cc}]
By Lemma \ref{ctfrlap} we have that in $B_1$
 \[\frlap u(x) = C(n,s) \frac{\omega_n}{2} B(1-s,s).\]
We use Theorem \ref{theorem:thm2} and for $n\neq 2s$, we obtain that
	\begin{equation*}
		\begin{split}
		u(x) = &\; \int_{B_1} C(n,s) \frac{\omega_n}{2}  B(1-s,s) G(x,y) dy \\
	= &\;  C(n,s) \frac{\omega_n}{2} B(1-s,s)  \kappa(n,s) \int_{B_1} |x-y|^{2s-n} \bigg( \int_0^{r_0(x,y)} \frac {t^{s-1}}{(t+1)^{\frac{n}{2}}}\, dt\bigg) \, dy.
		\end{split}
	\end{equation*}
We compute this identity in zero and have that
 \begin{equation}
		\begin{split}
		&1= C(n,s) \frac{\omega_n}{2} B(1-s,s)  \kappa(n,s) \int_{B_1} |y|^{2s-n} \bigg( \int_0^{\frac{1-|y|^2}{|y|^2}} \frac {t^{s-1}}{(t+1)^{\frac{n}{2}}}\, dt\bigg) \, dy .\label{constant}
		\end{split}
	\end{equation}
We compute the double integral, by using Fubini-Tonelli's theorem
 \begin{equation*}
		\begin{split}
		\int_{B_1} |y|^{2s-n} \bigg( \int_0^{\frac{1-|y|^2}{|y|^2}} \frac {t^{s-1}}{(t+1)^{\frac{n}{2}}}\, dt\bigg) \, dy  =&\; \omega_n   \int_0^1 \rho^{2s-1} \bigg( \int_0^{\frac{1-\rho^2}{\rho^2}} \frac {t^{s-1}}{(t+1)^{\frac{n}{2}}}\, dt\bigg) \, d\rho\\
	 = &\; \omega_n \int_0^{\infty} \frac {t^{s-1}}{(t+1)^{\frac{n}{2}}} \bigg( \int_0^{\frac{1}{\sqrt{t+1}}} \rho^{2s-1} \, d \rho\bigg) \, dt \\=&\; \frac{\omega_n}{2s} \int_0^{\infty} \frac {t^{s-1}}{(t+1)^{\frac{n}{2}+s}}\, dt =   \frac{\omega_n}{2s} B\bigg(s,\frac{n}{2}\bigg).
		\end{split}
	\end{equation*}
%It follows in  that
% \[1= C(n,s) \frac{\omega_n}{2}  B(1-s,s) \kappa(n,s)  \frac{\omega_n}{2s} B\Big(s,\frac{n}{2}\Big).\]
By inserting this, the value of $\kappa(n,s)$ from Theorem \ref{theorem:kns} and the measure of the $(n-1)$-dimensional unit sphere $\omega_n=({ 2\pi^{n/2}})/{\Gamma(n/2)}$ into \eqref{constant} and using \eqref{betagamma} we obtain that
%\[ 1=C(n,s) \frac{\pi^{\frac{n}2}} {\Gamma(\frac{n}2)} \Gamma(s)\Gamma(1-s) \frac{\Gamma(\frac{n}2)}{2^{2s}\pi^{\frac{n}2} \Gamma^2(s)} \frac{\pi^{\frac{n}2}} {s \Gamma(\frac{n}2)}\frac{\Gamma(s)\Gamma(\frac{n}2) }{\Gamma(s+\frac{n}2 ) } ,\]
%therefore
\[ C(n,s) = \frac{2^{2s} s \Gamma(\frac{n}2 +s)} {\pi^{\frac{n}2 }\Gamma(1-s)}.\]
For $n=2s$ we have 	that $\frlap u(x) = C\left(1, 1/2\right) \pi.$
Thanks to Theorem \ref{theorem:thm2}
	\[ u(x) =   C\left(1, \frac{1}2\right) \pi \int_{-1}^1 G(x,y)\, dy.\]
	Using formula \eqref{formn1s12} and computing $u$ at zero, we obtain that
		\[ \begin{aligned} 1 =  C\left(1, \frac{1}2\right)  \int_{-1}^1 \log \frac{ 1+\sqrt{1-y^2}}{|y|}\, dy
%			=&\; 2 C\left(1, \frac{1}2\right)  \int_0^1 \log  \frac{ 1+\sqrt{1-y^2}}{y}\, dy \\
			=  {\pi} C\left(1, \frac{1}2\right) .\end{aligned} \]
			Hence $ C\left(1, 1/2\right)  = {1}/{\pi}$ and this concludes the proof of the Theorem.
%In oder to compute the quotient $\displaystyle \frac{C(n,s)}{c(n,s)}$, we use the value of $c(n,s)$ from definition \eqref{ctcns} and the identity \eqref{gam3}. We obtain that
%\[\frac{ C(n,s)}{c(n,s)}= \frac{ 2^{2s} \Gamma(s+1)\Gamma(n/2+s)} {  \Gamma(n/2)},\]
%thus the desired result. We remark also that the term $\pi^{2s}$ depends on the normalization factor used in the Fourier transform.
\end{proof}

%%%%%%%%%%%%%%%%%%%%%%%%%%%%%%%%%%

\subsection{ Point inversion transformations and some useful integral identities}\label{appy}
The purpose of this subsection is to recall some basic geometric features of the point inversion, related to the so-called Kelvin transformation.

Let $r>0$ to be fixed.
%Unless otherwise specified, the results in the Appendix hold for $n\geq 1$ and any value of $s \in (0,1)$. We prove the results in the case $n>2s$, but they are easily adaptable to the case $n=1$.

\begin{defn}
Let $x_0\in B_r$ be a fixed point. The inversion with center $x_0$ is a point transformation that maps an arbitrary point $y \in {\Rn}\setminus \{x_0\}$ to the point $ \mathbf{K}_{x_0}(y)$ such that the points $y$, $x_0$, $ \mathbf{K}_{x_0}(y)$ lie on one line, $x_0$ separates $y$ and $\mathbf{K}_{x_0}(y)$ and
	\begin{equation}
	\mathbf{K}_{x_0}(y):=x_0- \frac{r^2-|x_0|^2}{|y-x_0|^2} \big( y-x_0\big). \label{trk}
	\end{equation}
\end{defn}
This is a bijective map from ${\Rn}\setminus \{ x_0\} $ onto itself. Of course, $\mathbf{K}_{x_0}\big( \mathbf{K}_{x_0}(x) \big) = x$. When this does not generate any confusion, we will use the notation $y^*:=\mathbf{K}_{x_0}(y)$ and $x^*:=\mathbf{K}_{x_0}(x)$ to denote the inversion of $y$ and $x$ respectively, with center at $x_0$.

\begin{remark}
It is not hard to see, from definition \eqref{trk}, that
	 \begin{equation}
	 |y^*-x_0||y-x_0|=r^2-|x_0|^2. \label{tr}
	\end{equation}

\end{remark}

\begin{prop}
\label{proposition:pointinv}
Let $x_0 \in B_r$ be a fixed point, and $x^*$ and $y^*$ be the inversion of $x \in {\Rn}\setminus \{x_0\}$ respectively $y \in {\Rn}\setminus \{x_0\}$ with center at $x_0$.
Then:
	\begin{subequations}
		\begin{align}
			\text{ a) }  & \text{points on the sphere } \partial B_r \text{ are mapped into points on the same sphere},  \quad \quad \quad\quad \quad \quad \quad \quad \quad \notag %label{mmmm1}
			\\
			 \text{ b) } & \text{points outside the sphere } \partial B_r  \text{ are mapped into points inside the sphere}, \notag % \label{mmmm2}
			 \\
	          		\text{ c) } &\frac{|y-x_0|^2}{(r^2-|x_0|^2)(r^2-|y|^2)}=\frac{1}{|y^*|^2-r^2},  \label{firsttr}	\\
			\text{ d) } & \frac{dy}{|y-x_0|^n}=\frac{dy^*}{|y^*-x_0|^n},  \label{dxtr} \\
			\text{ e) } & |y^*-x^*|=(r^2-|x_0|^2)\frac{|y-x|}{|y-x_0||x-x_0|}.  \label{sectr}	
		\end{align}
	\end{subequations}
\end{prop}

{The Kelvin point inversion transformation is well known (see, for instance, the Appendix in \cite{Landkof}) and elementary geometrical considerations can be used to prove this lemma. We give here a sketch of the proof. }
\begin{proof}[Sketch of the proof]
A simple way to prove claims a) to c) is to consider the first triangle in Figure \ref{fign:figure1}.

\begin{figure}[htpb]
	\hspace{0.85cm}
	\begin{minipage}[b]{0.98\linewidth}
	\centering
	\includegraphics[width=0.98\textwidth]{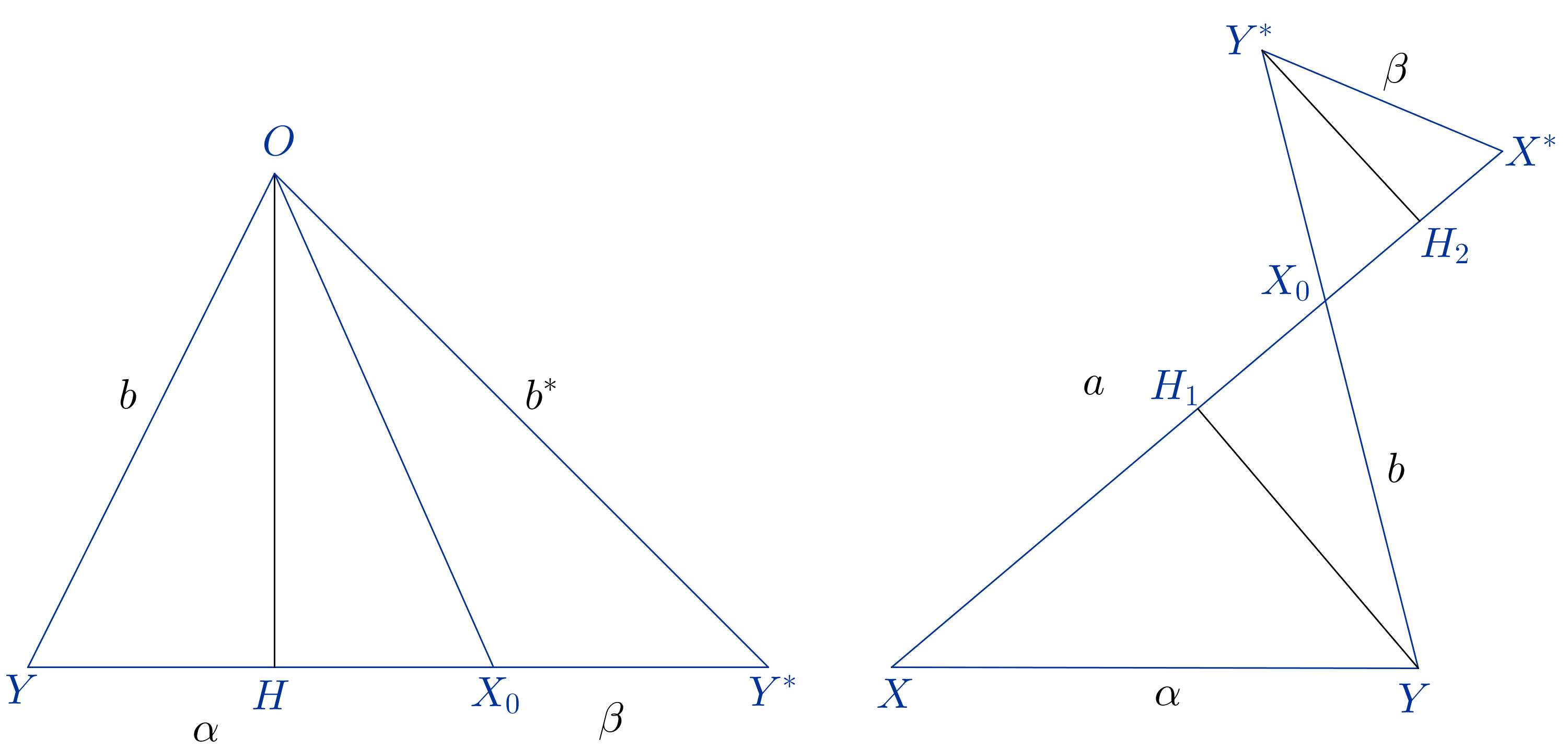}
	\caption{Inversion of $x, y$ with center at $x_0$}
	\label{fign:figure1}
	\end{minipage}
	\end{figure}
%We refer to Figure~\ref{fign:figure1} for the proof of claims a) to c) and to Figure~\ref{fign:figure2} for the claim e).
%
%
We denote
%We consider the triangle $ \triangle OYY^*$ with $X_0 \in YY^*$. Let
% $ a:=|OX_0|=|x_0| < r$,
 $b:=|OY|=|y|$, $b^*:=|OY^*|=|y^*|$, $\alpha:=|X_0Y|=|y-x_0|$ and $\beta := |X_0Y^*|=|y^*-x_0|$.
 Let $OH$ be the perpendicular from $O$ onto $YY^*$.
% , with $h:=|OH|$, and let t:=|YH|.
 We apply the Pythagorean Theorem in the three triangles $\triangle OYH, \, \triangle OHX_0, \, \triangle OHY^*$,  add  the equation \eqref{tr} and by solving the system, one gets that
	 \begin{equation*}
		 {b^*}^2= \frac{\beta r^2+\alpha r^2-\beta b^2}{\alpha}. \label{sol}
	\end{equation*} From this, claims a) to c) follow after elementary computations.

In order to prove d), without loss of generality, one can consider the point inversion of radius one with center at zero
	$y^*={-y}/{|y|^2} $ and take its derivative. Since the point inversion transformation is invariant under rotation, we can assume that $y=|y|e_1$ and the desired result plainly follows.

To prove e), see the Appendix in \cite{Landkof}, or consider the second triangle in Figure \ref{fign:figure1}. We denote
%We consider the triangles $\triangle X_0XY$ and $\triangle X_0X^*Y^*$. Let
$a :=|X_0X|=|x-x_0|$,
%$a^*:=|X_0X^*|=|x^*-x_0|$,
$b:= |X_0Y|=|y-x_0|$,
%$b^*:= |X_0Y^*|=|y^*-x_0|$,
$\alpha :=|XY|=|x-y|$ and $\beta:=|X^*Y^*|=|x^*-y^*|$. Let $YH_1$ and $Y^*H_2$ be perpendiculars from $Y$, respectively $Y^*$ onto the segment $XX^*$.
% and let $t_1:=|X_0H_1|$ and $t_2:= |X_0H_2|$.
 By applying the Pythagorean Theorem in the four triangles $\triangle X_0YH_1,\, \triangle XYH_1, \,  \triangle X_0Y^*H_2$, $\triangle X^*Y^*H_2$, adding relation \eqref{tr}  and using that $YH_1$ is parallel to $Y^*H_2$, one gets after solving the system that
\begin{equation*}  \beta=\frac{(r^2-|x_0|^2) \alpha}{ab}, \end{equation*}	which is the desired result.
%Thus \[ |x^*-y^*|=\frac{(r^2-|x_0|^2) |y-x|}{|y-x_0||x-x_0|},\]
%which completes the proof of Proposition \ref{proposition:pointinv}.\qedhere
\end{proof}

We present here a few detailed computations related to the functions $\Phi$, $A_r$ and $P_r$ and some other useful integral identities.

\begin{lemma} For any $r>0$
\begin{equation}  \int _{\obal} A_r(y) \, dy =1. \label{Ir} \end{equation}
\end{lemma}

\begin{proof}
%Let \[  I_r:= \int _{\obal} A_r(y) \, dy .\]
Using \eqref{smeandefn} and passing to polar coordinates we have that
 	\begin{equation*}
 		\begin{split}
 		 \int _{\obal} A_r(y) \, dy= &\;  c(n,s) \int _{\obal}   \frac {r^{2s} } {(|y|^2 -r^2)^s|y|^n}\, dy\\
 			= &\; c(n,s)  \omega_n  \int_r ^{\infty}   \frac { r^{2s} }{\rho  (\rho^2-r^2)^s} \, d\rho.
 		\end{split}
 	\end{equation*}
%where $\omega_n={ 2\pi^{n/2}}/{\Gamma(n/2)}$  is the measure of the $(n-1)$-dimensional unit sphere.
We change the variable $z= (\rho/r)^2-1$ and have that
	\begin{equation}
\int _{\obal} A_r(y) \, dy
%  =    c(n,s)  \omega_n  \int_1 ^{\infty}  \frac {1}{t(t^2 -1)^s} \, dt
  = \frac {c(n,s)}{2} 	 \omega_n  \int_0 ^{\infty}  \frac  {1} {(z+1) z^s} \, dz .  \label{Ircalc}   	 \end{equation}
%We apply identities \eqref{beta} and \eqref{betagamma} and use identity \eqref{gam3} to obtain that
%	\begin{equation} \int_0^{\infty}    \frac  {1} {(z+1) z^s} \, dz  = \Gamma (1-s) \Gamma(s) =  \frac{\pi}{\sin{\pi s} }.\label{betappl} \end{equation}
Using \eqref{betas} and the definition \eqref{ctcns} of $c(n,s)$, it follows that $ \int _{\obal} A_r(y) \, dy= 1$, as desired.
\end{proof}

\begin{lemma} \label{lemmaip}
For any $r>0$ and any $x \in B_r$
\begin{equation} \int_{\obal}P_r(y,x) \, dy  =1. \label{Ip} \end{equation}
\end{lemma}

\begin{proof}
%Let \[  I_p(x) := \int_{\obal}P_r(y,x) \, dy.\]
%From the definition \eqref{poissondefn} of $P_r$ we have that \[\int_{\obal}P_r(y,x) \, dy= c(n,s) \int_{\obal} \Bigg (\frac {r^2-|x|^2}{|y|^2-r^2}\Bigg)^s \frac {1}{|x-y|^n}  \, dy .\]

We make the proof for $n>3$. However, the results hold for $n\leq 3$. We change variables using the hyperspherical coordinates with radius $\rho>0$ and angles $\theta, \theta_1, \dots, \theta_{n-3} \in [0, \pi], \theta_{n-2} \in [0,2\pi]$
	\begin{equation}
		\begin{split} \label{chofvarhypersph}
		 y_1= &\rho \sin\theta\sin\theta_1 \dots \sin\theta_{n-3} \sin \theta_{n-2}\\
 	             y_2=  &\rho\sin \theta \sin \theta_1 \dots \sin \theta_{n-3} \cos \theta_{n-2} \\
   	             y_3= &\rho \sin \theta \sin \theta_1 \dots \cos \theta_{n-3} \\
   	       \dots \\
		 y_n=&\rho \cos\theta.
		\end{split}
	\end{equation}
	The Jacobian of the transformation is given by $ \rho^{n-1} \sin ^{n-2}\theta  \sin ^{n-3}\theta_1 \dots \sin \theta_{n-3} $.
We only remark that for $n=3$ the usual spherical coordinates can be used $	y_1= \rho \sin\theta\sin\theta_1,  y_2=  \rho\sin \theta \cos \theta_1  \mbox{ and } y_3= \rho \cos \theta $, while for $n=2$ and $n=1$ similar computations can be performed.

Without loss of generality and up to rotations, we assume that $x =|x|e_n$ to obtain the identity $ |x-y|^2 = \rho^2+|x|^2-2|x| \rho\cos\theta$ (see Figure~\ref{fign:xycalc} for clarity). With this change of coordinates, we obtain	
	\begin{equation*}
		\begin{split}
		& \int_{\obal}P_r(y,x) \, dy\\
		=& c(n,s) (r^2  - |x|^2)^s 2\pi \prod_{k=1}^{n-3}\int_0^{\pi} \sin^k \theta \,d\theta
   \int_r^{\infty}\int_0^{\pi}  \frac{\rho^{n-1}\sin^{n-2} \theta \, d\theta \, d\rho }{(\rho^2-r^2)^s (\rho^2+|x|^2 - 2\rho |x| \cos \theta)^{n/2}} .
		\end{split}
	\end{equation*}
%	where, without loss of generality and up to rotations, we have assumed that $x^* =|x^*|e_n$ to obtain the identity $ |x^*-y^*|^2 = \rho^2+|x^*|^2-2|x^*| \rho\cos\theta$ (refer to Figure \ref{fign:xycalc}).
	\begin{figure}[htb]
	\hspace{0.85cm}
	\begin{minipage}[b]{0.68\linewidth}
	\centering
	\includegraphics[width=0.68\textwidth]{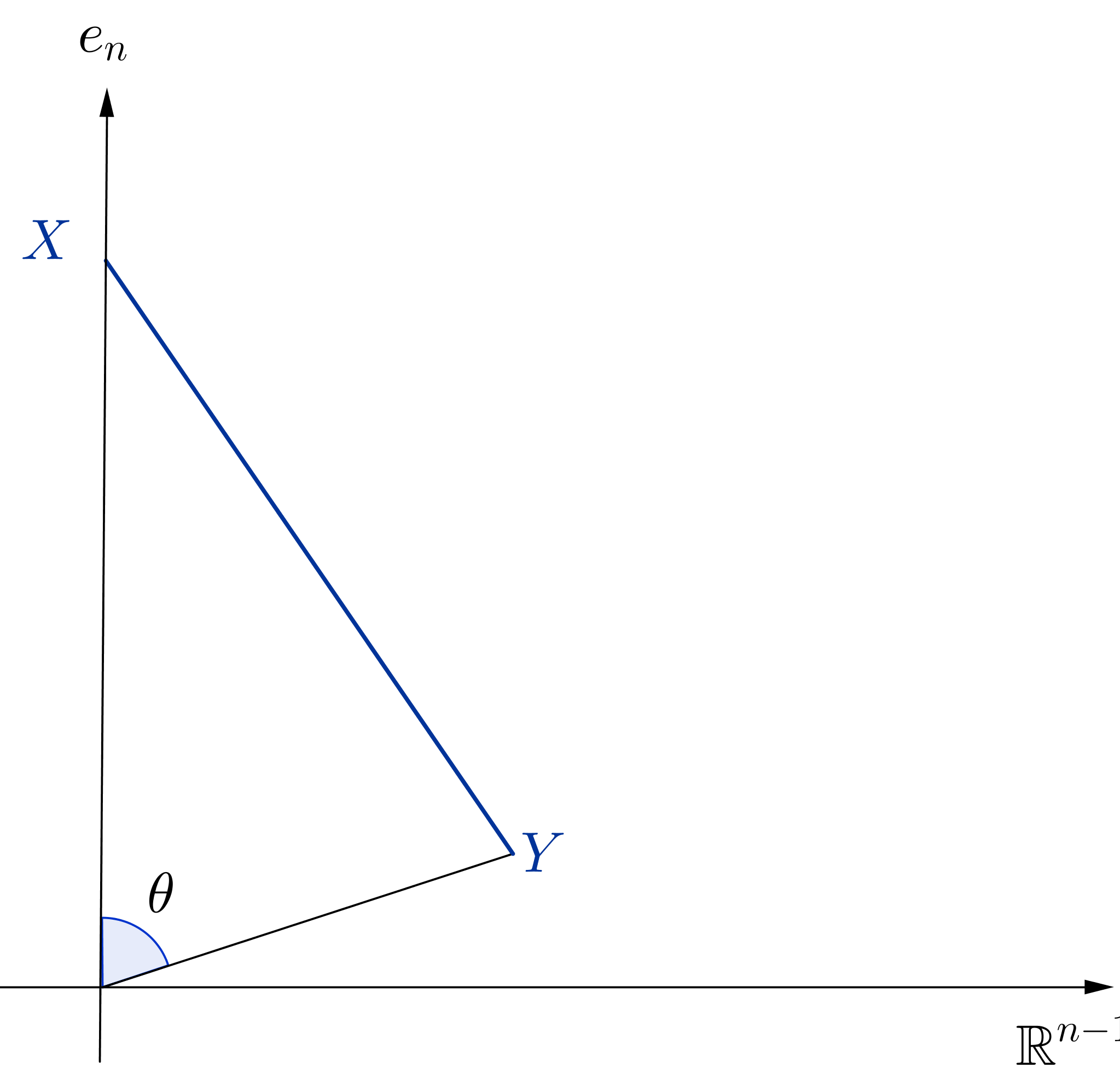}
	\caption{}
	\label{fign:xycalc}
	\end{minipage}
	\end{figure}
We do the substitution $\bar r= {r}/{|x|} $ and $\bar \rho= {\rho}/{ |x|}$ but still use $r$ and $\rho$ for simplicity and we remark that now $\rho>1$ and $r>1$.
 We obtain that
%	\begin{equation*}
%		\begin{split}
%		  \int_{\obal}P_r(y,x) \, dy
%		  = &\; c(n,s) (r^2-1)^s 2\pi \prod_{k=1}^{n-3}\int_0^{\pi} \sin^k\theta_{n-k-2}\, d\theta_{n-k-2}  \\
% &\;  \int_{r}^{\infty} \int_0^{\pi} \frac {|x|^{n-1} \rho^{n-1} \sin^{n-2} \theta \, |x| }{|x|^n \,  (\rho^2-r^2)^s(\rho^2+1-2\rho\cos \theta)^{n/2} }\,  d\theta \, d\rho.
%		\end{split}
%	\end{equation*}
%We use, for simplicity, $\rho$ instead of $\rho_1$, and $r$ instead of $r_1$,  We have that
	\begin{equation}
\label{ipxxx}
		\begin{split}
	&  \int_{\obal}P_r(y,x) \, dy\\
%	  = &\;c(n,s) (r^2-1)^s 2\pi \prod_{k=1}^{n-3}\int_0^{\pi}  \sin^k\theta_{n-k-2} \,d\theta_{n-k-2} \\
%		&\;  \int_r^{\infty} \int_0^{\pi} \frac  {\rho^{n-1} \sin^{n-2}\theta}{(\rho^2-r^2)^s(\rho^2+1-2\rho \cos\theta)^{n/2}}\, d\theta \, d\rho \\
		 =&\; c(n,s) (r^2-1)^s 2\pi  \prod_{k=1}^{n-3}\int_0^{\pi}  \sin^k\theta\,  d\theta
\int_r^{\infty} \frac {\rho^{n-1}}{(\rho^2-r^2)^s}\bigg( \int_0^{\pi} \frac  {\sin^{n-2}\theta \,d\theta} {(\rho^2+1-2\rho \cos\theta )^{n/2}}  \bigg)\, d\rho.
		\end{split}
	\end{equation}
Let
	\begin{equation*}
	 i (\rho):=\int_0^{\pi} \frac {\sin^{n-2}\theta}{(\rho^2-2\rho \cos\theta +1)^{n/2}} \, d\theta.
	\end{equation*}
We claim that, given that $\rho >1$
\begin{equation} \label{aahhh}
	 i(\rho)=\frac {1} {\rho^{n-2} (\rho^2-1)}\int_0^{\pi} \sin^{n-2} \theta \, d\theta.
	\end{equation}
To prove this, we use the following change of coordinates
\begin{equation}  \frac {\sin\theta}{\sqrt {\rho^2 -2\rho \cos\theta +1 }} = \frac {\sin \alpha }{\rho}. \label{changeofvar} \end{equation}
We have that
	\begin{equation} d\theta = \Bigg( 1- \frac{\cos \alpha}{\sqrt{\rho^2-\sin^2\alpha}}\Bigg) \, d\alpha. \label{chvar} \end{equation}
%	Since \[1- \frac{\cos \alpha}{\sqrt{\rho^2-\sin^2\alpha}}>0,\] the changes of variables is well defined.
{ To see this, one takes the derivative of the relation \eqref{changeofvar}
	\begin{equation}
		 \frac{(\rho \cos \theta -1)(\rho - \cos \theta)} {(\rho^2- 2\rho \cos \theta +1)^{\frac{3}{2}}}\, d\theta = \frac{\cos \alpha}{\rho}\, d \alpha \label{capduedue}
	\end{equation}
and obtains with some manipulations of  identity \eqref{changeofvar} that}
%	 \[ \frac{\cos^2 \alpha}{\rho^2} = \frac{(\rho \cos \theta -1)^2}{\rho^2(\rho^2- 2\rho \cos \theta +1)}  \]
%and also
%	\[ 1- \frac{\sin^2 \alpha}{\rho^2} =\frac {(\rho-\cos \theta)^2}{\rho^2- 2\rho \cos \theta +1}.\]
%Therefore we have that
%	\begin{equation} \cos \alpha = \frac{\rho \cos \theta -1 }{(\rho^2- 2\rho \cos \theta +1)^{\frac{1}{2}}} \label{cos123} \end{equation}
%and also 		\begin{equation} \sqrt{\rho^2 - \sin^2\alpha } = \frac{\rho (\rho-\cos \theta) }{(\rho^2- 2\rho \cos \theta +1)^{\frac{1}{2}}} \label{sin}.
%			\end{equation}
%We then compute
%	\begin{equation} \sqrt{\rho^2 - \sin^2\alpha} - \cos \alpha = (\rho^2- 2\rho \cos \theta +1)^{\frac{1}{2}} \label{prodsincos}  \end{equation}
%and notice that by multiplying \eqref{cos123} and \eqref{sin} and dividing by \eqref{prodsincos} we obtain that
\[  \frac{ \cos \alpha  \sqrt{\rho^2 - \sin^2\alpha} }{ \sqrt{\rho^2 - \sin^2\alpha} -\cos \alpha} =\frac{\rho(\rho \cos \theta -1)(\rho - \cos \theta)} {(\rho^2- 2\rho \cos \theta +1)^{\frac{3}{2}}} .\]
%By inserting this into \eqref{capduedue} we arrive at \eqref{chvar}, as desired.
Now by changing variables we obtain that
	\begin{equation*}
		\begin{split}
		 i(\rho)= &\;\int_0^{\pi} \frac {\sin^{n-2}\theta}{(\rho^2-2\rho \cos\theta +1)^{n/2}} \, d\theta  \\
		=&\; \frac{1}{\rho^{n-2}} \int_0^{\pi}  \frac{\sin ^{n-2}\alpha \, d\alpha}{(\sqrt{\rho^2- \sin^2 \alpha} -\cos \alpha ) \sqrt{\rho^2- \sin^2 \alpha}} \\
= &\; \frac{1}{\rho^{n-2}} \int_0^{\pi}   \frac{\sin ^{n-2}\alpha (\sqrt{\rho^2- \sin^2 \alpha} +\cos \alpha ) \, d\alpha}{ (\rho^2-1) \sqrt{\rho^2- \sin^2 \alpha}} \\
= &\;\frac{1}{\rho^{n-2}(\rho^2-1) } \bigg( \int_0^{\pi} \sin ^{n-2} \alpha  \, d\alpha + \int_0^{\pi} \frac{\sin^{n-2} \alpha \, \cos \alpha}{   \sqrt{\rho^2- \sin^2 \alpha}}\, d\alpha\bigg).
		\end{split}
	\end{equation*}
By symmetry \[\int_0^{\pi} \frac{\sin^{n-2} \alpha \, \cos \alpha}{   \sqrt{\rho^2- \sin^2 \alpha}}\, d\alpha  =0,\]
therefore \[ i(\rho)= \frac{1}{\rho^{n-2}(\rho^2-1) }  \int_0^{\pi} \sin ^{n-2} \alpha  \, d\alpha .\]
We substitute this into \eqref{ipxxx} and obtain that
%	\begin{equation*}
%		\begin{split}
%		\int_{\obal}P_r(y,x) \, dy= &\; c(n,s) (r^2-1)^s 2\pi \prod_{k=1}^{n-3} \int_0^{\pi}\sin^k\theta_{n-k-2} d\theta_{n-k-2} \\
%	 &\; \int_r^{\infty} \frac{\rho^{n-1}}{(\rho^2-r^2)^s} \frac {1}{\rho^{n-2} (\rho^2-1)}
%		 \Big( \int_0^{\pi} \sin^{n-2} \alpha \,  d\alpha \Big) \, d\rho.
%		\end{split}
%	\end{equation*}
	\[\int_{\obal}P_r(y,x) \, dy=c(n,s) (r^2-1)^s 2\pi \prod_{k=1}^{n-2} \int_0^{\pi} \sin^k\theta\,  d\theta  \int_r^{\infty} \frac{\rho\,  d\rho}{  (\rho^2-r^2)^s (\rho^2-1)}.\]
We claim that
	\begin{equation}
	\pi \prod_{k=1}^{n-2} \int_0^{\pi}\sin^k\theta d\theta = \frac{{\pi}^{n/2}}{\Gamma (n/2)}.
	\label{sinstuff}
	\end{equation}
To prove this, we integrate by parts and obtain that
	\begin{equation*}
		\begin{split}
		 I_k=  \int_0^{\pi} \sin^k\theta \, d\theta
%		 = \int_0^{\pi}  \sin^{k-1}\theta (-\cos\theta)' d\theta \\
%		= -\cos\theta\sin^{k-1}\theta \Big|_0^{\pi} +(k- 1) \int_0^{\pi} \sin^{k-2}\theta \cos^2 \theta \, d\theta \\
		=(k-1) \int_0^{\pi} \sin^{k-2}\theta \, d\theta - (k-1)\int_0^{\pi}  \sin^k \theta \, d \theta  ,
		\end{split}
	\end{equation*}
which implies that \[  I_k=  \frac{k-1}{k}\int_0^{\pi}  \sin^{k-2}\theta \, d\theta = \frac{k-1}{k} I_{k-2}.  \]
Thus we have
	\begin{equation*}
                         I_k=
		\displaystyle \begin{cases}
			\displaystyle \frac{k-1}{k}\frac{k-3}{k-2}\dots \frac{1}{2}I_0 \quad   \text{  if $k$ even}, \\
			\displaystyle  \frac{k-1}{k}\frac{k-3}{k-2}\dots \frac{2}{3}I_1\quad  \text{  if $k$ odd},
		 \end{cases}
	\end{equation*}
with $I_0=\pi$ and $I_1=2$, and the claim \eqref{sinstuff} follows after elementary computations. And so
	\[\int_{\obal}P_r(y,x) \, dy=c(n,s) (r^2-1)^s  \frac{\pi^{n/2}}{\Gamma(n/2)}\int_r^{\infty} \frac {2\rho}{(\rho^2-r^2)^s(\rho^2-1)} \, d\rho. \]
We change variable $ \frac{\rho^2-r^2}{r^2-1}=z$ and obtain
	\begin{equation*}
		\begin{split}
                              \int_{\obal}P_r(y,x) \, dy
%                              &\; c(n,s) (r^2-1)^s  \frac{\pi^{n/2}}{\Gamma(n/2)}\int_0^{\infty} \frac { 1}{t^s (t+r^2-1)}\, dt\\
                               	=   c(n,s)   \frac{\pi^{n/2}}{\Gamma(n/2)} \int_0^{\infty} \frac {1}{z^s \, (z+1)}\, dz.
		\end{split}
	\end{equation*}
We use \eqref{betas} and the value of $c(n,s)$ from \eqref{ctcns} and obtain that
	\begin{equation*} \int_{\obal}P_r(y,x) \, dy= 1. \end{equation*}
This completes the proof of Lemma \ref{lemmaip}.
\end{proof}

\begin{lemma}
For any $r>0$ and any $x \in B_r $
	\begin{equation}  c(n,s) \int_{B_r} \frac {dy}{(r^2-|y|^2)^s |x-y|^{n-2s}} =1 \label{If} .\end{equation}
\end{lemma}

\begin{proof}
%Let \begin{equation*}
%		\begin{split}
%	I_f(x):=& \;c(n,s) \int_{B_r} \frac {dy}{(r^2-|y|^2)^s |x-y|^{n-2s}} \\
%		=&\;c(n,s) \int_{B_r} \bigg(\frac{|x-y|^2}{r^2-|y|^2} \bigg)^s\frac {dy}{|x-y|^n}  .
%	\end{split}
%	\end{equation*}
Let $y^*$ be the inversion of $y$ with center at $x$ (notice that $ |y^*|>r$). Then by using \eqref{firsttr}
%we have that  \[ \frac{|x-y|^2}{r^2-|y|^2}= \frac{r^2-|x|^2}{|y^*|^2-r^2},\]
and \eqref{dxtr}
% 	\[ \frac{dy}{|x-y|^n}=\frac{dy^*}{|x-y^*|^n} .\]
we obtain that
	\begin{equation*}\int_{B_r} \frac {dy}{(r^2-|y|^2)^s |x-y|^{n-2s}}=\int_{\obal} \Bigg(\frac {r^2-|x|^2}{|y^*|^2-r^2}\Bigg)^s \frac{dy^*}{|x-y^*|^n}.
	\end{equation*}
From identity \eqref{Ip} the desired result immediately follows.
\end{proof}

\begin{lemma}
For any $r>0$ and any $x \in \obal$
\begin{equation}\int_{\obal}A_r(y) \Phi(x-y)\, dy = \Phi(x)\label{Ifu}. \end{equation}
\end{lemma}

\begin{proof}
%Let \[ I_f^u(x):=\int_{\obal}A_r(y) \Phi(x-y)\, dy .\]
We prove the claim for $n\neq 2s$. We insert definitions \eqref{smeandefn} and \eqref{fundsolution} and obtain that
 \begin{equation*}
	\begin{split}
			\int_{\obal}A_r(y) \Phi(x-y)\, dy  =  r^{2s}c(n,s) a(n,s)\int_{\obal} \frac{1}{(|y|^2-r^2)^s|y|^n |x-y|^{n-2s}}\, dy.
%		=&\; r^{2s}c(n,s) a(n,s)\int_{\obal} \frac{1}{(|y|^2-r^2)^s}\frac{1}{|x-y|^{n-2s}} \frac{dy}{|y|^n}.
	\end{split}
\end{equation*}
Let $x^*$ and $y^*$ be the inversion of $x$, respectively $y$ with center at $0$. Using identities \eqref{tr},
%we have that  \[  |y||y^*|=r^2, \quad |x||x^*|=r^2,\]  with $|x^*|< r$.
\eqref{firsttr}
% we have\[ |y|^2-r^2=\frac {r^2 \big( r^2-|y^*|^2 \big)} {|y^*|^2},\]
 \eqref{sectr}
  %\[ |x-y|= r^2 \frac{|x^*-y^*|}{|x^*||y^*|}.\]
and \eqref{dxtr}
%, we have that \[\frac{dy}{|y|^n}=\frac{dy^*}{|y^*|^n}.\]
we obtain that	 \begin{equation*}
		\begin{split}
			\int_{\obal}A_r(y) \Phi(x-y)\, dy
%			&\; r^{2s}c(n,s) a(n,s)  \int_{B_r}  \frac{   |y^*|^{2s}\,  |x^*|^{n-2s}\,|y^*|^{n-2s} }  { r^{2s}\, \big( r^2-|y^*|^2 \big)^s \, r^{2(n-2s)} \, |x^*-y^*|^{n-2s} }\, \frac{dy^*}{|y^*|^n} \\
%			= &\;c(n,s) a(n,s)  \bigg( \frac{|x^*|}{r^2}\bigg)^{n-2s}\int_{B_r} \frac{dy^*}{|x^*-y^*|^{n-2s}\,  \big( r^2-|y^*|^2 \big)^s} \\
			= \frac{c(n,s)   a(n,s)}{|x|^{n-2s}}\int_{B_r} \frac{dy^*}{|x^*-y^*|^{n-2s}\,  \big( r^2-|y^*|^2 \big)^s}.
				\end{split}
	\end{equation*}
From \eqref{If} it follows that
 	\begin{equation*}
			\int_{\obal}A_r(y) \Phi(x-y)\, dy =  \frac{a(n,s)}{|x|^{n-2s}},
	\end{equation*}
and thus the desired result.

We now prove the claim for $n=2s$, assuming $r=1$. We have that
\[ \int_{\obal}A_r(y) \Phi(x-y)\, dy =  \frac{-1}{\pi^2}\int_{|y|>1} \frac {\log|y-x|}{ \sqrt{ y^2-1} |y|} dy.\]
We perform the change of variables $v= \frac{1}{y}$, with $|v|\leq 1$. We set $w:= \frac{1}{x}$, hence $|w|\leq 1$.
%We obtain that
%	\begin{equation*}
%		\begin{split}
%		  &dy= \frac{-1}{v^2}dv, 	\\
%		&y^2-1= \frac{1-v^2}{v^2}, \\
%		&x-y= \frac{v-w}{v}x.		
%		\end{split}
%	\end{equation*}
Then we have that
	\[ \int_{\obal}A_r(y) \Phi(x-y)\, dy   = \frac{-1}{\pi^2} \int_{|v|\leq 1}\bigg( \log \frac{|v-w|}{|v|} + \log|x|\bigg)\frac{dv}{ \sqrt{1-v^2}}.\]
We use the following result (see \cite{conto}, page 549)
\begin{equation}
	   \int_{|v|\leq 1} \log |v-a| \frac{dv}{ \sqrt{1-v^2}}=  \begin{cases}  -\pi \log 2, &\text{ if } |a|\leq 1\\
								 \pi \log (|a|+(a^2-1)^{1/2}) - \pi \log 2, &\text{ if } |a|\geq 1.
	\end{cases}
	\label{logid}
	\end{equation}
We thus obtain \[\int_{\obal}A_r(y) \Phi(x-y)\, dy = -\frac{1}\pi\log|x|,\]
which concludes the proof of the Lemma.
\end{proof}

\begin{lemma}
For any $r>0$, let $x_0 \in B_r$ be a fixed point.  For any $x \in \obal$
	 \begin{equation} \int_{\obal}P_r(y,x_0) \Phi(x-y) \, dy  = \Phi(x-x_0). \label{Ipu} \end{equation}
\end{lemma}

\begin{proof}
%Let \[ I_p^u(x) := \int_{\obal}P_r(y,x_0) \Phi(x-y) \, dy .\]
We prove the claim for $n\neq 2s$.
We have that  \begin{equation*}
	\begin{split}
		\int_{\obal}P_r(y,x_0) \Phi(x-y) \, dy = 	c(n,s) a(n,s) \int_{\obal} \frac{(r^2-|x_0|^2)^s|x-y|^{2s-n}\, dy} {(|y|^2-r^2)^s|y-x_0|^n} .
	\end{split}
\end{equation*}
Let $x^*$ and $y^*$ be the inversion of $x$, respectively $y$ with center at $x_0$. From \eqref{tr},
%we recall that
%	\begin{equation*}
%		  |y-x_0||y^*-x_0|= r^2-|x_0|^2,  \quad  |x-x_0||x^*-x_0|=r^2-|x_0|^2,
%	\end{equation*}
%with $|x^*|< r$.
%By
 \eqref{firsttr}
% we have that
%		\[ |y|^2-r^2 =\frac{(r^2-|x_0|^2)(r^2-|y^*|^2)}{|y^*-x_0|^2}.\]
%Then,
\eqref{dxtr} and
% \[  \frac{dy}{|y-x_0|^n}=\frac{dy^*}{|y^*-x_0|^n}, \]
%and by
\eqref{sectr}
 we have that
%	 \[ |x-y|= (r^2-|x_0|^2)\frac{|x^*-y^*|}{|x^*-x_0||y^*-x_0|}.\]
%Consequently
	\begin{equation*}\begin{split}
				&  \int_{\obal}P_r(y,x_0) \Phi(x-y) \, dy \\ =&\; c(n,s) a(n,s) \frac{|x^*-x_0|^{n-2s}  }{(r^2-|x_0|^2)^{n-2s}  } \int_{B_r}  \frac {dy^*}{(r^2-|y^*|^2)^s|y^*-x^*|^{n-2s}}.
		\end{split} \end{equation*}
Using \eqref{If}, we obtain that \begin{equation*} \int_{\obal}P_r(y,x_0) \Phi(x-y) \, dy =  \frac{a(n,s)}{|x-x_0|^{n-2s}},\end{equation*}
which concludes the proof for $n\neq 2s$.

We now prove the claim for $n=2s$, assuming $r=1$. We have that
\[ \int_{\obal}P_r(y,x_0) \Phi(x-y) \, dy = \frac{-1}{\pi^2} \int_{|y|>1} \sqrt{ \frac{1-{x_0}^2}{y^2-1}}  \frac{\log|y-x|}{|y-x_0|} dy.\]
We perform the change of variables $v=\frac{yx_0-1}{y-x_0}$, noticing that $|v|\leq 1$. We set $w:= \frac{x x_0-1}{x-x_0}$, hence $|w|\leq 1$.
%We obtain that
%	\begin{equation*}
%		\begin{split}
%		  &dy= \frac{1-x_0^2}{(v-x_0)^2} dv, 	\\
%		&y^2-1= \frac{(1-v^2)(1-x_0^2)}{(v-x_0)^2}, \\
%		&x-y= \frac{v-w}{v-x_0}(x-x_0),\\
%		&x_0-y= \frac{1-x_0^2}{v-x_0}.
%		\end{split}
%	\end{equation*}
Then we have that
	\[\int_{\obal}P_r(y,x_0) \Phi(x-y) \, dy =\frac{-1}{\pi^2} \int_{|v|\leq 1} \bigg( \log  \frac{|v-w|}{|v-x_0|}+ \log |x-x_0|\bigg) \frac{ dv}{\sqrt{1-v^2}}.\]
We use identity \eqref{logid} and obtain \[\int_{\obal}P_r(y,x_0) \Phi(x-y) \, dy = -\frac{1}{\pi}\log|x-x_0|,\]
which concludes the proof.
\end{proof}
 We emphasize here two computations that we used in the proof of Lemma \ref{lemmaip}, namely identities \eqref{aahhh} and \eqref{sinstuff}.

\begin{prop} For any $\tau >1$
	\begin{equation} \label{prop1}  \int_0^{\pi} \frac {\sin^{n-2}\theta}{(\tau^2-2\tau \cos\theta +1)^{n/2}} \, d\theta = \frac {1} {\tau^{n-2} (\tau^2-1)}\int_0^{\pi} \sin^{n-2} \alpha \, d\alpha .\end{equation}
\end{prop}

\begin{prop}
\begin{equation} \pi \prod_{k=1}^{n-2} \int_0^{\pi}\sin^k\theta \, d\theta = \frac{{\pi}^{n/2}}{\Gamma (n/2)}.\label{prop2} \end{equation}
\end{prop}
In the next Proposition we introduce yet another integral identity.
\begin{prop}
\label{proposition:uss} Let $\alpha, \beta \in \mathbb{R}$ such that $ \big| \frac{\alpha}{\alpha +\beta}\big| < 1$. Then
\[  \int_0^\alpha \frac{(\alpha-x)^{s-1}}{x^s(\beta+x)}\, dx  = \frac{\pi}{\sin(\pi s)} \frac{(\alpha+\beta)^{s-1}} {\beta^s}.\]
\end{prop}

\begin{proof}

We change the variable $x =\alpha t$ and obtain that
	\begin{equation*}
		 \int_0^\alpha \frac{(\alpha-x)^{s-1}}{x^s(\beta+x)}\, dx = \frac{1}{\beta}  \int_0^1 t^{-s} (1-t)^{s-1} \bigg(1+ \frac{\alpha}{\beta}t\bigg)^{-1}\, dt.
	\end{equation*}
We use the integral definition \eqref{inthyp} of the hypergeometric function for $a=1$, $b=1-s$, $c=1$ and $w= -\frac{\alpha}{\beta}$ (since $|t|<1$, the integral is convergent) and we obtain that
	\[  \int_0^1 t^{-s} (1-t)^{-s} \bigg(1+ \frac{\alpha}{\beta}t\bigg)^{-1}\, dz =  \frac{\Gamma(s) \Gamma(1-s)}{\Gamma(1)} F\bigg(1,1-s,1,-\frac{\alpha}{\beta}\bigg).\]
Now we use the linear transformation \eqref{hyp3} and compute
	\[ F\bigg(1,1-s,1,-\frac{\alpha}{\beta}\bigg) = \bigg(\frac{\alpha+\beta}{\beta} \bigg)^{s-1} F\bigg(1-s,0,1,\frac{\alpha}{\alpha+\beta}\bigg)  .\]
We use the Gauss expansion in \eqref{gausshyp} and notice that for $k>0$, all the terms of the sum vanish. We are left with only with the term $k=0$ and obtain that
	\[ F\bigg(1-s,0,1,\frac{\alpha}{\alpha+\beta}\bigg)=1.\]
Furthermore, from \eqref{betas} it follows that
	\begin{equation*}\int_0^\alpha \frac{(\alpha-x)^{s-1}}{x^s(\beta+x)}\, dx=  \frac{\pi}{\sin \pi s} \frac{(\alpha+\beta)^{s-1}}{\beta^s}.\qedhere \end{equation*}
\end{proof}

We explicitly compute here another integral that was used in our computations, namely :

\begin{prop}For any $s\in (0,1/2]$ we have that
\begin{equation} \label{ctcomp1111} \int_0^\infty t^{2s-2} \sin t \, dt =-\cos(\pi s) \Gamma(2s-1).\end{equation}
\end{prop}
\begin{proof}
We have that
\begin{equation}\label{sti1} \int_0^{\infty} t^{2s-2} \sin t \, dt= - \text{Im} \int_0^{\infty}  t^{2s-2} e^{-it}\, dt. \end{equation}
We consider the closed path $\Omega_{\rho} =\partial\Big( \big([0,\rho]\times [0,\rho] \big) \cap B_\rho(0) \Big)$. We take the contour integral $\int_{\Omega_{\rho}} z^{2s-2} e^{-z}\, dz $, and let $\gamma_{\rho} = \partial  B_\rho(0) \cap  \big([0,\rho]\times [0,\rho] \big) $ (the boundary of the quarter of the circle). By Cauchy's Theorem, the contour integral is 0 (there are no poles inside $\Omega_{\rho}$), therefore
\[ \int_0^{\rho} t^{2s-2}e^{-t} \, dt + \int_{\gamma_{\rho}} z^{2s-2} e^{-z}\, dz - i \int_0^{\rho} (it)^{2s-2} e^{-it}\, dt =0.\] Integrating along $\gamma_{\rho}$, by using polar coordinates $z={\rho}e^{i\theta}$ and then the change of variables $\cos \theta =t$ we have that
	\[ \begin{split} \bigg| \int_{\gamma_{\rho}} z^{2s-2} e^{-z} \, dz \bigg| =&\; \bigg| \int_0^{\pi/2} {\rho}^{2s-1} e^{i\theta (2s-1)} e^{-{\rho}e^{i\theta}} \, d\theta\bigg|
	\leq {\rho}^{2s-1} \bigg| \int_0^{\pi/2} e^{-{\rho}\cos \theta}\, d\theta \bigg| \\
	 				=&\; {\rho}^{2s-1} \bigg| \int_0^1 \frac{ e^{-{\rho}t}}{\sqrt{1-t^2} }\, dt \bigg| \\
	 				\leq&\; {\rho}^{2s-1} e^{-{\rho}/2}  \bigg| \int_{1/2}^1 (1-t)^{-1/2}\, dt \bigg|   + \overline  c {\rho}^{2s-1}  \bigg|  \int_0^{1/2}  e^{-{\rho}t} \, dt \bigg| \\
	 				=&\; c  {\rho}^{2s-1} e^{-{\rho}/2} + \overline c {\rho}^{2s-2} (e^{-{\rho}/2}-1) .
\end{split}\]  Hence \[ \lim_{{\rho}\to \infty} \int_{\gamma_{\rho}} z^{2s-2} e^{-z} \, dz =0\]
and we are left only with the integrals along the real and the imaginary axis, namely
	\[ \int_0^{\infty} t^{2s-2}e^{-t} \, dt = i^{2s-1} \int_0^{\infty} t^{2s-2} e^{-it} \, dt .\] Here the left hand side returns the Gamma function according to definition \eqref{ABRAMOWITZ}. We compute $i^{1-2s} =\big( \cos (\pi/2) + i\sin(\pi /2)\big) ^{1-2s}= \sin(\pi s)+ i\cos(\pi s)$ and in \eqref{sti1} we obtain that
		 \[   \int_0^{\infty} t^{2s-2} \sin t \, dt= - \Gamma(2s-1)  \text{Im} \Big(   \sin(\pi s)+ i\cos(\pi s)\Big) = -\cos(\pi s) \Gamma(2s-1).\]
		 This concludes the proof of the Proposition.
\end{proof}
\section{Schauder estimates for the fractional Laplacian}
%\begin{abstract}
%We present an elementary approach for the proof of  Schauder estimates for the equation
%$(-\Delta)^s u(x)=f(x), \,0<s<1$, with $f$ having a modulus of continuity $\omega_f$, 
%based on the Poisson representation formula and 
%dyadic ball approximation argument. We give the explicit modulus of continuity of 
%$u$ in balls $B_r(x)\subset \R^n$ in terms of $\omega_f$.
%\end{abstract}
In what follows we assume that $n\geq 2$ and consider $s\in (0,1)$ to be a fixed quantity.
Let $f$ be a given H\"{o}lder continuous function and $u$ solving
\eqlab{\label{fondsolcon} (-\Delta)^s u=f \quad \mbox{ in } \; B_1.}
%, with $f$ having a modulus of continuity $\omega_f$, 
We want to study the regularity of $u$ using a very simple method 
based on the Poisson representation formula
and dyadic ball approximation argument. 
%In order to formulate our results it is convenient to introduce some notations and 
%basic knowledge on the fractional Laplacian and on some related kernels. 
For regularity up to the boundary of weak solutions of the Dirichlet problem, see the very nice paper \cite{oton}.  

More precisely, we prove here that given  $f\in C^{0,\alpha}(B_1)\cap C(\overline B_1)$, then on the half ball $u$ has the regularity of $f$ increased by $2s$. 
\begin{theorem} \label{schthm}Let $s\in(0,1)$,  $\alpha <1$ and $f\in C^{0,\alpha}(B_1)\cap C(\overline B_1)$ be a given function with modulus of continuity $\omega(r):=\sup_{|x-y|<r} |f(x)-f(y)|$.
 Let  $u\in  L^{\infty}(\Rn)\cap C^1(B_1)$ be a pointwise solution of $\frlap u=f$ in $B_1$. Then for any $x,y \in B_{1/2}$ and denoting $\delta:=|x-y|$ we have that for $ s\leq 1/2$ 
  \eqlab{\label{mm1}    |u(x)-u(y)|\leq &\; C_{n,s}  \bigg( \delta  \|u\|_{L^{\infty}(\Rn\setminus B_1)} +\delta \sup_{\overline B_1}|f| + \int_0^{c \delta} \omega(t)t^{2s-1}\, dt + \delta \int_{\delta}^1  \omega(t)t^{2s-2}\, dt\bigg) }
  while for $s>1/2$ 
  \eqlab{\label{mm} 
 |Du(x)-Du(y)|\leq &\; C_{n,s}  \bigg(\delta  \|u\|_{L^{\infty}(\Rn\setminus B_1)} +\delta \sup_{\overline B_1}|f| + \int_0^{c \delta} \omega(t)t^{2s-2}\, dt+ \delta \int_{\delta}^1  \omega(t)t^{2s-3}\, dt\bigg) ,}
 where $C_{n,s}$ and $ c$ are positive dimensional constants.
\end{theorem}

 There are other approaches to prove Schauder estimates for the fractional order operators with more general kernels 
 see  \cite{Dong} and references therein. Here we follow the one proposed by Xu-Jia Wang in \cite{wang} which is based only on the 
 higher order derivative estimates, that we state here in Lemma \ref{estimate1} and on a maximum principle, given in Lemma \ref{estimate2}. 
  \smallskip

 \medskip 
 
%--------------------------------------
% SUBSECTION Active scalars
%--------------------------------------
%\section{Active scalar equation}
One of the motivations to study \eqref{fondsolcon} comes from 
the active scalars (see \cite{Constantin}). The 2D incompressible Euler equation  
\begin{eqnarray}
\left\{
\begin{array}{lll}
\omega_t+v\nabla \omega=0\\
v=(\partial_2\psi, -\partial_1\psi)\\
\omega=\Delta\psi
\end{array}
\right.
\end{eqnarray}
is one of the well-known active scalar equations. Here $v$ is the 
velocity, $\omega$ the vorticity, $\psi$ the stream function. 

\smallskip

The uniqueness was proved  by Yudovich (see \cite{russo}) under the condition that 
$\omega(t)\in L^{\infty}(0, T; L^1(\R^2))\cap  L^{\infty}(0, T; L^\infty(\R^2))$. 
Observe that by the Biot -Savart law one has that $v=k*\omega$, where 
$$k(x)=\frac{x^\perp}{2\pi |x|^2}.$$ 
Clearly $k\in L^p_{loc}(\R^2), \,1\leq p<2$
and $k\in L^q(\R^2),\, q>2$ near infinity, implying that one must assume 
$\omega\in L^{p_o}(\R^2)\cap L^{q_o}(\R^2), \,p_o<2<q_o$ to make sure that $v=k*\omega$
is well defined.  In particular $p_o=1,\, q_o=\infty$ will do.

\medskip 

A generalization of the 2D Euler equation is the 
quasigeostrophic active scalar 
\begin{eqnarray}
\left\{
\begin{array}{lll}
\omega_t+v\nabla \omega=0\\
v=(\partial_2\psi, -\partial_1\psi)\\
-\omega=(-\Delta)^\frac12 \psi
\end{array}
\right.
\end{eqnarray}
or more generally when one takes $-\omega=(-\Delta)^s \psi, \,0<s<1$. 
Thus this leads to the study of $k_s*(\Delta)^{-s}\omega$
where $n=2$ and 
$$k_s(x)=\nabla^\perp \frac{C_{n,s}}{|x|^{n-2\sigma}}.$$
We see that the regularity of the stream function can be 
concluded from that of $\omega$ via Schauder estimates.
\medskip 

%-------------------------------------
% Schauder history
%-------------------------------------$v_t+v\nabla v=\nabla p$ writing it in terms of vorticity $\omega$ 

%--------------------------------------
% SECTION
%--------------------------------------
 \subsection{H\"{o}lder estimates for the Riesz  potentials}\label{newton}
In the next Lemma, we establish that given a bounded function with bounded support, its convolution with the function $\Phi$ defined in \eqref{fundsolution} is H\"{o}lder continuous. 
 \begin{lemma}\label{lem1}
 Let $s\in (0,1/2)\cup (1/2,1)$ be fixed. Let $\Omega \subseteq \Rn$ be a bounded set, the function $f\in L^{\infty}(\Rn)$ be supported in $\Omega$ and $u$ be defined as
 \eqlab{  \label{lem1f2}u(x):= \int_{\Rn} \frac{f(y)}{|x-y|^{n-2s}} \, dy.}Then 
 $u\in C^{0,2s}(\Rn)$ for $s<1/2$ and $u\in C^{1,2s-1}$ for $s>1/2$. 
 \end{lemma}

The proof of this Lemma takes inspiration from \cite{russo}, where some bounds are obtained in the case $s=1/2$. Check also Lemma 3.1 in \cite{samko} for other considerations.
 
\begin{proof} Let $s<1/2$ be fixed. We consider $x_1,x_2 \in \Rn$ and denote by $\delta:=|x_1-x_2|$. We notice that in the course of the proof, the constants may change value from line to line. We have that
%\eqlab{ \label{lem1eq1} 
\bgs{|u(x_1)&-u(x_2)| \leq |f\|_{L^{\infty}(\Rn)}  \int_{\Omega  }\bigg| \frac{1}{|x_1-y|^{n-2s}} -\frac{1}{|x_2-y|^{n-2s}} \bigg| \, dy .
%\\
%\leq \al \|f\|_{L^{\infty}(\Rn)} \Bigg[ \int_{\Omega \cap \{ |x_1-y|\leq 2\delta\} }\frac{dy}{|x_1-y|^{n-2s}} +  \int_{\Omega \cap \{ |x_1-y|\leq 2\delta\}} \frac{dy}{|x_2-y|^{n-2s}}  \\ \al 
%+ 
% \int_{\Omega \setminus \{ |x_1-y|\leq 2\delta\}} \bigg| \frac{1}{|x_1-y|^{n-2s} }-\frac{1}{|x_2-y|^{n-2s}} \bigg|\, dy \Bigg]\\
% =:\al\|f\|_{L^{\infty}(\Rn)}\big(  I_1 +I_2+I_3 \big). 
  }
%By passing to polar coordinates we have that 
% \[ I_1 \leq C_n  \int_0^{2\delta} \rho ^{2s-1} \, d\rho = C_{n,s} \delta^{2s}.\]
% At the same manner, noticing that $|x_2-y|\leq |x_2-x_1|+|x_1-y|\leq 3\delta$ we have 
% \[ I_2  \leq C_n   \int_0^{3\delta}\rho^{2s-1}\, d\rho = C_{n,s} \delta^{2s}.\]
% For $I_3$, we see that $|x_2-y|\geq|x_1-y|- |x_1-x_2|\geq \delta$. The function $|x-y|^{2s-n}$ is differentiable at each point of the segment $x_1x_2$ and using the mean value theorem we have that for $x^\star$ on the segment $x_1x_2$ 
% \[\left| \frac{1}{|x_1-y|^{n-2s}} -\frac{1}{|x_2-y|^{n-2s}} \right| \leq C_n \frac{|x_1-x_2|}{|x^\star-y|^{n-2s+1}}    =C_n \delta \frac{1}{|x^\star-y|^{n-2s+1}} .\]
%It follows that
% \[ I_3\leq C_n \delta  \int_{\Omega \setminus \{ |x_1-y|\leq 2\delta\}}  \frac{1}{|x^\star-y|^{n-2s+1}}\, dy.\]
% Since $|x^\star-x_1| \leq \delta \leq \frac{1}2 |x_1-y|$, we have that $|x^\star-y|\geq \frac{1}2 |x_1-y|$. Passing to polar coordinates, since $2s<1$, we get that
% 	\eqlab{ \label{forln1} I_3\leq &\;C_n\delta   \int_{2\delta}^{\infty}  \rho ^{2s-2}\, d\rho = C_{n,s} \delta^{2s}  .}
 	 	Using \eqref{exressch} (with $\Omega$ instead of $B_R$) we obtain that
\eqlab{\label{lem1eq2} |u(x_1)-u(x_2)| \leq \|f\|_{L^{\infty}(\Rn)} C \delta^{2s},}where $C=C({n,s}) $ is a positive constant.  
To prove the bound for $s>1/2$, thanks to Lemma 4.1 in \cite{trudy} we have that
\eqlab{\label{fundsolderiv} D u(x)=\int_{\Omega} D \Phi(x-y) f(y)\, dy=\int_{\Omega} \frac{f(y) }{|x-y|^{n-2s+1}}\, dy.}
The proof then follows as for $s<1/2$, and one gets that
\bgs{ |D u(x_1) \al-D u(x_2)|  \leq \al \|f\|_{L^{\infty}(\Rn)} \int_{\Omega} \bigg|\frac{1}{|x_1-y|^{n-2s}} -\frac{1}{|x_2-y|^{n-2s}} \bigg| \, dy
%
%%\Bigg[ \int_{\Omega \cap \{ |x_1-y|\leq 2\delta\} }\frac{dy}{|x_1-y|^{n-2s+1}} +  \int_{\Omega \cap \{ |x_1-y|\leq 2\delta\}} \frac{dy}{|x_2-y|^{n-2s+1}} \\ \al 
%+ 
% \int_{\Omega \setminus \{ |x_1-y|\leq 2\delta\}}  \bigg| \frac{1}{|x_1-y|^{n-2s+1} }-\frac{1}{|x_2-y|^{n-2s+1}}\bigg|\, dy \\
 \leq \al 
 C \|f\|_{L^{\infty}(\Rn)}  \delta^{2s-1},} where $C=C({n,s})$ is a positive constant. 
 This concludes the proof of the Lemma.
\end{proof}
\begin{remark} 
On $\Omega$ one has the following bounds. For $x_1,x_2\in \Omega$
 \sys [|u(x_1)-u(x_2) |\leq] {&\; C |x_1-x_2|\Big(1+	|x_1-x_2|^{2s-1}\Big)    &\mbox{ for } \quad &s<  1/2\\
 &\;C|x_1-x_2| \Big(1+\big|\ln|x_1-x_2|\big|  \Big) & \mbox{ for } \quad &s= 1/2 ,} and
 \bgs{ |Du(x_1)-Du(x_2)|\leq C|x_1-x_2| \Big(1+|x_1-x_2|^{2s-2} \Big)&\mbox{\; for } \quad &s>  1/2,
 }
where $C=C(n,s,f, \Omega) $ is a positive constant depending on $n, s$ 
%the dimension of the space, the fractional parameter $s$, 
the $L^{\infty}$ norm of $f$ and the diameter of $\Omega$.  
%
%To see these, it is enough to modify \eqref{forln1} as follows. Denoting $R:=\mbox{diam } \Omega$, since $|x_1|<R$ for $s<1/2$ we get that 
%\bgs{ I_3\leq &\;C_n \delta   \int_{2\delta}^{2R }  \rho ^{2s-2}\, d\rho = C_{n,s} \delta( \delta^{2s-1} - R^{2s-1})   \leq C_{n,s} \delta(\delta^{2s-1}+1).}
%For $s=1/2$, we have $I_1, I_2$ are bounded by $C_{n,s} \delta$ and
% 	 	\[I_3  \leq C_{n,s} \delta \ln \frac{R}{\delta} \leq C_{n,s,R} (1+|\ln \delta|).\]
% 	 	From this and the bounds established in the proof of Lemma \ref{lem1}, the estimates in this remark plainly follow.
 	 	\end{remark} 
 \subsection{Some useful estimates}
 
 In this subsection we introduce some useful estimates, using the representation formulas in Theorems \ref{theorem:poissonsolution} and \ref{theorem:DPL}. The interested reader can also check \cite{fall}, where Cauchy-type estimates for the derivatives of $s$-harmonic functions are proved using the Riesz and Poisson kernel.
 
 We fix $r>0$.
 \begin{lemma}\label{estimate1} Let $u\in  L^\infty(\Rn)\cap C(\Rn\setminus B_r) $ be such that $\frlap u(x)=0$  for any $x$ in $B_r$. Then for any $\alpha \in \N^n_0$
 	\eqlab{\label{estimate1} \|D^\alpha u \|_{L^{\infty}(B_{r/2})} \leq c r^{-|\alpha|} \|u\|_{L^\infty(\Rn\setminus B_r)},} where $c=c(n,s,\alpha)$ is a positive constant.
 \end{lemma}
 \begin{proof}
 We notice that it is enough to prove \eqref{estimate1} for $r=1$, i.e.
 \eqlab{\label{estimate16} \|D^\alpha u \|_{L^{\infty}(B_{1/2})} \leq c \|u\|_{L^\infty(\Rn\setminus B_1)}.}   Indeed, if \eqref{estimate16} holds, then by rescaling, namely letting $y=rx$ and $v(y)=u(x)$   for $x\in B_1$, we have that  $D^\alpha u(x) =r^{|\alpha|}D^\alpha v(y) $. Hence $r^{|\alpha|}| D^\alpha v(y)|=|D^\alpha u(x)|\leq c\|u\|_{L^{\infty}(\Rn\setminus B_1) }= c\|v\|_{L^\infty(\Rn\setminus B_r)}$  and one gets the original estimate for any $r$.
 
 We use the representation formula given in Theorem \ref{theorem:DPL}. By inserting definition \eqref{poissondefn}, we have that in $B_1$
 \bgs{ u(x) = &\;\int_{\Rn \setminus B_1} u(y) P_1(y,x)\, dy\\
 =&\; c(n,s) \int_{\Rn \setminus B_r} u(y)\frac{(1-|x|^2)^s}{(|y|^2-1)^s} \frac{dy}{|x-y|^n}.}
 Let $x\in B_{1/2}$. We take the j$^{th}$ derivative of $u$ and have that
  \bgs{ &\;  D_j u(x) \\ =&\; c(n,s) \int_{\Rn \setminus B_1} u(y)D_j \lrq{ \frac{(1-|x|^2)^s}{(|y|^2-1)^s} \frac{1}{|x-y|^n}}\, dy\\
   =&\; c(n,s) \int_{\Rn \setminus B_r}  \frac{u(y)}{(|y|^2-1)^s} \lrq{ \frac{(-2sx_j) (1-|x|^2)^{s-1} }{|x-y|^n} + (-n) \frac{(1-|x|^2)^s (x_j-y_j)}{|x-y|^{n+2}} \, }dy.}
   Therefore renaming the constants (even from line to line),
   \eqlab{ \label{derivb}|Du(x)|\leq c_{n,s} \int_{\Rn\setminus B_1} \frac{|u(y)|}{(|y|^2-1)^s} \lrq{\frac{|x|(1-|x|^2)^{s-1} }{|x-y|^n} + \frac{(1-|x|^2)^s}{|x-y|^{n+1}} } \, dy.}
Given that $|x|\leq 1/2$ we have that $ 3 /4 \leq 1-|x|^2 \leq 1$ and $|x-y|\geq |y|/2$ and so 
   \bgs{|Du(x)|\leq c_{n,s}\|u\|_{L^\infty(\Rn\setminus B_1)}  \int_{\Rn\setminus B_1} \lrq{ \frac{1}{(|y|-1)^s |y|^n} +\frac{1}{(|y|-1)^s |y|^{n+1}} } dy.}
   Passing to polar coordinates and renaming the constants, we have that
   \bgs{|Du(x)|\leq c_{n,s}\|u\|_{L^\infty(\Rn\setminus B_1)} \lrq{ \int_1^\infty (\rho-1)^{-s}\rho^{-1}\, d\rho + \int_1^\infty (\rho-1)^{-s}\rho^{-2}\, d\rho}.}
 Now we compute
 \[\int_1^\infty (\rho-1)^{-s}\rho^{-1}\, d\rho =\int_1^{2} (\rho-1)^{-s}\rho^{-1}\, d\rho+\int_{2}^\infty (\rho-1)^{-s}\rho^{-1}\, d\rho \leq C \] and likewise,
  \[\int_1^\infty (\rho-1)^{-s}\rho^{-2}\, d\rho  \leq C .\]
  It follows that 
  \bgs{|Du(x)|\leq c_{n,s} \|u\|_{L^\infty(\Rn\setminus B_1)} \quad \mbox{ for any } x\in B_{1/2}.} By reiterating the computation, we obtain the conclusion for the $\alpha$ derivative. This proves the estimate \eqref{estimate16}, thus \eqref{estimate1} by rescaling.
 \end{proof}
 
 \begin{lemma}\label{estimate2} Let $f\in C^{0,\eee}( B_r)\cap C(\overline B_r)$ be a given function and $u\in C^1(B_r)\cap L^{\infty}(\Rn)$ be a pointwise solution of
\[\begin{cases}
		    \frlap u= f   \qquad &\mbox{ in } {B_r} , 
 		\\  u= 0   \qquad &\mbox{ in } {\Rn \setminus B_r}. 
	\end{cases}\] 
 					Then
 					\eqlab{\label{estimate2} \|u\|_{L^{\infty}(B_r)}\leq c r^{2s}\sup_{\overline B_r}|f|,}
 					where $c=c(n,s)$ is a positive constant. Furthermore, for $s>1/2$ 
 					\eqlab{\label{estimate3} \|Du\|_{L^{\infty}(B_{r/2})}\leq \overline c r^{2s-1} \sup_{\overline B_r}|f|,} where $\overline c=\overline c(n,s)$ is a positive constant.
  \end{lemma}
 \begin{proof}
  We notice that it is enough to prove \eqref{estimate2} and \eqref{estimate3} for $r=1$, i.e.
 \eqlab{\label{estimate26} \ \|u\|_{L^{\infty}(B_1)}\leq c \sup_{\overline B_1}|f|} and\eqlab{\label{estimate36} \|Du\|_{L^{\infty}(B_{1/2})}\leq \overline c \sup_{\overline B_1}|f|.}
    Indeed, by rescaling, we let $y=rx$ and  $v(y)=u(x)$ we have that  $\frlap v(y)=r^{-2s}\frlap u (x)$, while $r D v (y) =D u(x)$  and one gets the original estimates for any $r$.
    
 We take $\tilde f$ to be a continuous extension of $f$, namely let $\tilde f \in C^{0,\eee}_c(\Rn) $ be such that
 \sys[ \tilde f=]{ & f &\mbox{ in } &B_1 \\
 				&0 &\mbox{ in } &\Rn \setminus B_{3/2}} and $\sup_{\Rn}| \tilde f|\leq C \sup_{ \overline B_1} |f|$. 
 Let
 \eqlab{ \label{defu0} \tilde u(x):=\tilde f* \Phi(x) =a(n,s)\int_{\Rn} \frac{\tilde f(y)}{|x-y|^{n-2s}}\, dy.}
Then $\tilde u \in L_s^1(\Rn)\cap C^{2s+\eee}(\Rn)$ (according to Lemmata \ref{lem:uff1} and \ref{grlem12})
%$ (see Theorem 2.3 in \cite{bucur} for the proof) and  $\tilde u \in C^{2s+\eee}(\Rn)$, according to Lemma \ref{lem12}. 
Thanks to Theorem \ref{theorem:poissonsolution}, we have that 
 $ \frlap \tilde u=\tilde f	$ pointiwse in $\Rn$. Hence, thanks to the definition of $\tilde f$,  $ \frlap (\tilde u-u)=0$ in $B_1$. Moreover, $\tilde u-u =\tilde u$ in $\Rn \setminus B_1$ and from Theorem \ref{theorem:DPL} we have that in $B_1$
\eqlab{\label{defu01}  (\tilde u-u)(x) =\int_{\Rn\setminus B_1} \tilde u(y)P_1(y,x)\, dy.}	We notice at first that  by definition \eqref{defu0} and passing to polar coordinates, we obtain for any positive constant $\tilde c$ that
 \eqlab{  \label{u01} \|\tilde u\|_{L^{\infty}(B_{\tilde c } )} \leq a_{n,s} \sup_{\Rn} |\tilde f| \int_0^{\tilde c +3/2}\rho^{2s-1}\, d\rho \leq c_{n,s}  \sup_{\overline B_1}|f| .}
By renaming constants, we also have that
 \eqlab{ \label{thing2} \|\tilde u-u\|_{L^{\infty}(B_1)} \leq &\;\int_{ B_{2}\setminus B_1} |\tilde u(y)| P_1(y,x)\, dy + \int_{\Rn\setminus   B_{2}} |\tilde u(y)| P_1(y,x)\, dy \\
 \leq &\; \|\tilde u\|_{L^{\infty}(B_{2})} + I\\
 \leq &\; c_{n,s}  \sup_{\overline B_1}|f|  + I.}
 Inserting the definition \eqref{poissondefn} and using for $|y|\geq 2$ the bounds $|y-x|\geq |y|/2$ and $|y|^2-1\geq |y|^2/2$ we have that
 \bgs{ I\leq &\;c(n,s)   \int_{\Rn\setminus   B_{2}}  \frac {|\tilde u(y)|}{(|y|^2-1)^s |x-y|^n}\, dy\\
 \leq &\; c_{n,s}    \int_{\Rn\setminus   B_{2}} \frac {|\tilde u(y)|}{|y|^{n+2s} }\, dy .} 
 We estimate the $L_s^1$ norm of $\tilde u$ as follows
 \eqlab{\label{thing6} \|\tilde u\|_{L_s^1(\Rn\setminus B_{2})} = &\; \int_{\Rn\setminus B_{2}} \frac{|\tilde u(y)|}{|y|^{n+2s}}\, dy\\
 \leq &\;a(n,s)\int_{\Rn\setminus B_{2}} |y|^{-n-2s} \lrq{ \int_{B_{3/2}} \frac{|\tilde f(t)|}{|y-t|^{n-2s}} \, dt} \,dy\\
 \leq &\; a(n,s)\sup_{\Rn}|\tilde f|  \int_{B_{3/2}}  \lrq{ \int_{\Rn\setminus B_{2}} |y|^{-n-2s}|y-t|^{2s-n}\, dy}\, dt.}
 We use that $|y-t|\geq |y|/4$ and passing to polar coordinates we get that
  \eqlab{\label{thing66} \|\tilde u\|_{L_s^1(\Rn\setminus B_{2})} \leq  a_{n,s}\sup_{\Rn}|\tilde f|    \int_{2}^\infty \rho^{-n-1} \, d\rho =a_{n,s} \sup_{\overline B_1}| f|  .}
Hence 
 \[ I \leq  c_{n,s}  \sup_{\overline B_1} |f|.\]
 It follows in \eqref{thing2} (eventually renaming the constants) that
\eqlab{ \|\tilde u-u\|_{L^{\infty}(B_1)} \leq  c_{n,s}\sup_{\overline B_1} |f|. \label{fb1} }
 By the triangle inequality, we have that 
 \bgs{ \|u\|_{L^{\infty}(B_1)} \leq \|\tilde u\|_{L^{\infty}(B_1)}  +\|\tilde u-u\|_{L^{\infty}(B_1)} .}
Hence by using \eqref{u01} and \eqref{fb1} we obtain that
\[ \|u\|_{L^{\infty}(B_1)} \leq c_{n,s} \sup_{\overline B_1} |f|,\]
that is the desired estimate \eqref{estimate26}, hence \eqref{estimate2} after rescaling.

In order to prove \eqref{estimate36}, we take $x\in B_{1/2}$ and obtain by the triangle inequality that
\eqlab{ \label{thing7} |Du(x)|\leq |D(\tilde u-u)(x)|+|D\tilde u(x)|.}
We notice that in the next computations the constants may change value from line to line.
 By using \eqref{defu01} and \eqref{derivb}, for $|x|\leq 1/2$ (hence $|y-x|\geq |y|/2$) we obtain that
 \eqlab{\label{thing4} |D(\tilde u-u)(x)|\leq &\; c_{n,s} \int_{\Rn\setminus B_1}\frac{|\tilde u(y)|}{(|y|^2-1)^s|y|^{n}}\, dy \\ &\; +c_{n,s} \int_{\Rn\setminus B_1}\frac{|\tilde u(y)|}{(|y|^2-1)^s|y|^{n+1}}\, dy \\
 =&\; c_{n,s} (I_1+I_2) .}
 We compute by passing to polar coordinates that
 \bgs{  \int_{B_{2}\setminus B_1}\frac{|\tilde u(y)|}{(|y|^2-1)^s|y|^{n}}\, dy  \leq c_{n,s} \|\tilde u\|_{L^{\infty}(B_{2 })} \leq c_{n,s} \sup_{\overline B_1} |f|, } according to \eqref{u01}.
 Moreover, for $|y|\geq 2 $ we have that $|y|^2-1\geq |y|^2/2$ and so
 \bgs{ \int_{\Rn\setminus B_{2 }}\frac{|\tilde u(y)|}{(|y|^2-1)^s|y|^{n}}\, dy \leq   \int_{\Rn\setminus B_{2 }}\frac{|\tilde u(y)|}{|y|^{2s+n}}\, dy  \leq a_{n,s} \sup_{\overline B_1}|f|}
 thanks to \eqref{thing6} and \eqref{thing66}. Hence \[I_1\leq c_{n,s} \sup_{\overline B_1}|f|.\]
 We split also integral $I_2$ into two and by passing to polar coordinates, we get that
 \bgs{ \int_{B_{2 }\setminus B_1} \frac{|\tilde u(y)|}{(|y|^2-1)^s|y|^{n+1}}\, dy \leq &\; c_{n,s} \|\tilde u\|_{L^\infty(B_{2 })} \leq c_{n,s}   \sup_{\overline B_1}|f| } again by \eqref{u01}. 
 Also, using definition \eqref{defu0} of $\tilde u$ and for $|y|\geq2$ the fact that $|y|^2-1\geq |y|^2/2$, we get
 \[  \int_{\Rn \setminus B_{2 }} \frac{|\tilde u(y)|}{(|y|^2-1)^2|y|^{n+1}}\, dy  \leq a(n,s) \int_{\Rn \setminus B_{2 }}  |y|^{-n-2s-1} \lrq{\int_{B_{3/2}} \frac{|\tilde f(t)|}{|y-t|^{n-2s}} \, dt} \, dy.\] We have that $|y-t|\geq |y|/4$ and obtain that
 \[  \int_{\Rn \setminus B_{2 }} \frac{|\tilde u(y)|}{(|y|^2-1)^2|y|^{n+1}}\, dy  \leq a_{n,s}    \sup_{\overline B_1}|f|.\]
 It follows that 
 \bgs{ I_2\leq c_{n,s}   \sup_{\overline B_1}|f|.} Inserting the bounds on $I_1$ and $I_2$ into \eqref{thing4}, we finally obtain that
 \eqlab{\label{thing5} |D(\tilde u -u)(x)|\leq c_{n,s}  \sup_{\overline B_1}|f|.}
 On the other hand, for $s>1/2$, using \eqref{fundsolderiv} we get that
 \[ D\tilde u(x)=a(n,s) \int_{B_{3 /2}} \frac{\tilde f(y) }{|x-y|^{n-2s+1} }\, dy\] and therefore by passing to polar coordinates 
 \bgs{ |D\tilde u(x)|\leq &\;a_{n,s} \sup_{\overline B_1 }|f| \int_{B_{3 /2}} |x-y|^{2s-n-1}\, dy \leq a_{n,s} \sup_{\overline B_1 }|f|  \int_0^{2 } \rho^{2s-2}\, d\rho   \\&\;=a_{n,s}   \sup_{\overline B_1 } |f| .}
 This and \eqref{thing5} finally allow us to conclude from \eqref{thing7} that
 \[ |Du(x)|\leq \overline c   \sup_{\overline B_1} |f|\] for any $x\in B_{1/2}$, therefore the bound in \eqref{estimate36}. From this after rescaling, we obtain the  estimate in \eqref{estimate3}. 
   \end{proof}

%--------------------------------------
% SECTION
%--------------------------------------
 \subsection{A proof of Schauder estimates} 
In this subsection we give a simple proof of some Schauder estimates related to the fractional Laplacian, as stated in Theorem \ref{schthm}.
As we see by substituting in \eqref{mm1} and \eqref{mm} that $\omega(r)\leq C r^\alpha$, we obtain for $s\leq 1/2$ 
\[ |u(x)- u(y) |\leq  C_{n,s}\delta \lr{  \|u\|_{L^{\infty}(\Rn\setminus B_1)} + \sup_{\overline B_1}|f|   + \delta^{\alpha+2s-1}},\] 
hence $u\in C^{0,2s+\alpha}(B_{1/2})$ as long as $\alpha <1-2s$ and Lipschitz if $\alpha>1-2s$. For $s>1/2$ we have that
	\[ |Du(x)-D u(y) |\leq  C_{n,s}\delta \lr{  \|u\|_{L^{\infty}(\Rn\setminus B_1)} + \sup_{\overline B_1}|f|   + \delta^{\alpha+2s-2}}.\] Hence if $\alpha \leq 2-2s$ then $u\in C^{1,\alpha+2s-1}(B_{1/2})$ while for  $2-2s\leq \alpha <1$ the derivative $Du$ is Lipschitz in $B_{1/2}.$ 
The proof takes its inspiration from \cite{wang}, where a similar result is proved for the classical case of the Laplacian. 

We prove here the case $s>1/2$, noting that for $s\leq 1/2$ the proof follows in the same way, using the lower order estimates.
\begin{proof}[Proof of Theorem \ref{schthm}] For $k=1,2,\dots$, we denote by $B_k:=B_{\rho^{k}}(0)$, where $\rho =1/2$ and let $u_k$ be a solution of 
\sys{&\frlap u_k = f(0) &\mbox{ in } &B_k\\
	&u_k=u &\mbox{ in } &\Rn \setminus B_k.}
Then we have that
	\sys{&\frlap (u_k-u) = f(0)-f &\mbox{ in } &B_k\\
	&u_k-u=0 &\mbox{ in } &\Rn \setminus B_k.}
	We remark that in the next computations, the constants may change value from line to line. \\
		Thanks to \eqref{estimate2}, we get that
	\eqlab{\label{my1} \|u_k-u\|_{L^{\infty}(B_k)} \leq &\;c_{n,s} \rho^{2ks}\sup_{B_k} |f(0)-f|\\
	\leq &\; c_{n,s} \rho^{2ks} \omega(\rho^k).}
		Using \eqref{estimate3}, we obtain that
	\eqlab{\label{my3} \|D (u_k-u)\|_{L^\infty(B_{k+1})} \leq c_{n,s} \rho^{(2s-1)k}\omega(\rho^k).} From here, sending $k$ to infinity, for $s>1/2$ it yields that
	\eqlab{ \label{lim1} \lim_{k \to \infty} D u_k(0) =D u(0).} 
Furthermore,
\sys{&\frlap (u_k-u_{k+1}) =0 &\mbox{ in } &B_{k+1}\\
	&u_k-u_{k+1}=u_k-u  &\mbox{ in } &B_k \setminus B_{k+1}\\
	&u_k-u_{k+1}=0 &\mbox{ in } &\Rn \setminus B_k,} hence from \eqref{estimate1} we have that
	\[ \|D(u_k-u_{k+1})\|_{L^{\infty}(B_{k+2})} \leq c_{n,s}  \rho^{-(k+1)} \sup_{B_k\setminus B_{k+1}} |u_k-u|\] and
	\bgs{ \|D^2(u_k-u_{k+1})\|_{L^{\infty}(B_{k+2})} \leq c_{n,s}  \rho^{-2(k+1)} \sup_{B_k\setminus B_{k+1}} |u_k-u|.}
	Using now \eqref{my1}, we get that
	\eqlab{\label{my2}\|D(u_k-u_{k+1})\|_{L^{\infty}(B_{k+2})}   \leq c_{n,s} \rho^{(2s-1)k}\omega(\rho^k)} and
\eqlab{\label{y1}\|D^2(u_k-u_{k+1})\|_{L^{\infty}(B_{k+2})} \leq c_{n,s}  \rho^{(2s-2)k} \omega(\rho^k).}
Let us fix $s>1/2$. Then for any given point $z$ near the origin we have that
\eqlab{\label{bla1} |Du(z)-Du(0)|\leq &\; |Du_k(z)-Du(z)|+ |Du_k(0)-Du(0)|+|Du_k(z)-Du_k(0)| \\ = &\;A_1+A_2+A_3.}
For $k \in \N^*$ fixed, we take $z$ such that $\rho^{k+2}\leq |z|\leq \rho^{k+1}$. Using \eqref{my3}  we get that
\[ A_1 \leq c_{n,s} \rho^{(2s-1)k}\omega(\rho^k).\] 
Taking into account \eqref{lim1} and using \eqref{my2}, we have that
\[ A_2 \leq \sum_{j=k}^{\infty} |Du_j(0)-D u_{j+1}(0)| \leq c_{n,s}\sum_{j=k}^{\infty} \rho^{(2s-1)j}\omega(\rho^j),\] therefore by renaming the constants
\bgs{  A_1+A_2 \leq &\;   c_{n,s} \rho^{(2s-1)k}\omega(\rho^k) + c_{n,s}\sum_{j=k}^{\infty} \rho^{(2s-1)j}\omega(\rho^j) \\ &\; \leq c_{n,s}\sum_{j=k}^{\infty} \rho^{(2s-1)j}\omega(\rho^j) .}  
 For the positive constant $c_s= (2s-1) / \lr{{\rho^{1-2s}-1}}$ and any $j =k, k+1, \dots$ we have that 
\[\rho^{(2s-1)j}=c_s \int_{\rho^{j}}^{\rho^{j-1}} t^{2s-2}\, dt.\]
Since $\omega$ is a increasing function, we obtain that
\bgs{\omega(\rho^j) \rho^{(2s-1)j}  = c_s\, \omega(\rho^j) \int_{\rho^{j}}^{\rho^{j-1}} t^{2s-2}\, dt \leq c_s \int_{\rho^{j}}^{\rho^{j-1}} \omega(t) t^{2s-2}\, dt,}
hence given that $8|z|\geq \rho^{k-1}$
\bgs{ \sum_{j=k}^{\infty} \rho^{(2s-1)j}\omega(\rho^j) \leq &\;  c_s \sum_{j=k}^{\infty} \int_{\rho^{j}}^{\rho^{j-1}} \omega(t) t^{2s-2}\, dt
\leq c_s \int_0^{\rho^{k-1}} \omega(t) t^{2s-2}\, dt
% \leq   &\;c_s \lr{ \int_0^{|z|} \omega(t)t^{2s-2}\, dt + \int_{\rho^{k+2}}^{\rho^{k-1}}\omega(t)t^{2s-2}\, dt  }
 \\
 \leq  &\; {c_s  \int_0^{8|z|} \omega(t)t^{2s-2}\, dt } .}
  Therefore, \eqlab{\label{a12} A_1 + A_2 \leq c_{n,s} \int_0^{8|z|} \omega(t)t^{2s-2}\, dt.} 
Moreover,  for $j=0, 1, \dots, k-1$ we consider $h_j:= u_{j+1}-u_{j}$ and have that
	\bgs{ A_3 \leq \sum_{j=0}^{k-1} |Dh_j(z)-Dh_j(0)| + |Du_0(z)-Du_0(0)|.}
	By the mean value theorem, there exists $\theta \in(0,|z|)$ such that
	\[ |Dh_j(z)-Dh_j(0)|\leq |z||D^2h_j(\theta)|\] and since $|z|\leq \rho^{k+1}$, thanks to \eqref{y1} we obtain that 
	\bgs{ |D^2h_j(\theta)| \leq  c_{n,s} \rho^{(2s-2)j} \omega(\rho^{j}).}
	Hence
	\[\sum_{j=0}^{k-1}|Dh_j(z)-Dh_j(0)| \leq c_{n,s} |z|  \sum_{j=0}^{k-1} \rho^{(2s-2)j}\omega(\rho^{j})
	=c_{n,s} |z|  \Big( \sup_{\overline B_1}|f| +  \sum_{j=1}^{k-1} \rho^{(2s-2)j}\omega(\rho^{j})\Big).\]
As previously done, we have that for the positive constant $c_s=(2-2s)/\lr{{1-\rho^{2-2s}}}$  and $j=1,\dots, k-1$
\[\rho^{(2s-2)j}= c_s \int_{\rho^{j}}^{\rho^{j-1}} t^{2s-3}\, dt\]and since $\omega$ is increasing
\bgs{\omega(\rho^{j}) \rho^{(2s-2)j} \leq c_s \int_{\rho^{j}}^{\rho^{j-1}} \omega(t) t^{2s-3}\, dt.} It follows that
\bgs{ \sum_{j=1}^{k-1} \rho^{(2s-2)j}\omega(\rho^{j}) \leq &\;  c_s \sum_{j=1}^{k-1} \int_{\rho^{j}}^{\rho^{j-1}} \omega(t) t^{2s-3}\, dt
\leq c_s  \int_{\rho^{k-1}}^1 \omega(t)t^{2s-3}\, dt\\
 \leq   &\;c_s \int_{|z|}^1 \omega(t)t^{2s-3}\, dt, } since $|z|\leq \rho^{k-1}$.  Therefore, \[\sum_{j=1}^{k-1}|Dh_j(z)-Dh_j(0)| \leq c_{n,s} |z| \int_{|z|}^1  \omega(t)t^{2s-3}\, dt.\]
 Moreover, let 
 \[ v_0(x):= k_{n,s} f(0) (1-|x|^2)^s_+ \quad \mbox{ for } x \in \Rn . \]  
 Then using the result in Subsection \ref{ctfrlap} (check there also the explicit value of $ k_{n,s}$)
 %(see Section 2.6 in \cite{nonlocal}, or the general result in \cite{dyda}), for the appropriate value of $ k_{n,s}$, 
 we have in $B_1$ that  $\frlap v_0(x) =f(0)$. Then the function $u_0-v_0$ is $s$-harmonic in $B_1$, with boundary data $u$. We have that
 \[ |Du_0(z)-Du_0(0)|\leq |z| |D^2 u_0(\theta)|\leq |z|\lr{ |D^2(u_0-v_0) (\theta)| + |D^2 v_0(\theta)|}.\]
  Using the estimate in \eqref{estimate1} we have  for $\theta \in (0,|z|)$
    \[ |D^2(u_0-v_0) (\theta)|\leq  c_{n,s}\,  \, \|u\|_{L^{\infty}(\Rn\setminus B_1)}.\] Moreover, $|D^2 v_0(\theta)|$ is bounded. It follows that  
    \[ |Du_0(z)-Du_0(0)|\leq c_{n,s} |z| \|u\|_{L^{\infty}(\Rn\setminus B_1) } ,\] hence
 \[ A_3 \leq c_{n,s}|z| \lr{ \sup_{\overline B_1}|f| + \|u\|_{L^{\infty}(\Rn\setminus B_1) } + \int_{|z|}^1  \omega(t)t^{2s-3}\, dt}.\]
 
 Inserting this and \eqref{a12} into \eqref{bla1} we finally obtain that 
 \bgs{ |Du(z)-Du(0)| \leq &\; C_{n,s} \bigg[ |z|\lr{  \|u\|_{L^{\infty}(\Rn\setminus B_1)} +\sup_{\overline B_1}|f|}  + \int_0^{c|z|} \omega(t)t^{2s-2}\, dt \\  &\;+ |z| \int_{|z|}^1  \omega(t)t^{2s-3}\, dt \bigg].} From this the conclusion plainly follows.
  This concludes the proof of the Theorem.
\end{proof}
\chapter{Extension problems} \label{S:3}

\begin{abstract}
{We discuss in this chapter an extension procedure for two integral (nonlocal) operators, the fractional Laplacian and the Marchaud derivative. We present at first two applications for the fractional Laplacian: the water wave model and a model related to crystal dislocations, making clear how the extension problem appears in these models. We then discuss in detail this harmonic extension problem via the Fourier transform.
Furthermore, we prove that the (nonlocal) Marchaud fractional derivative in $\mathbb{R}$ can be obtained from a parabolic extension problem with an extra (positive) variable as the operator that maps the heat conduction equation to the Neumann condition.
Some properties of the fractional derivative are deduced from those of the local operator. In particular we prove a Harnack inequality for Marchaud-stationary functions. }
\end{abstract}

\bigskip 
\bigskip

We dedicate this chapter to obtaining the fractional Laplacian and the Marchaud derivative from an extension procedure, as the behavior on the trace of two local operators, defined in an extra-dimension space.
We present at first two applications, the water wave model and the
Peierls-Nabarro model related to crystal dislocations. We show that the extension operator related to the half-Laplacian arises in the theory of water waves of irrotational, incompressible, inviscid fluids in the small amplitude, long wave regime. The mathematical framework of crystal dislocation is related to the Peierls-Nabarro model and in this context we obtain that at a macroscopic level, dislocations tend to concentrate at single points, following the natural periodicity of the crystal.  We then discuss\footnote{Though we do not develop this approach here,
it is worth mentioning that extended problems arise naturally
also from the probabilistic interpretation described in Chapter~\ref{chap2}.
Roughly speaking, a stochastic process with jumps in~$\R^n$
can often be seen as the ``trace''
of a classical stochastic process in~$\R^n\times [0,+\infty)$
(i.e., each time that the classical stochastic process 
in~$\R^n\times [0,+\infty)$
hits~$\R^n\times\{0\}$ it induces a jump process over~$\R^n$).
Similarly, stochastic process with jumps may also be seen
as classical processes at discrete, random, time steps.}
in detail the extension problem via the Fourier transform. We conclude this chapter by discussing the extension operator related to the Marchaud fractional derivative. As an application of this, we give a proof of a Harnack inequality for Marchaud-stationary functions.

\section{The harmonic extension of the fractional Laplacian}
\begin{center}
\begin{figure}[htpb]
	\hspace{0.6cm}
	\begin{minipage}[b]{0.8\linewidth}
	\centering
	\includegraphics[width=0.8\textwidth]{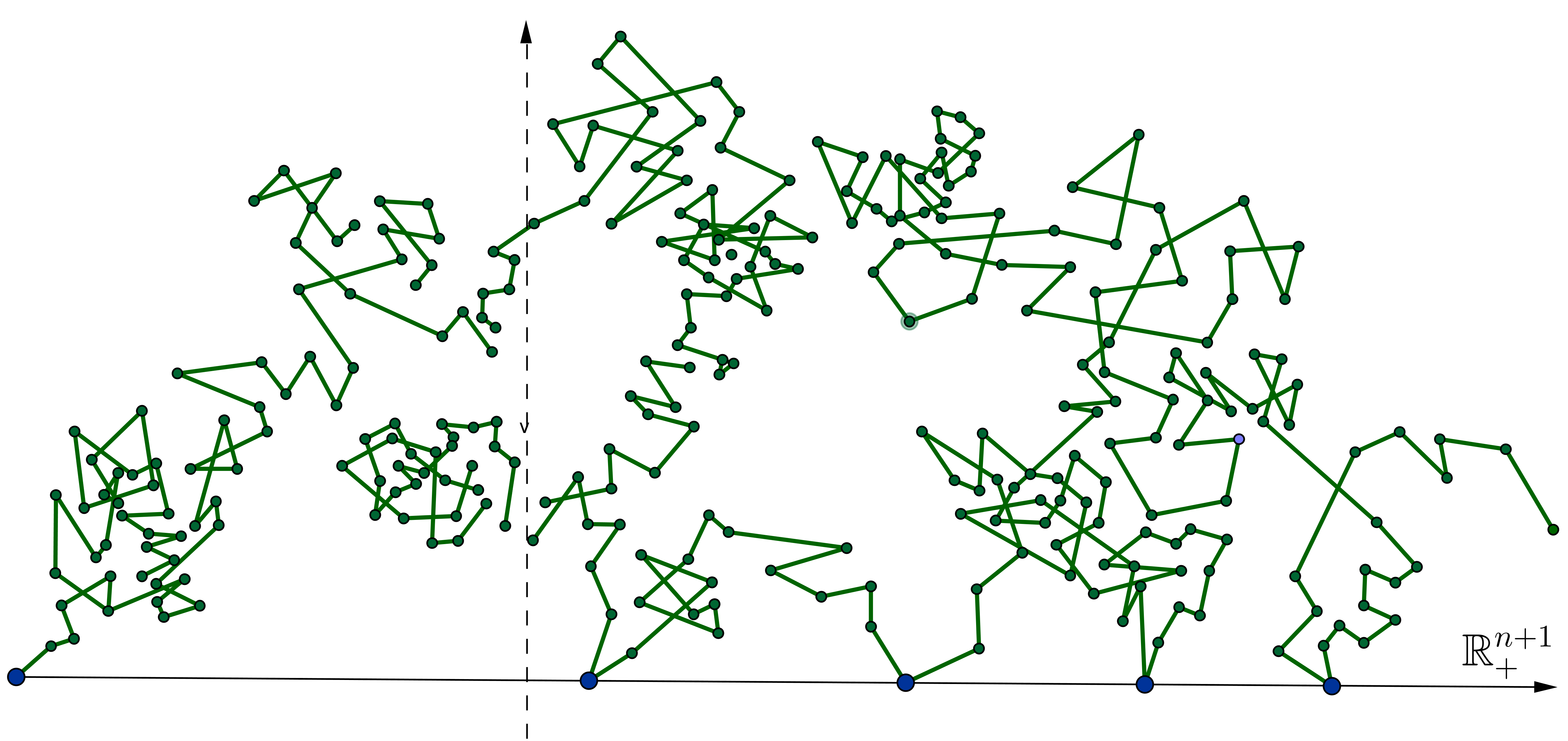}
	\caption{The random walk with jumps in $\Rn$ can be seen as a classical random walk in $\R^{n+1}$}   
	\label{fign:rwmm}
	\end{minipage}
\end{figure} 
\end{center}
The harmonic extension of the fractional Laplacian in the framework
considered here
is due to Luis Caffarelli and Luis Silvestre (we
refer to \cite{CAFSIL} for details). We also recall that this extension procedure
was obtained by S.~A. Mol{\v{c}}anov and E.~Ostrovski{\u\i}
in~\cite{MO-OS} by probabilistic methods (roughly speaking
``embedding'' a long jump random walk in~$\R^n$
into a classical random walk in one dimension more, see Figure \ref{fign:rwmm}).

The idea of this extension procedure
is that the nonlocal operator $\frlap $ acting on functions defined on $\Rn$ may be reduced to a local operator, acting on functions defined in the higher-dimensional half-space $\R^{n+1}_+ :=\Rn \times (0,+\infty).$ Indeed, take $U \colon \R^{n+1}_+  \to \R$ solution to the equation 	
\syslab{\label{lapextop} 
 &\text{div} \Big(y^{1-2s} \nabla U(x,y)\Big) =0 &&  \mbox{in}\quad \R^{n+1}_+\\
&U(x,0)=u(x)&&\mbox{in} \quad \Rn.}
Then up to constants one has that\[ -\lim_{y\to 0^+} \Big( y^{1-2s} \partial_y U (x,y) \Big)= \frlap u(x).\]

\subsection{Water wave model}

Let us consider the half space $\R_{+}^{n+1} = {\Rn} \times (0,+\infty)$ endowed with the coordinates $x \in {\Rn}$ and $y\in (0,+\infty)$. We show that the half-Laplacian (namely when $s=  1/2 )$ arises when looking for a harmonic function in $\R_{+}^{n+1}$ with given data on $\Rn \times \{y=0\}$. Thus, let us consider the following local Dirichlet-to-Neumann problem: 
	\begin{equation*}
		\begin{cases}
		\Delta U = 0 & \text{ in }  \R_{+}^{n+1} ,\\
		U(x,0) =u(x) & \text{ for }  x\in  { \Rn}.
		\end{cases}
	\end{equation*}
The function $U$ is the harmonic extension of $u$, we write $U=E u$, and define the operator $\mathcal{L}$ as 
	\eqlab { \label{D1} \mathcal{L} u (x):= -\partial_y U (x,0) .}
We claim that 
	\eqlab{\label{D2} \mathcal{L} = \sqrt{-\Delta_x}, }
	in other words
	\[ \mathcal{L} ^2 = -\Delta_x. \]
Indeed, by using the fact that $E(\mathcal{L} u) = -\partial_y U$ (that can be proved, for instance, by using the Poisson kernel representation for the solution), we obtain that	
	\[ \begin{split}
		\mathcal{L}^2 u (x) =\; & \mathcal{L}\big( \mathcal{L}u \big) (x) \\
		=\; &-\partial_y E\big( \mathcal{L}u\big)  (x,0) \\
		=\; & -\partial_y \big( -\partial_y  U\big) (x,0) \\
		=\; &\big( \partial_{yy} U + \Delta_x U- \Delta_x U \big)(x,0)\\
		=\; &	\Delta U (x,0) - \Delta u (x)	\\
		=\; &  -\Delta u(x),
	\end{split}\]
which concludes the proof of~\eqref{D2}.

One  remark in the above calculation lies in the choice of the sign of the square root of the operator. Namely, if we set~$\tilde{\mathcal{L}}u(x):=\partial_y U(x,0)$, the same computation as above would give that~$\tilde{\mathcal{L}}^2 = -\Delta$. In a sense, there is no surprise that a quadratic equation offers indeed two possible solutions. But a natural question is how to choose the ``right'' one.

There are several reasons to justify the sign convention in \eqref{D1}. One 
reason
is given by spectral theory, that makes the (fractional) Laplacian a negative defined operator. Let us discuss a purely geometric justification, in the simpler $n=1$-dimensional case. We wonder how the solution of the problem
\begin{equation}\label{EQ-D1-1}
 \begin{cases}
(-\Delta)^s u=\,1    &\mbox{ in }\;  (-1,1),\\
	 u=\,0   &\mbox{ in } \; \R \setminus (-1,1). \end{cases}
\end{equation}
should look like in the extended variable~$y$. First of all, 
by Maximum Principle (recall Theorems~\ref{THM-MA-1}
and~\ref{THM-MA-1-STRONG}),
we have that~$u$ is positive\footnote{As a matter of fact,
the solution of~\eqref{EQ-D1-1} is explicit and it is given by~$(1-x^2)^s$,
up to dimensional constants (see Section \ref{ctfrlap}). See also~\cite{DYDA} for a list of functions
whose fractional Laplacian can be explicitly computed
(unfortunately, differently from the classical cases, explicit computations
in the fractional setting
are available only for very few functions).}
when~$x\in(-1,1)$
(since this is an $s$-superharmonic function,
with zero data outside). 

Then the harmonic extension~$U$ in~$y>0$ of a 
function~$u$ which is positive in~$(-1,1)$ and vanishes outside~$(-1,1)$ 
should have the shape of an elastic membrane over the halfplane~$\R^2_+$ 
that is constrained to the graph of~$u$ on the trace~$\{y=0\}$.
\begin{center}
\begin{figure}[htpb]
	\hspace{0.9cm}
	\begin{minipage}[b]{0.90\linewidth}
	\centering
	\includegraphics[width=0.90\textwidth]{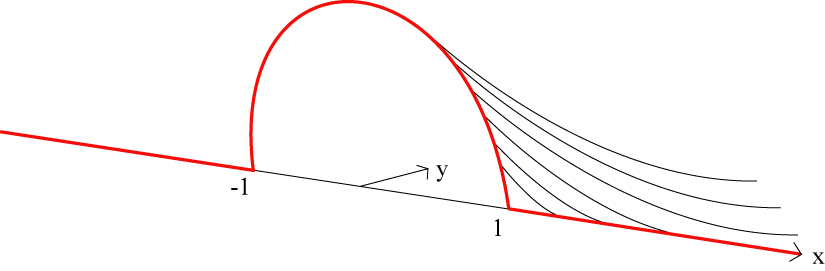}
	\caption{The harmonic extension}   
	\label{fign:FE}
	\end{minipage}
\end{figure} 
\end{center}
We give a picture
of this function~$U$ in Figure~\ref{fign:FE}. Notice from the picture that $\partial_y U(x,0)$ is negative, for any~$x\in(-1,1)$. Since $(-\Delta)^s u (x)$ is positive, we deduce that, to make our picture consistent with the maximum principle, we need to take the sign of~${\mathcal{L}}$ opposite to that of~$\partial_y U(x,0)$. This gives a geometric justification of~\eqref{D1}, which is only based on maximum principles (and on ``how classical harmonic functions look like''). 
\begin{center}
\begin{figure}[htpb]
	\hspace{0.6cm}
	\begin{minipage}[b]{0.70\linewidth}
	\centering
	\includegraphics[width=0.70\textwidth]{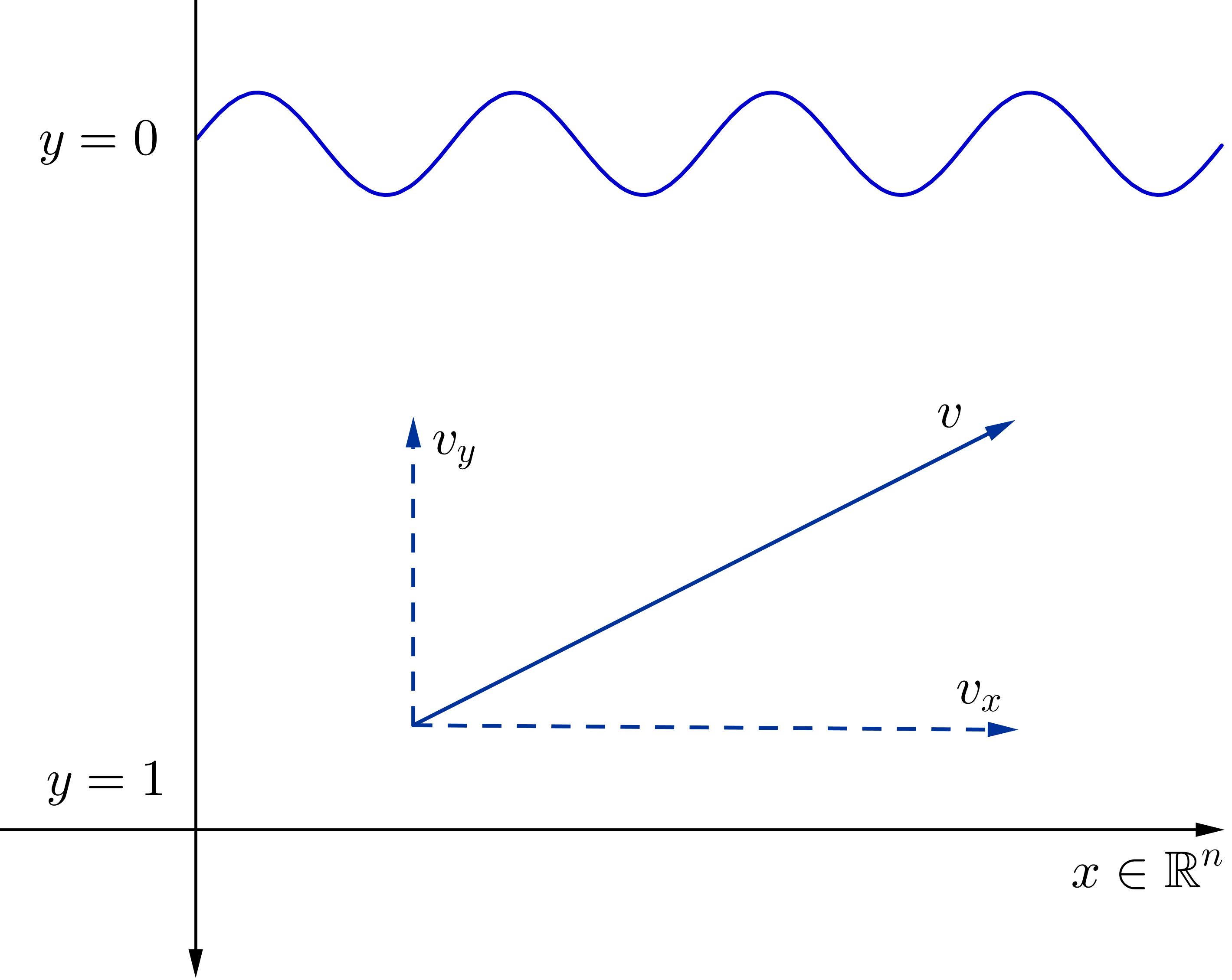}
	\caption{The water waves model}   
	\label{fign:ww}
	\end{minipage}
\end{figure} 
\end{center}

We show now that the operator $\mathcal{L}$ arises in the theory of water waves of irrotational, incompressible, inviscid fluids in the small amplitude, long wave regime.\\ 
Consider a particle moving in the sea, which is, for us, the space $\Rn \times (0,1)$, where the bottom of the sea is set at level $1$ and the surface at level $0$ (see Figure \ref{fign:ww}). The velocity of the particle is $v\colon \Rn \times (0,1) \to \R^{n+1}$ and we write $v(x,y)=\big(v_x(x,y),v_y(x,y)\big)$, where $v_x \colon \Rn \times(0,1) \to \Rn$ is the horizontal component and $v_y\colon \Rn \times(0,1)\to \R$ is the vertical component. We are interested in the vertical velocity of the water at the surface of the sea. \\
In our model, the water is incompressible, thus $\text{div } v=0 $ in $\Rn \times (0,1)$.  Furthermore, on the bottom of sea (since water cannot penetrate into the sand), the velocity has only a non-null horizontal component, hence $v_y(x,1)=0$. Also, in our model we
assume that there are no vortices: at a mathematical level,
this gives that~$v$ is irrotational, thus we may write it as the gradient of a function $U \colon \R^{n+1} \to \R$. This says that the vertical component of the velocity at the surface of the sea is $ v_y(x,0)=\partial_y U(x,0)$.
We are led to the problem 
	\begin{equation} \label{ww}
		\begin{cases}
		\Delta U = 0 & \text{ in }  \R_{+}^{n+1} ,\\
		 \partial_yU (x,1) =0 & \text{ in }  {\Rn}, \\
		U(x,0) =u(x) & \text{ in }  {\Rn}.
		\end{cases}
	\end{equation}

Let $\mathcal {L}$ be, as before, the operator $\mathcal{L}u(x):=-\partial_yU(x,0)$.
We solve the problem \eqref{ww} by using the Fourier transform and, up to a normalization factor, we obtain that
	\[ \mathcal{L}u=  \mathcal{F}^{-1} \bigg( |\xi| \frac{ e^{|\xi|} - e^{-|\xi|}}{e^{|\xi|}+e^{-|\xi|}} \widehat{u}(\xi) \bigg).\]
Notice that for large frequencies $\xi$, this operator is asymptotic to the square root of the Laplacian: 
	\[ \mathcal{L} u \simeq \mathcal{F}^{-1} \bigg( |\xi| \widehat{u}(\xi) \bigg) = \sqrt{-\Delta} u.\]

The operator $\mathcal{L}$ in the two-dimensional case has an interesting property, that is in analogy to a conjecture of
De Giorgi (the forthcoming Section \ref{sbsdg} 
will give further details about it):
more precisely,
one considers
entire, bounded, smooth, monotone solutions of the 
equation $\mathcal{L} u= f(u)$ for given $f$, and proves
that the solution only depends on one variable. More precisely:
\begin{theorem}\label{DLV}
Let $f \in C^1(\R)$ and $u$ be a bounded smooth solution of 
 	\begin{equation*}
		\begin{cases}
		\mathcal{L}u = f(u)  & \text{ in } \R^2,\\
		\partial_{x_2} u >0   & \text{ in } \R^2.
		\end{cases}
	\end{equation*}
Then there exist a direction $\omega \in S^1$ and a function $u_0\colon \R \to \R$ such that, for any $x \in \R^2$, 
	\[ u(x)=u_0(x \cdot \omega).\]
\end{theorem}
See Corollary 2 in \cite{LV09} for a proof of Theorem \ref{DLV} and to Theorem 1 in \cite{LV09} for a more general result (in higher dimension).

\subsection{Crystal dislocation}

A crystal is a material whose atoms are displayed in a regular way.
Due to some impurities in the material or to an external stress,
some atoms may move from their rest positions. The system
reacts to small modifications by pushing back towards
the equilibrium. Nevertheless, slightly larger modifications may lead
to plastic deformations.
Indeed, if an atom dislocation is of the order
of the periodicity size of the crystal, it can be perfectly
compatible with the behavior of the material
at a large scale, and it can lead to a permanent modification.

Suitably superposed atom dislocations
may also produce macroscopic deformations of the material,
and the atom dislocations may be moved
by a suitable external force,
which may be more effective if it happens to be compatible with the periodic
structure of the crystal.

These simple considerations may be framed into a mathematical
setting, and they also have concrete
applications in many industrial branches
(for instance,
in the production of a soda can, in order to change the shape of an aluminium
sheet, it is reasonable to believe that applying the right force to it can be simpler and less expensive than melting the metal).

It is also quite popular (see e.g.~\cite{CATERPILLAR})
to describe
the atom dislocation motion in crystals
in analogy
with the movement of caterpillar 
(roughly speaking, it is less expensive for
the caterpillar to produce a defect in the alignment
of its body and to dislocate this displacement, rather
then rigidly translating his body on the ground).

The mathematical framework of crystal dislocation
presented here is related to the 
Peierls-Nabarro model, that is
a hybrid model in which a discrete dislocation 
occurring along a slide line is incorporated in a continuum medium.
The 
total energy in the Peierls-Nabarro model combines the elastic energy of 
the material in reaction to the single dislocations, and the potential energy of 
the misfit along the glide plane. The main result is that, at a 
macroscopic level, dislocations tend to concentrate at single points, 
following the natural periodicity of the crystal.
\begin{center}
	\begin{figure}[htpb]
	\hspace{0.7cm}
	\begin{minipage}[b]{0.7\linewidth}
	\centering
	\includegraphics[width=0.7\textwidth]{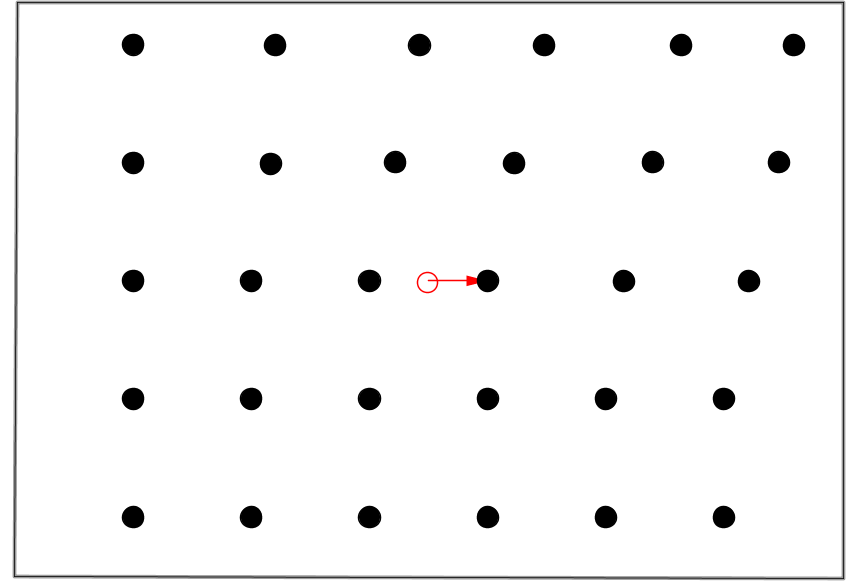}
	\caption{Crystal dislocation}   
	\label{fign:figure1}
	\end{minipage}
	\end{figure}
\end{center}

To introduce a mathematical framework
for crystal dislocation,
first, we ``slice'' the crystal with a plane. The mathematical setting will be then, by symmetry arguments, the half-plane $\R^2_{+} =\{ (x,y)\in   \R^2 \text{ s.t. } y \geq 0\}$ and the glide line will be the $x$-axis. In a crystalline structure, the atoms display periodically. Namely, the atoms on the $x$-axis have the preference of occupying integer sites. If atoms move out of their rest position due to a misfit, the material will have an elastic reaction, trying to restore the crystalline configuration. The tendency is to move back the atoms to their original positions, or to recreate, by translations, the natural periodic configuration. This effect may be modeled by defining $v^0(x):=v(x,0)$ to be the discrepancy between the position of the atom $x$ and its rest position. Then, the misfit energy is 
	\begin{equation}\label{YU-1}
\mathcal{M}(v^0):=\int_{\R} W\Big(v^0(x)\Big) \, dx,\end{equation}
where~$W$ is a smooth periodic potential,
normalized in such a way that~$W(u+1)=W(u)$ for any $u\in \R$ and~$0=W(0)<
W(u)$ for any~$u\in(0,1)$. We also assume that the minimum of~$W$
is nondegenerate, i.e.~$W''(0)>0$.

We consider
the dislocation function $v(x,y)$ on the half-plane $\R^2_{+}$. The
elastic energy of this model is given by
	\begin{equation}\label{YU-2}
 \mathcal{E}(v):=\frac{1}{2} \int_{\R^2_{+}} \Big|\nabla v(x,y)\Big|^2 \, dx \, dy.\end{equation}
The total energy of the system is therefore
	\begin{equation}\label{YU-3}
\mathcal{F} (v):= \mathcal{E}(v)+ \mathcal{M}(v^0) =\frac{1}{2} \int_{\R^2_{+}} \Big|\nabla v(x,y)\Big|^2 \, dx \, dy+\int_{\R} W\Big(v(x,0)\Big) \, dx.\end{equation}
Namely, the total energy of the system is the superposition
of the energy in~\eqref{YU-1},
which tends to settle all the atoms in their rest position
(or in another position equivalent to it from the
point of view of the periodic crystal),
and the energy in~\eqref{YU-2},
which is the elastic energy of the material itself.
\bigskip

Notice that some approximations have been performed
in this construction. For instance,
the atom dislocation changes the structure of the crystal itself:
to write~\eqref{YU-1}, one is making the assumption
that the dislocations of the single atoms do not destroy
the periodicity of the crystal at a large scale,
and it is indeed this ``permanent'' periodic structure
that produces the potential~$W$.

Moreover, in~\eqref{YU-2}, we are supposing that
a ``horizontal'' atom displacement
along the line~$\{y=0\}$ causes a horizontal
displacement at~$\{y=\epsilon\}$ as well.
Of course, in real life, if an atom at~$\{y=0\}$ moves, say, to the right,
an atom at level~$\{y=\epsilon\}$ is dragged to the right as well,
but also slightly downwards towards the slip line~$\{y=0\}$.
Thus, in~\eqref{YU-2} we are neglecting this ``vertical''
displacement. This approximation
is nevertheless reasonable since, on the one hand,
one expects the vertical displacement to be negligible with
respect to the horizontal one and, on the other hand,
the vertical periodic structure of the crystal tends to avoid
vertical displacements of the atoms outside the periodicity range
(from the mathematical point of view,
we notice that taking into account vertical displacements
would make the dislocation function vectorial,
which would produce a system of equations, rather than one single
equation for the system).

Also, the initial assumption of slicing the crystal
is based on some degree of simplification,
since this comes to studying dislocation curves in spaces
which are ``transversal'' to the slice plane.

In any case, we will take these (reasonable, after all)
simplifying assumptions
for granted, we will study their
mathematical consequences and see how the results
obtained agree with the physical experience.\bigskip

To find the Euler-Lagrange 
equation associated to~\eqref{YU-3}, let us consider a perturbation~$\phi \in C_0^{\infty}(\R^2)$, with $\varphi (x):=\phi(x,0)$ and let $v$ be a minimizer. Then
\[ \frac{d}{d\eee} \mathcal{F} (v+\eee \phi) \Big|_{\eee=0} =0,\]
which gives
\[ \int_{\R^2_+} \nabla v\cdot \nabla \phi \, dx \,dy + \int_{\R} W'(v^0)  \varphi \, dx=0.\]
Consider at first the case in which $\text{supp} \phi\cap \partial \R^2_+ = \emptyset$, thus $\varphi=0$. By the Divergence Theorem we obtain that 
	\[\int_{\R^2_+} \phi \, \Delta v \, dx \, dy=0 \quad \mbox{ for any } \phi\in C_0^{\infty}(\R^2),\] 
thus $\Delta v =0$ in $\R^2_+$. 
If $\text{supp} \phi\cap \partial \R^2_+ \neq \emptyset$ then we have that
	\[ \begin{split}
		0=&\;\int_{\R^2_+} \text{div} (\phi \nabla v)\, dx \, dy + \int_{\R} W'(v^0)  \varphi \, dx \\
		=&\; \int_{\partial \R^2_+} \phi \frac{\partial v}{\partial \nu} \, dx  + \int_{\R} W'(v^0)  \varphi \, dx \\
			=&\; - \int_{\R} \varphi \frac{\partial v}{\partial y} \, dx  + \int_{\R} W'(v^0)  \varphi \, dx 
			\end{split}
			\]
			for an arbitrary $\varphi \in C_0^{\infty}(\R)$ therefore $\displaystyle\frac{\partial v}{\partial y}(x,0)=W'(v^0(x))$ for $x \in \R$. 
Hence the critical points of $\mathcal{F}$ are solutions of the problem
	\begin{equation*}
		\begin{cases}
		\Delta v(x,y)=0 &\text{ for } x\in {\R} \text{ and } y>0, \\
		v(x,0)=v^0(x) &\text{ for  } x \in \R, \\
		\partial_y v(x,0)=W'\Big(v(x,0)\Big)  & \text{ for } x \in \R
		\end{cases}
	\end{equation*}
and up to a normalization constant, recalling \eqref{D1} and \eqref{D2}, we have that
	\[-\sqrt{-\Delta} v(x,0)=W'\big(v(x,0)\big), \text{ for any } x \in \R.\]
The corresponding parabolic evolution equation is $ \partial_t v (x,0)= -\sqrt{-\Delta} v(x,0)- W'\big(v(x,0)\big)$.\bigskip

After this discussion, one is lead to
consider the more general case of the fractional Laplacian $\frlap$ for any~$s\in(0,1)$
(not only the half Laplacian), and the corresponding parabolic equation
	\[\partial_t v=  - \frlap v - W'(v) +\sigma,\]
where $\sigma$ is a (small) external stress.\\
If we take the lattice of size $\epsilon$ and rescale $v$ and $\sigma$ as
	\[v_{\epsilon} (t,x)= v\bigg( \frac{t}{\epsilon^{1+2s}}, \frac{x}{\epsilon}\bigg)\quad \mbox{ and } \quad \sigma = \eee^{2s} \sigma \bigg( \frac{t}{\epsilon^{1+2s}}, \frac{x}{\epsilon}\bigg),\]
then the rescaled function satisfies 
	\begin{equation}
		\partial_t v_{\epsilon} = \displaystyle \frac{1}{\epsilon} \big( - \frlap v_{\epsilon} -\frac{1}{\epsilon^{2s}} W'(v_{\epsilon}) + \sigma \big)\text{ in } (0, +\infty) \times \R
		\label{pareq}
	\end{equation}
with the initial condition
	 \[v_{\epsilon}(0,x)= v_{\epsilon}^0(x) \text{ for } x \in \R.\]
To suitably choose the initial condition $v_{\epsilon}^0$, we introduce the basic layer\footnote{As a matter of fact,
the solution of~\eqref{allencahn} coincides
with the one
of a one-dimensional fractional Allen-Cahn equation, 
that will be discussed in further detail in the forthcoming
Section \ref{sbsac}.}
solution $u$, that is, the unique solution of the problem
	\begin{equation}
		\begin{cases}
		- \frlap u(x)=W'(u)  \quad \text { in } \R, \\
		u'>0 \text{ and } u(-\infty)=0, u(0) = 1/2, u(+\infty)=1.
		\label{allencahn}
		\end{cases}
	\end{equation}
For the existence of such solution and its main properties see \cite{PSV13} and \cite{CS15}. Furthermore, the solution decays polynomially at $\pm \infty$ (see \cite{DPV15} and \cite{DFV14}), namely
	\begin{equation}
		\bigg| u(x) - H(x) + \frac {1}{2sW''(0)}\frac{x}{|x|^{1+2s}} \bigg| \leq \frac{C}{|x|^{\vartheta}} \quad \text { for any } x \in \Rn,
		\label{layersol}
	\end{equation}
where $\vartheta>2s$ and~$H$ is the Heaviside step function
	\[H(x) =\begin{cases} 1, & x\geq 0\\
						0,& x<0.
			\end{cases} \]
We take the initial condition in~\eqref{pareq}
to be the superposition of transitions all occurring with the same orientation, i.e. we set
	\begin{equation} \label{cdin1} v_{\epsilon}(x,0):= \frac{\epsilon^{2s}}{W''(0)}\sigma(0,x)+ \sum_{i=1}^N u\bigg(\frac{x-x_i^0}{\epsilon}\bigg),\end{equation}
where $x_1^0,\dots, x_N^0$ are $N$ fixed points. 

\begin{center}\begin{figure}[htpb]
	\hspace{0.6cm}
	\begin{minipage}[b]{0.8\linewidth}
	\centering
	\includegraphics[width=0.8\textwidth]{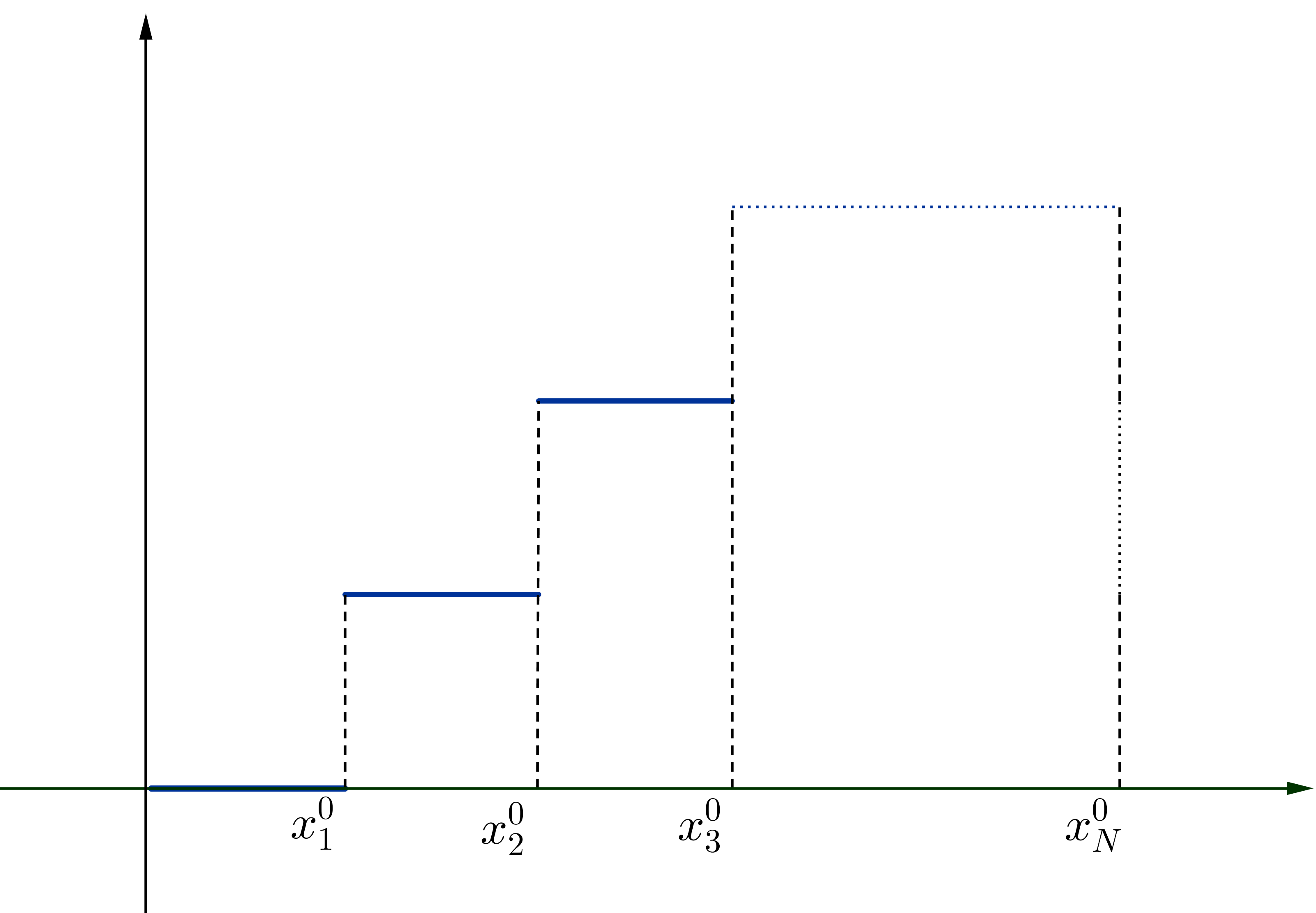}
	\caption{The initial datum when $\eee \to 0$}   
	\label{fign:cryind}
	\end{minipage}
	\end{figure}
\end{center}
The main result in this setting is that the solution~$v_\epsilon$ approaches,
as~$\epsilon\to0$, the superposition of step functions.
The discontinuities of the limit function occur at some points~$\big(
x_i(t)\big)_{i=1, \dots, N}$ which
move accordingly to the following\footnote{
The system of ordinary differential equations
in~\eqref{solx}
has been extensively studied in~\cite{MONNEAU-FORCADEL-IMBERT}.}
dynamical system
	\begin{equation}
		\begin{cases}
			\dot{x_i} = \gamma \bigg( -\sigma (t,x_i) +\displaystyle \sum_{j\neq i} \displaystyle \frac{x_i-x_j} {2s |x_i-x_j|^{2s+1}}\bigg)  \quad \text{ in } (0,+\infty) ,\\
			x_i(0)=x_i^0,
		\end{cases}
	\label{solx}
	\end{equation}
where 
	\begin{equation} \label{cdg}\gamma = \bigg(\int_{\R} (u')^2\bigg)^{-1}.\end{equation}

More precisely,
the main result obtained here is the following.

\begin{theorem}\label{VISCOSITY}
There exists a
unique viscosity solution of 	
	\begin{equation*}
		\begin{cases}
		\partial_t v_{\epsilon} = \displaystyle \frac{1}{\epsilon} \Big( - \frlap v_{\epsilon} -\frac{1}{\epsilon^{2s}} W'(v_{\epsilon}) + \sigma \Big)&\text{ in } (0, +\infty) \times \R,\\
		v_{\epsilon}(0,x)=  \displaystyle \frac{\epsilon^{2s}}{W''(0)}\sigma(0,x)+ \sum_{i=1}^N u\bigg(\frac{x-x_i^0}{\epsilon}\bigg)&\text{ for } x \in \R
		\end{cases}
	\end{equation*}
such that 
	\begin{equation}\label{SENSE}
\lim_{\epsilon \to 0} v_{\epsilon} (t,x) = \sum_{i=1}^N H\big(x-x_i(t)\big),\end{equation}
where $\big(x_i(t)\big)_{i=1, \dots, N}$ is solution to \eqref{solx}.
\end{theorem}

We refer to \cite{GM12} for the case $\displaystyle s=\frac{1}{2}$, 
to~\cite{DPV15} for the case 
$s>\displaystyle \frac{1}{2}$, and \cite{DFV14} for the case~$s<\displaystyle 
\frac{1}{2}$
(in these papers, it is also carefully stated
in which sense the limit in~\eqref{SENSE} holds true).

We would like to give now a formal (not rigorous) justification of the ODE system in \eqref{solx} that drives the motion of the transition layers.

\begin{proof}[Justification of ODE system \eqref{solx}]
We assume for simplicity that the external stress $\sigma $ is null. We use the notation $\simeq$ to denote the equality up to negligible terms in $\epsilon$. Also, we denote
	\[ u_i(t,x):=u\bigg(\frac{x-x_i(t)}{\epsilon}\bigg)\] 
and, with a slight abuse of notation 
	\[u'_i(t,x):= u'\bigg(\frac{x-x_i(t)}{\epsilon}\bigg).\]
By $\eqref{layersol}$ we have that the layer solution is approximated by
	\begin{equation}
	 	u_i(t,x) \simeq H\bigg(\frac{x-x_i(t)}{\epsilon}\bigg) - \frac{\epsilon^{2s}\big(x-x_i(t)\big)}{2s W''(0)\big|x-x_i(t)\big|^{1+2s}}.
		\label{ui}
	\end{equation}
We use the assumption that the solution $v_\epsilon$ is well approximated by the sum of $N$ transitions and write
	\[v_\epsilon(t,x)\simeq \sum_{i=1}^Nu_i(t,x) = \sum_{i=1}^N u\Big(\frac{x-x_i(t)}{\epsilon}\Big).\]
For that
	\[\partial_t v_\epsilon (t,x)= -\frac{1}{\epsilon} \sum_{i=1}^N u'_i(t,x) \dot{x_i}(t)\]
and, since the basic layer solution $u$ is the solution of $\eqref{allencahn}$, we have that 		
	\[ \begin{split}
		 -\frlap v_\epsilon \simeq \; &-\sum_{i=1}^N \frlap u_i(t,x) 					=  \frac{1}{\epsilon^{2s}}\sum_{i=1}^N  \frlap u \Big(\frac{x-x_i(t)}{\epsilon}\Big) \\
					= \;& \frac{1}{\epsilon^{2s}} \sum_{i=1}^N W' \bigg(u\Big(\frac{x-x_i(t)}{\epsilon}\Big)\bigg)
					=  \frac{1}{\epsilon^{2s}} \sum_{i=1}^N  W' \big(u_i(t,x)\big).		
		\end{split} \]
Now, returning to the parabolic equation $\eqref{pareq}$ we have that
	\begin{equation}
		 -\frac{1}{\epsilon}\sum_{i=1}^N u'_i(t,x) \dot{x_i}(t) = \frac{1}{\epsilon^{2s+1} } \bigg( \sum_{i=1}^N  W' \big(u_i(t,x)\big) - W'\Big(\sum_{i=1}^Nu_i(t,x) \Big) \bigg).
		\label{upareq}
	\end{equation}
Fix an integer $k$ between $1$ and $N$, multiply $\eqref{upareq}$ by $u_k'(t,x)$ and integrate over $\R$. We obtain
	\begin{equation}
		\begin{split}
		  -\frac{1}{\epsilon}& \sum_{i=1}^N   \dot{x_i}(t) \int_{\R} u'_i(t,x) u_k'(t,x) \, dx  \\ 
		  \;&=  \frac{1}{\epsilon^{2s+1} }  \bigg(  \sum_{i=1}^N \int_{\R}   W' \big(u_i(t,x)\big) u'_k(t,x) \, dx - \int_{\R}  W'\Big(\sum_{i=1}^Nu_i(t,x)\Big)u_k'(t,x) \, dx  \bigg).
		\label{ukpar}
		\end{split}
	\end{equation}
We compute the left hand side
of~\eqref{ukpar}. First, we take the $k^{\text{th}}$ term of the sum (i.e.
we consider the case~$i=k$). By using the change of variables
	\begin{equation}
		y:= \frac{x-x_k(t)}{\epsilon}
		\label{chvar}
	\end{equation}
we have that
	\begin{equation}
		\begin{split}
		 -\frac{1}{\epsilon}\dot{x_k}(t) \int_{\R} (u_k')^2(t,x) \,dx =&\;-\frac{1}{\epsilon}\dot{x_k}(t) \int_{\R} (u')^2\bigg(\frac{x-x_k(t)}{\epsilon}\bigg) \,dx \\	
												=&\;-\dot{x_k}(t) \int_{\R} (u')^2 (y) \, dy 
												=  -\frac{\dot{x_k}(t)}{\gamma},
		\end{split}
		\label{first}
	\end{equation}
where $\gamma$ is defined by \eqref{cdg}.\\
Then, we consider
the $i^{\text{th}}$ term of the sum on the left hand side
of~\eqref{ukpar}.
By performing again the substitution $\eqref{chvar}$, we see that this term
is
	\[\begin{split}
			-\frac{1}{\epsilon}\dot{x_i}(t) \int_{\R} u_i'(t,x)u_k'(t,x) \,dx = \;&- \frac{1}{\epsilon} \dot{x_i}(t) \int_{\R} u'\bigg(\frac{x-x_i(t)}{\epsilon}\bigg)u'\bigg(\frac{x-x_k(t)}{\epsilon}\bigg) \, dx\\
			= \;& -\dot{x_i}(t) \int_{\R} u'\bigg(y+\frac{x_k(t)-x_i(t)}{\epsilon}\bigg)u'(y) \, dy \\
			\simeq\; & \, 0,
			\end{split}\]
where, for the last equivalence we have used that for  $\epsilon$ small, $ \displaystyle u' \bigg(y + \frac{x_k(t)-x_i(t)}{\epsilon}\bigg)$ is  asymptotic to $u'(\pm\infty) =0$.

We consider the first member on the right hand side of the identity \eqref{ukpar}, and, as before, take the $k^{\text{th}}$ term of the sum. We do the substitution \eqref{chvar} and have that
	\[ \begin{split}
	\frac{1}{\epsilon} \int_{\R} W' \big(u_k(t,x) \big) u_k'(t,x) \, dx =&\;\int_{\R} W' \big(u(y) \big) u'(y) \, dy	\\	=&\; W\big(u(y)\big)\Big|_{-\infty}^{+\infty} 
			=W(1)-W(0)=0\end{split}\]
by the periodicity of $W$.
Now we use \eqref{ui}, the periodicity of $W'$ and we perform a Taylor expansion, noticing that $W'(0)=0$. We see that
		\[ \begin{split}
			 		W' \big(u_i(t,x)\big) \simeq & \;W' \Bigg( H\bigg(\frac{x-x_i(t)}{\epsilon}\bigg) - \frac{\epsilon^{2s}\big(x-x_i(t)\big)}{2s W''(0)\big|x-x_i(t)\big|^{1+2s}}\Bigg)\\
			 	\simeq & \;  W' \bigg( - \frac{\epsilon^{2s}\big(x-x_i(t)\big)}{2s W''(0)\big|x-x_i(t)\big|^{1+2s}}\bigg)\\
				\simeq & \;\frac{-\epsilon^{2s} \big(x-x_i(t)\big) }{ 2s \big|x-x_i(t)\big|^{1+2s} } .
		\end{split}\]		
Therefore, the $i^{\text{th}}$ term of the sum on the right hand side of the identity \eqref{ukpar} for $i\neq k$, by using the above approximation and doing one more time the substitution \eqref{chvar},  for $\epsilon$ small becomes
	\begin{equation}	
		\begin{split}\label{second}
		\frac{1}{\epsilon}\int_{\R} W' \big(u_i(t,x) \big) u_k'(t,x) \, dx = & \; -\frac{1}{\epsilon} \int_{\R} \frac{\epsilon^{2s} \big(x-x_i(t)\big) }{ 2s \big|x-x_i(t)\big|^{1+2s}} u' \bigg(\frac{x-x_k(t)}{\epsilon} \bigg) \, dx \\	
		=&  \;- \int_\R  \frac{\epsilon^{2s} \big( \epsilon y +x_k(t)-x_i(t)\big) }{ 2s \big|\epsilon y +x_k(t)-x_i(t)\big|^{1+2s}} u' (y) \, dy \\
		\simeq & \;- \frac{\epsilon^{2s} \big( x_k(t)-x_i(t)\big) }{ 2s \big|x_k(t)-x_i(t)\big|^{1+2s}} \int_\R u' (y) \, dy \\	
		= & \; - \frac{\epsilon^{2s} \big( x_k(t)-x_i(t)\big) }{ 2s \big|x_k(t)-x_i(t)\big|^{1+2s}}.
		\end{split}
	\end{equation}
We also observe that, for $\epsilon$ small, the second member on the right hand side of the identity \eqref{ukpar}, by using the change of variables \eqref{chvar}, reads
		\[ \begin{split}
		  		\frac{1}{\epsilon}\int_{\R} W'\Big(\sum_{i=1}^Nu_i(t,x)\Big)& u_k'(t,x) \, dx  \\ = &\;\frac{1}{\epsilon}\int_\R W'\Big(u_k(t,x) + \sum_{i\neq k} u_i(t,x)\Big)u_k'(t,x) \, dx \\ 
		=&\; \int_\R W'\bigg(u(y) + \sum_{i\neq k} u\Big(y + \frac{x_k(t)-x_i(t)}{\epsilon}\Big)\bigg) u'(y) \, dy.
		\end{split}\]
For $\epsilon$ small, $\displaystyle u\bigg(y + \frac{x_k(t)-x_i(t)}{\epsilon}\bigg) $ is asymptotic either to $u(+\infty)=1$ for $x_k>x_i$, or to $u(-\infty)=0$ for $x_k<x_i$. By using the periodicity of $W$, it follows that
	\begin{equation*}
		 \frac{1}{\epsilon}\int_{\R} W'\Big(\sum_{i=1}^Nu_i(t,x)\Big)u_k'(t,x) \, dx  = \int_\R W'\Big(u(y) \Big)u'(y) \, dy= W(1) -W(0)=0,
	\end{equation*}
again by the asymptotic behavior of $u$.
Concluding, by inserting the results \eqref{first} and \eqref{second} into \eqref{ukpar} we get that 
	\begin{equation*}
		\frac{\dot{x_k}(t)}{\gamma} =  \sum_{i\neq k} \frac{  x_k(t)-x_i(t) }{ 2s \big|x_k(t)-x_i(t)\big|^{1+2s}},
	\end{equation*}
which ends the justification of the system \eqref{solx}.
\end{proof}

We recall that, till now,
in Theorem~\ref{VISCOSITY}
we considered the initial data as a superposition of transitions all 
occurring with the same orientation (see \eqref{cdin1}),
i.e. the initial dislocation is a monotone function
(all the atoms are initially moved to the right).

Of course, for concrete applications,
it is interesting to consider also
the case in which
the atoms may dislocate in both directions,
i.e.
the transitions can occur with different orientations
(the atoms may be initially 
displaced to the left or to the right of their equilibrium position).

To model the different orientations of the dislocations,
we introduce a parameter $\xi_i \in \{-1,1\}$
(roughly speaking~$\xi_i=1$ corresponds to a dislocation to the right
and~$\xi_i=-1$ to a dislocation to the left). 

The main result in this case is the following (see~\cite{PV15}):

\begin{theorem}\label{VISCOSITY2}
There exists a viscosity solution of 	
	\begin{equation*}
		\begin{cases}
		\partial_t v_{\epsilon} = \displaystyle \frac{1}{\epsilon} \Big( - \frlap v_{\epsilon} -\frac{1}{\epsilon^{2s}} W'(v_{\epsilon}) + \sigma_{\epsilon} \Big)&\text{ in } (0, +\infty) \times \R,\\
		v_{\epsilon}(0,x)=  \displaystyle \frac{\epsilon^{2s}}{W''(0)}\sigma(0,x)+ \sum_{i=1}^N u\bigg(\xi_i \frac{x-x_i^0}{\epsilon}\bigg)&\text{ for } x \in \R
		\end{cases}
	\end{equation*}
such that 
	\[\lim_{\epsilon \to 0} v_{\epsilon} (t,x) = \sum_{i=1}^N H\Big(\xi_i \big(x-x_i(t)\big)\Big),\]
where $\big(x_i(t)\big)_{i=1, \dots, N}$ is solution to 
\begin{equation}
		\begin{cases}
			\dot{x_i} = \gamma \bigg( -\xi_i \sigma (t,x_i) + \displaystyle\sum_{j\neq i} \xi_i \xi_j \displaystyle \frac{x_i-x_j} {2s |x_i-x_j|^{2s+1}}\bigg)  \quad \text{ in } (0,+\infty) ,\\
			x_i(0)=x_i^0.
		\end{cases}
	\label{solx1}
	\end{equation} 
\end{theorem}

We observe that
Theorem~\ref{VISCOSITY2} reduces to Theorem~\ref{VISCOSITY}
when~$\xi_1=\dots=\xi_n=1$. In fact, the case discussed
in Theorem~\ref{VISCOSITY2} is richer than the one in Theorem~\ref{VISCOSITY},
since, in the case of different initial orientations,
collisions can occur, i.e. it may happen that~$x_i(T_c)=x_{i+1}(T_c)$
for some~$i\in\{1,\dots,N-1\}$ at a collision time~$T_c$.

For instance, in the case $N=2$, for $\xi_1=1$ and $\xi_2=-1$
(two initial dislocations with different orientations)
we have that
	\[ \begin{split}&\text{ if } \sigma \leq 0  \mbox{ then } T_c\leq\frac{s\theta_0^{1+2s}}{(2s+1)\gamma},\\
		&\text{ if } \theta_0 < (2s \|\sigma\|_{\infty})^{-\frac{1}{2s}} \mbox{ then } T_c \leq \frac{s\theta_0^{1+2s}}{\gamma(1-2s \theta_0\|\sigma\|_{\infty} )},
		\end{split} \]
where~$\theta_0:=x_2^0-x_1^0$ is the initial
distance between the dislocated atoms.
That is, if either the external force has the right sign,
or the initial distance is suitably small with respect to the
external force, then the dislocation time is finite,
and collisions occur in a finite time
(on the other hand, when these conditions are violated,
there are examples in which collisions do not occur).\bigskip

This and more general cases of collisions, with precise
estimates on the collision times, are discussed in detail
in~\cite{PV15}.\bigskip

An interesting feature of the system
is that the dislocation function~$v_\epsilon$
does not annihilate
at the collision time. More precisely, in the appropriate
scale, we have that~$v_\epsilon$ at the collision time
vanishes outside the collision
points, but it still preserves a non-negligible
asymptotic contribution exactly
at the collision points.
A formal statement is the following (see~\cite{PV15}):

\begin{theorem}\label{J78K}
Let~$N=2$ and assume that a
collision occurs. Let~$x_c$ be the collision point, namely $x_c=x_1(T_c)=x_2(T_c)$. Then 
\begin{equation} \label{C1} \lim_{t\to T_c} \lim_{\eee\to 0} v_\eee(t,x)=0 \quad \mbox { for any } \quad x\neq x_c,\end{equation} but
\begin{equation}\label{C2}\limsup_{\substack{ {t\to T_c}\\{\eee\to 0}}} v_\eee(t,x_c) \geq 1.\end{equation} 
\end{theorem}

Formulas~\eqref{C1} and~\eqref{C2} describe what happens in the crystal at the collision time. On the one hand, formula~\eqref{C1} states that at any point that is not the collision point and at a large scale, the system relaxes at the collision time. On the other hand, formula~\eqref{C2} states that the behavior at the collision points at the collision time is quite ``singular''. Namely, the system does not relax immediately
(in the appropriate scale). As a matter of fact, in related numerical simulations
(see e.g. \cite{RAABE}) one may notice that the dislocation
function may persists after collision and, in higher dimensions,
further collisions may change the topology of the dislocation curves.

What happens is that a slightly larger time is needed before the system relaxes exponentially fast: a detailed description of this relaxation phenomenon
is presented in \cite{RELAX}. For instance,
in the case~$N=2$,
the dislocation function decays to zero
exponentially fast, slightly after collision, as given by the following
result:

\begin{theorem}\label{thmexponentialdecay}
Let~$N=2$, $\xi_1=1$, $\xi_2=-1$, and let
$v_\epsilon$ be the solution given by Theorem~\ref{VISCOSITY2},
with~$\sigma\equiv0$.
Then there exist
$\epsilon_0>0$, $c>0$, $T_\epsilon>T_c$ and~$\rho_\epsilon>0$
satisfying
\begin{eqnarray*}
&& \lim_{\epsilon\to0} T_\epsilon=T_c\\
{\mbox{and }}&& \lim_{\epsilon\to0}
\varrho_\epsilon=0\end{eqnarray*}
such that for any $\epsilon<\epsilon_0$ we have
\begin{equation}\label{4.23bis}
|v_\epsilon (t,x)|\leq \varrho_\epsilon
e^{c\frac{T_\epsilon-t}{\epsilon^{2s+1}}},\quad\text{for any }x\in\R\text{ and
}t\geq T_\epsilon.\end{equation} \end{theorem}

The estimate in~\eqref{4.23bis} states, roughly speaking, that
at a suitable time~$T_\epsilon$ (only slightly bigger than the collision
time~$T_c$) the dislocation function gets below a small threshold~$\rho_\epsilon$,
and later it decays exponentially fast (the constant of this
exponential becomes large when $\epsilon$ is small).

The reader may compare Theorem~\ref{J78K}
and~\ref{thmexponentialdecay} and notice that different asymptotics
are considered by the two results.
A result
similar to Theorem~\ref{thmexponentialdecay}
holds for a larger number of dislocated atoms.
For instance,
in the case of three atoms with alternate
dislocations, one has that, slightly after collision,
the dislocation function decays
exponentially fast to
the basic layer solution. More precisely (see again~\cite{RELAX}),
we have that:

\begin{theorem}\label{mainthmbeforecoll3}
Let~$N=3$, $\xi_1=\xi_3=1$, $\xi_2=-1$, and let
$v_\epsilon$ be the solution given by Theorem~\ref{VISCOSITY2},
with~$\sigma\equiv0$.
Then there exist
$\epsilon_0>0$, $c>0$, 
$T_\epsilon^1,T_\epsilon^2>T_c$ and~$\rho_\epsilon>0$
satisfying
\begin{eqnarray*}
&& \lim_{\epsilon\to0} T_\epsilon^1=
\lim_{\epsilon\to0} T_\epsilon^2
=T_c,\\
{\mbox{and }}&& \lim_{\epsilon\to0}
\varrho_\epsilon=0\end{eqnarray*}
and points~$\bar y_\epsilon$ and~$\bar z_\epsilon$ satisfying
$$ \lim_{\epsilon\to0}
|\bar z_\epsilon-\bar y_\epsilon|=0$$
such that for any $\epsilon<\epsilon_0$ we have
\begin{equation}
\label{SDF-1}
v_\epsilon(t,x)\leq
   u\left(\frac{x-\bar y_\epsilon}{\epsilon}\right)+
\varrho_\epsilon e^{-\frac{c
(t-T^1_\epsilon)}{\epsilon^{2s+1}}},\qquad
   \text{for any }x\in\R \text{ and }t\geq T^1_\epsilon,
\end{equation}
and
\begin{equation}
\label{SDF-2}
v_\epsilon(t,x)\geq
u\left(\frac{x-\bar z_\epsilon}{\epsilon}\right)
-\varrho_\epsilon e^{-\frac{c
(t-T^2_\epsilon)}{\epsilon^{2s+1}}},\qquad
\text{for any }x\in\R
\text{ and }t\geq
T^2_\epsilon,\end{equation}
where $u$ is the basic layer solution introduced in~\eqref{allencahn}.
\end{theorem}

Roughly speaking, formulas~\eqref{SDF-1}
and~\eqref{SDF-2}
say that for times~$T^1_\epsilon$, $T^2_\epsilon$
just slightly bigger than the collision time~$T_c$,
the dislocation function~$v_\epsilon$ gets trapped
between two basic layer solutions (centered at points~$\bar y_\epsilon$
and~$\bar z_\epsilon$), up to a small error.
The error gets eventually to zero, exponentially fast in time,
and the two basic layer solutions which trap~$v_\epsilon$
get closer and closer to each other as~$\epsilon$
goes to zero (that is, the distance between~$\bar y_\epsilon$
and~$\bar z_\epsilon$ goes to zero with~$\epsilon$).

We refer once more
to~\cite{RELAX} for a series of figures
describing in details
the results of Theorems~\ref{thmexponentialdecay}
and~\ref{mainthmbeforecoll3}. 
We observe that
the results presented in Theorems \ref{VISCOSITY}, \ref{VISCOSITY2}, \ref{J78K}, \ref{thmexponentialdecay} and \ref{mainthmbeforecoll3} describe the crystal at different space and time scale.
As a matter of fact, the mathematical study of a crystal typically goes from an atomic description (such as a classical discrete model presented by Frenkel-Kontorova
and Prandtl-Tomlinson) to a macroscopic scale in which a plastic deformation occurs.

In the theory discussed here, we join this atomic and macroscopic scales by a series
of intermediate scales, such as a microscopic scale, in which the Peierls-Nabarro model is introduced,
a mesoscopic scale, in which we studied the dynamics of the dislocations (in particular,
Theorems \ref{VISCOSITY} and \ref{VISCOSITY2}),
in order to obtain at the end a macroscopic theory
leading to the relaxation of the model to a permanent
deformation (as given in Theorems \ref{thmexponentialdecay} and \ref{mainthmbeforecoll3} ,
while Theorem \ref{J78K} somehow describes the further intermediate
features between these schematic scalings).

\subsection{An approach to the extension problem via the Fourier transform}

We will discuss here the extension operator of the fractional Laplacian
via the Fourier transform approach (see~\cite{CAFSIL} and~\cite{stingatorrea2} 
for other approaches and further results and also~\cite{FGMT},
in which a different extension formula is obtained
in the framework of the Heisenberg groups). 
%
%Some readers may find the details of this part rather technical:
%if so, she or he can jump directly to Section~\ref{S:NP} on page~\pageref{S:NP:P},
%without affecting the subsequent reading.

We fix at first a few pieces of
notation. We denote points in~$\R^{n+1}_+:=\R^n\times(0,+\infty)$ as~$X=(x,y)$, with~$x\in\R^n$ and~$y>0$. When taking gradients in~$\R^{n+1}_+$, we write~$\nabla=(\nabla_x,\partial_y)$, where~$\nabla_x$ is the gradient in~$\R^n$. Also, in~$\R^{n+1}_+$, we will often take the
Fourier transform in the variable~$x$ only, for fixed~$y>0$. We also set \[ a:=1-2s\in(-1,1).\]

We will consider the fractional Sobolev space $\widehat H^s(\Rn)$
defined as the 
set of functions $u$ that satisfy
\[\|u\|_{L^2(\Rn)}+[\widehat u]_G <+\infty,\]
where
	\[ [v]_G:=\sqrt{ \int_{\Rn} |\xi|^{2s}\,|v(\xi)|^2\,d\xi}. \]
For any~$u\in W^{1,1}_{\rm loc}((0,+\infty))$, we consider the functional 	
	\begin{equation} \label{Gu}
		 G(u):=\int_0^{+\infty}t^a \Big(\big|u(t)\big|^2+\big|u'(t)\big|^2\Big) \,dt.
	\end{equation}
By Theorem~4 of~\cite{SeV14}, we know that the functional~$G$ attains its minimum among all the functions~$u\in W^{1,1}_{\rm loc}((0,+\infty))\cap C^0([0,+\infty))$ with~$u(0)=1$. We call~$g$ such minimizer and
	\begin{equation}\label{F-0}
		C_\sharp:=G(g)=\min_{{u\in W^{1,1}_{\rm loc}((0,+\infty)) 
		\cap C^0([0,+\infty))}\atop{u(0)=1}} G(u).
	\end{equation}

The main result of this subsection is the following.

\begin{theorem} \label{thmext}
Let~$u\in \mathcal{S}(\R^n)$ and let
	\begin{equation}\label{F-7-bis--}
		U(x,y):={\mathcal{F}}^{-1} \Big(\widehat u(\xi)\,g(|\xi|y)\Big).
	\end{equation}
Then
	\begin{equation}\label{78}
		{\rm div}\, (y^a \nabla U)=0
	\end{equation}
for any~$X=(x,y)\in\R^{n+1}_+$. 
In addition,
	\begin{equation}\label{0S}
		-y^a\partial_y U \Big|_{\{y=0\}} = C_\sharp (-\Delta)^s u
	\end{equation}
in~$\Rn$, both in the sense of distributions and as a pointwise limit.
\end{theorem}

In order to prove Theorem \ref{thmext}, we need to make some preliminary computations. At first, let us recall a few useful properties of the minimizer function $g$ of the operator $G$ introduced in \eqref{Gu}.

We know from formula~(4.5) in~\cite{SeV14} that 
	\begin{equation}\label{gg}
			0\le g\le 1,
	\end{equation}
and from formula~(2.6) in~\cite{SeV14} that
	\begin{equation}\label{89}
		g'\le0.
	\end{equation}
We also cite formula~(4.3) in~\cite{SeV14}, according to which~$g$ is a solution of
	\begin{equation}\label{89-bis}
		g''(t)+a t^{-1} g'(t)=g(t)
	\end{equation}
for any~$t>0$, and formula~(4.4) in~\cite{SeV14}, according to which
	\begin{equation}\label{89-bis2}
		\lim_{t\rightarrow0^+} t^a g'(t)\,=\,-C_\sharp.
	\end{equation}

Now, for any~$V\in W^{1,1}_{\rm loc}(\R^{n+1}_+)$ we set
	\[ [V]_a := \sqrt{\int_{\R^{n+1}_+} y^a |\nabla V(X)|^2\,dX}.\]
Notice that~$[V]_a$ is well defined (possibly infinite) on such space. Also, one can compute~$[V]_a$ explicitly in the following interesting case:

\begin{lemma} \label{lem1ext}
Let~$\psi\in \mathcal{S}(\Rn)$
and
	\begin{equation}\label{F-7}
		U(x,y):={\mathcal{F}}^{-1} \Big(\psi(\xi)\,g(|\xi|y)\Big).
	\end{equation}
Then
	\begin{equation}\label{F-8}
		[U]_a^2 = C_\sharp\,[\psi]_G^2.
	\end{equation}
\end{lemma}
 
\begin{proof} By~\eqref{gg}, for any fixed~$y>0$, the function~$\xi\mapsto \psi(\xi)\,g(|\xi|y)$ belongs to~$L^2(\Rn)$, and so we may consider its (inverse) Fourier transform. This says that the definition of~$U$ is well posed. 

By the inverse Fourier transform definition \eqref{invF}, we have that 
	\[ \begin{split}
		 \nabla_x U(x,y)= &\; \nabla_x
			\int_{\Rn} \psi(\xi)\,g(|\xi|y) \,e^{ix\cdot\xi}\,d\xi \\
			=&\;
			\int_{\Rn} i\xi\psi(\xi)\,g(|\xi|y) \,e^{ix\cdot\xi}\,d\xi\\
			=&\;{\mathcal{F}}^{-1} \Big( i\xi\psi(\xi) g(|\xi|y)\Big) (x).
	\end{split} \]
Thus, by Plancherel Theorem,
	\[ \int_{\Rn} |\nabla_x U(x,y)|^2\,dx =\int_{\Rn} \big|\xi\psi(\xi) g(|\xi|y)\big|^2\,d\xi.\]
Integrating over~$y>0$, we obtain that
	\begin{equation}\label{F-2}
		\begin{split}
			\int_{\R^{n+1}_+} y^a |\nabla_x U(X)|^2\,dX\,&
			=\int_{\Rn} |\xi|^2 \big|\psi(\xi)\big|^2 \left[\int_0^{+\infty} 
			y^a \big|g(|\xi|y)\big|^2 \,dy\right]\,d\xi\\
			&=\int_{\Rn} |\xi|^{1-a} \big|\psi(\xi)\big|^2 \left[\int_0^{+\infty}
			t^a \big|g(t)\big|^2 \,dt\right]\,d\xi
			\\ &=\int_0^{+\infty}
			t^a \big|g(t)\big|^2 \,dt\cdot
			\int_{\Rn} |\xi|^{2s} \big|\psi(\xi)\big|^2 \, d\xi
		\\ &= [\psi]_{G}^2 \;\int_0^{+\infty}
			t^a \big|g(t)\big|^2 \,dt.
		\end{split}
	\end{equation}
Let us now prove that the  following identity is well posed
	\begin{equation}\label{4.3bis}
\partial_y U(x,y)= {\mathcal{F}}^{-1} \Big(|\xi|\,\psi(\xi)\,g'(|\xi|y)\Big) .\end{equation}
For this, we observe that
	\begin{equation}\label{98}
		|g'(t)|\le C_\sharp t^{-a}.
	\end{equation}
To check this, we define~$\gamma(t):= t^a \,|g'(t)|$. {F}rom~\eqref{89} and~\eqref{89-bis}, we obtain that
	\[ \gamma'(t) =-\frac{d}{dt}\big(t^a g'(t)\big)=-t^a \,\big(g''(t)+a t^{-1} g'(t)\big)=-t^a g(t)\le0.\]
Hence
$$ \gamma(t)\le \lim_{\tau\rightarrow 0^+} \gamma(\tau)=
\lim_{\tau\rightarrow 0^+} \tau^a |g'(\tau)|= C_\sharp,$$
where formula~\eqref{89-bis2} was used in the last identity,
and this establishes~\eqref{98}.\\
{F}rom~\eqref{98} we have that~$|\xi|\,|\psi(\xi)|\,|g'(|\xi|y)|\le
 C_\sharp y^{-a}\,|\xi|^{1-a}\,|\psi(\xi)|\in L^2(\Rn)$, and so~\eqref{4.3bis}
follows.\\
Therefore, by~\eqref{4.3bis} and the Plancherel Theorem,
$$ \int_{\Rn} |\partial_y U(x,y)|^2\,dx= \int_{\Rn} |\xi|^2\,\big|\psi(\xi)\big|^2\,\big|g'(|\xi|y)\big|^2\,d\xi.$$
Integrating over~$y>0$ we obtain
\[ \begin{split}
	\int_{\R^{n+1}_+} y^a |\partial_y U(x,y)|^2\,dx =&\;
	\int_{\Rn} |\xi|^2\,\big|\psi(\xi)\big|^2\,\left[\int_0^{+\infty}
	y^a \big|g'(|\xi|y)\big|^2\,dy \right]\,d\xi
	\\ =&\;
	\int_{\Rn} |\xi|^{1-a} \,\big|\psi(\xi)\big|^2\,\left[\int_0^{+\infty}
	t^a \big|g'(t)\big|^2 dt \right]\,d\xi
	\\ =&\;
	\int_0^{+\infty}
	t^a \big|g'(t)\big|^2 dt\cdot
	\int_{\Rn} |\xi|^{2s} \,\big|\psi(\xi)\big|^2\,d\xi
	\\ =&\;
	 [\psi]_{G}^2 \; \int_0^{+\infty}
	t^a \big|g'(t)\big|^2 dt.
\end{split} \]
By summing this with~\eqref{F-2}, and recalling~\eqref{F-0}, we obtain the desired result $[U]_a^2 = C_\sharp\,[\psi]_G^2$. This concludes the proof of the Lemma.
\end{proof}

Now, given~$u \in L^1_{\rm loc}(\Rn)$, we consider the space~$X_u$ of all the functions~$V\in W^{1,1}_{\rm loc}(\R^{n+1}_+)$ such that, for any~$x\in\Rn$, the map~$y\mapsto V(x,y)$ is in~$C^0\big([0,+\infty)\big)$, with~$V(x,0)=u(x)$ for any~$x\in\Rn$.
Then the problem of minimizing~$[\,\cdot\,]_a$ over~$X_u$ has a somehow explicit solution.

\begin{lemma}\label{L7s}
Assume that~$u\in \mathcal{S}(\Rn)$. Then
\begin{equation}\label{A6}
\min_{V\in X_u} [V]_a^2 = [U]_a^2
=C_\sharp\,[\widehat u]_G^2,\end{equation}

\begin{equation}\label{F-7-bis}
U(x,y):={\mathcal{F}}^{-1} \Big(\widehat u(\xi)\,g(|\xi|y)\Big).\end{equation}
\end{lemma}

\begin{proof} We remark that~\eqref{F-7-bis} is simply~\eqref{F-7}
with~$\psi:=\widehat u$, and by Lemma \ref{lem1ext} we have that \[ [U]_a^2= C_{\sharp}[\widehat u]^2_G.\]
Furthermore, we claim that
\begin{equation}\label{98-bis}
U\in X_u. \end{equation}
In order to prove this, we first observe that
\begin{equation}\label{98-tris}
|g(T)-g(t)|\le \frac{C_\sharp\,|T^{2s}-t^{2s}|}{2s} .
\end{equation}
To check this, without loss of generality, we may suppose that~$T\ge t\ge0$.
Hence, by~\eqref{89} and~\eqref{98},
	\[ \begin{split}
		|g(T)-g(t)|\le  \int_t^T |g'(r)|\,dr 
		\le  C_\sharp\int_t^T r^{-a}\,dr
		= \frac{C_\sharp\,(T^{1-a}-t^{1-a})}{1-a},
	\end{split}\]
that is~\eqref{98-tris}.\\
Then, by~\eqref{98-tris}, for any~$y$, $\tilde y\in(0,+\infty)$,
we see that
\[ \Big| g(|\xi|\,y)-g(|\xi|\,\tilde y)\Big|\le
\frac{C_\sharp\,|\xi|^{2s} |y^{2s}-{\tilde y}^{2s}|}{2s} .\]
Accordingly,
	\[ \begin{split}
		\big|U(x,y)-U(x,\tilde y)\big|
		=& \; \bigg|{\mathcal{F}}^{-1} \bigg(\widehat u(\xi)\,\Big( g(|\xi|\,y)-g(|\xi|\,\tilde y)\Big)\bigg)\bigg|\\
		 \le & \; 	\int_{\Rn}  	\Big|\widehat u(\xi)\,\Big( g(|\xi|\,y)-g(|\xi|\,\tilde y)\Big)\Big|\,d\xi \\ 
		\le & \; 	\frac{C_\sharp\,|y^{2s}-{\tilde y}^{2s}|}{2s} \int_{\Rn} |\xi|^{2s} |\widehat u(\xi)|\,d\xi,
	\end{split} \]
and this implies~\eqref{98-bis}.

Thanks to~\eqref{98-bis} and~\eqref{F-8}, in order to complete the proof
of~\eqref{A6},
it suffices to show that, for any~$V\in X_u$, we have that
\begin{equation}\label{A6-c}
[V]_a^2 \ge [U]_a^2.\end{equation}
To prove this,
let us take~$V\in X_u$. Without loss of generality,
since~$[U]_a<+\infty$ thanks to~\eqref{F-8},
we may suppose that~$[V]_a<+\infty$. Hence, 
fixed a.e.~$y>0$, we have that
$$ y^a \int_{\Rn}|\nabla_x V(x,y)|^2\,dx
\,\le\,
y^a \int_{\Rn}| \nabla V(x,y)|^2\,dx\,<\,+\infty,$$
hence the map~$x\in |\nabla_x V(x,y)|$ belongs to~$L^2(\Rn)$.
Therefore, by Plancherel Theorem,
\begin{equation}\label{F-9}
\int_{\Rn}|\nabla_x V(x,y)|^2\,dx =
\int_{\Rn}\Big|{\mathcal{F}}\big(\nabla_x V(x,y)\big) (\xi)\Big|^2\,d\xi
.\end{equation}
Now using the Fourier transform definition (see \eqref{transF})
\begin{eqnarray*}  {\mathcal{F}}\big(\nabla_x V(x,y)\big) (\xi)=
\int_{\Rn} \nabla_x V(x,y)\,e^{-ix\cdot\xi}\,dx
=\int_{\Rn} i\xi\,V(x,y)\,e^{-ix\cdot\xi}\,dx =
i\xi \,{\mathcal{F}}\big(V(x,y)\big) (\xi),\end{eqnarray*}
hence~\eqref{F-9} becomes
\begin{equation}\label{F-11}
\int_{\Rn}|\nabla_x V(x,y)|^2\,dx =
\int_{\Rn} |\xi|^2 \,|{\mathcal{F}}\big(V(x,y)\big) (\xi)|^2\,d\xi.\end{equation}
On the other hand
$$ {\mathcal{F}}\big(\partial_y V(x,y)\big) (\xi)=
\partial_y {\mathcal{F}}\big(V(x,y)\big) (\xi)$$
and thus, by Plancherel Theorem,
$$ \int_{\Rn} |\partial_y V(x,y)|^2 \,dx=
\int_{\Rn}\big| {\mathcal{F}}\big(\partial_y V(x,y)\big) (\xi)\big|^2\,d\xi
=\int_{\Rn} |\partial_y {\mathcal{F}}\big(V(x,y)\big) (\xi)|^2\,d\xi.$$
We sum up this latter result with identity \eqref{F-11} and we use the notation~$\phi(\xi,y):=
{\mathcal{F}}\big(V(x,y)\big) (\xi)$
to conclude that
\begin{equation}\label{A5} 
\int_{\Rn}|\nabla V(x,y)|^2\,dx =
\int_{\Rn} |\xi|^2 \,|\phi(\xi,y)|^2 + |\partial_y \phi(\xi,y)|^2\,d\xi.\end{equation}
Accordingly, integrating over~$y>0$, we deduce that
\begin{equation}\label{F-13}
[V]_a^2 =\int_{\R^{n+1}_+} y^a\Big(
 |\xi|^2 \,|\phi(\xi,y)|^2 + |\partial_y \phi(\xi,y)|^2 \Big)\,d\xi\,dy.
\end{equation}
Let us first consider the integration over~$y$, for any fixed $\xi\in
\Rn\setminus\{0\}$, that we now omit from the notation when
this does not generate any confusion. We set~$h(y):= 
\phi(\xi, |\xi|^{-1} y)$.
We have that~$h'(y)=|\xi|^{-1} \partial_y\phi(\xi, |\xi|^{-1} y)$ and
therefore, using the substitution~$t=|\xi|\,y$, we obtain
	\begin{equation}\label{F-17}
		\begin{split}
		&\int_0^{+\infty} y^a \Big(|\xi|^2\,|\phi(\xi, y)|^2 +  \big|\partial_y\phi(\xi, y)\big|^2 \Big)\,dy
		\\ = &\; |\xi|^{1-a} 
		\int_0^{+\infty} t^a \Big( 
		|\phi(\xi, |\xi|^{-1} t)|^2 +
		|\xi|^{-2} \big|\partial_y\phi(\xi, |\xi|^{-1} t)\big|^2
		\Big)\,dt\\
		= &\; |\xi|^{1-a}
		\int_0^{+\infty} t^a \Big( 
		|h(t)|^2 +|h'(t)|^2
		\Big)\,dt
		\\ = &\; |\xi|^{2s} \,G(h).
\end{split}\end{equation}
Now, for any~$\lambda\in\R$, we show that
\begin{equation}\label{F-15}
\min_{w\in W^{1,1}_{\rm loc}((0,+\infty))
\cap C^0([0,+\infty))} {w(0)=\lambda} G(w) = \lambda^2 \,C_\sharp.
\end{equation}
Indeed, when~$\lambda=0$, the trivial function is an allowed
competitor and~$G(0)=0$, which gives~\eqref{F-15}
in this case. If, on the other hand,~$\lambda\ne 0$, given~$w$
as above with~$w(0)=\lambda$ we set~$w_\lambda(x):=\lambda^{-1} w(x)$.
Hence we see that~$w_\lambda(0)=1$ and thus~$G(w)=
\lambda^2 \, G(w_\lambda)\le \lambda^2\,G(g)=\lambda^2\,C_\sharp$, due to the minimality
of~$g$. This proves~\eqref{F-15}.
{F}rom~\eqref{F-15} and the fact that
$$h(0)=
\phi(\xi, 0)={\mathcal{F}}\big(V(x,0)\big) (\xi) = \widehat u(\xi),$$
we obtain that
$$ G(h)\ge C_\sharp\, \big|\widehat u(\xi)\big|^2.$$
As a consequence, we get from~\eqref{F-17} that
$$ \int_0^{+\infty} y^a \Big(
|\xi|^2\,|\phi(\xi, y)|^2 +
\big|\partial_y\phi(\xi, y)\big|^2
\Big)\,dy \ge C_\sharp\,|\xi|^{2s} \,\big|\widehat u(\xi)\big|^2.$$
Integrating over~$\xi\in\Rn\setminus\{0\}$ we obtain that
\[ \int_{\R^{n+1}_+} y^a \Big(
|\xi|^2\,|\phi(\xi, y)|^2 +
\big|\partial_y\phi(\xi, y)\big|^2
\Big)\,d\xi\,dy \ge  C_\sharp\, [\widehat u]_G^2.\]
Hence, by~\eqref{F-13},
\[ [V]_a^2 \ge C_\sharp\, [\widehat u]_G^2,\]
which proves~\eqref{A6-c},
and so~\eqref{A6}.
\end{proof}

We can now prove the main result of this subsection.

\begin{proof}[Proof of Theorem \ref{thmext}] Formula~\eqref{78} follows from the minimality
property in~\eqref{A6}, by writing that
$[U]_a^2 \le [U+\epsilon \varphi]_a^2$
for any~$\varphi$ smooth and compactly supported inside~$\R^{n+1}_+$
and any~$\epsilon\in\R$.

Now we take~$\varphi\in C^\infty_0(\Rn)$ (notice that its support
may now hit~$\{y=0\}$). We define $u_\epsilon:= u+\epsilon\varphi$,
and~$U_\epsilon$
as in~\eqref{F-7-bis--}, with~$\widehat u$ replaced by~$\widehat u_\epsilon$
(notice that~\eqref{F-7-bis--}
is nothing but~\eqref{F-7-bis}), hence we will be able to exploit
Lemma~\ref{L7s}.\\
We also set
$$\varphi_* (x,y):= {\mathcal{F}}^{-1} \Big(\widehat\varphi(\xi)\,g(|\xi|y)\Big).$$
We observe that
\begin{equation}\label{9956}
\varphi_* (x,0)=
{\mathcal{F}}^{-1} \Big(\widehat\varphi(\xi)\,g(0)\Big)=
{\mathcal{F}}^{-1} \Big(\widehat\varphi(\xi)\Big)=\varphi(x)\end{equation}
and that
$$ U_\epsilon= U+\epsilon
{\mathcal{F}}^{-1} \Big(\widehat\varphi(\xi)\,g(|\xi|y)\Big)=
U+\epsilon\varphi_*.$$
As a consequence
$$ [U_\epsilon]_a^2 \, =[U_\epsilon]_a^2 +2\epsilon
\int_{\R^{n+1}_+} y^a \nabla U\cdot \nabla  \varphi_*\,dX+o(\epsilon).$$
Hence, using~\eqref{78}, \eqref{9956} and the Divergence Theorem,
	\begin{equation}\label{67}
		\begin{split}
			[U_\epsilon]_a^2= &\; [U]_a^2 +2\epsilon 	\int_{\R^{n+1}_+} {\rm div}\,
			\Big( \varphi_* \; y^a\nabla_X U \Big) \,dX+o(\epsilon)\\
			 =&\;   [U]_a^2 -2\epsilon \int_{\R^{n}\times\{0\}}
			\varphi \; y^a\partial_y U \,dx+o(\epsilon).
		\end{split}
	\end{equation}
Moreover, from Plancherel Theorem, and the fact that the image of~$\varphi$ is in the reals,
\[ \begin{split}
[\widehat u_\epsilon]_G^2 =&\; [\widehat u]_G +2\epsilon \int_{\Rn}
|\xi|^{2s}\widehat u(\xi)\,\overline{\widehat \varphi(\xi)}\,d\xi+o(\epsilon)\\
=&\;[\widehat u]_G +2\epsilon \int_{\Rn}
{\mathcal{F}}^{-1}\Big( |\xi|^{2s}\widehat u(\xi)\Big)(x)
\,\overline{\varphi(x)}\,dx+o(\epsilon)\\
=&\;[\widehat u]_G +2\epsilon \int_{\Rn} (-\Delta)^s u(x)\,
{\varphi(x)}\,dx+o(\epsilon).
\end{split} \] 
By comparing this with~\eqref{67} and recalling \eqref{A6}
we obtain that
\[ \begin{split}
& [U]_a^2 -2\epsilon
\int_{\R^{n}\times\{0\}}
\varphi \; y^a\partial_y U \,dx+o(\epsilon)=[U_\epsilon]_a^2
 = C_\sharp [u_\epsilon]_G^2
\\=&\; 
C_\sharp [\widehat u]_G +2C_\sharp
\epsilon \int_{\Rn} (-\Delta)^s u(x)\,
{\varphi(x)}\,dx+o(\epsilon)
\\=&\;[U]_a^2+2C_\sharp
\epsilon \int_{\Rn} (-\Delta)^s u\,
{\varphi}\,dx+o(\epsilon) 
\end{split}\]
and so
$$ -\int_{\R^{n}\times\{0\}}
\varphi \; y^a\partial_y U \,dx = C_\sharp
\int_{\Rn} (-\Delta)^s u\,
{\varphi}\,dx,$$
for any~$\varphi\in C^\infty_0(\Rn)$, that is the
distributional formulation of~\eqref{0S}.

Furthermore, by \eqref{F-7-bis--}, we have that
$$ y^a \partial_y
U(x,y)={\mathcal{F}}^{-1} \Big(|\xi|\,\widehat u(\xi)\,y^a\,g(|\xi|y)\Big)
= {\mathcal{F}}^{-1} \Big(|\xi|^{1-a}\,\widehat u(\xi)\,(|\xi|y)^a\,g(|\xi|y)\Big)
.$$
Hence, by~\eqref{89-bis2}, we obtain
\[ \begin{split}
 	\lim_{y\rightarrow0^+} y^a \partial_y U(x,y) =&\; -C_\sharp {\mathcal{F}}^{-1} \Big(|\xi|^{1-a}\,\widehat u(\xi) \Big)\\
 	=&\;-C_\sharp {\mathcal{F}}^{-1} \Big(|\xi|^{2s}\,\widehat u(\xi)  \Big)\\
	=&\;-(-\Delta)^s u(x),
	\end{split} \]
that is the pointwise limit formulation of~\eqref{0S}. This concludes the proof of Theorem \ref{thmext}.
\end{proof}

 \section{An extension problem for the fractional derivative defined by Marchaud}
 
% 
%In the literature there are several definitions of fractional derivatives (see, for instance, the monographs \cite{samkokilbas}, \cite{Samkobook} or \cite{Butzer_Westphal} for an historical introduction). In particular, 
The purpose of this Section is to introduce an extension operator for the fractional derivative introduced by Marchaud and to prove a Harnack inequality for stationary functions (in the sense of Marchaud).  \\
%we are interested in the notion given , see \cite{Marchaud}, who introduced two types of fractional derivatives.
 The left and the right Marchaud fractional derivative of order $s \in(0,1)$ (see \cite{samkokilbas}, formulas 5.57 and 5.58) are respectively defined as follows:
\begin{equation} \label{mdefct}
{\bold{D}}^s_{\pm} f(t)=\frac{s}{\Gamma(1-s)}\int_{0}^{\infty}\frac{f(t)-f(t\mp \tau)}{\tau^{1+s}}d\tau.
\end{equation}
 These fractional derivatives are well defined when $f$ is a bounded, locally H{\"o}lder continuous function in $\mathbb{R}.$  Indeed, we assume that\footnote{Indeed, we have that \begin{equation*}
	\begin{split}
	\bigg|\int_0^\infty \frac{f(t)-f(t-\tau)}{\tau^{s+1}}\, d\tau\bigg| \leq  2\|f\|_{L^{\infty}(\mathbb{R})} \int_1^\infty  \frac{1}{\tau^{s+1}}\, d\tau +\|f\|_{C^{\bar \gamma}(\R)}  \int_0^1  \tau^{\bar{\gamma}-s-1} \, d\tau<\infty,\end{split}\end{equation*}
			given that $\bar{\gamma}>s$.} 
$f\in C^{\bar{\gamma}}(\mathbb{R}),$ for $s<\bar{\gamma} \leq 1$ and $f\in L^{\infty}(\mathbb{R})$. 
    In addition, we just recall here that  the Marchaud derivative can be defined for $s\in (0,n)$ and $n \in \mathbb{N}$, as
\[\bold{D}_{\pm}^s f(t) = \frac{\{s\}}{\Gamma(1-\{s\})} \int_{0}^{\infty} \frac{f^{[s]} (t)-f^{[s]}(t\mp\tau)}{\tau^{1+\{s\}}}\, d \tau,\]
where $[s]$ and $\{s\}$ denote, respectively, the integer and the fractional part of $s$. Our work focuses on the case $n=1$ and, in the first part of the paper, on the left fractional derivative, that we can write using a change of variable, neglecting the constant and omitting for simplicity the subscript symbol $+$, as:
	\begin{equation}\label{frader}
	\begin{split} \bold{D}^s f (t):=  \int_{0}^\infty \frac{f(t)-f(t-\tau)}{\tau^{s+1}}\, d\tau= \int_{-\infty}^t \frac{f(t)-f(\tau)}{(t-\tau)^{s+1}}\, d\tau.
	 \end{split}
	 \end{equation} 
%A short remark on  the right counterpart of the Marchaud fractional derivative is given in Subsection \ref{rightderiv}. 
%Moreover, 
We consider (\ref{frader}) as the definition of our fractional derivative without taking care of what happens when $s\to 0^+$ or $s\to 1^-$. We just remark that $\bold{D}_{\pm}^s\varphi\to\varphi$ as $s\to 0^+$ and $D_{\pm}^s\varphi\to \varphi'$ as $s\to 1^-.$ Indeed, as $s\to 0^+$ or $s\to1^-$, the integral term in (\ref{mdefct}) does not converge but one is able to pass to the limit using the constant term (which, in those cases, plays a fundamental role). 
%Nevertheless, in the Appendix, we briefly discuss this behavior using the definition given in (\ref{mdefct}), since in those cases, the constant plays a fundamental role.

 The operator  $\bold{D}^s$ naturally arises when dealing with a family of singular/degenerate parabolic problems (which, for $s={1}/{2}$, reduces to the heat conduction problem) on the positive half-plane, with a positive space variable and  for all times, namely for $(x,t)\in [0,\infty)\times \mathbb{R} $.  

% More precisely, if $s\not=\frac{1}{2}$ the parabolic equation is  while if $s=\frac{1}{2}$ we get the heat equation.
% 
\noindent In order to construct this extension operator, we exploit the idea recently revisited in \cite{CAFSIL}. In that paper, the fractional Laplacian was characterized via an extension procedure, by means of a degenerate second order elliptic local operator. 
 
Considering the function $\varphi$ of one variable, formally representing the time variable, our approach relies on constructing a parabolic local operator by adding an extra variable, say the space variable, on the positive half-line, and working on the extended plane $ [0,\infty)\times \mathbb{R} $.
 
 The heuristic argument can be described in the simplest case $s={1}/{2}$ as follows. Let $\varphi:\mathbb{R}\to \mathbb{R}$ be a ``good'' function and $U$ be a solution of the problem
  \begin{equation} \label{pro1}\displaystyle
 \left\{ 
% \begin{array}{ll}
  \begin{split}
 &\frac{\partial U}{\partial t} =\frac{\partial^2 U}{\partial x^2}, & & (x,t) \in (0,\infty)\times\mathbb{R}\\
 		&U(0,t)=\varphi(t), & &t \in \mathbb{R}. 
\end{split}
%\end{array}		
		\right.
		\end{equation}
We point out that this is not the usual Cauchy problem for the heat operator, but a heat conduction problem. 

It is known that, without extra assumptions, we can not expect to have a unique solution of the problem (\ref{pro1}), see \cite{tichonov}, Chapter 3.3. Nevertheless, if we denote by $T_{1/2}$ the operator that associates to $\varphi$ the partial derivative ${\partial U}/{\partial x},$ whenever $U$ is sufficiently regular, we have that
$$
T_{1/2}T_{1/2}\varphi=\frac{d\varphi}{dt}.
$$
That is $T_{1/2}$ acts like an half derivative, indeed 
$$
\frac{\partial}{\partial x}\frac{\partial U}{\partial x}(x,t)=\frac{\partial U}{\partial t}(x,t) \underset{x \rightarrow 0^+} \longrightarrow  \frac{d\varphi(t)}{dt}.
$$
The solution of the problem (\ref{pro1}) under the reasonable assumptions that $\varphi$ is bounded and H{\"o}lder continuous, is explicitly known (check \cite{tichonov}, Chapter 3.3) to be
\begin{equation*}\begin{split} U(x,t)=&\, c {x} \int_{-\infty}^t \displaystyle e^{-\frac{x^2}{4(t-\tau)}}{(t-\tau)^{-\frac{3}{2}}}\varphi(\tau)\, d\tau\\
= &\,c {x}  \int_{0}^{\infty} \displaystyle e^{-\frac{x^2}{4\tau}}{\tau^{-\frac{3}{2}}}\varphi (t-\tau)\, d\tau,
\end{split}
\end{equation*}
where the last line is obtained with a change of variables. Using $\displaystyle t ={x^2}/{(4\tau)}$ and the integral definition \eqref{ABRAMOWITZ} of the Gamma function  we have that 
\[ \int_0^\infty x e^{-\frac{x^2}{4\tau}} \tau ^{-\frac{3}2} \, d\tau = 2 \int_0^\infty e^{-t} t^{-\frac{1}2}\, dt = 2 \Gamma\left(\frac{1}{2}\right).\]
Hence,
$$ \frac{U(x,t)-U(0,t)}{x}= c\int_{0}^{\infty}e^{-\frac{x^2}{4\tau}}{\tau^{-\frac{3}{2}}}\left(\varphi(t-\tau)-\varphi(t)\right)d\tau,$$
choosing $c$ that takes into account the right normalization. This yields, by passing to the limit, that
\begin{equation*}
-\lim_{x\to 0^+}\frac{U(x,t)-U(0,t)}{x}=c\int_{0}^{\infty}\frac{\varphi(t)-\varphi(t-\tau)}{\tau^{\frac{3}{2}}}d\tau.
\end{equation*}
Hence, with the right choice of the constant, we get exactly $\bold{D}^{1/2}\varphi$ (see \eqref{frader}), i.e. the Marchaud derivative  of order $1/2$ of $\varphi.$

Now we are in position to state our main result.
 
\begin{theorem} \label{teo:mainstat}
Let $s\in (0,1) $ and $\bar{\gamma}\in (s,1]$ be fixed. Let $\varphi \in C^{\bar{\gamma}}(\mathbb{R})$ 
be a bounded function 
and let $U\colon [0,\infty)\times \mathbb{R}\to \mathbb{R}$ be a solution of the problem 
\begin{equation}\label{marprob1}
\left\{
%\begin{array}{ll} 
\begin{split}
	&\frac{\partial U}{\partial t} (x,t)= \frac{1-2s}{x} \frac{\partial U }{\partial x}(x,t)+ \frac{\partial^2 U }{\partial x^2}(x,t), & &    (x,t)\in(0,\infty)\times \mathbb{R}\\
	 & U(0,t)=\varphi(t),  &&t\in\mathbb{R}\\
	  & \lim_{x \to +\infty} U(x,t)=0,&&t\in\mathbb{R}.
%\end{array}
\end{split}
\right.
\end{equation}
Then $U$ defines the extension operator for $\varphi$, such that
\begin{equation*}
\bold{D}^s \varphi(t)=-\lim_{x\to 0^+} c_s x^{-2s}(U(x,t)- \varphi(t)),\quad \mbox{ where }\quad c_s= 4^s\Gamma(s).
\end{equation*}
 \end{theorem}

We notice that one can write
	\begin{equation} \label{mainstat1} \bold{D}^s \varphi(t)= -\lim_{x\to 0^+} c_sx^{1-2s}\frac{\partial U}{\partial x}(x,t),
	\end{equation}
	in analogy with formula (3.1) in \cite{CAFSIL}.

\begin{remark}
The extension operator satisfies, as one would expect, up to constants that
 	\[ \bold{D}^{1-s} \bold{D}^{s}  \varphi(t)= \varphi'(t).\] 
 	Indeed, using \eqref{mainstat1} and thanks to \eqref{marprob1} we have that
\begin{equation*}
\begin{split}
\bold{D}^{1-s} \bold{D}^s \varphi(t) =&  \lim_{x\to 0^+} x^{2s-1} \frac{\partial }{\partial x}\left(x^{1-2s} \frac{\partial U}{\partial x}(x,t)\right)\\=  & \lim_{x\to 0^+} \frac{\partial^2{U}}{\partial x^2}(x,t) + \frac{1-2s}x \frac{\partial U}{\partial x}(x,t)\\
=&\lim_{x\to 0^+} \frac{\partial {U}}{\partial t}(x,t)=\frac{\partial {U}}{\partial t}(0,t)=\varphi'(t) .
\end{split}
\end{equation*}
	\end{remark}
	
An interesting application that follows from this extension procedure is a Harnack inequality for Marchaud-stationary functions in an interval $J \subseteq {\mathbb{R}},$ namely for functions that satisfy  $\bold{D}^s \varphi=0$  in  $J.$ This fact is not obvious, indeed the set of functions determined by fractional-stationary functions (on an interval) is nontrivial, see Section \ref{capdens}.  
\begin{theorem}\label{teoHarn}
Let $s\in (0,1)$. There exists a positive constant $\gamma$ such that, if $\bold{D}^s\varphi=0$ in an interval $J\subseteq \mathbb{R}$ and $\varphi\geq 0$ in $\mathbb{R}$, then  
\begin{equation*}\label{Harnack_intro}
\sup_{[t_0-\frac{3}{4}\delta,t_0-\frac{1}{4}\delta]}\varphi\leq \gamma \inf_{[t_0+\frac{3}{4}\delta,t_0+\delta]}\varphi
\end{equation*}
for every $t_0\in \mathbb{R}$ and for every  $\delta >0$ such that $[t_0-\delta,t_0+\delta]\subset J$.
\end{theorem}

The previous result can be deduced from the Harnack inequality proved in \cite{CS} for some degenerate parabolic operators (see also \cite{FKS} for the elliptic setting). In particular, the constant $\gamma$ used in  Theorem \ref{teoHarn} is the same that appears in the parabolic case in \cite{CS}. 

In addition, we remark that Theorem \ref{teoHarn} does not give the usual Harnack inequality for elliptic operators, where the comparison between the supremum and the infimum is done on the same set, e.g. the same metric ball.
This Harnack inequality for the Marchaud-stationary functions inherits the behavior of its parabolic extension. 

We point out at this point the very interesting paper \cite{Torreault}. Indeed, after we have submitted our paper, we learnt from professor Jos\' e L. Torrea about the results contained in  his joint paper where an extension procedure for a class of operators has been studied.

\subsection{The extension parabolic problem}

In this subsection we find a solution of the system \eqref{marprob1}. At first, we introduce a particular kernel, that acts as the Poisson kernel. We then look for a particular solution of the system by means of the Laplace transform, and in this way we show how the solution arises. Finally, by a straightforward check, it yields that indeed the indicated solution satisfies the problem \eqref{marprob1}.

% \subsection{Properties of the kernel $\Psi_s$}
We 
%In this section we introduce and 
study
at first the properties of a kernel, that acts as the Poisson kernel for the problem \eqref{marprob1}. The readers can see Section $3$ in \cite{stinga1}, where this kernel is studied in a more general framework.\\
\noindent We define  for every $x\in \mathbb{R},$
\begin{equation*} \label{kern} \Psi_s(x,t) := \left\{
\begin{split} &\frac{1}{4^s \Gamma(s)} x^{2s} e^{-\frac{x^2}{4t}} t^{-s-1},& \mbox{ if }& t> 0 ,\\
								&0, & \mbox{ if }& t\leq 0 .
								\end{split}
								\right.
								\end{equation*} 
	Also, let
	\begin{equation*}
	\psi_s(t) :=\left\{
	  \begin{split}& \frac{1}{4^s \Gamma(s)} e^{-\frac{1}{4t}} t^{-s-1},& \mbox{ if }& t> 0 ,\\
								&0, & \mbox{ if }& t\leq 0 
								\end{split}
								\right.
								\end{equation*}
and notice that
	\begin{equation} \label{bla4} \int_\mathbb{R} \Psi_s(x,t)\, dt= \int_\mathbb{R} \psi_s(t)\, dt.
	\end{equation}
	Indeed, we have by changing the coordinate $\tau={t}/{x^2}$ that
	\begin{equation*}
	\begin{split}  \int_{\mathbb{R}} \Psi_s(x,t)\, dt  =& \frac{1}{4^s\Gamma(s)} \int_0^\infty x^{2s}e^{-\frac{x^2}{4t}} t^{-s-1} \, dt\\ 
	&=\frac{1}{4^s\Gamma(s)} \int_0^\infty e^{-\frac{1}{4\tau}} \tau^{-s-1} \, d\tau\\
	&= \int_\mathbb{R} \psi_s(t)\, dt.
	\end{split}
	\end{equation*}
	
The kernel $\Psi_s$ satisfies also the following property:
\begin{equation}\label{kencalc1} \int_{\mathbb{R}} \Psi_s(x,t)\, dt = 1.
\end{equation}
Indeed, by changing the variable $ t =  {1}/{(4\tau)}$ we get that
	\begin{equation} \label{bla2}
	\begin{split}
	\int_\mathbb{R} \psi_s(\tau)\, d\tau = \frac{1}{4^s\Gamma(s)}   \int_0^\infty e^{-\frac{1}{4\tau}} \tau^{-s-1} \, d\tau
	 =  \frac{1}{\Gamma(s)} \int_0^\infty e^{-t} t^{s-1}\, d t =1,
	 \end{split}
	 \end{equation}
	 thanks to the integral definition of the Gamma function (see \eqref{ABRAMOWITZ}).
	% \[ \Gamma(s) =\int_0^\infty e^{-t} t^{s-1}\, d t,\]
	 It follows from \eqref{bla4} that
				  $$ \int_{\mathbb{R}} \Psi_s(x,t)\, dt  =1.  $$ 		

Taking the Laplace transform
% (see for instance \cite{dyke} for details) 
 of the kernel $\Psi_s$ (see e.g. \cite{dyke} for details on this integral transform), we have the following result involving the modified Bessel function of the second kind $\bold K_s,$ see \cite{MaGoberSo} and
 \cite{ABRAMOWITZ}, \textsection 9.6. We use here the notation $\Re \omega >0$ to denote the real part of a complex number $\omega$.
% . We have the property:
 \begin{lemma}The Laplace transform of the function $\psi_s \in L^1(\mathbb{R})$ is
\begin{equation} \label{fourfpsi} \mathcal{L} (\psi_s)(\omega)= \frac{1}{2^{s-1}\Gamma(s)} \omega^{\frac{s}2}\bold K_s(\sqrt{\omega}) \mbox{ for } \Re \omega >0.\end{equation}
Moreover, the Laplace transform with respect to the variable $t$ of the kernel $\Psi_s \in L^1(\mathbb{R}, dt)$ is
\begin{equation} \label{calcKpsi} \mathcal{L} (\Psi_s)(x,\omega)=\frac{1}{2^{s-1}\Gamma(s)} x^s \omega^{\frac{s}2}\bold K_s(x\sqrt{\omega}) \mbox{ for } \Re \omega >0. 
\end{equation}
\end{lemma}	
\begin{proof}
If one proves claim \eqref{fourfpsi}, the identity \eqref{calcKpsi} follows by changing the variable $\tau=t/x^2$.
%We have that
% $$ \mathcal{L} (\Psi_s)(x,\omega) = \frac{1}{4^s\Gamma(s)}\int_0^\infty x^{2s} e^{-\frac{x^2}{4t} } t^{-s-1} e^{-\omega t} \, dt.$$
%  Taking $t=x^2\tau$ (and recalling that $x>0$), we obtain
% $$ \mathcal{L}  (\Psi_s)(x,\omega) =  \frac {1}{4^s\Gamma(s)}\int_0^\infty  e^{-\frac{1}{4\tau} } \tau^{-s-1} e^{-\omega (x^2 \tau)} \, d\tau\\
% = \mathcal{L}  (\psi_s)(x^2 \omega).$$ 
For $\Re a>0$ and  $\omega\in \mathbb{C}$ with $\Re\omega >0$, as stated in formula 5.34 in \cite{oberl}, we have that
 	\[\mathcal{L} \left(t^{\gamma-1} e^{-\frac{a}{t} }\right) = 2 \left(\frac{a}{\omega}\right)^{\frac{\gamma}2} \bold K_\gamma \left(2(a\omega)^{\frac{1}{2}}\right)  .\] 
 Taking $\gamma=-s$ and $a=1/4$, recalling that $\bold K_s=\bold K_{-s}$, we obtain that
% \[ \int_0^\infty e^{-\frac{1}{4\tau}}  \tau^{-s-1} e^{-\omega \tau}\, d\tau= 2^{s+1} {\omega}^{\frac{s}{2}}\bold K_s(\sqrt{ \omega}).\] 
%Hence, multiplying by  $4^s\Gamma(s)$ we get:
	\[ \mathcal{L} (\psi_s)(\omega) =\frac{1}{4^s\Gamma(s)} \mathcal{L} \Big(e^{-\frac{1}{4\tau}}  \tau^{-s-1} \Big) = \frac{1}{2^{s-1}\Gamma (s) } {\omega}^{\frac{s}{2}}\bold K_s(\sqrt{ \omega})\] and thus \eqref{fourfpsi}. This concludes the proof of the Lemma.
\end{proof}			
   
%\subsection{Existence of the solution}
%\label{exuniq}

We recall now a useful result (see \cite{FerFra}, Proposition 4.1) involving the modified Bessel function of the second kind. 
\begin{prop}{\label{prop:FerFra}} 
If $-\infty <\alpha<1$, the boundary value problem 
	\begin{equation} \label{proby1}
	\left\{
	\begin{split}
	%\begin{array}{ll}
		 &x^{\alpha} y''(x) =y(x), & \mbox{in } &(0,\infty)\\
		 &y(0)=1, &  &\\
		 &\lim_{x \to \infty} y(x) =0. &&
		% \end{array}
		\end{split}
		 \right.
		 \end{equation}
		has a solution $y \in C^{2-\alpha} \left([0,\infty)\right)$ of the form
	\begin{equation*}
	%\label{soly1} 
			y(x)= c_{k} x^{\frac{1}2}\bold{K}_{\frac{1}{2k}} \left(\frac{t^k}{k}\right),
			\end{equation*} where $c_k$ is the positive constant
			\[ c_k = \frac{2^{1-\frac{1}{2k}} k^{-\frac{1}{2k}}}{\Gamma \left(\frac{1}{2k}\right)} \quad \mbox{ and } \quad  k := \frac{2-\alpha}{2}.\]
				\end{prop}

We show in the next rows how the solution of the problem \eqref{marprob1} arises, using the Laplace transform. So, we look for a possible candidate of a solution
% of the problem \eqref{marprob1}. In particular,  we put ourselves 
 in the simplified situation in which $U$ has a sub-exponential growth in $t$, and in which the function $\varphi$ is zero on the negative semi-axis $(-\infty,0]$.
 Under this additional hypothesis, we take the Laplace transform in $t$ of the system \eqref{marprob1}. Since the Laplace transform of the derivative of a function gives
$$ \mathcal{L} (f')(\omega) =  \omega \mathcal{L}f(\omega),$$ we get that
\begin{equation*}\left\{
%\begin{array}{ll}
\begin{split}
	& \omega \mathcal{L} U(x,\omega) = \frac{1-2s}x \frac{\partial \mathcal{L} U}{\partial x}(x,\omega) + \frac{\partial^2{ \mathcal{L} U}}{\partial x^2}(x,\omega), 	&&\mbox{in } \quad (0,\infty)\times \mathbb{C}\\
	&\mathcal{L} U(0,\omega)=\mathcal{L} \varphi(\omega), & &\mbox{in } \quad\mathbb{C}\\
	&\lim_{x \to +\infty} \mathcal{L}U(x,\omega)=0, &&\mbox{in } \quad\mathbb{C}.
	%\end{array}
	\end{split}
	\right.
	\end{equation*}
We define for any fixed $\omega \in \mathbb{C}$
\begin{equation} \label{deff} f(x):= \mathcal{L} U(x,\omega),
\end{equation}
 then $f $ must be a solution of the system
 	\begin{equation} \label{probf1}
	\left\{
\begin{split}
 	  &\omega  f(x)= \frac{1-2s}x  f'(x) +  f''(x),	 &\quad\mbox{in } &(0,\infty)\\
	&f(0)=\mathcal{L} \varphi(\omega)  &&\\
	&\lim_{x \to +\infty} f(x)=0.&&
\end{split}
\right.
\end{equation}	
	We assume here that for any $\omega\in \mathbb{C}$, $ \mathcal{L} \varphi(\omega)  \neq 0$.\\
	We take in Proposition \ref{prop:FerFra}, $\alpha={(2s-1)}/{s}$ (notice for $s\in(0, 1)$ that $\alpha\in (-\infty,1)$) and $y(x)$ to be the solution there introduced. We claim that taking
		\begin{equation*}\label{chvarfy} f(x) = \mathcal{L} \varphi (\omega) y\left( \omega^s \left(\frac{x}{2s}\right)^{2s}\right),
		\end{equation*}
		$f(x)$ is a solution of the system \eqref{probf1}.
%		one obtains that $f$ satisfies problem \eqref{probf1}.
	Indeed, 
$f(0)= \mathcal{L} \varphi(\omega)$ and
%		\[ y'\left( \omega^s \left(\frac{x}{2s}\right)^{2s}\right)= \frac{1}{\mathcal{L} \varphi(\omega) }\omega^{-s} (2s)^{2s-1} x^{1-2s} f' (x),\]
%		and 
		\begin{equation*}\begin{split} y''\left( \omega^s \left(\frac{x}{2s}\right)^{2s}\right) &=  f''(x) \frac{1}{\mathcal{L} \varphi(\omega)} \omega^{-2s} (2s)^{4s-2} x^{2-4s} \\ 
		& + f'(x)\frac{1-2s}{\mathcal{L} \varphi (\omega)}(2s)^{4s-2}\omega^{-2s}  x^{1-4s} .
		\end{split}
		\end{equation*}
		Since $y(x)$ satisfies the system \eqref{proby1} we have that
			\[ \left( \omega^s \left(\frac{x}{2s}\right)^{2s}\right)^{\frac{2s-1}s} y''\left( \omega^s \left(\frac{x}{2s}\right)^{2s}\right) =y\left(\omega^s \left(\frac{x}{2s}\right)^{2s}\right) .\] This implies that 
			\[\omega f(x)=f''(x) +(1-2s) x^{-1} f'(x),\] which yields that $f$ is a solution of \eqref{probf1}. \\
Now, from Proposition \ref{prop:FerFra} we have $k={1}/{(2s)}$ and
	\begin{equation*}
	% \label{soly2}
	 y(x)= \frac{2^{1-s}(2s)^s}{\Gamma(s)}x^{\frac{1}{2}} \bold{K}_s\left(2s x^{\frac{1}{2s}}\right).
	\end{equation*}
	And so we get that
	$$ f(x)= \mathcal{L} \varphi(\omega) \frac{2^{1-s}}{\Gamma(s)} \omega^{\frac{s}2} {x^s} \bold{K}_s(x\sqrt{\omega} ) .$$
	We use \eqref{deff}, take the 
%	\[ \mathcal{L} U(x,\omega) = \mathcal{L}\varphi(\omega) \frac{2^{1-s}}{\Gamma(s)} \omega^{\frac{s}2} {x^s} \bold{K}_s(x\sqrt{\omega} ).\]
inverse Laplace transform, and recall that the pointwise product is taken into the convolution product to obtain that
		\begin{equation*} U(x,t)= \frac{2^{1-s}}{\Gamma(s)} \varphi * \mathcal{L}^{-1} \left( \omega^{\frac{s}2}{x^s}  \bold{K}_s(x\sqrt{\omega} ) \right)(t).
		\end{equation*}
			And so, using \eqref{calcKpsi}, we get the following representation formula for the system \eqref{marprob1}: 
%		$$ \mathcal{L} (\Psi_s)(x,\omega)= \frac{\Gamma(s) }{2^{1-s}}  x^s \omega^{\frac{s}2} \bold K_s \left(x \sqrt{\omega}\right).$$
%			Hence in \eqref{bla1} we obtain the following Laplace convolution
			\begin{equation*}\label{bla3} U(x,t) = \varphi*\Psi_s(x,t) = \int_0^t \Psi(x,\tau) \varphi(t-\tau) \,d\tau.\end{equation*}

We recall that we obtained the above formula by taking the function $\varphi$ to be vanishing in $(-\infty,0)$. 
However, it is reasonable to suppose that this formula holds true also for a function that is not a signal. Hence,  we take $\varphi$
% is defined on the entire axis  $\mathbb{R}$,  without the assumption that
that does not vanish in $(-\infty,0)$ and claim that $\varphi*\Psi_s$ still defines a solution of the problem \eqref{marprob1}.
Indeed, we show the following existence theorem:

\begin{theorem}\label{teo:sol}
There exists a continuous solution of the problem \eqref{marprob1} given by
\begin{equation*} U(x,t)	=   \Psi_s(x,\cdot)*\varphi (t) := \int_{\mathbb{R}} \Psi_s(x,\tau)\varphi (t-\tau)\, d\tau.
\end{equation*}
More precisely (inserting the definition \eqref{kern}) we have that
	\begin{equation}\label{solsts2} U(x,t)=\frac{1}{4^s \Gamma(s)} x^{2s} \int_0^{\infty} e^{-\frac{x^2}{4\tau}} \tau^{-s-1}  \varphi(t-\tau) \, d\tau.
	\end{equation}
\end{theorem}
\begin{proof}
We define
	\begin{equation*} A_{x,\tau}:=  
	\left\{
	\begin{split}  
	&e^{-\frac{x^2}{4\tau}}\tau^{-s-1}, &\mbox{if } &\tau>0\\
	&0, &\mbox{if } &\tau\leq 0 
	\end{split}
	\right.
	\end{equation*}
	and notice that
		\begin{equation*}
		\frac{ \partial A_{x,\tau}}{\partial x}= \left\{
		\begin{split} &-\frac{x}{2\tau}A_{x,\tau}, &\mbox{if } &\tau>0\\
	 &0,&\mbox{if } &\tau\leq 0.
	 \end{split}
	 \right.
	 \end{equation*}
	Let 
		$$ V(x,t):=4^s\Gamma(s) U(x,t) =  x^{2s}\int_\mathbb{R} A_{x,\tau}\varphi(t-\tau)\, d\tau,$$ where we have introduced the notation $A_{x,\tau}$ into \eqref{solsts2}.
Taking the derivative with respect to $x$ of $V(x,t)$ we have that
	$$ \frac{\partial V}{\partial x} (x,t)= 2s x^{2s-1} \int_\mathbb{R} A_{x,\tau} \varphi(t-\tau) \, d\tau -\frac{x^{2s+1}}2  \int_\mathbb{R} \frac{A_{x,\tau}}{\tau} \varphi(t-\tau) \, d\tau,$$
	and that
	\begin{equation*}\begin{split}\frac{\partial^2{V}}{\partial x^2}(x,t)=&\, 2s(2s-1) x^{2s-2}  \int_\mathbb{R} A_{x,\tau} \varphi(t-\tau) \, d\tau\\ &\,
	 - \frac{(4s+1)x^{2s}}{2}  \int_\mathbb{R} \frac{A_{x,\tau}}{\tau} \varphi(t-\tau) \, d\tau + \frac{x^{2s+2}}4  \int_\mathbb{R} \frac{A_{x,\tau}}{\tau^2} \varphi(t-\tau) \, d\tau.
	\end{split}
	\end{equation*}
Then, by changing variables, we write
	\[ V(x,t)=x^{2s} \int_\mathbb{R} A_{x,t-\tau} \varphi(\tau) \, d\tau ,\] and taking the derivative with respect to $t$, we get that
	\[\frac{\partial { V}}{\partial t}(x,t)= x^{2s} \int_\mathbb{R}  \left[x^2 \frac{A_{x,t-\tau}}{4(t-\tau)^2} \varphi(\tau) -(s+1) \frac{A_{x,t-\tau}}{(t-\tau) }\varphi(\tau)\right]\, d\tau.\]
	We change back variables to obtain	
	\[\frac{\partial { V}}{\partial t}(x,t)  =x^{2s+2}\int_\mathbb{R}  \frac{A_{x,\tau}}{4\tau^2 } \varphi(t-\tau)\, d\tau  -(s+1)x^{2s} \int_\mathbb{R} \frac{A_{x,\tau}}{\tau }\varphi(t-\tau) \, d\tau.\]
	By substituting these computations, we obtain that indeed $V$, hence $U$ by the definition of $V$, satisfies the equation 
	\[ \frac{\partial { U}}{\partial t} (x,t)= \frac{1-2s}{x}\frac{\partial U}{\partial x} (x,t)+\frac{\partial^2{U}}{\partial x^2}(x,t).\] 
		Moreover, using for $x$ large enough the bound
		\[ x^{2s}e^{-\frac{x^2}{4\tau}} \leq M e^{-\frac{1}{4\tau}},\] 
		thanks to the Dominated Convergence Theorem and the limit
			\[\lim_{x\to +\infty} x^{2s} e^{-\frac{x^2}{4\tau}} =0,\]
			it yields that \[\lim_{x\to +\infty} U(x,t)=0.\]
Furthermore, in \eqref{solsts2} by changing the variable $\tilde \tau = {\tau}/{x^2}$ (but still writing $\tau$ as the variable of integration), we have that
		$$ U(x,t)= \frac{1}{4^s\Gamma(s)} \int_0^\infty e^{-\frac{1}{4\tau}} \tau^{-s-1} \varphi(t-\tau x^2) \, d\tau.$$
		Since $\varphi$ is bounded, by the Dominated Convergence Theorem, we have that
		$$\lim_{x\to 0^+} U(x,t) = \frac{\varphi(t)}{4^s\Gamma(s)}  \int_0^\infty e^{-\frac{1}{4\tau}} \tau^{-s-1}  \, d\tau
		=\varphi(t),$$
		according to \eqref{bla2}. This proves the continuity up to the boundary of the solution $U,$ concluding the proof of the Theorem. 
%		that the function $U$ defined in \eqref{solsts1} is a continuous solution to the problem \eqref{marprob1}.	 
\end{proof}

%\subsection{Relation with the Marchaud fractional derivative}

%\setcounter{section}{3}
%\setcounter{equation}{0}\setcounter{theorem}{0}

We  prove here that
%the relation between the parabolic equation given in \eqref{marprob1} studied in Subsection \ref{exuniq} and the Marchaud fractional derivative. Namely, 
 the Marchaud derivative is obtained as the trace operator of the extension given by the solution of problem \eqref{marprob1} obtained in Theorem \ref{teo:sol}. Namely, we prove the following theorem.
% of \eqref{marprob1}.

\begin{proof}[Proof of Theorem \ref{teo:mainstat}]
By inserting the expression of $U(x,t)$ from \eqref{solsts2}, we compute 
\begin{equation*}
\begin{split}& \lim_{x \to 0^+} x^{-2s} \left(U(x,t)-\varphi(t)\right)\\
& = \lim_{x \to 0^+}x^{-2s}\left( \frac{1}{4^s\Gamma(s)} \int_0^\infty x^{2s}e^{-\frac{x^2}{4\tau}} \tau^{-s-1} \varphi(t-\tau) \, d\tau - \varphi (t)\right).
\end{split}
\end{equation*}
Recalling property \eqref{kern} of the kernel, we have that
	\begin{equation*}\begin{split} \lim_{x \to 0^+} x^{-2s} &\left(U(x,t)-\varphi(t)\right) \\
	&= \lim_{x \to 0^+} \frac{x^{-2s}}{4^s\Gamma(s)} \int_0^\infty x^{2s}e^{-\frac{x^2}{4\tau}} \tau^{-s-1} \left( \varphi(t-\tau) - \varphi (t)\right) \, d\tau\\
	&= \lim_{x\to 0^+} \frac{1}{4^s \Gamma(s)} \int_0^\infty e^{-\frac{x^2}{4\tau}}  \frac{ \varphi(t-\tau)-\varphi(t)}{\tau^{s+1}}\, d\tau.
	\end{split}
	\end{equation*}
Now \[ e^{-\frac{x^2}{4\tau}}  \leq 1\] and since $\varphi$ is bounded, we have that 
	\[ \frac{|\varphi(t-\tau)-\varphi(t) |}{\tau^{s+1}} \leq 2M \tau^{-s-1} \in L^1\left((1,\infty)\right).\]
	On the other hand, recalling that $\varphi$ is $C^{\bar{\gamma}}(\mathbb{R})$ we have that
		\[ |\varphi(t) - \varphi(t-\tau)|\leq c\tau^{\bar{\gamma}}.\]
		Hence, since $\bar{\gamma}>s$,
			\[    \frac{|\varphi(t-\tau)-\varphi(t) |}{\tau^{s+1}} \leq  c \tau^{\bar{\gamma}-s-1} \in L^1\left((0,1)\right).\]
			Using the Dominated Converge Theorem, we obtain
\begin{equation}\label{derivative_right}
 \begin{split}\lim_{x \to 0^+} x^{-2s} \left(U(x,t)-\varphi(t)\right) &= \frac{1}{4^s \Gamma(s)} \int_0^\infty \lim_{x\to 0^+} e^{-\frac{x^2}{4\tau}}  \frac{ \varphi(t-\tau)-\varphi(t)}{\tau^{s+1}} \, d\tau \\
				&= \frac{1}{4^s \Gamma(s)} \int_0^\infty\frac{ \varphi(t-\tau)-\varphi(t)}{\tau^{s+1}}\, d\tau.
				\end{split}
				\end{equation}
	And so for $c_s=4^s\Gamma(s),$
		\begin{equation*}
		\begin{split} -  c_s \lim_{x \to 0^+} x^{-2s} \left(U(x,t)-\varphi(t)\right) = \int_0^\infty\frac{ \varphi(t)-\varphi(t-\tau) }{\tau^{s+1}}\, d\tau
		= \bold{D}^s \varphi(t)
		\end{split}
		\end{equation*}
		by definition \eqref{frader}. This concludes the proof of Theorem \ref{teo:mainstat}.
\end{proof}

We make a short remark on the right Marchaud fractional derivative (denoted by ${\bold{D}}^s_-\varphi$) and the backward equation. The following result is true:
\begin{theorem}
Let $s\in (0,1) $ and $\bar{\gamma}\in (s,1]$ be fixed. Let $\varphi \in C^{\bar{\gamma}}(\mathbb{R})$ be a bounded function and let $U_-\colon [0,\infty)\times \mathbb{R}\to \mathbb{R}$ be a solution of the problem
\begin{equation}
 \label{prob2}
\left\{ \begin{split}
 &  -\frac{\partial {U(x,t)}}{\partial t} = \frac{1-2s}x  \frac{\partial{U(x,t)}}{\partial x} + \frac{\partial^2{U(x,t)}}{\partial x^2}, &  & (x,t)\in(0,\infty)\times \mathbb{R}\\
	& U(0,t)=\varphi(t), &&t\in\mathbb{R}\\
	  & \lim_{x \to +\infty} U(x,t)=0. && 
\end{split}
\right.
\end{equation}

Then $U_-$ defines the extension operator for $\varphi$, such that
\begin{equation*}
%\label{mainstat2} 
\bold{D}^s_- \varphi(t)=-\lim_{x\to 0^+} c_s x^{-2s}(U_-(x,t)- \varphi(t)) , \quad \mbox{ where } \quad c_s= 4^s\Gamma(s).
\end{equation*}
\end{theorem}

%We do not repeat all the computations, that are very similar to the case of the left Marchaud-derivative ${\bold{D}}^s\equiv {\bold{D}}^s_+.$
The proof follows similarly to the proof of Theorem \ref{teo:mainstat}.
We only point out that if $U_-$ is a solution of \eqref{prob2}, then $U_-(x,t)=U(x,-t),$ where $U$ is the solution of the differential equation in \eqref{marprob1}.
% (\ref{teo:mainstat}).
% As in Theorem \ref{teo:sol} 
% %and keeping in mind (\ref{solsts2}),
% we get that
%
%\begin{equation*}
%%\label{solsts2_back} 
%U_-(x,t)=\frac{1}{4^s \Gamma(s)} x^{2s} \int_0^{\infty} e^{-\frac{x^2}{4\tau}} \tau^{-s-1}  \varphi(t+\tau) \, d\tau.
%\end{equation*}

%
%Recalling the computations in (\ref{derivative_right}) and the properties of the kernel $\Psi_s$ (see formula \eqref{kencalc1}),  we obtain that
%
%\begin{equation*}
%\begin{split}
%\lim_{x \to 0^+} x^{-2s} &\left(U_-(x,t)-\varphi(t)\right)\\ = &\,\lim_{x \to 0^+} \frac{x^{-2s}}{4^s\Gamma(s)} \int_0^\infty x^{2s}e^{-\frac{x^2}{4\tau}} \tau^{-s-1} \left( \varphi(t+\tau) - \varphi (t)\right) \, d\tau\\
%	= &\,\lim_{x\to 0^+} \frac{1}{4^s \Gamma(s)} \int_0^\infty e^{-\frac{x^2}{4\tau}}  \frac{ \varphi(t+\tau)-\varphi(t)}{\tau^{s+1}} \, d\tau.
%\end{split}
%\end{equation*}	
%Then, using the same argument as in the proof of Theorem \ref{teo:mainstat}, we conclude that
%\begin{equation*}
%\begin{split}
%\lim_{x \to 0^+} x^{-2s} \left(U_-(x,t)-\varphi(t)\right) &= \frac{1}{4^s \Gamma(s)} \int_0^\infty \frac{ \varphi(t+\tau)-\varphi(t)}{\tau^{s+1}}\, d\tau,
%\end{split}
%\end{equation*}	
%that is 
%$$
%{\bold{D}}_-^s\varphi(t)=-c_s\lim_{x \to 0^+} x^{-2s} \left(U_-(x,t)-\varphi(t)\right). 
%$$
%It is worth to say that ${\bold{D}}_-^{1-s}{\bold{D}}_-^s\varphi(t)=-\displaystyle \frac{d\varphi}{dt}.$ Hence, using a different notation we can write that
%$$
%{\bold{D}}_+^s\varphi(t)=\left(\frac{d}{dt}\right)^s\varphi,\quad {\bold{D}}_-^s\varphi(t)=\left(-\frac{d}{dt}\right)^s\varphi.
%$$

\subsection{Applications: a Harnack inequality for Marchaud-stationary functions}

In this subsection
%part of the paper 
we prove a Harnack inequality for functions that have a vanishing Marchaud derivative in a bounded interval $J$, namely we prove here Theorem \ref{teoHarn}. At this purpose, we use a known Harnack inequality for degenerate parabolic operators, that can be found in \cite{CS}, see Theorem 2.1. There, the result is given in its generality, in $\mathbb{R}^n$.  For the reader's convenience we recall in Proposition \ref{Chiarenza_Serapioni} this result in the case $n=1.$ 

% \subsection{Preliminary notions}
We 
%would like to 
point out that the result given in \cite{CS} was proved for $n\geq 3$. Nevertheless the same proof works also for $n = 1$ with some adjustments. We recall here the hypotheses we need, adapted in our case $n=1$. It is worth to say that this problem has been studied in a more general fashion in \cite{GuWhee2} and \cite{GuWhee1}.

The degenerate parabolic 
\begin{equation}\label{specific_heat}
w(x)\frac{\partial u}{\partial t}=\frac{\partial }{\partial x}{}\left(w(x)\frac{\partial u}{\partial x}\right),
\end{equation}
is given in $Q=(-R,R)\times (0,T)$, for $R>0$. The weight $w$ has to satisfy an integrability condition (also known as a Muckehoupt, or $A_2$ weight condition), given by  
\begin{equation}\label{2-weight}
\sup_{J}\left(\frac{1}{|J|}\int_{J}w(x)\, dx\right)\, \left(\frac{1}{|J|}\int_{J}\frac{1}{w (x)}\, dx\right)=c_0<\infty,
\end{equation}
for any interval $J \subseteq (-R,R)$. The constant $c_0$ is indicated as the $A_2$ constant of $w$. \\
In this particular case we give here in \eqref{specific_heat}, the conductivity coefficient (i.e. the coefficient in front of the $x$ derivative) and the specific heat (the coefficient of the $t$ derivative) coincide. 
A more general form of the equation in $\mathbb R$ can be given in these terms: 
\begin{equation}\label{genst1}
w(x)\frac{\partial u}{\partial t}=\frac{\partial u}{\partial x}\left(a(x)\frac{\partial u}{\partial x}\right),
\end{equation}
i.e. when the conductivity and the specific heat are not equal.
In that case, one has to require, besides condition \eqref{2-weight}, that
	\[ \lambda^{-1} w(x) \leq a(x)\leq \lambda w(x). \] 

In addition we consider the functional space
\[ W:=\left\{u\in L^2(0,T;H^{1}_0(J,w)) \mbox{ s.t. }\quad \frac{\partial u}{\partial t}\in  L^2(0,T;L^2(J,w))\right\}.\]
We denote here by $L^2(J,w),$ the Banach space of measurable functions $u$ with finite weighted norm
$$
\|u\|_{2,w;J}=\left(\int_{J}|u|^2 w \,dx\right)^{1/2}<\infty,
$$
by $H^1(J,w)$ the completion of $C^{\infty}(\overline{J})$ under the norm
$$
\|u\|_{1,w;J}=\left(\int_{J}(u^2+|\partial_x u|^2)w \, dx\right)^{1/2}
$$
and by $H^1_0 (J,w)$ the completion of $C^{\infty}_0(J)$ under the norm
$$
\|u\|_{1,w;J}=\left(\int_{J}|\partial_x u|^2w \,dx\right)^{1/2}.
$$ The time dependent Sobolev space $L^2\left(0,T; H_0^1(J,w)\right)$ is defined as the set of all measurable functions $ u$ such that
 \[ \|u\|_{L^2\left(0,T; H_0^1(J,w)\right)}  := \left(  \iint_{ J \times(0,T)} |u(x,t)|^2 w(x)\, dx \, dt \right)^{\frac{1}2}< \infty.\] 

In this setting, we introduce the notion of weak solution of the problem \eqref{specific_heat}.
\begin{defn}
We say that $u\in L^2(0,T;H^1(J,w))$ is a weak solution of \eqref{specific_heat} in $J \times (0,T)$ if, for every $\eta\in W,$ such that $\eta(x,0)=\eta(x,t)$ for any $x\in J$, we have that
\begin{equation*}\label{weak_solution}
\iint_{J\times (0,T)} w(x) \left(\frac{\partial u}{\partial x}  \frac{\partial {\eta}}{\partial x}  - u\frac{\partial{\eta}}{\partial t}\right)\, dx\, dt=0.
\end{equation*}
\end{defn}

We have the next proposition (see for the proof Theorem $2.1$ in \cite{CS}). 
\begin{prop}\label{Chiarenza_Serapioni}
Let $u$ be a positive solution in $(-R,R)\times (0,T)$ of (\ref{specific_heat}) and assume that condition \eqref{2-weight} holds, with constant $c_0$. Then there exists $\gamma=\gamma(c_0)>0$ such that
\begin{equation*}
\sup_{\left(-\frac{\rho}2,\frac{\rho}2\right)\times\left(t_0-\frac{3\rho^2}4 , t_0-\frac{\rho^2}{4}\right)} u\leq \gamma \inf_{\left(-\frac{\rho}2,\frac{\rho}2\right)\times\left( t_0+\frac{3\rho^2}4 , t_0+\rho^2\right)}u
\end{equation*}
% \label{paraharnack}
holds for $t_0\in (0,T)$ and any $\rho$ such that $0<\rho<R/2$ and $[t_0-\rho^2,t_0+\rho^2]\subset (0,T)$.
\end{prop}
\begin{remark}
The reader can easily imagine the general situation in any dimension as explicated in Theorem 2.1 in \cite{CS}, where the coefficient $a(x)$ in \eqref{genst1} is a matrix and the domains are cylinders.  
We have stated the Harnack inequality in $ (0,T).$ Nevertheless with a change of coordinates in space and time, we can always say that the Harnack inequality holds in any subset of $(R_1,R_2)\times(\tau_1,\tau_2),$ where $R_1,R_2,\tau_1,\tau_2\in \mathbb{R}.$ 
\end{remark}
%\subsection{Reflection of the solution}
We consider here that $\bold{D}^s \varphi (t)=0$ in an interval $J$. By taking the reflection of the solution of problem \eqref{marprob1}, we obtain a solution in a weak sense of \eqref{marprob1} across $x=0$.
 
It is useful to introduce a weak version of the limit $ \displaystyle \lim_{x\to 0^+} x^{1-2s} \partial_x U(x,t)$. In this sense, we have:
\begin{defn}
We say that in a weak sense
$$  \lim_{x\to 0^+} x^{1-2s} \frac{\partial{U}}{\partial x}(x,t) =0$$
if and only if, for any  $\eta\in W$ such that $\eta(x,0)=\eta(x,t)$ for any $x\in J$, we have that
\begin{equation}\label{weak_limit} \lim_{x\to 0^+} \int_0^T x^{1-2s} \frac{\partial U}{\partial x} \,\eta\,dt =0.\end{equation}
\end{defn}
\begin{lemma}\label{tildeu}
Let $U\colon \mathbb{R} \times [0,\infty)\to \mathbb{R}$ be a solution of the problem (\ref{marprob1}) such that, in a weak sense,
$  \displaystyle\lim_{x\to 0^+} x^{1-2s} \partial_x U(x,t)=0.$
Then the extension 
\begin{equation*}
\tilde U(x,t):=\left\{ 
\begin{split}
&U(x,t), && (x,t)\in [0,+\infty)\times (0,T)\\
&U(-x,t), & &(x,t)\in (-\infty,0)\times (0,T)
\end{split}
\right.
\end{equation*}
is a weak solution of
 \begin{equation}\label{weaksol}
\frac{\partial (|x|^{1-2s}U)}{\partial t} (x,t)= \frac{\partial }{\partial x}\left(|x|^{1-2s}\frac{\partial U}{\partial x}(x,t)\right)
\end{equation}
in $(-R,R)\times (0,T)$.
\end{lemma}
\begin{proof}
We claim that the extension $\tilde U$ is a weak solution of \eqref{weaksol}, hence that 
	\begin{equation}
	\label{bla22} 
	\int_{(-R,R)\times(0,T)} |x|^{1-2s}\left( \frac{\partial { \tilde U}}{\partial x} \frac{\partial { \eta}}{\partial x} - \tilde U \frac{\partial{\eta}}{\partial t}\right)\, dx\, dt =0. \end{equation}
	We compute, integrating by parts
	\begin{equation*}\begin{split}  \int_0^T  & \left(  \int_0^R x^{1-2s} \frac{\partial { \tilde U}}{\partial x} \, \frac{\partial { \eta}}{\partial x} \, dx\right)\,dt \\ &=   \int_0^T R^{1-2s} \frac{\partial { \tilde U}}{\partial x} (R,t) \, \eta(R,t)  \, dt - \lim_{x\to 0} \int_0^T x^{1-2s}\frac{\partial U}{\partial x} \,  \eta \, dt \\ 
	& - \int_0^T \left( \int_0^R \frac{\partial }{\partial x} \left( x^{1-2s} \frac{\partial { \tilde U}}{\partial x}  \right) \eta \, dx \right)\, dt \\
	&=  \int_0^T R^{1-2s} \frac{\partial { \tilde U}}{\partial x} (R,t)   \eta (R,t) \, dt  -\int_0^T\left( \int_0^R x^{1-2s}  \frac{\partial{\tilde U}}{\partial t} \, \eta\, dx\right)\, dt ,
	\end{split}
	\end{equation*}
	where we have used the weak limit in \eqref{weak_limit} and the fact that $\tilde U$ solves equation \eqref{weaksol}.
In the same way, one obtains that
 	\begin{equation*}\begin{split}  \int_0^T  & \left(  \int_{-R}^0 (-x)^{1-2s} \frac{\partial { \tilde U}}{\partial x} \, \frac{\partial { \eta}}{\partial x} \, dx\right)\,dt =  \int_0^T R^{1-2s} \frac{\partial { \tilde U}}{\partial x}  (-R,t)  \eta (-R,t)  \, dt  \\ 
	& -\int_0^T \left( \int_{-R}^0 (-x)^{1-2s}   \frac{\partial{\tilde U}}{\partial t} \, \eta \, dx\right)\, dt ,
	\end{split}
	\end{equation*}
 	therefore, by summing up, 
 	\begin{equation*}\begin{split} \int_{(-R,R)\times(0,T)} & |x|^{1-2s}\frac{\partial {\tilde U}}{\partial x}  \frac{\partial { \eta}}{\partial x} \, dx \, dt  \\ 
	&=  \int_0^T R^{1-2s} \left(\frac{\partial { \tilde U}}{\partial x} (R,t)  \eta(R,t)   - \frac{\partial { \tilde U}}{\partial x}(-R,t) \eta (-R,t)\right) \, dt \\ 
	& -\int_0^T \left(\int_{-R}^R |x|^{1-2s}  \frac{\partial{\tilde U}}{\partial t} \, \eta \, dx\right) \, dt.
\end{split}
\end{equation*}	
Hence
\begin{equation*}\begin{split}	\int_{(-R,R)\times(0,T)} &     |x|^{1-2s} \left( \frac{\partial {\tilde U}}{\partial x} \frac{\partial { \eta}}{\partial x} - \tilde U \frac{\partial{\eta}}{\partial t}\right)\, dx\, dt  \\ 
&=     \int_0^T R^{1-2s}\left( \frac{\partial { \tilde U} }{\partial x}(R,t) \eta(R,t)  - \frac{\partial { \tilde U}}{\partial x}(-R,t)  \eta (-R,t) \right) \, dt  \\ 
&   - \int_0^T\left( \int_{-R}^R |x|^{1-2s} \left(\frac{\partial{\tilde U}}{\partial t} \, \eta - \tilde U \frac{\partial{ \eta}}{\partial t}\right) \, dx\right)\, dt \\
	&=  \int_0^T R^{1-2s} \left( \frac{\partial { \tilde U}}{\partial x} (R,t) \eta(R,t)  - \frac{\partial { \tilde U}}{\partial x}(-R,t)  \eta (-R,t)  \right) \, dt \\ 
	&-   \int_{-R}^R|x|^{1-2s} 	\left( \tilde U (x,T)  \eta (x,T)  - \tilde U(x,0)\eta(x,0) \right) \, dx\\
	&= 0,
	\end{split}
	\end{equation*}
since $\eta(x,T)=\eta(x,0)=0$ and $\eta(R,t)=\eta(-R,t)=0$. This is the claim in \eqref{bla22}, and we conclude the proof of the Lemma.
\end{proof}

%\subsection{The Harnack inequality for Marchaud stationary functions} 

We show now that the Harnack inequality for Marchaud stationary functions can be deduced from the Harnack inequality associated with the extension operator.

The interested reader can also see \cite{CAFSIL} for the proof (using the extension operator) of the Harnack inequality for the fractional Laplacian, and \cite{DCKP14}  for the inequality for other types of nonlocal operators. In addition, we also point out \cite{FerFra} for the case of the fractional subelliptic operators in Carnot groups and \cite{stingatorrea2} for the fractional harmonic oscillator.
\begin{proof}[Proof of Theorem \ref{teoHarn}]
We consider $U$ to be the extension of $\varphi$, as introduced in Theorem \ref{teo:mainstat}. Since $\varphi$ is nonnegative, given the explicit solution $U$ in Theorem \ref{teo:sol}, the function $U$ is also positive. Now, we reflect $U$ and obtain $\tilde U$, as we have done in Lemma \ref{tildeu}. \\
We prove at first the theorem when $J=(0,T)$. Since $\bold{D}^s \varphi (t)=0$ in $(0,T)$, we have by definition that
	\[ \lim_{x\to 0^+} x^{-2s}  \frac{\partial {U}}{\partial x}(x,t)=0,\]
	and thanks to Lemma \ref{tildeu}, we obtain that $\tilde U$ is a weak solution of \eqref{weaksol} in, say, $(-R,R)\times(0,T)$ for a fixed arbitrary $R>0$. Moreover, the function $|x|^{1-2s}$ satisfies the condition \eqref{2-weight}, and according to Proposition \ref{Chiarenza_Serapioni}, we have that
	
	\begin{equation*}
\sup_{\left(-\frac{\rho}2,\frac{\rho}2\right)\times\left(t_0-\frac{3\rho^2}4 , t_0-\frac{\rho^2}{4}\right)} \tilde U\leq \gamma \inf_{\left(-\frac{\rho}2,\frac{\rho}2\right)\times\left( t_0+\frac{3\rho^2}4 , t_0+\rho^2\right)}\tilde U.
\end{equation*}
	%	\[ \sup_{\left(t_0-\frac{3\rho}4 , t_0-\frac{\rho}{4}\right)\times-\left(-\frac{\rho}2,\frac{\rho}2\right)} u\leq \gamma \inf_{\left(t_0+\frac{3\rho}4 , t_0+1\right)\times-\left(\frac{\rho}2,\frac{\rho}2\right)}u\] for any $\rho$ such that $0<\rho<R/2$ and $[t_0-\rho^2,t_0+\rho^2]\subset (0,T)$.
	It suffices now to slice the domain at $x=0$ to obtain that
		\[\sup_{\left(t_0-\frac{3\rho^2}4 , t_0-\frac{\rho^2}{4}\right)} U(0,t) \leq \gamma \inf_{\left(t_0+\frac{3\rho^2}4 , t_0+\rho^2\right)}U (0,t) ,\]
		hence
		\[ \sup_{\left(t_0-\frac{3\rho^2}4 , t_0-\frac{\rho^2}{4}\right)}\varphi(t)\leq \gamma \inf_{\left(t_0+\frac{3\rho^2}4 , t_0+\rho^2\right) }\varphi(t) \]
		for any $\rho$ such that $0<\rho<R/2$ and $[t_0-\rho^2,t_0+\rho^2]\subset (0,T)$.\\
			Now, in order to prove that the Harnack inequality holds on any interval $J \subset \mathbb{R}$, one considers a translation of $U$, namely for any $\theta\in \mathbb{R}$, the function $U_\theta (x,t):=U(x,t+\theta)$, and reflects it as in Lemma \ref{tildeu}. Then $\tilde U_\theta$ is a weak solution of \eqref{weaksol}, and $\tilde U_\theta(0,t)=\varphi(t+\theta)$.
		One obtains then, as a consequence of the Harnack inequality for the solution $U_\theta$, the following:
\[ \sup_{\left(t_0-\frac{3\rho^2}4 , t_0-\frac{\rho^2}{4}\right)}\varphi(t+\theta)\leq \gamma \inf_{\left(t_0+\frac{3\rho^2}4 , t_0+\rho^2\right) }\varphi(t+\theta) \]
for any $\rho$ such that $0<\rho<R/2$ and $[t_0-\rho^2,t_0+\rho^2]\subset (0,T)$. Therefore
	\[ \sup_{\left(t_0-\frac{3\rho^2}4 , t_0 -\frac{\rho^2}{4}\right)}\varphi(t)\leq \gamma \inf_{\left (t_0+\frac{3\rho^2}4 , t_0+\rho^2\right) }\varphi(t) \]
for any $\rho$ such that $0<\rho<R/2$ and $[t_0-\rho^2,t_0+\rho^2]\subset (\theta,T+\theta)$. As $\theta$ and $R$ are arbitrary, one concludes that
	\[ \sup_{\left(t_0-\frac{3\delta}4 , t_0 -\frac{\delta}{4}\right)}\varphi(t)\leq \gamma \inf_{\left (t_0+\frac{3\delta}4 , t_0+\delta\right) }\varphi(t) \]
	for any $\delta>0$ such that $[t_0-\delta, t_0+\delta]\subset J$.
		This concludes the proof of Theorem \ref{teoHarn}.
\end{proof}
\begin{remark}\label{Holder_regularity_ine}

We would like to point out that the Harnack type inequality obtained in Theorem \ref{teoHarn} can be equivalently stated as follows.
 Let us define for every $\delta>0$ and for every $\tau\in \mathbb{R}$ the sets:
\begin{equation*}\begin{split} &I(\tau,\delta)= [\tau-\frac{15}{8}\delta,\tau+\frac{1}{8}\delta],\\
&I^+(\tau,\delta)= [\tau-\frac{15}{8}\delta,\tau-\frac{7}{4}\delta],\\
&I^-(\tau,\delta)= [\tau-\frac{1}{8}\delta,\tau+\frac{1}{8}\delta].
\end{split}
\end{equation*}
With this notation, the Harnack inequality gives that for every  $I(\tau,\delta)\subset J$
$$
\sup_{I^+(\tau,\delta)} \varphi\leq \gamma \inf_{I^-(\tau,\delta)}\varphi.
$$
\end{remark}

\chapter{Some nonlocal nonlinear stationary equations}\label{S:NP:2}

\begin{abstract}We deal in this chapter with some nonlocal nonlinear stationary type problems. We first take a look at a problem connected to solitary solutions of nonlinear dispersive wave equations, in particular that arises in the study of the fractional Schr\"{o}dinger equation when looking for standing waves. More precisely, we discuss here the existence of a solution that concentrates at interior points of the domain, of  the probelm
\begin{equation*}
		\begin{cases}
		\varepsilon^{2s} \frlap u +u = u^p &\text{ in } \Omega \subset \Rn\\ 
		u=0 &\text{ in } \Rn \setminus \Omega, 
		\end{cases}
			\end{equation*} for $p\in (1, 2_s^{\star}-1)$, where 
		$ \cx={2n}/({n-2s}) $ is the critical fractional Sobolev exponent and $\varepsilon$ is a small parameter. Moreover, we prove the existence of a positive solution of the nonlinear and nonlocal elliptic equation in $\Rn$ 
\[ (-\Delta)^s u =\varepsilon h u^q+u^{\cx-1} \]
in the convex case $1\leq q<\cx-1$, where $\varepsilon$ is a small parameter and $h$ is a given bounded, integrable function.  		
	The problem has a variational structure and we prove the existence of a solution using the classical Mountain-Pass Theorem. We work here with the harmonic extension of the fractional Laplacian, which allows us to deal with a weighted (but possibly degenerate) local operator, rather than with a nonlocal energy. 
	\end{abstract}

\bigskip 
\bigskip
	
In this chapter, we study some nonlocal nonlinear problems of stationary type. 	
Let $s \in (0,1)$ be the fractional parameter, $n>2s$ be the dimension of the reference space, and $\varepsilon>0$ be a small parameter. We consider the so-called fractional Sobolev exponent defined for $n>2s$ as  
		\[ 2_s^{\star}:= 	\frac{2n}{n-2s}.\]
				
\section{A  nonlocal nonlinear stationary Schr\"{o}dinger type equation}
The type of problems introduced in this section are connected to solitary solutions of nonlinear dispersive wave equations (such as the Benjamin-Ono equation, the Benjamin-Bona-Mahony equation and the fractional Schr\"{o}dinger equation). In this section, only stationary equations are studied and we redirect the reader to \cite{vazquez_nonlinear, vazquez_recent} for the study of evolutionary type equations. 

We discuss the following nonlocal nonlinear Schr\"{o}dinger equation 
	\begin{equation}
		\begin{cases}
		\varepsilon^{2s} \frlap u +u = u^p &\text{ in } \Omega \subset \Rn\\ 
		u=0 &\text{ in } \Rn \setminus \Omega, 
		\end{cases}
		\label{sch}
	\end{equation}
in the subcritical case $p\in (1, 2_s^{\star}-1)$, namely when $p\in \displaystyle \bigg(1, \frac{n+2s}{n-2s}\bigg)$.
We study the existence of a solution that concentrates at interior points of the domain, points that depend on the global geometry of the domain. Moreover, we point out a simple consequence of
the Uncertainty Principle, which can be seen as
a fractional Sobolev inequality
in weighted spaces.

This equation \ref{sch} arises in the study of the 
fractional Schr\"{o}dinger equation when looking for standing waves. 
Namely, the fractional Schr\"{o}dinger equation
considers solutions~$\Psi=\Psi(x,t):\R^n\times\R\to {\mathbb{C}}$ of
\begin{equation}\label{PLANCK}
i\hslash\partial_t \Psi = \big( \hslash^{2s} (-\Delta)^s +V\big)\Psi,\end{equation}
where~$s\in(0,1)$, $\hslash$ is the reduced
Planck constant and~$V=V(x,t,|\Psi|)$ is a potential.
This equation
is of interest in quantum mechanics
(see e.g.~\cite{L00} and the appendix in~\cite{DPDV14}
for details and physical motivations). 
Roughly speaking, the quantity~$|\Psi(x,t)|^2\,dx$
represents the probability density of finding a quantum particle
in the space region~$dx$ and at time~$t$.

\noindent The simplest solutions of~\eqref{PLANCK} are the ones
for which this probability density is independent of time,
i.e.~$|\Psi(x,t)|=u(x)$ for some~$u:\R^n\to[0,+\infty)$.
In this way, one can write $\Psi$ as~$u$ times
a phase that oscillates (very rapidly) in time: that is
one may look for solutions of~\eqref{PLANCK} of the form
$$ \Psi(x,t) := u(x)\, e^{i\omega t/\hslash},$$
for some frequency~$\omega\in\R$.
Choosing~$V=V(|\Psi|)=-|\Psi|^{p-1}=-u^{p-1}$,
a substitution into~\eqref{PLANCK} gives that
$$
\Big( \hslash^{2s} (-\Delta)^s u +\omega u - u^p \Big)\, e^{i\omega t/\hslash}
=\hslash^{2s} (-\Delta)^s \Psi -
i\hslash\partial_t \Psi +V\Psi=0,$$
which is~\eqref{sch}
(with the normalization convention~$\omega:=1$ and~$\varepsilon:=\hslash$).

The goal of this section is to construct solutions of problem \eqref{sch} that concentrate at interior points of the domain $\Omega$ for sufficiently small values of $\varepsilon$. We perform
a blow-up of the domain, defined as
	\[\Omega_\varepsilon:=\displaystyle \frac{1}{\varepsilon} \Omega=\bigg\{  \frac{x}{\varepsilon}, x\in \Omega\bigg\}.\]
We can also rescale the solution of \eqref{sch} on $\Omega_\varepsilon$, 
	\[u_\varepsilon(x)=u(\varepsilon x).\]
The problem \eqref{sch} for $u_\varepsilon$ then reads
	\begin{equation}
		\begin{cases}
		 \frlap u +u = u^p &\text{ in } \Omega_\varepsilon\\ 
		u=0 &\text{ in } \Rn \setminus \Omega_\varepsilon.
		\end{cases}
		\label{dsch}
	\end{equation} 
When~$\varepsilon\to0$, the domain~$\Omega_\varepsilon$ invades the whole of the space.
Therefore, it is also natural to 
consider (as a first approximation)
the equation on the entire space  
	\begin{equation}
	 	\frlap u +u=u^p \text{ in } \Rn.
		\label{entsch}
	\end{equation}
The first result that
we need  is that there exists an entire positive radial least energy solution $w \in H^s(\Rn)$ of \eqref{entsch}, called the \emph{ground state solution}. 
Here follow some relevant results on this. The interested reader can find their proofs in \cite{FLS13}.
\begin{enumerate}
	\item The ground state solution $w\in H^s(\Rn)$ is unique up to translations. 
	\item The ground state solution $w\in H^s(\Rn)$ is nondegenerate, i.e., the derivatives $D_i w$ are solutions to the linearized equation
		\begin{equation}
		 	\frlap Z +Z= pZ^{p-1}.
			\label{linsch}
		\end{equation} 
	\item The ground state solution $w\in H^s(\Rn)$ decays polynomially at infinity, namely there exist two constants $\alpha, \beta >0$  such that 
		\[ \alpha |x|^{-(n+2s)} \leq u(x) \leq \beta |x|^{-(n+2s)}.\] 
\end{enumerate}
Unlike the fractional case, we remark that for the (classical) Laplacian, at infinity  the ground state solution decays exponentially fast. We also refer to~\cite{FL1D} for the one-dimensional case.

The main theorem of this section establishes the existence of a solution that concentrates at interior points of the domain for sufficiently small values of $\varepsilon$. This concentration phenomena is written
in terms of the ground state solution~$w$.
Namely, the first approximation for the solution
is exactly the ground state~$w$,
scaled and concentrated at an appropriate point~$\xi$ of
the domain. More precisely, we have:

\begin{theorem}\label{THSP}
If $\varepsilon$ is sufficiently small, there exist a point $\xi \in \Omega$ and a solution $U_\varepsilon$ of the problem \eqref{sch} such that
			\[ \bigg| U_\varepsilon (x)- w\Big(\frac{x-\xi}{\varepsilon}\Big)\bigg| \leq C \varepsilon^{n+2s},\]
and $\text{dist}(\xi, \partial \Omega)\geq \delta>0$. Here, $C$ and $\delta$ are constants independent of $\varepsilon$ or $\Omega$, and the
function $w$ is the ground state solution of problem \eqref{entsch}.
\label{nnscheq}
\end{theorem}

The concentration point~$\xi$ in Theorem~\ref{THSP}
is influenced by the global geometry of the domain.
On the one hand, when~$s=1$, the point~$\xi$ is the one that maximizes
the distance from the boundary. On the other hand,
when~$s\in(0,1)$, such simple characterization of~$\xi$
does not hold anymore: in this case, $\xi$
turns out to be asymptotically the maximum of a (complicated, but rather explicit)
nonlocal functional: see \cite{DPDV14} for more details.

We state here the basic idea of the proof of Theorem~\ref{THSP}
(we refer again to \cite{DPDV14} for more details).
\begin{proof}[Sketch of the proof of Theorem \ref{nnscheq}] In this proof, we make use of the Lyapunov-Schmidt procedure. Namely, rather than looking for the
solution in an infinite-dimensional functional space,
one decomposes the problem into two orthogonal subproblems.
One of these problem is still infinite-dimensional,
but it has the advantage to bifurcate from a known object
(in this case, a translation of the ground state).
Solving this auxiliary subproblem does not provide a true
solution of the original problem, since a leftover
in the orthogonal direction may remain. To kill this
remainder term, one solves a second subproblem,
which turns out to be finite-dimensional
(in our case, this subproblem is set in~$\R^n$,
which corresponds to the action of the translations
on the ground state).

A structural advantage of the problem considered
lies in its variational structure.
Indeed, equation~\eqref{dsch} is 
the Euler-Lagrange equation of
the energy functional
	\eqlab{ \label{peren} 
I_\varepsilon(u) =\frac{1}{2} \int_{\Omega_\varepsilon} \Big( \frlap u(x)+u(x) \Big) u(x) \, dx - \frac{1}{p+1} \int_{\Omega_\varepsilon} u^{p+1} (x) \, dx }
for any $u \in H^s_0(\Omega_\varepsilon) := \{ u\in H^s(\R^n) \text { s.t. } u=0 \text{ a.e. in } \Rn \setminus \Omega_\varepsilon\}$.
Therefore, the problem reduces to finding critical points of~$I_\varepsilon$.

To this goal,
we consider the \emph{ground state} solution $w$ and  for any $\xi \in \Rn$ we let $w_\xi:=w(x-\xi)$. For a given $\xi \in \Omega_\varepsilon$ a first approximation $\bar{u}_\xi$ for the solution of problem \eqref{dsch} can be taken as the solution of the linear problem 
	\begin{equation}
		\begin{cases}
		 \frlap \overline u_\xi +\overline u_\xi = w_\xi^p &\text{ in } \Omega_\varepsilon,\\ 
		\overline u_\xi=0 &\text{ in } \Rn \setminus \Omega_\varepsilon.
		\end{cases}
		\label{dlsch}
	\end{equation}
The actual solution will be obtained as a small perturbation of $\bar u_\xi$ for a suitable point $\xi = \xi (\varepsilon)$.
We define the operator $\mathcal{L}:=\frlap +I$, where~$I$ is the identity and we notice that $\mathcal{L}$ has a unique fundamental solution that solves
	\[\mathcal{L}\Gamma = \delta_0 \quad \text{ in } \Rn.\]
The Green function $G_\varepsilon$ of the operator $\mathcal{L}$ in $\Omega_\varepsilon$ satisfies 	
	\begin{equation}
		\begin{cases}
		\mathcal{L} G_\varepsilon(x,y) = \delta_y(x)  &\text{ if } x \in \Omega_\varepsilon,\\ 
		 G_\varepsilon (x,y)=0 &\text{ if } x \in \Rn \setminus \Omega_\varepsilon. 
		\end{cases}
		\label{gsch}
	\end{equation}
It is convenient to introduce 
the regular part of $G_\varepsilon$, which is often called
the Robin function. This function is defined by
	\begin{equation}\label{7.8bis}
H_\varepsilon (x,y):= \Gamma(x-y) - G_\varepsilon(x,y) \end{equation}
and it satisfies, for a fixed $y \in \Rn$,
	\begin{equation*}
		\begin{cases}
		\mathcal{L} H_\varepsilon(x,y) = 0 &\text{ if } x \in \Omega_\varepsilon,\\ 
		H_\varepsilon (x,y)=\Gamma(x-y) &\text{ if } x \in \Rn \setminus \Omega_\varepsilon. 
		\label{rsch}
		\end{cases}
	\end{equation*}
We have that 
	\[\overline u_\xi (x)= \int_{\Omega_\varepsilon} \overline u_\xi (y)\delta_0(x-y) \, dy,\]
and by \eqref{gsch} that
	\[\overline u_\xi (x)=  \int_{\Omega_\varepsilon} \overline u_\xi (y)\mathcal{L} G_\varepsilon(x,y) \, dy.\]
The operator $\mathcal{L}$ is self-adjoint and thanks to the above identity and to equation \eqref{dlsch} it follows that 
	\[ 	 \overline u_\xi (x)=  \int_{\Omega_\varepsilon} \mathcal{L} \overline u_\xi (y) G_\varepsilon(x,y) \, dy
					=  \int_{\Omega_\varepsilon}
w_\xi^p(y) G_\varepsilon(x,y)\, dy.
		\]
So, we use~\eqref{7.8bis} and we obtain that
$$ \overline u_\xi (x)=
\int_{\Omega_\varepsilon} w_\xi^p(y) \Gamma(x-y) \, dy-\int_{\Omega_\varepsilon} w_\xi^p(y) H_\varepsilon(x,y)\, dy.$$
Now we notice that, since $w_\xi$ is solution of \eqref{entsch} and $\Gamma$ is the fundamental solution of $\mathcal{L}$, we have that
		\[ 
			\int_{\Rn} w_\xi^p(y)\Gamma(x-y) \, dy =\int_{\Rn} \mathcal{L}w_\xi(y) \Gamma(x-y) \, dy
											= \int_{\Rn} w_\xi (y)  \mathcal{L}\Gamma(x-y) \, dy
											=w_\xi (x).
		\]
Therefore we have obtained that 
 		\begin{equation}
	 \overline u_\xi (x) = w_\xi (x) -\int_{\Rn \setminus \Omega_\varepsilon} w_\xi^p(y) \Gamma(x-y) \, dy- \int_{\Omega_\varepsilon} w_\xi^p(y) H_\varepsilon(x,y)\, dy.
		\label{uexp}
	\end{equation}
Now we can insert~\eqref{uexp} into the energy functional~\eqref{peren}
and expand the errors
in powers of~$\varepsilon$.
For $\text{dist}(\xi,\partial \Omega_\varepsilon)\geq \displaystyle \frac{\delta}{\varepsilon}$ with $\delta$ fixed and small, the energy of $\overline u_\xi$ is a perturbation of the energy of the ground state~$w$
and one finds (see Theorem 4.1 in~\cite{DPDV14})
that
	\begin{equation}
		 I_\varepsilon(\overline u_\xi) = I(w) + \frac{1}{2} \mathcal{H}_\varepsilon(\xi) + \mathcal{O}(\varepsilon^{n+4s}),
	\label{en}
	\end{equation}
where 
	\[\mathcal{H}_\varepsilon (\xi):= \int_{\Omega_\varepsilon}\int_{\Omega_\varepsilon} H_\varepsilon(x,y) w_\xi^p(x)w_\xi^p(y) \, dx \, dy\]
and $I$ is the energy computed on the whole space $\Rn$, namely
	\begin{equation*}
	I(u) =\frac{1}{2} \int_{\Rn} \Big( \frlap u(x)+u(x) \Big) u(x) \, dx - \frac{1}{p+1} \int_{\Rn} u^{p+1} (x) \, dx.
	\label{enfuncrn}
	\end{equation*}
In particular, $I_\varepsilon(\overline u_\xi)$
agrees with a constant (the term~$I(w)$), plus a functional over a finite-dimensional space
(the term~$\mathcal{H}_\varepsilon (\xi)$, which only depends on~$\xi\in\R^n$),
plus a small error.

We remark that the solution~$\overline u_\xi$
of equation~\eqref{dlsch} which can be obtained from~\eqref{uexp}
does not provide a solution for the original problem~\eqref{dsch}
(indeed, it only solves~\eqref{dlsch}):
for this,
we look for solutions $u_\xi$ of \eqref{dsch} as perturbations of~$\overline u_\xi$,
in the form
	\begin{equation}\label{POj}
u_\xi :=\overline u_\xi +\psi.\end{equation}
The perturbation functions $\psi$ are considered among those vanishing outside $\Omega_\varepsilon$ and orthogonal to the space $\mathcal{Z}=\text{Span}(Z_1,\dots,Z_n)$, where $Z_i=\displaystyle \frac{\partial w_\xi}{\partial x_i}$ are solutions of the linearized equation \eqref{linsch}.
This procedure somehow ``removes the degeneracy'',
namely we look for the corrector~$\psi$ in
a set where the linearized operator is invertible.
This makes it
possible, fixed any~$\xi\in\R^n$, 
to find~$\psi=\psi_\xi$ such that the function~$u_\xi$,
as defined in~\eqref{POj}
solves the equation 
	\eqlab{ \label{qsci} \frlap u_\xi +u_\xi=u_\xi^p +\sum_{i=1}^n c_i Z_i \quad  \text{ in }\Omega_\varepsilon.}
That is, $u_\xi$ is solution of
the original equation~\eqref{dsch}, up to an error
that lies in the tangent space of the translations
(this error is exactly the price that we pay
in order to solve the corrector equation for~$\psi$
on the orthogonal of the kernel, where the operator is nondegenerate).
As a matter of fact (see Theorem~7.6 in~\cite{DPDV14} for details)
one can see that the corrector~$\psi=\psi_\xi$ is of order~$\varepsilon^{n+2s}$.
Therefore, one can compute~$I_\varepsilon (u_\xi)
=I_\varepsilon (\overline u_\xi+\psi_\xi)$ as a higher order
perturbation of~$I_\varepsilon (\overline u_\xi)$.
{F}rom~\eqref{en}, one obtains that
\begin{equation}\label{9HU88I}
I_\varepsilon (u_\xi)=
I(w) +\frac{1}{2}\mathcal{H}_\varepsilon(\xi)+\mathcal{O}(\varepsilon^{n+4s}),\end{equation}
see Theorem~7.17 in~\cite{DPDV14} for details.

Since this energy expansion now depends only on~$\xi\in\R^n$,
it is convenient
to define the operator $J_\varepsilon \colon \Omega_\varepsilon \to \R$ as
	\[ J_\varepsilon(\xi):=I_\varepsilon (u_\xi).\]
This functional is often
called the reduced energy functional.
{F}rom~\eqref{9HU88I}, we conclude that
	\begin{equation}\label{JK L}
J_\varepsilon(\xi)=
I(w) +\frac{1}{2}\mathcal{H}_\varepsilon(\xi)+\mathcal{O}(\varepsilon^{n+4s}).\end{equation}
The reduced energy~$J$ plays an important role in this framework
since critical points of~$J$ correspond to true solutions
of the original equation~\eqref{dsch}. More precisely
(see Lemma 7.16 in~\cite{DPDV14}) one has that
$c_i=0$ for all $i=1, \dots, n$ 
in~\eqref{qsci}
if and only if 
	\begin{equation}\label{NEAT}
\frac{\partial J_\varepsilon}{\partial \xi}(\xi)=0.\end{equation}
In other words, when $\varepsilon$ approaches $0$, to find
concentration points, it is enough to find critical points
of~$J$,
which is a finite-dimensional problem.
Also, critical points for~$J$ come from critical points of~$\mathcal{H}_\varepsilon$,
up to higher orders, thanks to~\eqref{JK L}.
The issue is thus to prove that~$\mathcal{H}_\varepsilon$
does possess critical points
and that these critical points survive after
the small error of
size~$\varepsilon^{n+4s}$: in fact, we show that~$\mathcal{H}_\varepsilon$
possesses a minimum,
which is stable for perturbations.
For this, one needs
a bound for the Robin function $H_\varepsilon$ from above and below.
To this goal,
one builds a barrier function $\beta_\xi$ defined for $\xi \in \Omega_\varepsilon$  and $x \in \Rn$ as
	\[ \beta_\xi (x) := \int_{\Rn \setminus \Omega_\varepsilon} \Gamma(z-\xi)\Gamma(x-z) \, dz.\]
Using this function in combination with suitable maximum principles,
one obtains the existence of a constant $c\in (0,1)$ such that
	\[ c H_\varepsilon (x,\xi) \leq \beta_\xi(x) \leq c^{-1}H_\varepsilon(x,\xi),\]
 for any $x\in \Rn$ and any $\xi \in \Omega_\varepsilon$ with $\text{dist} (\xi, \partial \Omega_\varepsilon)>1$, see Lemma~2.1
in~\cite{DPDV14}. {F}rom this it follows that  
	\begin{equation}\label{ATT}
\mathcal{H}_\varepsilon (\xi) \simeq d^{-(n+4s)},\end{equation}
for all points $\xi \in \Omega_\varepsilon$ such that $d \in [5, \delta / \varepsilon]$.
So, one considers the domain~$\Omega_{\varepsilon,\delta}$ of the points
of~$\Omega_\varepsilon$ that lie at distance more than~$\delta/\varepsilon$
from the boundary of~$\Omega_\varepsilon$. By~\eqref{ATT},
we have that
  \begin{equation}\label{ATT-1}
\mathcal{H}_\varepsilon (\xi) \simeq \frac{\varepsilon^{n+4s}}{\delta^{n+4s}}
\;{\mbox{ for any }}\;\xi\in \partial \Omega_{\varepsilon,\delta}.
\end{equation}
Also, up to a translation, we may suppose that~$0\in\Omega$.
Thus, $0\in\Omega_{\varepsilon}$ and its distance from~$\partial\Omega_\varepsilon$
is of order~$1/\varepsilon$ (independently of~$\delta$).
In particular, if~$\delta$ is small enough, we have that~$0$
lies in the interior of~$ \Omega_{\varepsilon,\delta}$,
and~\eqref{ATT} gives that
$$ \mathcal{H}_\varepsilon (0) \simeq \varepsilon^{n+4s}.$$
By comparing this with~\eqref{ATT-1}, we see that~$\mathcal{H}_\varepsilon $
has an interior minimum in~$ \Omega_{\varepsilon,\delta}$.
The value attained at this minimum
is of order~$\varepsilon^{n+4s}$,
and the values attained at the boundary of~$\Omega_{\varepsilon,\delta}$
are of order~$\delta^{-n-4s}\varepsilon^{n+4s}$,
which is much larger than~$\varepsilon^{n+4s}$, if~$\delta$ is small enough.
This says that the interior minimum of~$\mathcal{H}_\varepsilon $
in~$ \Omega_{\varepsilon,\delta}$ is nondegenerate and it survives to
any perturbation of order~$\varepsilon^{n+4s}$, if~$\delta$ is small enough.

This and~\eqref{JK L} imply that~$J$ has also an interior 
minimum at some point~$\xi$ in~$ \Omega_{\varepsilon,\delta}$. By 
construction, this point~$\xi$ satisfies~\eqref{NEAT},
and so this completes the proof of
Theorem~\ref{THSP}.
\end{proof}

The variational argument in the proof above (see in particular~\eqref{NEAT})
has a classical and neat geometric interpretation. Namely, the ``unperturbed'' functional (i.e. the one with~$\varepsilon=0$) has a very degenerate geometry, since it has a whole manifold of minimizers with the same energy: this manifold corresponds to the translation of the ground state~$w$, namely it is of the form~$M_0:=\{ w_\xi,\;\xi\in\R^n\}$ and, therefore, it can be identified with~$\R^n$.
\begin{center}
\begin{figure}[htb]
\begin{minipage}[b]{0.65\linewidth}
	\centering
	\includegraphics[width=0.65\textwidth]{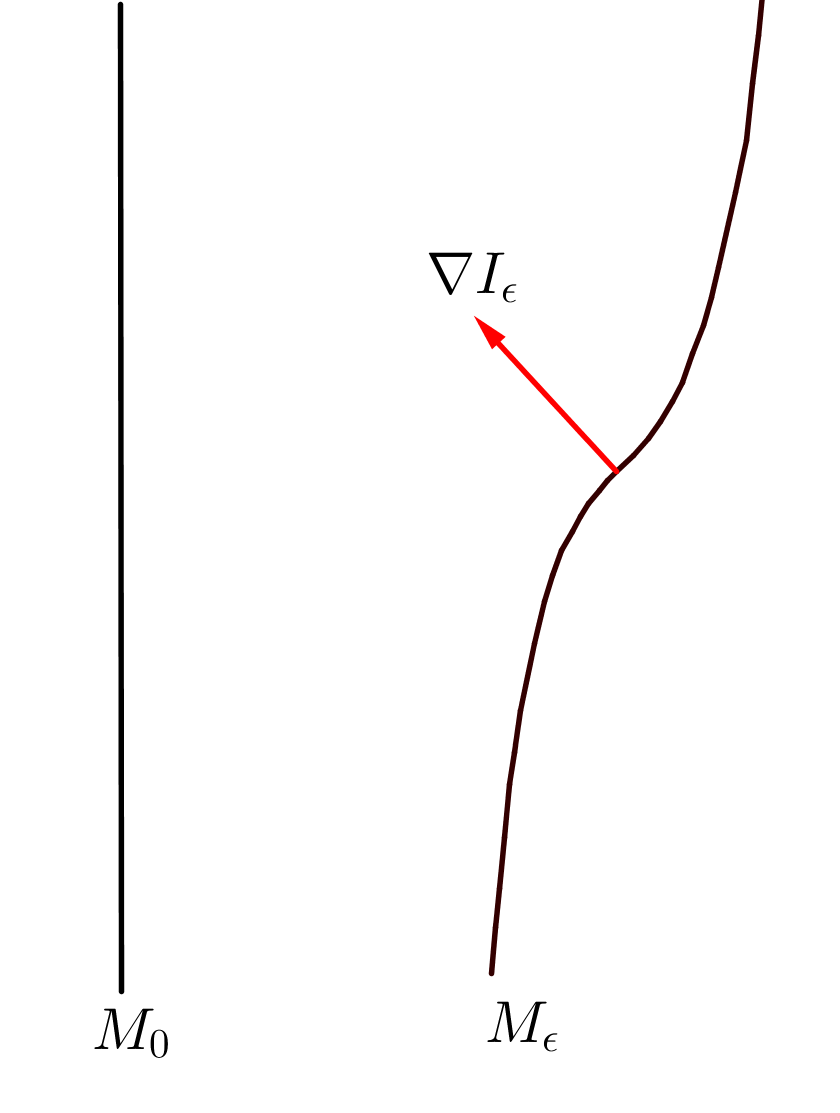}
	\caption{Geometric interpretation}   
	\label{fign:seq}
	\end{minipage}
\end{figure}
\end{center}
For topological arguments, this degenerate picture may constitute a serious obstacle to the existence of critical points for the ``perturbed'' functional (i.e. the one with~$\varepsilon\ne0$). As an obvious example, the reader may think of the function of two variables~$f_\varepsilon:\R^2\to\R$ given by~$f_\varepsilon(x,y):=x^2+\varepsilon y$.
When~$\varepsilon=0$, this function attains its minimum along the manifold~$\{x=0\}$, but all the critical points on this manifold are ``destroyed'' by the
perturbation when~$\varepsilon\ne0$ (indeed~$\nabla f_\varepsilon(x,y) =(2x,\varepsilon)$ never vanishes).

In the situation described in the proof of Theorem~\ref{nnscheq}, this pathology does not occur, thanks to the nondegeneracy provided in~\cite{FLS13}.
Indeed, by the nondegeneracy of the unperturbed critical manifold, when~$\varepsilon\ne0$ one can construct a manifold, diffeomorphic to the original one 
(in our case of the form~$M_\varepsilon:=\{ \overline u_\xi+\psi(\xi),\;\xi\in\R^n\}$), that enjoys the special feature of ``almost annihilating'' the gradient of the functional, up to vectors parallel to the original manifold~$M_0$ (this is the meaning of formula \eqref{qsci}).

Then, if one finds a minimum of the functional constrained to~$M_\varepsilon$, the theory of Lagrange multipliers (at least in finite dimension) would suggest that the gradient is normal to~$M_\varepsilon$. That is, the gradient of the functional is, simultaneously, parallel to~$M_0$ and orthogonal to~$M_\varepsilon$. But since~$M_\varepsilon$ is diffeomorphically
close to~$M_0$, the only vector with this property is the null vector, hence this argument provides the desired critical point.
\bigskip

We also recall that the fractional Schr\"odinger equation
is related to a nonlocal 
canonical quantization, which in turn produces
a nonlocal Uncertainty Principle.
In the classical setting, one considers the momentum/position
operators, which are defined in~$\R^n$, by
\begin{equation}\label{CQP}
P_k:=-i\hslash\partial_k \ {\mbox{
and }} \
Q_k:=x_k\end{equation}
for~$k\in\{1,\dots,n\}$. Then, the
Uncertainty Principle states that the operators~$P=(P_1,\dots,P_n)$
and~$Q=(Q_1,\dots,Q_n)$ do not commute (which makes
it practically impossible to measure simultaneously both momentum
and position).
Indeed, in this case a simple computation shows that
\begin{equation}\label{PLj89aaLL}
[Q,P]:=\sum_{k=1}^n [Q_k,P_k] = i\hslash n.\end{equation}
The nonlocal analogue of this quantization may be formulated
by introducing a nonlocal momentum, i.e. by replacing the operators
in~\eqref{CQP} by
\begin{equation}\label{CQ}
P_k^s:=-i\hslash^s\partial_k (-\Delta)^{\frac{s-1}2} \ {\mbox{
and }} \
Q_k:=x_k.\end{equation}
In this case, 
%using that the Fourier transform 
%of the product is the 
%convolution of the
%Fourier transforms,
one has that
\eqlab{\label{723}
%(\widehat x_k * g)(\xi)\,&= 
 {\mathcal{F}}\big(x_k {\mathcal{F}}^{-1} g(x)\big) (\xi)
&= \int_{\R^n} \,dx \int_{\R^n} \,dy \; e^{2\pi ix\cdot (y-\xi)} x_k g(y)
\\
&= \frac1{2\pi i}
\,\int_{\R^n} \,dx \int_{\R^n} \,dy \;
\partial_{y_k} e^{2\pi ix\cdot (y-\xi)} g(y)
= \frac{i}{2\pi}\int_{\R^n} \,dx \int_{\R^n} \,dy \;
e^{2\pi ix\cdot (y-\xi)} \partial_k g(y)\\
&= \frac{i}{2\pi}
\int_{\R^n} \,dx\;
e^{-2\pi ix\cdot \xi} {\mathcal{F}}^{-1}(\partial_k g)(x)
= \frac{i}{2\pi} {\mathcal{F}}\big({\mathcal{F}}^{-1}(\partial_k g)\big)(\xi)
= \frac{i}{2\pi} \partial_k g(\xi),
}
for any test function~$g$.
In addition,
$${\mathcal{F}}(P_k^s f)= (2\pi)^{s}\hslash^s
\xi_k\,|\xi|^{s-1}\widehat f.$$
Therefore, given any test function~$\psi$, using this with $f:=\psi$
and~$f:=x_k\psi$,
and also~\eqref{723} with $g:={\mathcal{F}}(P_k^s \psi)$
and~$g:=\widehat\psi$,
we obtain that
\bgs{
& {\mathcal{F}} \big(Q_k P_k^s \psi(x)-P_k^s Q_k\psi(x)\big)
= {\mathcal{F}} \big(x_k P_k^s \psi(x)\big) -{\mathcal{F}} \big (P_k^s (x_k\psi(x)) \big)\\
%=& \widehat x_k * {\mathcal{F}}(P_k^s \psi(x))
%-{\mathcal{F}} \big (P_k^s (x_k\psi(x)) \big)
=& \frac{i}{2\pi} \partial_k {\mathcal{F}}( P_k^s \psi)(\xi)
- (2\pi)^{s}\hslash^s
\xi_k\,|\xi|^{{s-1}} {\mathcal{F}}(x_k\psi(x))(\xi)
\\ 
=& (2\pi)^{s-1} i \hslash^s
\partial_k \big( \xi_k\,|\xi|^{{s-1}}\widehat \psi(\xi)\big)
- (2\pi)^{s}\hslash^s
\xi_k\,|\xi|^{{s-1}} \widehat x_k* \widehat\psi(\xi)\\
= &(2\pi)^{s-1} i \hslash^s
\partial_k \big( \xi_k\,|\xi|^{{s-1}}\widehat \psi(\xi)\big)
- (2\pi)^{s-1}i\hslash^s
\xi_k\,|\xi|^{{s-1}}\partial_k\widehat \psi(\xi)
\\ =& (2\pi)^{s-1} i \hslash^s 
\partial_k \big( \xi_k\,|\xi|^{{s-1}}\big)\,\widehat \psi(\xi)
=
(2\pi)^{s-1} i \hslash^s 
\left( |\xi|^{{s-1}} +(s-1)
\xi_k^2\,|\xi|^{s-3}\right)\,\widehat \psi(\xi)
.}
Consequently, by summing up,
\[ \mathcal F \big([Q,P^s]\psi)=(2\pi )^{s-1} i \hslash^s\,|\xi|^{s-1}\,\left(n +s-1\right)\, \widehat \psi (\xi).\] 
So, by taking the anti-transform,
\eqlab{ \label{s-COMM-QP}
 [Q,P^s]\psi=
i \hslash^s\,\left(
n +{s-1}
\right)\,
{\mathcal{F}}^{-1}
\big( (2\pi |\xi|)^{{s-1}} \widehat\psi\big)= 
i\hslash^s 
\,(n+s-1)\,
(-\Delta)^{\frac{s-1}2}\psi.}
Notice that, as $s\rightarrow1$, this formula reduces to the
the classical Heisenberg Uncertainty Principle in~\eqref{PLj89aaLL}.

%\subsection{From the nonlocal Uncertainty Principle
%to a fractional weighted inequality}

\bigskip

Now we point out a simple consequence of
the Uncertainty Principle
in formula~\eqref{s-COMM-QP}, which can be seen as
a fractional Sobolev inequality
in weighted spaces. The result
(which boils down to known formulas as~$s\to1$)
is the following:

\begin{prop}
For any~$u\in \mathcal{S}(\R^n)$, we have that
\begin{equation*}
\Big\| \, (-\Delta)^{\frac{s-1}{4}} u \, \Big\|_{L^2(\R^n)}^2
\le \frac{2}{n+s-1} \,
\Big\|\, |x|u \, \Big\|_{L^2(\R^n)}
\,\Big\| \,\nabla (-\Delta)^{\frac{s-1}{2}} u\,
\Big\|_{L^2(\R^n)}.\end{equation*}
\end{prop}

\begin{proof}
The proof is a general argument in operator theory.
Indeed, suppose that there are two operators~$S$ and~$A$,
acting on a space with a scalar
Hermitian product. Assume that~$S$ is self-adjoint
and~$A$ is anti-self-adjoint, i.e.
$$ \langle u,Su\rangle=\langle Su,u\rangle
\;{\mbox{ and }}\;
\langle u,Au\rangle=-\langle Au,u\rangle,$$
for any~$u$ in the space.
Then, for any~$\lambda\in\R$,
\begin{eqnarray*}
\| (A+\lambda S)u\|^2&=&
\|Au\|^2 + \lambda^2 \|Su\|^2 + \lambda\Big(
\langle Au,Su\rangle +\langle Su,Au\rangle\Big)\\
&=&
\|Au\|^2 + \lambda^2 \|Su\|^2 + \lambda
\langle (SA-AS)u,u\rangle .
\end{eqnarray*}
Now we apply this identity in the space~$C^\infty_0(\R^n)\subset L^2(\R^n)$,
taking~$S:=Q_k=x_k$ and~$A:=i P_k^s=\hslash^s\partial_k (-\Delta)^{\frac{s-1}2}$
(recall~\eqref{CQ}
and notice that~$iP_k^s$ is anti-self-adjoint,
thanks to the integration by parts formula).
In this way, and using~\eqref{s-COMM-QP}, we obtain that
\begin{eqnarray*}
0&\le&
\sum_{k=1}^n
\| (iP_k^s+\lambda Q_k)u\|^2_{L^2(\R^n)}\\
&=&
\sum_{k=1}^n \left[
\|iP_k^s u\|^2_{L^2(\R^n)}
+ \lambda^2 \|Q_k u\|^2_{L^2(\R^n)} 
+ i\lambda \langle [Q_k ,P_k^s] u,u\rangle_{L^2(\R^n)}\right]
\\ &=&
\hslash^{2s} \Big\| \,\nabla (-\Delta)^{\frac{s-1}{2}} u\,\Big\|_{L^2(\R^n)}^2
+\lambda^2 \Big\|\, |x|u \, \Big\|_{L^2(\R^n)}^2 + i^2 \,\lambda\,(n + s - 1)\,\hslash^{s}
\langle (-\Delta)^{\frac{s-1}{2}} u,u\rangle_{L^2(\R^n)}
\\ &=&
\hslash^{2s} \Big\| \,\nabla (-\Delta)^{\frac{s-1}{2}} u\,\Big\|_{L^2(\R^n)}^2
+\lambda^2 \Big\|\, |x|u \, \Big\|_{L^2(\R^n)}^2-\lambda\,(n + s -1)\,\hslash^{s}
\Big\| \, (-\Delta)^{\frac{s-1}{4}} u \, \Big\|_{L^2(\R^n)}^2.
\end{eqnarray*}
Now, if~$u\not\equiv0$,
we can optimize this identity by choosing
$$\lambda:=\frac{ 
(n + s -1)\,\hslash^{s}
\Big\| \, (-\Delta)^{\frac{s-1}{4}} u \, \Big\|_{L^2(\R^n)}^2}{2\,
\Big\|\, |x|u \, \Big\|_{L^2(\R^n)}^2 }$$
and we obtain that
$$ 0\le
\hslash^{2s} \Big\| \,\nabla (-\Delta)^{\frac{s-1}{2}} u\,\Big\|_{L^2(\R^n)}^2
- \frac{(n + s -1)^2\,\hslash^{2s}
\Big\| \, (-\Delta)^{\frac{s-1}{4}} u \, \Big\|_{L^2(\R^n)}^4}{4\,
\Big\|\, |x|u \, \Big\|_{L^2(\R^n)}^2},$$
which gives the desired result.
\end{proof}

%%%%%%%%%%%%%%%%%%%%%%%%%%%%%

\section[A fractional elliptic problem in $\Rn$]{A fractional elliptic problem in $\Rn$ with critical growth and convex nonlinearities}

\noindent  The goal of this section is to prove the existence of a positive solution to the convex problem 
\begin{equation} \label{scrproblem}(-\Delta)^s u =\varepsilon h u^q+u^{\cx-1} \qquad\mbox{ in }\mathbb{R}^n,
\end{equation}
where \textcolor{black}{$s\in (0,1)$, $n>2s$, $1\leq q <2^{\star}_s-1$ are given quantities}, $\varepsilon>0$ is a small parameter, and $h$ is a function satisfying suitable summability conditions. 
%The main operator in this problem is the fractional Laplacian, defined by
%\[ (-\Delta)^s u(x)  = C(n,s) \mbox{ P.V. } \int_{\Rn} \frac{u(x)-u(y)}{|x-y|^{n+2s}}\, dy\]
%	for any $x\in \Rn$ \textcolor{black}{and for a function $u$ regular enough\footnote{It is enough to take $u\in\mathcal{S}(\Rn)$ (the Schwartz space of rapidly decreasing functions), or in $ L^{\infty}(\Rn)$ and $C^2$ in a neighborhood of $x$, to have a pointwise definition of the fractional Laplacian. Check also \cite{silvestre} for a refinement of the space of definition.}}, where $C(n,s)$ is a positive constant. For details on this operator and applications, see \cite{nonlocal}. See also \cite{hitch} for an introduction to fractional Sobolev spaces. 
%\medskip
The main result of this section goes as follows.
	\begin{theorem}\label{scrtheorem}
	Let $q\in [1, \cx-1)$ and $h$ be such that
	\eqlab{\label{h1} & h\in L^1(\Rn)\cap L^{\infty}(\Rn)\quad \mbox{and}\\
	& \mbox{there exists a ball }B\subset \Rn \mbox{ such that } \inf_B h>0.\\
	& \mbox{If }n\in (2s,6s), \mbox{suppose in addition }h\geq 0.}
	Let $\varepsilon>0$ be a small parameter. Then the problem \eqref{scrproblem}
 admits a positive (mountain-pass) solution, provided that
	 $n>{\frac{2s(q+3)}{q+1} }$.
	\end{theorem}
		
The literature concerning problems with this type of nonlinearities is large and deep in the classical case, see for instance \cite{ABC, AGP1, AGP, ALM, CW, Cing}, among others. In particular, in \cite{AGP} A. Ambrosetti, J. Garc\'ia-Azorero and I. Peral studied \eqref{scrproblem} for $s=1$. There, \textcolor{black}{the existence of solutions} is proved by means of two different techniques: bifurcation and concentration-compactness. In the first case, they construct solutions for the whole range $0<q<2^{\star}_s-1$ as small perturbations of the solutions to the problem
$$-\Delta u = u^{2^{\star}-1}\qquad \mbox{ in }\mathbb{R}^n,\qquad u>0,$$
by using a Lyapunov-Schmidt reduction. \textcolor{black}{
On the other hand, the authors also prove the existence of two solutions for $0<q<1$ (that is, the concave-convex problem) by applying an argument of concentration-compactness type (in the spirit of \cite{lions1, lions2}).}

The fractional counterpart of these results is as follows. In \cite{DMPV}, a solution to \eqref{scrproblem} for $0<q<2^{\star}_s-1$ is obtained by means of a Lyapunov-Schmidt reduction. Indeed, the authors prove the existence of a function $w_\varepsilon$ (which goes to zero \textcolor{black}{in a suitable space with $\varepsilon\to 0$}) so that, for some $\mu\in (0,+\infty)$ and $\xi\in \mathbb{R}^n$, $z_{\mu,\xi}+w_\varepsilon$ solves the problem \eqref{scrproblem}, where
\begin{equation} \label{zmu} z_{\mu,\xi}(x)=\mu^{\frac{2s-n}{2}}z\left(\frac{x-\xi}{\mu}\right),\qquad z(x)=\frac{c_*}{(1+|x|^2)^{\frac{n-2s}{2}}}\end{equation}
is a solution of
$$(-\Delta)^s u = u^{2^{\star}_s-1}\qquad \mbox{ in }\mathbb{R}^n,\qquad u>0.$$
\textcolor{black}{Moreover, in \cite{maria} for the range $0<q<1$ the authors use the concentration-compactness principle to prove the existence of two solutions for the problem \eqref{scrproblem}}
%    The analogous result in the fractional case can be found in \cite{maria} 
    (see also \cite{Bego, SV-2, SV-1} for related problems in the nonlocal case).
\medskip

\textcolor{black}{In this section, we solve the problem \eqref{scrproblem}
in the fractional case $s\in (0,1)$ and in the range $1\leq q<2^{\star}_s-1$, using a concentration-compactness principle. Notice that in our problem the two nonlinearities are convex, and the geometry of the functional suggests the existence of one solution instead of two. In order to prove the existence of a solution we use, roughly speaking, the following strategy:}

(i) we consider the energy functional associated to \eqref{scrproblem} and we prove that it satisfies some compactness condition (Palais-Smale condition) under a certain energy level.

(ii) we build a sequence of functions with an appropriate geometry (of Mountain Pass type) whose energy lies below the critical level found in (i). 

(iii) we apply the Mountain Pass Lemma (see \cite{mps}) to pass to the limit, getting a solution.\\
There are two fundamental points here: to identify the energy level, and to find the appropriate sequence. \textcolor{black}{We point out that, in the concave-convex (fractional) problem studied in \cite{maria}}, the geometry derived from the concave term (the functional has a minimum of negative energy) helps to prove that the sequence stays below the critical level. However, in our problem both nonlinearities are convex, and the proof gets more involved. Indeed, if one adapts straightforwardly the compactness result in \cite[Proposition 4.2.1]{maria} and builds the sequence in the standard way (by considering the path along the Sobolev minimizers), then the arguments to prove that the energy of the sequence is small enough do not work. \\
Thus, the study of \eqref{scrproblem} will first require a finer analysis of the compactness properties of the functional. More precisely, we will have to improve the estimates of the functional in order to get a slightly higher critical level.
Accordingly, once we have found this new critical level, we perform a more careful analysis of the energy of the sequence given by the minimizers. We will finally conclude by applying the Mountain Pass Lemma in the standard way.

\textcolor{black}{We remark here that in this work we also overcome 
%It is worth to point out here that there is 
a flaw found in \cite{AGP}, where the classical problem is studied}; indeed, to prove compactness (Proposition 2.1 therein) they state that the critical energy level $c_\varepsilon$ has to satisfy
$$c_\varepsilon<\frac{1}{n}S^{{n}/2}-C\varepsilon^{\frac{2^{\star}}{2^{\star}-(q+1)}}.$$
Nevertheless, if one follows the proof it arises that, in order to reach the contradiction, it has to be required that
$$c_\varepsilon<\frac{1}{{n}}S^{{n}/2}-C\varepsilon^{\frac{2^{\star}}{2^{\star}-(q+1)}}-C\varepsilon.$$
Notice that what we are saying here is that the compactness holds below a lower critical level, and thus it will be more difficult to find the sequence in (ii). This flaw was already fixed in \cite{maria} in the fractional, \textcolor{black}{concave-convex} case (see Proposition 4.2.1), where the authors consider the lower level and find the appropriate sequence. 
%\textcolor{black}{Our paper thus also gives a proof of the existence of a solution of problem \eqref{scrproblem} in the case $s=1$ and for $q\in[1,\cx)$, with the correct lower critical level.}

\bigskip

\textcolor{black}{We make now some preliminary observations on the problem that we study.} We see at first that if $h$ satisfies conditions \eqref{h1}, then also 
		\bgs{\label{h2} h\in L^m(\Rn) \mbox{ for any } m \in (1,+\infty).}
Furthermore, in the case $n\in (2s,6s)$ we need to ask $h$ to be positive. This restriction arises again from the study of the energy of the Sobolev minimizers. As we commented before, we would like to control the energy of the sequence that we will construct (and that will be based on the functions $z_{\mu,\xi}$, \textcolor{black}{see \eqref{zmu}}), and thus we would like the negative terms to be as large as possible. In particular, if one looks at the $q$-order term, we hope that the part where $h$ is positive \textcolor{black}{dominates over} the part where it is negative. To have this, we will center the function $z_{\mu,\xi}$ in the ball where $h$ is positive, so that the mass is concentrated there. However, it can be easily seen that for low dimensions the mass of the tails of the minimizers is too large and it annihilates the mass in the positive part of $h$. This computation gives an idea of why the necessity of \textcolor{black}{requiring} $h\geq 0$ for $n\in (2s,6s)$, but the detailed restriction can be found in Section \ref{proofThm}.
\medskip

The section is organized as follows: in Subsection \ref{strategy} we provide the functional framework that will be needed, as well as some auxiliary results related to compactness and geometry properties. Subsection \ref{PSC} is devoted to the proof of the Palais-Smale condition for the energy functional, and Subsection \ref{minmax} to construct the sequence with mountain pass geometry and whose energy level lies below the critical one. Finally, in Subsection \ref{proofThm} we prove Theorem \ref{scrtheorem}.

	\subsection{Functional framework and preliminary computations}\label{strategy}

\noindent \textcolor{black}{We introduce at first some notations.} Let us denote by $\Rp:=\Rn\times (0,+\infty)$ the $n+1$ dimensional half-space, by $X=(x,y)\in \Rp$ a $n+1$ dimensional vector, having $x\in \Rn$ and $y>0$, and take $a:=1-2s$.  Moreover, for $x\in \Rn$ and $r>0$ we write $B_r(x)$ \textcolor{black}{(shorted to $B_r$ when $x=0$)} for the ball in $\Rn$ centered at $x$ with radius $r$, i.e.
	  \[ B_r(x):= \{ x'\in \Rn \mbox{ s.t. } |x-x'|<r\},\]
	  and for $X\in \Rp$ and $r>0$ we write $B_r^+(X)$ for the ball in $\Rp$ centered at $X$ with radius $r$, that is
	  \[    B_r^+(X):= \{ X'\in \Rp \mbox{ s.t. } |X-X'|<r\}.\]
Let us introduce first the seminorm
		\[ [u]^2_{\dot H^s(\Rn)}: = \iint_{\R^{2n}} \frac{|u(x)-u(y)|^2}{|x-y|^{n+2s}}\, dx\, dy,\] 
and define  the space $\dot H^s(\Rn)$ as the completion of the Schwartz space of rapidly decreasing smooth functions, with respect to the norm $[\,\cdot\,]_{\dot H^s(\Rn)} + \|\cdot\|_{L^{2^{\star}_s}(\mathbb{R}^n)}$. 
For the sake of simplicity, from now on we will use the notation $\|\cdot\|$ for the $\Lst$ norm.
	
\begin{defn}
We say that $u\in \dot H^s(\Rn)$ is a (weak) solution of $(-\Delta)^s u=f$ in $\Rn$ for a given $f\in L^\beta(\Rn)$ where $\beta:=2n/(n+2s)$ if
	\[ \frac{C(n,s)}{2} \iint_{\R^{2n}} \frac{ \left(u(x)-u(y)\right)\left(\varphi(x)-\varphi(y)\right)}{|x-y|^{n+2s}} \, dx \, dy = \int_{\Rn} f(x)\varphi(x)\, dx\] for every $\varphi \in \dot H^s(\Rn)$.
\end{defn}

\noindent Nevertheless, instead of directly working in this framework, we will transform the problem into a local one by using the extension due to L. Caffarelli and L. Silvestre (see Chapter \ref{S:3} and the original result in the paper \cite{CAFSIL}). 
%
%Thus, the operator $(-\Delta)^s$ can be obtained as the trace of a local (possibly singular and degenerate) operator acting on the half space.  Given $U\colon \Rp\to\R$ that satisfies
%		\syslab{ \label{extprc}&\mbox{div} (y^a \nabla U)= 0 & &\mbox{ in } \Rp,\\
%			&	U(x,0)=u(x) &&\mbox{ in } \Rn,}
%			it holds that, up to constants,
%			\[ (-\Delta)^s u (x)= -\lim_{y \to 0^+} y^a \partial_y U(x,y).\]
Let $U$ be the extension operator defined in \eqref{lapextop} an denote 
			\[ [U]_a^{\star}:=\left( \kappa_s\int_{\R^{n+1}} y^a |\nabla U|^2 \, dX\right)^{1/2} ,\] where $\kappa_s$ is a normalization constant. We then define the spaces
			\[ \dot H_a^s (\R^{n+1}) := \overline{ C_0^\infty(\R^{n+1})}^{[\cdot]_a^{\star}}\]
			and
			\[ \Hsa := \Big\{ U:=\tilde U\Big|_{\Rp} \mbox{ s.t. }\tilde U \in \dot H^s_a (\R^{n+1}), \tilde U(x,y)=\tilde U(x,-y) \mbox{ a.e. in } \Rn \times \R \Big\}.\]
The norm in $\Hsa$ is, neglecting the constants,
	\bgs{\label{norm} [U]_a:= \lr{\int_{\R^{n+1}_+} y^a|\nabla U|^2 \, dX }^{1/2}.} 
	        
\noindent So, finding a solution $u\in \dot H^s(\Rn)$ of the nonlocal problem $(-\Delta)^s u=f(u)$ (thanks to \eqref{lapextop}) is equivalent to finding a solution $U\in \Hsa$ of the local problem
\sys{&\mbox{div} \left(y^a\nabla U\right) =0 && \mbox{ in } \Rp,\\
& -\lim_{y\to 0^+} y^a \partial_y U = f(u) && \mbox{ in } \Rn .}
Since we are looking for positive solutions of \eqref{scrproblem}, we will consider the problem
\eqlab{\label{eqpls} (-\Delta)^s u =\varepsilon h u _+^q + u_+^{\cx-1} \quad \mbox{ in } \Rn}
and (according to the considerations above) its equivalent formulation 
\syslab{\label{equiveqplus} &\mbox{div} \left(y^a\nabla U\right) =0 && \mbox{ in } \Rp,\\
& -\lim_{y\to 0^+} y^a \partial_y U (x,y)= \varepsilon h U _+^q(x,0) + U_+^{\cx-1}(x,0) && \mbox{ in } \Rn.}
In particular, we say that $U\in \Hsa$ is a (weak) solution on the problem \eqref{equiveqplus} if
\bgs{\label{weaksolU} \int_{\Rp} y^a\langle \nabla U,\nabla \varphi\rangle \, dX =\int_{\Rn} \lr{ \varepsilon h(x) U_+^q(x,0) +U_+^{\cx-1}(x,0)} \varphi(x,0)\, dx,}
for every $\varphi \in \Hsa$. Furthermore, the energy functional associated to the problem \eqref{equiveqplus} is 
\bgs{ \label{oper} \Fl (U):=\frac{1}2 \int_{\Rp}y^a |\nabla U|^2 \, dX - \frac{\varepsilon}{q+1} \int_{\Rn} h(x) U_+^{q+1} (x,0)\, dx -\frac{1}{\cx}\int_{\Rn} U_+^{\cx}(x,0)\, dx.} 
In particular $\Fl \in C^1(\Hsa)$ and for any $U,V\in \Hsa$
\bgs{\label{operderiv}&  \langle \Fl'(U),V\rangle\\ =&\; \int_{\Rp} y^a\langle \nabla U,\nabla V\rangle \, dX - \varepsilon \int_{\Rn} h(x) U_+^q (x,0)V(x,0) - \int_{\Rn} U_+^{\cx-1}(x,0) V(x,0)\, dx.}
The purpose of the section from here on is to prove the existence of a critical point $U$ of the operator $\Fl$. Then, $U$ is a solution of \eqref{equiveqplus} and therefore $u:=U(\cdot, 0)$ is a solution of \eqref{eqpls}. Moreover, one can prove that any nontrivial solution $u$ of \eqref{eqpls} (hence its extension $U$) is nonnegative, and therefore a true solution of \eqref{scrproblem} (see for this \cite[Proposition 2.2.3]{maria}).
\medskip

It is known that (up to constants) the harmonic extension of the fractional Laplacian gives an isometry between $\dot H^s(\Rn)$ and $\Hsa$, i.e.
	\eqlab{\label{isom}    [u]_{\dot H^s(\Rn)} = [U]_a.}	We recall that the Sobolev embedding in $\dot H^s(\Rn)$ gives that
	\[ S\|u\|^2 \leq [u]^2_{\dot H^s(\Rn)},\] where $S$ is the best constant of the Sobolev embedding of $\dot H^s(\Rn)$ (see for instance \cite[Theorem 6.5]{galattica}). As a consequence, we have the following inequality,
	\begin{prop}[Trace inequality] \label{traceIneq}
Let $U \in \Hsa$. Then
\bgs{ S\|U(\cdot, 0)\|^{\textcolor{black}{2}}\leq [U]_a^2.}
\end{prop}			
\noindent In \cite[Theorem 1.1]{costi} the best Sobolev constant and the fractional Sobolev minimizers are explicitly computed. The form of the fractional Sobolev minimizer is given by
	\eqlab{\label{SobMin}  z(x):= \frac{c_\star}{(1+|x|^2)^{\frac{n-2s}{2}}}} for a positive constant $c_\star=c_\star(n,s)$. 

We introduce for $r\in(1,+\infty)$ the weighted Lebesgue space endowed with the norm
\[\|U\|_{L^r(\Rp,y^a)} :=\lr{\int_{\Rp}y^a |U|^r\, dX}^{1/r}.\] The following result gives a continuous Sobolev embedding of the space $\Hsa$ into the weighted Lebesgue space for a particular value of $r$. See for the proof \cite[Proposition 3.1.1]{maria}.

\begin{prop}[Sobolev embedding] \label{sob}There exists a constant $\widehat S>0$ such that for all $U\in \Hsa$ it holds that
\bgs{ \lr{ \int_{\Rp} y^a|U|^{2\gamma} \, dX}^{1/2\gamma} \leq \widehat S \lr{ \int_{\Rp} y^a |\nabla U|^2 \, dX}^{1/2},}
where $\gamma=1+2/(n-2s)$. 
\end{prop}

\noindent In the next proposition, we prove a useful integral inequality that will be frequently used.
\begin{prop} \label{qineq} Let $1\leq q <\cx-1$. Assume $u\in L^{2^{\star}_s}(\mathbb{R}^n)$  and $h\in L^m(\mathbb{R}^n)$ with $m=\frac{2^{\star}_s}{2^{\star}_s-(q+1)}$. Then
\bgs{\label{bound3} \left| \int_{\Rn} h(x) u^{q+1}(x) \, dx\right| \leq \|h\|_{L^m(\Rn)} \|u\|^{q+1}.}
\end{prop}

\begin{proof} We use the Hölder inequality to deduce that
	\bgs{ \left| \int_{\Rn} h(x) u^{q+1} (x) \, dx \right| \leq \al \int_{\Rn} |h(x)| u^{q+1} (x) \, dx
	\leq  \lrq{\int_{\Rn}|h|^{\frac{\cx}{\cx-q-1} } \, dx }^{\frac{\cx-q-1}{\cx}} \lrq{ \int_{\Rn} |u| ^{\cx} \, dx }^{\frac{q+1}{\cx}}\\
	\leq \al \|h\|_{L^m(\Rn)} \|u\|^{q+1}  } for $m= \frac{\cx}{\cx-q-1}>1$, and so the inequality is proved. 
\end{proof}

\noindent The next proposition is the equivalent of \cite[Lemma 4.1.1]{maria} in the case $q\geq1$ and goes as follows.
\begin{prop}\label{convres}
Let $v_k \in L^{\cx}(\Rn,[0,+\infty))$ be a sequence converging to some $v$ in $L^{\cx}(\Rn)$. Then for any $r>1$
\[ \lim_{k \to +\infty} \int_{\Rn} |v_k^r(x) -v^r(x)|^{\frac{\cx}{r}} \, dx =0.\]
\end{prop}

\begin{proof}
For any $a\geq b\geq 0$ and any $r>1$ we see that
\[ a^r-b^r =r\int_{b}^{a} t^{r-1}\, dt \leq r a^{r-1} (a-b) \leq r(a^{r-1}+b^{r-1})(a-b).\]
Exchanging $a$ with $b$, we conclude that
\[ \big| a^r-b^r \big|\leq r (a+b)^{r-1}|a-b|.\]
Then by the Hölder inequality we have that
\bgs{\int_{\Rn} |v_k^r(x)-v^r(x)|^{\frac{\cx}{r}}  \, dx \leq &\; r^{\frac{\cx}{r}} \lr{\int_{\Rn} (v_k+v)^{\cx }\, dx} ^{(r-1)/r} \lr{ \int_{\Rn} |v_k-v|^{\cx}\, dx}^{1/r} \\
\leq &\; r^{\frac{\cx}{r}} \
\|v_k+v\|^{\frac{(r-1) \cx}{r}} \|v_k-v\|^{\frac{\cx}{r}}   .} 
Using the convergence $\|v_k-v\|\to 0$ (from which it also follows that $\|v_k+v\|$ is uniformly bounded), the conclusion plainly follows.
%of $v_k$ we have that 
%$\|v_k+v\|\leq \|v_k\|+\|v\|$ is uniformly bounded and that $\|v_k-v\|$ goes to zero as $k $ goes to infinity. From here 
\end{proof}

\noindent Another useful result is given in \cite[Lemma 4.2.4]{maria}. We just notice that now, for $q>1$, the statement goes as follows:
\begin{prop}\label{proppqs}
Let $m:= \frac{\cx}{\cx-(q+1)}$. Then there exists a positive constant $\bar C$ depending on $n,s,q$ and $\|h\|_{L^m(\Rn)}$ such that, for any $\alpha>0$,
\bgs{ \frac{s}n \alpha^{\cx} -\varepsilon\lr{\frac{1}2-\frac{1}{q+1}}\|h\|_{L^m(\Rn)}\alpha^{q+1}\geq -\bar C\eps^{\frac{\cx}{\cx-(q+1)}}.}
\end{prop}

\subsection{Palais-Smale condition}\label{PSC}
\noindent The main result of this Section is the following.
\begin{theorem}\label{PSCthm} 
There exists~$\bar C, c_1>0$, depending on~$h, q, n$
and~$s$, such that the following statement
holds true. 

Let $\{U_k\}_{k\in \N}\subset \Hsa$ be a sequence satisfying
\begin{enumerate}
\item[(i)]$\displaystyle\lim_{k\to+\infty}\mathcal{F}_\varepsilon(U_k)= 
c_\varepsilon$, with 
\begin{equation*}\begin{split}\label{ceps}
&c_\varepsilon+c_1\varepsilon^{1+\delta} +\overline C \varepsilon^{\frac{\cx}{\cx-(q+1)}}<
\dfrac{s}{n}S^{\frac{n}{2s}}\qquad  \hbox{ if }n\geq 6s,\\
&c_\varepsilon+c_1\varepsilon^{1+\delta}<
\dfrac{s}{n}S^{\frac{n}{2s}} \qquad \hbox{ if }n\in (2s,6s),
\end{split}\end{equation*} 
where $\delta>0$ and $S$ is the Sobolev constant appearing in Proposition~\ref{traceIneq},
\item[(ii)]$\displaystyle\lim_{k\to+\infty}\mathcal{F}'_\varepsilon(U_k)= 0.$
\end{enumerate}
Then there exists a subsequence, still denoted by~$\{U_k\}_{k\in\mathbb{N}}$, 
which is strongly convergent in $\dot{H}^s_a(\mathbb{R}^{n+1}_+)$ as~$k\to+\infty$.
\end{theorem}

\noindent \textcolor{black}{Here, the limit in $(ii)$ is to be intended as
\[ \lim_{k\to +\infty}  \|\mathcal F'(U_k) \|_{\mathcal L(E,E)} = \lim_{k\to +\infty} \sup_{V\in E,\|V\|_{E}=1} \left|\scp{ \mathcal F'(U_k),V }\right| = 0, \] where we denote by $\mathcal L(E,E)$ the space of all linear functionals from $E$ to $E$.}
\begin{remark}
As we commented in the introduction, one of the key points in this work is to slightly improve the critical level in such a way that further on we can build a sequence whose energy lies below it. This is precisely the role played by the parameter $\delta$ in the previous theorem. We can not drop this term (that will cause important difficulties) but we can choose $\delta$ large enough so that we can neglect it when $\varepsilon\rightarrow 0$.
\end{remark}

\noindent We recall at first a concentration-compactness principle, stated in \cite[Proposition 3.2.3]{maria} and proved there. This principle is based on the original results by P.L Lions in \cite{lions1,lions2} (in particular  in \cite[Lemma 2.3]{lions2}). For this, we recall the next definitions.
\begin{defn}
A sequence $\{U_k\}_{k\in \N}$ is tight if for every $\mu>0$ there exists $\rho>0$ such that for any $k\in \N$
\[ \int_{\Rp\setminus B_{\rho}^+}y^a |\nabla U_k|^2 \, dX\leq \mu.\]
\end{defn}
\begin{defn}\label{measconv}
Let $\{\mu_k\}_{k\in \N}$ be a sequence of measures on a topological space $X$. We say that $\mu_k $ converges to $\mu$ on $X$ if and only if
\[\lim_{k\to +\infty} \int_X \varphi d\mu_k =\int_X \varphi \, d\mu \quad \mbox{ for any } \varphi \in C_0(X).\]
\end{defn}
\noindent Then the principle goes as follows.
\begin{prop}[Concentration-Compactness Principle]\label{CCP}
Let $\{U_k\}_{k\in \N}$ be a bounded and tight sequence in $\Hsa$ such that $U_k$ converges weakly to $U$ in $\Hsa$. Let $\mu,\nu$ be two nonnegative measures on $\Rp$ respectively $\Rn$ such that (in the sense of Definition \ref{measconv})
\bgs{\label{measconv1} \lim_{k \to +\infty} y^a|\nabla U_k|^2 =\mu }
and
\bgs{\label{measconv2} \lim_{k \to +\infty} | U_k(x,0)|^{2_s^{\star}} =\nu. }
 Then there exists a set $J$ that is at most countable and three families $\{x_j\}_{j\in J}\in \Rn$, $\{\nu_j\}_{j\in J}$ and $\{\mu_j\}_{j\in J}$ with $\nu_j, \mu_j\geq0$ such that
\begin{flalign*}
\mbox{ (i) } &\nu=|U(x,0)|^{2_s^{\star}}+\sum_{j\in J} \nu_j \delta_{x_j} \qquad \qquad \qquad\qquad\qquad\qquad\qquad\qquad\qquad\qquad\qquad\qquad\qquad\qquad\qquad \\
\mbox{ (ii) }& \mu \geq y^a|\nabla U|^2 + \sum_{j\in J}\mu_j \delta_{\{x_j,0\}}\\
\mbox{ (iii) }&\mu_j\geq S\nu_j^{2/2_s^{\star}} \mbox{ for all } j \in J.
\end{flalign*}

\end{prop}
\noindent We prove that a sequence $\{U_k\}_{k\in\N}$ satisfying the assumptions in Theorem \ref{PSCthm} is bounded. \textcolor{black}{A slighter more general result is given in the following Lemma.}

\begin{lemma}\label{Ukisbd}
Let $\varepsilon, \kappa>0$ and let $\{U_k\}_{k\in \N} \subset \Hsa$ be a sequence that satisfies
\eqlab{ \label{bound1} |\Fl (U_k)| + \sup_{V\in \Hsa, \,  [V]=1}| \scp{\Fl' (U_k),V}|\leq \kappa }
for any $k\in \N$. Then there exists $M>0$ such that for any $k\in \N$
	\eqlab{\label{bound}  [U_k]_a\leq M.}
\end{lemma}

\begin{proof}
We suppose by contradiction that for every $M>0$ there exists $k\in \N$ such that
	\eqlab{ \label{ra}  [U_k]_a>M.}
Thanks to \eqref{bound1} we have that
	\bgs{  \kappa \geq \al \Fl (U_k) = \frac{1}{2} [U_k]_a^2 -\frac{\varepsilon}{q+1} \int_{\Rn} h(x) (U_k)_+ ^{q+1}(x,0)\, dx -\frac{1}{\cx}\int_{\Rn} (U_k)_+^{\cx} (x,0)  \, dx.}
Using also the bound in \eqref{bound3}, we obtain that
	\eqlab{ \label{e1}   \,[U_k]_a^2 \leq \al 2 \kappa + \frac{2\varepsilon}{q+1} \int_{\Rn} h(x) (U_k)_+ ^{q+1}(x,0)\, dx +\frac{2}{\cx}\int_{\Rn} (U_k)_+^{\cx} (x,0) \, dx \\
		\leq \al 2 \kappa + \frac{2 \varepsilon}{q+1} \|h\|_{L^m(\Rn)} \|(U_k)_+\|^{q+1} +\frac{2}{\cx}\| (U_k)_+\|^{\cx} .}
Thus, from this and \eqref{ra}, we deduce that also \textcolor{black}{for every $\tilde{M}>0$ one can find $k\in\mathbb{N}$ so that}
	\eqlab{ \label{ra2} \|\ukp\|>\tilde{M}.} 
Consider now the function $f\colon (0,\infty) \to (0,\infty)$ defined as 
	\[ f(\tau):=\frac{\tau^{q+1}}{\tau^{\cx}}.\] 
Since $q+1<\cx$ we have that
	\[\lim_{\tau \to \infty}f(\tau)=0,\] 
and hence, for any $\delta>0$ there exists $\tau_\delta>0$ such that for every $\tau>\tau_\delta$, one has that $f(\tau)< \delta$. Hence, fixing  $0<\delta<1$, by \eqref{ra2} we can assume
\begin{equation}\label{lowBound1}
\|(U_k)_+\|>\tau, \qquad \|\ukp\|^{q+1}\leq \delta\|\ukp\|^{\cx},\qquad \forall \; \tau>\tau_\delta.
\end{equation}	
Therefore, by Proposition \ref{traceIneq} \textcolor{black}{there exists $k\in \N$ such that }
\begin{equation}\label{lowBound2}
[U_k]_a>\tau S^{1/2}, \qquad \forall \; \tau>\tau_\delta.
\end{equation}
	\textcolor{black}{Using \eqref{lowBound1} and} \eqref{e1} we obtain that
	\eqlab{ \label{bound2}   	[U_k]_a^2  \leq 2\kappa+ \lr{\delta \frac{2\varepsilon}{q+1} \|h\|_{L^m(\Rn)}+ \frac{2}{\cx}  } \|\ukp\|^{\cx}.}
	
\noindent On the other hand, considering the quotient $U_k / [U_k]_a$ from \eqref{bound1} we get that
	\[ |\scp{\Fl'(U_k),U_k}|\leq \kappa [U_k]_a.\]
	From this and the fact that $|\Fl(U_k)|\leq \kappa$, for $q>1$ we have that
		\eqlab{\label{usel} \kappa (1+[U_k]_a)\geq \al \Fl(U_k) - \frac{1}{2}\scp{\Fl'(U_k),U_k} \\ =\al \varepsilon\lr{\frac{1}{2} -\frac{1}{q+1} }\int_{\Rn} h(x) \ukp^{q+1}(x,0)\, dx + \frac{s}{n} \|\ukp \|^{\cx},}
		\textcolor{black}{recalling that
		\[ \frac{1}2-\frac{1}{\cx} =\frac{s}n.\]}
Thanks to the bound in \eqref{bound3}, it follows that
	\bgs{ \frac{s}{n} \|\ukp\|^{\cx} \leq\al   \kappa (1+[U_k]_a) + \varepsilon\lr{\frac{1}{2} -\frac{1}{q+1} } \|h\|_{L^m(\Rn)} \|\ukp\|^{q+1}.}
We use \eqref{lowBound1} again and we obtain that 
			\bgs{ \frac{s}{n} \|\ukp\|^{\cx} \leq\al \kappa (1+[U_k]_a) +\delta \varepsilon \lr{\frac{1}{2} -\frac{1}{q+1} } \|h\|_{L^m(\Rn)}  \|\ukp\|^{\cx}  .}
			Thus
					\bgs{\lrq{ \frac{s}{n}-  \delta \varepsilon \lr{\frac{1}{2} -\frac{1}{q+1} } \|h\|_{L^m(\Rn)}  } \|\ukp\|^{\cx} \leq\al \kappa (1+[U_k]_a)  ,}
					which for $\delta$ small enough, implies that
						\[ c \|\ukp\|^{\cx}\leq \kappa(1+[U_k]_a).\] 
Notice that for $q=1$ the inequality above immediately follows from \eqref{usel}. 
						This, together with \eqref{bound2}, yields 
						\[ [U_k]_a^2\leq C_1+ C_2 [U_k]_a \] 
						for suitable positive constants $C_1, C_2$, both independent of $k$. Choosing $\tau$ large enough in \eqref{lowBound2} we contradict this inequality and conclude the proof.
\end{proof}
\noindent Furthermore, a sequence $\{U_k\}_{k\in \N} \subset \Hsa$ that satisfies the hypotheses of Theorem \ref{PSCthm} is tight, as stated in the next Lemma.

\begin{lemma}\label{tightness} Let $\{U_k\}_{k\in \N} \subset \Hsa$ be a sequence that satisfies the hypothesis of Theorem \ref{PSCthm}. Then for any $\eta>0$ there exists $\rho>0$ such that for any $k\in \N$ it holds that
\[ \int_{\Rp\setminus B_{\rho}^+} y^a|\nabla U_k|^2 \, dX + \int_{\Rn \setminus \{B_\rho \cap\{y=0\}\}} (U_k)^{\cx} (x,0)\, dx<\eta.\] In particular, the sequence $\{U_k\}_{k\in \N}$ is tight.
\end{lemma}
\begin{proof}
First we notice that~\eqref{bound1}
holds in this case, due to conditions~(i) and~(ii) in
Theorem~\ref{PSCthm}. Hence,
Lemma \ref{Ukisbd} 
%%% and Proposition~\ref{traceIneq} 
gives that the sequence 
$\{U_k\}_{k\in\mathbb{N}}$ is uniformly bounded in $\dot{H}^s_a(\mathbb{R}^{n+1}_+)$, and thus
\begin{equation}\begin{split}\label{weak convergence-1}
& U_k\rightharpoonup U \quad \hbox{ in }\dot{H}^s_a(\mathbb{R}^{n+1}_+) \quad {\mbox{ as }}k\to+\infty \\
{\mbox{and }}& U_k\rightarrow U\quad  \hbox{ a.e. in }\mathbb{R}^{n+1}_+\quad {\mbox{ as }}k\to+\infty.
\end{split}\end{equation}

\noindent We now proceed by contradiction. Suppose that there exists $\eta_0>0$ 
such that for all $\rho>0$ there exists~$k=k(\rho)\in\N$ such that
\begin{equation}\label{contrad}
\int_{\mathbb{R}^{n+1}_+\setminus B_\rho^+}{y^a|\nabla U_k|^2\,dX}
%+\int_{\mathbb{R}^{n+1}_+\setminus B_\rho^+}{y^a|U_k|^{2\gamma}\,dX}
+\int_{\mathbb{R}^n\setminus\{B_\rho\cap\{y=0\}\}}{(U_k)_+^{2^{\star}_s}(x,0)\,dx}
\geq \eta_0.
\end{equation}
We observe that 
\begin{equation}\label{forse0} 
k\to+\infty \quad {\mbox{ as }}\rho\to+\infty.
\end{equation}
Indeed, let us take a sequence $\{\rho_i\}_{i\in\mathbb{N}}$ such that $\rho_i\rightarrow +\infty$ as $i\rightarrow +\infty$, and suppose that $k_i:=k(\rho_i)$ given by \eqref{contrad} is a bounded sequence. That is, the set $F:=\{k_i:\;i\in\N\}$ is a finite set of integers.

Hence, there exists an integer $k^\star$ so that we can extract a subsequence $\{k_{i_j}\}_{j\in\N}$ satisfying $k_{i_j}=k^\star$ for any $j\in\N$. Therefore,
\begin{equation}\label{contrad2}
\int_{\mathbb{R}^{n+1}_+\setminus B_{\rho_{i_j}}^+}{y^a|\nabla U_{k^\star}|^2\,dX}
%+\int_{\mathbb{R}^{n+1}_+\setminus B_{\rho_{i_j}}^+}{y^a|U_{k^\star}|^{2\gamma}\,dX}
+\int_{\mathbb{R}^n\setminus\{B_{\rho_{i_j}}\cap\{y=0\}\}}{(U_{k^\star})_+^{2^{\star}_s}(x,0)\,dx}
\geq \eta_0,
\end{equation}
for any $j\in \N$. 
But on the other hand, since~$U_{k^\star}$
belongs to~$\dot{H}^s_a(\mathbb{R}^{n+1}_+)$ 
(and so $U_{k^\star}(\cdot,0)\in L^{2^{\star}_s}(\R^n)$ 
thanks to Proposition \ref{traceIneq}), for $j$ large enough there holds
\begin{equation*}
\int_{\mathbb{R}^{n+1}_+\setminus B_{\rho_{i_j}}^+}{y^a|\nabla U_{k^\star}|^2\,dX}
%+\int_{\mathbb{R}^{n+1}_+\setminus B_{\rho_{i_j}}^+}{y^a|U_{k^\star}|^{2\gamma}\,dX}
+\int_{\mathbb{R}^n\setminus\{B_{\rho_{i_j}}\cap\{y=0\}\}}{(U_{k^\star})_+^{2^{\star}_s}(x,0)\,dx}
\leq \frac{\eta_0}{2},
\end{equation*}
which is a contradiction with \eqref{contrad2}.
This shows~\eqref{forse0}. 

Now, since $U$ given in~\eqref{weak convergence-1} belongs to~$\in\dot{H}^s_a(\mathbb{R}^{n+1}_+)$, by Propositions~\ref{traceIneq} 
and~\ref{sob} 
we have that for a fixed $\varepsilon>0$, there exists $r_\varepsilon>0$ such that
$$\int_{\mathbb{R}^{n+1}_+\setminus B_{r_\varepsilon}^+}{y^a|\nabla U|^2\,dX}
+\int_{\mathbb{R}^{n+1}_+\setminus B_{r_\varepsilon}^+}{y^a|U|^{2\gamma}\,dX}
+\int_{\mathbb{R}^n\setminus\{B_{r_\varepsilon}\cap\{y=0\}\}}{|U(x,0)|^{2^{\star}_s}\,dx}<\varepsilon^\alpha,$$
with $\alpha>\gamma$ 
 and $\gamma$ defined in Proposition \ref{sob}. Notice that, without loss of generality, we can assume that 
\begin{equation}\label{eps to zero}
{\mbox{$r_\varepsilon\to +\infty$ as $\varepsilon\to 0$.}} 
\end{equation}
On the other hand, since $h\in L^m(\R^n)$ for every $m\in(1,+\infty)$, in particular we can assure the existence of a radius $\bar{r}_\varepsilon$ such that
\begin{equation}\label{heps}
\|h\|_{L^{m}(\R^n\setminus B_{\bar{r}_\varepsilon})}\leq \varepsilon^{\beta},
\end{equation}
with $m$ satisfying $\frac{1}{m}=1-\frac{q+1}{2^{\star}_s}$ and $\beta>\alpha/\gamma-1$.\\
Moreover, by \eqref{bound} and again by Propositions~\ref{traceIneq}
and~\ref{sob}, there exists $\tilde{M}>0$ such that
\begin{equation}\label{boundk}
\int_{\mathbb{R}^{n+1}_+}{y^a|\nabla U_k|^2\,dX}+\int_{\mathbb{R}^{n+1}_+}{y^a|U_k|^{2\gamma}\,dX}
+\int_{\mathbb{R}^n}{|U_k(x,0)|^{2^{\star}_s}\,dx}\leq \tilde{M}.
\end{equation}
Let \textcolor{black}{\eqlab{\label{chooser} r:=\max\{r_\varepsilon,\bar{r}_\varepsilon\}.}} Now let $j_\varepsilon\in\mathbb{N}$ be the integer part of $\frac{\tilde{M}}{\varepsilon^\alpha}$. Notice that~$j_\varepsilon$ tends to~$+\infty$ as~$\varepsilon$ 
tends to~0. We also set
$$ I_l:=\{(x,y)\in\mathbb{R}^{n+1}_+:r+l\leq |(x,y)|\leq r+(l+1)\},\;l=0,1,\cdots,j_\varepsilon.$$
Thus, from~\eqref{boundk} we get
\begin{eqnarray*}
(j_\varepsilon+1)\varepsilon^\alpha &\geq &
\frac{\tilde{M}}{\varepsilon^\alpha}\varepsilon^\alpha \\&\ge &
 \sum_{l=0}^{j_\varepsilon}\left({\int_{I_l}{y^a|\nabla U_k|^2\,dX}
 +\int_{I_l}{y^a|U_k|^{2\gamma}\,dX}
+\int_{I_l\cap\{y=0\}}{|U_k(x,0)|^{2^{\star}_s}\,dx}}\right),
\end{eqnarray*}
and this implies the existence of $\bar{l}\in\{0,1,\cdots, j_\varepsilon\}$ such that, 
up to a subsequence, 
\begin{equation}\label{epsBound}
\int_{I_{\bar{l}}}{y^a|\nabla U_k|^2\,dX}+\int_{I_{\bar{l}}}{y^a|U_k|^{2\gamma}\,dX}
+\int_{I_{\bar{l}}\cap\{y=0\}}{|U_k(x,0)|^{2^{\star}_s}\,dx}\leq \varepsilon^\alpha .
\end{equation}
We take now a cut-off function~$\chi\in C^\infty_0(\R^{n+1}_+,[0,1])$, 
such that
\begin{equation}\label{3.4bis}
\chi(x,y)=\begin{cases}
1,\quad |(x,y)|\leq r+\bar{l}\\
0,\quad |(x,y)|\geq r+(\bar{l}+1),
\end{cases}
\end{equation}
and 
\begin{equation}\label{3.4bisbis}
|\nabla \chi|\leq 2.
\end{equation} 
We also define 
\begin{equation}\label{3.4ter}
V_k:=\chi U_k \quad {\mbox{ and }}\quad W_k:=(1-\chi)U_k.
\end{equation}
We estimate
\begin{equation}\begin{split}\label{math F}
&|\langle \mathcal{F}'_\varepsilon(U_k)-\mathcal{F}'_\varepsilon(V_k),V_k\rangle |\\
&\quad = \bigg|\int_{\mathbb{R}^{n+1}_+}{y^a\langle\nabla U_k,\nabla V_k\rangle\,dX}
-\varepsilon \int_{\mathbb{R}^{n}}{h(x) (U_k)_+^q(x,0)\,V_k(x,0)\,dx}\\
&\qquad\quad -\int_{\mathbb{R}^{n}}{(U_k)_+^{\cx-1}(x,0)\,V_k(x,0)\,dx}
-\int_{\mathbb{R}^{n+1}_+}{y^a\langle\nabla V_k,\nabla V_k\rangle\,dX}\\
&\qquad\quad +\varepsilon \int_{\mathbb{R}^{n}}{h(x) (V_k)_+^{q+1}(x,0)\,dx}
+\int_{\mathbb{R}^{n}}{(V_k)_+^{\cx}(x,0)\,dx}\bigg|.
\end{split}\end{equation}
First, we observe that
\begin{equation}\begin{split}\label{AA}
&\bigg|\int_{\mathbb{R}^{n+1}_+}{y^a\langle\nabla U_k,\nabla V_k\rangle\,dX}-\int_{\mathbb{R}^{n+1}_+}{y^a\langle\nabla V_k,\nabla V_k\rangle\,dX}\bigg|\\
&\qquad\leq \int_{I_{\overline{l}}}{y^a|\nabla U_k|^2|\chi||1-\chi|\,dX}+\int_{I_{\overline{l}}}{y^a|\nabla U_k||U_k||\nabla\chi|\,dX}\\
&\qquad\qquad +2\int_{I_{\overline{l}}}{y^a|U_k||\nabla U_k||\nabla\chi||\chi|\,dX}+\int_{I_{\overline{l}}}{y^a|U_k|^2|\nabla \chi|^2\,dX}\\
&\qquad =:A_1+A_2+A_3+A_4.
\end{split}\end{equation}
By \eqref{epsBound}, we have that $A_1\leq C\varepsilon^\alpha $, for some $C>0$. 
Furthermore, by the H\"older inequality, \eqref{3.4bisbis} and \eqref{epsBound}, we obtain
\begin{eqnarray*}
A_2 &\leq& 2\int_{I_{\overline{l}}}{y^a|\nabla U_k||U_k|\,dX}\leq 2\left(\int_{I_{\overline{l}}}{y^a|\nabla U_k|^2\,dX}\right)^{1/2}\left(\int_{I_{\overline{l}}}{y^a| U_k|^2\,dX}\right)^{1/2}\\
&\leq& 2\varepsilon^{\alpha/2}  \left(\int_{I_{\overline{l}}}{y^a|U_k|^{2\gamma}\,dX}\right)^{1/{2\gamma}}
\left(\int_{I_{\overline{l}}}y^{a}\,dX \right)^{\frac{\gamma-1}{2\gamma}}.
%{y^{(a-\frac{a}{\gamma})m}\,dX}\right)^{1/2m},
\end{eqnarray*}
%where $m=\dfrac{n+2-2s}{2}$. 
Since $a=(1-2s)>-1$,
%$\left(a-\dfrac{a}{\gamma}\right)m=a=(1-2s)>-1$,  
the second integral is finite, 
and therefore, for $\varepsilon<1$,
\begin{equation*}
A_2\leq \tilde{C}\varepsilon^{\alpha/2}  \left(\int_{ I_{\overline{l} }}{y^a|U_k|^{2\gamma}\,
dX}\right)^{1/{2\gamma}}\leq C\varepsilon^{\alpha/2} \varepsilon^{\alpha/2\gamma}\le C\eps^{\alpha/\gamma},
\end{equation*}
where \eqref{epsBound} was used again. 
In the same way, we get that $A_3\leq C\varepsilon^{\alpha/\gamma}$. Finally, 
\begin{equation*}
A_4\leq C\left(\int_{I_{ \overline{l} }}{y^a|U_k|^{2\gamma}\,dX}\right)^{1/{\gamma}}
\left(\int_{I_{\overline{l}}}{y^{a}\,dX}\right)^{\frac{\gamma-1}{\gamma}}\leq C\eps^{\alpha/\gamma}.
\end{equation*}
Using this information in \eqref{AA}, since $\alpha>\alpha/\gamma$ we obtain that 
$$ \bigg|\int_{\mathbb{R}^{n+1}_+}{y^a\langle\nabla U_k,\nabla V_k\rangle\,dX}
-\int_{\mathbb{R}^{n+1}_+}{y^a\langle\nabla V_k,\nabla V_k\rangle\,dX}\bigg|
\le C\eps^{\alpha/\gamma}, $$
up to renaming the constant $C$. \\
On the other hand by \eqref{3.4ter} and \eqref{epsBound}, 
\begin{eqnarray*}
\bigg|\int_{\mathbb{R}^n}{\Big( (U_k)_+^{\cx-1}(x,0)\,V_k(x,0)-(V_k)_+^{\cx}(x,0)\Big)\,dx}\bigg|
&\le &\int_{\mathbb{R}^n}{|1-\chi^{\cx-1}||\chi| |U_k(x,0)|^{\cx}\,dx}\\
&\leq& C\int_{I_{\overline{l}}\cap\{y=0\}}{|U_k(x,0)|^{2^{\star}_s}\,dx}\leq C\eps^{\alpha}.
\end{eqnarray*}
In the same way, applying the H\"older inequality, one obtains
\begin{equation}\begin{split}\label{ealpha}
&\bigg|\varepsilon\int_{\mathbb{R}^n}{h(x)\,\left((U_k)_+^q(x,0)\,V_k(x,0)-(V_k)_+^{q+1}(x,0)\right)\,dx}\bigg|
\\&\qquad \le \varepsilon\int_{\mathbb{R}^n}{|h(x)|\,|1-\chi^q||\chi| |U_k(x,0)|^{q+1}\,dx}\\
&\qquad \leq C\, \varepsilon\|h\|_{L^\infty(\R^n)}\,
\int_{I_{\overline{l}}\cap\{y=0\}}{|U_k(x,0)|^{2^{\star}_s}\,dx}\leq C\eps^{1+\alpha}.
\end{split}\end{equation}

\noindent All in all, plugging these observations in \eqref{math F}, we obtain that 
\begin{equation}\label{boundV}
|\langle \mathcal{F}'_\varepsilon(U_k)-\mathcal{F}'_\varepsilon(V_k),V_k\rangle|
\leq C\eps^{\alpha/\gamma}.
\end{equation}
Likewise, one can see that
\begin{equation}\label{boundW}
|\langle \mathcal{F}'_\varepsilon(U_k)-\mathcal{F}'_\varepsilon(W_k),W_k\rangle|
\leq C\eps^{\alpha/\gamma}.
\end{equation}

\noindent Now we claim that 
\begin{equation}\label{fprimeV}
|\langle \mathcal{F}'_\varepsilon(V_k),V_k\rangle|\leq C\eps^{\alpha/\gamma}+o_k(1),
\end{equation}
where $o_k(1)$ denotes (here and in the rest of this section)
a quantity that tends to 0 as $k$ tends to $+\infty$. 
For this, we first observe that 
\begin{equation}\label{bbbb}
[V_k]_a\le C\mbox{ and }[W_k]_a\leq C,
\end{equation}
for some $C>0$. Indeed, recalling \eqref{3.4ter} and using \eqref{3.4bis} 
and \eqref{3.4bisbis}, we have 
\begin{eqnarray*}
[V_k]_a^2 &=& \int_{\R^{n+1}_+}y^a|\nabla V_k|^2\,dX \\ 
&=& \int_{\R^{n+1}_+}y^a|\nabla\chi|^2|U_k|^2\,dX + 
\int_{\R^{n+1}_+}y^a\,\chi^2|\nabla U_k|^2\,dX + 2\int_{\R^{n+1}_+}y^a\,\chi\,U_k
\ \langle \nabla U_k, \nabla\chi\rangle\,dX\\
&\le & 4 \int_{I_{\overline{l}} }y^a| U_k|^2\,dX + [U_k]_a^2 +
C\left(\int_{ I_{\overline{l}}}y^a|\nabla U_k|^2\,dX\right)^{1/2}\, 
\left(\int_{I_{\overline{l}}}y^a |U_k|^2\,dX\right)^{1/2}\\
&\le & C \left(\int_{I_{\overline{l}} }y^a| U_k|^{2\gamma}\,dX\right)^{1/\gamma} 
+ [U_k]_a^2 +
C\, [U_k]_a\, \left(\int_{I_{\overline{l}}}y^a |U_k|^{2\gamma}\,dX\right)^{1/2\gamma}, 
\end{eqnarray*}
where the H\"older inequality was used in the last two lines.  
Hence, from Proposition \ref{sob} and using \eqref{bound}, we obtain \eqref{bbbb}. The estimate for $W_k$ can be proved analogously.

Now, we notice that 
\begin{eqnarray*}
|\langle \mathcal{F}'_\varepsilon(V_k),V_k\rangle| \le 
|\langle \mathcal{F}'_\varepsilon(V_k)-\mathcal{F}'_\varepsilon(U_k),V_k\rangle| + 
|\langle \mathcal{F}'_\varepsilon(U_k),V_k\rangle| \le  
C\,\varepsilon^{\alpha/\gamma} +|\langle \mathcal{F}'_\varepsilon(U_k),V_k\rangle|,
\end{eqnarray*}
thanks to \eqref{boundV}. 
Thus, from \eqref{bbbb} and assumption (ii) in Theorem \ref{PSCthm} 
we get the desired claim in \eqref{fprimeV}. 

Analogously (but making use of \eqref{boundW}), one can see that  
\begin{equation}\label{fprimeW}
|\langle \mathcal{F}'_\varepsilon(W_k),W_k\rangle|\leq C\eps^{\alpha/\gamma}+o_k(1),
\end{equation}
Let us consider first the case $n\geq 6s$. From now on, we divide the proof in three main steps: 
we first show lower bounds for $\mathcal{F}_\varepsilon(V_k)$ 
and $\mathcal{F}_\varepsilon(W_k)$ (see Step 1 and Step 2, respectively), 
and then in Step 3 we obtain a lower bound for $\mathcal{F}_\varepsilon(U_k)$, 
which will give a contradiction with the hypotheses on $\mathcal{F}_\varepsilon$, 
and so the conclusion of Lemma \ref{tightness}.

\medskip

\noindent {\it Step 1: Lower bound for $\mathcal{F}_\varepsilon(V_k)$.} 
Recalling  that 
$$ \frac12-\frac{1}{\cx}=\frac{s}{n}$$ we have by Proposition \ref{qineq} that
\begin{equation*}\begin{split}
\mathcal{F}_\varepsilon(V_k)&-\frac{1}{2}\langle \mathcal{F}'_\varepsilon(V_k), V_k\rangle 
= \left(\frac{1}{2}-\frac{1}{\cx}\right)\|(V_k)_+(\cdot,0)\|^{\cx}\\
&\;\;+\varepsilon\left(\frac{1}{2}-\frac{1}{q+1}\right)\int_{\R^n}h(x)(V_k)_+^{q+1}(x,0)\,dx\\
&\geq \frac{s}n \|(V_k)_+(\cdot,0)\|^{\cx} - \varepsilon\left(\frac{1}{2}-\frac{1}{q+1}\right)\|h\|_{L^m(\Rn)} \|(V_k)_+(\cdot,0)\|^{q+1},
\end{split}\end{equation*}
and by Proposition \ref{proppqs} and \eqref{fprimeV} we get that
\begin{equation}\label{LowerBoundV}
\mathcal{F}_\varepsilon(V_k)\geq -C\eps^{\alpha/\gamma}- \overline C \varepsilon^{\frac{\cx}{\cx-(q+1)}}+o_k(1).
\end{equation}
\\

\noindent {\it Step 2: Lower bound for $\mathcal{F}_\varepsilon(W_k)$.} 
First of all, by the definition of $W_k$ in \eqref{3.4ter} \textcolor{black}{(recall that $W_k$ is supported in $\Rn \setminus B_{r+\overline l}\subset \Rn\setminus B_{\bar r_{\varepsilon}}$, using also \eqref{chooser})}, by Proposition \ref{qineq} and \ref{traceIneq}, using \eqref{heps} and \eqref{bbbb}, we have that 
\begin{equation}\begin{split}\label{upBoundWq}
\,& \bigg|\varepsilon\int_{\mathbb{R}^n}{ h(x)(W_k)_+^{q+1}(x,0)\,dx }\bigg|
\leq \varepsilon\int_{\mathbb{R}^n\setminus B_{\bar{r}_\varepsilon}}{ |h(x)|(W_k)_+^{q+1}(x,0)\,dx }\\
\leq\,& \varepsilon\, \|h\|_{L^m(\mathbb{R}^n\setminus B_{\bar{r}_\varepsilon})}
\|(W_k)_+(\cdot,0)\|^{q+1}
\leq \varepsilon\, C\, \|h\|_{L^m(\mathbb{R}^n\setminus B_{\bar{r}_\varepsilon})}[W_k]_a^{q+1}
\leq C\eps^{1+\beta},
\end{split}\end{equation}
where $1+\beta>\alpha/\gamma$. Thus, from \eqref{fprimeW} 
% from \eqref{boundW}????} 
we get that
\begin{equation}\begin{split}\label{Wbound}
&\bigg|\int_{\mathbb{R}^{n+1}_+} {y^a|\nabla W_k|^2\,dX}
-\int_{\mathbb{R}^n}  {(W_k)_+^{\cx} (x,0)\,dx} \bigg|
\\ &\qquad\le \left|\langle \mathcal{F}'_\varepsilon(W_k),W_k\rangle\right| + 
\left|\varepsilon\int_{\mathbb{R}^n}{ h(x) (W_k)_+^{q+1}(x,0)\,dx}\right|\\
&\qquad \leq C\varepsilon^{\alpha/\gamma} + o_k(1).
\end{split}\end{equation}
%where~was also used in the last passage. 
Moreover, notice that $W_k=U_k$ in $\R^{n+1}_+\setminus B_{r+\overline{l}+1}$ 
(recall \eqref{3.4bis} and \eqref{3.4ter}). 
Hence, using \eqref{contrad} with $\rho:=r+\overline{l}+1$, we get 
\begin{equation}\begin{split}\label{espero}
&\int_{\mathbb{R}^{n+1}_+\setminus B^+_{r+\bar{l}+1}}{y^a|\nabla W_k|^2\,dX}
+\int_{\mathbb{R}^n\setminus\{B_{r+\bar{l}+1}\cap\{y=0\}\}}{(W_k)_+^{2^{\star}_s}(x,0)\,dx}\\
&\qquad =\int_{\mathbb{R}^{n+1}_+\setminus B^+_{r+\bar{l}+1}}{y^a|\nabla U_k|^2\,dX}
+\int_{\mathbb{R}^n\setminus\{B_{r+\bar{l}+1}\cap\{y=0\}\}}
{(U_k)_+^{2^{\star}_s}(x,0)\,dx}
\geq \eta_0,
\end{split}\end{equation}
for $k=k(\rho)$. 
We observe that $k$ tends to $+\infty$ as $\varepsilon\to 0$, 
thanks to \eqref{forse0} and \eqref{eps to zero}. 

From \eqref{espero} we obtain that either 
$$ \int_{\mathbb{R}^n\setminus\{ B_{r+\bar{l}+1} \cap\{y=0\}\} }
{(W_k)_+^{2^{\star}_s}(x,0)\,dx}
\ge\frac{\eta_0}{2}$$
or
$$ \int_{\mathbb{R}^{n+1}_+\setminus B^+_{r+\bar{l}+1}}{y^a|\nabla W_k|^2\,dX}
\ge\frac{\eta_0}{2}.
$$
In the first case, we get that 
$$\int_{\mathbb{R}^n} { (W_k)_+^{2^{\star}_s}(x,0)\,dx}\ge \int_{\mathbb{R}^n\setminus\{B_{r+\bar{l}+1}\cap\{y=0\}\}}
{(W_k)_+^{2^{\star}_s}(x,0)\,dx}\ge 
\frac{\eta_0}{2}.$$
In the second case, taking $\varepsilon$ small (and so $k$ large enough), by \eqref{Wbound} 
we obtain that
\begin{equation*}\begin{split}
\int_{\mathbb{R}^n}{(W_k)_+^{\cx}(x,0)\,dx} & \geq 
\int_{\mathbb{R}^{n+1}_+}{y^a|\nabla W_k|^2\,dX}-C\varepsilon^{\alpha/\gamma}-o_k(1)
\\ &\geq \int_{\mathbb{R}^{n+1}_+\setminus B^+_{r+\bar{l}+1}}
{y^a|\nabla W_k|^2\,dX}-C\varepsilon^{\alpha/\gamma}-o_k(1)>\frac{\eta_0}{4}.
\end{split}\end{equation*}
Hence, in both cases we have that 
\begin{equation}\label{lowBoundWp}
\int_{\mathbb{R}^n}{(W_k)_+^{\cx}(x,0)\,dx} >\frac{\eta_0}{4}
\end{equation}
for $\varepsilon$ small and $k$ large enough. We now define $\psi_k:=\alpha_kW_k$, with
$$ \alpha_k^{\cx-2}:=\frac{[W_k]_a^2}{\|(W_k)_+(\cdot,0)\|^{\cx}}.$$
Notice that from \eqref{fprimeW} we have that
\begin{eqnarray*}
[W_k]_a^2 &\le & \|(W_k)_+(\cdot,0)\|^{\cx}
+\left|\varepsilon\int_{\R^n}h(x)(W_k)_+^{q+1}(x,0)\,dx\right|
+C\,\varepsilon^{\alpha/\gamma} +o_k(1)\\
&\le & \|(W_k)_+(\cdot,0)\|^{\cx}
+C\,\varepsilon^{\alpha/\gamma} +o_k(1),
\end{eqnarray*}
where \eqref{upBoundWq} was used in the last line. 
Hence, thanks to \eqref{lowBoundWp}, we get that 
\begin{equation}\label{star-1}
\alpha_k^{\cx-2}\leq 1+C\varepsilon^{\alpha/\gamma}+o_k(1).\end{equation}
Also, we notice that for this value of $\alpha_k$, we have the following chain of identities,
$$[\psi_k]^2_a=\alpha_k^2[W_k]_a^2=\alpha^{\cx}_k
\|(W_k)_+(\cdot,0)\|^{\cx}
=\|(\psi_k)_+(\cdot,0)\|^{\cx}.$$
Thus, by Proposition \ref{traceIneq} and~\eqref{isom}, we obtain 
\begin{eqnarray*}
&& S\leq \frac{[\psi_k(\cdot,0)]^2_{\dot{H}^s(\mathbb{R}^n)}}
{\|(\psi_k)_+(\cdot,0)\|^2}
=\frac{[\psi_k]_a^2}{\|(\psi_k)_+(\cdot,0)\|^{2}}=\frac{\|(\psi_k)_+(\cdot,0)\|^{\cx}}
{\|(\psi_k)_+(\cdot,0)\|^{2}}
=\|(\psi_k)_+(\cdot,0)\|^{\frac{4s}{n-2s}}.
\end{eqnarray*}
Consequently,
$$\|(W_k)_+(\cdot,0)\|^{\cx}
=\frac{\|(\psi_k)_+(\cdot,0)\|^{\cx}}{\alpha_k^{\cx}}
\geq S^{n/2s}\frac{1}{\alpha_k^{\cx}}.$$
This, together with \eqref{star-1}, gives that 
\begin{equation}\begin{split}\label{swkbla}
S^{n/2s}\leq\;&(1+C\varepsilon^{\alpha/\gamma}+o_k(1))^{\frac{\cx}{\cx-2}}
\|(W_k)_+(\cdot,0)\|^{\cx}\\
\leq\;& \|(W_k)_+(\cdot,0)\|^{\cx}+C\varepsilon^{\alpha/\gamma}+o_k(1).
\end{split}\end{equation}
%Also, recalling  that 
%$$ \frac12-\frac{1}{\cx}=\frac{s}{n}$$
We get that 
\begin{eqnarray*}
\mathcal{F}_\varepsilon(W_k)-\frac{1}{2}\langle \mathcal{F}'_\varepsilon(W_k),W_k\rangle
&=&\frac{s}{n}\|(W_k)_+(\cdot,0)\|^{\cx}\\
&&\qquad +\varepsilon\left(\frac{1}{2}-\frac{1}{q+1}\right)
\int_{\mathbb{R}^n}{h(x)(W_k)_+^{q+1}(x,0)\,dx}\\
&\geq&\frac{s}{n}S^{n/2s} -C\varepsilon^{\beta+1}-C\varepsilon^{\alpha/\gamma}+o_k(1),
\end{eqnarray*}
where we have used \eqref{upBoundWq} to estimate the $(q+1)$-order term.
Finally, using also \eqref{fprimeW} and the fact that $\beta+1>\alpha/\gamma$, we get 
\begin{equation}\label{LowBoundFW}
\mathcal{F}_\varepsilon(W_k)\geq \frac{s}{n}S^{n/2s}-C\varepsilon^{\alpha/\gamma}+o_k(1).
\end{equation}
\\

\noindent {\it Step 3: Lower bound for $\mathcal{F}_\varepsilon(U_k)$.} 
We first observe that by definition
we can write 
\begin{equation}\label{adwetperigyrejh}
U_k=(1-\chi)U_k+\chi U_k=W_k+V_k.\end{equation} 
Therefore
\begin{equation}\begin{split}\label{sumF}
\mathcal{F}_\varepsilon(U_k) = &\, \mathcal{F}_\varepsilon(V_k)+\mathcal{F}_\varepsilon(W_k)
+\int_{\mathbb{R}^{n+1}_+}{y^a\langle\nabla V_k,\nabla W_k\rangle\,dX} \\
&\quad +\frac{1}{\cx}\int_{\mathbb{R}^n}{(V_k)_+^{\cx}(x,0)\,dx}
+\frac{\varepsilon}{q+1}\int_{\mathbb{R}^n}{h(x)(V_k)_+^{q+1}(x,0)\,dx}\\
&\quad +\frac{1}{\cx}\int_{\mathbb{R}^n}{(W_k)_+^{\cx}(x,0)\,dx}
+\frac{\varepsilon}{q+1}\int_{\mathbb{R}^n}{h(x)(W_k)_+^{q+1}(x,0)\,dx}\\
&\quad -\frac{1}{\cx}\int_{\mathbb{R}^n}{(U_k)_+^{\cx}(x,0)\,dx}
-\frac{\varepsilon}{q+1}\int_{\mathbb{R}^n}{h(x)(U_k)_+^{q+1}(x,0)\,dx}.
\end{split}\end{equation}
On the other hand,
\begin{eqnarray*}
&& \int_{\mathbb{R}^{n+1}_+}{y^a\langle\nabla V_k,\nabla W_k\rangle\,dX} \\
&&\qquad = \frac{1}{2}\int_{\mathbb{R}^{n+1}_+}{y^a\langle\nabla U_k-\nabla V_k,\nabla V_k\rangle\,dX} +\frac{1}{2}\int_{\mathbb{R}^{n+1}_+}{y^a\langle\nabla U_k-\nabla W_k,\nabla W_k\rangle\,dX}.
\end{eqnarray*}
Also
\begin{eqnarray*}
&&\langle \mathcal{F}_\varepsilon'(U_k)-\mathcal{F}_\varepsilon'(V_k),V_k
\rangle \\&=& 
\int_{\R^{n+1}_+}y^a\langle\nabla U_k-\nabla V_k, \nabla V_k\rangle\,dX \\
&&\qquad - \varepsilon \int_{\R^n}h(x)(U_k)_+^q(x,0)\,V_k(x,0)\,dx - 
\int_{\R^n}(U_k)_+^{\cx-1}(x,0)\,V_k(x,0)\,dx\\
&&\qquad +\varepsilon \int_{\R^n}h(x)(V_k)_+^{q+1}(x,0)\,dx + 
\int_{\R^n}(V_k)_+^{\cx}(x,0)\,dx,
\end{eqnarray*}
and 
\begin{eqnarray*}
&&\langle \mathcal{F}_\varepsilon'(U_k)-\mathcal{F}_\varepsilon'(W_k),W_k
\rangle \\&=& 
\int_{\R^{n+1}_+}y^a\langle\nabla U_k-\nabla W_k, \nabla W_k\rangle\,dX \\
&&\qquad - \varepsilon \int_{\R^n}h(x)(U_k)_+^q(x,0)\,W_k(x,0)\,dx - 
\int_{\R^n}(U_k)_+^{\cx-1}(x,0)\,W_k(x,0)\,dx\\
&&\qquad +\varepsilon\int_{\R^n}h(x)(W_k)_+^{q+1}(x,0)\,dx + 
\int_{\R^n}(W_k)_+^{\cx}(x,0)\,dx.
\end{eqnarray*}
Hence, plugging the three formulas above into \eqref{sumF} we get
\begin{equation*}\begin{split}
\mathcal{F}_\varepsilon(U_k) = &\, \mathcal{F}_\varepsilon(V_k)+\mathcal{F}_\varepsilon(W_k)+\frac12
\langle \mathcal{F}_\varepsilon'(U_k)-\mathcal{F}_\varepsilon'(V_k),V_k
\rangle +\frac12\langle \mathcal{F}_\varepsilon'(U_k)-\mathcal{F}_\varepsilon'(W_k),W_k
\rangle \\
&\quad +\frac{1}{\cx}\int_{\mathbb{R}^n}{(V_k)_+^{\cx}(x,0)\,dx}
+\frac{\varepsilon}{q+1}\int_{\mathbb{R}^n}{h(x)(V_k)_+^{q+1}(x,0)\,dx}\\
&\quad +\frac{1}{\cx}\int_{\mathbb{R}^n}{(W_k)_+^{\cx}(x,0)\,dx}
+\frac{\varepsilon}{q+1}\int_{\mathbb{R}^n}{h(x)(W_k)_+^{q+1}(x,0)\,dx}\\
&\quad -\frac{1}{\cx}\int_{\mathbb{R}^n}{(U_k)_+^{\cx}(x,0)\,dx}
-\frac{\varepsilon}{q+1}\int_{\mathbb{R}^n}{h(x)(U_k)_+^{q+1}(x,0)\,dx}\\
&\quad +\frac{\varepsilon}{2}
\int_{\R^n}h(x)(U_k)_+^q(x,0)\,V_k(x,0)\,dx + 
\frac12 \int_{\R^n}(U_k)_+^{2^{\star}_s-1}(x,0)\,V_k(x,0)\,dx\\
&\quad -\frac{\varepsilon}{2}\int_{\R^n}h(x)(V_k)_+^{q+1}(x,0)\,dx - 
\frac12\int_{\R^n}(V_k)_+^{\cx}(x,0)\,dx\\
&\quad + \frac{\varepsilon}{2} \int_{\R^n}h(x)(U_k)_+^q(x,0)
\,W_k(x,0)\,dx 
+\frac12 \int_{\R^n}(U_k)_+^{\cx-1}(x,0)\,W_k(x,0)\,dx\\
&\quad -\frac{\varepsilon}{2}\int_{\R^n}h(x)(W_k)_+^{q+1}(x,0)\,dx - \frac12
\int_{\R^n}(W_k)_+^{\cx}(x,0)\,dx.
\end{split}\end{equation*}
Therefore, using~\eqref{boundV} and~\eqref{boundW} we obtain that 
\begin{equation*}\begin{split}
\mathcal{F}_\varepsilon(U_k) \ge &\, \mathcal{F}_\varepsilon(V_k)+\mathcal{F}_\varepsilon(W_k)\\
&\quad +\frac{1}{\cx}\int_{\mathbb{R}^n}{(V_k)_+^{\cx}(x,0)\,dx}
+\frac{1}{\cx}\int_{\mathbb{R}^n}{(W_k)_+^{\cx}(x,0)\,dx}  -\frac{1}{\cx}\int_{\mathbb{R}^n}{(U_k)_+^{\cx}(x,0)\,dx}\\
&\quad +\frac12 \int_{\R^n}(U_k)_+^{\cx-1}(x,0)\,V_k(x,0)\,dx  +\frac12 \int_{\R^n}(U_k)_+^{\cx-1}(x,0)\,W_k(x,0)\,dx  \\
&\quad - \frac12\int_{\R^n}(V_k)_+^{\cx}(x,0)\,dx
 - \frac12 \int_{\R^n}(W_k)_+^{\cx}(x,0)\,dx  \\
 &\quad -\varepsilon\left(\frac{1}{2}-\frac{1}{q+1}\right)\int_{\mathbb{R}^n}{h(x)(W_k)_+^{q+1}(x,0)\,dx}\\
 &\quad-\varepsilon\left(\frac{1}{2}-\frac{1}{q+1}\right)\int_{\mathbb{R}^n}{h(x)(V_k)_+^{q+1}(x,0)\,dx}\\
 &\quad - \frac{\varepsilon}{q+1} \int_{\Rn} h(x) \ukp^{q+1}(x,0)\, dx + \frac{\varepsilon}2 \int_{\Rn} h(x)\ukp^q(x,0)V_k(x,0)\, dx \\  
&\quad + \frac{\varepsilon}2 \int_{\Rn} h(x)\ukp^q(x,0)W_k(x,0)\, dx  -C\varepsilon^{\alpha/\gamma},
\end{split}\end{equation*}
for some positive~$C$. We use identity~\eqref{adwetperigyrejh} to write
\[ \ukp^{\cx-1}(V_k+W_k)=\ukp^{\cx}\quad \mbox{ and }\quad \ukp^{q+1}=\ukp^q(V_k+W_k),\]
%\begin{eqnarray*}
%&&\int_{\R^n}(U_k)_+^{\cx-1}(x,0)\,V_k(x,0)\,dx + \int_{\R^n}(U_k)_+^{\cx-1}(x,0)\,W_k(x,0)\,dx\\
%&=& \int_{\R^n}(U_k)_+^{\cx-1}(x,0)\,\big(V_k(x,0)+W_k(x,0)\big)\,dx\\
%&=&\int_{\R^n}(U_k)_+^{\cx}(x,0)\,dx.
%\end{eqnarray*}
and obtain that
\eqlab{\label{qwqweteyryhg}
\mathcal{F}_\varepsilon(U_k) \ge &\, \mathcal{F}_\varepsilon(V_k)+\mathcal{F}_\varepsilon(W_k)\\ 
&\quad+\bigg( \frac12-\frac{1}{\cx}\bigg) \lrq{ \int_{\mathbb{R}^n}(U_k)_+^{\cx}(x,0)\,dx -\int_{\mathbb{R}^n}(V_k)_+^{\cx}(x,0)\,dx 
-\int_{\mathbb{R}^n}(W_k)_+^{\cx}(x,0)\,dx }\\ 
&\quad +  \varepsilon\lr{ \frac{1}{2}-\frac1{q+1} }\int_{\mathbb{R}^n}{h(x)\left[(U_k)_+^{q}(x,0)V_k(x,0) - (V_k)_+^{q+1}(x,0)\right]\,dx} \\ 
&\quad + \varepsilon\lr{ \frac{1}{2}-\frac1{q+1} }\int_{\mathbb{R}^n}{h(x)\left[(U_k)_+^{q}(x,0)W_k(x,0) - (W_k)_+^{q+1}(x,0)\right]\,dx} -C\varepsilon^{\alpha/\gamma}.
} 
Using \eqref{ealpha}, reasoning in the same way for the term with $W_k$ \textcolor{black}{and recalling that $1+\alpha>\alpha/\gamma$ we get that} 
\begin{equation*}\begin{split}
\mathcal{F}_\varepsilon(U_k) \ge  &\, \mathcal{F}_\varepsilon(V_k)+\mathcal{F}_\varepsilon(W_k)\\
&\quad +\frac{s}n \int_{\mathbb{R}^n}\left(
(U_k)_+^{2^{\star}_s}(x,0)-(V_k)_+^{2^{\star}_s}(x,0)-(W_k)_+^{2^{\star}_s}(x,0)\right)\,dx-C\varepsilon^{\alpha/\gamma}\\
= &\, \mathcal{F}_\varepsilon(V_k)+\mathcal{F}_\varepsilon(W_k)\\
&\quad +\frac{s}n \int_{\mathbb{R}^n}(U_k)_+^{2^{\star}_s}(x,0)\left(1-\chi^{2^{\star}_s}(x,0)-(1-\chi(x,0))^{2^{\star}_s}\right)\,dx-C\varepsilon^{\alpha/\gamma},
\end{split}\end{equation*}
where~\eqref{3.4ter} was used in the last line. 
Also, since $\cx>2$ and 
\begin{equation}\label{chipos} 1-\chi^{\cx}(x,0)-(1-\chi(x,0))^{\cx}\ge0 \quad {\mbox{ for any }}x\in\R^n,
\end{equation}
we get
\begin{equation*}\begin{split}
\mathcal{F}_\varepsilon(U_k) \ge  &\mathcal{F}_\varepsilon(V_k)+\mathcal{F}_\varepsilon(W_k) -C\eps^{\alpha/\gamma}. 
\end{split}\end{equation*}
%\mathcal{F}_\varepsilon(U_k) \ge  &\mathcal{F}_\varepsilon(V_k)+\mathcal{F}_\varepsilon(W_k)\textcolor{black}{-\varepsilon\left(\frac{1}{2}-\frac{1}{q+1}\right)\int_{\mathbb{R}^n}{h(x)(W_k)_+^{q+1}(x,0)\,dx}}\\
% &\textcolor{black}{-\varepsilon\left(\frac{1}{2}-\frac{1}{q+1}\right)\int_{\mathbb{R}^n}{h(x)(V_k)_+^{q+1}(x,0)\,dx}}-C\textcolor{black}{\varepsilon^{\alpha/\gamma}}.
This, \eqref{LowerBoundV} and \eqref{LowBoundFW} imply that 
\begin{equation*}
\mathcal{F}_\varepsilon(U_k)\geq 
\frac{s}{n}S^{n/2s}-c_1\eps^{\alpha/\gamma}- \overline C \varepsilon^{\frac{\cx}{\cx-(q+1)}}+o_k(1).
\end{equation*}
Hence, taking the limit as $k\to+\infty$ we obtain that 
$$ c_\varepsilon=\lim_{k\to+\infty}\mathcal{F}_\varepsilon(U_k)\geq 
\frac{s}{n}S^{n/2s}-c_1\eps^{\alpha/\gamma}  -\overline C \varepsilon^{\frac{\cx}{\cx-(q+1)}},$$
which is a contradiction with assumption (i) of Theorem \ref{PSCthm}. 
This concludes the proof of Lemma \ref{tightness} in the case $n\geq 6s$.
\medskip

Consider now $n\in (2s,6s)$. In such a case, 
%using the positivity of $h$, 
one easily sees that
\begin{equation*}\begin{split}
\mathcal{F}_\varepsilon(V_k)&-\frac{1}{2}\langle \mathcal{F}'_\varepsilon(V_k), V_k\rangle 
= \left(\frac{1}{2}-\frac{1}{\cx}\right)\|(V_k)_+(\cdot,0)\|^{\cx}\\
&\;\;+\varepsilon\left(\frac{1}{2}-\frac{1}{q+1}\right)\int_{\R^n}h(x)(V_k)_+^{q+1}(x,0)\,dx\\
&\geq \varepsilon\left(\frac{1}{2}-\frac{1}{q+1}\right)\int_{\R^n}h(x)(V_k)_+^{q+1}(x,0)\,dx,
\end{split}\end{equation*}
and by \eqref{fprimeV} we get that
\begin{equation}\label{LowerBoundV2}
\mathcal{F}_\varepsilon(V_k)\geq \varepsilon\left(\frac{1}{2}-\frac{1}{q+1}\right)\int_{\R^n}h(x)(V_k)_+^{q+1}(x,0)\,dx-C\eps^{\alpha/\gamma}+o_k(1).
\end{equation}
On the other hand, proceeding analogously to the previous case \textcolor{black}{(check \eqref{swkbla})}, we obtain
\begin{eqnarray*}
\mathcal{F}_\varepsilon(W_k)-\frac{1}{2}\langle \mathcal{F}'_\varepsilon(W_k),W_k\rangle
&=&\frac{s}{n}\|(W_k)_+(\cdot,0)\|_{L^{\cx}(\mathbb{R}^n)}^{\cx}\\
&&\qquad +\varepsilon\left(\frac{1}{2}-\frac{1}{q+1}\right)
\int_{\mathbb{R}^n}{h(x)(W_k)_+^{q+1}(x,0)\,dx}\\
&\geq&\frac{s}{n}S^{n/2s}+\varepsilon\left(\frac{1}{2}-\frac{1}{q+1}\right)
\int_{\mathbb{R}^n}{h(x)(W_k)_+^{q+1}(x,0)\,dx}\\
&&\qquad -C\varepsilon^{\alpha/\gamma}+o_k(1).
\end{eqnarray*}
Thus, using also \eqref{fprimeW}, we get 
\begin{equation}\label{LowBoundFW2}
\mathcal{F}_\varepsilon(W_k)\geq \frac{s}{n}S^{n/2s}+\varepsilon\left(\frac{1}{2}-\frac{1}{q+1}\right)
\int_{\mathbb{R}^n}{h(x)(W_k)_+^{q+1}(x,0)\,dx}-C\varepsilon^{\alpha/\gamma}+o_k(1).
\end{equation}
Now, using the positivity of $h$, from \eqref{qwqweteyryhg} and \eqref{chipos} we get
\bgs{
\mathcal{F}_\varepsilon(U_k) \ge &\, \mathcal{F}_\varepsilon(V_k)+\mathcal{F}_\varepsilon(W_k)\\ 
&\quad+\bigg(\frac12 - \frac{1}{\cx}\bigg) \lrq{\int_{\mathbb{R}^n}(U_k)_+^{\cx}(x,0)\,dx  -\int_{\mathbb{R}^n}(V_k)_+^{\cx}(x,0)\,dx 
-\int_{\mathbb{R}^n}(W_k)_+^{\cx}(x,0)\,dx 
 }\\ 
&\quad -  \varepsilon\lr{ \frac{1}{2}-\frac1{q+1} }\int_{\mathbb{R}^n}{h(x)\left[(V_k)_+^{q+1}(x,0)+(W_k)_+^{q+1}(x,0)\right]\,dx}  -C\varepsilon^{\alpha/\gamma}\\
\ge &\, \mathcal{F}_\varepsilon(V_k)+\mathcal{F}_\varepsilon(W_k)\\ 
&\quad -  \varepsilon\lr{ \frac{1}{2}-\frac1{q+1} }\int_{\mathbb{R}^n}{h(x)\left[(V_k)_+^{q+1}(x,0)+(W_k)_+^{q+1}(x,0)\right]\,dx}  -C\varepsilon^{\alpha/\gamma},\\
\ge &\, \frac{s}{n}S^{n/2s}-C\varepsilon^{\alpha/\gamma}+o_k(1),
} 
where we have used \eqref{LowerBoundV2} and \eqref{LowBoundFW2} in the last line. Passing to the limit as $k\rightarrow\infty$ we reach a contradiction with assumption (i) of Theorem \ref{PSCthm} and thus we finish the proof of Lemma \ref{tightness} in the case $n\in (2s,6s)$.
\end{proof}

\noindent Knowing that the sequence $\{U_k\}_{k\in \N}$ is bounded and tight, one can use the Concentration Compactness principle and prove Theorem \ref{PSCthm}. More precisely, one applies Proposition \ref{CCP} for the positive sequence $\{\ukp\}_{k\in \N}$, which is also bounded and tight, to obtain that
\begin{equation*}\begin{split}
\ukp^{2^{\star}_s}(\cdot,0)&\xrightarrow[k\to +\infty]{} \nu=\overline{U}^{2^{\star}_s}(\cdot,0)+\sum \nu_j\delta_{x_j},\\
y^a|\nabla \ukp|^2&\xrightarrow[k\to +\infty] {}\mu\geq y^a|\nabla\overline{U}|^2+\sum\mu_j\delta_{(x_j,0)},
\end{split}\end{equation*}
and then, following the steps in \cite[Proof of Proposition 4.2.1]{maria} and using Proposition \ref{proppqs}, one deduces $\nu_j=\mu_j=0$ for every $j$. Finally, proceeding as in \cite[Proposition 4.2.1]{maria} (using Proposition \ref{convres} instead of  \cite[Lemma 4.1.1]{maria}) the strong convergence in $\dot{H}^s_a(\mathbb{R}^{n+1}_+)$ follows, and thus Theorem \ref{PSCthm} holds. 
\subsection{Bound on the minmax value and geometry of the functional}\label{minmax}
\noindent The purpose of this subsection is to show that the minmax value of the Mountain Pass Lemma lies below the critical threshold given in Theorem \ref{PSCthm}. To see this, the idea is to find a path where the maximum value of the functional is smaller than this critical level (and so the minmax). We obtain such path by working with the fractional Sobolev minimizers, explicitly computed in formula \eqref{SobMin}. 
One considers, as done in \cite[Section 6.5]{maria}, the ball $B$ given in \eqref{h1}  and takes $\mu_0 >0$ and $\xi \in \Rn$  to be the radius and the center of $B$ respectively. Namely, one has that 
\[\inf_{B_{\mu_0}(\xi)} h >0.\] 
Let $\bar \phi \in C_0^{\infty}(B_{\mu_0}(\xi),[0,1])$ be a cut-off function such that $\bar \phi(x) = 1$ in $B_{\frac{\mu_0}2}(\xi)$. Translating and rescaling the function $z$ in \eqref{SobMin} we define 
\eqlab{\label{zmmu} z_{\mu,\xi}(x):=\mu^{\frac{2s-n}2}z\lr{ \frac{x-\xi}{\mu}},\qquad \mu>0.} 
Let $\bar Z_{\mu,\xi}$ be the extension of $\bar \phi z_{\mu,\xi}$, as defined in  \eqref{lapextop}. With some manipulations (check Section \textcolor{black}{6.5} in \cite{maria}), one has that
\eqlab{\label{zsmal1} \|z_{\mu,\xi}\|^2 =S^{\frac{n-2s}{4s}}} 
and that
\eqlab{\label{zbar1} [\bar Z_{\mu,\xi} ]_a^2 = [\bar \phi z_{\mu,\xi} ]^2_{\dot H^s(\Rn)} \leq S^{\frac{n}{2s}}+C\mu^{n-2s}.}
Moreover, we have the following result.
\begin{lemma} There exists $C=C(n,s,\mu_0)>0$ such that
\label{lemzm} \[\|\bar Z_{\mu, \xi} (\cdot, 0)\|^{\cx} \geq \|z_{\mu,\xi}\|^{\cx} -C\mu^n.\]
\end{lemma}
\begin{proof}
In the next computations, the constant may change value from line to line. Using that  $\bar\phi=1$ on $B_{{\mu_0}/2}(\xi)$, we have that
\bgs{ \|z_{\mu,\xi}\|^{\cx}-\|\bar Z_{\mu,\xi}(\cdot,0)\|^{\cx} =&\; \int_{\Rn}(1-\bar \phi^{\cx})z_{\mu,\xi}^{\cx}\, dx   = \int_{\Rn \setminus B_{\frac{\mu_0}2}(\xi) } (1-\bar \phi^{\cx})z_{\mu,\xi}^{\cx}\, dx \\
   \leq &\; \mu^{-n} \int_{\Rn \setminus B_{\frac{\mu_0}2}(\xi)} z^{\cx}\lr{\frac{x-\xi}\mu} \, dx .}
%   since $\phi=1$ on $B_{\frac{\mu_0}/2}(\xi)$. 
   Making the change of variable $y= (x-\xi)/\mu$ and inserting definition \eqref{SobMin} we get
   \bgs{ \|z_{\mu,\xi}\|^{\cx}-\|\bar Z_{\mu,\xi}(\cdot,0)\|^{\cx}   = \int_{\Rn\setminus B_{\frac{\mu_0}{2\mu}}} z^{\cx}(y) \, dy 
    \leq c_{\star}^{\cx}\int_{\Rn\setminus B_{\frac{\mu_0}{2\mu}}} |y|^{-2n}\, dy \leq C \mu^n,} where $C$ depends on $n,s, \mu_0$.
   This proves the lemma.
\end{proof}

\noindent Let $t> 0$. We consider the path $t\bar Z_{\mu,\xi}$ and compute the energy along it. Namely, we focus on  obtaining an upper bound for 
\[ \Fl (t\bar Z_{\mu,\xi}) = \frac{t^2}2 [\bar Z_{\mu,\xi}]_a^2 - \frac{t^{\cx}}{\cx} \|\bar Z_{\mu,\xi}(\cdot,0)\|^{\cx} - \frac{\varepsilon}{q+1}\int_{\Rn} h(x) \lr{ t\bar Z_{\mu,\xi}(x,0)}^{q+1}\, dx  \] \textcolor{black}{and proving that it stays below the critical threshold  given in Theorem \ref{PSCthm}. Of course, if $t=0$ the energy level is zero, and  for $\varepsilon$ small enough, this is trivially fulfilled.}
We introduce  the following Lemmata.
\begin{lemma}\label{lemqp1}
Let $n>\frac{2s(q+3)}{q+1}$. There exists $\mu^{\star}<\mu_0/2$ such that for any $t> 0$ and any $\mu\in(0,\mu^\star)$  
\bgs{\label{cl3} \int_{\Rn} h(x) \lr{ t \bar Z_{\mu,\xi}(x,0)}^{q+1} \, dx \geq C\; t^{q+1}\;  \mu^{\frac{(2s-n)(q+1)}2+n}  ,}
where $C=C(n,s,h,\mu_0,\mu^\star)$ is a positive constant.\end{lemma}
\begin{proof}
Notice that since $q<\cx-1$, we have that $ {\frac{(2s-n)(q+1)}2+n} >0$. 
 \textcolor{black}{Given the definition of $\bar \phi$
 % vanishes outside ${B_{\mu_0}(\xi)}$ and is equal to $1$ in  ${B_{\mu_0/2}(\xi)}$, 
 we have that}
\eqlab{ \label{bla12}
&\; \int_{\Rn} h(x) \lr{ t \bar Z_{\mu,\xi}(x,0)}^{q+1} \, dx  
	= \int_{\Rn} h(x) t^{q+1} \bar \phi^{q+1} z_{\mu,\xi}^{q+1} (x)\, dx \\
=&\;\int_{B_{\mu_0}(\xi)}h(x) t^{q+1} \bar \phi^{q+1} z_{\mu,\xi}^{q+1} (x)\, dx  
\geq  t^{q+1} \inf_{B_{\mu_0}(\xi)} h  \, \int_{B_{\frac{\mu_0}2}(\xi)} z_{\mu,\xi}^{q+1}(x)\, dx,}
 recalling also that $\inf_{B_{\mu_0}(\xi)} h$ is positive. Thus, using \eqref{zmmu}, changing the variable $y=(x-\xi)/\mu$ and inserting definition \eqref{SobMin} we obtain that
\bgs{  \int_{B_{\frac{\mu_0}2}(\xi)} z^{q+1}_{\mu,\xi}(x)\, dx =&\; \mu^{\frac{(2s-n)(q+1)}2+n} \int_{B_{\frac{\mu_0}{2\mu}}} z^{q+1}(y)\, dy \\=&\;   \textcolor{black}{C}\mu^{\frac{(2s-n)(q+1)}2+n}  \int_{B_{\frac{\mu_0}{2\mu}}} (1+|y|^2)^{\frac{(2s-n)(q+1)}2} \, dy,}\textcolor{black}{where $C=C(n,s)>0$.}
Passing to polar coordinates and taking $\mu$ small enough, say $\mu<\mu_0/2$ we get that
\bgs{  &\;  \int_{B_{\frac{\mu_0}{2\mu}}}(1+|y|^2)^{\frac{(2s-n)(q+1)}2} \, dy\geq \textcolor{black}{ \int_{B_{\frac{\mu_0}{2\mu}}}|y|^{(2s-n)(q+1)}\, dy} \\
%\int_0^{\frac{\mu_0}{2\mu}} (1+\rho^2)^{\frac{(2s-n)(q+1)}2}  \rho^{n-1} \, d\rho\\
 \geq   &\; c_{n,s} \int_1^{\frac{\mu_0}{2\mu}} \rho^{{(2s-n)(q+1)}}  \rho^{n-1} \, d\rho 
= c_{n,s} \frac{\lr{\frac{\mu_0}{2\mu}}^{(2s-n)(q+1)+n}-1}{(2s-n)(q+1)+n}   .}
We have that $(2s-n)(q+1) +n <0$ and renaming the constants we obtain that
\bgs{  \int_{B_{\frac{\mu_0}{2\mu}}}(1+|y|^2)^{\frac{(2s-n)(q+1)}2} \, dy  \geq c_{n,s} \lr{\mu_0^{(n-2s)(q+1)- n}-(2 \mu)^{(n-2s)(q+1)- n} } \geq  C_{n,s,\mu_0,\mu^\star}, }  for any  $\mu \in  (0,\mu^{\star})$ , $\mu^{\star}<\mu_0/2$, where $C_{n,s,\mu_0,\mu^\star}$ designates a positive constant. Hence 
\bgs{  \int_{B_{\frac{\mu_0}2}(\xi)} z^{q+1}_{\mu,\xi}(x)\, dx \geq  C_{n,s,\mu_0,\mu^\star} \mu^{\frac{(2s-n)(q+1)}2+n},}
and from \eqref{bla12} it follows
\bgs{ \int_{\Rn} h(x) \lr{ t \bar Z_{\mu,\xi}(x,0)}^{q+1} \, dx   \geq &\;C  t^{q+1} \mu^{\frac{(2s-n)(q+1)}2+n} ,}
where $C$ is a positive constant that depends on $n,s, h, \mu_0$ and $\mu^\star$. 
\end{proof}

\noindent Let $\mu^\star$ be fixed as in Lemma \ref{lemqp1}. We want to prove now that the energy level along the path induced by $t\bar Z_{\mu,\xi}$, $t>0$, stays below the critical threshold  given in Theorem \ref{PSCthm} for $\mu<\mu^{\star}$. With this purpose, we state the next result.
\begin{lemma}\label{haha}
There exists $\mu_1 \in(0, \mu_0)$ such that
\eqlab{ \label{cl1} \lim_{t\to +\infty} \sup_{\mu \in (0,\mu_1)}\Fl(t\bar Z_{\mu,\xi}) =-\infty.}
% In particular, there exists $T_1>0$ such that
%\[ \sup_{\mu\in(0,\mu_1)} \Fl(t\bar Z_{\mu,\xi}) \leq 0.\]
Furthermore, if $n>\frac{2s(q+3)}{q+1}$, for any $\mu \in ( 0, \min\{\mu^\star,\mu_1\})$
\eqlab{ \label{cl2} \sup_{t\geq 0} \Fl (t \bar Z_{\mu,\xi}) <\frac{s}n S^{\frac{n}{2s}} +C_1 \mu^{ n-2s} + o(\mu^{n-2s})  - C_3  \varepsilon\mu^{\frac{(2s-n)(q+1)}2+n},}
where $\mu^{\star}$ was given in Lemma \ref{lemqp1}.
\end{lemma}
\begin{proof}
%By Lemma \ref{lemqp1}, in particular we know

Thanks to \eqref{zbar1}, \eqref{zsmal1} and Lemma \ref{lemzm} we have that
\eqlab{ \label{bla11}   \frac{t^2}2 [\bar Z_{\mu,\xi}]^2_a - \frac{t^{\cx}}{\cx} \|\bar Z_{\mu,\xi}(\cdot,0)\|^{\cx}  \leq &\;\frac{t^2}2 \lr{S^{\frac{n}{2s}} +C_1\mu^{n-2s} } -\frac{t^{\cx}}{\cx} \lr{\|z_{\mu,\xi}\|^{\cx}  -C_2\mu^n} \\
\leq  &\; \frac{t^2}2 \lr{S^{\frac{n}{2s}} +C_1\mu^{n-2s} }  -  \frac{t^{\cx}}{\cx} \lr{S^{\frac{n}{2s}} -C_2\mu^n}.}
\textcolor{black}{From \eqref{bla12} it follows that for any $\mu\in(0,\mu^{\star})$}
\[\int_{\Rn} h(x) \lr{ t\bar Z_{\mu,\xi}(x,0)}^{q+1}\, dx  \geq 0,\] and therefore
\bgs{ \Fl (t\bar Z_{\mu,\xi}) = &\;\frac{t^2}2 [\bar Z_{\mu,\xi}]^2_a - \frac{t^{\cx}}{\cx} \|\bar Z_{\mu,\xi}(\cdot,0)\|^{\cx} - \frac{\varepsilon}{q+1}\int_{\Rn} h(x) \lr{ t\bar Z_{\mu,\xi}(x,0)}^{q+1}\, dx \\
 \leq&\; \frac{t^2}2 \lr{S^{\frac{n}{2s}} +C_1\mu^{n-2s} }  -  \frac{t^{\cx}}{\cx} \lr{S^{\frac{n}{2s}} -C_2\mu^n}.  } 
Now, there exists $\mu_1\in(0,\mu_0)$ small enough such that $S^{n/2s}-C_2\mu^n$ is positive and hence, sending $t $ to $+\infty$ and recalling that $\cx>2$, we obtain
\[ \lim_{t\to +\infty} \Fl (t\bar Z_{\mu,\xi})=-\infty\] for any $\mu \in (0,\mu_1)$. This proves \eqref{cl1}. 

To obtain \eqref{cl2}, we use \eqref{bla11} and taking any  $\mu \in ( 0, \min\{\mu^\star,\mu_1\})$, by Lemma \ref{lemqp1} we have that
\bgs{ \Fl (t\bar Z_{\mu,\xi})   \leq \frac{t^2}2 \lr{S^{\frac{n}{2s}} +C_1\mu^{n-2s} }  -  \frac{t^{\cx}}{\cx} \lr{S^{\frac{n}{2s}} -C_2\mu^n} -C_3 \varepsilon \frac{t^{q+1}}{q+1} \mu^{\frac{(2s-n)(q+1)}2+n} .} By renaming the constants, we obtain
\eqlab{\label{fltx}    \Fl (t\bar Z_{\mu,\xi}) \leq  S^{\frac{n}{2s}} g(t),}
where 
\[g(t):=  \frac{t^2}2 \lr{1 +C_1\mu^{n-2s} }  -  \frac{t^{\cx}}{\cx} \lr{1 -C_2\mu^n} - C_3\eps \frac{t^{q+1}}{q+1}   \mu^{\frac{(2s-n)(q+1)}2+n}   .\] We compute the first derivative of $g$ and have that
\eqlab{ \label{blaaa} g'(t)= t\lrq{ (1+C_1 \mu^{n-2s}) -t^{\cx-2} (1-C_2\mu^n) -C_3 \varepsilon t^{q-1} \mu^{\frac{(2s-n)(q+1)}2+n}} .}
Let 
$$f(t):= (1+C_1 \mu^{n-2s}) -t^{\cx-2} (1-C_2\mu^n)\quad  \hbox{ and } \quad h(t):= C_3 \varepsilon t^{q-1} \mu^{\frac{(2s-n)(q+1)}2+n}.$$ 
Looking for a critical point of $g$ is equivalent to looking for a solution of $f(t)=h(t)$. We notice that $f(t)=0$ has the solution
\[ \alpha =\lr{ \frac{1+C_1 \mu^{n-2s}} { 1-C_2\mu^n}}^{\frac{n-2s}{4s}},\] which is positive for $\mu\in(0,\mu_1)$. Moreover, $f$ is strictly decreasing on $(0,+\infty)$ for any $\mu \in (0,\mu_1) $, $h$ is strictly increasing  on $(0,+\infty)$ (recalling that $q\textcolor{black}{\geq}1$) and 
\bgs{ f(0)>0, \quad  f(\alpha)=0 , \quad \mbox{and} \quad h(0)=0 , \quad h(\alpha)>0.}
From this it follows that there exists (and is unique) $t_{\mu} \in (0,\alpha)$ such that $f(t_{\mu})=h(t_{\mu})$ (hence $g'(t_{\mu})=0$). Notice also that $g'(t)>0$ on $(0,t_{\mu})$ and   $g'(t)<0$ on $(t_{\mu},+\infty)$. This implies that $g(t_{\mu})$ is a maximum. 
%Then from \eqref{blaaa} we have that
%\[ C_3 \varepsilon \frac{t_{\mu} ^{q+1}}{q+1} \mu^{\frac{(2s-n)(q+1)}2+n} = \frac{t_{{\mu,\varepsilon}}^2}{q+1} (1+C_1\mu^{n-2s}) - \frac{t_{\mu}^{\cx} }{q+1} (1-C_2\mu^n). \]
Now, denoting by 
\[F(t):= \frac{t^2}2 \lr{1 +C_1\mu^{n-2s} }  -  \frac{t^{\cx}}{\cx} \lr{1 -C_2\mu^n} \]
we have that $F'(t)= tf(t) >0$ on $(0,\alpha)$, hence $F(t_{\mu})\leq F(\alpha)$. 
%and
%\[ G(t) :=  -\frac{t^2}{q+1} (1+C_1\mu^{n-2s}) +\frac{t^{\cx} }{q+1} (1-C_2\mu^n),\] 
  On the other hand, $t_{\mu}>0$ and there exists $\delta >0$ independent on $\varepsilon$ and $ \mu$ such that $t_{\mu}\geq \delta$.
  Indeed, since $g'(t_{\mu})=0$, one has from \eqref{blaaa} that
  \[ 1<1 + C_1 \mu^{n-2s}  = t_{\mu}^{\cx-2}  (1-C_2\mu^n)+ C_3 \varepsilon t_{\mu}^{q-1} \mu^{\frac{(2s-n)(q+1)}2+n} <t_{\mu}^{\cx-2}  + C_3 t_{\mu}^{q-1} \] for any $\mu \in (0,\mu_1) $ and  $\varepsilon \in (0,1)$ and this implies the claim. 
  And so by renaming $C_3$ (that will depend on $\delta$ also) \textcolor{black}{and computing $F(\alpha)$} we have that 
  \bgs{ g(t)\leq  g(t_{\mu} )=&\;F(t_{\mu} )  - C_3 \varepsilon \frac{t_{\mu} ^{q+1}}{q+1} \mu^{\frac{(2s-n)(q+1)}2+n}  \leq F(\alpha)- C_3  \varepsilon\mu^{\frac{(2s-n)(q+1)}2+n}   \\
\leq &\; \lr{\frac{1}2-\frac{1}{\cx} } (1+C_1\mu^{n-2s})^{\frac{n}{2s}} (1-C_2\mu^{n})^{\frac{2s-n}{2s}}  - C_3  \varepsilon\mu^{\frac{(2s-n)(q+1)}2+n}    \\
%= &\;\frac{s}n (1+C_1\mu^{n-2s})^{\frac{n}{2s}} (1-C_2\mu^{n})^{\frac{2s-n}{2s}}\\
=&\; \frac{s}n +C_1 \mu^{ n-2s} + o(\mu^{n-2s})  - C_3  \varepsilon\mu^{\frac{(2s-n)(q+1)}2+n} .} 
Renaming the constants, from \eqref{fltx} we have that \textcolor{black}{for any $\mu \in (0,\min\{\mu^{\star},\mu_1\})$}
 \[ \Fl (t\bar Z_{\mu,\xi}) \leq \frac{s}n S^{\frac{n}{2s}} +C_1 \mu^{ n-2s} + o(\mu^{n-2s})  - C_3  \varepsilon\mu^{\frac{(2s-n)(q+1)}2+n}.\] \textcolor{black}{This concludes the proof of Lemma \ref{haha}.}
\end{proof}

\subsection{Proof of Theorem \ref{scrtheorem}}\label{proofThm}

\noindent Let us take $\mu=\varepsilon^\beta$ with $\beta$ satisfying
\eqlab{ \label{beta1}\frac{2}{n(q+1)-2s(q+3)} <\beta<\frac{\delta}{\frac{(2s-n)(q+1)}{2}+n},}
and $\delta>0$ large enough to have both conditions satisfied
\textcolor{black}{(notice that both denominators are positive by hypothesis)}. This gives in particular that
\[\beta(n-2s)>1+\beta\left[ \frac{(2s-n)(q+1)}2+n \right].\] 
\noindent\textcolor{black}{ Consider now the case $n\in (2s,6s)$.} For $\varepsilon$ small enough, from Lemma \ref{haha} and renaming the constants, we obtain
\begin{equation*}\begin{split}
c_\varepsilon+C\varepsilon^{1+\delta}&\leq \frac{s}{n}S^{n/2s}+C\varepsilon^{\beta(n-2s)}+o(\varepsilon^{\beta(n-2s)})-C\varepsilon^{1+\beta\left({\frac{(2s-n)(q+1)}{2}+n}\right)}\\
&\leq \frac{s}{n}S^{n/2s}-C\varepsilon^{1+\beta\left({\frac{(2s-n)(q+1)}{2}+n}\right)}<\frac{s}{n}S^{n/2s},
\end{split}\end{equation*}
that is assumption (i) of Theorem \ref{PSCthm} for $n\in(2s,6s)$. 
\medskip

On the other hand, if $n\geq 6s$ we have that
$$q\geq 1 >\frac{n}{n-s}\frac{n}{n-2s}-1,$$
\textcolor{black}{which assures that \[ \frac{2}{n(q+1)-2s(q+3)} <\frac{\cx}{(n-2s)[\cx-(q+1)]} .\] So we pick now $\beta$ with the additional condition \bgs{ \frac{2}{n(q+1)-2s(q+3)}<\beta<\frac{\cx}{(n-2s)[\cx-(q+1)]} }(still taking $\delta>0$ such that \eqref{beta1} is satisfied).
In particular we have that \[ 1+\beta\left[\frac{(2s-n)(q+1)}2+n \right]< \frac{\cx}{\cx-(q+1)}.\] 
%and hence, for every $\varepsilon$ it is possible to choose $\mu$ small enough in Lemma \ref{haha} so that
Therefore for $\varepsilon$ small enough, we get from Lemma \ref{haha} that
\bgs{ c_\varepsilon+c_1\eps^{1+\delta}+\overline c\eps^{\frac{\cx}{\cx-(q+1)}}<&\; \frac{s}n +C\eps^{\beta(n-2s)} -C \varepsilon^{1+\beta\left({\frac{(2s-n)(q+1)}{2}+n}\right)}\\ <&\;\frac{s}{n}S^{n/2s} -C  \varepsilon^{1+\beta\left({\frac{(2s-n)(q+1)}{2}+n}\right)} <\frac{s}{n}S^{n/2s},}
that is assumption (i) of Theorem \ref{PSCthm} for $n>6s$.}
%
%\begin{equation*}
%\mu^{n-2s}<\varepsilon\mu^{\frac{(2s-n)(q+1)}{2}+n}\qquad\mbox{ and }\qquad\varepsilon^{\frac{2^{\star}_s}{2^{\star}_s-(q+1)}}<\varepsilon\mu^{\frac{(2s-n)(q+1)}{2}+n}.
%\end{equation*}
%Therefore, choosing $\varepsilon$ small and $\delta>0$ large enough, there holds
%\begin{equation*}\begin{split}
%c_\varepsilon+C\varepsilon^{1+\delta}+c_1\varepsilon^{\frac{2^{\star}_s}{2^{\star}_s-(q+1)}}<\frac{s}{n}S^{n/2s}.
%\end{split}\end{equation*}

\noindent Hence, Theorem \ref{PSCthm} yields that the operator $\Fl$ satisfies the Palais-Smale condition. Moreover, Lemma \ref{haha} assures that it has the geometry of Mountain Pass type and therefore we conclude the existence of a critical point of $\Fl$. This implies the existence of a positive solution of \eqref{scrproblem} and concludes the proof of Theorem \ref{scrtheorem}. 

\chapter{Nonlocal phase transitions }\label{S:NP}
\begin{abstract}{We consider in this chapter a nonlocal phase transition model, in particular described by the Allen-Cahn equation. We deal here with a two-phase transition model, in which a fluid can reach two pure phases forming an interface of separation. The aim is to describe the pattern and the separation of the two phases, focusing on the study of
long range interactions that naturally leads to the analysis of phase transitions and interfaces of nonlocal type. The formation of the interface is driven by a variational principle, and here the kinetic energy is modified to take into account far away changes in phase (though the influence is weaker and weaker towards infinity). A fractional analogue of a conjecture of De Giorgi, that deals with possible one-dimensional symmetry of entire solutions naturally arises from treating this model, and will be consequently presented. }
\end{abstract}

\bigskip 
\bigskip

We consider
a nonlocal phase transition model, \label{S:NP:P} in particular described by the Allen-Cahn equation. A fractional analogue of
a conjecture of De Giorgi, that deals with possible one-dimensional symmetry of entire solutions, naturally arises from treating this model, and will be consequently presented. There is a very interesting connection with nonlocal minimal surfaces, that will be studied in Chapter \ref{nlms}.

  We introduce briefly the classical case\footnote{We would like to thank Alberto Farina who, during a summer-school in Cortona (2014), gave a beautiful introduction on phase transitions in the classical case.}. The Allen-Cahn equation has various applications, for instance, in the study of interfaces (both in gases and solids), in the theory of superconductors and superfluids or in cosmology.  We deal here with a two-phase transition model, in which a fluid can reach two pure phases (say $1$ and $-1$) forming an interface of separation. The aim is to describe the pattern and the separation of the two phases. 
  
The formation of the interface is driven by a variational principle. Let $u(x)$ be the function describing the state of the fluid at position $x$ in a bounded region $\Omega$. As a first guess, the phase separation
can be modeled
via the minimization of the energy 
  \[ \E_0(u)= \int_{\Omega} W \big(u(x)\big) \, dx,\]
where $W$ is a double-well potential. More precisely, $W\colon [-1,1]\to [0,\infty)$ such that 
	\eqlab{ \label{dwp} W\in C^2\left([-1,1]\right), W(\pm 1)=0, W>0 \mbox{ in } (-1,1),\\
	 W'(\pm 1)= 0 \mbox{ and } W''(\pm 1) >0. }
The classical example is 
\begin{equation}\label{DF-WELL}
W(u):=\displaystyle \frac{(u^2-1)^2}{4}.\end{equation} 
On the other hand, the functional in~$\E_0$ 
produces an ambiguous
outcome, since any function $u$ that attains
only the values $\pm 1$ is a minimizer for the energy. That is,
the energy functional in~$\E_0$ alone cannot detect any
geometric feature of the interface.

To avoid this, one is led to consider an additional energy term that penalizes 
the formation of unnecessary interfaces. The typical energy
functional provided by this procedure has the form
\begin{equation}\label{5.2PRE}
\mathcal{E} (u):= \int_{\Omega} W\big(u(x)\big)\, dx + \frac{\eee^2}{2} \int_{\Omega} |\nabla u(x)|^2 \, dx.\end{equation}
In this way, the potential energy that forces the pure phases is compensated by a small term, that is due to the elastic effect of the reaction of the particles. 
As a curiosity, we point out that
in the classical mechanics framework, the analogue
of~\eqref{5.2PRE} is a Lagrangian action of a particle,
with~$n=1$, $x$ representing a time coordinate and $u(x)$
the position of the particle at time~$x$. In this framework
the term involving the square of the derivative of~$u$
has the physical meaning of a
kinetic energy. With a slight abuse of notation, we will keep
referring to the gradient term in~\eqref{5.2PRE}
as a kinetic energy. Perhaps a more appropriate term would be elastic
energy, but in concrete applications also 
the potential may arise from elastic reactions,
therefore the only purpose of these names in our framework
is to underline the fact that~\eqref{5.2PRE}
occurs as a superposition of two terms,
a potential one, which only depends on $u$,
and one, which will be called kinetic, which only depends
on the variation of $u$ (and which, in principle, possesses no
real ``kinetic'' feature). 
	
	The energy minimizers will be smooth functions, taking values between $-1$ and $1$, forming layers of interfaces of $\eee$-width. If we send $\eee \to 0$, the transition layer will tend to a minimal surface. To better explain this, consider the energy
	\eqlab{\label{gle1}  J(u)= \int \frac{1}2 |\nabla u|^2 +W(u) \,dx,} whose minimizers solve the Allen-Cahn equation
	\begin{equation}\label{ALC} -\Delta u+ W'(u)=0. \end{equation}  
In particular, for the explicit potential in~\eqref{DF-WELL},
equation~\eqref{ALC} reduces (up to normalizations constants) to
\begin{equation}\label{ALC-DG} -\Delta u= u-u^3.\end{equation}
In this setting,
the behavior of $u$ in large domains reflects into the behavior of the rescaled function $u_{\eee} (x)=u\big( \frac{x}{\eee}\big)$ in $B_1$. Namely, the minimizers of $J$ in $B_{1/ \eee}$ are the minimizers of $J_\eee$ in $B_1$, where $J_\eee$ is the rescaled energy functional
	\eqlab{ J_{\eee} (u)= \int_{B_1} \frac{\eee}2 |\nabla u|^2 + \frac{1}\eee W(u) \, dx.    \label{rescef1} }
We notice then that
	\[ J_\eee(u) \geq \int_{B_1} \sqrt{ 2W(u)}\, |\nabla u| \, dx\] 
which, using the Co-area Formula, gives 	\[  J_\eee(u) \geq
\int_{-1}^1 \sqrt{2W(t)}\,\mathcal {H}^{n-1} \left(\{u=t\}\right) \, dt.\] 
The above formula may suggest that the minimizers of~$J_\eee$ 
have the tendency to minimize the $(n-1)$-dimensional measure of
their level sets. It turns out that indeed the level sets of the
minimizers of $J_\eee$ converge 
to a minimal surface as $\eee \to 0$: for more details see, for instance,
\cite{savinminimal} and the references therein.
															
In this setting, a famous De Giorgi conjecture comes into place. In the
late 70's, De Giorgi conjectured that entire, smooth, monotone (in one direction), bounded solutions of~\eqref{ALC-DG}
in the whole of~$\R^n$ 
are necessarily one-dimensional, i.e., there exist $\omega \in S^{n-1}$ and $u_0: \R\to \R$ such that 	
	\[ u(x)=u_0(\omega \cdot x) \quad \mbox{for any} \quad x\in \Rn.\] In other words, the conjecture above asks
if the level sets of the entire, smooth, monotone (in one direction),
bounded solutions are necessarily hyperplanes, at least in dimension $n\leq 8$. 

One may wonder why the number eight has a relevance in the problem above.
A possible explanation for this is given by the Bernstein Theorem,
as we now try to describe.
 
The Bernstein problem asks on whether or not all 
minimal graphs (i.e.
surfaces that locally minimize the perimeter and that
are graphs in a given direction)
in $\R^n$ must be necessarily affine.
This is indeed true in dimensions $n$ at most eight.
On the other hand, in dimension $n\geq 9$ there are global minimal graphs that are not hyperplanes (see e.g.~\cite{GIUSTI}).

The link between the problem of Bernstein and the conjecture of De Giorgi could be suggested
by the fact that
minimizers approach minimal surfaces in the limit.
In a sense, if one is able to prove that the limit interface is a hyperplane
and that this rigidity property gets inherited by the level sets of
the minimizers~$u_\eee$ (which lie nearby such limit hyperplane),
then, by scaling back, one obtains that the level sets
of~$u$ are
also hyperplanes. Of course, this link between the two
problems, as stated here, is only heuristic,
and much work is needed to deeply understand the connections
between the problem of Bernstein and the conjecture of De Giorgi.
We refer to~\cite{FARINA-VALDINOCI-STATE} for a more detailed introduction
to this topic.
\bigskip
	
We recall that this conjecture by De Giorgi
was proved for $n\leq 3$, see~\cite{GGUI, BCN97, AC00}.
Also,  the case $4\leq n \leq 8$ with the additional assumption that 	
\begin{equation} \label{limdgs}
 \lim_{x_n\to \pm \infty} u(x',x_n)=\pm 1, \quad \mbox{for any} \quad x'\in \R^{n-1}
\end{equation}
was proved in \cite{S09}. 

For $n\geq 9$ a counterexample can be found in \cite{PKW08}.  Notice 
that, if the above limit is uniform (and the De Giorgi conjecture with 
this additional assumption is known as the Gibbons conjecture), the 
result extends to all possible $n$ (see for 
instance~\cite{Farina-Gibbons, FARINA-VALDINOCI-STATE}
for further details).

The goal of the next part of this thesis
is then to discuss an analogue of these questions
for the nonlocal case and present related results.

\section{The fractional Allen-Cahn equation} \label{sbsac} 

The extension of the Allen-Cahn equation in~\eqref{ALC}
from a local to a nonlocal 
setting has theoretical interest and concrete applications. 
Indeed, the study of
long range interactions naturally leads to the analysis of
phase transitions and interfaces of nonlocal type.

Given an open domain 
$\Omega\subset \Rn$ and the double well potential $W$
(as in~\eqref{DF-WELL}), our goal here is 
to study the fractional Allen-Cahn equation	
\[ \frlap u +W'(u)=0 \quad \mbox{in} \quad \Omega, \]
for~$s\in(0,1)$ (when~$s=1$, this equation reduces to~\eqref{ALC}).
The solutions are the critical points of the nonlocal energy 	
	\begin{equation} \label{enfac}
		 \mathcal{E}(u,\Omega) := \int_{\Omega} W\big(u(x)\big) \, dx + \frac12
		 \iint_{\R^{2n}\setminus (\C \Omega)^2} \frac{|u(x)-u(y)|^2}{|x-y|^{n+2s} }\, dx\, dy,	
	\end{equation}
up to normalization constants that we omitted for simplicity.
The reader can compare~\eqref{enfac} with~\eqref{5.2PRE}. Namely, in~\eqref{enfac} the kinetic energy
is modified, in order to take into account long range interactions.
That is, the new kinetic energy still depends on the variation of the
phase parameter.
But, in this case, far away changes in phase may influence each other
(though the influence is weaker and weaker towards infinity).
	
Notice that in the nonlocal framework, we prescribe the function on 
$\C \Omega \times \C \Omega$ and consider the kinetic energy on the remaining regions (see Figure \ref{fign:eac}). 
The prescription of values in~$\C \Omega \times \C \Omega$
reflects into the fact that the domain of integration of the kinetic
integral in~\eqref{enfac} is~$\R^{2n}\setminus (\C \Omega)^2$.
Indeed, this is perfectly compatible with
the local case in~\eqref{5.2PRE}, where the domain of
integration of the kinetic term was simply~$\Omega$.
To see this compatibility, one may think that 
the domain of integration of the kinetic energy
is simply the complement of
the set in which the values of the functions are prescribed.
In the local case of~\eqref{5.2PRE}, the values are prescribed on~$\partial\Omega$, or, one may say, in~$\C \Omega$: then the domain of integration
of the kinetic energy is the complement of~$\C \Omega$,
which is simply~$\Omega$. 
In analogy with that, in the nonlocal case of~\eqref{enfac},
the values are prescribed on~$\C \Omega\times\C \Omega=(\C \Omega)^2$,
i.e. outside~$\Omega$ for both the variables $x$ and~$y$.
Then, the kinetic integral is set on the complement of~$(\C \Omega)^2$,
which is indeed~$\R^{2n}\setminus(\C \Omega)^2$.

Of course, the potential energy has local features, both
in the local and in the nonlocal case,
since in our model the nonlocality only occurs in the kinetic interaction,
therefore the potential integrals are set over~$\Omega$
both in~\eqref{5.2PRE} and in~\eqref{enfac}.

For the sake of shortness, given disjoint sets~$A$, $B\subseteq\R^n$
we introduce the notation
\begin{equation*} u(A,B):=\int_{A}\int_{B}  \frac{|u(x)-u(y)|^2}{|x-y|^{n+2s} }\, dx\, dy,\end{equation*}
and we write the new kinetic energy in~\eqref{enfac} as
\begin{equation}\label{kenac}
 \mathcal {K} (u,\Omega) = \frac{1}{2} u(\Omega, \Omega) + u(\Omega, \C \Omega).
\end{equation} 

%\begin{center}
\begin{figure}[htpb]
%\sidecaption
%	\hspace{0.8cm}
%	\begin{minipage}[b]{0.95\linewidth}
%	\centering
	\includegraphics[width=0.40\textwidth]{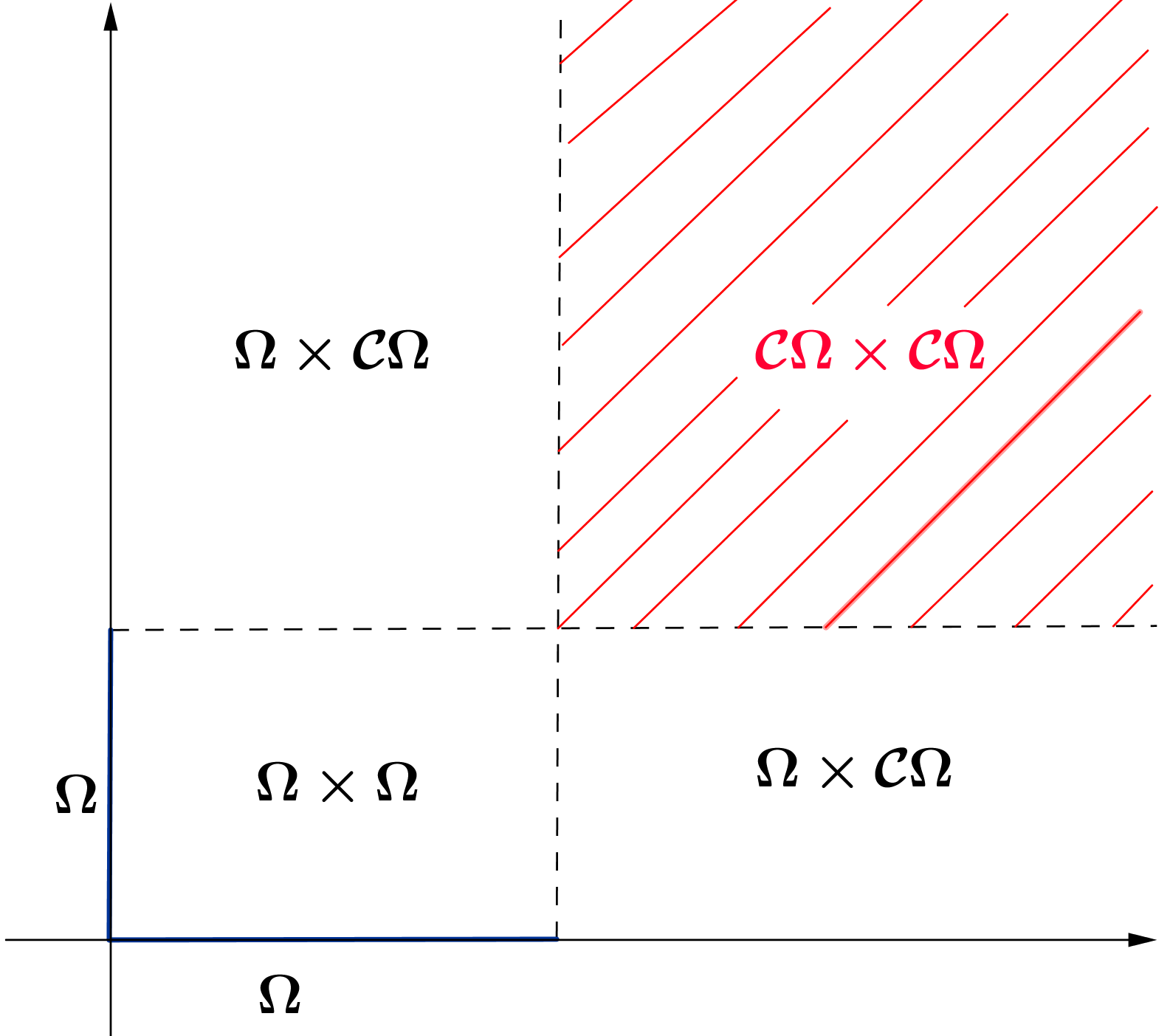}
	\caption{The kinetic energy}   
	\label{fign:eac}
%	\end{minipage}
\end{figure}
%\end{center}
Let us define the energy minimizers and provide a density estimate for the minimizers.
\begin{defn}
The function $u$ is a minimizer for the energy $\E$ in $B_R$ if 
$ \E (u, B_R)\leq \E (v,B_R) $ for any $v$ such that $u=v $ outside $B_R$. 
\end{defn}

The energy of the minimizers satisfy the following
uniform bound property on large balls.
\begin{theorem} \label{acenest1}
Let $u$ be a minimizer in $B_{R+2}$ for a large $R$, say $R\geq1$. Then 
	\begin{equation}\label{THANKS}
\lim_{R \to +\infty} \frac{1}{R^n} \E (u,B_R) =0.\end{equation} More precisely,
	\[ \E(u,B_R) \leq 
					\begin{cases}
						CR^{n-1} \quad &\mbox{if} \quad s\in \Big(\frac{1}{2},1\Big),\\
						CR^{n-1}\log R \quad &\mbox{if} \quad s=\frac{1}{2},\\
						CR^{n-2s} \quad &\mbox{if} \quad s\in \Big(0,\frac{1}{2}\Big).
					\end{cases}
	\]
	Here, $C$ is a positive constant depending only on $n, s$ and $W$.
\end{theorem}
Notice that for $\displaystyle s\in \Big(0,\frac{1}{2}\Big)$, $R^{n-2s}>R^{n-1}$. These estimates are optimal (we
refer to \cite{SV14} for further details).
\begin{proof}
%\begin{center}
\begin{figure}[htpb]
%\sidecaption
%	\hspace{0.8cm}
%	\begin{minipage}[b]{0.95\linewidth}
%	\centering
	\includegraphics[width=0.75\textwidth]{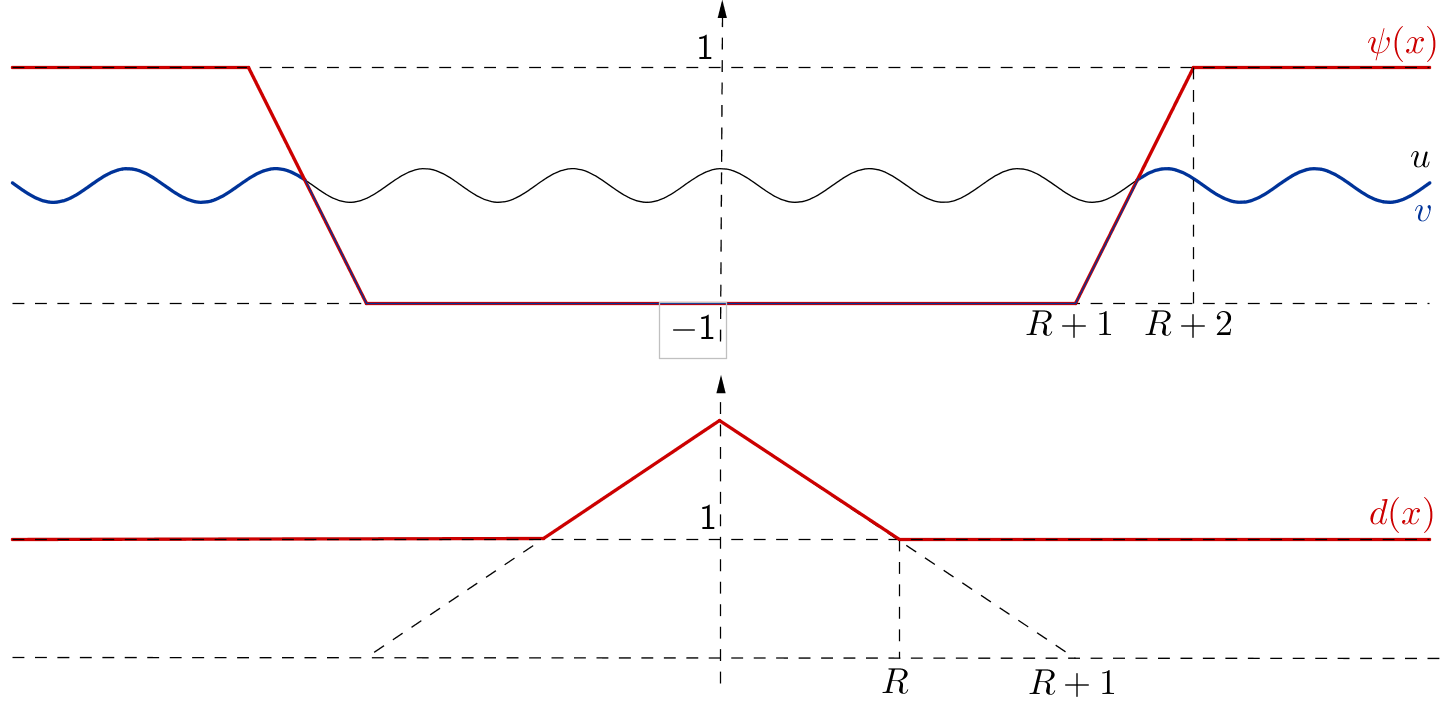}
	\caption{The functions $\psi$, $v$ and $d$}   
	\label{fign:psid}
%	\end{minipage}
\end{figure}
%\end{center}
We remark that throughout this proof the constants may change value from line to line. 
We introduce at first some auxiliary functions. Let 
	\[\psi(x):=-1+2 \min\Big\{ (|x|-R-1)_+,1\Big\}, \quad v(x):=\min \Big\{u(x), \psi(x)\Big\},\]\[ d(x):= \max\Big\{ (R+1-|x|),1 \Big\}.\]
	Then, for $|x-y|\leq d(x)$ we have that 
	\begin{equation}\label{acpsi1} |\psi(x)-\psi(y)|\leq \frac{2 |x-y|}{d(x)}.\end{equation}
	Indeed, if $|x|\leq R$, then $d(x)=R+1-|x|$ and \[ |y|\leq |x-y|+|x|\leq d(x)+|x|\leq R+1,\] thus $\psi(x)-\psi(y)=0$ and the inequality is trivial. Else, if $|x|\geq R$, then $d(x)=1$,
and so the inequality is assured by the Lipschitz continuity of $\psi$ (with $2$ as the Lipschitz constant).
	
Also, we prove that we have the following estimates for the function $d$:
	\begin{equation} \label{acdx1} \int_{B_{R+2}} d(x)^{-2s}\, dx \leq 
		\begin{cases} 
			 CR^{n-1}  \quad & \mbox{if}  \quad s\in \Big(\frac{1}{2},1\Big),\\
			 CR^{n-1}\log R   \quad &\mbox{if} \quad s=\frac{1}{2},\\
			 CR^{n-2s}  \quad & \mbox{if}  \quad s\in \Big(0,\frac{1}{2}\Big) ,
		\end{cases}
	\end{equation}
	where $C=C(n,s)>0$.
	To prove this, we observe that
in the ring $B_{R+2}\setminus B_R$, we have~$d(x)=1$. Therefore,
the contribution to the integral in~\eqref{acdx1}
that comes from the ring~$B_{R+2}\setminus B_R$
is bounded by the measure of the ring, and so it
is of order $R^{n-1}$, namely
\begin{equation}\label{5.10bis}
\int_{B_{R+2}\setminus B_R} d(x)^{-2s}\,dx=|B_{R+2}\setminus B_R|\le CR^{n-1},
\end{equation}
for some~$C>0$.
We point out that this order is always negligible with respect to
the right hand side of~\eqref{acdx1}.\\
Therefore, to complete the proof of~\eqref{acdx1}, it only remains to estimate
the contribution to the integral coming from~$B_R$.
For this, we use
polar coordinates and perform
the change of variables $t=\rho/(R+1)$. In this way, we obtain that
	\[ \begin{split}
		\;&\int_{B_R} d(x)^{-2s} \, dx = C\,\int_0^R \frac{\rho^{n-1}}	{(R+1-\rho)^{2s}} \, d\rho = C\,(R+1)^{n-2s} \int_0^{1-\frac{1}{R+1}} t^{n-1} (1-t)^{-2s} \, dt\\
			\leq \;& C\,(R+1)^{n-2s} \int_0^{1-\frac{1}{R+1}}  (1-t)^{-2s} \, dt,
	 \end{split}\]
for some dimensional constant~$C>0$. 
Now we observe that
	\[  \int_0^{1-\frac{1}{R+1}}  (1-t)^{-2s} \, dt \leq
	 	\begin{cases} \displaystyle \int_0^1 (1-t)^{-2s} \, dt =C \quad &\mbox{if} \quad s\in \Big(0, \frac{1}{2}\Big),\\
	 					-\log (1-t) \Big|_0^{1-\frac{1}{R+1}} \leq \log R \quad &\mbox{if} \quad s= \frac{1}{2},\\
	 					-\frac{(1-t)^{1-2s}}{1-2s} \Big|_0^{1-\frac{1}{R+1}} \leq C R^{2s-1} \quad &\mbox{if} \quad s\in \Big( \frac{1}{2},1 \Big) .
		\end{cases}
	\]
The latter two formulas and~\eqref{5.10bis}
imply~\eqref{acdx1}.\\ 
Now, we
define the set \[ A:= \{ v=\psi\} \] and notice that $B_{R+1}\subseteq A \subseteq B_{R+2}$. We prove
that for any $x\in A$ and any $y\in \C A$ 
	\begin{equation}\label{8uUtY}
		|v(x)-v(y)| \leq \max\Big\{ |u(x)-u(y)|,|\psi(x)-\psi(y)|\Big\}.	
	\end{equation}
Indeed, for $x\in A$ and $y\in \C A$ we have that
\[ 	v(x) =\psi(x)\leq u(x) \quad \mbox{and} \quad  v(y)=u(y)\leq \psi(y),\]
therefore
	\[ v(x)-v(y)\leq u(x)-u(y) \quad \mbox{and} \quad v(y)-v(x)\leq \psi(y)-\psi(x),\]
	which establishes~\eqref{8uUtY}. This leads to 
	\begin{equation} \label{acv1}
		v(A,\C A) \leq u(A,\C A)+\psi(A,\C A).
	\end{equation}
Notice now that \[\E(u,B_{R+2})\leq \E(v,B_{R+2})\] since $u$ is a minimizer in $B_{R+2}$ and $v=u$ outside $B_{R+2}$. We have that
	\[ \begin{split}
		\E (u,B_{R+2}) =\;& \frac{1}{2} \;u(B_{R+2},B_{R+2}) + u(B_{R+2}, \C B_{R+2}) +\int_{B_{R+2}} W(u) \, dx\\
						=\;& \frac{1}{2}\; u(A,A) + u(A,\C A) \\
						\;& + \frac{1}{2} \;u(B_{R+2}\setminus A,B_{R+2}\setminus A) + u(B_{R+2}\setminus A, \C B_{R+2}) \\
						\;& + \int_{A} W(u) \, dx+ \int_{B_{R+2}\setminus A} W(u) \, dx.
	\end{split}\]
	Since $u$ and $v$ coincide on $\C A$, by using the inequality \eqref{acv1} we obtain that
	\[ \begin{split}
		0\;\leq&\;  \E(v,B_{R+2})- \E(u,B_{R+2}) =\E(v,A)- \E(u,A)  \\= &\; \frac{1}{2}\; v(A,A)-\frac{1}{2} u(A,A) + v(A,\C A) -u (A,\C A) + \int_{A} \Big(W(v)-W(u)\Big) \, dx \\
		 \leq & \;  \frac{1}{2} \; v(A,A)-\frac{1}{2} u(A,A)  + \psi (A,\C A) +  \int_{A}\Big( W(v)-W(u) \Big)\, dx.
	\end{split}\]
Moreover, $v=\psi$ on $A$ and we have that
		\[ \frac{1}{2}\; u(A,A) + \int_{A} W(u) \, dx \leq  \frac{1}{2} \; \psi (A, A) + \psi (A, \C A)  +\int_{A} W(\psi) \, dx )=\E(\psi,A),\]
		and therefore, since $B_{R+1}\subseteq A \subseteq B_{R+2}$,
		\begin{equation}\label{acbr1} \frac{1}{2} \;u(B_{R+1},B_{R+1}) + \int_{B_{R+1}} W(u)\, dx \leq \E(\psi,B_{R+2}).\end{equation}
We estimate now $\E(\psi, B_{R+2})$. For a fixed $x \in B_{R+2}$ we observe that
		\[ \begin{split} \int_{\Rn}  \frac{|\psi(x)-\psi(y)|^2}{|x-y|^{n+2s}} \, dy = \al \int_{|x-y|\leq d(x)} \frac{|\psi(x)-\psi(y)|^2}{|x-y|^{n+2s}} \, dy  + \int_{|x-y|\geq d(x)} \frac{|\psi(x)-\psi(y)|^2}{|x-y|^{n+2s}} \, dy \\
				\leq \al  C \bigg( \frac{1}{d(x)^{2}} \int_{|x-y|\leq d(x)} |x-y|^{-n-2s+2} \, dy + \int_{|x-y|\geq d(x)} |x-y|^{-n-2s} \, dy \bigg),\end{split}\]
		where we have used \eqref{acpsi1} and the boundedness of $\psi$. Passing to polar coordinates, we have that
		\[ \begin{split}
			\int_{\Rn} \frac{|\psi(x)-\psi(y)|^2}{|x-y|^{n+2s}} \, dy \leq&\; C\bigg( \frac{1}{d(x)^{2}} \int_0^{d(x) } \rho^{-2s+1} \, d\rho + \int_{d(x)}^{\infty} \rho^{-2s-1}\, d\rho \bigg) \\ =&\;C d(x)^{-2s}. 
	\end{split}\]
Recalling
that~$\psi(x)=-1$ on $B_{R+1}$ and $W(-1)=0$, we obtain that
 \[ \begin{split}
 	\E(\psi, B_{R+2}) =&\;\int_{B_{R+2}}\int_{\Rn} \frac{|\psi(x)-\psi(y)|^2}{|x-y|^{n+2s}} \, dy \, dx + \int_{B_{R+2}} W(\psi) \, dx \\
 	\leq&\; C\int_{B_{R+2}} d(x)^{-2s} dx + \int_{B_{R+2}\setminus B_{R+1}} W(\psi)\,dx .
 \end{split}\]
Therefore, making use of~\eqref{acdx1},
\begin{equation}\label{acpsi2} \E(\psi, B_{R+2}) \leq \begin{cases} 
			CR^{n-1} \quad &\mbox{if} \quad s\in \Big(\frac{1}{2},1\Big),\\
			CR^{n-1}\log R \quad &\mbox{if} \quad s= \frac{1}{2},\\
			CR^{n-2s} \quad &\mbox{if} \quad s\in \Big(0,\frac{1}{2}\Big).
		\end{cases}\end{equation} 
For what regards the right hand-side of inequality \eqref{acbr1}, we have that
		\eqlab{\label{bla111}
			\frac{1}{2} \;u(B_{R+1},B_{R+1}) + \int_{B_{R+1}} W(u)\, dx  \geq&\;   \frac{1}{2} \; u(B_{R},B_{R}) +u(B_R, B_{R+1}\setminus B_R) \\
			&\; +\int_{B_{R}} W(u)\, dx . }
We prove now that
\begin{equation}\label{5.16bis}
u(B_R,\C B_{R+1})  \leq  C\int_{B_{R+2}} d(x)^{-2s}\, dx .\end{equation}
For this, we observe that if $x\in B_R$, then~$d(x)=R+1-|x|$.
So, if~$x\in B_R$ and~$y \in \C B_{R+1}$, then
\[|x-y|\geq |y|-|x| \geq R+1-|x| =d(x). \] Therefore, by changing variables $z=x-y$ and then passing to polar coordinates, we have that
	\bgs{ u(B_R,\C B_{R+1}) \leq&\; 4\,\int_{B_R} dx \int_{\C B_{d(x)}} |z|^{-n-2s} \, dz\\
				\leq&\; C\,\int_{B_R} dx \int_{d(x) }^{\infty} \rho^{-2s-1} \, d\rho \\
				=&\;C\, \int_{B_R} d(x)^{-2s} \, dx .}
This establishes~\eqref{5.16bis}.\\
	 Hence, by \eqref{acdx1} and~\eqref{5.16bis}, we have that
	 	\begin{equation} \label{acpsi3} 
u(B_R,\C B_{R+1})  \leq  C\int_{B_{R+2}} d(x)^{-2s}\, dx \leq 
		\begin{cases} 
			 CR^{n-1}  \quad & \mbox{if}  \quad s\in \Big(\frac{1}{2},1\Big),\\
			 CR^{n-1}\log R   \quad &\mbox{if} \quad s=\frac{1}{2},\\
			 CR^{n-2s}  \quad & \mbox{if}  \quad s\in \Big(0,\frac{1}{2}\Big) .
		\end{cases}\end{equation}
We also observe that, by adding $u(B_R,\C B_{R+1}) $ to inequality \eqref{bla111}, we obtain that
	\[ \begin{aligned} 	\frac{1}{2} \al  u(B_{R+1},B_{R+1}) + \int_{B_{R+1}} W(u)\, dx  + u(B_R,\C B_{R+1}) 
	\\
		&\qquad\ge   \frac{1}{2} \; u(B_{R},B_{R}) +u(B_R, B_{R+1}\setminus B_R) +\int_{B_{R}} W(u) \, dx +u(B_R,\C B_{R+1})
\\ &\qquad=\E(u,B_R)
.\end{aligned}\] 
This and \eqref{acbr1} give that 
		\[ \E(u,B_R)\leq \E(\psi,B_{R+2} )+  u(B_R,\C B_{R+1}).\] 
Combining this with the estimates in \eqref{acpsi2} and \eqref{acpsi3},
we obtain the desired result.
\end{proof}

Another type of estimate can be given in terms of the level sets of the minimizers (see Theorem 1.4 in \cite{SV14}).

\begin{theorem}\label{TY78UU}
Let u be a minimizer of $\E$ in $B_R$. Then for any $\theta_1, \theta_2 \in (-1,1)$ such that \[ u(0)>\theta_1\] we have that
there exist~$\overline R$ and~$C>0$ such that
	\[ \Big| \{ u>\theta_2\} \cap B_R\Big| \geq C R^n\]
if $R \geq \overline R(\theta_1, \theta_2)$. The constant $C > 0$ depends only on $n$, $s$ and $W$ and $\overline R(\theta_1, \theta_2)$ is
a large constant that depends also on $\theta_1$ and $\theta_2$.
\end{theorem}

The statement of Theorem \ref{TY78UU}
says that the level sets of minimizers always occupy
a portion of a large ball comparable to the ball itself.
In particular, both phases occur in a large ball,
and the portion of the ball occupied by each phase is comparable
to the one occupied by the other.\\
Of course, the simplest situation in which two phases split
a ball in domains with comparable, and in fact equal, size
is when all the level sets are hyperplanes.
This question is related to a fractional version
of a classical conjecture of
De Giorgi and
to nonlocal minimal surfaces, that we discuss in the following Section~\ref{sbsdg} and Chapter~\ref{nlms}.

Let us try now to give some details on the proof of the Theorem \ref{TY78UU} in the particular case in which $s$ is in the range $(0,1/2)$. The more general proof for all $s \in (0,1)$ can be found in \cite{SV14}, where one uses some estimates on the Gagliardo norm. In our particular case we will make use of the Sobolev inequality that we introduced in \eqref{OK:sob:p}. The interested reader can see \cite{SavVal11} for a more exhaustive explanation of the upcoming proof. 

\begin{proof}[Proof of Theorem \ref{TY78UU}]
Let us consider a smooth function $w$ such that $w=1$ on $\C B_R$ (we will take in sequel $w$ to be a particular barrier for $u$), and define 
	\[ v(x):= \min \{ u(x),w(x)\}.\] 
	Since $|u|\leq 1$, we have that $v=u$ in $\C B_R$. Calling $D=\left(\Rn\times \Rn\right)\setminus \left(\C B_R\times \C B_R \right)$ we have from definition \eqref{kenac} that
		\bgs{ \mathcal {K} (u-v,\al  B_R)  + \mathcal {K} (v,B_R) -\mathcal {K}(u,B_R)  \\ =\al
		\frac{1}{2} \iint_D \frac{ |(u-v)(x) -(u-v)(y)|^2 +|v(x)-v(y)|^2 -|u(x)-u(y)|^2}{|x-y|^{n+2s}} \, dx\, dy.}
		Using the identity $|a-b|^2 +b^2-a^2=2b(b-a)$ with $a=u(x)-u(y)$ and $b=v(x)-v(y)$ we get
		\bgs{ \mathcal {K} (u-v,B_R) \al  + \mathcal {K} (v,B_R) -\mathcal {K}(u,B_R) \\
		= \al  \iint_D \frac{ \left( (u-v)(x) -(u-v)(y)\right)\left(v(y)-v(x)\right)} {|x-y|^{n+2s}} \, dx\, dy.} 
	Since $u-v=0$ on $\C B_R$ we can extend the integral to the whole space $\Rn \times \Rn$, hence 
		\bgs{ \mathcal {K}   (u-v,B_R) \al + \mathcal {K} (v,B_R) -\mathcal {K}(u,B_R) \\
		= \al  \iint_{\Rn\times\Rn} \frac{ \left( (u-v)(x) -(u-v)(y)\right)\left(v(y)-v(x)\right)} {|x-y|^{n+2s}} \, dx\, dy.}
	Then  by changing variables and using the anti-symmetry of the integrals, we notice that
		\bgs{ \al \iint_{B_R\times B_R}    \frac{ \left( (u-v)(x) -(u-v)(y)\right)\left(v(y)-v(x)\right)} {|x-y|^{n+2s}} \, dx\, dy \\
		=\al   \iint_{B_R\times B_R}  \frac{ (u-v)(x) \left(v(y)-v(x)\right)} {|x-y|^{n+2s}} \, dx\, dy  - \iint_{B_R\times B_R} \frac{(u-v)(y)\left(v(y)-v(x)\right)} {|x-y|^{n+2s}} \, dx\, dy \\
		=\al 2  \iint_{B_R\times B_R}   \frac{ (u-v)(x) \left(v(y)-v(x)\right)} {|x-y|^{n+2s}} \, dx\, dy}
		and
		\bgs{ \al \iint_{B_R\times \C B_R}  \frac{ \left( (u-v)(x) -(u-v)(y)\right)\left(v(y)-v(x)\right)} {|x-y|^{n+2s}} \, dx\, dy \\
		 \al +   \iint_{\C B_R \times B_R}  \frac{ \left( (u-v)(x) -(u-v)(y)\right)\left(v(y)-v(x)\right)} {|x-y|^{n+2s}} \, dx\, dy  \\
			=\al \iint_{B_R\times \C B_R}   \frac{ (u-v)(x) \left(v(y)-v(x)\right)} {|x-y|^{n+2s}} \, dx\, dy -\iint_{\C B_R \times B_R} \frac{ (u-v)(y) \left(v(y)-v(x)\right)} {|x-y|^{n+2s}} \, dx\, dy  \\
			= \al 2 \iint_{B_R\times \C B_R}  \frac{  (u-v)(x) \left(v(y)-v(x)\right)} {|x-y|^{n+2s}} \, dx\, dy.}
	Therefore
		\bgs{ \al\mathcal {K}   (u-v,B_R)  + \mathcal {K} (v,B_R) -\mathcal {K}(u,B_R) \\
		= \al 2 \iint_{\Rn\times\Rn}   \frac{\left( u(x)-v(x)\right)\left(v(y)-v(x)\right) }{|x-y|^{n+2s}}\,dx \, dy= 2 \int_{\Rn} (u(x)-v(x))\left(\int_{\Rn} \frac{ v(y)-v(x)}{|x-y|^{n+2s}}\, dy\right)\, dx\\
		=\al 2 \int_{B_R\cap\{u>v=w\}} (u(x)-w(x))\left(\int_{\Rn} \frac{ v(y)-w(x)}{|x-y|^{n+2s}}\, dy\right)\, dx \\ \leq  \al  2 \int_{B_R\cap\{u>v=w\}} (u-w)(x)\left(\int_{\Rn} \frac{ w(y)-w(x)}{|x-y|^{n+2s}}\, dy\right)\, dx\\
		=\al 2 \int_{B_R\cap\{u>w\}} (u-w)(x)\left(- \frlap w \right)(x)\, dx.}
Hence
	\bgs{  \mathcal {K}  (u-v,B_R)  
	 \leq  \mathcal {K} (u,B_R) -\mathcal {K}(v,B_R) + 2 \int_{B_R\cap\{u>w\}} (u-w)\left(- \frlap w \right)\, dx.} By adding and subtracting the potential energy, we have that 
	\bgs{  \mathcal {K}  (u-v,B_R) \leq \al \E(u,B_R)-\E(v,B_R) +\int_{B_R} \left(W(v)-W(u)\right) \, dx \\ \al + 2 \int_{B_R\cap\{u>w\}} (u-w)\left(- \frlap w \right)\, dx} and since $u$ is minimal in $B_R$,
	\eqlab{\label{kuvr1} \mathcal {K}  (u-v,B_R) \leq \al   \int_{B_R\cap \{ u>w=v\}} \left(W(w)-W(u)\right) \, dx + 2 \int_{B_R\cap\{u>w\}} (u-w)\left(- \frlap w \right)\, dx.}
We deduce from the properties in \eqref{dwp} of the double-well potential $W$ that there exists a small constant $c>0$ such that
	\bgs{ &W(t)-W(r) \geq c(1+r)(t-r) +c(t-r)^2 & \mbox{ when } & -1\leq r\leq t\leq -1+c \\
		&W(r)-W(t) \leq \frac{1+r}{c} &\mbox{ when } &-1\leq r\leq t\leq 1.}
	We fix the arbitrary constants $\theta_1$ and $\theta_2$, take $c$ small as here above. Let then \[ \theta_{\star}:=\min\{\theta_1,\theta_2,-1+c\}.\] It follows that
	\eqlab{ \label{wwu} \int_{B_R\cap \{ u>w\} }\al \left(W(w)-W(u)\right) dx \\
	=\al  \int_{B_R\cap \{ \theta_{\star}>u>w\} }\left(W(w)-W(u)\right) dx 	+ \int_{B_R\cap \{ u>\max\{\theta_{\star},w\} \} }\left(W(w)-W(u)\right) dx \\
	\leq \al -c  \int_{B_R\cap \{ \theta_{\star}>u>w\} } (1-w)(u-w) \,dx - c\int_{B_R\cap \{ \theta_{\star}>u>w\} }(u-w)^2\, dx \\ 
	\al+ \frac{1}{c} \int_{B_R\cap \{ u>\max\{\theta_{\star},w\} \} } (1+w)\, dx\\
	\leq  \al -c  \int_{B_R\cap \{ \theta_{\star}>u>w\} } (1-w)(u-w) \,dx + \frac{1}{c} \int_{B_R\cap \{ u>\max\{\theta_{\star},w\} \} } (1+w)\, dx.}
Therefore, in \eqref{kuvr1} we obtain that 
\eqlab{\label{kuvr2} \mathcal {K}  (u-v,B_R) \leq \al   -c  \int_{B_R\cap \{ \theta_{\star}>u>w\} } (1-w)(u-w) \,dx + \frac{1}{c} \int_{B_R\cap \{ u>\max\{\theta_{\star},w\} \} } (1+w)\, dx  \\ \al+ 2 \int_{B_R\cap\{u>w\}} (u-w)\left(- \frlap w \right)\, dx.} 
We introduce now a useful barrier in the next Lemma (we just recall here Lemma 3.1 in \cite{SV14} - there the reader can find how this barrier is build):
\begin{lemma}\label{bW1} Given any $\tau\geq 0$ there exists a constant $C> 1$ (possibly depending on $n, s$ and $\tau$) such that: for any $R\geq C$ there exists a rotationally symmetric function $w \in C\left(\Rn, [-1+CR^{-2s},1]\right)$ with $w=1$ in $\C B_R$ and such that for any $x\in B_R$ one has that
	\eqlab{ \label{w1} \frac{1}{C} (R+1-|x|)^{-2s} \leq 1+w(x)\leq C (R+1-|x|)^{-2s} \quad \mbox{and} }
	\eqlab{\label{w2} -\frlap w(x)\leq \tau (1+w(x)).}
\end{lemma} 
Taking $w$ as the barrier introduced in the above Lemma, thanks to \eqref{kuvr2} and to the estimate in \eqref{w2}, we have that
	\bgs{ \mathcal {K}  (u-v,B_R) \leq \al  -c  \int_{B_R\cap \{ \theta_{\star}>u>w\} } (1+w)(u-w) \,dx \\ \al + \frac{1}{c} \int_{B_R\cap \{ u>\max\{\theta_{\star},w\} \} } (1+w)\, dx  \\ \al + 		2\tau \int_{B_R\cap\{u>w\}} (u-w)(1+w) \, dx .}
Let then $\tau=\frac{c}2$, and we are left with
	\bgs{ \mathcal {K}  (u-v,B_R) \leq  \al c \int_{B_R\cap \{ u>\max\{\theta_{\star},w\} \}} (u-w)(1+w)\, dx \\ \al + \frac{1}{c} \int_{B_R\cap \{ u>\max\{\theta_{\star},w\} \} } (1+w)\, dx\\
	\leq \al C_1 \int_{B_R\cap \{ u>\max\{\theta_{\star},w\} \} } (1+w)\, dx ,}
	with $C_1$ depending on $c$ (hence on $W$).
Using again Lemma \ref{bW1}, in particular the right hand side inequality in \eqref{w1}, we have that
	\[ 	  \mathcal {K}  (u-v,B_R) \leq C_1 \cdot C \int_{B_R\cap \{ u>\max\{\theta_{\star},w\} \} } (R+1-|x|)^{-2s}.\]
We set
	\eqlab{ \label {vrrr} V(R):= |B_R\cap \{ u>\theta_{\star} \} | } and the Co-Area formula then gives
	\eqlab{ \label{strange1}  \mathcal {K}  (u-v,B_R) \leq C_2 \int_0^R (R+1-t)^{-2s} V'(t)\, dt,}
	where $C_2$ possibly depends on $n,s,W$.
	
	 We use now the Sobolev inequality \eqref{OK:sob:p} for $p=2$, applied to $u-v$ (recalling that the support of $u-v$ is a subset of $B_R$) to obtain that
	 	\eqlab{\label{sobk1} \mathcal K(u-v,B_R) \al =  \mathcal K(u-v,\Rn) = \iint_{\Rn \times \Rn} \frac{ |(u-v)(x)-(u-v)(y)|^2}{|x-y|^{n+2s}}\, dx\, dy \\
	 	\geq \al  \tilde C \|u-v\|^2_{L^{\frac{2n}{n-2s}}(\Rn)} = \tilde C \|u-v\|^2_{L^{\frac{2n}{n-2s}}(B_R)}.}
From \eqref{w1} one has that
	\[ w(x)\leq C(R+1-|x|)^{-2s} -1.\] We fix $K$ large enough so as to have $R\geq 2K$ and in $B_{R-K}$ 
		\[ w(x) \leq C(1+K)^{-2s} -1 \leq -1 + \frac{1+\theta_{\star}}2.\]
		Therefore in $B_{R-K}\cap \{u>\theta_{\star} \}$  we have that
		\[ |u-v|\geq u-w \geq u+1- \frac{1+\theta_{\star}}2\geq \frac{1+\theta_{\star}}2 .\] Using definition \eqref{vrrr}, this leads to
			\bgs{ \|u-v\|^2_{L^{\frac{2n}{n-2s}}(B_R)} =\al  \left( \int_{B_R} |u-v|^{\frac{2n}{n-2s}} dx\right) ^{\frac{n-2s}{n}}
			\geq   \left( \frac{1+\theta_{\star}}2\right)^{\frac{2n}{n-2s}} \left( \int_{ B_{R-K}\cap \{u>\theta_{\star} \} } dx \right)^{\frac{n-2s}{n}} \\
			\geq\al C_3 V(R-K)^{\frac{n-2s}{n}}.}
In \eqref{sobk1}  we thus have  	\[ \mathcal K (u-v,B_R) \geq \tilde C_3 V(R-K)^{\frac{n-2s}{n}}\] and from \eqref{strange1} it follows that
	\bgs{ C_4 V(R-K)^{\frac{n-2s}{n}} \leq \int_0^R (R+1-t)^{-2s} V'(t) \, dt.}
Let $R\geq \rho \geq 2K$. Integrating the latter integral from $\rho$ to $\displaystyle \frac{3\rho}2$ we have that
	\bgs{ \al C_4 \frac{\rho}2 V(\rho-K)^{\frac{n-2s}{n}} \leq  C_4   \int_\rho^{\frac{3\rho}2} V(R-K)^{\frac{n-2s}{n}} \, dR \\
		\leq \al \int_0^{\frac{3\rho}2}\left(\int_0^R (R+1-t)^{-2s} V'(t) \, dt\right) \, dR 	
		= \int_0^{\frac{3\rho}2} V'(t)  \left(\int_0^{\frac{3\rho}2} (R+1-t)^{-2s} \, dR\right)\, dt \\
		=\al\int_0^{\frac{3\rho}2} V'(t)  \frac{ \left(\frac{3\rho}2+1-t\right)^{1-2s}-1}{1-2s} \, dt.}
Since $1-2s>0$, one has for large $\rho$ that $\displaystyle \left(\frac{3\rho}2+1-t\right)^{1-2s}-1 \leq (2\rho)^{1-2s}$, hence, noticing that the function $V$ is nondecreasing,
	\bgs{  \frac{\rho}2 V(\rho-K)^{\frac{n-2s}{n}} \leq  C_5 \rho^{1-2s} \int_0^{2\rho} V'(t)\, dt
	\leq  C_5 \rho^{1-2s} V(2\rho).}
	Therefore 	\eqlab{ \label{rhovk1} \rho^{2s} V(\rho-K)^{\frac{n-2s}{n}} \leq2 C_5 V(2\rho).}
Now we use an inductive argument as in Lemma 3.2 in \cite{SV14}, that we recall here:
\begin{lemma} \label{indargl} Let $\sigma, \mu\in (0,\infty), \nu \in (\sigma,\infty) $ and $\gamma, R_0,C\in (1,\infty)$. Let $V \colon (0,\infty)\to (0,\infty)$ be a nondecreasing function. For any $r\in [R_0,\infty)$, let $\alpha(r) := \min\left\{ 1, \displaystyle \frac{ \log V(r)} {\log r} \right\} $. Suppose that $V(R_0)>\mu$  and 	\[ r^\sigma \alpha(r) V(r)^{\frac{\nu-\sigma}{\nu}} \leq C V(\gamma r),\]
for any $r\in [R_0,\infty)$. Then there exist $c\in (0,1)$ and $R_{\star} \in [R_0,\infty)$, possibly depending on $\mu,\nu, \gamma, R_0, C$ such that 
\[ V(r) >c r^{\nu} ,\]
for any $r\in [R_{\star},\infty)$. 
\end{lemma}

For $R$ large, one obtains from \eqref{rhovk1} and Lemma \ref{indargl} that
	\[ V(R) \geq c_0 R^n,\] for a suitable $c_0\in(0,1)$.
Let now \[ \theta^{\star} :=\max \{ \theta_1, \theta_2, -1+c\} .\] We have that
	\eqlab { \label{uteta}|\{ u>\al  \theta^{\star}\} \cap B_R | + |\{ \theta_{\star}<u<\theta^{\star}\}\cap B_R | 
	=  |\{ u>\theta_{\star}\} \cap B_R |
	 =  V(R)\geq c_0 R^n.}
Moreover, from \eqref{acenest1} we  have that for some $\overline c>0$
		\[ \E(u,B_R) \leq \overline c R^{n-2s} ,\] therefore
			\bgs{ \overline c R^{n-2s} \geq \al \E(u,B_R) \geq  \int_{ \{ \theta_{\star}<u<\theta^{\star}\}\cap B_R} W(u)\, dx 
			\geq   \inf_{t\in (\theta_{\star},\theta^{\star})}W(t)\, |\{ \theta_{\star}<u<\theta^{\star}\}\cap B_R |  .}
			From this and \eqref{uteta} we have that
				\bgs{ c_0 R^n \leq \overline C R^{n-2s} + |\{ u>\theta^{\star}\} \cap B_R |  ,}
				and finally \bgs{ |\{ u>\theta^{\star}\} \cap B_R |  \geq C R^n,}
				with $C$ possibly depending on $n,s,W$. This concludes the proof of Theorem \ref{TY78UU} in the case $s\in (0,1/2)$.
\end{proof}

\section{A nonlocal version of a conjecture by De Giorgi} \label{sbsdg}

In this section we consider the fractional 
counterpart of
the conjecture by De Giorgi that was discussed before in the classical case.
Namely,
we consider the nonlocal Allen-Cahn equation 
	\[ \frlap u + W'(u)=0 \quad \mbox{in} \quad \Rn, \]
where $W$ is a double-well potential, and $u$ is smooth, bounded and monotone in one direction, namely $|u|\leq 1$ and $\partial_{x_n} u >0$. We wonder if it is also true, at least in low
dimension, that $u$ is one-dimensional. In this case, the conjecture was initially proved for $n=2$ and $s=\frac{1}{2}$ in \cite{CM05}. In the case $n=2$, for any $s \in (0,1)$, the result is proved using the harmonic extension of the fractional Laplacian in \cite{CS15} and \cite{SV09}. For $n=3$, the proof can be found in \cite{CC10} for $s\in \Big[\frac{1}{2},1\Big]$. The conjecture is still open for $n=3$ and  $s\in \Big[0,\frac{1}{2}\Big]$ and for $n\geq 4$. Also, the Gibbons conjecture (that is, the 
De Giorgi conjecture with the additional condition that the limit in \eqref{limdgs} is uniform) is also true for any $s \in (0,1)$
and in any dimension~$n$, see~\cite{INDIANA}.
 
To keep the discussion as simple as possible,
we focus here on the case $n=2$ and any $s\in (0,1)$, providing an alternative proof that does not make use of the harmonic extension. This part is completely new and not available in the literature. The proof is indeed quite
general and it will be further exploited in \cite{CV15}.
 
 We define (as in \eqref{kenac}) the total energy of the system to be 
 \begin{equation}\label{5.17bis}
\E(u, B_R) =  \mathcal{K}_R(u) + \int_{B_R} W(u) dx,\end{equation}
 where the kinetic energy is 
\begin{equation}  \label{dgkenac} {\mathcal{K}}_R(u):=\frac{1}{2}
\iint_{Q_R} \frac{|u(x)-u(\bar x)|^2}{|x-\bar x|^{n+2s}}\,dx\,d\bar x,\end{equation}
and~$Q_R:= \R^{2n}\setminus (\C B_R)^2= (B_R\times B_R)\cup (B_R\times\C B_R)\cup
(\C B_R\times B_R)$. We recall that the kinetic energy can also be written as
	\begin{equation} \label{dgkr}  {\mathcal{K}}_R(u) = \frac{1}{2} u(B_R,B_R) + u(B_R,\C B_R),\end{equation}
	where for two sets $A,B$ \begin{equation} \label{dguab} u(A,B)= \int_A\int_B \frac{|u(x)-u(\bar x)|^2}{|x-\bar x|^{n+2s}}\,dx\,d\bar x.\end{equation}
	
The main result of this section is the following. 

\begin{theorem}\label{dgdim2}
Let $u$ be a minimizer of the energy defined in \eqref{5.17bis}
in any ball of $\R^2$. Then $u$ is $1$-D, i.e.  there exist~$\omega \in S^{1}$ and $u_0: \R\to \R$ such that 	
	\[ u(x)=u_0(\omega \cdot x) \quad \mbox{for any} \quad x\in \R^2.\] 	
\end{theorem}

The proof relies on the following estimate for the kinetic energy, that we prove
by employing a domain deformation technique.   

\begin{lemma} \label{endg}
Let~$R>1$, $\varphi\in C^\infty_0(B_1)$. Also, for any~$y\in\R^n$, let
\begin{equation}\label{5252}
\Psi_{R,+}(y):=y+\varphi \Big(\frac{ y}{R}\Big)\,e_1
\ {\mbox{ and }} \ \Psi_{R,-}(y):=y-\varphi \Big(\frac{ y}{R}\Big)\,e_1.\end{equation}
Then, for large~$R$, the maps~$\Psi_{R,+}$ and~$\Psi_{R,-}$
are diffeomorphisms on~$\R^n$.
Furthermore, if we define~$u_{R,\pm}(x):= u(\Psi_{R,\pm}^{-1}(x))$,
we have that
\begin{equation}\label{DG01}
{\mathcal{K}}_R (u_{R,+})+{\mathcal{K}}_R (u_{R,-})-2
{\mathcal{K}}_R (u)\le \frac{C}{R^2}{\mathcal{K}}_R (u),
\end{equation}
for some~$C>0$. 
\end{lemma}

\begin{proof} First of all, we compute the Jacobian of~$\Psi_{R,\pm}$.
For this, we write~$\Psi_{R,+,i}$ to denote the~$i^{\text{th}}$
component
of the vector~$\Psi_{R,+}=(\Psi_{R,+,1},\cdots,\Psi_{R,+,n})$
and we observe that
\begin{equation}\label{JA}
\frac{\partial \Psi_{R,+,i}(y)}{\partial y_j}=
\frac{\partial }{\partial y_j} \Big(y_i\pm 
\varphi \Big(\frac{ y}{R}\Big) \delta_{i1} \Big) = \delta_{ij} \pm \frac{ 1}{R}
\partial_j\varphi \Big(\frac{  y}{R}\Big) \delta_{i1}.\end{equation}
The latter term is bounded by~${\mathcal{O}}(R^{-1})$, and
this proves that~$\Psi_{R,\pm}$ is a diffeomorphism if~$R$ is large enough.

For further reference, we point out that if~$J_{R,\pm}$
is the Jacobian determinant of~$\Psi_{R,\pm}$, then
the change of variable
\begin{equation}\label{JA0}
x:=\Psi_{R,\pm}(y),\qquad
\bar x:=\Psi_{R,\pm}(\bar y)\end{equation}
gives that
\begin{equation*}
	\begin{split}
		dx\,d\bar x \;=\;& J_{R,\pm}(y)\,J_{R,\pm}(\bar y)\,dy\,d\bar y \\
		 =\;& 		\bigg(1\pm  \Big(\frac{ 1}{R} \Big)\partial_1 \varphi \Big(\frac{ y}{R}\Big) +{\mathcal{O}}\Big(\frac{ 1}{R^2}\Big)\bigg)
		\bigg(1\pm  \frac{ 1}{R}\partial_1 \varphi \Big(\frac{ \bar y}{R} \Big)+{\mathcal{O}}\Big(\frac{ 1}{R^2}\Big) \bigg) dy d\bar y
		\\  \;=\;& 1\pm  \frac{ 1}{R} \partial_1 \varphi \Big(\frac{ y}{R}\Big) \pm   \frac{ 1}{R}
		\partial_1 \varphi \Big(\frac{ \bar y}{R} \Big) +{\mathcal{O}}\Big(\frac{ 1}{R^2}\Big)	
		\,dy\,d\bar y
		,\end{split}
\end{equation*}
thanks to~\eqref{JA}. Therefore
\begin{equation}\label{JA8}
\begin{split}
& \frac{|u_{R,\pm}(x)-u_{R,\pm}(\bar x)|^2}{|x-\bar x|^{n+2s}}\,dx\,d\bar x\\
&\quad=\frac{\big|u(\Psi_{R,\pm}^{-1}(x))-
u(\Psi_{R,\pm}^{-1}(\bar x))\big|^2}{
|\Psi_{R,\pm}^{-1}(x)-\Psi_{R,\pm}^{-1}(\bar x)|^{n+2s}}\cdot
\left(
\frac{ |x-\bar x|^2
}{|\Psi_{R,\pm}^{-1}(x)-\Psi_{R,\pm}^{-1}(\bar x)|^2}\right)^{-\frac{n+2s}{2}}\,dx\,d\bar x
\\ &\quad= \frac{|u(y)-u(\bar y)|^2}{|y-\bar y|^{n+2s}} \cdot \left( 
\frac{ \Big|\Psi_{R,\pm}(y)- \Psi_{R,\pm}(\bar y)\Big|^2 }{|y-\bar 
y|^2}\right)^{-\frac{n+2s}{2}} \\&\qquad\cdot \Bigg( 1\pm  \frac{ 1}{R} \partial_1 \varphi \Big(\frac{ y}{R}\Big) \pm   \frac{ 1}{R}
		\partial_1 \varphi \Big(\frac{ \bar y}{R} \Big) +{\mathcal{O}}\Big(\frac{ 1}{R^2}\Big)	\Bigg)
		\,dy\,d\bar y. \end{split} \end{equation} Now, for 
any~$y$, $\bar y\in\R^n$ we calculate 
	\begin{equation}
		\label{J1}
			\begin{split} 
			&\Big|\Psi_{R,\pm}(y)-\Psi_{R,\pm}(\bar y)\Big|^2\\
			&\quad= \Big| (y-\bar y)\pm \bigg(\varphi \Big(\frac{y}{R}\Big)-\varphi \Big(\frac{\bar y}{R}\Big)\bigg) \,e_1\Big|^2 \\
			 &\quad= |y-\bar y|^2 +\bigg|\varphi \Big(\frac{y}{R}\Big)-\varphi \Big(\frac{\bar y}{R}\Big)\bigg|^2 \pm 2 \bigg(\varphi \Big(\frac{y}{R}\Big) -\varphi \Big(\frac{\bar y}{R}\Big) \bigg) \,(y_1-\bar y_1). 					\end{split}
	\end{equation}
 Notice also  that 
 	\begin{equation}\label{j9} 
 		\bigg|\varphi \Big(\frac{y}{R}\Big)-\varphi \Big(\frac{\bar y}{R}\Big)\bigg| \le \frac{1}{R} \|\varphi\|_{C^1(\R^n)} |y-\bar y|,
	\end{equation} 
hence~\eqref{J1} becomes
\[ \frac{ \Big|\Psi_{R,\pm}(y)-\Psi_{R,\pm}(\bar y)\Big|^2 }{|y-\bar y|^2} 
=1+\eta_{\pm}\]
 where
  \begin{equation}\label{JA7} \eta_{\pm}:= \frac{ \bigg|\varphi \Big(\frac{y}{R}\Big)-\varphi \Big(\frac{\bar y}{R}\Big)\bigg|^2}{|y-\bar y|^2}\pm 2\frac{ \bigg(\varphi \Big(\frac{y}{R}\Big)-\varphi \Big(\frac{\bar y}{R}\Big)\bigg)  \,(y_1-\bar y_1)}{|y-\bar y|^2}={\mathcal{O}}\Big(\frac{1}{R}\Big).
\end{equation} 

As a consequence 
\[  \left( \frac{ \Big|\Psi_{R,\pm}(y)- \Psi_{R,\pm}(\bar y)\Big|^2 }{|y-\bar y|^2}\right)^{-\frac{n+2s}{2}}\\ = 
(1+\eta_\pm)^{-\frac{n+2s}{2}}= 1-\frac{n+2s}{2} \eta_\pm +{\mathcal{O}}(\eta_\pm^2). \]
 We plug this information into~\eqref{JA8} and use~\eqref{JA7} to 
obtain 
\begin{eqnarray*} && \frac{|u_{R,\pm}(x)-u_{R,\pm}(\bar x)|^2}{|x-\bar x|^{n+2s}}\,dx\,d\bar x\\ 
&=& \frac{|u(y)-u(\bar y)|^2}{|y-\bar y|^{n+2s}} \cdot \left(1-\frac{n+2s}{2} \eta_\pm +{\mathcal{O}}\Big(\frac{1}{R^2}\Big) \right) \\
&&\cdot \bigg(1\pm \frac{1}{R}\partial_1 \varphi \Big(\frac{y}{R}\Big)\pm \frac{1}{R} \partial_1 \varphi \Big(\frac{\bar y}{R} \Big)+{\mathcal{O}}\Big(\frac{1}{R^2}\Big)\bigg) 
\,dy\,d\bar y\\ 
&=& \frac{|u(y)-u(\bar y)|^2}{|y-\bar y|^{n+2s}} \cdot \Bigg[ 1-\frac{n+2s}{2} \eta_\pm + \, \bigg(  \pm \frac{1}{R} \partial_1 \varphi \Big(\frac{ y}{R} \Big)   \pm \frac{1}{R} \partial_1 \varphi \Big(\frac{\bar y}{R}\Big) \bigg) \\
	&&+{\mathcal{O}}\Big(\frac{1}{R^2}\Big)\Bigg]\,dy\,d\bar y.\end{eqnarray*}
Using this and the fact that $$ \eta_+ \,+\, \eta_-= 2\,\frac{ \bigg|\varphi \Big( \frac{y}{R}\Big)-\varphi \Big(\frac{\bar y}{R}\Big)\bigg|^2}{|y-\bar y|^2}={\mathcal{O}}\Big(\frac{1}{R^2}\Big),$$ thanks 
to~\eqref{j9},
we obtain \begin{eqnarray*} && 
\frac{|u_{R,+}(x)-u_{R,+}(\bar x)|^2}{|x-\bar x|^{n+2s}}+ 
\frac{|u_{R,-}(x)-u_{R,-}(\bar x)|^2}{|x-\bar x|^{n+2s}} \,dx\,d\bar 
x\\ &=& \frac{|u(y)-u(\bar y)|^2}{|y-\bar y|^{n+2s}} \cdot \left( 2 
+{\mathcal{O}}\Big(\frac{1}{R^2}\Big)\right)
\,dy\,d\bar y 
.\end{eqnarray*}
Thus, if we integrate over~$Q_R$
we find that
$$ {\mathcal{K}}_R(u_{R,+})+{\mathcal{K}}_R(u_{R,-})
= 2{\mathcal{K}}_R(u) +
\iint_{Q_R} {\mathcal{O}}\Big(\frac{1}{R^2}\Big)\,
\frac{|u(x)-u(\bar x)|^2}{|x-\bar x|^{n+2s}}\,dx\,d\bar x.$$
This establishes~\eqref{DG01}.
\end{proof}

 \begin{proof}[Proof of Theorem \ref{dgdim2}]
We organize this proof into four steps. 
 
\noindent \textbf{Step 1.} \textbf{A geometrical consideration}\\
In order to prove that the level sets are flat, it suffices to prove that $u$ is monotone in any direction.  Indeed, if $u$ is monotone in any direction, the level set  $\{ u=0\}$ is both convex and concave, thus it is flat. 
 \bigskip
 
\noindent  \textbf{Step 2.} \textbf{Energy estimates}\\
 	 Let $\varphi \in C_0^{\infty}(B_1)$ such that $\varphi = 1 $ in $B_{1/2}$, and let $e=(1,0)$. We define as in Lemma \ref{endg} 
		\[ \Psi_{R,+}(y):=y+\varphi \Big(\frac{ y}{R}\Big)\,e \ {\mbox{ and }} \ \Psi_{R,-}(y):=y-\varphi \Big(\frac{ y}{R}\Big)\,e,\]
which are diffeomorphisms for large $R$, and the functions ~$u_{R,\pm}(x):= u(\Psi_{R,+}^{-1}(x))$. Notice that 
	 \begin{align}
		&u_{R,+} (y)= u(y) \; &\text{for} \; &y \in \C B_R \label{urpiur1}\\
		&u_{R,+} (y)= u(y-e) \; &\text{for} \; &y \in B_{R/2}\label{urpiur2} .
		\end{align} 
     By computing the potential energy, it is easy to see that
    	\[ \begin{split}
    	\int_{B_R} W(u_{R,+}(x)) \, dx &+\int_{B_R} W(u_{R,-}(x)) \, dx - 2 \int_{B_R} W(u(x)) \, dx \\
    		&\leq \frac{C}{R^2}  \int_{B_R} W(u(x)) \, dx.\end{split}\]
Using this and~\eqref{DG01}, we obtain
the following estimate for the total energy
    	\begin{equation}\label{5.29bis}
\E(u_{R,+},B_R)+\E(u_{R,-},B_R) - 2\E(u,B_R) \leq \frac{C}{R^2} \E(u,B_R).\end{equation}  
Also, since $u_{R,\pm}=u$ in $\C B_R$, we have that
    	\[ \E(u,B_R) \leq \E(u_{R,-},B_R).\]
This and~\eqref{5.29bis} imply that
    		\begin{equation} \label {bla13} \E(u_{R,+},B_R) -\E(u,B_R) \leq \frac{C}{R^2} \E(u,B_R).\end{equation}
As a consequence of this estimate and~\eqref{THANKS},
it follows that
  		\begin{equation} \label{dge11} \lim_{R \to+ \infty} \Big(\E(u_{R,+},B_R) -\E(u,B_R) \Big)=0.
		\end{equation}
		
\bigskip 

  \noindent  \textbf{Step 3.} \textbf{Monotonicity}\\
     We claim that $u$ is monotone. Suppose by contradiction that $u$ is not monotone. 
     That is, up to translation and dilation, we suppose that the value of~$u$ at the 
     origin stays above the values of~$e$ and~$-e$, with $e:=(1,0)$, i.e. 
     \[ u(0) > u(e) \;{\mbox{ and }}\;u(0)>u(-e).\]
     Take $R$ to be large enough, say $R>8$. Let now 
  	\begin{equation}\label{5.31bis}
v_R(x):= \min \big\{u(x),u_{R,+}(x)\big\} \quad \mbox{and} \quad w_R(x):= \max \big\{u(x),u_{R,+}(x)\big\} .\end{equation}
  	By \eqref{urpiur1} we have that $v_R =w_R =u $ outside $B_R$. Then, since $u$ is a minimizer in $B_R$ and $w_R=u$ outside $B_R$, we have that
  		\begin{equation}\label{oPGHll} \E(w_R,B_R) \geq \E(u,B_R) .\end{equation}
Moreover, the sum of the energies of
the minimum and the maximum is less than or equal to the sum
of the original energies: this is obvious in the local case,
since equality holds, and in the nonlocal case the proof
is based on the inspection of the different integral contributions, see e.g.
formula~(38) in~\cite{PSV13}. So we have that
\[  \E(v_R,B_R) + \E(w_R,B_R) \leq\E(u,B_R) +\E(u_{R,+},B_R)\]
hence, recalling~\eqref{oPGHll},
	 \begin{equation} \label{dge12}   \E(v_R,B_R)\leq   \E(u_{R,+},B_R) .\end{equation} 
  	
We claim that~$v_R$ is not identically neither $u$, nor $u_{R,+}$. Indeed, since $u(0)= u_{R,+}(e)$ and $u(-e)= u_{R,+}(0)$ we have that
 	\[\begin{split}	 v_R(0) \;=&\;\min \big\{u(0),u_{R,+}(0)\big\}=
\min\big\{u(0),u(-e)\big\}\\ =&\;u(-e) = u_{R,+}(0) <u(0)  \quad \mbox{and} \\
 					v_R(e) \;= &\;\min \big\{u(e),u_{R,+}(e)\big\}=\min\big\{u(e),u(0)\big\}\\
=&\;u(e) <u(0) = u_{R,+}(e).  \end{split}\]
 					By continuity of $u$ and $u_{R,+}$, we have that
 					\begin{equation} \label{dgcuur1} \begin{split} v_R \;=&\;u_{R,+}<u  \mbox{ in a neighborhood of } 0 \quad \mbox{and} \\
 									  v_R \;=&\;u <u_{R,+} \mbox{ in a neighborhood of } e .\end{split}\end{equation}
 We focus our attention on the energy in the smaller ball $B_2$. We claim that $v_R$ is not minimal for $\E(\cdot, B_2)$. Indeed, if $v_R$ were minimal in $B_2$, then on $B_2$ both $v_R$ and $u$ would satisfy the same equation. However, $v_R \leq u$ in $\R^2$ by definition and $v_R=u$ in a neighborhood of $e$ by the second statement in \eqref{dgcuur1}. The Strong Maximum Principle implies that they coincide everywhere, which contradicts the first line in \eqref{dgcuur1}. 

Hence $v_R$ is not a minimizer in $B_2$. Let then $v^*_R$ be a minimizer of $\E(\cdot, B_2)$, that agrees with $v_R$ outside the ball $B_2$, and we define the positive quantity
 	\begin{equation}\delta_R: = \E(v_R,B_2) -  \E(v^*_R,B_2).  \label{dgclmc1} \end{equation} 
  
 We claim that 
 \begin{equation}\label{remains}
{\mbox{as $R$ goes to infinity, $\delta_R$ remains bounded away from
zero.}}\end{equation} 
To prove this, we assume by contradiction that \begin{equation} \label{dgclaim2} \displaystyle \lim_{R \to+ \infty} \delta_R=0.\end{equation} 
Consider $\tilde u$ to be the translation of $u$, that is~$\tilde u(x):= u(x-e)$.
Let also
\[ m (x):= \min \big\{ u(x), \tilde u(x)\big\}.\] 
We notice that in $B_{R/2}$ we have that $\tilde u (x)= u_{R,+}(x)$. 
This and~\eqref{5.31bis}
give that
\begin{equation}\label{5.31ter}
{\mbox{$m=v_R$ in $B_{R/2}$.}}\end{equation}
Also, from \eqref{dgcuur1} and~\eqref{5.31ter},
it follows that $m$ cannot be identically neither $u$ nor $\tilde u$,
and
\begin{equation} \label{dgcuur2} \begin{split} m \;<&\; u  \mbox{ in a neighborhood of } 0 \quad \mbox{and} \\
 									  m \;=&\;u  \mbox{ in a neighborhood of } e .\end{split}\end{equation}
Let $z$ be a competitor for $m$ in the ball $B_2$, that agrees with $m$ outside $B_2$. We take a
cut-off function $\psi\in C_0^{\infty}(\Rn)$ such that $\psi = 1 $ in $B_{R/4}$, $\psi =0$ in $\C B_{R/2}$. Let \[z_R(x):= \psi(x) z(x) + \big(1- \psi(x)\big) v_R(x).\]  Then we have that $z_R=z $ on $B_{R/4}$ and 
\begin{equation}\label{STAR}
{\mbox{$z_R=v_R$ on $\C B_{R/2}$. }}\end{equation}
In addition, by~\eqref{5.31ter}, we have that~$z=m=v_R$
in~$B_{R/2}\setminus B_2$. So, it follows that
	\[ z_R(x)=\psi(x)v_R(x)+(1-\psi(x))v_R(x)= v_R(x)  =z(x)\quad \mbox{on} \quad B_{R/2}\setminus B_2 .\]
This and~\eqref{STAR} imply that
$z_R=v_R$ on $\C B_2$. 

We summarize
in the next lines these useful identities (see also Figure \ref{fign:Enest}). 
	\begin{align*}
		&\text{in }\; B_2  & & u_{R,+}= \tilde u, \quad m=v_R, \quad z=z_R\\
		&\text{in }\; B_{R/2}\setminus B_2  & & u_{R,+}= \tilde u, \quad  v^*_R=v_R=m=z=z_R \\
		&\text{in }\; B_R \setminus B_{R/2}  & & v^*_R=v_R=z_R,\quad  m=z\\	
		&\text{in }\; \C B_R   & &  u_{R,+}=u=v_R=v^*_R=z_R, \quad m=z.	
	\end{align*}		
%\begin{center}
\begin{figure}[htpb]
%\sidecaption
%	\hspace{0.8cm}
%	\begin{minipage}[b]{0.95\linewidth}
%	\centering
	\includegraphics[width=0.80\textwidth]{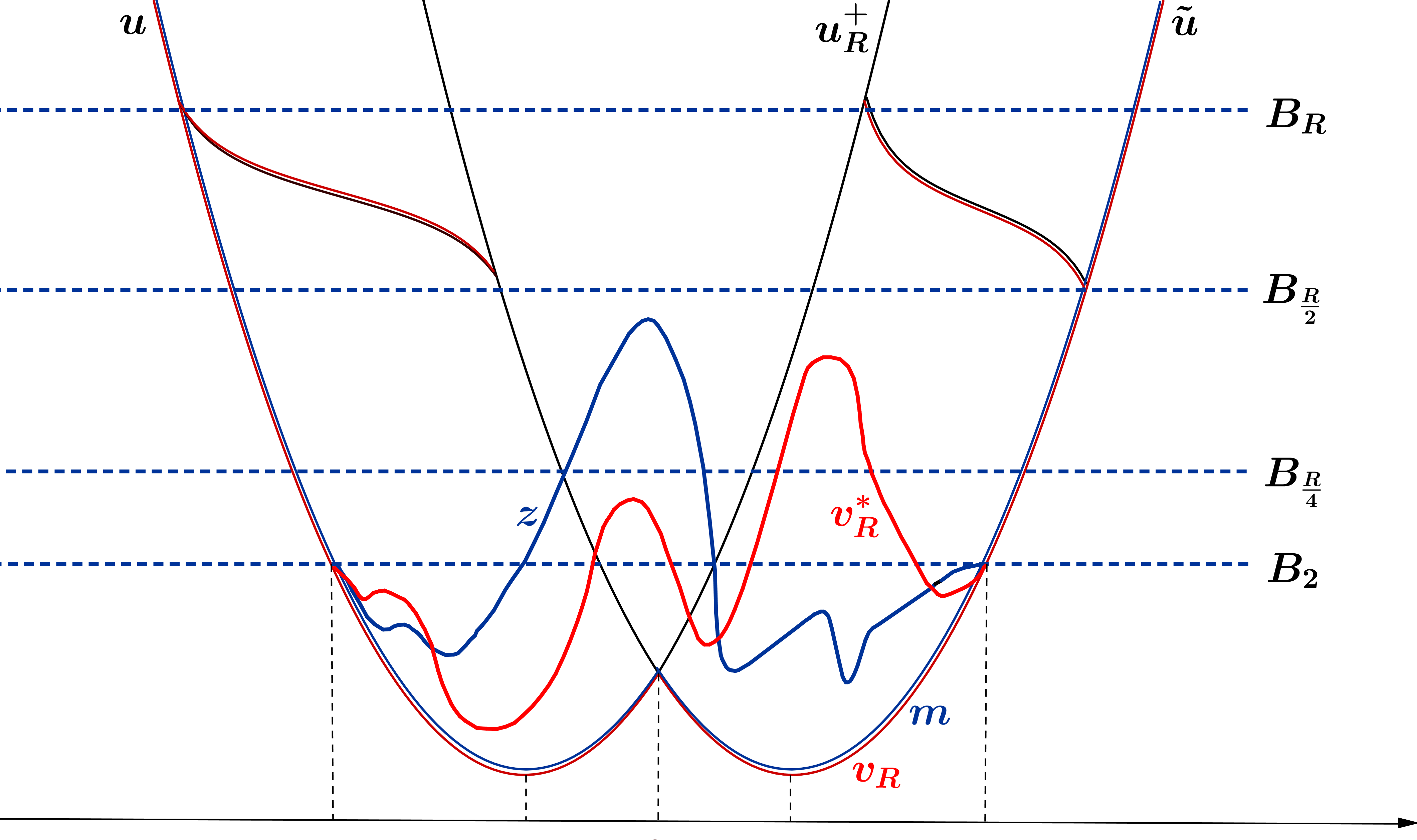}
	\caption{Energy estimates}  
	\label{fign:Enest}
%	\end{minipage}
	\end{figure} 
%	\end{center}										
We compute now
		\[ \begin{split} \E(m&\;,B_2) -\E(z,B_2) \\
		= &\; \E(m,B_2)-\E(v_R,B_2)+ \E(v_R,B_2)-\E(z_R,B_2) +\E(z_R,B_2)-\E(z,B_2) .\end{split}\]  \\
By the definition of $\delta_R$ in~\eqref{dgclmc1}, we have that
	\begin{equation}\label{dgeqc}\begin{split} \E(m&\;,B_2) -\E(z,B_2) \\							
				= &\;	\E(m,B_2)-\E(v_R,B_2)+\delta_R + \E(v^*_R,B_2) -\E(z_R,B_2)+\E(z_R,B_2)-\E(z,B_2).
				\end{split} \end{equation}							
Using the formula for the kinetic energy given in \eqref{dgkr} together with \eqref{dguab} we have that 
\[\begin{split} \E(m&\; ,B_2)-\E(v_R,B_2) \\
= &\; \frac{1}{2} m(B_2,B_2)+ m(B_2,\C B_2) + \int_{B_2} W\big(m(x)\big)\, dx \\
&\;- \frac{1}{2} v_R(B_2,B_2) -v_R(B_2,\C B_2) - \int_{B_2}W\big(v_R(x)\big) \, dx.\end{split} \]	
Since $m=v_R$ on  $B_{R/2}$ (recall~\eqref{5.31ter}), we obtain 
\[\begin{split} \E(m&\; ,B_2)-\E(v_R,B_2) \\	=&\; \int_{B_2} \, dx \int_{\C B_{R/2}} \, dy \frac{ |m(x)-m(y)|^2-|m(x)-v_R(y)|^2}{|x-y|^{n+2s}} . 
				\end{split} \]	
Notice now that~$m$ and $v_R$ are bounded on $\Rn$ (since so is
$u$).
Also, if~$x\in B_2$ and~$y\in \C B_{R/2}$ we have that~$|x-y|\ge|y|-|x|\ge |y|/2$ if~$R$ is large.
Accordingly,
	\begin{equation}\label{uno} \E(m ,B_2)-\E(v_R,B_2) \leq  C \int_{B_2} \, dx \int_{\C B_{R/2}}\frac{1}{|y|^{n+2s}} \, dy  \le
	CR^{-2s}, 
				\end{equation}
up to renaming constants.
Similarly, 
$z_R=z$ on $B_{R/2}$ and we have the same bound  
\begin{equation}\label{due}
\E(z_R,B_2)-\E(z,B_2) \leq  C R^{-2s}.  \end{equation}
Furthermore, since $v_R^*$ is a minimizer for $\E(\cdot, B_2)$ and $v_R^*=z_R$ outside of $B_2$, we have that
 \[ \E(v^*_R,B_2) -\E(z_R,B_2)\leq 0.\]   				  			
Using this, \eqref{uno} and~\eqref{due}        
in \eqref{dgeqc}, it follows that
 	\[ \E(m,B_2)-\E(z,B_2) \leq CR^{-2s} +\delta_R.\] 
Therefore, by sending~$R\to+\infty$ and using again~\eqref{dgclaim2}, we
obtain that
 	\begin{equation}\label{MKI} \E(m,B_2) \leq \E(z,B_2) .\end{equation}
We recall that $z$ can be any competitor for $m$, that coincides with $m$ outside of $B_2$. Hence, formula~\eqref{MKI} means that $m$ is a minimizer for $\E(\cdot, B_2)$.
On the other hand,
$u$ is a minimizer of the energy in any ball.
Then, both $u$ and $m$ satisfy the same equation in $B_2$. Moreover, they coincide in a neighborhood of $e$, as stated in the second line of \eqref{dgcuur2}. By the Strong Maximum Principle, they have to coincide on $B_2$, but this contradicts the
first statement of \eqref{dgcuur2}. The proof of~\eqref{remains} is thus complete.
\bigskip

Now, since $v_R^*=v_R$ on $\C B_2$, from definition \eqref{dgclmc1} we
have that 
 	\[ \delta_R  = \E(v_R,B_R)-\E(v_R^*,B_R) .\]
Also, $\E (v_R^*,B_R)\geq \E(u,B_R)$, 
thanks to the minimizing property of~$u$. 
Using these pieces of information and inequality \eqref{dge12}, it follows that 
  	\[ 		\delta_R  \leq  \E(u_{R,+},B_R) - \E(u,B_R) .\]
Now, by  sending $R \to +\infty$ and using~\eqref{remains}, we have that
	\[  \lim_{R \to +\infty} \E(u_{R,+},B_R) - \E(u,B_R) > 0 ,\] which contradicts \eqref{dge11}. This implies that indeed $u$ is monotone, and this concludes the proof of this Step.

\bigskip  					  						  						  					
 \noindent \textbf{Step 4.} \textbf{Conclusions}\\ 
         In Step 3, we have proved that $u$ is monotone, in any given direction~$e$.
Then, Step 1 gives the desired result.
This concludes the proof of Theorem \ref{dgdim2}.
 \end{proof}

We remark that the exponent two in the energy estimate \eqref{DG01} is related to the expansions of order two and not to the dimension of the space. Indeed, the energy estimates hold for any $n$. However, the two power in the estimate \eqref{DG01} allows us to prove the fractional
version of De Giorgi conjecture only in dimension two. In other words, the proof of Theorem \ref{dgdim2} is not applicable for $n> 2$. One can verify this by checking the limit in \eqref{dge11}
\bgs{ \lim_{R \to+ \infty} \Big(\E(u_{R,+},B_R) -\E(u,B_R) \Big)=0,}
which was necessary for the Proof of Theorem \ref{dgdim2} in the case $n=2$.
 We know from Theorem \ref{acenest1} that
	\[ \lim_{R \to+ \infty} \frac{C}{R^n} \E(u,B_R) =0.\]  Confronting this result with inequality \eqref{bla13}
	    		\bgs{ \E(u_{R,+},B_R) -\E(u,B_R) \leq \frac{C}{R^2} \E(u,B_R),}
 we see that we need to have $n=2$ in order for the the limit in \eqref{dge11} to be zero.  
\bigskip

Of course, the one-dimensional symmetry property in
Theorem \ref{dgdim2} is inherited by the spatial homogeneity
of the equation, which is translation and rotation invariant.
In the case, for instance, in which the potential also depends 
on the space variable, the level sets of the (minimal) solutions
may curve, in order to adapt themselves to the spatial inhomogeneity.

Nevertheless, in the case of periodic dependence, it
is possible to construct minimal solutions whose
level sets are possibly not planar, but still remain at a bounded distance
from any fixed hyperplane. As a typical result in this direction,
we recall the following one (we refer to~\cite{MATTEO-ENRICO}
for further details on the argument):
%\begin{center}
\begin{figure}[htpb]
%\sidecaption
%	\hspace{0.8cm}
%	\begin{minipage}[b]{0.95\linewidth}
%	\centering
	\includegraphics[width=0.70\textwidth]{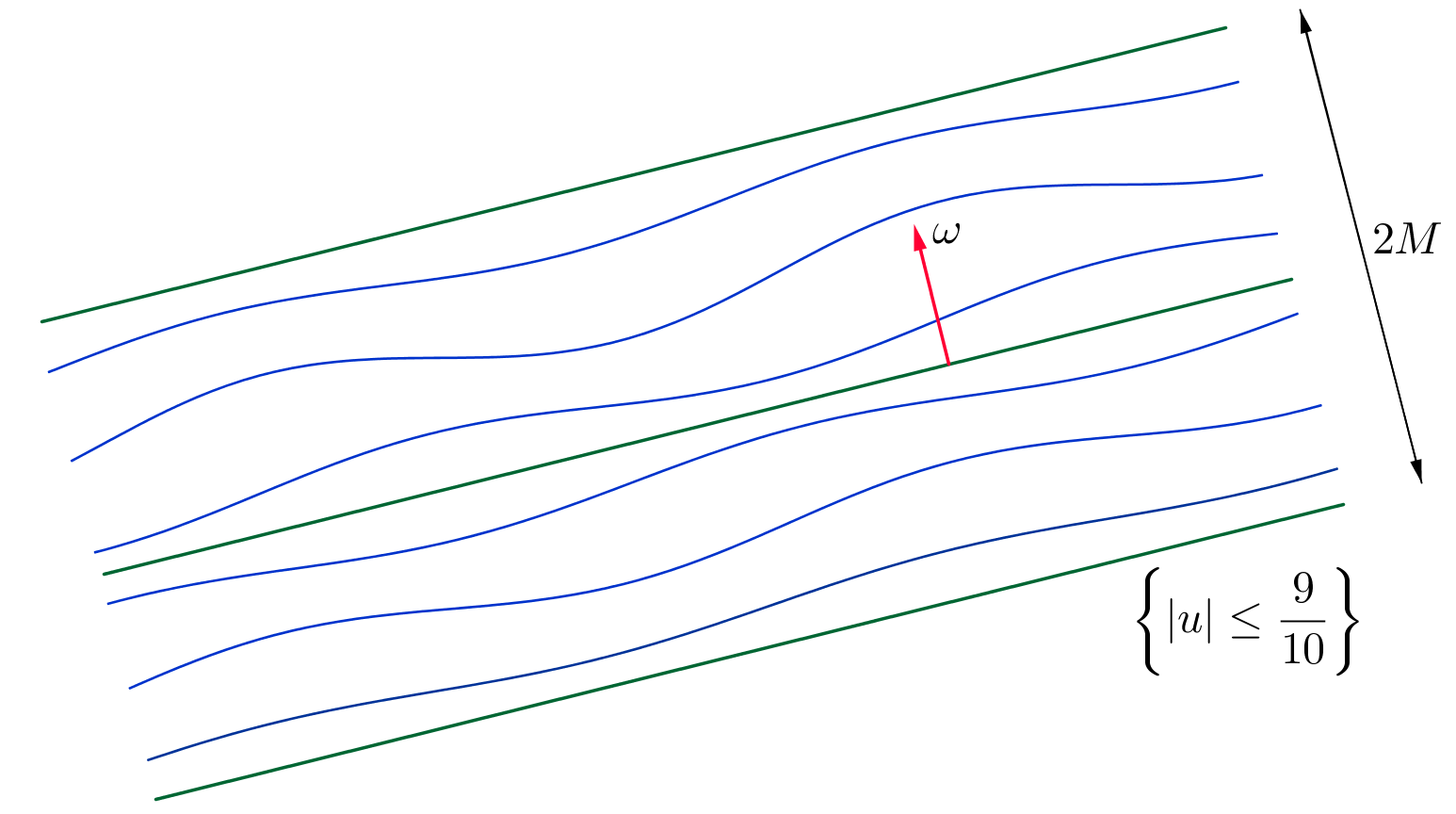}
	\caption{Minimal solutions in periodic medium}  
	\label{fign:minsperm}
%	\end{minipage}
	\end{figure} 
%	\end{center}
\begin{theorem}
Let~$Q_+>Q_->0$ and~$Q:\R^n\to [Q_-,Q_+]$. Suppose that~$Q(x+k)=Q(x)$
for any~$k\in \Z ^n$. Let us consider, in any ball~$B_R$, the energy
defined by
$$ \E(u, B_R) =  \mathcal{K}_R(u) + \frac14 \,\int_{B_R} Q(x)\,(1-u^2)^2 dx,$$
where the kinetic energy ${\mathcal{K}}_R(u)$ is defined as in~\eqref{dgkenac}.

Then, there exists a constant $M> 0$, 
such that, given any~$\omega \in \partial B_1$,
there exists a minimal solution~$u_\omega$ of 
$$ (-\Delta)^s u_\omega (x)= Q(x)\,(u_\omega(x)-u^3_\omega(x))
\qquad{\mbox{ for any }}x\in\R^n$$
for which the level sets $\{|u_\omega | \le\frac{9}{10}\}$ are contained in the 
strip~$\{x\in\R^n {\mbox{ s.t. }} |\omega\cdot x|\le M\}$.

Moreover, if $\omega$ is rotationally dependent,
i.e. if there exists~$k_o\in\Z^n$ such that~$\omega\cdot k_o=0$,
then~$u_\omega$ is periodic with respect to~$\omega$, i.e.
$$ u_\omega(x)=u_\omega(y) {\mbox{ for any $x$, $y\in\R^n$ such that 
$x-y=k$ and $\omega\cdot k=0$.}} $$
\end{theorem}

\chapter{Nonlocal minimal surfaces}\label{nlms}
\begin{abstract}{In this chapter, we introduce nonlocal minimal surfaces. We first discuss a Bernstein type result in any dimension, namely the property that an $s$-minimal graph  in $\R^{n+1}$ is flat (if no singular cones exist in dimension $n$) and prove that an $s$-minimal surface whose prescribed data is a subgraph, is itself a subgraph. The non-existence of nontrivial $s$-minimal cones in dimension $2$ is then proved. Moreover, some boundary regularity properties will be discussed at the end of this chapter: quite surprisingly, and differently from the classical case, nonlocal minimal surfaces do not always attain boundary data in a continuous way (not even in low dimension). A possible boundary behavior is, on the contrary, a combination of stickiness to the boundary and smooth separation from the adjacent portions.
Furthermore, in the last section we deal with the asymptotic behavior as $s\to 0^+$
of the fractional mean curvature, and with the behavior of $s$-minimal surfaces when  $s\in(0,1)$ is small in a bounded and connected open set with $C^2$ boundary $\Omega\subset \Rn$. We classify the behavior of $s$-minimal surfaces with respect to the fixed exterior data (i.e. the $s$-minimal set fixed outside of $\Omega$). So, for $s$ small and depending
on the data at infinity,
the $s$-minimal set can be either empty in $\Omega$, fill all $\Omega$, 
or possibly develop a wildly oscillating boundary. 
Also, we prove the continuity of the fractional mean curvature in all variables, for $s\in [0,1]$. Using this, we see that as the parameter $s$ varies, the fractional mean curvature may change sign. }
\end{abstract}

\bigskip 
\bigskip

In this chapter, we deal with nonlocal minimal surfaces, as introduced in \cite{nms} in 2010 (see also \cite{MILAN} for a preliminary introduction to some properties of nonlocal minimal surfaces). In particular, following the approach of De Giorgi (for classical minimal surfaces), we introduce the fractional perimeter and look for minimizers in bounded open sets with respect to some fixed exterior data. The boundaries of such (nonlocal minimal) sets are called nonlocal minimal surfaces (and are indeed smooth almost everywhere). We give in this chapter some notions on this subject, outline some nice recent achievements and also present a new result on a stickiness phenomena when the fractional parameter is small. So, in this Chapter \ref{nlms} 
\begin{itemize}
\item we prove that $s$-minimal graphs  in $\R^{n+1}$ are flat if no singular cones exist in dimension $n$ (and this is related to a known Bernstein problem),
\item we prove that minimizers with respect to the exterior data that is a subgraph, is a subgraph also inside the domain,
\item we prove that nontrivial minimal cones in dimension two do not exist (which implies, according to the first item, that $s$-minimal graphs in $\R^3$ are flat),
\item we discuss some nice examples of boundary regularity and stickiness phenomena.
\end{itemize}
In the last Section \ref{luke} we focus on the behavior of $s$-minimal surfaces for small values of the fractional parameter. In particular
\begin{itemize}
\item we give the asymptotic behavior of the fractional mean curvature as $s\to 0^+$,
\item we prove the continuity of the fractional mean curvature in all variables for $s\in [0,1]$,
\item when $s\in (0,1)$ is small we classify the behavior of $s$-minimal surfaces, in dependence of the exterior data at infinity.
\end{itemize}   
To give more details on the last item, we prove that when the fractional parameter is small and the exterior data at infinity occupies (in measure, with respect to the weight) less than half the space, then nonlocal minimal surfaces completely stick at the boundary (that is, they are empty inside the domain), or become ``topologically dense'' in their domain. An analogues result, that is nonlocal minimal surfaces fill the domain or become dense, is obtained when the exterior data occupies in the appropriate sense more than half the space (so this threshold is  optimal).

\medskip

Let $s\in (0,1)${\footnote{
%\begin{remark}\label{ssigma}
We point out that we use the fractional parameter $s$  differently from the previous (and the following) chapters. Indeed, it substitutes the $2s \in(0,2)$ power used up until now in the kernel defining our nonlocal operators. To give a more precise idea, let us denote $\sigma:=2s\in (0,2)$ and write our singular kernel kernel (check \eqref{frlap2def} or \eqref{PV-1}) as $|x-y|^{-n-\sigma}$. In the present chapter, the important thing is that $\sigma$ will take into account only half of the interval of definition, that is $\sigma\in (0,1)$, and this is equivalent to having (the original) $s\in (0,1/2)$. As a notation, we nonetheless writes $s$ instead of $\sigma$, hence we take $s\in (0,1)$. This is just a matter of notation, however we will make clear why we need to take the power in the kernel smaller than $1$ and not up until $2$, in the upcoming Theorem \ref{TH1-SV-GAMMA}.
%\end{remark}}
}.\label{ssigma}

We introduce the fractional perimeter\footnote{The next measure theoretic assumptions are assumed throughout this chapter.
 Up to modifying $E\subset\R^n$ on a set of measure zero we can assume (see e.g. Appendix C of \cite{Myfractal})
that $E$ contains the measure theoretic interior
$
E_{int}:=\Big\{x\in\R^n\,|\,\exists\,r>0\textrm{ s.t. }|E\cap B_r(x)|=\frac{\omega_n}n r^n\Big\}\subset E,
$
the complementary $\C E$ contains its measure theoretic interior
$
E_{ext}:=\{x\in\R^n\,|\,\exists\,r>0\textrm{ s.t. }|E\cap B_r(x)|=0\}\subset\Co E,
$
and the topological boundary of $E$ coincides with its measure theoretic boundary, $\partial E=\partial^-E, \partial^-E:=\R^n\setminus(E_{int}\cup E_{ext})
=\{x\in\R^n\,|\,0<|E\cap B_r(x)|<\omega_nr^n\textrm{ for every }r>0\}.
$
In particular, we remark that both $E_{int}$ and $E_{ext}$ are open sets.} Let $\Omega \subset \Rn$ be an open bounded set, and let $E \subset \Rn$ be fixed outside of $\Omega$. We consider minimizers of the $H^{s/2}$ norm 
	\begin{equation*}
		\begin{split}
		||\chi_E||^2_{H^{\frac{s}2}} =& \int_{\Rn} \int_{\Rn} \frac{|\chi_E(x)-\chi_E(y)|^2}{|x-y|^{n+s}} \, dx \, dy\\
							=& 2 \int_{\Rn} \int_{\Rn} \frac{\chi_E(x)\chi_{\C E}(y)}{|x-y|^{n+s}} \, dx \, dy.
		\end{split}
	\end{equation*}
Notice that only the interactions between $E$ and $\C E$ contribute to the norm. \\
In order to define the fractional perimeter of $E$ in $\Omega$, we need to clarify the contribution of $\Omega$ to the $H^{\frac{s}2}$ norm here introduced. Namely, as $E$ is fixed outside $\Omega$, we aim at minimizing the ``$\Omega$-contribution'' to the norm among all measurable sets that ``vary'' inside $\Omega$. We consider thus interactions between $E\cap \Omega$ and $\C E$ and between $E\setminus \Omega$ and $\Omega \setminus E$, 
neglecting the data that is fixed outside~$\Omega$
and that does not contribute to the minimization of the norm (see Figure \ref{fign:FracPer}).
%\begin{center}
\begin{figure}[htpb]
%\sidecaption
%	\hspace{0.8cm}
%	\begin{minipage}[b]{0.95\linewidth}
%	\centering
	\includegraphics[width=0.85\textwidth]{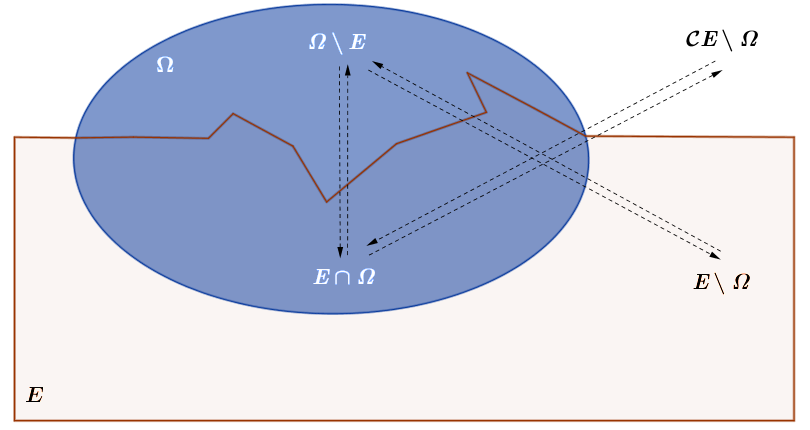}
	\caption{Fractional Perimeter}  
	\label{fign:FracPer}
%	\end{minipage}
	\end{figure} 
%	\end{center}
We define the interaction $I(A,B)$ of two disjoint subsets of $\Rn$ as
	\begin{equation}\label{nmsi1}
		\begin{split}
		I(A,B):=&\int_A \int_B \frac{dx\, dy}{|x-y|^{n+s}}
			= \int_{\Rn} \int_{\Rn} \frac{ \chi_A(x) \chi_B (x) }{|x-y|^{n+s}}\, dx \, dy.
		\end{split}
	\end{equation}
Then (see~\cite{nms}),
one defines the nonlocal $s$-perimeter functional of $E$ in $\Omega$ as	
\begin{equation} \label{nmspf1} \text{Per}_s(E,\Omega) := I(E\cap \Omega,\C E)  + I(E\setminus \Omega,\Omega \setminus E). \end{equation}
Equivalently, one may write
	\[ \text{Per}_s(E,\Omega) = I( E\cap \Omega,\Omega \setminus E) + I(E\cap \Omega , \C \Omega \setminus E)+ I(E\setminus \Omega,\Omega \setminus E). \]

\begin{defn}
Let~$\Omega$ be an open set of~$\R^n$.
A measurable set $E\subset \Rn$ is $s$-minimal in~$\Omega$
if $\text{Per}_s(E,\Omega)$ is finite and if,
for any measurable set~$F$ such that $E\setminus\Omega =F\setminus \Omega$, we have that
	\[ \text{Per}_s(E,\Omega) \leq \text{Per}_s(F,\Omega).\]
A measurable set is $s$-minimal in $\Rn$ if it is $s$-minimal in any ball $B_r$, where $r > 0$.
\end{defn}

The boundaries of $s$-minimal sets are referred to as \emph{nonlocal minimal surfaces}.
\medskip
 
 We discuss briefly the behavior of the perimeter as $s$ tends to $1$ and to $0$.\\
When~$ s\to 1^-$,
the fractional perimeter~$\text{Per}_s$ 
approaches the classical perimeter, see \cite{BBM}. See also \cite{DAVILA} for the precise limit in the class
of functions with bounded variations, \cite{uniform,regularity} for a geometric
approach towards regularity and \cite{PONCE, gammaconv} for an approach based on $\Gamma$-convergence. See also \cite{vS-W-2004} for a different proof and Theorem 2.22 in \cite{lukes} and the references therein for related discussions.
%as proved in \cite{CV11}.
A simple, formal statement (up to renormalizing constants) is the following:
\begin{theorem}\label{TP12}
Let $R>0$ and $E$ be a set with finite perimeter in $B_R$. Then 
	\[\lim_{s \to 1 }(1 -s) \text{Per}_s(E,B_r)=\text{Per }(E,B_r) \]
for almost any $r\in (0,R)$.
\end{theorem}

The behavior of~$\text{Per}_s$ as~$s\to0^+$ is slightly more involved.
In principle, the limit as~$s\to0^+$ of~$\text{Per}_s$
is, at least locally, related to the Lebesgue measure
(see e.g.~\cite{VGMAZYA}). Nevertheless, the situation
is complicated by the terms coming from infinity,
which, as~$s\to0^+$, become of greater and greater importance.
We define in this sense the contribution from infinity of a set as
\begin{equation}\label{E:LIM:LE}
\alpha(E) := \lim_{s\to0^+} s
\,\int_{E\setminus B_1}\frac{dy}{|y|^{n+s}}.\end{equation}
We will study in more detail this quantity in the next Section \ref{luke}.\\
It is proved in~\cite{asympt1} that,
if~$\text{Per}_{s_o}(E,\Omega)$ is finite for some~$s_o\in(0,1)$, and $\alpha(E)$ exists, then
\begin{equation}\label{C:I:PA}
\lim_{s\to0} s \,\text{Per}_{s}(E,\Omega)=
\left(\omega_n-\alpha(E)\right)\,|E\cap\Omega| + \alpha(E) \,|\Omega\setminus E|.\end{equation}
We remark that, using polar coordinates,
$$ 0\le \alpha(E)\le 
\lim_{s\to0^+} s
\,\int_{\R^n\setminus B_1}\frac{dy}{|y|^{n+s}}
= 
\lim_{s\to0^+} s\,\omega_n
\int_{1}^{+\infty} \rho^{-1-s}\,d\rho =\omega_n,$$
therefore~$\alpha(E)\in [0,\omega_n]$ plays the role
of a convex interpolation parameter in
the right hand-side of~\eqref{C:I:PA} (up to normalization constants).  \\
In this sense, formula~\eqref{C:I:PA} may be interpreted by saying
that, as~$s\to0^+$, the $s$-perimeter concentrates itself
on two terms that are ``localized'' in the domain~$\Omega$,
namely~$|E\cap\Omega|$ and~$|\Omega\setminus E|$. Nevertheless,
the proportion in which these two terms count is given by
a ``strongly nonlocal'' interpolation parameter, namely
the quantity~$\alpha(E)$ in~\eqref{E:LIM:LE}
which ``keeps track''
of the behavior of~$E$ at infinity.\\
As a matter of fact, to see how~$\alpha(E)$ is influenced
by the behavior of~$E$ at infinity, one can compute~$\alpha(E)$
for the particular cases, as in Subsection \ref{sectexamples}. For instance, taking $E$ a cone, then $\alpha(E)$ gives in this case exactly the opening of the cone.
%of a cone.
%For instance, if~$\Sigma\subseteq \partial B_1$, 
%with 
%\[ \mathfrak{o} :=\mathcal H^{n-1}(\Sigma),\] 
%%%\[\frac{|\Sigma|}{\omega_n}=: b \in[0,1],\]
%and $E$ is the cone over~$\Sigma$ (that is~$E:= \{ tp,\;\, p\in\Sigma,\;\,
%t\ge0\}$), we have that
%$$ \alpha(E) = \lim_{s\to0^+} s\,\mathfrak o\,\int_1^{+\infty}\rho^{-1-s}\,d\rho
%=\mathfrak o,$$
%that is~ We point out that we will explicitly compute the contribution from infinity of more sets in Subsection \ref{sectexamples}.\\
We also remark that, in general, the limit
in~\eqref{E:LIM:LE} may not exist, even for smooth sets:
indeed, it is possible that the set~$E$ ``oscillates'' wildly at infinity,
say from one cone to another one, leading to the non-existence of the limit in~\eqref{E:LIM:LE}.\\
Moreover, we point out that
the existence of the limit
in~\eqref{E:LIM:LE}
is equivalent to the existence of the limit in~\eqref{C:I:PA},
except in the very special case~$
|E\cap\Omega| = |\Omega\setminus E|$, in which the limit in~\eqref{C:I:PA}
always exists. That is, the following alternative holds true:
\begin{itemize}
\item if $|E\cap\Omega| \ne |\Omega\setminus E|$, then the
limit in~\eqref{E:LIM:LE} exists if and only if
the limit in~\eqref{C:I:PA} exists,
\item if $|E\cap\Omega| = |\Omega\setminus E|$, then the
limit in~\eqref{C:I:PA} always exists (even when the one
in~\eqref{E:LIM:LE} does not exist), and
$$ \lim_{s\to0} \text{Per}_{s}(E,\Omega)=
\omega_n\,|E\cap\Omega| =\omega_n\,|\Omega\setminus E|.$$
\end{itemize}

	We define now the $s$-fractional mean curvature of a set $E$ at a point $q\in\partial E$ as the principal value integral 
\eqlab{ \label{nmc} \I_s[E](q):=P.V.\int_{\R^n}\frac{\chi_{\C E}(y)-\chi_E(y)}{|y-q|^{n+s}}\,dy,}
that is
\[\I_s[E](q):=\lim_{\rho\to0^+}\I_s^\rho[E](q),\qquad\textrm{where}\qquad
\I_s^\rho[E](q)=\int_{
\C B_\rho(q)}\frac{\chi_{\C E}(y)-\chi_E(y)}{|y-q|^{n+s}}\,dy.\]

The fractional mean curvature gives the Euler-Lagrange equation corresponding to the $s$-perimeter functional $\text{Per}_s$.  Indeed, in analogy with the case of
 classical minimal surfaces, which have zero mean curvature, if $E$ is $s$-minimal in $\Omega$, then
\eqlab{\label{ELsmin}\I_s[E]=0,\qquad\textrm{on}\quad\partial E\cap\Omega,}
in an appropriate viscosity sense (see Theorem 5.1 of\cite{nms}).\\
Actually, by exploiting the interior regularity theory of $s$-minimal sets, the equation is satisfied in the classical sense
in a neighborhood of every ``viscosity point'' (see Appendix A in \cite{elsulbordo}). That is, if $E$ has at $p\in\partial E\cap\Omega$ a tangent ball (either interior or exterior), then $\partial E$ is $C^\infty$ in $B_r(p)$, for some $r>0$ small enough, and
\[\I_s[E](x)=0,\qquad\forall\,x\in\partial E\cap B_r(p).\] 
Moreover, if $\Omega$ has a $C^2$ boundary, then the Euler-Lagrange equation (at least as an inequality) holds also at a point $p\in\partial E\cap\partial\Omega$,
%provided that the boundary $\partial E$ and the boundary $\partial\Omega$ do not intersect ``transversally'' in $p$ (see Theorem 1.1 in \cite{elsulbordo}).
%
%In \cite{nms} it is proved that $s$-minimizers satisfy a suitable integral equation (see in particular Theorem 5.1
%in \cite{nms}), that is the Euler-Lagrange equation 
%If~$E$ is $s$-minimal in $\Omega$ and~$\partial E$ is smooth enough,
%this Euler-Lagrange equation can be written as
%	\begin{equation} \label{ELsmin} 
%\int_{\Rn} \frac{\chi_E(x_0+y) -\chi_{\Rn\setminus E} (x_0+y) }{|y|^{n+s}}\, dy  =0,\end{equation}
%for any $x_0\in \Omega \cap \partial E$.
% 
%	
%	
%	\eqlab{ \label{nmc} H_E^s(x_0) := \int_{\Rn} \frac{
%\chi_E(y) -\chi_{\C E}(y) }{|y-x_0|^{n+2s}}\, dy.}

%In this way, equation~\eqref{ELsmin} can be written as~$\I[E](x)=0$ at $x\in \partial E$.
%Therefore,

It is also suggestive to think that the function~$\tilde\chi_E:=
\chi_{\C E}-\chi_{E}$ averages out to zero at the points on~$\partial E$,
if~$\partial E$ is smooth enough, since at these points
the local contribution of~$E$ compensates the one of~$\C E$.
Using this notation, for $x_0\in \partial E$, one may take the liberty of writing
\begin{eqnarray*}
\I_s[E](x_0)&=&\frac12
\int_{\Rn} \frac{\tilde\chi_{E}(x_0+y)+\tilde\chi_{E}(x_0-y)}{|y|^{n+s}}\, dy\\
&=&
\frac12
\int_{\Rn} \frac{\tilde\chi_{E}(x_0+y)+\tilde\chi_{E}(x_0-y)-2\tilde\chi_E(x_0)}{
|y|^{n+s}}\, dy
\\&=&-\frac{-(-\Delta)^{\frac{s}2} \tilde\chi_{E}(x_0)}{C(n,s)},\end{eqnarray*}
using the notation of~\eqref{frlap2def}.
Using this suggestive representation, the
Euler-Lagrange equation in~\eqref{ELsmin} becomes
$$ {\mbox{$(-\Delta)^{\frac{s}2} \tilde\chi_{E}=0$
along~$\partial E\cap \Omega$.}}$$\\
%We refer to \cite{Abaty} for further details on this argument.
For the main properties of the fractional mean curvature, we refer to \cite{Abaty}.
In particular,  it is proved there in Theorem 12 that for a set $E\subset \Rn$ with $C^2$ boundary and any $x\in \partial E$, one has 
\[ \lim_{s \to 1} \mathcal (1-s)I_s[E] (x) = \omega_{n-1}H[E](x),\]  where $H$ is the classical mean curvature of $E$ at the point $x$ (with the convention that we take $H$ such that the curvature of the ball is a positive quantity).  See also \cite{regularity}.

It is also worth recalling that
the nonlocal perimeter functionals find applications
in motions of fronts by nonlocal mean curvature (see
e.g.~\cite{CAFFA-SOUG, IMBERT, CHAMBOLLE}),
problems in which aggregating and disaggregating terms
compete towards an equilibrium (see e.g.~\cite{I5}
and~\cite{I4}) and nonlocal free boundary problems
(see e.g.~\cite{CAFFA-SAVIN-VALDINOCI} and~\cite{DIPIERRO-SAVIN-VALDINOCI}).
See also~\cite{VGMAZYA}
and~\cite{VISENTIN}
for results related to this type of problems.
\bigskip

We point out that in order to find minimal surfaces we are looking for sets of minimal perimeter (this was first done by De Giorgi in the classical case). However, it is necessary to prove that indeed the boundaries of $s$-minimal sets are smooth surfaces.  In the case of the local perimeter
functional, it is known indeed that the boundaries of minimal sets are smooth in dimension $n\leq 7$. Moreover, if $n\geq 8$  minimal surfaces are smooth except on a small singular set of Hausdorff dimension $n-8$.
Differently from the classical case,
the regularity theory for $s$-minimizers is still quite open.
We present here some of the partial results obtained in this direction:

\begin{theorem}\label{THM 5.8} In the plane, $s$-minimal sets are smooth. More precisely:\\ 
a) If $E$ is an $s$-minimal set in $\Omega \subset \R^2$, then $\partial E \cap \Omega$ is a $C^{\infty}$-curve.\\
b) Let $E$ be $s$-minimal in $\Omega\subset \Rn$ and let $\Sigma_E \subset \partial E\cap\Omega$ denote its
singular set. Then $\mathcal{H}^d (\Sigma_E)=0$ for any $d>n-3$.
\end{theorem}
See \cite{SV13} for the proof of this results
(as a matter of fact, in~\cite{SV13} only $C^{1,\alpha}$ regularity is proved,
but then \cite{bootstrap} proved that $s$-minimal sets with~$C^{1,\alpha}$-boundary are automatically~$C^\infty$).
Further regularity results of the $s$-minimal surfaces can be found in 
\cite{regularity}. There, a regularity theory
when $s$ is near $\displaystyle 1$ is stated, as we see in the following Theorem:

\begin{theorem}\label{THM 5.9} There exists $\epsilon_0 \in (0,1)
%\displaystyle \frac{1}{2}\Big)
$ such that if $s \geq 1 -\epsilon_0$,  then\\
a) if $n\leq 7$, any $s$-minimal set is of class $C^{\infty}$, \\
b) if $n=8$ any $s$-minimal surface is of class $C^\infty$ except, at most, at countably many isolated points, \\
c) any $s$-minimal surface is of class $C^\infty$ outside a closed set $\Sigma$ of Hausdorff dimension $n-8$.
\end{theorem}

\section{Graphs and $s$-minimal surfaces} 
Minimal surfaces that are graphs are called minimal graphs, and they reduce to hyperplanes
if~$n\le8$ (this is called the Bernstein property,
which was also discussed at the beginning of the Chapter \ref{S:NP}).
If $n\geq9$, there exist global minimal graphs that are not affine
(see e.g.~\cite{GIUSTI}).

We will focus the upcoming material on two interesting results related to graphs: a Bernstein type result, namely the property that an $s$-minimal graph  in $\R^{n+1}$ is flat (if no singular cones exist in dimension $n$); we will then prove that an $s$-minimal surface whose prescribed data is a subgraph, is itself a subgraph.

The first result is the following theorem:
\begin{theorem}\label{figv}
Let $E= \{ (x,t) \in \Rn \times \R \;\big|\; t<u(x)\}$ be an $s$-minimal graph, and assume there are no singular cones in dimension $n$ (that is, if $\mathcal{K} \subset \Rn$ is an $s$-minimal cone, then $\mathcal{K}$ is a half-space). Then $u$ is an affine function (thus $E$ is a half-space).
\end{theorem}

To be able to prove Theorem \ref{figv}, we recall some useful auxiliary results. 
In the following lemma we state a dimensional reduction result (see Theorem 10.1 in \cite{nms}).
\begin{lemma} \label{dimrid}
Let $E=F\times \R$. Then if $E$ is $s$-minimal if and only if $F$ is $s$-minimal.
\end{lemma}
We define then the blow-up and blow-down of the set $E$ are, respectively
	\[E_0:=\lim_{r\to 0}E_r \quad \mbox{and} \quad E_\infty: =\lim_{r\to +\infty} E_r, \quad \mbox{where} \quad E_r=\frac{E}{r}.\]
A first property of the blow-up of $E$ is the following (see Lemma 3.1 in \cite{FV13}).
\begin{lemma}
If $E_\infty$ is affine, then so is $E$.
\label{halfspace}
\end{lemma}

We recall also a regularity result for the $s$-minimal surfaces (see \cite{FV13} and \cite{bootstrap} for details and proof).
\begin{lemma}\label{reg}  Let~$E$ be $s$-minimal. Then:\\
a) If $E$ is Lipschitz, then $E$ is~$C^{1,\alpha}$.\\
b) If $E$ is~$C^{1,\alpha}$, then $E$ is~$C^{\infty}$.
\end{lemma}

We give here a sketch of the proof of Theorem \ref{figv} (see \cite{FV13} for all the details).
\begin{proof}[Sketch of the proof of Theorem \ref{figv}]
If $E\subset \R^{n+1}$ is an $s$-minimal graph, then the blow-down $E_\infty$ is an $s$-minimal cone (see Theorem 9.2 in \cite{nms} for the proof of this statement).  By applying the dimensional reduction argument in Lemma \ref{dimrid} we obtain an $s$-minimal cone in dimension $n$. According to the assumption  that no singular $s$-minimal cones exist in dimension $n$, it follows that necessarily $E_\infty$ can be singular only at the origin.\\
We consider a bump function $w_0 \in C^{\infty}( \mathbb{R}, [0,1])$ such that 
	\begin{equation*}
		\begin{split}
		&w_0(t)=0 \text{ in } \bigg(-\infty, \frac{1}{4}\bigg)\cup \bigg(\frac{3}{4}, +\infty\bigg) \\
		& w_0(t) =1 \text{ in } \bigg(\frac {2}{5},\frac{3}{5}\bigg)\\
		&w(t)=w_0(|t|).
		\end{split}
	\end{equation*}
The blow-down of $E$ is	
	\[E_\infty =\big\{(x',x_{n+1}) \; \big| \; x_{n+1}\leq u_\infty(x')\big\} .\]
For a fixed $\sigma \in \partial B_1$, let
	\[F_t:=\big\{(x',x_{n+1}) \; \big| \;  x_{n+1}\leq u_{\infty} \big(x'+t\theta w(x')\sigma\big) -t\big\}\]
	be a family of sets, where $t\in (0,1)$ and $\theta >0$.
Then for $\theta$ small, we have that 
\begin{equation}\label{ABC567}
{\mbox{$F_1$ is below $E_\infty$.}}\end{equation} Indeed, suppose by contradiction that this is not true. Then, there exists $\theta_k \to 0$ such that
	\begin{equation}\label{UF78988799} u_\infty\big(x'_k+\theta_kw(x'_k)\sigma\big)-1 \geq u_\infty (x'_k).\end{equation}
But $x'_k \in \text{supp}w$, which is compact, therefore $\displaystyle x'_\infty:=\lim_{k\to +\infty} x'_k$ belongs to the support of $w$, and $w(x'_\infty)$ is defined. Then, by sending $k \to +\infty$
in~\eqref{UF78988799}
we have that
	\[u_\infty(x'_\infty) -1\geq u_\infty(x'_\infty),\]
which is a contradiction.
This establishes~\eqref{ABC567}.

Now consider the smallest $t_0\in (0,1)$ for which $F_t$ is below $E_\infty$. Since $E_\infty$ is a graph, then $F_{t_0}$ touches $E_\infty$ from below in one point $X_0=(x'_0,x_{n+1}^0)$, where $x'_0 \in \text{supp} w$. 
Since $E_\infty$ is $s$-minimal, we have that the nonlocal mean curvature (defined in \eqref{nmc}) of the boundary is null. 
Also, since $F_{t_0}$ is a $C^2$ diffeomorphism of $E_\infty$ we have that
	\begin{equation}\label{PLO90-1}
\I_s[F_{t_0}](p) \simeq \theta t_0,\end{equation}
and there is a region where $E_\infty$ and $F_{t_0}$ are well separated by $t_0$, thus
	\[\big|\big(E_\infty \setminus F_{t_0}\big) \cap \big(B_3
\setminus B_2\big)\big| \geq ct_0,\]
for some~$c>0$.
Therefore, we see that
	\[\I_s[F_{t_0}](p) =\I_s[F_{t_0}](p)  -I_s[E_\infty](p) \geq c t_0.\]
This and~\eqref{PLO90-1} give that~$\theta t_0 \ge c t_0$, for some~$c>0$
(up to renaming it). 
If~$\theta$ is small enough, this implies that $t_0=0$.

In particular,
we have proved that there exists $\theta >0$ small enough such that,
for any $t\in (0,1)$ and any~$\sigma\in\partial B_1$, we have that 
	\[  u_\infty\big(x'+t \theta w(x')\sigma\big)-t \leq u_\infty (x').\]
This implies that
	\[\frac{u_\infty\big(x'+t \theta w(x')\sigma\big)-u_\infty (x')}{t\theta}\leq \frac{1}{\theta},\]
hence, letting $t\to 0$, we have that 
	\[ \nabla u_\infty (x')w(x')\sigma \leq \frac{1}{\theta}, \text { for any } x\in \Rn\setminus \{0\}, \text{ and } \sigma \in B_1.\]
We recall now
that $w=1$ in $B_{3/5}\setminus B_{2/5}$ and $\sigma$ is arbitrary in $\partial B_1$.
Hence, it follows that
	\[|\nabla u_\infty(x)| \leq \frac{1}{\theta}, \text{ for any } x \in B_{3/5}\setminus B_{2/5}.\]
Therefore $u_\infty$ is globally Lipschitz. By the regularity statement in 
Lemma \ref{reg}, we have that~$u_\infty$ is $C^\infty$.
This says that~$u$ is
smooth also at the origin, hence (being a cone) it follows that~$E_\infty$ is necessarily
a half-space. Then by Lemma \ref{halfspace}, we conclude that E is a half-space as well.
\end{proof}

We introduce in the following theorem another interesting property related to $s$-minimal surfaces, in the case in which the fixed given data outside a domain is a subgraph. In that case, the $s$-minimal surface itself is a subgraph. Indeed:

\begin{theorem}\label{thmgraph}
Let $\Omega_0$ be a bounded open subset of $\R^{n-1}$ with boundary of class $C^{1,1}$ and let $\Omega:=\Omega_0 \times \R$. Let $E$ be an $s$-minimal set in $\Omega$. Assume that
		\eqlab{\label{thgr2}E\setminus \Omega =\{x_n <u(x'), \, x'\in \R^{n-1}\setminus \Omega_0\}}
		for some continuous function $u\colon \R^{n-1} \to \R$. Then
		\[ E\cap \Omega =\{x_n<v(x'), \, x' \in \Omega_0\}\]
		for some function $v\colon \R^{n-1}\to \R$.  
\end{theorem}

The reader can see \cite{graph}, where this theorem and the related results are proved; here, we only state the preliminary results needed for our purposes
and focus on the proof of Theorem \ref{thmgraph}. The proof relies on a
sliding method, more precisely, we take a translation of $E$ in the $n^{\mbox{th}}$ direction, and move it until it touches $E$ from above. If the set $E\cap \Omega$ is a subgraph, then, up to a set of measure $0$, the contact between the translated $E$ and $E$, will be $E$ itself.

 However, since we have no information on the regularity of the minimal surface, we need at first to ``regularize'' the set by introducing the notions of supconvolution and subconvolution. With the aid of a useful result related to the sub/supconvolution of an $s$-minimal surface, we proceed then with the proof of the Theorem \ref{thmgraph}.
  %\begin{center}
\begin{figure}[htpb]
%\sidecaption
%	\hspace{0.8cm}
%	\begin{minipage}[b]{0.95\linewidth}
%	\centering
	\includegraphics[width=0.50\textwidth]{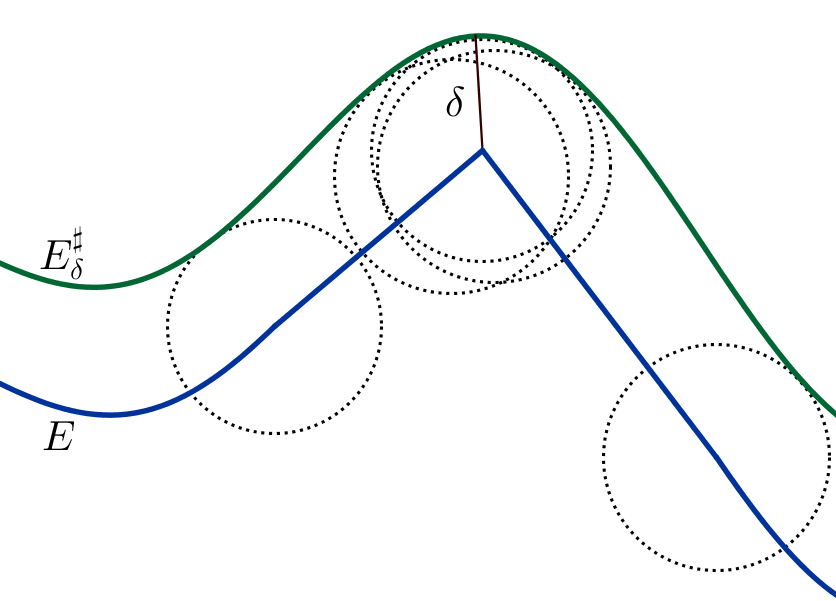}
	\caption{The supconvolution of a set}  
	\label{fign:sgrcor1}
%	\end{minipage}
	\end{figure} 
%	\end{center} 

\noindent The supconvolution of a set $E\subseteq \Rn$ is given by 
 	\[ E_{\delta}^{\sharp} := \bigcup_{x\in E} \overline {B_{\delta} (x)}.\]
 	In an equivalent way, the supconvolution can be written as 
 	\[ E_{\delta}^{\sharp} =  \bigcup_{ { v\in \Rn}\atop {|v|\leq \delta}}   (E+v)  .\] 
 	Indeed, we consider $\delta>0$ and an arbitrary $x\in E$. Let $y\in \overline{ B_\delta (x)}$ and we define $v:=y-x$. Then
 		\[ |v|\leq|y-x|\leq \delta\quad \mbox{and} \quad y=x+v \in E+v. \]
 Therefore $\overline{B_\delta(x)} \subseteq E+v$ for $|v|\leq \delta$. In order to prove the inclusion in the opposite direction, one notices that taking $y\in E+v$ with $|v|\leq \delta$ and defining $x:= y-v$, it follows that
 \[ |x-y|=|v|\leq \delta.\] Moreover, $x\in (E+v)-v=E$ and the inclusion $E+v \in \overline{B_\delta(x)}$ is proved.
  	
 	      On the other hand, the subconvolution is defined as
 		\[ E_{\delta}^{\flat} := \Rn \setminus \left( (\Rn \setminus E)_{\delta}^{\sharp}\right).\] Now, the supconvolution of~$E$
is a ``regularized'' version of~$E$, whose nonlocal minimal curvature
is smaller than the one of~$E$, i.e.:
\begin{equation}\label{oforUI89}
\int_{\R^n} \frac{ \chi_{\C E_{\delta}^{\sharp}}(y)
-\chi_{ E_{\delta}^{\sharp} }(y)
}{|x-y|^{n+s}} \,dy \le
\int_{\R^n} \frac{ \chi_{CE}(y)
-\chi_{ E }(y)
}{|\tilde x-y|^{n+s}} \,dy
\le0,
\end{equation}
for any~$x\in\partial E_{\delta}^{\sharp}$, where~$\tilde x:= x-v\in
\partial E$ for some~$v\in\R^n$ with~$|v|=\delta$. 
Then, by construction, the set~$E+v$ lies in~$E_{\delta}^{\sharp}$,
and this implies~\eqref{oforUI89}.
%The proof of~\eqref{oforUI89} can be done by taking~$x\in \partial 
%E_{\delta}^{\sharp}$ and noticing that~$\tilde x:= x-v\in
%\partial E$ for some~$v\in\R^n$ with~$|v|=\delta$.
Similarly, one has that the opposite inequality holds
for the subconvolution of~$E$, namely for any $x\in \partial E_{\delta}^{\flat}$
\begin{equation}\label{oforUI89-bis}
\int_{\R^n} \frac{ \chi_{\C E_{\delta}^{\flat}}(y)
-\chi_{ E_{\delta}^{\flat} }(y)
}{|x-y|^{n+s}} \,dy 
%\geq
%\int_{\R^n} \frac{ \chi_{\R^n\setminus E}(y)
%-\chi_{ E }(y)
%}{|x-y|^{n+2s}} \,dy
\ge0,
\end{equation}
By~\eqref{oforUI89} and~\eqref{oforUI89-bis}, we obtain:
 		 \begin{prop}\label{posubsup}
 		 Let $E$ be an $s$-minimal set in $\Omega$. Let $p\in \partial E_{\delta}^{\sharp}$ and assume that $\overline {B_\delta(p)}\subseteq\Omega$. Assume also that $E_{\delta}^{\sharp}$ is touched from above by a translation of $E_{\delta}^{\flat}$, namely there exists $\omega \in \Rn$ such that 
 		 \[ E_{\delta}^{\sharp} \subseteq E_{\delta}^{\flat} +\omega\] 
 		 and 
 		 \[ p\in (\partial E_{\delta}^{\sharp} )\cap (\partial E_{\delta}^{\flat}+\omega).\] Then \[E_{\delta}^{\sharp} = E_{\delta}^{\flat }+\omega. \]
 		 \end{prop}

\begin{proof}[Proof of Theorem \ref{thmgraph}]
One first remarks is that
the $s$-minimal set does not have spikes which go to infinity: more precisely,
one shows that
	\eqlab{\label{thgr1} \Omega_0\times (-\infty, -M)\subseteq E\cap \Omega \subseteq \Omega_0 \times (-\infty,M) }
for some $M\geq 0$.
The proof of \eqref{thgr1} can be performed by sliding horizontally a large ball,
see~\cite{graph} for details.

After proving~\eqref{thgr1}, one can deal with the core of
the proof of Theorem \ref{thmgraph}.
The idea is to
slide $E$ from above until it touches itself and analyze what happens at the contact points.
For simplicity, we will assume here that the function~$u$
is uniformly continuous (if~$u$ is only continuous,
the proof needs to be slightly modified since the subconvolution
and supconvolution that we will perform may create new touching
points at infinity).
%\begin{center}
\begin{figure}[htpb]
%\sidecaption
%	\hspace{0.8cm}
%	\begin{minipage}[b]{0.95\linewidth}
%	\centering
	\includegraphics[width=0.75\textwidth]{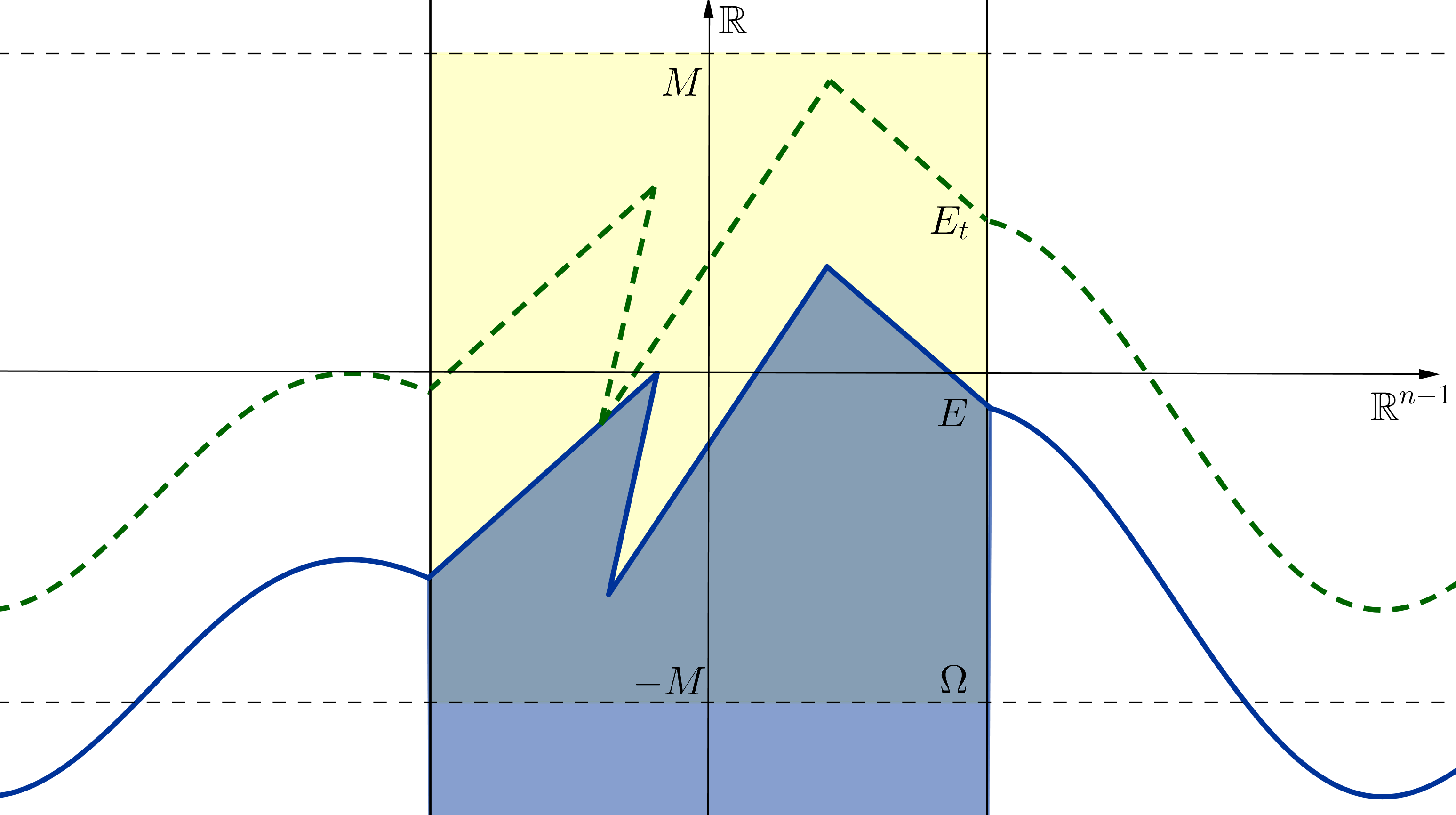}
	\caption{Sliding $E$ until it touches itself at an interior point}  
	\label{fign:thmgrph1}
%	\end{minipage}
	\end{figure} 
	%\end{center}
At this purpose, we consider $E_t=E+t e_n$ for $t\geq 0$. Notice that, by \eqref{thgr1}, if $t \geq 2M$, then $E\subseteq E_t$. Let then $ t$ be the smallest for which the inclusion $E\subseteq E_t$ holds.
We claim that  $t=0$.
		 If this happens, one may consider 
		 \[ v=\inf\{\tau \;\big|\; (x,\tau)\in \C E\}\]
and, up to sets of measure $0$, $E\cap \Omega_0$ is the subgraph of $v$.
 
 The proof is by contradiction, so let us assume  that $ t>0$. According to \eqref{thgr2}, the set $E\setminus \Omega$ is a subgraph, hence the contact points between $\partial E$ and $\partial E_t$ must lie in $\overline \Omega_0 \times \R$. Namely, only two 
possibilities may occur: the contact point is
interior (it belongs to  $\Omega_0 \times\R)$, or it is at the boundary (on $\partial \Omega_0 \times\R$).
So, calling $p$ the contact point, one may have\footnote{As a matter
of fact, the number of contact
points may be higher than one, and even infinitely many
contact points may arise. So, to be rigorous,
one should distinguish the case in which all
the contact points are interior and the case in which
at least one contact point lies on the boundary.

Moreover, since the surface may have vertical portions
along the boundary of the domain, one needs to carefully define
the notion of contact points (roughly speaking, one needs to take
a definition for which the vertical portions
which do not prevent the sliding are not in the contact set).

Finally, in case the contact points are all interior, it is also
useful to perform the sliding method in a slighltly reduced domain,
in order to avoid that the supconvolution method produces
new contact points at the boundary (which may arise from
vertical portions of the surfaces).

Since we do not aim to give a complete proof of Theorem~\ref{thmgraph} here,
but just to give the main ideas and underline the additional
difficulty, we refer to~\cite{graph} for the full details of these arguments.}
that
	\eqlab{\label{caseone}  {\mbox{either }} p\in   \Omega_0 \times \R\quad \mbox{or}}
	\eqlab{\label{casetwo}  p\in \partial \Omega_0 \times \R. \quad \mbox{ }}

We deal with the first case in~\eqref{caseone} (an example of this behavior is depicted in Figure \ref{fign:thmgrph1}).
We consider $E_{\delta}^{\sharp} $ and $E_{\delta}^{\flat}$ to be the supconvolution, respectively the subconvolution of $E$.  We then slide the subconvolution until it touches the supconvolution. More precisely, let $\tau>0$ and we take a translation of the subconvolution, $E_{\delta}^{\flat}+ \tau e_n$. For $\tau$ large, we have that $E_{\delta}^{\sharp}\subseteq E_{\delta}^{\flat}+ \tau e_n$ and we consider $\tau_\delta$ to be the smallest for which such inclusion holds. We have (since $t$ is positive by assumption) that
	\[ \tau_\delta\geq \frac{t}{2}>0.\]
	Moreover, for $\delta$ small, the sets $\partial E_\delta^{\sharp} $ and $\partial (E_{\delta}^{\flat}+ \tau_\delta e_n)$ have a contact point which, according to \eqref{caseone}, lies in $\Omega_0\times \R$. Let $p_{\delta}$ be such a point, so we may write
 \[  p_\delta \in   (\partial E_\delta^{\sharp} ) \cap \partial (E_{\delta}^{\flat}+ \tau_\delta e_n) \quad \mbox{and } \quad   p_\delta \in \Omega_0\times \R.\]
 Then, for $\delta$ small (notice that $\overline{B_{\delta}(p)}\subseteq \Omega$), Proposition \ref{posubsup} yields that 
 \[E_\delta^\sharp =E_\delta^\flat +\tau_\delta e_n.\]
 Considering $\delta$ arbitrarily small, one obtains that
 \[E=E+\tau_0 e_n,\quad \mbox{with} \quad \tau_0>0.\] 
 But $E$ is a subgraph outside of $\Omega$, and this provides a contradiction. Hence, the claim that $t=0$ is proved.

Let us see that we also obtain a contradiction when supposing that $t>0$ and that the second case \eqref{casetwo} holds. Let
 	\[ p=(p',p_n) \quad \mbox{and } \quad p \in  (\partial E)\cap (\partial E_t) . \]
 	Now, if one takes sequences $a_k\in \partial E$ and $b_k\in \partial E_t$, both that tend to $p$ as $k $ goes to infinity, since $E\setminus \Omega $ is a subgraph and $t>0$, necessarily $a_k, b_k$ belong to $\Omega$. Hence
 	 		\eqlab{ \label{pinclos} p\in \overline { (\partial E) \cap \Omega}  \cap \overline {(\partial E_t)\cap \Omega}.} 
 	Thanks to Definition 2.3 in \cite{nms}, one obtains that $E$ is a variational subsolution in a neighborhood of $p$. In other words, if $A\subseteq E\cap\Omega$ and $p\in \overline A$, then 
 			\[ 0\geq \mbox{Per}_s( E,\Omega)-\mbox{Per}_s(E\setminus A, \Omega) = I(A,\C E) -I(A, E\setminus A)\] 
 			(we recall the definition of $I$ in \eqref{nmsi1} and of the fractional perimeter $Per_s$ in \eqref{nmspf1}).
  According to Theorem 5.1 in \cite{nms}, this implies in a viscosity sense (i.e. if $E$ is touched at $p$ from outside by a ball), that
 			\eqlab{\label{curvgz} \int_{\Rn}\frac{\chi_{\C E}(y)-\chi_{ E}(y)}{|p-y|^{n+s}}\, dy\leq 0.} 
In order to obtain an estimate on the fractional mean curvature in the strong sense, we consider the translation of the point $p$ as follows: 
\[p_t=p-t e_n=(p',p_n-t)=(p',p_{n,t}).\] Since $t>0$, one may have that either $p_n\neq u(p')$, or $p_{n,t}\neq u(p')$. 
 			
 			These two possibilities can be dealt with in a similar way, so we just continue with the proof in the case $p_n\neq u(p')$ (as is also exemplified in Figure \ref{fign:thmgrph2}).  Taking $r>0$ small, the set $B_r(p)\setminus \Omega$ is contained entirely in $E$ or in its complement. Moreover, one has from \cite{regint} that $\partial E\cap B_r(p) $ is a $C^{1,\frac{1+s}2}$-graph in the direction of the normal to $\Omega$ at $p$. That is: in Figure \ref{fign:thmgrph2}
the set~$E$ is $C^{1,\frac{1+s}2}$, hence in the vicinity of~$p=(p',p_n)$, it appears to be sufficiently smooth.
 %\begin{center}
\begin{figure}[htpb]
%\sidecaption
%	\hspace{0.8cm}
%	\begin{minipage}[b]{0.95\linewidth}
%	\centering
	\includegraphics[width=0.75\textwidth]{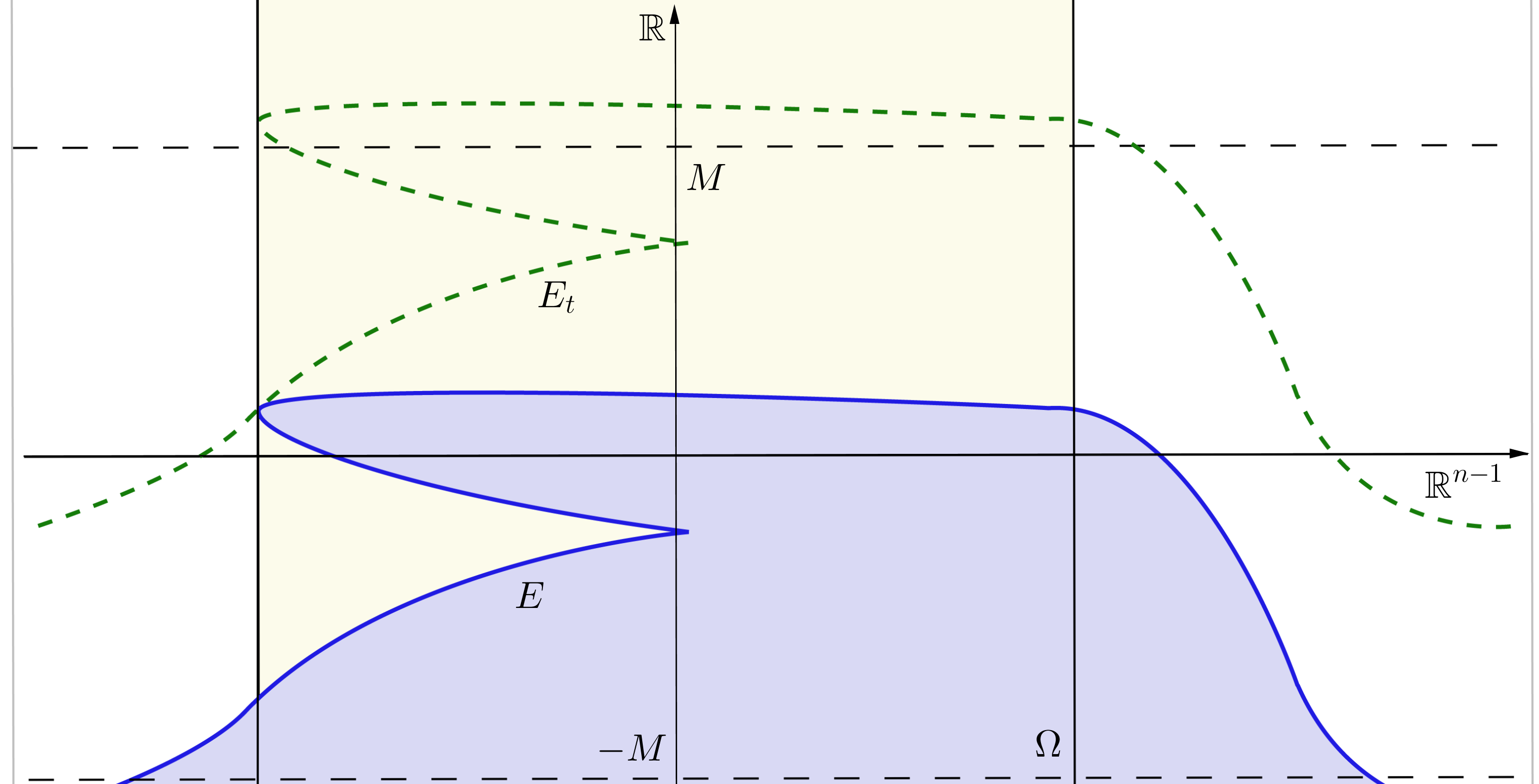}
	\caption{Sliding $E$ until it touches itself at a boundary point}  
	\label{fign:thmgrph2}
%	\end{minipage}
	\end{figure} 
%	\end{center}

So, let $\nu(p)=(\nu'(p),\nu_n(p))$ be the normal in the interior direction, then up to a rotation and since $\Omega$ is a cylinder (hence $\nu_n(p)=0$), we can write $\nu(p)=e_1$. Therefore, there exists a function $\Psi$ of class $ C^{1,\frac{1+s}2}$ such that $p_1=\Psi(p_2,\dots,p_n)$ and, in the vicinity of $p$, we can write $\partial E$ as the graph $G=\{x_1=\Psi(x_2,\dots,x_n)\}$. \\ 			 
 			 Given \eqref{pinclos}, we deduce that there exists a sequence $p_k\in G$ such that $p_k\in \Omega$ and $p_k \to p$ as $k\to \infty$. From this it follows that there exists a sequence of points $p_k\to p$ such that 
 			 \eqlab {\label{pkE1} \partial E \mbox{ in the vicinity of }p_k \mbox{ is a graph of class }C^{1,\frac{1+s}2}} 
 			 and 
 			  	 \eqlab{\label{pkE2} \int_{\Rn}\frac{\chi_{\C E}(y)-\chi_{ E}(y)}{|p_k-y|^{n+s}}\, dy =0.}
From \eqref{pkE1} and~\eqref{pkE2}, and using 
a pointwise version of the Euler-Lagrange equation (see~\cite{graph}
for details), we have that
 			 \bgs{ \int_{\Rn}\frac{\chi_{\C E}(y)-\chi_{E}(y)}{|p-y|^{n+s}}\, dy =0.}	
 		Now, $E\subset E_t$ for $t$ strictly positive, hence
 			\eqlab{\label{curvgz11} \int_{\Rn}\frac{\chi_{\C E_t}(y)-\chi_{E_t}(y)}{|p-y|^{n+s}}\, dy <0.}	 	
 		Moreover, we have that the set $\partial E_t \cap B_{\frac{r}4}(p)$ must remain on one side of the graph $G$, namely one could have that
 			\bgs{ & E_t\cap	B_{\frac{r}4}(p) \subseteq \{x_1\leq \Psi(x_2,\dots,x_n)\} \mbox{ or }\\
 				& E_t\cap	B_{\frac{r}4}(p) \supseteq \{x_1\geq \Psi(x_2,\dots,x_n)\} .}
 	Given again \eqref{pinclos}, 	we deduce that there exists a sequence $\tilde p_k\in \partial E_t\cap \Omega$ such that $\tilde p_k \to p$ as $k\to \infty$ and $ \partial E_t\cap \Omega $ in the vicinity of $\tilde p_k$ is touched by a surface lying in $E_t$, of class $ C^{1,\frac{1+s}2}$. 
 	Then 	\bgs{ \int_{\Rn}\frac{\chi_{\C E_t}(y)-\chi_{E_t}(y)}{|\tilde p_k-y|^{n+s}}\, dy \geq 0.} 
Hence, making use of
a pointwise version of the Euler-Lagrange equation (see~\cite{graph}
for details), we obtain that
 	\bgs{ \int_{\Rn}\frac{\chi_{\C E_t}(y)-\chi_{E_t}(y)}{| p-y|^{n+s}}\, dy \geq 0.} But this is a contradiction with \eqref{curvgz11}, and this concludes the proof of Theorem~\ref{thmgraph}.
 	\end{proof}

On the one hand, one may think that Theorem \ref{thmgraph} has to be well-expected. On the other hand, it is far from being obvious, not only because the proof is not trivial, but also because the statement itself almost risks to be false, especially at the boundary.
Indeed we will see in Theorem \ref{STI-DSV-2}
that the graph property is close to fail at the boundary of the domain, where the $s$-minimal surfaces may present vertical tangencies and stickiness phenomena (see Figure \ref{eST2}).

\section{Non-existence of singular cones in dimension $2$}We now prove the non-existence of singular
$s$-minimal cones in dimension $2$, as stated in the next result
(from this, the more general statement in Theorem~\ref{THM 5.8}
follows after a blow-up procedure):

\begin{theorem}\label{THIS}
If $E$ is an $s$-minimal cone in $\mathbb{R}^2$, then $E$ is a half-plane.
\label{noconestwo}
\end{theorem}

We remark that, as a combination of
Theorems \ref{figv} and~\ref{THIS},
we obtain the following result of Bernstein type:

\begin{corollary}
Let $E= \{ (x,t) \in \Rn \times \R\;\big|\; t<u(x)\}$ be an $s$-minimal graph, and assume that $n \in \{1, 2\}$. Then $u$ is an affine function.
\end{corollary}

%\begin{center}
\begin{figure}[htpb]
%\sidecaption
%	\hspace{0.8cm}
%	\begin{minipage}[b]{0.95\linewidth}
%	\centering
	\includegraphics[width=0.40\textwidth]{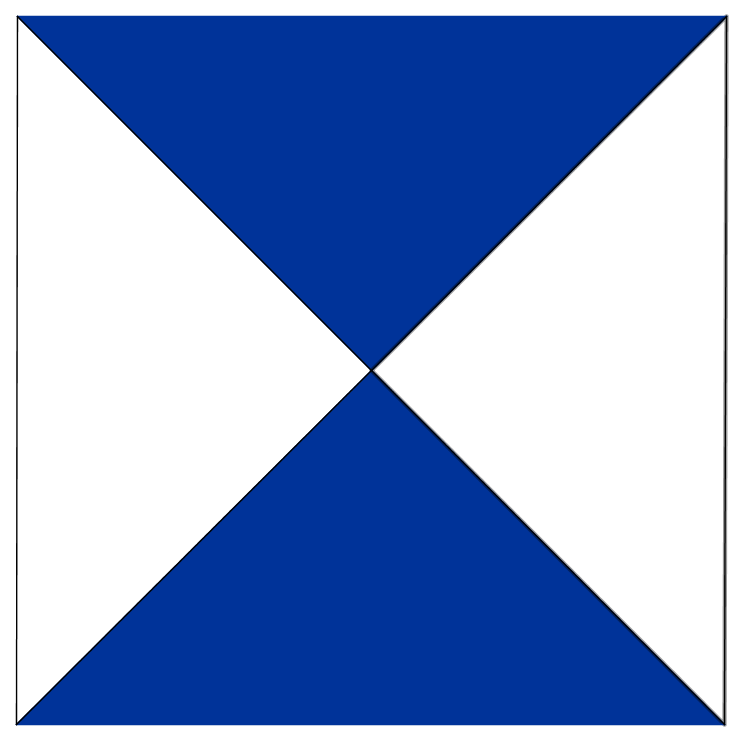}
	\caption{The cone $\mathcal{K}$}   
	\label{fign:Cone}
%	\end{minipage}
	\end{figure}
	%\end{center}
Let us first consider a simple example, given by the cone in the plane, drawn in Figure \ref{fign:Cone},
	\[\mathcal{K} := \Big\{  (x,y)\in \R^2 \;\big|\; y^2>x^2\Big\}.\]

\begin{prop}\label{THAT}
The cone $\mathcal{K}$ depicted in 
Figure \ref{fign:Cone} is not $s$-minimal in $\R^2$.
\end{prop} 

Notice that, by symmetry, one can prove that $\mathcal{K}$ satisfies \eqref{ELsmin}
(possibly in the viscosity sense). On the other hand, Proposition~\ref{THAT} gives that~${\mathcal{K}}$
is not $s$-minimal. This, in particular, provides an example of a set that
satisfies the Euler-Lagrange equation 
in~\eqref{ELsmin}, but is not~$s$-minimal (i.e.,
the Euler-Lagrange equation 
in~\eqref{ELsmin}
is implied by, but not necessarily equivalent to, the $s$-minimality property). 

\begin{proof}[Proof of Proposition~\ref{THAT}]
The proof of the non-minimality of $\mathcal{K}$ is due to an original idea by Luis Caffarelli.

Suppose by contradiction that the cone $\mathcal{K}$ is minimal in $\R^2$. 
We add to $\mathcal K$ a small square adjacent to the origin
(see Figure \ref{fign:Cone1}), and call $\mathcal{K}'$ the set obtained. Then $\mathcal{K}$ and $\mathcal{K}'$ have the same $s$-perimeter. This is due to the interactions considered in the $s$-perimeter functional and the unboundedness of the regions. We remark  that in Figure \ref{fign:Cone1} we draw bounded regions, of course, sets $A, B, C, D, A' ,B' ,C'$ and $D'$ are actually unbounded.
%\begin{center}
\begin{figure}[htpb]
%\sidecaption
%	\hspace{0.8cm}
%	\begin{minipage}[b]{0.95\linewidth}
%	\centering
	\includegraphics[width=0.80\textwidth]{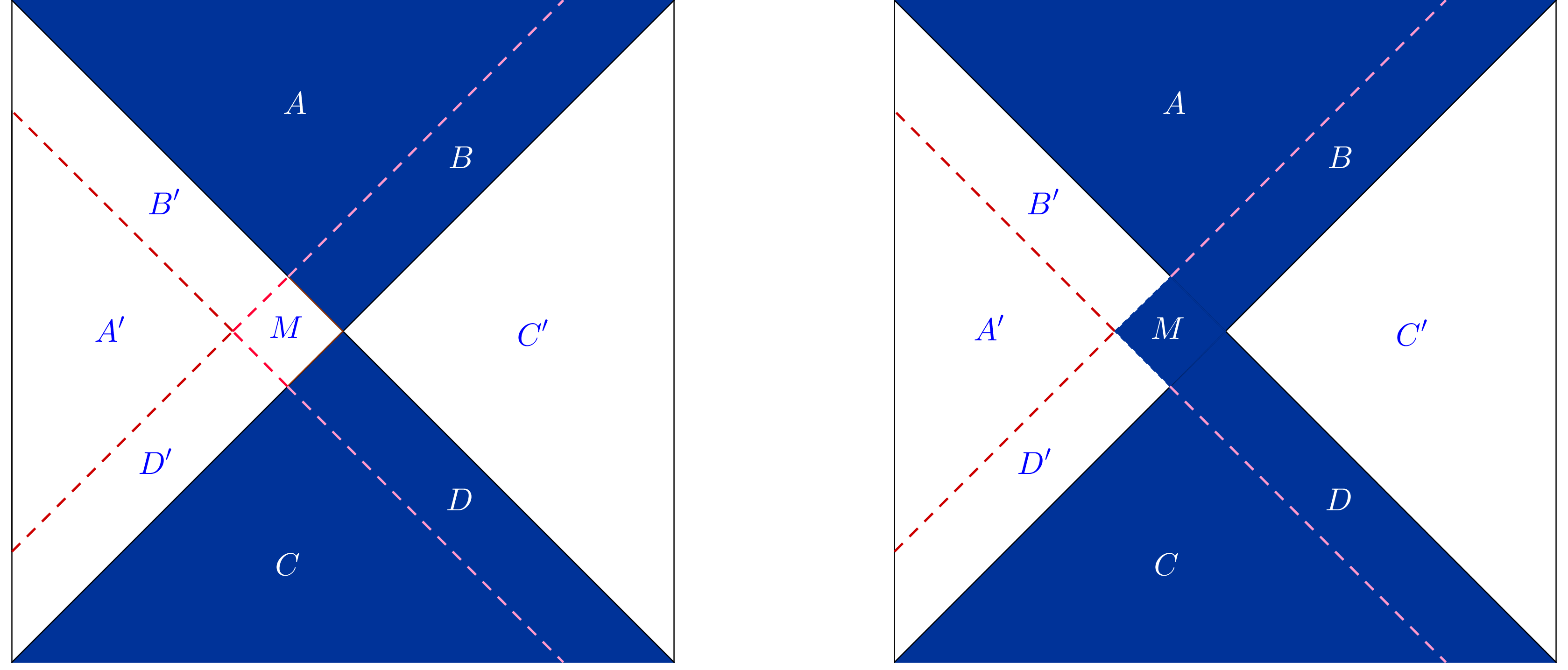}
	\caption{Interaction of $M$ with $A, B, C, D, A', B', C', D'$}   
	\label{fign:Cone1}
%	\end{minipage}
\end{figure}
%\end{center}
\noindent Indeed, we notice that in the first image, the white square $M$ interacts with the dark regions $A, B, C, D$, while in the second the now dark square $M$ interacts with the regions $A', B', C', D'$, and all the other interactions are unmodified. Therefore, the difference between the $s$-perimeter of $\mathcal{K}$ and that of $\mathcal{K'}$ consists only of the interactions $I(A,M)+I(B,M)+I(C,M)+I(D,M)-I(A',M)- I(B',M)-I(C',M)-I(D',M)$. But $A \cup B=A' \cup B'$ and $C\cup D=C'\cup D'$ (since these sets are all unbounded), therefore the difference is null, and the $s$-perimeter of $\mathcal{K}$ is equal to that of $\mathcal{K}'$. Consequently, $\mathcal{K}'$ is also $s$-minimal, and therefore it satisfies the Euler-Lagrange equation in \eqref{ELsmin} at the origin. But this leads to a contradiction, since the the dark region now contributes more than the white one, namely
	\[ \int_{\R^2} \frac{\chi_{\mathcal{\C K}'}(y) - \chi_{ \mathcal{K}'}(y)} {|y|^{2+s} } \, dy<0.\]
Thus $\mathcal{K}$ cannot be $s$-minimal, and this concludes our proof.
\end{proof}
%\begin{center}
\begin{figure}[htpb]
%\sidecaption
%	\hspace{0.8cm}
%	\begin{minipage}[b]{0.95\linewidth}
%	\centering
	\includegraphics[width=0.55\textwidth]{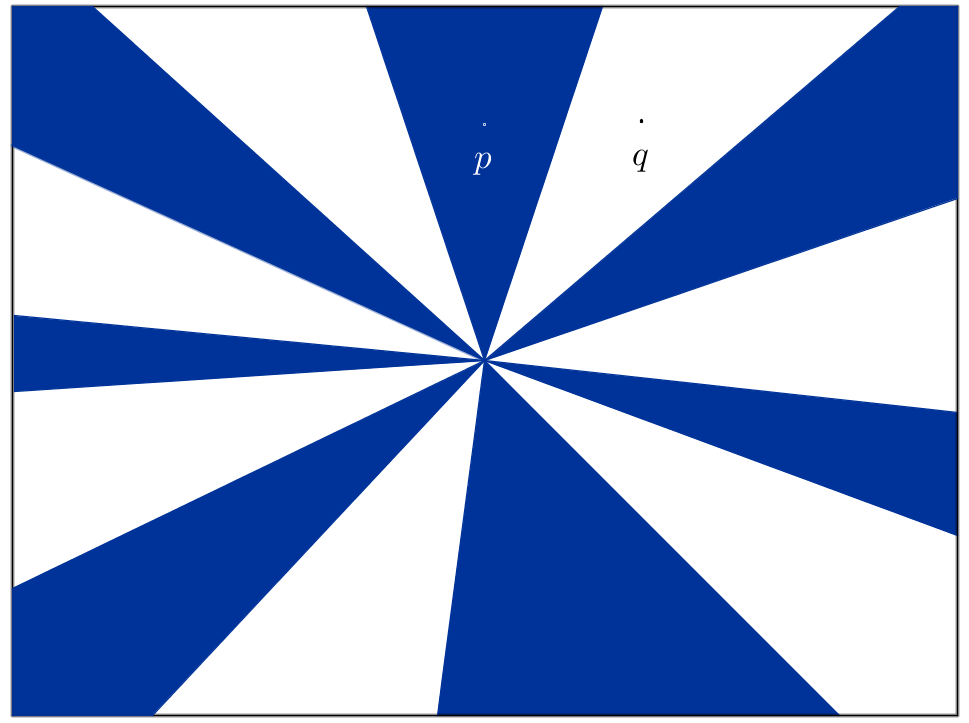}
	\caption{Cone in $\R^2$}   
	\label{fign:Cone2}
%	\end{minipage}
\end{figure}
%\end{center}

This geometric argument cannot be extended to a more general case (even, for instance, to a cone in $\R^2$ made of many sectors, see Figure \ref{fign:Cone2}).
As a matter of fact, the proof of Theorem~\ref{noconestwo}
will be completely different than the one of Proposition~\ref{THAT} and it will rely on
an appropriate domain perturbation argument.

The proof of Theorem~\ref{noconestwo} that we present here
is actually different than the original one in~\cite{SV13}.
Indeed,
in \cite{SV13}, the result was proved
by using the harmonic extension for the fractional Laplacian.
Here, the extension will not be used; furthermore, 
the proof follows the steps of Theorem \ref{dgdim2} and we will recall here just the main ingredients.

\begin{proof}[Proof of Theorem \ref{noconestwo}]

The idea of the proof is the following: if $E\subset \R^2$ is an $s$-minimal cone, then let $\tilde{E}$ be a perturbation of the set $E$ which coincides with a translation of $E$ in $B_{R/2}$ and with $E$ itself outside $B_R$. Then the difference between the energies of $\tilde{E}$ and $E$ tends to $0$ as $R\to +\infty$. This implies that also the energy of $E \cap \tilde{E}$ is arbitrarily close to the energy of $E$. On the other hand if $E$ is not a half-plane, the set $\tilde{E} \cap E$ can be modified locally to decrease its energy by a fixed small amount and we reach a contradiction.

The details of the proof go as follows.
Let \[ u:= \chi_{\C E} - \chi_E.\] We use now 
definition \eqref{dguab} and Theorem \ref{acenest1} with $s\in (0,1)$ instead of $\sigma:=2s$ as the power of the denominator (check the observation in the footnote at page \pageref{ssigma} 
and Theorem \ref{TH1-SV-GAMMA} further on). We have that
	\[ u(B_R,B_R)= 8 I(E\cap B_R,\C E \cap B_R)\] and\[ u(B_R, {\C}B_R)=4I(B_R\cap E,\C E\setminus B_R) + 4I(\C E \cap B_R,E\setminus B_R),\]
	thus (up to constants that we neglect)	
	\eqlab{ \label{peren1} \text{Per}_s (E,B_R) = \mathcal{K}_R (u),}
	where $\mathcal{K}_R (u) $ is given in \eqref{dgkenac} and $\text{Per}_s(E,B_R)$ is the $s$-perimeter functional defined in \eqref{nmspf1}. Then $E$ is $s$-minimal if $u$ is a minimizer of the energy $\K_R$ in any ball $B_R$, with $R>0$. 
	
	Now, we argue by contradiction, and suppose that $E$ is an $s$-minimal cone different from the half-space. Up to rotations, we may suppose that a sector of $E$ has an angle smaller than $\pi$ and is bisected by $e_2$. Thus there exists $M\geq 1$ and $p\in E\cap B_{M}$ on the $e_2$-axis such that (see
Figure \ref{fign:Cone2})
\[ p\pm e_1 \in \C E. \]  
We take
$\varphi\in C^\infty_0(B_1)$, such that 
$\varphi (x)=1 $ in $B_{1/2}$.
For $R$ large (say $R > 8M$), we define 
\[ \Psi_{R,+}(y):=y+\varphi \Big(\frac{ y}{R}\Big)\,e_1 .\] 
We point out that,
for $R$ large, $\Psi_{R,+}$ is a diffeomorphism on~$\R^2$.\\
Furthermore, we define~$u_R^+(x):= u(\Psi_{R,+}^{-1}(x))$. Then
	\begin{equation*}
		\begin{aligned}
		&u_R^+(p)= u(p-e_1) &\text{ for } &p\in B_{2M} \\
	{\mbox{and }}\;\;
	& u_R^+(p)= u(p) &\text { for } &p \in \C B_R.
		\end{aligned}
 \end{equation*}  
We recall the estimate obtained in \eqref{DG01}, that, combined with the minimality of $u$, gives  
 	\begin{equation*}	 {\mathcal{K}}_R( u_R^+)- {\mathcal{K}}_R( u)\leq \frac{C}{R^2} \mathcal{K}_R(u).\end{equation*}
But $u$ is a minimizer in any ball, and by the energy estimate in Theorem \ref{acenest1} we have that
   \[ \mathcal{K}_R(u_R^+)  -\mathcal{K}_R(u)  \leq CR^{-s}.\]
This implies that
		\begin{equation} \label{urplus11} \lim_{R \to+ \infty} \mathcal{K}_R(u_R^+)  -\mathcal{K}_R(u) =0.
		\end{equation}
Let now
	\[ v_R(x):= \max \{ u(x), u_R^+(x) \} \quad \quad \text{ and } \quad \quad w_R(x):=\min \{ u(x), u_R^+(x) \}.\]
 We claim that $v_R$ is not identically $u$ nor $u_R^+$. Indeed, since $p\pm e_1  \in \C E$ and $p\in E$
	\[\begin{split}  &u_R^+ (p) = u(p-e_1) = (\chi_{\C E} - \chi_{ E} )(p-e_1) =1 \quad \mbox{and}  \\
				&u(p) = (\chi_{\C E} - \chi_{E} )(p) =- 1 .\end{split}\] 
On the other hand,
	\[\begin{split}  & u_R^+(p+e_1) = u(p)=-1 \quad \mbox{and}  \\
& u(p+e_1) = (\chi_{\C E} - \chi_{ E} )(p+e_1) = 1 .\end{split}\] 
By the continuity of $u$ and $u_R^+ $, we obtain that 
	\begin{equation}\label{neigP}
		v_R=u_R^+ > u \text{ in a neighborhood of } p \end{equation} and
		\begin{equation}\label{neigPe}  v_R =u >  u_R^+ \text{ in a neighborhood of } p+e_1.
		\end{equation}
By the minimality of~$u$,
	\[\mathcal K_R (u) \leq \mathcal K_R (v_R).\]  
Moreover (see e.g.
formula~(38) in~\cite{PSV13}), 
 	\[ \mathcal K_R (v_R)+\mathcal K_R (w_R)\le
\mathcal K_R (u) +\mathcal K_R (u_R^+).   \]
The latter two formulas give that 
	\begin{equation}
	\mathcal K_R (v_R) \leq \mathcal K_R (u_R^+).
	\label{wrur}
	\end{equation} 
We claim that 
\begin{equation}\label{I*8ui}
{\mbox{$v_R$ is not minimal for $\mathcal K_{2M}$}}\end{equation}
with respect to compact perturbations in $B_{2M}$.
Indeed, assume by contradiction that $v_R$ is minimal, then in $B_{2M}$ both $v_R$ and $u$ would satisfy the same equation. 
Recalling~\eqref{neigPe}
and applying the Strong Maximum Principle, 
it follows that $u=v_R$ in $B_{2M}$, which contradicts \eqref{neigP}.
This establishes~\eqref{I*8ui}.

Now, we consider a minimizer~$u_R^*$ of~$\mathcal K_{2M}$
among the competitors that agree with~$v_R$ outside~$B_{2M}$.
Therefore, we can define
	\[  \delta_R  : = \mathcal K_{2M} (v_R)- \K_{2M} (u_R^*). \]
In light of~\eqref{I*8ui}, we have that~$\delta_R>0$.

The reader can now compare
Step 3 in the proof of Theorem \ref{dgdim2}. There we proved that 
\begin{equation}\label{89*9*9}
{\mbox{$\delta_R$
remains bounded away from zero as~$R\to+\infty$.}}\end{equation}
Furthermore, since $u_R^*$ and $v_R$ agree outside $B_{2M}$ we obtain that 
	\[\mathcal K_R (u_R^*) +\delta_R = \mathcal K_R (v_R).\]
Using this, \eqref{wrur} and the minimality of $u$, we obtain that
	\[\delta_R  =\mathcal K_R (v_R) - \mathcal K_R (u_R^*) \leq \mathcal K_R (u_R^+) -\mathcal K_R (u).\]
Now we send $R$ to infinity, recall \eqref{urplus11} and~\eqref{89*9*9},
and
we reach a contradiction. Thus, $E$ is a half-space, and this concludes the proof of Theorem \ref{noconestwo}.
\end{proof}

As already mentioned,
the regularity theory for $s$-minimal sets is still widely open.
Little is known beyond Theorems~\ref{THM 5.8}
and~\ref{THM 5.9}, so it would be very
interesting to further investigate the regularity of~$s$-minimal surfaces
in higher dimension
and for small~$s$.\bigskip

It is also interesting to recall that if the $s$-minimal surface~$E$
is a subgraph of some function~$u:\R^{n-1}\to\R$
(at least in the vicinity of some point~$x_0=(x_0', u(x'_0))\in\partial E$)
then the Euler-Lagrange~\eqref{ELsmin} can be written
directly in terms of~$u$.
For instance (see formulas~(49) and~(50) in~\cite{bootstrap}),
under appropriate smoothness assumptions on~$u$, formula~\eqref{ELsmin}
reduces to
\begin{eqnarray*}
0&=&
\int_{\Rn} \frac{\chi_{\C E}(x_0+y) -\chi_{E} (x_0+y) }{|y|^{n+s}}\,dy
\\ &=&\int_{\R^{n-1}} F\left( \frac{u(x'_0+y')-u(x'_0)}{|y'|}\right)\,
\frac{\zeta(y')}{|y'|^{n-1+s}}\,dy'+\Psi(x'_0),\end{eqnarray*}
for suitable~$F$ and~$\Psi$, and a cut-off function~$\zeta$
supported in a neighborhood of~$x_0'$.

%\begin{center}
\begin{figure}[htpb]
%\sidecaption
%	\hspace{0.8cm}
	\begin{minipage}[b]{0.95\linewidth}
%	\centering
        \includegraphics[width=0.85\textwidth]{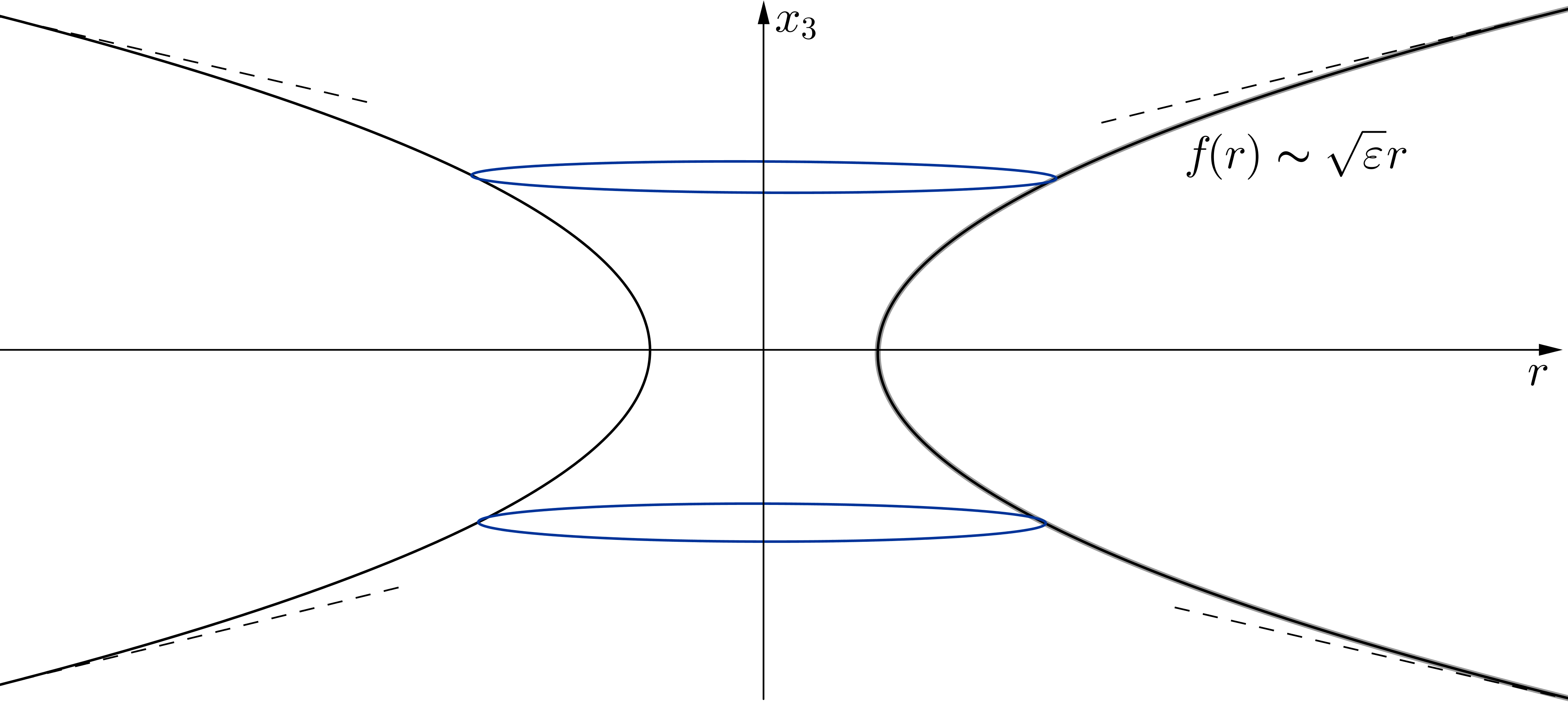}
        \caption{A nonlocal catenoid}
        \label{FRAC_CATENOID}
    \end{minipage}
\end{figure}
%\end{center}

Regarding the regularity problems
of the $s$-minimal surfaces, let us mention the recent papers~\cite{Lawson}
and~\cite{Lawson2}. Among other very interesting results,
it is proved there that suitable singular cones
of symmetric type are unstable up to dimension~$6$
but become stable in dimension~$7$ for small~$s$
(these cones can be seen as the nonlocal analogue
of the Lawson cones in the classical minimal surface theory,
and the stability property is in principle weaker than minimality,
since it deals with the positivity of the second order derivative
of the functional).

This phenomenon may suggest the conjecture that the
$s$-minimal surfaces 
may develop singularities in dimension~$7$ and higher
when~$s$ is sufficiently small.
\bigskip

In~\cite{Lawson2}, interesting examples of
surfaces with vanishing nonlocal mean curvature
are provided for~$s$ sufficiently close to~$1$. Remarkably, the surfaces in~\cite{Lawson2}
are the nonlocal analogues of the catenoids, but, differently
from the classical case
(in which catenoids grow logarithmically),
they approach a singular cone at infinity, see Figure~\ref{FRAC_CATENOID}.

Also, these nonlocal catenoids are highly unstable from the
variational point of view, since they possess infinite Morse index
(differently from the standard catenoid, which has Morse index
equal to one, i.e. it is, roughly speaking,
a minimizer in any functional direction with the exception of one).
\bigskip

Moreover, in~\cite{Lawson2}, there are also examples
of surfaces with vanishing nonlocal mean curvature
that can be seen as the nonlocal analogues of two
parallel hyperplanes. Namely,
for~$s$ sufficiently close to~$1$, there exists
a surface of revolution made of two sheets
which are the graph of a radial function~$f=\pm f(r)$.
When~$r$ is small, $f$ is of the order of~$1+(1-s) r^2$,
but for large~$r$ it becomes of the order of~$\sqrt{1-s}\cdot r$.
That is, the two sheets ``repel each other'' and produce a linear
growth at infinity. When~$s$ approaches $1$ the two sheets are
locally closer and closer to two parallel hyperplanes,
see Figure~\ref{FRAC_2_SHEETS}.

The construction above may be extended to build families
of surfaces with vanishing nonlocal mean curvature
that can be seen as the nonlocal analogue of $k$
parallel hyperplanes, for any~$k\in\N$. These $k$-sheet surfaces
can be seen as the bifurcation, as~$s$ is close to~$1$,
of the parallel hyperplanes~$\{x_n = a_i\}$, for~$i\in\{1,\dots,k\}$,
where the parameters~$a_i$ satisfy the constraints
\begin{equation}\label{1.9}
a_1 >\dots> a_k,\qquad \sum_{i=1}^k a_i=0\end{equation}
and the balancing relation
\begin{equation}\label{1.10}
a_i =2\sum_{{1\le j\le n}\atop{j\ne i}}
\frac{(-1)^{i+j+1}}{a_i-a_j}.
\end{equation}
%\begin{center}
\begin{figure}[htpb]
%\sidecaption
%	\hspace{0.8cm}
%	\begin{minipage}[b]{0.95\linewidth}
%	\centering
        \includegraphics[width=0.85\textwidth]{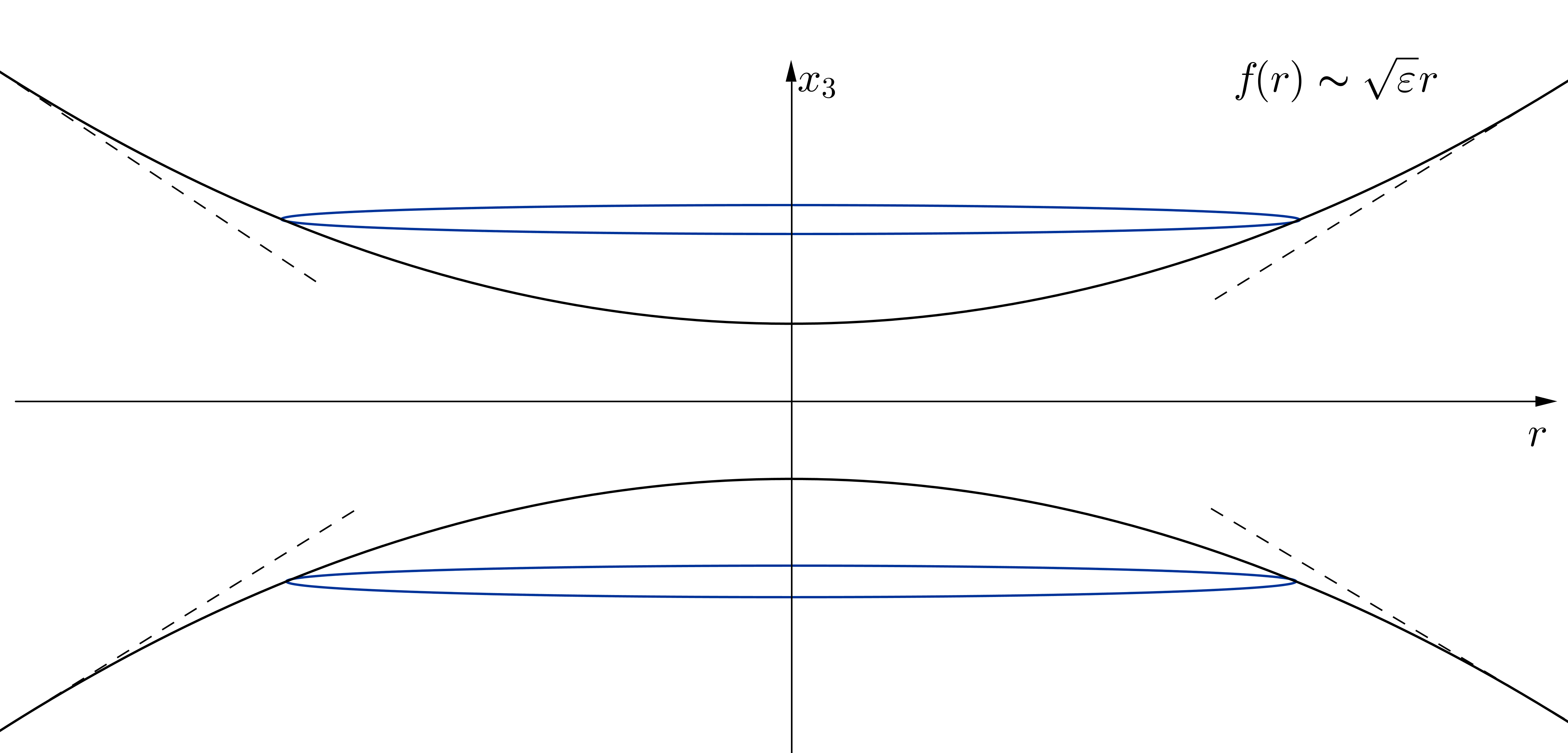}
        \caption{A two-sheet surface with vanishing fractional mean curvature}
        \label{FRAC_2_SHEETS}
     %  \end{minipage}
\end{figure}
%\end{center}
It is actually quite interesting to observe that
solutions of~\eqref{1.10} correspond to (nondegenerate)
critical points of
the functional
$$ E(a_1,\dots,a_k):=\frac12\sum_{i=1}^k a_i^2+
\sum_{{1\le j\le n}\atop{j\ne i}}(-1)^{i+j}\log|a_i-a_j|$$
among all the $k$-ples~$(a_1,\dots,a_k)$ that satisfy~\eqref{1.9}.
\bigskip

These bifurcation techniques
rely on a careful expansion of the fractional perimeter functional
with respect to normal perturbations. That is, if~$E$
is a (smooth) set with vanishing fractional mean curvature,
and~$h$ is a smooth and compactly supported perturbation,
one can define, for any~$t\in\R$,
$$ E_h(t):= \{ x+t h(x)\nu(x),\; x\in\partial E\},$$
where~$\nu(x)$ is the exterior normal of~$E$ at~$x$.
Then, the second variation of the perimeter of~$E_h(t)$
at~$t=0$ is (up to normalization constants)
\begin{eqnarray*}
&& \int_{\partial E} \frac{h(y)-h(x)}{|x-y|^{n+s}}
\,d{\mathcal{H}}^{n-1}(y) + h(x)\,
\int_{\partial E} \frac{\big(\nu(x)-\nu(y)\big)\cdot\nu(x)}{|x-y|^{n+s}}
\,d{\mathcal{H}}^{n-1}(y) \\ &=&
\int_{\partial E} \frac{h(y)-h(x)}{|x-y|^{n+s}}
\,d{\mathcal{H}}^{n-1}(y) + h(x)\,
\int_{\partial E} \frac{1-\nu(x)\cdot\nu(y)}{|x-y|^{n+s}}
\,d{\mathcal{H}}^{n-1}(y)
.\end{eqnarray*}
Notice that the latter integral is non-negative, since
$\nu(x)\cdot\nu(y)\le1$.
The quantity above, in dependence of the perturbation~$h$,
is called, in jargon, ``Jacobi operator''.
It encodes an important geometric information,
and indeed, as~$s\to1$, it approaches the classical operator
$$ \Delta_{\partial E} h +|A_{\partial E}|^2\, h,$$
where~$\Delta_{\partial E}$ is the Laplace-Beltrami operator
along the hypersurface~$\partial E$ and~$|A_{\partial E}|^2$
is the sum of the squares of the
principal curvatures.
\bigskip

Other interesting sets that possess constant nonlocal mean curvature
with the structure of onduloids have been recently constructed 
in \cite{davilaaa} and 
in~\cite{cabre-fall-weth}. This type of sets are periodic in a given direction 
and their construction has perturbative nature (indeed, the sets are 
close to a slab in the plane).

It is interesting to remark that the planar objects constructed
in~\cite{cabre-fall-weth} have no counterpart in the local framework, since
hypersurfaces of constant classical mean curvature with an onduloidal structure
only exist
in~$\R^n$ with~$n\ge3$: once again, this is a typical nonlocal effect,
in which the nonlocal mean curvature at a point is influenced by
the global shape of the set.\bigskip

While unbounded sets with constant nonlocal mean curvature
and interesting geometric features
have been constructed in~\cite{Lawson2, cabre-fall-weth},
the case of smooth and bounded sets is always geometrically trivial.
As a matter of fact, it has been recently proved independently in~\cite{cabre-fall-weth} and~\cite{figalli-maggi-ciraolo}
that bounded sets with smooth boundary
and constant mean curvature
are necessarily balls (this is the analogue
of a celebrated result by Alexandrov for surfaces
of constant classical mean curvature).
\bigskip 

\section{Boundary regularity}The boundary regularity of the nonlocal minimal surfaces
is also a very interesting, and surprising, topic.
Indeed, differently from the classical case,
nonlocal minimal surfaces do not always attain boundary data
in a continuous way (not even in low dimension).
A possible boundary behavior is, on the contrary,
a combination of stickiness to the boundary
and smooth separation from
the adjacent portions. Namely, the nonlocal minimal surfaces
may have a portion that sticks at the boundary and that
separates from it in a~$C^{1,\frac{1+s}2}$-way.
%\begin{center}
\begin{figure}[htpb]
%\sidecaption
%	\hspace{0.8cm}
%	\begin{minipage}[b]{0.95\linewidth}
%	\centering
        \includegraphics[width=0.40\textwidth]{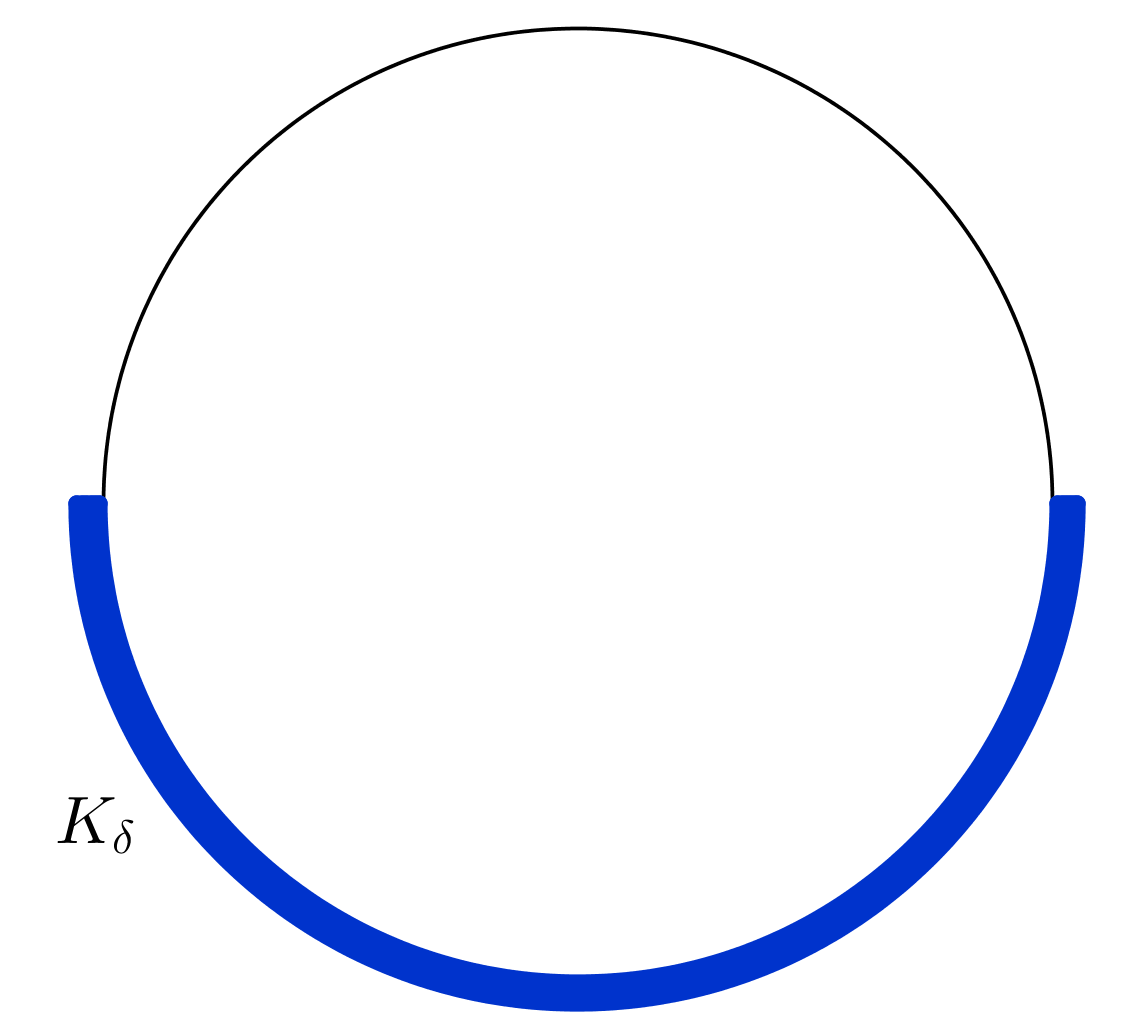}
        \caption{Stickiness properties of Theorem~\ref{STI-DSV-1}.}
        \label{eST1}
   %     \end{minipage}
\end{figure}
%\end{center}
As an example, we can
consider, for any~$\delta>0$, the spherical cap
$$ K_\delta := \big( B_{1+\delta}\setminus B_1\big)\cap \{x_n<0\},$$
and obtain the following stickiness result:

\begin{theorem}\label{STI-DSV-1}
There exists~$\delta_0>0$, depending on~$n$ and~$s$, such that
for any~$\delta\in(0,\delta_0]$, we have that
the $s$-minimal set in~$B_1$
that coincides with~$K_\delta$ outside~$B_1$
is~$K_\delta$ itself.

That is, the $s$-minimal set with datum~$K_\delta$ outside~$B_1$
is empty inside~$B_1$.
\end{theorem}

The stickiness property of Theorem~\ref{STI-DSV-1}
is depicted in Figure~\ref{eST1}.

Other stickiness examples occur at the sides of slabs in the plane.
For instance,
given~$M>1$, one can 
consider the $s$-minimal set~$E_M$ in~$(-1,1)\times\R$
with datum outside~$(-1,1)\times\R$ given by
the ``jump'' set~$J_M:=J^-_M \cup J^+_M$,
where
\begin{eqnarray*}
&& J^-_M:= (-\infty,-1]\times (-\infty,-M)
\\{\mbox{and }}&&J^+_M:= [1,+\infty)\times (-\infty,M).\end{eqnarray*}
Then, if~$M$ is large enough,
the minimal set~$E_M$ sticks at the boundary of the slab:

\begin{theorem}\label{STI-DSV-2}
There exist~$M_o>0$, $C_o>0$, depending on~$s$, such that
if~$M\ge M_o$ then
\begin{eqnarray}
&& [-1,1)\times [C_o M^{\frac{1+s}{2+s}},M]\subseteq E_M^c \label{SL-011}
\\{\mbox{and }}&& (-1,1]\times [-M,-C_o M^{\frac{1+s}{2+s}}]\subseteq E_M.
\label{SL-012}
\end{eqnarray}
\end{theorem}

The situation of Theorem~\ref{STI-DSV-2}
is described in Figure~\ref{eST2}.
%\begin{center}
\begin{figure}[htpb]
%\sidecaption
%	\hspace{0.8cm}
%	\begin{minipage}[b]{0.95\linewidth}
%	\centering
        \includegraphics[width=0.75\textwidth]{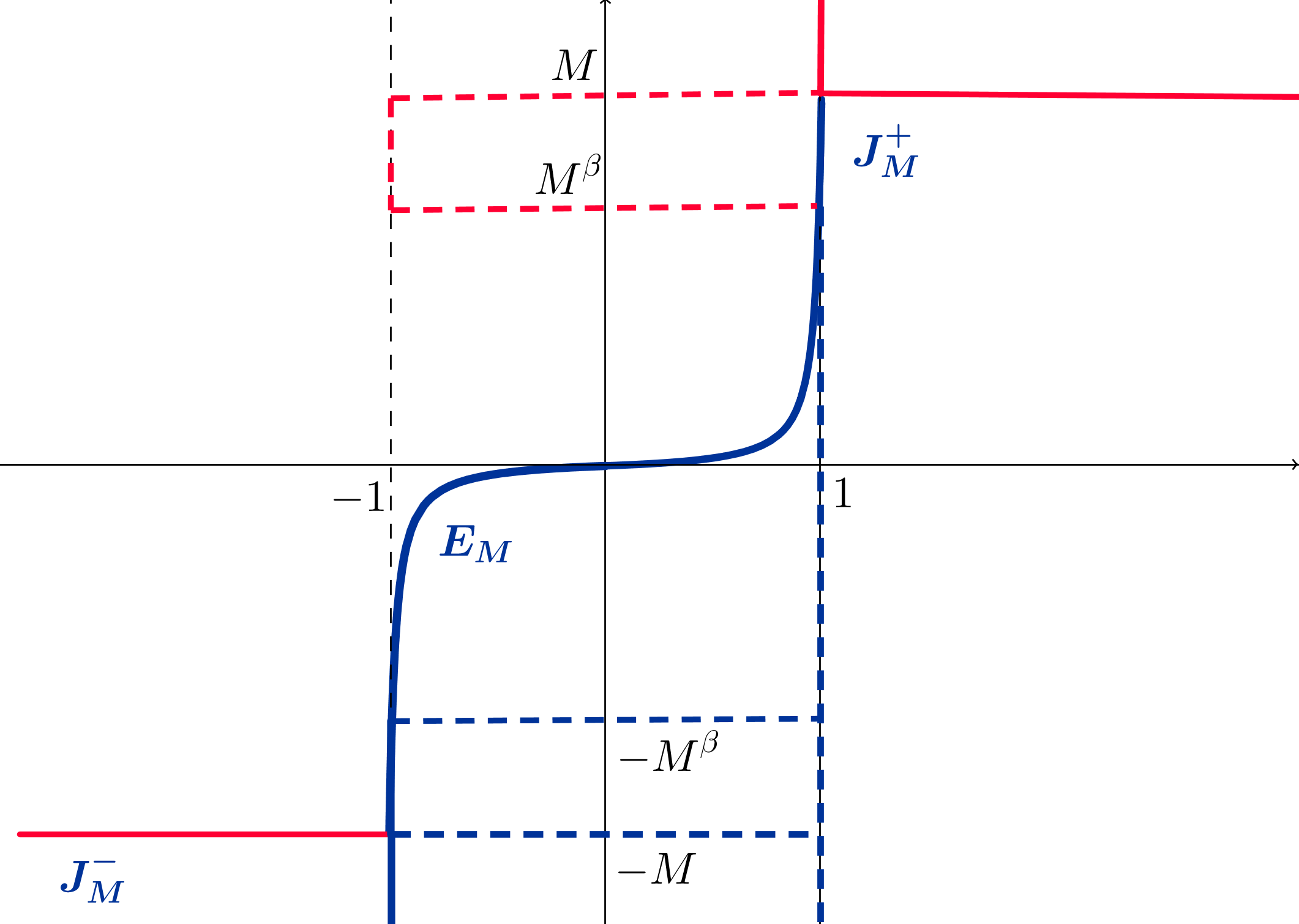}
        \caption{Stickiness properties of Theorem~\ref{STI-DSV-2}.}
        \label{eST2}
%        \end{minipage}
\end{figure}
%\end{center}
We mention that the ``strange''
exponent~$\frac{1+s}{2+s}$ in~\eqref{SL-011}
and~\eqref{SL-012} is optimal.\bigskip

For the detailed proof of
Theorems~\ref{STI-DSV-1} and~\ref{STI-DSV-2},
and other results on the 
boundary behavior
of nonlocal minimal surfaces, see~\cite{STICK}.
Here, we limit ourselves to give some 
heuristic motivation and a sketch of the proofs.
\bigskip

As a motivation for the (somehow unexpected) stickiness
property at the boundary,
one may look at Figure~\ref{eST1} and argue like this.
In the classical case, corresponding to~$s=1$,
independently on the width~$\delta$,
the set of minimal perimeter in~$B_1$ will always be
the half-ball~$B_1\cap \{x_n<0\}$. 

Now let us take~$s<1$.
Then, the half-ball~$B_1\cap \{x_n<0\}$ cannot be an $s$-minimal
set, since the nonlocal mean curvature, for instance, at the origin
cannot vanish. Indeed, the origin ``sees'' the complement 
of the set in a larger proportion than the set itself.
More precisely, in~$B_1$ (or even in~$B_{1+\delta}$)
the proportion of the set is
the same as the one of the complement, but outside~$B_{1+\delta}$
the complement of the set is dominant. Therefore, to ``compensate''
this lack of balance, the $s$-minimal set for~$s<1$ has to bend
a bit. Likely, the $s$-minimal set in this case will have the tendency
to become slightly convex at the origin, so that, at least nearby,
it sees a proportion of the set which is larger than the proportion of the complement
(we recall that, in any case, the proportion of the complement will
be larger at infinity, so the set needs to compensate at least near
the origin). But when~$\delta$ is very small, it turns
out that this compensation is not sufficient to obtain the desired
balance between the set and its complement: therefore,
the set has to ``stick'' to the half-sphere, in order to drop
its constrain to satisfy a vanishing nonlocal mean curvature equation.

Of course some quantitative estimates
are needed to make this argument work,
so we describe the sketch of the rigorous proof of
Theorem~\ref{STI-DSV-1} as follows.

\begin{proof}[Sketch of the proof of Theorem~\ref{STI-DSV-1}]
First of all, one checks that for any fixed~$\eta>0$,
if~$\delta>0$ is small enough, we have that
the interaction between~$B_1$ and~$B_{1+\delta}\setminus B_1$
is smaller than~$\eta$. In particular, by comparing with
a competitor that is empty in~$B_1$, by minimality we obtain that
\begin{equation}\label{ETA:DE} {\text{Per}}_s(E_\delta, B_1)\le \eta,\end{equation}
where we have denoted by~$E_\delta$ the
$s$-minimal set in~$B_1$
that coincides with~$K_\delta$ outside~$B_1$.

Then, one checks that
\begin{equation}\label{small:n}
{\mbox{the boundary of~$E_\delta$
can only lie in a small neighborhood of~$\partial B_1$}}\end{equation}
if~$\delta$ is sufficiently small.\\
Indeed, if, by contradiction,
there were points of~$\partial E_\delta$
at distance larger than~$\epsilon$
from~$\partial B_1$, then one could find
two balls of radius comparable to~$\epsilon$,
whose centers lie at distance larger than~$\epsilon/2$
from~$\partial B_1$ and at mutual distance smaller than~$\epsilon$,
and such that one ball is entirely contained in~$B_1\cap E_\delta$
and the other ball is entirely contained in~$B_1\setminus E_\delta$
(this is due to a 
Clean Ball Condition, see Corollary 4.3 in~\cite{nms}).
As a consequence,
${\text{Per}}_s(E_\delta, B_1)$ is bounded from below
by the interaction of these two balls, which is at least of
the order of~$\epsilon^{n-s}$. Then, 
we obtain a contradiction with~\eqref{ETA:DE}
(by choosing~$\eta$ much smaller than~$\epsilon^{n-s}$,
and taking~$\delta$ sufficiently small).

This proves~\eqref{small:n}. {F}rom this, it follows
that 
\begin{equation}\label{HJ:pr:ba}
\begin{split}
&{\mbox{the whole  set~$E_\delta$ must lie in a small neighborhood of~$\partial B_1$.}}\end{split}\end{equation}
Indeed, if this were not so, by~\eqref{small:n} the set~$E_\delta$
must contain a ball of radius, say~$1/2$. Hence,
${\text{Per}}_s(E_\delta, B_1)$ is bounded from below
by the interaction of this ball against~$\{x_n>0\}\setminus B_1$,
which would produce a contribution of order one,
which is in contradiction with~\eqref{ETA:DE}.

Having proved~\eqref{HJ:pr:ba}, one can use it to
complete the proof of
Theorem~\ref{STI-DSV-1} employing a geometric argument.
Namely, one considers the ball~$B_\rho$,
which is outside~$E_\delta$ for small~$\rho>0$, in virtue of~\eqref{HJ:pr:ba},
and then enlarges~$\rho$ untill it touches~$\partial E_\delta$.
If this contact occurs at some point~$p\in B_1$,
then the nonlocal mean curvature of~$E_\delta$ at~$p$ must be zero.
But this cannot occur (indeed, we know by~\eqref{HJ:pr:ba}
that the contribution of~$E_\delta$ to the nonlocal mean curvature
can only come from a small neighborhood of~$\partial B_1$,
and one can check, by estimating integrals, that this is not
sufficient to compensate the outer terms in which the complement of~$E_\delta$
is dominant).

As a consequence, no touching point between~$B_\rho$
and~$\partial E_\delta$ can occur in~$B_1$, which shows that~$E_\delta$
is empty inside~$B_1$ and completes the proof of
Theorem~\ref{STI-DSV-1}.
\end{proof}

As for the proof of Theorem~\ref{STI-DSV-2}, 
the main arguments are based on sliding a ball of suitably large radius
till it touches the set, with careful quantitative estimates.
Some of the details are as follows (we refer to~\cite{STICK}
for the complete arguments).

\begin{proof}[Sketch of the proof of Theorem~\ref{STI-DSV-2}]
The first step is to prove
a weaker form of stickiness as the one claimed in Theorem~\ref{STI-DSV-2}.
Namely, one shows that
\begin{eqnarray}
&& [-1,1)\times [c_oM\,,M]\subseteq E_M^c \label{SL-011-PRE}
\\{\mbox{and }}&& (-1,1]\times [-M,\,-c_o M]\subseteq E_M,
\label{SL-012-PRE}
\end{eqnarray}
for some~$c_o\in(0,1)$. Of course, the statements in~\eqref{SL-011}
and~\eqref{SL-012} are stronger than
the ones in~\eqref{SL-011-PRE}
and~\eqref{SL-012-PRE} when~$M$ is large, since~${\frac{1+s}{2+s}}<1$,
but we will then obtain them later in a second step.

To prove~\eqref{SL-011-PRE},
one takes balls of radius~$c_oM$ and centered at~$\{x_2=t\}$,
for any~$t\in [c_o M,\,M]$. One slides these balls from left to right,
till one touches~$\partial E_M$. When~$M$ is large enough
(and~$c_o$ small enough) this contact point cannot lie
in~$\{|x_1|<1\}$. This is due to the fact that
at least the sliding ball lies outside~$E_M$,
and the whole~$\{x_2>M\}$ lies outside~$E_M$ as well.
As a consequence, these contact points see a proportion of~$E_M$
smaller than the proportion of the complement
(it is true that
the whole of~$\{x_2<-M\}$ lies inside~$E_M$,
but this contribution comes from further away than the ones just
mentioned, provided that~$c_o$ is small enough).
Therefore, contact points cannot satisfy a vanishing mean curvature
equation and so they need to lie on the boundary of the domain
(of course, careful quantitative estimates are
necessary here, see~\cite{STICK},
but we hope to have given an intuitive sketch of the computations
needed).

In this way, one sees that all the portion
$[-1,1)\times [c_oM\,,M]$
is clean from the set~$E_M$
and so~\eqref{SL-011-PRE} is established
(and~\eqref{SL-012-PRE} can be proved similarly).

Once~\eqref{SL-011-PRE} and~\eqref{SL-012-PRE} are established, one uses them
to obtain the strongest form expressed
in~\eqref{SL-011} and~\eqref{SL-012}. For this,
by~\eqref{SL-011-PRE} and~\eqref{SL-012-PRE},
one has only to take care of points
in~$\{ |x_2|\in [C_o M^{ \frac{1+s}{2+s} }, \,c_oM]\}$.
For these points, one can use again a sliding method,
but, instead of balls, 
one has to use suitable surfaces obtained by appropriate
portions of balls and adapt the calculations in order to evaluate
all the contributions arising in this way.

The computations are not completely obvious (and once again we refer
to~\cite{STICK} for full details), but the idea is, once again,
that contact points that are
in the set $\{ |x_2|\in [C_o M^{ \frac{1+s}{2+s} }, \,c_oM]\}$
cannot satisfy the balancing relation prescribed by
the vanishing nonlocal mean curvature equation.
\end{proof}

The stickiness property discussed above also has an interesting
consequence in terms of the ``geometric stability'' of
the flat $s$-minimal surfaces. 
For instance, rather surprisingly, 
the flat lines in the plane are ``geometrically
unstable'' nonlocal minimal surfaces,
in the sense that an arbitrarily small and compactly supported
perturbation can produce a
stickiness phenomenon at the boundary of the domain.
Of course, the smaller the perturbation, the smaller
the stickiness phenomenon, but it is quite relevant that
such a stickiness property
can occur for arbitrarily small (and ``nice'')
perturbations. This means that $s$-minimal flat objects,
in presence of a perturbation, may not only ``bend''
in the center of the domain, but rather ``jump''
at boundary points as well.
\bigskip

To state this phenomenon in a mathematical framework, one can consider,
for fixed~$\delta>0$ the planar sets
\begin{eqnarray*}
&& H:=\R\times(-\infty,0),\\
&& F_-:=(-3,-2)\times [0,\delta)\\
{\mbox{and }}&&
F_+:= (2,3)\times [0,\delta).\end{eqnarray*} 
One also fixes a set~$F$
which contains~$H\cup F_-\cup F_+$
and denotes by~$E$ be the
$s$-minimal set in~$(-1,1)\times\R$ among all the sets that coincide
with~$F$ outside~$(-1,1)\times\R$. 
%\begin{center}
\begin{figure}[htpb]
%\sidecaption
%	\hspace{0.8cm}
%	\begin{minipage}[b]{0.95\linewidth}
%	\centering
        \includegraphics[width=0.75\textwidth]{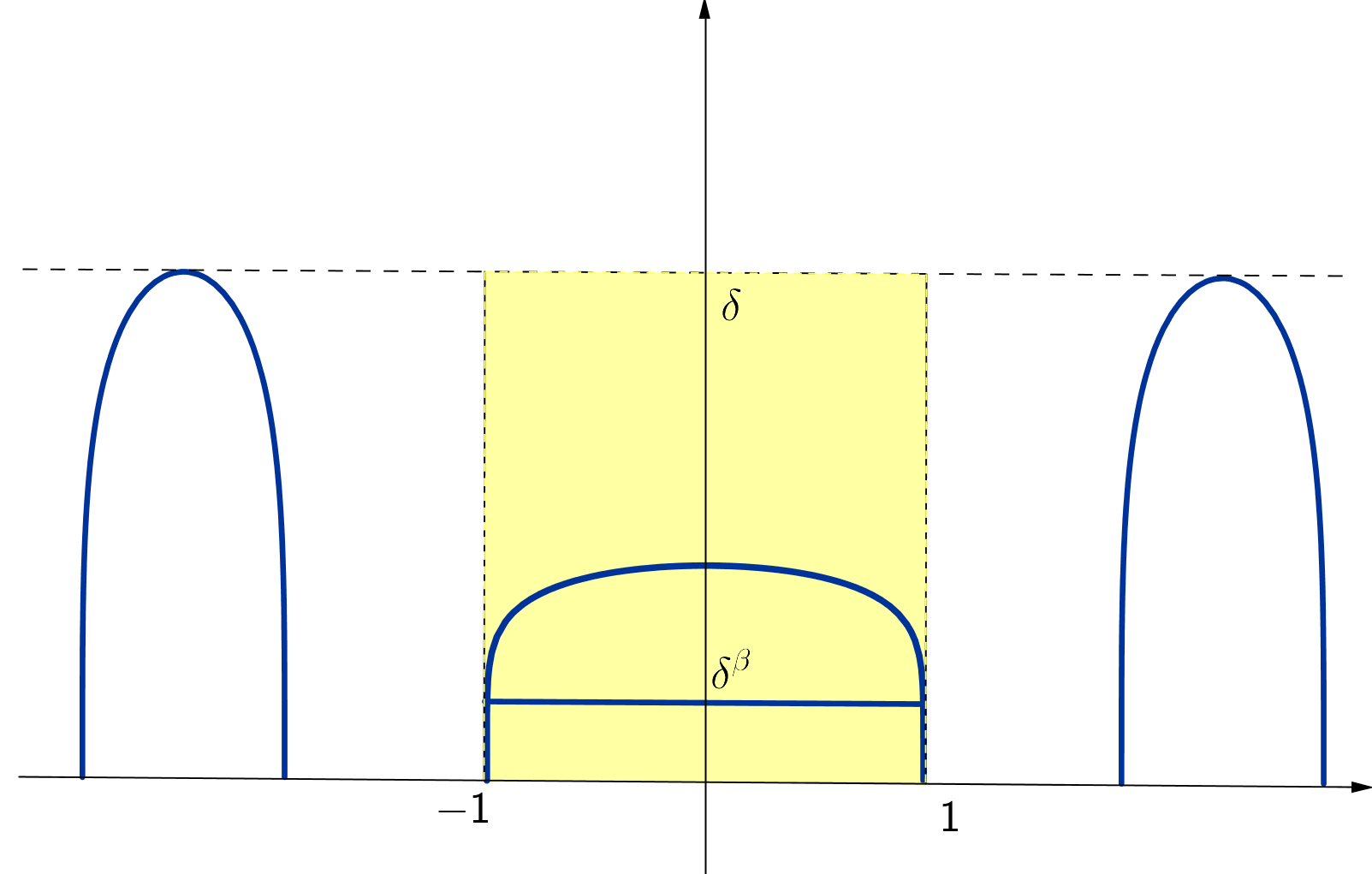}
        \caption{The stickiness/instability property in Theorem~\ref{UNS},
with~$\beta:=\frac{2+\epsilon_0}{1-s}$}
        \label{PUNS}
%        \end{minipage}
\end{figure}
%\end{center}
Then, this set~$E$ sticks at the boundary of the domain, according
to the next result:

\begin{theorem}\label{UNS}
Fix~$\epsilon_0>0$ arbitrarily small.
Then, there exists~$\delta_0>0$, possibly depending on~$\epsilon_0$,
such that, for any~$\delta\in(0,\delta_0]$,
$$ E\supseteq (-1,1)\times (-\infty, \delta^{\frac{2+\epsilon_0}{1-s}} ].$$
\end{theorem}

The stickiness/instability property in Theorem~\ref{UNS}
is depicted in Figure~\ref{PUNS}. We remark that
Theorem~\ref{UNS}
gives a rather precise quantification
of the size of the stickiness in terms of the size
of the perturbation: namely the size of the stickiness in Theorem~\ref{UNS}
is larger than the size of the perturbation
to the power~$\beta:=\frac{2+\epsilon_0}{1-s}$,
for any~$\epsilon_0>0$ arbitrarily small.
Notice that~$\beta\to+\infty$
as~$s\to1$, consistently
with the fact that
classical minimal surfaces do not stick at the boundary.
\bigskip

The proof of Theorem~\ref{UNS}
is based on the construction of suitable auxiliary barriers.
These barriers are used to detach a portion of the set
in a neighborhood of the origin and their construction
relies on some compensations of nonlocal integral terms.
In a sense, the building blocks of these barriers are
``self-sustaining solutions'' that
can be seen as the geometric counterparts of
the $s$-harmonic function~$x_+^s$ discussed in Section~\ref{shf1}.

Indeed, roughly speaking, like the function~$x_+^s$,
these barriers ``see'' a proportion of the set
in~$\{x_1<0\}$ larger than what is produced by their tangent plane,
but a proportion smaller than that at infinity, due to their sublinear
behavior. Once again, the computations needed to
check such a balancing conditions are a bit involved,
and we refer to~\cite{STICK} for the complete details.
%\begin{center}
\begin{figure}[htpb]
%\sidecaption
%	\hspace{0.8cm}
%	\begin{minipage}[b]{0.95\linewidth}
%	\centering
        \includegraphics[width=0.65\textwidth]{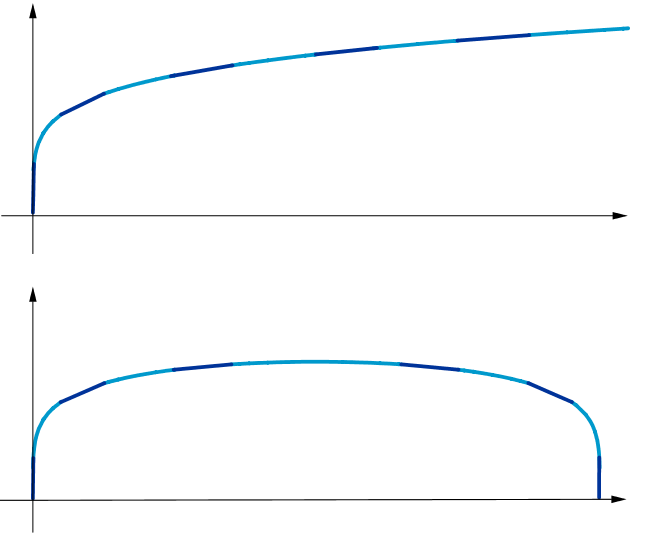}
        \caption{Auxiliary barrier for the proof of Theorem~\ref{UNS}}
        \label{barr}
%      \end{minipage}
\end{figure}
%\end{center}
\bigskip

To conclude this section, we make a remark on the connection between 
solutions of the fractional Allen-Cahn equation and $s$-minimal 
surfaces. Namely, a suitably scaled version of
the functional in \eqref{enfac}
$\Gamma$-converges to either the classical perimeter
or the nonlocal perimeter functional, depending
on the fractional parameter.
The $\Gamma$-convergence
is a type of convergence of functionals
that is compatible with the
minimization of the energy, and turns out to be very useful when dealing with variational problems indexed by a parameter. This notion was introduced by De Giorgi, see e.g.~\cite{degg}
for details. 

Let us denote $\sigma:=2s \in(0,2)$ in the definition of the functional in \eqref{enfac}. This choice is related to the observation in the footnote at page \pageref{ssigma}. 
In the nonlocal case, some care is needed to introduce
the ``right'' scaling
of the functional, which
comes from the dilation invariance of the 
space coordinates and possesses a nontrivial energy
in the limit.
For this, one takes first the rescaled energy functional
 	\[ J_\eee (u, \Omega) := \eee^{\sigma} \K (u,\Omega)+ 
\int_{\Omega} W(u)\, dx,\]
where $\mathcal{K}$ is the kinetic energy defined in \eqref{kenac} (where we replace $2s$ with $\sigma$).
Then, one considers the functional
\sys[F_\eee (u,\Omega) :=]
	 {&  \eee^{-\sigma} J_\eee(u, \Omega) & & \mbox{ if } \sigma \in(0,1), \\
	&|\eee \log \eee|^{-1} J_\eee(u, \Omega) & & \mbox{ if }\sigma=1,\\
	& \eee^{-1} J_\eee(u, \Omega) & & \mbox{ if }\sigma\in(1,2). }
The limit functional of~$F_\eee$ as~$\eee\to0$ depends on~$\sigma$.
Namely, when~$\sigma \in(0,1)$, the limit functional is
(up to dimensional constants that we neglect) the fractional
perimeter, i.e.
\syslab[F(u,\Omega):=]
{\label{FF1} 
&{\text{Per}}_{\sigma}(E,\Omega)
&& \mbox{ if } u|_\Omega = \chi_E -\chi_{\C E}, \; \mbox{ for some set } E\subset
\Omega\\
& +\infty & & {\mbox{ otherwise.}} }
On the other hand, when~$\sigma \in[1,2)$, the limit functional of~$F_\eee$
is (again, up to normalizing constants) the classical
perimeter, namely
\syslab[ F(u, \Omega):=] {\label{FF2}
&{\text{ Per}} (E,\Omega)
&& {\mbox{ if }} u|_\Omega = \chi_E -\chi_{ \C E},\; \mbox{ for some set }E\subset
\Omega \\
& +\infty && {\mbox{ otherwise,}}
}
That is, the following limit statement holds true:

\begin{theorem}\label{TH1-SV-GAMMA}
Let~$\sigma \in(0,\,2)$. Then,
$F_\eee$ $\Gamma$-converges to~$F$, as defined
in either~\eqref{FF1} or~\eqref{FF2},
depending on whether~$\sigma\in(0,1)$ or~$\sigma \in[1,2)$.
\end{theorem}

For precise statements and further details, see~\cite{savingamma}.

We remark here that Theorem \ref{TH1-SV-GAMMA} clarifies now why we take $\sigma \in (0,1)$ (that we denoted and will denote in this chapter by $s$) when defining our nonlocal operators of this chapter, i.e. the fractional perimeter and the fractional mean curvature (check again the footnote at page \pageref{ssigma}).

Additionally, we remark that the level sets
of the minimizers of
the functional in \eqref{enfac}, 
after a homogeneous scaling in the space variables,
converge
locally uniformly to minimizers 
either of the fractional perimeter (if $\sigma\in(0,1)$)
or of the classical perimeter (if~$\sigma \in[1,2)$):
that is, the ``functional'' convergence stated in
Theorem~\ref{TH1-SV-GAMMA}
has also a ``geometric'' counterpart: for this, see
Corollary~1.7 in~\cite{SV14}.

One can also interpret Theorem~\ref{TH1-SV-GAMMA} by saying that
a nonlocal phase transition possesses two parameters, $\eee$
and~$s$/$\sigma$. When~$\eee\to0$, the limit interface
approaches a minimal surface either in the fractional
case (when~$\sigma\in(0,1)$) or in the classical case (when~$\sigma\in[1,2)$).
This bifurcation at~$\sigma=1$ somehow states that
for lower values of~$\sigma$ the nonlocal phase transition
possesses a nonlocal interface in the limit, but for larger values of~$\sigma$
the limit interface is characterized only by local features
(in a sense, when~$\sigma\in(0,1)$ the ``surface tension effect''
is nonlocal, but for~$\sigma\in[1,2)$ this effect localizes).

It is also interesting to compare Theorems~\ref{TP12}
and~\ref{TH1-SV-GAMMA}, since
the bifurcation at~$\sigma=1$ detected by Theorem~\ref{TH1-SV-GAMMA}
is perfectly compatible with the limit behavior of the
fractional perimeter, which reduces to the classical
perimeter exactly for this value of~$\sigma$, as stated in Theorem~\ref{TP12}.

\section[Complete stickiness at the boundary]{Complete stickiness at the boundary of nonlocal minimal surfaces for small values of the fractional perimeter}[Stickiness at the boundary]\label{luke}

In this section, we deal with the behavior of nonlocal minimal surfaces when the fractional parameter (that we denote by $s\in (0,1)$) is small. In particular
\begin{itemize}
\item we give the asymptotic behavior of the fractional mean curvature as $s\to 0^+$, 
\item we classify the behavior of $s$-minimal surfaces, in dependence of the exterior data at infinity.
\end{itemize}  
Moreover, we prove the continuity of the fractional mean curvature in all variables for $s\in [0,1]$.

The results in this section take their inspiration from \cite{asympt1,STICK}. It is a known result, see  \cite[Corollary 5.3]{nms} that when the exterior data  is a half-space, the $s$-minimal set itself is the same half-space.
On the other hand, as we prove here, by just removing some small set from the half-space, for $s$ small enough the $s$-minimal set completely sticks to the boundary. 

This section is organized as follow. We give some  preliminary results on the contribution from infinity of sets in Subsection \ref{contr_infty}. 

In Subsection \ref{classify}, we consider an exterior data ``occupying at infinity'' in measure, with respect to an appropriate weight, less than an half-space. To be precise

\eqlab{\label{mainhyp} \alpha(E_0)<\frac{\omega_n}{2}. } 
In this hypothesis:
\begin{itemize}
\item we give some asymptotic estimates of the density, in particular showing that when $s$ is small enough, $s$-minimal sets cannot fill their domain. 
\item  we give some estimates on the fractional mean curvature. In particular we show that if a set $E$ has an exterior tangent ball of radius $\delta$ at some point $p\in \partial E$,  then the $s$-fractional mean curvature  of $E$ at $p$ is strictly positive for every $s<s_\delta$.
\item we prove that when the fractional parameter is small and the exterior data at infinity occupies (in measure, with respect to the weight) less than half the space, then $s$-minimal sets completely stick at the boundary (that is, they are empty inside the domain), or become ``topologically dense'' in their domain. A similar result, which says that nonlocal minimal surfaces fill the domain or their complementaries become dense, can be obtained in the same way, when the exterior data occupies in the appropriate sense more than half the space (so this threshold is somehow optimal,). 
\item we narrow the set of minimal sets that become dense in the domain for $s$ small. As a matter of fact, if the exterior data does not completely surrounds the domain, $s$-minimal sets completely stick at the boundary. 
\end{itemize}
In Subsection \ref{sectexamples}, we provide some examples in which we are able to explicitly compute the contribution from infinity of sets. 
Subsection \ref{cont} contains the continuity of the fractional mean curvature operator in all its variables for $s\in[0,1]$.  As a corollary, we show that for $s\to 0^+$ the fractional mean curvature at a regular point of the boundary of a set, takes into account only the behavior of that set at infinity.  Furthermore, the continuity property implies that the mean curvature at a regular point on the boundary of a set may change sign, as $s$ varies, 
depending on the signs of the two asymptotics as $s\to1^-$ and $s\to0^+$. 

In the last Section \ref{appendicite} we collect some useful results that we use in this paper.

%%%%%%%%%%%%%%SECTION%%%%%%%%%%%%%%%%%%%%

\subsection{Statements of the main results}

 We remark that the quantity $\alpha$ (defined in \eqref{E:LIM:LE}) may not exist (see Example 2.8 and 2.9 in \cite{asympt1}). For this reason, we also define
\eqlab{\label {baralpha1} \overline \alpha (E):= \limsup_{s\to 0^+} s\int_{\C B_1} \frac{\chi_E(y)}{|y|^{n+s}}\, dy ,\quad \quad \underline \alpha(E) := \liminf_{s\to 0^+} s\int_{\C B_1} \frac{\chi_E(y)}{|y|^{n+s}}\, dy.}

This set parameter plays an important role in describing
the asymptotic behavior of the fractional mean curvature as~$s\to0^+$
for unbounded sets. As a matter of fact, the limit as~$s\to0^+$
of the fractional mean curvature for a {\em bounded} set
is a positive, universal constant (independent of the set),
see e.g. (Appendix~B in \cite{senonlocal}).
On the other hand, this asymptotic behavior changes for {\em unbounded}
sets, due to the set function $\alpha(E)$, as described explicitly
in the following result:
  \begin{theorem}\label{asympts}
Let $E\subset\Rn$ and let $p\in\partial E$ be such that $\partial E$ is $C^{1,\gamma}$ near $p$,
for some $\gamma\in(0,1]$. Then
\bgs{& \liminf_{s\to0^+} s\,\I_s[E](p) =\omega_n -2 \overline \alpha(E)\\
& \limsup_{s\to0^+} s\,\I_s[E](p) =\omega_n-2 \underline\alpha(E).}
\end{theorem}
We notice that if~$E$ is bounded, then $\underline\alpha(E)
=\overline\alpha(E)=\alpha(E)=0$, hence Theorem~\ref{asympts}
reduces in this case to formula~(B.1) in  \cite{senonlocal}.
Actually, we can estimate the fractional mean curvature from below (above) uniformly with respect to the radius of the exterior
(interior) tangent ball to $E$. To be more precise, if there exists an exterior tangent ball at $p\in \partial E$ of radius $\delta>0$, then for every $s<s_\delta$ we have 
\[ \liminf_{\rho \to 0^+} s\,\I^\rho_s[E](p)\geq \frac{\omega_n -2\overline \alpha(E)}4.\]    
More explicitly, we have the following result:

   \begin{theorem}\label{positivecurvature}
Let $\Omega\subset\Rn$ be a bounded open set. 
Let $E_0\subset\C\Omega$ be such that
\eqlab{\label{weak_hp_beta}\overline \alpha(E_0)<\frac{\omega_n}2,}
and let
\[\beta=\beta(E_0):=\frac{\omega_n-2\overline \alpha(E_0)}4.\] We define 
\eqlab{\label{delta_wild_index_def}
\delta_s=\delta_s(E_0):=e^{-\frac{1}{s}\log \frac{\omega_n+2\beta}{\omega_n+\beta}} ,}
for every $s\in(0,1)$.
Then, there exists $s_0=s_0(E_0,\Omega)\in(0,\frac{1}{2}]$ such that, if $E\subset\Rn$ is such that $E\setminus\Omega=E_0$
and $E$ has an exterior tangent ball
of radius (at least) $\delta_\sigma$, for some $\sigma\in(0,s_0)$, at some point $q\in\partial E\cap\overline{\Omega}$, then
\eqlab{\label{unif_pos_curv}\liminf_{\rho\to0^+}\I_s^\rho[E](q)\geq\frac{\beta}{s}>0,\qquad\forall\,s\in(0,\sigma].}
 \end{theorem}

Given an open set $\Omega\subset\R^n$ and $\delta\in\R$, we consider the open set
\[\Omega_\delta:=\{x\in\R^n\,|\,\bar{d}_\Omega(x)<\delta\},\]
where $\bar{d}_\Omega$ denotes the signed distance function from $\partial\Omega$, negative inside $\Omega$.

It is well known (see e.g. \cite{trudy,Ambrosio})
that if $\Omega$ is bounded and $\partial \Omega$ is of class $C^2$, then the distance function is also of class $C^2$
in a neighborhood of $\partial\Omega$. Namely, there exists $r_0>0$
such that
\[\bar{d}_\Omega\in C^2(N_{2r_0}(\partial\Omega)),\quad
\mbox{ where }\quad N_{2r_0}(\partial\Omega):=\{x\in\R^n\,|\,|\bar{d}_\Omega(x)|<2r_0\}.\]
As a consequence, since $|\nabla\bar{d}_\Omega|=1$,
the open set $\Omega_\delta$ has $C^2$ boundary for every $|\delta|<2r_0$.
For a more detailed discussion, see Appendix A.2 and the references cited therein.

The constant $r_0$ will have the above meaning throughout this paper.

\smallskip

We give the next definition.
\begin{defn}\label{wild}
   Let $\Omega\subset \Rn$ be an open, bounded set.
   We say that a set $E$ is $\delta$-{dense} in $\Omega$ for some fixed $\delta>0$ if $|B_\delta(x)\cap E|>0$ for any $x\in \Omega$ for which $B_\delta(x)\subset\subset\Omega$.
  \end{defn}
  
\noindent Notice that if $E$ is $\delta$-dense  then $E$ cannot have an exterior tangent ball of radius greater or equal than $\delta$ at any point $p\in \partial E\cap \Omega_{-\delta}$.

\noindent We observe that the notion for a set of
being $\delta$-dense is a ``topological'' notion,
rather than a measure theoretic one. 
Indeed, $\delta$-dense sets need not be ``irregular'' nor ``dense'' in the measure theoretic sense (see Remark \ref{deltadance}).

\smallskip
With this definition and using Theorem \ref{positivecurvature} we obtain the following classification.
\begin{theorem}\label{THM}
  Let $\Omega$ be an  bounded and  connected open set with $C^2$ boundary. Let $E_0\subset \C \Omega$ such that\[\overline \alpha(E_0)<\frac{\omega_n}{2}.\]  
 Then the following two results hold.\\
  A)  Let $s_0$ and $\delta_s$ be as in Theorem \ref{positivecurvature}. There exists
  $s_1=s_1(E_0,\Omega)\in (0,s_0]$ such that if $s<s_1$ and $E$ is an $s$-minimal set in $\Omega$ with exterior data $E_0$, then either
     \bgs{(A.1) \;  E\cap \Omega=\emptyset \quad  \mbox{ or} \quad\; (A.2)\;  E \mbox{ is } \delta_s-\mbox{dense}.}
    %     Let $\Omega$ be an  bounded and  connected open set with $C^2$ boundary. Let $E_0\subset \C \Omega$ such that\[\overline \alpha(E_0)<\frac{\omega_n}{2}.\] %and let $s_0$ and $\delta_s$ be as in Theorem \ref{positivecurvature}.  
%Then,
 \noindent
 B) Either \\
(B.1) there exists
  $\tilde s=\tilde s(E_0,\Omega)\in (0,1)$ such that if $E$ is an $s$-minimal set in $\Omega$ with exterior data $E_0$ and $s\in(0,\tilde s)$, then
     \bgs{  E\cap \Omega=\emptyset,}
     or \\
    (B.2)    there exist  $\delta_k \searrow 0$, $s_k \searrow 0$ and a sequence of sets  $E_k$ such that each $E_k$ is $s_k$-minimal in $\Omega$ with exterior data $E_0$ and for every $k$
     \bgs{ \partial E_k \cap B_{\delta_k}(x) \neq \emptyset \quad \forall \; B_{\delta_k}(x)\subset\subset \Omega.}
     \end{theorem}

We remark here that Definition \ref{wild} allows the $s$-minimal surface 
to completely fill $\Omega$. The next theorem states that for $s$ small enough (and $\overline \alpha(E)<\omega_n/2$) we can exclude  this possibility.
\begin{theorem}\label{notfull}

Let $\Omega\subset\R^n$ be a bounded open set of finite classical perimeter and let $E_0\subset\C \Omega$ be such that
\[\overline{\alpha}(E_0)<\frac{\omega_n}{2}.\]
For every $\delta>0$ and every $\gamma\in(0,1)$ there exists $\sigma_{\delta,\gamma}=\sigma_{\delta,\gamma}(E_0,\Omega)\in(0,\frac{1}{2}]$ such that if $E\subset\R^n$ is $s$-minimal in $\Omega$, with exterior data $E_0$ and $s<\sigma_{\delta,\gamma}$, then
\eqlab{\label{1666}\big|(\Omega\cap B_\delta(x))\setminus E\big|\ge\gamma\, \frac{\omega_n-2\overline{\alpha}(E_0)}{\omega_n-\overline{\alpha}(E_0)}\big|\Omega\cap B_\delta(x)\big|,\qquad\forall\,x\in\overline{\Omega}.}
\end{theorem}

\begin{remark}
 Let $\Omega$ and $ E_0$ be as in Theorem \ref{notfull} and fix $\gamma=\frac{1}{2}$.
 \begin{enumerate}
\item Notice that we can find $\bar \delta >0$ and  $\bar x \in \Omega$ such that
\[ B_{2\bar\delta} (\bar x ) \subset \Omega.\]
Now if $s<\sigma_{\bar \delta,\frac{1}{2}}$ and $E$ is $s$-minimal in $\Omega$ with respect to $E_0$, \eqref{1666} says that
\[ |B_{\bar\delta} (\bar x ) \cap \C E|>0.\] 
Then (since the ball is connected), either $B_{\bar\delta} (\bar x ) \subset \C E$ or there exists a point
\[x_0\in\partial E\cap \overline B_{\bar\delta} (\bar x ).\]
In this case, since $d(x_0, \partial \Omega )\ge\bar \delta$, Corollary 4.3 of \cite{nms} implies that
\[B_{\bar\delta c_s}(z)\subset\C E\cap B_{\bar \delta}(x_0)\subset\C E\cap\Omega\] for some $z$, where $c_s\in(0,1]$ denotes the constant of the clean ball condition (as introduced in Corollary 4.3 in \cite{nms}) and depends only on $s$ (and $n$). In both case, there exists a ball of radius $\bar \delta c_s$ contained in $\C E \cap \Omega$. 
\item If $s<\sigma_{\bar \delta,\frac{1}{2}}$ and $E$ is $s$-minimal and $\delta_s$-dense, then 
we have that
\[\delta_s>c_s\bar \delta.\]
On the other hand, we have an explicit expression for $\delta_s$, given in \eqref{delta_wild_index_def}. Therefore, if one could prove that $c_s$ goes to zero slower than $\delta_s$, one could exclude the existence of $s$-minimal sets that are $\delta_s$-dense (for all sufficiently small $s$). 
\end{enumerate}
\end{remark}
\smallskip

     An interesting result is related to $s$-minimal sets whose exterior data does not completely surround $\Omega$. In this case, the $s$-minimal set, for small values of $s$, is always
empty in $\Omega$. More precisely:

 \begin{theorem}\label{boundedset}
  Let $\Omega$ be a  bounded and  connected open set with $C^2$ boundary. Let $E_0\subset \C \Omega$ such that  \[\overline \alpha(E_0)<\frac{\omega_n}{2},\]
 and let $s_1$ be as in Theorem \ref{THM}. Suppose that there exists $R>0$ and $x_0\in \partial \Omega$ such that \[B_R(x_0)\setminus \Omega \subset \C E_0.\]  Then, there exists $s_3=s_3(E_0,\Omega)\in(0,s_1]$ such that if $s<s_3$ and $E$ is an $s$-minimal set in $\Omega$ with exterior data $E_0$, then 
    \[  E\cap \Omega=\emptyset .\]
     \end{theorem}

We notice that Theorem \ref{boundedset} prevents the existence 
of $s$-minimal sets that are $\delta$-dense (for any $\delta$).    

\begin{remark}
The indexes $s_1$ and $s_3$ are defined as follows
\[s_1:=\sup\{s\in(0,s_0)\,|\,\delta_s<r_0\}\]
and
\[s_3:=\sup\Big\{s\in(0,s_0)\,\big|\,\delta_s<\frac{1}{2}\min\{r_0,R\}\Big\}.\]
Clearly, $s_3\leq s_1\leq s_0$.
\end{remark}

\begin{remark} We point out that condition \eqref{weak_hp_beta}
is somehow optimal. Indeed,
when $\alpha(E_0)$ exists and 
\[ \alpha(E_0)=\frac{\omega_n}2,\]
several configurations may occur, depending on the position of $\Omega$ with respect to the exterior data $E_0\setminus \Omega$. As an example, take 
\[ \mathfrak P =\{ (x',x_n) \; \big| \; x_n> 0\}.\] Then, for any $\Omega\subset \Rn$  bounded open set with $C^2$ boundary, the only $s$-minimal set with exterior data given by $\mathfrak P \setminus \Omega$ is $\mathfrak P$ itself. So, if $E$ is $s$-minimal with respect to $E_0\setminus \Omega$ then
\bgs{&\Omega\subset \mathfrak P & \quad \implies \quad &E\cap \Omega=\Omega\\
 & \Omega\subset \Rn \setminus \mathfrak P &\quad \implies \quad &E\cap \Omega=\emptyset.} 
 On the other hand, if one takes $\Omega= B_1$, then 
 \[  E\cap B_1 = \mathfrak P  \cap B_1.\] 
 
 As a further example, we consider the supergraph
 \[ E_0:=\{(x',x_n) \; \big| \; x_n > \tanh x_1\},\] for which we have that (see Example \ref{tanh})
 \[\alpha(E_0)=\frac{\omega_n}2.\]  Then for every $s$-minimal set in $\Omega$ with exterior data $E_0\setminus \Omega$, we have that
 \bgs{&\Omega\subset \{ (x',x_n) \; \big| \; x_n> 1\} & \quad \implies \quad &E\cap \Omega=\Omega\\
 & \Omega\subset \{ (x',x_n) \; \big| \; x_n<-1\} &\quad \implies \quad &E\cap \Omega=\emptyset.} 
Taking $\Omega=B_2$, we have by the maximum principle in Proposition \ref{maximum_principle}  that every set $E$ which is $s$-minimal in $B_2$, with respect to $E_0\setminus B_2$, satisfies
\bgs{  B_2\cap  \{ (x',x_n) \; \big| \; x_n> 1\}\subset E,   \qquad 
  B_2 \cap \{ (x',x_n) \; \big| \; x_n<-1\} \subset \C E .} 
 On the other hand, we are not able to establish what happens in $B_2\cap \{ (x',x_n) \; \big| \; -1<x_n< 1\} $.
\end{remark}

\begin{remark}
We notice that when $E$ is $s$-minimal in $\Omega$ with respect to $E_0$, then $\C E$ is $s$-minimal in $\Omega$ with respect to $\C E_0$. Moreover
\[ \underline \alpha(E_0) >\frac{\omega_n}{2} \qquad \implies \qquad \overline \alpha (\C E)< \frac{\omega_n}{2}.\]
So in this case we can apply Theorems \ref{positivecurvature}, \ref{THM}, \ref{notfull} and \ref{boundedset} to $\C E$ with respect to $\C E_0$. For instance, if
$E$ is $s$-minimal in $\Omega$ with exterior data $E_0$ with
\[ \underline \alpha(E_0) >\frac{\omega_n}{2}, \]
and $s<s_1(\C E_0, \Omega)$,
 then either
\[ E\cap \Omega=\Omega \qquad \mbox{ or }  \qquad  \C E \; \mbox{ is } \; \delta_s(\C E_0)-\mbox{dense}.\]
 The analogues of the just mentioned Theorems can be obtained similarly.
\end{remark}

We point out that from our main results and the last two remarks, we have a complete classification of nonlocal minimal surfaces when $s$ is small whenever
\[ \alpha(E_0)\neq  \frac{\omega_n}{2} .\] 

In the last Subsection \ref{cont} of the paper, we prove the continuity of the fractional mean curvature in all variables (see Theorem \ref{everything_converges} and Proposition \ref{propsto0}). As a consequence, we have the following result. 

\begin{prop}\label{rsdfyish}
Let $E\subset\R^n$ and let $p\in\partial E$ such that $\partial E$ is $C^{1,\alpha}$ in $B_R(p)$ for some
$R>0$ and $\alpha\in(0,1]$. Then the function
\[\I_{(-)}[E](-):(0,\alpha)\times(\partial E\cap B_R(p))\longrightarrow\R,
\qquad(s,x)\longmapsto\I_s[E](x)\]
is continuous.\\
Moreover, if $\partial E\cap B_R(p)$ is $C^2$ and for every $x\in \partial E\cap B_R(p)$ we define
\sys[ \tilde \I_s  {[}E{]} (x):=]{ &s(1-s)\I_s[E](x),  & \mbox{ for } &s\in (0,1) \\	
						&{\omega_{n-1}} H[E](x), &\mbox{ for } &s=1,}
then the function
\[\tilde \I_{(-)}[E](-):(0,1]\times(\partial E\cap B_R(p))\longrightarrow\R,
\qquad(s,x)\longmapsto \tilde \I_s[E](x)\]
is continuous.\\
Finally, if $\partial E\cap B_R(p)$ is $C^{1,\alpha}$ and $\alpha(E)$ exists, and if for every $x\in \partial E\cap B_R(p)$ we denote
\[\tilde\I_0[E](x):=\omega_n-2\alpha(E),\]
then the function
\[\tilde \I_{(-)}[E](-):[0,\alpha)\times(\partial E\cap B_R(p))\longrightarrow\R,
\qquad(s,x)\longmapsto \tilde \I_s[E](x)\]
is continuous.
\end{prop}

As a consequence of the continuity of the fractional mean curvature and the asymptotic result in  theorem (\ref{asympts}) we
establish that, by varying the fractional parameter $s$,
the nonlocal mean curvature may change sign at a point
where the classical mean curvature is negative, as one can observe in Theorem \ref{changeyoursign}.

\subsection{Contribution to the mean curvature coming from infinity}\label{contr_infty}
%  The proofs of the main theorems are based on some preliminary results. For this,
In this section, we study in detail the quantities $\alpha(E)$, $\overline\alpha(E),\underline \alpha(E)$) as defined in \eqref{E:LIM:LE}, \eqref{baralpha1}. As a first remark, notice that these definitions  are independent on the radius of the ball (see  Observation 3 in Subsection 3.3), so we have that for any $R>0$
\eqlab{ \label{baralpha}  \overline \alpha (E)= \limsup_{s\to 0^+} s\int_{\C B_R} \frac{\chi_E(y)}{|y|^{n+s}}\, dy, \quad \underline \alpha(E) := \liminf_{s\to 0^+} s\int_{\C B_R} \frac{\chi_E(y)}{|y|^{n+s}}\, dy .}
Notice that
\[ \overline \alpha(E) = \omega_n -\underline \alpha(\C E),  \quad \underline \alpha(E) = \omega_n - \overline \alpha(\C E).\] 
We define
\[ \alpha_s(q,r,E):=\int_{\C B_r(q)} \frac{\chi_{E}(y) }{|q-y|^{n+s}} \, dy .\]
Then, the quantity $\alpha_s(q,r,E)$
somehow ``stabilizes'' for small $s$ independently on how large or where we take the ball, as rigorously given by the following result:

\begin{prop}\label{unifrq}
Let $K\subset \Rn$ be a compact set and $[a,b]\subset \R$ be a closed interval. Then  \[\lim_{s\to 0^+}s|\alpha_s(q,r,E)-\alpha_s(0,1,E)| =0\quad \mbox{ uniformly in } q \in K, r\in [a,b].\]
Moreover,
for any bounded open set $\Omega\subset \Rn$ and any fixed $r>0$,
we have that \eqlab{\label{claimalpha}\limsup_{s\to 0^+} s\inf_{q\in \overline \Omega} \alpha_s(q,r,E)= \limsup_{s\to 0^+} s\sup_{q\in \overline \Omega} \alpha_s(q,r,E)=\overline \alpha(E).}

\end{prop}
\begin{proof}
Let $K\subset \overline B_{ R}$ and let $\eps\in (0,1)$ be a fixed positive small quantity (that we will take arbitrarily small further on), such that 
\[ R>(\eps r)/(1-\eps).\]
We notice that if $x\in B_r(q)$, we have that 
$|x|<r+|q|<{R}/{\eps},$
hence $B_R(q)\subset B{R/\eps}$.
%Since $B_r(q)\subset B_{R/\eps}$ and $|q-y|\leq |q|+|y|\leq (\eps+1)|y|$,
%we obtain that
%\eqlab{\label{primariga}\int_{\C B_r(q)}\frac{\chi_E(y) }{|q-y|^{n+s}}\, dy \geq &\; \int_{\C B_{ R/\eps}} \frac{\chi_E(y) }{|q-y|^{n+s}}\, dy  \geq (1+\eps)^{-n-s} \int_{\C B_{ R/\eps}} \frac{dy}{|y|^{n+s}}\\
%=&\; \int_{\C B_{ R/\eps}} \frac{dy}{|y|^{n+s}} -\eps(n+s) \int_{\C B_{ R/\eps}} \frac{dy}{|y|^{n+s}} + {o}(\eps^2).}
We write that
\[ \int_{\C B_r(q)}\frac{\chi_E(y)}{|q-y|^{n+s}}\, dy = \int_{\C B_{ R/\eps}} \frac{\chi_E(y)}{|q-y|^{n+s}}\, dy + \int_{ B_{ R/\eps}\setminus B_r(q)} \frac{\chi_E(y)}{|q-y|^{n+s}}\, dy.\]
Now $|y-q|\geq |y|-|q| \geq (1-\eps)|y|$, thus for any $q\in \overline B_R$ 
\eqlab{\label{secondaaa1} \int_{ \C B_{ R/\eps}} \frac{\chi_E(y)}{|q-y|^{n+s}} \, dy \leq &\; (1-\eps)^{-n-s} \int_{\C B_{ R/\eps}} \frac{\chi_E(y)}{|y|^{n+s}}\, dy \\ =&\; \int_{\C B_{ R/\eps}} \frac{\chi_E(y)}{|y|^{n+s}}  \, dy +  \eps(n+s)\int_{\C B_{ R/\eps}} \frac{\chi_E(y)}{|y|^{n+s}} + o(\eps^2)  . }
Moreover 
\eqlab{\label{secondaaa11} \int_{ B_{ R/\eps}\setminus B_r(q)} \frac{\chi_E(y)}{|q-y|^{n+s}}\, dy \leq &\; \int_{B_{ R/\eps}\setminus B_r(q)} \frac{dy}{|q-y|^{n+s}} \leq \omega_n \int_r^{R/\eps+ R} t^{-s-1}\, dt\\ = &\; \omega_n \frac{r^{-s}- R^{-s}\eps^s(1+\eps)^{-s} }{s} \leq \omega_n \frac{a^{-s}- R^{-s}\eps^s(1+\eps)^{-s} }{s} . }
Therefore
\bgs{  \alpha_s(q,r,E)&\; -\alpha_s(0,R/\eps,E) =  \int_{\C B_r(q)}\frac{\chi_E(y) }{|q-y|^{n+s}}\, dy -  \int_{ \C B_{ R/\eps}} \frac{\chi_E(y)}{|y|^{n+s}} \, dy\\ \leq &\; \eps(n+s)   \int_{ \C B_{ R/\eps}} \frac{\chi_E(y)}{|y|^{n+s}} \, dy + o(\eps^2) + \omega_n \frac{a^{-s}- R^{-s}\eps^s(1+\eps)^{-s} }{s}.}
Now, 
\bgs{ \alpha_s(0, R/\eps,E) -\alpha_s(0,r,E)\leq \int_{B_{R/\eps}\setminus B_r} \frac{dy}{|y|^{n+s}}= \omega_n \frac{r^{-s}-R^{-s}\eps^s}s \leq \omega_n \frac{a^{-s}-R^{-s}\eps^s}s  .}
Moreover,
\[ |\alpha_s(0,r,E)-\alpha_s(0,1,E)|\leq \omega_n \frac{ |1-r^{-s}|}{s} \leq\omega_n \frac{\max\{ |1-a^{-s}|, |1-b^{-s}\}} s. \]
So by the triangle inequality we have that
\bgs{ &|\alpha_s(q,r,E)-\alpha_s(0,1,E)|  \leq  \eps(n+s)   \int_{ \C B_{ R/\eps}} \frac{\chi_E(y)}{|y|^{n+s}} \, dy + o(\eps^2)  \\ &\;+ \frac{\omega_n}s \big[ {a^{-s}- R^{-s}\eps^s(1+\eps)^{-s} }+   {a^{-s}-R^{-s}\eps^s}+  {\max\{ |1-a^{-s}|, |1-b^{-s}\}}\big]. }
Hence,
it holds that
\[ \limsup_{s\to 0^+}s |\alpha_s(q,r,E)-\alpha_s(0,1,E)| \leq \eps n \overline \alpha(E) +o(\eps^2),\]
uniformly in $q\in K$ and in $r\in [a,b]$.\\
Letting $\eps \to 0^+$, we obtain that \[\limsup_{s\to 0^+}s|\alpha_s(q,r,E)-\alpha_s(0,1,E)| =0.\] 
Therefore, we conclude that
 \[\lim_{s\to 0^+}s|\alpha_s(q,r,E)-\alpha_s(0,1,E)| =0,\]
uniformly in $q\in K$ and in $r\in [a,b]$.

Now, we consider $K$ such that $K=\overline \Omega$.
Since $B_r(q)\subset B_{R/\eps}$ and $|q-y|\leq |q|+|y|\leq (\eps+1)|y|$,
we obtain that
\bgs{\label{primariga}\int_{\C B_r(q)}\frac{\chi_E(y) }{|q-y|^{n+s}}\, dy \geq &\; \int_{\C B_{ R/\eps}} \frac{\chi_E(y) }{|q-y|^{n+s}}\, dy  \geq (1+\eps)^{-n-s} \int_{\C B_{ R/\eps}} \frac{\chi_E(y)}{|y|^{n+s}}\, dy.\\
%=&\; \int_{\C B_{ R/\eps}} \frac{\chi_E(y)}{|y|^{n+s}}\, dy -\eps(n+s) \int_{\C B_{ R/\eps}} \frac{dy}{|y|^{n+s}} + {o}(\eps^2).
}
 Using  this and the inequalities in \eqref{secondaaa1} and \eqref{secondaaa11}, we have that 
\bgs{  (1+\eps)^{-n-s} \int_{\C B_{ R/\eps}} \frac{\chi_E(y)}{|y|^{n+s}} \,dy \leq &\int_{\C B_r(q)}\frac{\chi_E(y)}{|q-y|^{n+s}}\, dy\\ \leq&\;  
   (1-\eps)^{-n-s} \int_{\C B_R} \frac{\chi_E(y)}{|y|^{n+s}}\, dy + \omega_n \frac{a^{-s}-R^{-s}\eps^s(1+\eps)^{-s}}s. }
Passing to limsup  and using \eqref{claimalpha} it follows that
\bgs{ &(1+\eps)^{-n} \overline \alpha(E) \leq \limsup_{s\to 0^+} s\inf_{q\in \overline \Omega} \int_{\C B_r(q)}\frac{\chi_E(y)}{|q-y|^{n+s}} \, dy\\& \leq \limsup_{s\to 0^+} s\sup_{q\in \overline \Omega}\int_{\C B_r(q)}\frac{\chi_E(y)}{|q-y|^{n+s}} \, dy  \leq  (1-\eps)^{-n} \overline\alpha(E).}
Sending $\eps \to 0$  we obtain the conclusion.
\end{proof}

\begin{remark}\label{finmeas}
Let $E\subset \Rn$ be such that $|E|<\infty$. Then
\[ \alpha(E)=0.\]
Indeed,
\[ |\alpha_s(0,1,E)|\leq |E|,\]
hence
\[ \limsup_{s\to 0} s|\alpha_s(0,1,E)|=0.\]
 \end{remark}
 
Now, we discuss some useful properties of $\overline\alpha$. 
Roughly speaking, the quantity $\overline\alpha$ takes into account
the ``largest possible asymptotic opening'' of a set, and so it
possesses nice geometric features such as monotonicity, additivity
and geometric invariances. The detailed list of these properties is
the following:

\begin{prop}\label{subsetssmin} \quad \\ 
(i) (Monotonicity) Let $E,F\subset \Rn$ be such that for some $r>0$ and $q\in \Rn$\[ E\setminus B_r(q)\subset  F\setminus B_r(q).\] Then
\[\overline \alpha(E)\leq \overline \alpha(F).\]
%, \quad \quad \underline \alpha(E) \geq \underline (F).\]
(ii) (Additivity) Let $E,F\subset \Rn$ be such that for some $r>0$ and $q\in \Rn$ \[ (E\cap F)\setminus B_r(q)= \emptyset.\] Then
\[ \overline \alpha (E\cup F)\leq \overline \alpha(E)+\overline \alpha(F).\]
%, \quad \quad \underline \alpha(E\cup F) \geq \underline \alpha(E)+\underline \alpha(F).\] 
Moreover, if $\alpha(E), \alpha(F)$ exist, then $\alpha(E\cup F)$ exists and
\[ \alpha(E\cup F)= \alpha(E)+\alpha(F).\]
(iii) (Invariance with respect
to rigid motions) Let $E\subset \Rn$, $x\in \Rn$ and $\mathcal R \in \mathcal {SO}(n)$ be a rotation. Then
\[ \overline \alpha(E+x)=\overline \alpha(E) \quad{\mbox{ and }}\quad \overline \alpha( \mathcal R E)=\overline \alpha(E).\]
(iv) (Scaling) Let $E\subset \Rn$ and $\lambda >0$. Then for some $r>0$ and $q\in \Rn$
\[ \alpha_s(q,r,\lambda E) =  \lambda^{-s} 
\alpha_s\left(\frac{q}{\lambda},\frac{r}{\lambda},E\right)
\quad{\mbox{ and }}\quad
\overline \alpha(\lambda E) =\overline \alpha(E).\] 
(v) (Symmetric difference)
Let  $E, F\subset \Rn$. Then for every $r>0$ and $q\in \Rn$
\[ |\alpha_s(q,r,E)-\alpha_s(q,r,F)|\leq \alpha_s(q,r,E\Delta F).\]
As a consequence, if $|E\Delta F|<\infty$ and $\alpha(E)$  exists, then $\alpha(F)$ exists and  
\[\alpha(E)=\alpha(F). \]
\end{prop}

\begin{proof} (i) It is enough to notice that for every $s\in (0,1)$
\[ \alpha_s(q, r,E) \leq \alpha_s(q, r,F) .\]
Then, passing to limsup and recalling \eqref{claimalpha} we conclude that
\[\overline \alpha(E)\leq \overline \alpha(F).\]
%Also, notice that
%\[ \C F\setminus B_r(q)\subset \C E\setminus B_r(q) .\] Hence,
%argueing as above and using that $\underline\alpha (E)=\omega_n-\overline \alpha(E)$, we obtain
%the desired claim. \\
(ii)  We notice that for every $s\in (0,1)$ \[ \alpha_s(q,r,E\cup F) = \alpha_s(q,r,E)+\alpha_s(q,r, F) \]  
and passing to limsup and liminf as $s\to 0^+$  we obtain the desired claim.\\
(iii) By a change of variables,
we have that
\[ \alpha_s(0,1,E+x)= \int_{\C B_1} \frac{\chi_{E+x}(y)}{|y|^{n+s}}\, dy = \int_{\C B_1(-x)} \frac{\chi_{E}(y)}{|x+y|^{n+s}}\, dy= \alpha_s(-x,1,E).\]
Accordingly, the invariance by translation 
follows after passing to limsup and using \eqref{claimalpha}. 

In addition, the invariance by rotations is obvious, using a change of variables.\\
(iv) Changing the variable $y=\lambda x$ we deduce that 
\[ \alpha_s(q,r,\lambda E)=\int_{\C B_r(q)}\frac{\chi_{\lambda E} (y)}{|q-y|^{n+s}}\, dy
=\lambda ^{-s}\int_{\C B_{\frac{r}{\lambda}}(\frac{q}{\lambda})} \frac{\chi_E(x)}{|\frac{q}{\lambda}-x|^{n+s}}\, dx
=\lambda^{-s} \alpha_s\left(\frac{q}{\lambda},\frac{r}{\lambda},E\right).\] Hence, the claim follows by passing to limsup as $s\to 0^+$.\\
(v) We have that
\[ |\alpha_s(q,r,E)-\alpha_s(q,r,F)| \leq \int_{\C B_r(q)} \frac{ |\chi_{E}(y)-\chi_F(y)|}{|y-q|^{n+s}}\, dy = \int_{\C B_r(q)} \frac{ \chi_{E\Delta F}(y)}{|y-q|^{n+s}}\, dy = \alpha_s(q,r,E\Delta F).\]
The second part of the claim follows applying the Remark \ref{finmeas}.
\end{proof}

We recall the definition (see (3.1) in  \cite{asympt1})
\[\mu(E):=\lim_{s \to 0^+} s P_s(E,\Omega),\]
where $\Omega$ is a bounded open set with $C^2$ boundary.
Moreover, we define 
\[ \overline \mu(E)= \limsup_{s\to 0^+} sP_s(E,\Omega)\]
and give the following upper bound:
\begin{prop}\label{barmubaral} Let $\Omega\subset\Rn$ be a bounded open set with finite classical perimeter
and let $E_0\subset\C \Omega$. Then
\[ \overline \mu(E_0)= \overline \alpha(E_0) |\Omega|.\]
\end{prop}
\begin{proof}
Let $R>0$ be fixed such that $\Omega \subset B_R$ and $\eps\in(0,1)$ be small enough such that $R/\eps> R+1$. This choice of $\eps$ assures that $B_1(y)\subset B_{R/\eps}$. 
For any fixed $y\in \Omega$,
we have that
\[ \int_{\Rn} \frac{\chi_{E_0}(x)}{|x-y|^{n+s}}\, dx =  \int_{\C B_{R/\eps} } \frac{\chi_{E_0}(x)}{|x-y|^{n+s}}\, dx  + \int_{ B_{R/\eps}\setminus B_{1}(y)} \frac{\chi_{E_0}(x)}{|x-y|^{n+s}}\, dx + \int_{  B_{1}(y)} \frac{\chi_{E_0}(x)}{|x-y|^{n+s}}\, dx  .\]
Since $|x-y|\geq (1-\eps)|x|$ whenever $x\in \C B_{R/\eps}$, we get
\[ \int_{\C B_{R/\eps} } \frac{\chi_{E_0}(x)}{|x-y|^{n+s}}\, dx\leq (1-\eps)^{-n-s}\int_{\C B_{R/\eps} } \frac{\chi_{E_0}(x)}{|x|^{n+s}}\, dx.\]
Also we have that
\[ \int_{ B_{R/\eps}\setminus B_{1}(y)} \frac{\chi_{E_0}(x)}{|x-y|^{n+s}}\, dx\leq \omega_n\int_1^{\R/\eps+R} \rho^{-s-1}\, d\rho \leq \omega_n\frac{1- \left(\frac{R}{\eps}+R\right)^{-s}}s.\]
Also, we can assume that $s<1/2$ 
(since we are interested in what happens for $s\to 0$). In this way,
if $|x-y|<1$ we have that $|x-y|^{-n-s}\leq|x-y|^{-n-\frac{1}{2}}$, and so
\[\int_{  B_1(y)} \frac{\chi_{E_0}(x)}{|x-y|^{n+s}}\, dx\leq \int_{  B_1(y)} \frac{\chi_{E_0}(x)}{|x-y|^{n+\frac{1}2}}\, dx.\] 
Also, since $E_0\subset\C\Omega$, we have that 
\[\int_{  B_1(y)} \frac{\chi_{E_0}(x)}{|x-y|^{n+\frac{1}2}}\, dx  \leq \int_{  B_1(y)\setminus \Omega} \frac{dx}{|x-y|^{n+\frac{1}2}}
\leq \int_{\C \Omega} \frac{dx}{|x-y|^{n+\frac{1}2}}.  \] 
This means that
\[\int_{\Omega} \int_{  B_1(y)} \frac{\chi_{E_0}(x)}{|x-y|^{n+s}}\, dx\,dy\leq \int_{\Omega} \int_{\C \Omega} \frac{dx}{|x-y|^{n+\frac{1}2}}=P_{\frac{1}2}(\Omega)=c<\infty,\]
since $\Omega$ has a finite classical perimeter.
In this way, it follows that
\eqlab{\label{mumu1} sP_s(E_0,\Omega) = \int_{\Omega} \int_{\Rn}\frac{\chi_{E_0}(x) }{|x-y|^{n+s}} \, dx \, dy\leq  &\;s (1-\eps)^{-n-s} |\Omega|\int_{\C B_{R/\eps} } \frac{\chi_{E_0}(x)}{|x|^{n+s}}\, dx\\
&\; +\omega_n\Big(1- \Big(\frac{R}{\eps}+R\Big)^{-s}\Big)|\Omega| +sc. }
Furthermore, notice that if $x\in B_{R/\eps}$ we have that $|x-y|\leq (1+\eps)|x|$, hence
\[ \int_{\Rn}\frac{\chi_{E_0}(x)}{|x-y|^{n+s}}\, dx\geq \int_{\C B_{{R}/{\eps}}} \frac{\chi_{E_0}(x)}{|x-y|^{n+s}} \, dx \geq (1+\eps)^{-n-s} \int_{\C B_{{R}/{\eps}}} \frac{\chi_{E_0}(x)}{|x|^{n+s}} \, dx.  \]
Thus for any $\eps>0$
\[ sP_s(E_0,\Omega) \geq s |\Omega|  (1+\eps)^{-n-s}  \int_{\C B_{{R}/{\eps}}} \frac{\chi_{E_0}(x)}{|x|^{n+s}} \, dx .\] 
Passing to  limsup as $s\to 0^+$ here above and in \eqref{mumu1} it follows that
\[ (1+\eps)^{-n} \overline \alpha(E_0)\,  |\Omega| \leq \overline \mu(E_0) \leq (1-\eps)^{-n} \overline \alpha(E_0)\, |\Omega| .\]
Sending $\eps \to 0$, we obtain the desired conclusion.
\end{proof}

\subsection{Classification of nonlocal minimal surfaces for small $s$}\label{classify}

\noindent \textbf{Asymptotic estimates of the density (Theorem \ref{notfull}).}\label{sectnotfull} Now we prove Theorem \ref{notfull}. 
\begin{proof}
We argue by contradiction. Suppose that there exists $\delta>0$ and $\gamma\in(0,1)$ for which we can find a sequence $s_k\searrow0$,
a sequence of sets $\{E_k\}$ such that each $E_k$ is $s_k$-minimal in $\Omega$ with exterior data $E_0$, and a sequence
of points $\{x_k\}\subset\overline{\Omega}$ such that
\eqlab{\label{density_contrad_proof}\big|(\Omega\cap B_\delta(x_k))\setminus E_k\big|< \gamma\,\frac{\omega_n-2\overline{\alpha}(E_0)}{\omega_n-\overline{\alpha}(E_0)}\big|\Omega\cap B_\delta(x_k)\big|.}

First of all we remark that, since $\overline{\Omega}$ is compact, up to passing to subsequences we can suppose that $x_k\longrightarrow x_0$, for some $x_0\in\overline{\Omega}$. As a consequence, for every $\eps>0$
there exists $\tilde k_\eps$ such that
\eqlab{\label{subset_balls_eps}%\Omega\cap B_{\delta-\eps}(x_0)\subset
\Omega\cap B_\delta (x_k)\subset\Omega\cap B_{\delta+\eps}(x_0),\qquad
\forall\,k\ge\tilde k_\eps.}

We fix a small $\eps>0$. We will let $\eps\to0$ later on.

Since $E_k$ is $s_k$-minimal in $\Omega$, it is $s_k$-minimal also in every $\Omega'\subset\Omega$.
Thus, by \eqref{subset_balls_eps} and by minimality, for every $k\ge\tilde k_\eps$ we have
\bgs{P_{s_k}(E_k,\Omega\cap B_\delta(x_k))&\le P_{s_k}(E_k,\Omega\cap B_{\delta+\eps}(x_0))\\
&
\le P_{s_k}(E_0\cup(E_k\cap(\Omega\setminus B_{\delta+\eps}(x_0))),\Omega\cap B_{\delta+\eps}(x_0))\\
&
\le\int_{E_0}\int_{\Omega\cap B_{\delta+\eps}(x_0)}\frac{dy\,dz}{|y-z|^{n+s_k}}
+\int_{\Omega\setminus B_{\delta+\eps}(x_0)}\int_{\Omega\cap B_{\delta+\eps}(x_0)}\frac{dy\,dz}{|y-z|^{n+s_k}}\\
&
=:I^1_k+I^2_k.}
Now notice that the set $\Omega\cap B_{\delta+\eps}(x_0)$ has finite classical perimeter. %and
%\[E_0\subset\Co(\Omega\cap B_{\delta+\eps}(x_0)),
%\qquad\Omega\setminus B_{\delta+\eps}(x_0)\subset\Co(\Omega\cap B_{\delta+\eps}(x_0)).\]
Thus, by Proposition \ref{barmubaral} we find
\[\limsup_{k\to\infty}s_k I^1_k\le\overline{\alpha}(E_0)\big|\Omega\cap B_{\delta+\eps}(x_0)\big|,\]
and, since $\Omega$ is bounded
\[\limsup_{k\to\infty}s_k I^2_k\le\overline{\alpha}(\Omega\setminus B_{\delta+\eps}(x_0))\big|\Omega\cap B_{\delta+\eps}(x_0)\big|=0,\] 
for every $\eps>0$.
Therefore, letting $\eps\to0$,
\eqlab{\label{contrad_dens_limsup}\limsup_{k\to\infty}s_k P_{s_k}(E_k,\Omega\cap B_\delta(x_k))\le\overline{\alpha}(E_0)\big|\Omega\cap B_\delta(x_0)\big|.}

On the other hand, if $R>0$ is such that $\Omega\subset\subset B_R(q)$ for every $q\in\overline{\Omega}$, then
%arguing as in Proposition \ref{propnotfull} 
we have that
\bgs{P_{s_k}(E_k,\Omega\cap B_\delta(x_k))&\ge
\int_{E_k\cap(\Omega\cap B_\delta(x_k))}
\Big(\int_{\C E_k\setminus(\Omega\cap B_\delta(x_k))}\frac{dz}{|y-z|^{n+s_k}}\Big)dy\\
&
\ge\int_{E_k\cap(\Omega\cap B_\delta(x_k))}\Big(\int_{\C \Omega}\frac{\chi_{\C E_0}(z)}{|y-z|^{n+s_k}}\,dz\Big)dy\\
&\ge \int_{E_k\cap(\Omega\cap B_\delta(x_k))}\Big(\inf_{q\in \overline \Omega } \int_{\C \Omega}\frac{\chi_{\C E_0}(z)}{|q-z|^{n+s_k}}\,dz\Big)dy\\
 &
\ge\big|E_k\cap(\Omega\cap B_\delta(x_k))\big|\inf_{q\in\overline{\Omega}}\int_{\C B_R(q)}
\frac{\chi_{\C E_0}(z)}{|q-z|^{n+s_k}}\,dz.}
So, thanks to Proposition \ref{unifrq}
\bgs{\liminf_{k\to\infty}s_k P_{s_k}&(E_k,\Omega\cap B_\delta(x_k))\\
&
\ge\Big(\liminf_{k\to\infty}|E_k\cap(\Omega\cap B_\delta(x_k))\big|\Big)\Big(\liminf_{k\to\infty}s_k\,
\inf_{q\in\overline{\Omega}}\int_{\C B_R(q)}\frac{\chi_{\C E_0}(z)}{|q-z|^{n+s_k}}\,dz\Big)\\
&
=\big(\omega_n-\overline{\alpha}(E_0)\big)\Big(\liminf_{k\to\infty}|E_k\cap(\Omega\cap B_\delta(x_k))\big|\Big).}

By \eqref{density_contrad_proof} we have 
\bgs{|E_k\cap(\Omega\cap B_\delta(x_k))\big|=|\Omega\cap B_\delta(x_k)|-\big|(\Omega\cap B_\delta(x_k))\setminus E_k\big|
>\frac{(1-\gamma)\omega_n-(1-2\gamma)\overline{\alpha}(E_0)}{\omega_n-\overline{\alpha}(E_0)}\,|\Omega\cap B_\delta(x_k)|,}
and hence, since $x_k\longrightarrow x_0$,
\[\liminf_{k\to\infty}|E_k\cap(\Omega\cap B_\delta(x_k))\big|\geq
\frac{(1-\gamma)\omega_n-(1-2\gamma)\overline{\alpha}(E_0)}{\omega_n-\overline{\alpha}(E_0)}\,|\Omega\cap B_\delta(x_0)|.\]
Thus, recalling \eqref{contrad_dens_limsup} we obtain
\eqlab{\label{contrad_dens_liminf}
\overline{\alpha}(E_0)\,|\Omega\cap B_\delta(x_0)|\ge\liminf_{k\to\infty}s_k P_{s_k}&(E_k,\Omega\cap B_\delta(x_k))\ge
\big((1-\gamma)\omega_n-(1-2\gamma)\overline{\alpha}(E_0)\big)|\Omega\cap B_\delta(x_0)|.}

We remark that, since $x_0\in\overline{\Omega}$, we have
\[|\Omega\cap B_\delta(x_0)|>0,\]
hence we get
\[\overline{\alpha}(E_0)\ge(1-\gamma)\omega_n-(1-2\gamma)\overline{\alpha}(E_0)
\quad\textrm{ that is }\quad
(1-\gamma)\overline{\alpha}(E_0)\ge(1-\gamma)\frac{\omega_n}{2}.\]

Therefore, since $\gamma\in(0,1)$ and by hypothesis $\overline{\alpha}(E_0)<\frac{\omega_n}{2}$, we obtain a contradiction, concluding the proof.
\end{proof}

\begin{corollary}
Let $\Omega\subset\R^n$ be a bounded open set of finite classical perimeter and let $E_0\subset\C\Omega$ be such that
$\alpha(E_0)=0$. Let $s_k\in(0,1)$ be such that $s_k\searrow0$ and let $\{E_k\}$
be a sequence of sets such that each $E_k$ is $s_k$-minimal in $\Omega$ with exterior data $E_0$. Then
\[\lim_{k\to\infty}|E_k\cap\Omega|=0.\]
\end{corollary}
\begin{proof}
Fix $\delta>0$. Since $\overline{\Omega}$ is compact, we can find a finite number of points $x_1,\dots,x_m\in\overline{\Omega}$
such that
\[\overline{\Omega}\subset\bigcup_{i=1}^mB_\delta(x_i).\]
By Theorem \ref{notfull} we know that for every $\gamma\in(0,1)$ we can find a $k(\gamma)$ big enough such that
\[\big|E_k\cap(\Omega\cap B_\delta(x_i))\big|
\le\frac{(1-\gamma)\omega_n-(1-2\gamma)\overline{\alpha}(E_0)}{\omega_n-\overline{\alpha}(E_0)}\,|\Omega\cap B_\delta(x_i)|
=(1-\gamma)|\Omega\cap B_\delta(x_i)|,\]
for every $i=1,\dots,m$ and every $k\ge k(\gamma)$. Thus
\[|E_k\cap\Omega|\le(1-\gamma)\sum_{i=1}^m|\Omega\cap B_\delta(x_i)|,\]
for every $k\ge k(\gamma)$, and hence
\[\limsup_{k\to\infty}|E_k\cap\Omega|\le(1-\gamma)\sum_{i=1}^m|\Omega\cap B_\delta(x_i)|,\]
for every $\gamma\in(0,1)$. Letting $\gamma\longrightarrow1^-$ concludes the proof.
\end{proof}

We recall here that any set $E_0$ of finite measure has $\alpha(E_0)=0$ (check Remark \ref{finmeas}).

\noindent \textbf{Estimating the fractional mean curvature (Theorem \ref{positivecurvature}).}\label{estimatecurvature}
Thanks to the previous preliminary work, we are now in the position
of completing the proof of Theorem \ref{positivecurvature}.

\begin{proof}[Proof of Theorem \ref{positivecurvature}]
Let $R:=2\,\max\{1,\textrm{diam}(\Omega)\}$. First of all, \eqref{claimalpha} implies that
\bgs{ \liminf_{s\to 0^+} \bigg(\omega_n R^{-s} -2s\sup_{q\in \overline \Omega} \int_{\C B_R(q)}\frac{\chi_E(y)}{|q-y|^{n+s}} \, dy \bigg) = {\omega_n-2\overline \alpha(E_0)}=4 \beta.} Notice that by \eqref{weak_hp_beta}, $\beta>0$. Hence for every $s$ small enough, say $s<s'\leq\frac{1}{2}$ with $s'=s'(E_0,\Omega)$, we have that
\eqlab{\label{tildebeta}\omega_n R^{-s} - 2s\sup_{q\in \overline \Omega} \int_{\C B_R(q)}\frac{\chi_E(y)}{|q-y|^{n+s}} \, dy \geq \frac {7}2 \beta.}

Now, let $E\subset\Rn$ be such that $E\setminus\Omega=E_0$, suppose that $E$ has an exterior tangent ball
of radius $\delta<R/2$ at $q\in\partial E\cap\overline{\Omega}$, that is
\[B_\delta(p)\subset\C E\quad\textrm{and}\quad q\in\partial B_\delta(p),\]
and let $s<s'$.
Then for $\rho$ small enough (say $\rho<\delta/2$) we conclude that
\bgs{\label{uno} \mathcal{I}^\rho_s[E](q)= \int_{B_R(q)\setminus B_{\rho}(q)} \frac{\chi_{ \C E} (y)-\chi_{E}(y)}{|q-y|^{n+s}}\, dy + \int_{\C B_R(q)} \frac{\chi_{ \C E} (y)-\chi_{E}(y)}{|q-y|^{n+s}}\, dy.}

 Let $D_\delta= B_\delta(p) \cap B_\delta(p')$, where $p'$ is the symmetric of $p$ with respect to $q$, i.e. the ball $B_\delta(p')$ is the ball tangent to  $B_\delta(p)$ in $q$. Let also $K_\delta$ be the convex hull of $D_\delta$ and let $P_\delta:=K_\delta-D_\delta$. Notice that $B_\rho(q)\subset K_\delta \subset B_R(q)$ . Then
\bgs{ \int_{B_R(q)\setminus B_{\rho}(q)} \frac{\chi_{ \C E} (y)-\chi_{E}(y)}{|q-y|^{n+s}}\, dy=& \int_{D_\delta\setminus B_{\rho}(q)} \frac{\chi_{ \C E} (y)-\chi_{E}(y)}{|q-y|^{n+s}}\, dy +\int_{P_\delta\setminus B_{\rho}(q) } \frac{\chi_{ \C E} (y)-\chi_{E}(y)}{|q-y|^{n+s}}\, dy\\ &+\int_{B_R(q) \setminus K_\delta} \frac{\chi_{ \C E} (y)-\chi_{E}(y)}{|q-y|^{n+s}}\, dy.}
Since $B_\delta(p) \subset \C E$, by symmetry we obtain that
\bgs{  \int_{D_\delta \setminus B_{\rho}(q)} \frac{\chi_{ \C E} (y)-\chi_{E}(y)}{|q-y|^{n+s}}\, dy =\int_{B_\delta(p)\setminus B_{\rho}(q)} \frac{dy}{|q-y|^{n+s}} +   \int_{B_\delta(p')\setminus B_{\rho}(q) }\frac{\chi_{ \C E} (y)-\chi_{E}(y)}{|q-y|^{n+s}}\, dy  \geq 0.   }
Moreover, from Lemma 3.1 in \cite{graph} (here applied with $\lambda =1$) we have that
\[\bigg| \int_{P_\delta \setminus B_{\rho}(q)} \frac{\chi_{ \C E} (y)-\chi_{E}(y)}{|q-y|^{n+s}}\, dy \bigg|\leq  \int_{P_\delta} \frac{dy}{|q-y|^{n+s}}\leq  \frac{C_0}{1-s} {\delta^{-s} }
 ,\] 
%given that $s\leq 1/2$, 
with $C_0=C_0(n)>0$.
Notice that $B_\delta(q)\subset K_\delta$  so
\bgs{\bigg | \int_{B_R(q)\setminus K_\delta } \frac{\chi_{ \C E} (y)-\chi_{E}(y)}{|q-y|^{n+s}}\, dy\bigg | \leq &\;\int_{B_R(q)\setminus B_\delta(q)} \frac{dy}{|q-y|^{n+s}} = \omega_n  \frac{\delta^{-s}- R^{-s}}{s} .} 
%C \delta^{-s} \frac{1- (\delta/r)^{s}}{s}. }  
%For $x\in(0,1)$ we have that $(1-x^s)/s\leq -\log x$, and since $\delta/r\in(0,1)$,  we obtain that 
%%%%$[1- (\delta/r)^{s}]/s  \leq \log (r/\delta) $. It follows that
% \bgs{\bigg | \int_{B_r(q)\setminus K_\delta } \frac{\chi_{ \C E} (y)-\chi_{E}(y)}{|q-y|^{n+s}}\, dy\bigg | 
Therefore for every $\rho <\delta/2$
%for 
%$C_1$ and $C_2$ 
%$C_0$ positive, independent on $s$, 
one has that
\bgs{\label{due}\int_{B_R(q)\setminus B_\rho(q)} \frac{\chi_{ \C E} (y)-\chi_{E}(y)}{|q-y|^{n+s}}\, dy \geq -\frac{C_0}{1-s}\delta^{-s} - \frac{\omega_n}{s} \delta^{-s} + \frac{\omega_n}{s} R^{-s}  .}
%\bigg( \frac{C_1}{1-s} +C_2\log\frac{r}{\delta}\bigg).} 
Thus, using \eqref{tildebeta}
%On the other hand, we have that
%\eqlab{\label{comp} \int_{\C B_r(q)} \frac{\chi_{\C E}(y)-\chi_E(y)}{|q-y|^{n+s}}\, dy =&\;  \int_{\C B_r(q)} \frac{dy}{|q-y|^{n+s}} -2 \int_{E\setminus B_r(q)}\frac{dy}{|q-y|^{n+s}}\\
%\geq &\; \omega_n \frac{r^{-s}}{s} - 2 \sup_{q\in \overline\Omega} \int_{E\setminus B_r(q)}\frac{dy}{|q-y|^{n+s}} }
%and so
\eqlab{\label{uno} \mathcal{I}^\rho_s[E](q)=&\;\int_{B_R(q)\setminus B_\rho(q)} \frac{\chi_{ \C E} (y)-\chi_{E}(y)}{|q-y|^{n+s}}\, dy + \int_{\C B_R(q)} \frac{\chi_{ \C E} (y)-\chi_{E}(y)}{|q-y|^{n+s}}\, dy \\
\geq &\; -\frac{C_0}{1-s}\delta^{-s} - \frac{\omega_n}{s} \delta^{-s} + \frac{\omega_n}{s} R^{-s} + \int_{\C B_R(q)} \frac{dy}{|q-y|^{n+s}} -  2\int_{ \C B_R(q)} \frac{\chi_E(y)}{|q-y|^{n+s}}\,dy \\ 
\geq&\; - \delta^{-s} \Big(\frac{C_0}{1-s} +\frac{\omega_n} s\Big) +  \frac{\omega_n}s R^{-s} + \bigg(\frac{\omega_n}s R^{-s}- 2\sup_{q\in \overline \Omega} 
 \int_{ \C B_R(q)} \frac{\chi_E(y)}{|q-y|^{n+s}}\,dy\bigg)\\
 \geq&\; - \delta^{-s} \Big(\frac{C_0}{1-s} +\frac{\omega_n} s\Big) +  \frac{\omega_n}s R^{-s}  +\frac{7\beta}{2s}\\
 \geq &\; - \delta^{-s} \Big(2C_0 +\frac{\omega_n} s\Big) +  \frac{\omega_n}s R^{-s}  +\frac{7\beta}{2s},}
where we also exploited that $s<s'\leq 1/2$.
%since $s<s'\leq\frac{1}{2}$, we have
%\bgs{\I_s^\rho[E](q)&\geq-\delta^{-s}\Big(\frac{C_0}{1-s}+\frac{\omega_n}{s}\Big)+\frac{\omega_n}{s}R^{-s}
%+\beta_s(E_0,\Omega;R)\\
%&
%\geq-\delta^{-s}\Big(2C_0+\frac{\omega_n}{s}\Big)+\frac{\omega_n}{s}R^{-s}
%+\frac{7\beta}{2s}\\
%&
%=\frac{1}{s}\Big\{-\delta^{-s}\big((2C_0)s+\omega_n\big)+\omega_n R^{-s}+\frac{7}{2}\beta\Big\},}
%for every $\rho\in(0,\delta/2)$. 
Since $R>1$, we have
\[R^{-s}\to1^-,\qquad\textrm{as }s\to0^+.\]
Therefore we can find $s''=s''(E_0,\Omega)$ small enough such that
\[\omega_nR^{-s}\geq\omega_n-\frac{\beta}{2},\qquad\forall s<s''.\]

Now let
\[s_0=s_0(E_0,\Omega):=\min\Big\{s',s'',\frac{\beta}{2C_0}\Big\}.\]
Then, for every $s<s_0$ we have
\eqlab{\label{important_estimate_curv}\I_s^\rho[E](q)&\geq\frac{1}{s}\Big\{-\delta^{-s}\big((2C_0)s+\omega_n\big)+\omega_n R^{-s}+\frac{7}{2}\beta\Big\}\\
&
\geq\frac{1}{s}\big\{-\delta^{-s}(\omega_n+\beta)+\omega_n+3\beta\big\},}
for every $\rho\in(0,\delta/2)$.

Notice that if we fix $s\in(0,s_0)$, then for every
\[\delta\geq e^{-\frac{1}{s}\log\frac{\omega_n+2\beta}{\omega_n+\beta}}=:\delta_s(E_0),\]
we have that
\[-\delta^{-s}(\omega_n+\beta)+\omega_n+3\beta\geq\beta>0.\]
To conclude, we
let $\sigma\in(0,s_0)$ and suppose that $E$ has an exterior tangent ball of radius $\delta_\sigma$
at $q\in\partial E\cap\overline{\Omega}$.
Notice that, since $\delta_\sigma<1$, we have
\[-(\delta_\sigma)^{-s}(\omega_n+\beta)+\omega_n+3\beta\geq
-(\delta_\sigma)^{-\sigma}(\omega_n+\beta)+\omega_n+3\beta=\beta,\qquad\forall\,s\in(0,\sigma].\]
Then \eqref{important_estimate_curv} gives that
\[\liminf_{\rho\to0^+}\I_s^\rho[E](q)\geq\frac{\beta}{s}>0,\qquad\forall\,s\in(0,\sigma],\]
which concludes the proof.
\end{proof}

\begin{remark}
We remark that
\[\log\frac{\omega_n+2\beta}{\omega_n+\beta}>0,\]
thus
\[\delta_s\to0^+\qquad\textrm{as }s\to 0^+.\]
\end{remark}

%\textcolor{blue}{Farei una sezione su $\alpha$ e $\beta$ per far vedere che se esistono i limiti, allora sono uguali, non dipendono da $r$ e metterci $(i)$ e $(ii)$ del Remark. Pi\'u qualche esempio: gli insiemi limitati, i coni, il sopragrafico di una funzione limitata che cresce linearmente.}

As a consequence of Theorem \ref{positivecurvature}, we have that, as $s\to0^+$,
the $s$-minimal sets with small mass at infinity have small mass in $\Omega$.
The precise result goes as follows:

\begin{corollary}
Let $\Omega\subset\Rn$ be a bounded open set,
%and let $R:=2\,\max\{1,\textrm{diam}(\Omega)\}$.
let $E\subset\Rn$ be such that
\[\overline \alpha(E) <\frac{\omega_n}2,\]
and suppose that $\partial E$ is of class $C^2$ in $\Omega$.
Then, for every $\Omega'\subset\subset\Omega$ there exists
$\tilde s=\tilde s(E\cap\overline{\Omega'})\in(0,s_0)$ such that for every $s\in(0,\tilde s]$
\eqlab{\I_s[E](q)\geq\frac{\omega_n -2\overline \alpha(E)}{4s}>0,\qquad\forall\,q\in\partial E\cap\overline{\Omega'}.}
\end{corollary}

\begin{proof}
Since $\partial E$ is of class $C^2$ in $\Omega$ and $\Omega'\subset\subset\Omega$, the set $E$ satisfies a uniform exterior
ball condition of radius $\tilde\delta=\tilde\delta(E\cap\overline{\Omega'})$ in $\overline{\Omega'}$, meaning that
$E$ has an exterior tangent ball of radius at least $\tilde\delta$ at every point $q\in\partial E\cap\overline{\Omega'}$.

Now, since $\delta_s\to0^+$ as $s\to0^+$, we can find $\tilde s=\tilde s(E\cap\overline{\Omega'})
<s_0(E\setminus\Omega,\Omega)$,
small enough such that $\delta_s<\tilde\delta$ for every $s\in(0,\tilde s]$. Then we can conclude by applying Theorem \ref{positivecurvature}.
\end{proof}

%%%%%%%%%%%%%%%%%%%%%%
\noindent \textbf{Classification of $s$-minimal surfaces (Theorem \ref{THM}).}\label{alternative}
To classify the behavior of the $s$-minimal surfaces when $s$ is small,
we need to take into account the ``worst case scenario'',
that is the one in which the set behaves very badly in terms
of oscillations and lack of regularity. To this aim,
we make an observation about $\delta$-dense sets.

   \begin{center}
\begin{figure}[htpb]
	\hspace{0.79cm}
	\begin{minipage}[b]{0.79\linewidth}
	\centering
	\includegraphics[width=0.79\textwidth]{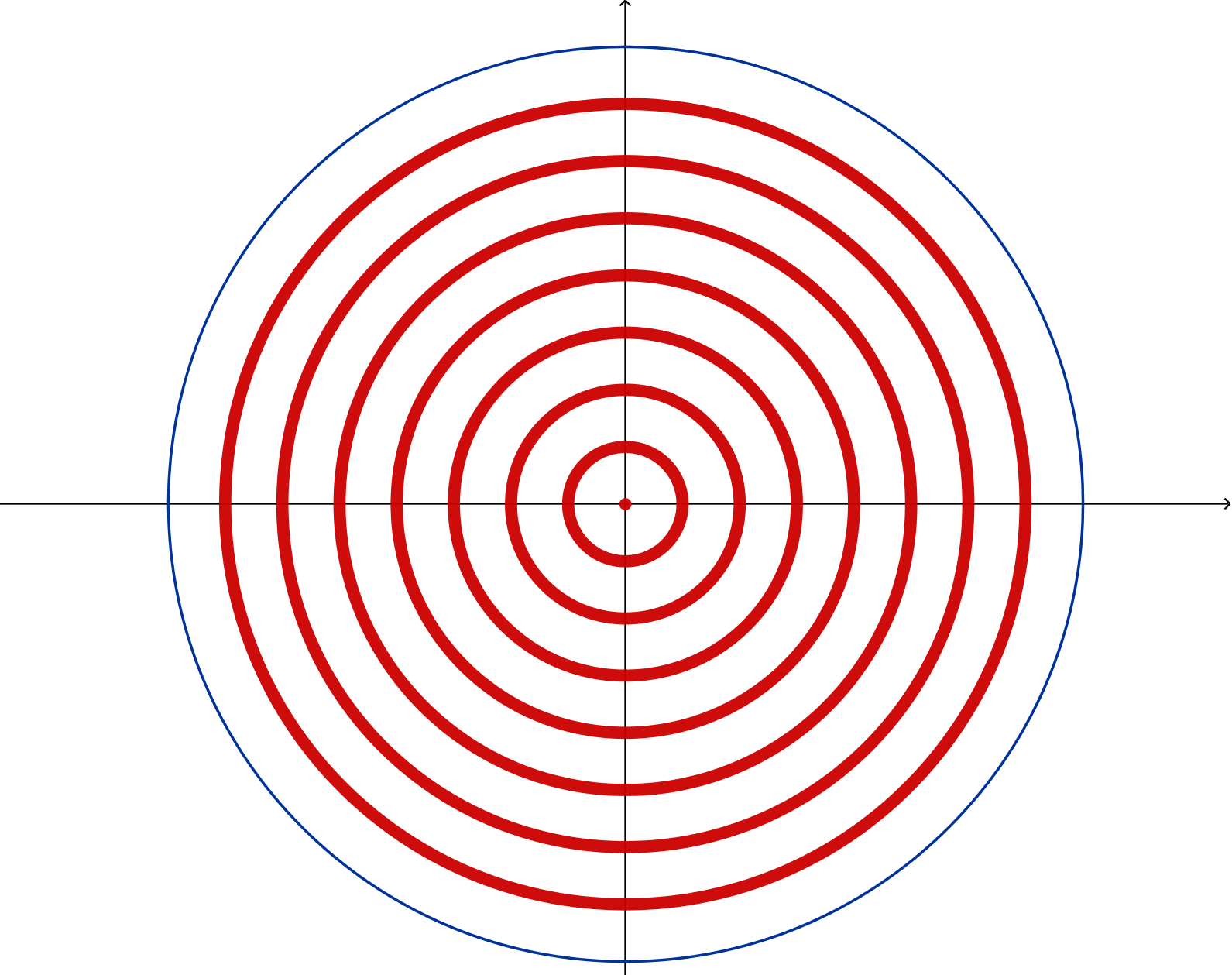}
	\caption{A $\delta$-dense set of measure $<\eps$}   
	\label{x3}
	\end{minipage}
\end{figure} 
	\end{center}
  \begin{remark}\label{deltadance}
  For every $k\ge1$ and every $\eps<2^{-k-1}$, we define the sets
\[\Gamma_k^\eps:=B_\eps\cup\bigcup_{i=1}^{2^k-1}\Big\{x\in\R^n\,\big|\,\frac{i}{2^k}-\eps<|x|<\frac{i}{2^k}+\eps\Big\}
\quad\textrm{ and }\quad\Gamma_k:=\{0\}\cup\bigcup_{i=1}^{2^k-1}\partial B_\frac{i}{2^k}.\]
%
%The measure of the set $\Gamma_k^\eps$ is
%\bgs{|\Gamma_k^\eps|
%=\frac{\omega_n}{n}\eps^n
%+\sum_{i=1}^{2^k-1}\omega_n\int_{\frac{i}{2^k}-\eps}^{\frac{i}{2^k}+\eps}r^{n-1}\,dr
%=\frac{\omega_n}{n}\Big(\eps^n+\sum_{i=1}^{2^k-1}\Big\{\Big(\frac{i}{2^k}+\eps\Big)^n-\Big(\frac{i}{2^k}-\eps\Big)^n\Big\}}
\noindent Notice that for every $\delta>0$ there exists $\tilde k=\tilde k(\delta)$ such that for every $k\geq\tilde k$ we have
\[B_\delta(x)\cap\Gamma_k \neq \emptyset,\qquad\forall\,B_\delta(x)\subset B_1.\]
Thus, for every $k\geq\tilde k(\delta)$ and $\eps<2^{-k-1}$, the set $\Gamma_k^\eps$ is $\delta$-dense in $B_1$.
\noindent Moreover, notice that
\[\Gamma_k=\bigcap_{\eps\in(0,2^{-k-1})}\Gamma_k^\eps\quad\textrm{ and }\quad\lim_{\eps\to0^+}|\Gamma_k^\eps|=0.\]
It is also worth remarking that the sets $\Gamma_k^\eps$ have smooth boundary.
\noindent In particular, for every $\delta>0$ and every $\eps>0$ small, we can find a set $E\subset B_1$ which is $\delta$-dense in $B_1$ and whose measure is $|E|<\eps$.
This means that we can find an open set $E$ with smooth boundary, whose measure is arbitrarily small
and which is ``topologically arbitrarily dense'' in $B_1$.
\end{remark}

We introduce a useful geometric observation.

\begin{prop}\label{tgball}
Let $\Omega\subset \Rn$ be a bounded and connected open set with $C^2$ boundary and let $\delta\in(0,r_0)$,
for $r_0$ given in \eqref{r01}. 
% be fixed
If $E$ is not $\delta$-{dense} in $\Omega$ 
% for $\delta<r_0$
and $|E\cap\Omega|>0$, 
%for $r_0$ given in \eqref{r01}, 
%then there exists an exterior tangent ball to $\partial E$, of radius (at least) $\delta$, centered at a point in $\Omega$, i.e.
then there exists a point $q\in\partial E\cap\Omega$ such that $E$ has an exterior tangent ball at $q$ of radius $\delta$ (contained in $\Omega$),
i.e.
there exist $p\in \C E\cap \Omega$ such that
\[ B_\delta(p)\subset\subset\Omega,\qquad q \in \partial B_{\delta}(p) \cap \partial E
\quad \mbox{ and }  \quad B_{ \delta}(p) \subset \C E.\] 
\end{prop}
\begin{proof}
Using Definition \ref{wild}, we have that there exists  $x\in \Omega$ for which $B_\delta(x)\subset\subset \Omega$ and $|B_\delta(x)\cap E|=0$, so $B_\delta(x)\subset E_{ext}$. If $B_\delta(x)$ is tangent to $\partial E$ then we are done.
%\noindent Else, let $d_0=\mbox{dist}(x,\partial \Omega)$.

Notice that
\[B_\delta(x)\subset\subset\Omega\quad\Longrightarrow\quad d(x,\partial\Omega)>\delta,\]
and let
\[\delta':=\min\{r_0,d(x,\partial\Omega)\}\in(\delta,r_0].\]
Now we consider the open set $\Omega_{-\delta'}\subset \Omega$
\[ \Omega_{-\delta'} :=\{ \bar d_{\Omega}<-\delta'\},\]  so $x\in \Omega_{-\delta'}$. According to Remark \ref{c21} and Lemma \ref{geomlem} we have that $\Omega_{-\delta'}$ has $C^2$ boundary and that 
\eqlab{\label{unifball} \Omega_{-\delta'} \mbox{ satisfies the uniform interior ball condition of radius at least } r_0.} 
 
 We have two possibilities:
 \eqlab{\label{inside} &\mbox{ i) } && \overline  E\cap \Omega_{-\delta'}\neq \emptyset \\ 
 				  &\mbox{ ii) }&&\emptyset \neq  \overline{E}\cap\Omega \subset \Omega \setminus \Omega_{-\delta'}.} 
 
 If i) happens, we pick any point $y\in \overline{E}\cap \Omega_{-\delta'}$.  
%to be such that 
%\[ \overline{ B_{\delta}(\bar p)}\subset \{x_n<0\}, \quad \mbox{ and so }\quad  \overline{ B_{\delta}(\overline p)]}\subset \C E_s,\]
% according to \eqref{allunder}.
%Let $\bar q\in \partial E_s \cap \Omega_{-r}$ be any point. 
The set $\overline{\Omega_{-\delta'}}$ is path connected (see Proposition \ref{retract}), so there exists a path $c:[0,1]\longrightarrow
\R^n$ that connects $x$ to $y$ and that stays inside $\overline{\Omega_{-\delta'}}$, that is
\[c(0)=x,\qquad c(1)=y\quad\textrm{ and }\quad c(t)\in\overline{\Omega_{-\delta'}},\quad\forall\,t\in[0,1].\]
Moreover, since $\delta<\delta'$, we have
\[B_\delta\big(c(t)\big)\subset\subset\Omega\qquad\forall\,t\in[0,1].\]
Hence,
we can ``slide the ball'' $B_{\delta}(x)$ along the path and we obtain the desired
claim thanks to Lemma \ref{slidetheballs}.

Now, if we are in the case ii) of \eqref{inside}, then $\Omega_{-\delta'}\subset E_{ext}$, so we dilate $\Omega_{-\delta'}$ until we first touch %the boundary of 
$\overline E$. That is, we consider 
\[\tilde \rho:=\inf\{\rho\in[0,\delta'] \; \big| \; \Omega_{-\rho}\subset E_{ext}\}.\]
Notice that by hypothesis $\tilde \rho>0$. Then 
\[\overline{\Omega_{-\tilde \rho}} \subset \overline {E_{ext}}=E_{ext}\cup\partial E. \]
If
\[ \partial \Omega_{-\tilde\rho} \cap \partial E=\emptyset\quad  \mbox { then }\quad  \overline{ \Omega_{-\tilde \rho}} \subset  E_{ext},\] 
hence we have that
\[ d= d\left(  \overline E\cap \Omega\setminus \Omega_{-\delta'}, \overline {\Omega_{-\tilde \rho}}\right)\in(0,\tilde \rho),\]
therefore
\[ \Omega_{-\tilde \rho} \subset \Omega_{-(\tilde \rho-d)}\subset E_{ext}.\]
This is in contradiction with the definition of $\tilde \rho$. Hence, 
%$\overline \Omega_{-\tilde \rho} \cap \overline E\neq \emptyset$ and it follows 
 there exists $q \in \partial \Omega_{-\tilde \rho} \cap \partial E$. 

Recall that, by definition of $\tilde\rho$, we have
$\Omega_{-\tilde \rho} \subset \C E.$ Thanks to \eqref{unifball}, there exists a tangent ball at $q$ interior to $\Omega_{-\tilde \rho}$, hence a tangent ball at $q$ exterior to $E$, of radius at least $r_0>\delta$. 
This concludes the proof of the lemma.
\end{proof}

\begin{proof}[Proof of Theorem \ref{THM}]
We begin by proving part $(A)$.\\
First of all, since $\delta_s\to0^+$, we can find $s_1=s_1(E_0,\Omega)\in(0,s_0]$ such that $\delta_s<r_0$
for every $s\in(0,s_1)$.

Now let $s\in(0,s_1)$ and let $E$ be $s$-minimal in $\Omega$, with exterior data $E_0$.

We suppose that  $E\cap \Omega\neq \emptyset$ and prove that $E$ has to be $\delta_s$-dense. 

Suppose by contradiction that $E$ is not $\delta_s$-dense. Then, in view of 
Proposition \ref{tgball}, there exists $p\in \C E\cap \Omega$ such that
\[ q \in \partial B_{\delta_s}(p) \cap (\partial E\cap\Omega) \quad \mbox{ and }  \quad B_{ \delta_s}(p) \subset \C E.\] 
Hence we use the Euler-Lagrange theorem at $q$, i.e.
\[I_s[E](q) \leq 0,\] 
to obtain a contradiction with Theorem \ref{positivecurvature}. This says that $E$ is not $\delta_s$-dense and concludes the proof of part $(A)$ of Theorem \ref{THM}.

Now we prove the part $(B)$ of the Theorem.\\
Suppose that point $(B.1)$ does not hold true. Then we can find a sequence $s_k\searrow0$ and a sequence of sets $E_k$
such that each $E_k$ is $s_k$-minimal in $\Omega$ with exterior data $E_0$ and
\[E_k\cap\Omega\not=\emptyset.\]
We can assume that $s_k<s_1(E_0,\Omega)$ for every $k$. Then part $(A)$ implies that each $E_k$ is $\delta_{s_k}$-dense,
that is
\[|E_k\cap B_{\delta_{s_k}}(x)|>0\quad\forall\,B_{\delta_{s_k}}(x)\subset\subset\Omega.\]
Fix $\gamma=\frac{1}{2}$, take a sequence $\delta_h\searrow0$ and let $\sigma_{\delta_h,\frac{1}{2}}$ be as in Theorem \ref{notfull}. Recall that $\delta_s\searrow0$ as $s\searrow0$.
Thus for every $h$ we can find $k_h$ big enough such that
\eqlab{\label{koala}s_{k_h}<\sigma_{\delta_h,\frac{1}{2}}\qquad\textrm{and}\qquad\delta_{s_{k_h}}<\delta_h.}
In particular, this implies
\eqlab{\label{koala1}|E_{k_h}\cap B_{\delta_h}(x)|\ge|E_k\cap B_{\delta_{s_{k_h}}}(x)|
>0\quad\forall\,B_{\delta_h}(x)\subset\subset\Omega,}
for every $h$. On the other hand, by \eqref{koala} and Theorem \ref{notfull}, we also have that
\eqlab{\label{koala2}|\C E_{k_h}\cap B_{\delta_h}(x)|>0\quad\forall\,B_{\delta_h}(x)\subset\subset\Omega.}
This concludes the proof of part $(B)$. Indeed, notice that since $B_{\delta_h}(x)$ is connected, \eqref{koala1} and \eqref{koala2}
together imply that
\[\partial E_{k_h}\cap B_{\delta_h}(x)\neq\emptyset\quad\forall\,B_{\delta_h}(x)\subset\subset\Omega.\]
\end{proof}

\noindent \textbf{Stickiness to the boundary is a typical behavior (Theorem \ref{boundedset}).}\label{sticky}
Now we show that the ``typical behavior'' of the nonlocal minimal surfaces
is to stick at the boundary whenever they are allowed to do it,
in the precise sense given by Theorem \ref{boundedset}.

\begin{proof}[Proof of Theorem \ref{boundedset}]
Let
\[\delta:=\frac{1}{2}\min\{r_0,R\},\]
and notice that (see Remark \ref{ext_unif_omega})
\[B_\delta(x_0+\delta\nu_\Omega(x_0))\subset B_R(x_0)\setminus\Omega\subset\C E_0.\]
Since $\delta_s\to0^+$, we can find $s_3=s_3(E_0,\Omega)\in(0,s_0]$ such that $\delta_s<\delta$ for every $s\in(0,s_3)$.

Now let $s\in(0,s_3)$ and let $E$ be $s$-minimal in $\Omega$, with exterior data $E_0$.

We claim that
\eqlab{\label{pf_bdedset_eq1}B_\delta(x_0-r_0\nu_\Omega(x_0))\subset E_{ext}.}
We observe that this is indeed a crucial step to
prove Theorem \ref{boundedset}. Indeed, once this is established,
by Remark \ref{ext_unif_omega} we obtain that
\[B_\delta(x_0-r_0\nu_\Omega(x_0))\subset\subset\Omega.\]
Hence, since $\delta_s<\delta$, we deduce from \eqref{pf_bdedset_eq1}
that $E$ is not $\delta_s$-dense.
Thus, since $s<s_3\leq s_1$, Theorem \ref{THM} 
implies that $E\cap\Omega=\emptyset$, which concludes the proof of Theorem \ref{boundedset}.

This, we are left to prove \eqref{pf_bdedset_eq1}. Suppose by contradiction that
\[\overline{E}\cap B_\delta(x_0-r_0\nu_\Omega(x_0))\not=\emptyset,\]
and consider the segment $c:[0,1]\longrightarrow\Rn$,
\[c(t):=x_0+\big((1-t)\delta-t\,r_0\big)\nu_\Omega(x_0).\]
Notice that
\[B_\delta\big(c(0)\big)\subset E_{ext}\quad\mbox{ and }\quad B_\delta\big(c(1)\big)\cap\overline{E}\not=\emptyset,\]
so
\[t_0:=\sup\Big\{\tau\in[0,1]\,\big|\,\bigcup_{t\in[0,\tau]}B_\delta\big(c(t)\big)\subset E_{ext}\Big\}<1.\]
Arguing as in Lemma \ref{slidetheballs}, we conclude that
\[B_\delta\big(c(t_0)\big)\subset E_{ext}\quad\mbox{ and }\quad \exists\,q\in\partial B_\delta\big(c(t_0)\big)\cap\partial E.\]
By definition of $c$, we have that either $q\in\Omega$ or
\[q\in\partial\Omega\cap B_R(x_0).\]
In both cases (see Theorem 5.1 in \cite{nms} and Theorem 1.1 in \cite{elsulbordo}) we have
\[\I_s[E](q)\leq0,\]
which gives a contradiction with Theorem \ref{positivecurvature}.
\end{proof}

%%%%%%%%%%%%%%%%%%% SECTION %%%%%%%%%%%%%%%%%%%%%%%%%%%%%%
 
   \subsection{The contribution from infinity of some supergraphs} \label{sectexamples} We compute  in this Subsection the contribution from infinity of some particular supergraphs.
 \begin{example}[The cone] \label{THECONE} Let $ S \subset  \mathbb S^{n-1}$ be a portion of the unit sphere, $\mathfrak o:=\Ha^{n-1}(S)$ and
  \[ C:=\{ t\sigma  \; \big| \; t\geq 0, \;\sigma\in  S)\}.\] 
  Then the contribution from infinity is given by the opening of the cone,
  \eqlab{\label{cony} \alpha(C) = \mathfrak o.}
  Indeed,
  \[\alpha_s(0,1,C)= \int_{\C B_1} \frac{\chi_C(y)}{|y|^{n+s}}\, dy = \Ha^{n-1}(S) \int_1^{\infty} t^{-s-1}\, dt=
  \frac{\mathfrak o}s,  \]
and we obtain the claim by passing to the limit. Notice that this says in particular that the contribution from infinity of a half-space is $\omega_n/2$.
  \end{example} 
  
%%%%%%%%%%%%%%%%%%%%%%%%%%%%%%%%%%%%
  \begin{center}
\begin{figure}[htpb]
	\hspace{0.99cm}
	\begin{minipage}[b]{0.99\linewidth}
	\centering
	\includegraphics[width=0.99\textwidth]{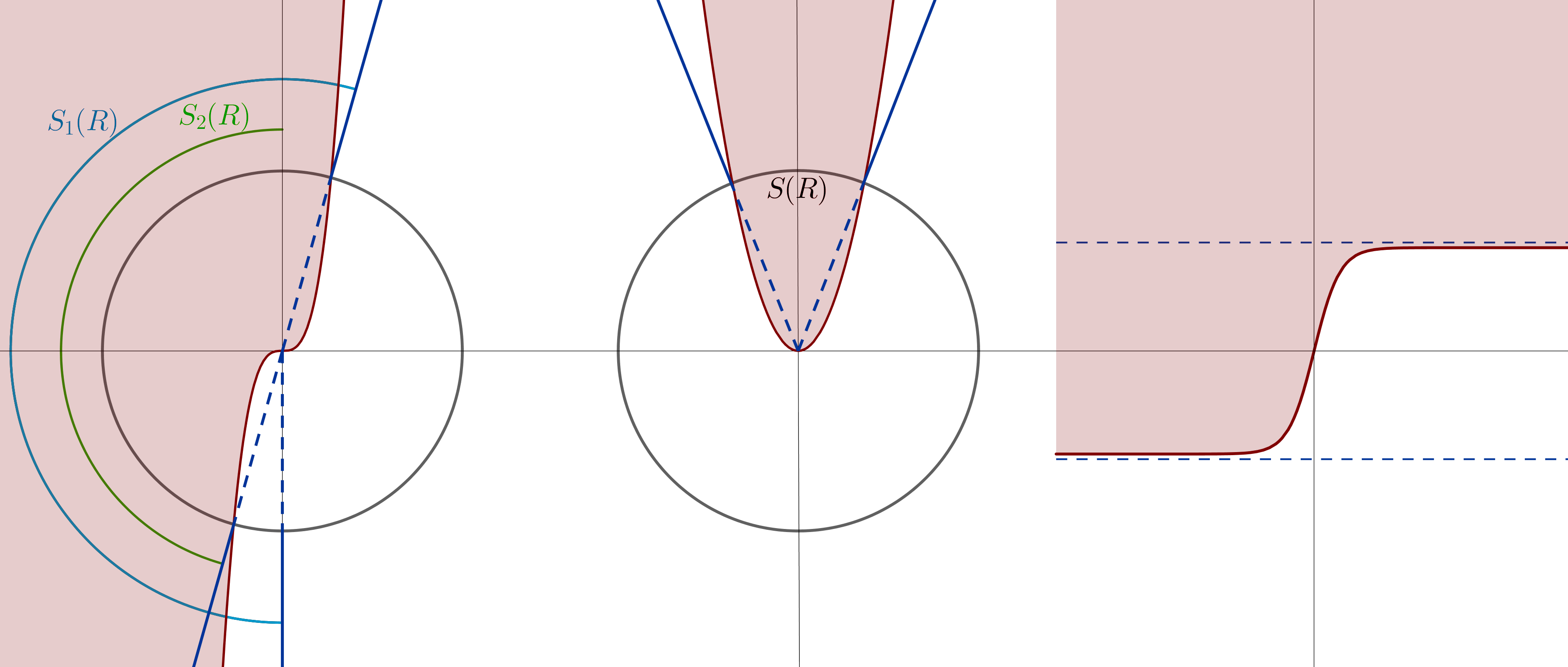}
	\caption{The contribution from infinity of $x^3$, $x^2$ and $\tanh x$}   
	\label{x3}
	\end{minipage}
\end{figure} 
	\end{center}
%%%%%%%%%%%%%%%%%%%%%%%%%%%%%%%%%%%%%%%%%%%%
	  \begin{example}[The parabola]
  We consider the supergraph 
  \[ E:=\{ (x',x_n) \; \big| \; x_n\geq |x'|^2\},\] and we show that, in this case, \[ \alpha(E)=0.\]
  In order to see this, we take any $R>0$, intersect the ball $B_R$ with the parabola and build a cone on this intersection (see the second picture in Figure \ref{x3}), i.e. we define
  \[ S(R):=\partial B_R \cap E , \quad \quad C_R =\{ t\sigma  \; \big| \; t\geq 0, \;\sigma\in  S(R)\}.\] 
  We can explicitly compute the opening of this cone, that is
  \[ \mathfrak o (R)= \Bigg( \arcsin \frac{\left(\sqrt{4R^2+1}-1\right)^{1/2}}{R\, \sqrt 2 }\Bigg)\frac{\omega_n}{\pi} .\]
  Since $E\subset C_R$ outside of $B_R$, thanks to the monotonicity property in Proposition \ref{subsetssmin} and to \eqref{cony}, we have that
  \[  \overline \alpha (E) \leq \overline \alpha(C_R)  =\mathfrak o(R).\]
  Sending $R\to \infty$, we find that
\[  \overline \alpha (E) =0, \quad \mbox{ thus } \quad \alpha(E)=0.\]
  \end{example}
  More generally, if we consider for any given $c, \eps>0$ a function $u$ such that
  \[ u(x')>c|x'|^{1+\eps}, \quad \mbox{ for  any }|x'|>R \mbox{ for  some } R>0\]
  and\[   E:=\{ (x',x_n) \; \big| \; x_n\geq u(x')\},\]
  then
\[ \alpha(E)=0.\]
On the other hand, if we consider a function that is not rotation invariant, things can go differently, as we see in the next example. 
  
%%%%%%%%%%%%%%%%%%%%%%%%%%
   \begin{example}[The supergraph of $x^3$]
 We consider the supergraph 
  \[ E:=\{ (x,y) \; \big| \; y\geq x^3\}.\] 
In this case, we show that \[\alpha(E) = \pi .\]
For this, given $R>0$, we intersect  $\partial B_R$ with $E$  and denote by $S_1(R)$ and $S_2(R)$ the arcs on the circle as the first picture in Figure \ref{x3}. We consider
the cones\[ C^1_R :=\{ t\sigma  \; \big| \; t\geq 0, \;\sigma\in  S_1(R)\}\, \quad  C^2_R :=\{ t\sigma  \; \big| \; t\geq 0, \;\sigma\in  S_2(R)\} \] and
notice that outside of $B_R$, it holds that $C_R^2\subset E\subset C_R^1$. 
Let  $\overline x_R $ be the solution of \[ x^6+x^2=R^2,\] that is the $x$-coordinate in absolute value of the intersection points $\partial B_R \cap \partial E$. Since $f(x)=x^6+x^2$ is increasing on $(0,\infty)$ and $ R^2=f(\overline x_R)<f(R^{1/3}),$ we have that  $\overline x_R< R^{1/3}$. Hence  
\[ \mathfrak o^1(R)= \pi + \arcsin \frac{\overline x_R}R  \leq \pi+  \arcsin \frac{R^{1/3}}R  ,\quad   \mathfrak o^2(R)  \geq \pi- \arcsin \frac{R^{1/3}}R .\] 
Thanks to the monotonicity property in Proposition \ref{subsetssmin} and to \eqref{cony} we have that
  \[  \overline \alpha (E) \leq\alpha(C_R^1)=  \mathfrak o^1(R), \quad \underline \alpha(E) \geq \alpha(C_R^2)= \mathfrak o^2(R)  \]
  and sending  $R\to \infty$ we obtain that
\[ \overline \alpha(E) \leq  \pi , \quad \underline \alpha(E) \geq \pi.\] Thus $\alpha(E)$ exists and we obtain the desired conclusion.
 \end{example}

 %%%%%%%%%%%%%%%%%%%%%%%%%%%%%%%%%%%%%%
 \begin{example}[The supergraph of a bounded function]\label{tanh}
 We consider the supergraph 
  \[ E:=\{ (x',x_n) \; \big| \; x_n\geq u(x') \}, \quad \quad{\mbox{with}}
\quad \quad \|u\|_{L^\infty(\Rn)} <M.\] 
We show that, in this case,
\[\alpha(E) = \frac{\omega_n}2 .\]
To this aim, let \bgs{ &\mathfrak{P}_1:=\{ (x',x_n)\; \big| \; x_n>M\} \\
	&\mathfrak{P}_2:=\{ (x',x_n)\; \big| \; x_n<-M\}.}
	We have that 
 \[   \mathfrak{P}_1 \subset E , \quad \quad \mathfrak{P}_2 \subset \C E . \] 
  Hence by Proposition \ref{subsetssmin}
  \[ \underline \alpha(E) \geq \overline \alpha (\mathfrak{P}_1)=\frac{\omega_n}2,\quad  \quad  
  \underline \alpha(\C E)\geq \overline\alpha (\mathfrak{P}_2)=\frac{\omega_n}2 .\]
  Since $\underline \alpha(CE) =\omega_n - \overline \alpha(E)$ we find that
   \[ \overline \alpha(E)\leq \frac{\omega_n}2,\]
  thus the conclusion. An example of this type is depicted in Figure \ref{x3} (more generally, the result holds for the supergraph in $\Rn$  
 $ \{ (x',x_n) \; \big| \; x_n\geq \tanh x_1\}$).
 \end{example}

 \begin{example}[The supergraph of a sublinear graph]\label{candygr}
   More generally, we can take the supergraph of a function that grows sublinearly at infinity, i.e.  
  \[  E:=\{(x',x_n)\;\big|\; x_n>u(x')\}, \qquad{\mbox{with}}\qquad 
\lim_{|x'|\to +\infty} \frac{|u(x')|}{|x'|}=0 .\]
In this case, we show that \[ \alpha(E)= \frac{\omega_n}2.\]
Indeed, for any $\eps>0$ we have that there exists $R=R(\eps)>0$ such that
\[ |u(x')|<\eps |x'|, \quad \forall \;|x'|>R.\]
We denote
\[ S_1(R):= \partial B_R \cap \{(x',x_n)\;\big|\; x_n>\eps|x'|\}, \qquad S_2(R):= \partial B_R \cap \{(x',x_n)\;\big|\; x_n<-\eps|x'|\}\]
and
\[ C_R^i=\{ t\sigma\; \big| \; t\geq 0, \; \sigma \in S_i(R)\}, \quad \mbox{for }\; i=1,2.\]
We have that outside of $B_R$
\[ C_R^1\subset E, \qquad C_R^2\subset \C E ,\] and 
\[ \alpha (C_R^1)= \alpha (C_R^2)= \frac{ \omega_n}{\pi} \left( \frac{\pi}2- \arctan \eps \right).\]
We use Proposition \ref{subsetssmin}, (i), and letting $\eps$ go to zero, we obtain that $\alpha(E)$ exists and 
\[ \alpha (E)= \frac{\omega_n}2.\]
%We see that
%\[ B_R\cap \{(x',x_n)\; \big| \; x_n=\eps|x'|\} =\{(x',x_n)\; \big|\; |x'|=\frac{R}{\sqrt{1+\eps^2}}, x_n=\eps|x'|\}.\]
 A particular example of this type is given by
 \[ E:=\{(x',x_n)\;\big|\; x_n>c|x'|^{1-\eps} \}, \quad \mbox{ when } |x'|>R\; \;  \mbox{ for some } \eps\in(0,1],\,c\in\R,\, R>0     .\]
 \end{example}
 In particular using the additivity property in Proposition \ref{subsetssmin} we can compute $\alpha$ for sets that lie between two graphs. 
 \begin{example}[The ``butterscotch hard candy'']
  \begin{center}
\begin{figure}[htpb]
	\hspace{0.89cm}
	\begin{minipage}[b]{0.89\linewidth}
	\centering
	\includegraphics[width=0.89\textwidth]{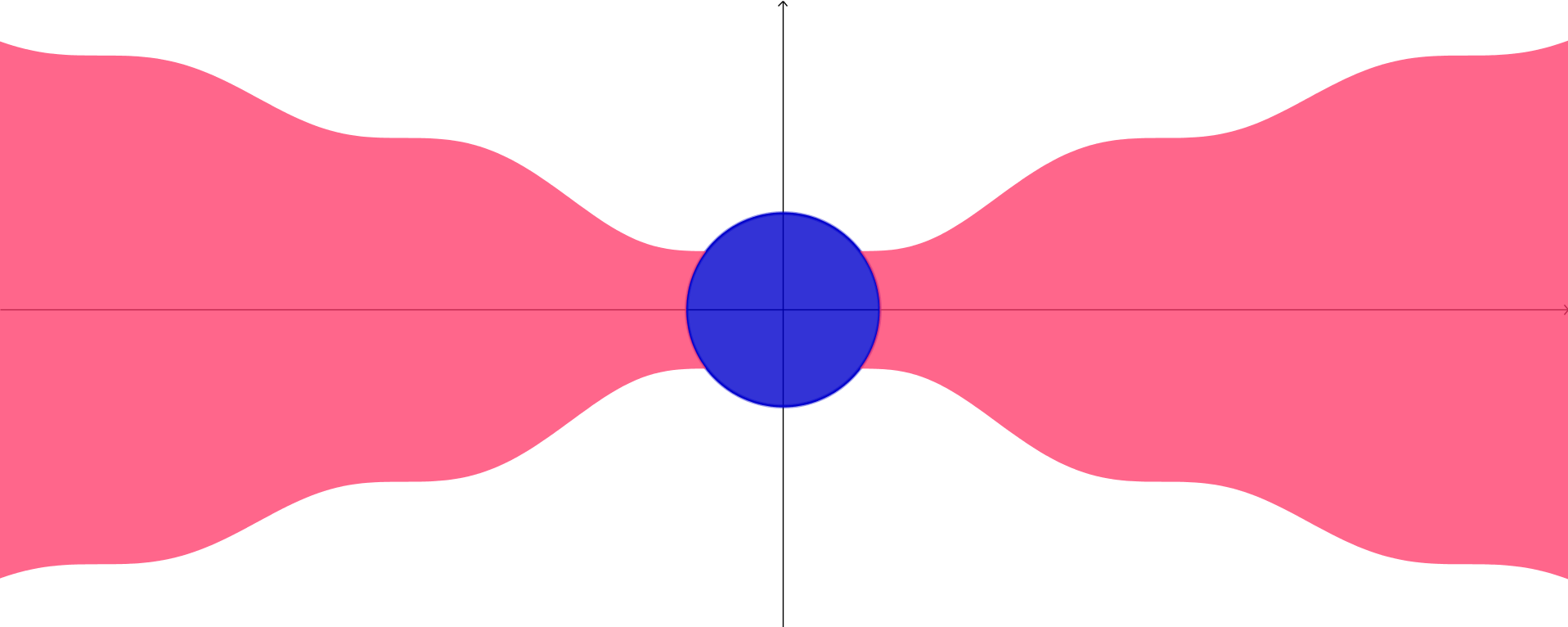}
	\caption{The ``butterscotch hard candy'' graph}   
	\label{candy_pic}
	\end{minipage}
\end{figure} 
	\end{center}
Let~$E\subset\R^n$ be such that
\[ E\cap\{|x'|>R\}\subset
\{(x',x_n)\;\big|\; |x'|>R\;,\;\,|x_n|<c|x'|^{1-\eps} \},   \;\; 
\; \;  \mbox{ for some } \eps\in(0,1],\,c>0,\,  R>0     ,\]
(an example of such a set $E$ is given in Figure \ref{candy_pic}).
In this case, we have that \[\alpha(E)=0.\]
Indeed, we can write $E_1:=E\cap\{|x'|>R\}$ and $E_2:=
E\cap\{|x'|\le R\}$. Then,
using
%the notations and
 the computations in Example \ref{candygr}, we have by
the monotonicity and the additivity properties in Proposition \ref{subsetssmin} that
\[ \overline{\alpha}(E_1) %\leq {\omega_n}- \alpha(C_R^1)-\alpha(C_R^2) =0.
\le \alpha\big(\{x_n>-c|x'|^{1-\eps}\}\big)-\alpha\big(\{x_n>c|x'|^{1-\eps}\}\big)=0.\]
Moreover, $E_2$ lies inside $\{|x_1|\le R\}$. Hence, again by Proposition \ref{subsetssmin} and by Example \ref{THECONE},
we find
\[\overline{\alpha}(E_2)\le\alpha\big(\{|x_1|\le R\}\big)=\alpha\big(\{x_1\le R\}\big)-\alpha\big(\{x_1<-R\}\big)=0.\]
Consequently, using again the additivity property in Proposition \ref{subsetssmin}, we obtain that
$$ \overline\alpha(E)\le \overline\alpha(E_1)+\overline\alpha(E_2)
%\le\overline\alpha(E_1)+\overline\alpha(E_2')
=0,$$
that is the desired result.
\end{example}
%%%%%%%%%%%%%%%%%%%%%%%%%%%%%%%%%%%%%%
We can also compute $\alpha$ for sets that have different growth ratios in different directions. For this, we have the following example.
   \begin{example}[The supergraph of a superlinear function on a small cone]
  We consider a set lying in the half-space, deprived of a set that grows linearly at infinity.  We denote by $\tilde S$ the portion of the sphere given by 
 \bgs{ \tilde{S}:=\Big\{\sigma \in \mathbb S^{n-2}\,\Big|\,\sigma=(&\cos\sigma_1, \sin \sigma_1\cos\sigma_2,\dots, \sin \sigma_1\dots\sin \sigma_{n-2}),\\ &  \mbox{ with } \sigma_i\in \lr{\frac{\pi}2-\bar \eps, \frac{\pi}2+\bar \eps},\;  i=1,\dots,n-2\Big\} ,} where $\overline \eps \in (0,{\pi}/{2})$.
 For $x_0\in \Rn$ and $k>0$ we define  the supergraph $E\subset \Rn$ as
\bgs{\label{sigma}  E:=\big\{ (x',x_n)\in \Rn \; &\big| \; x_n\geq u(x')\big \} \quad   \mbox{ where }  \quad
 u(x')=\alig{ & \, k|x'-x_0'| &\mbox{ for } &x'\in X, \\
		 & \, 0 &\mbox{ for } &x'\notin X,  }\\
	  &X= \{ x'   \in \R^{n-1}  \mbox{ s.t. }x' = t\sigma +x_0',\,  \sigma \in  \tilde S\}.   }
We remark that $X\subset \{x_n=0\}$ is the cone ``generated'' by $\tilde S$ 
and centered at $x_0$.	 Then
\eqlab{\label{alphasigma} \alpha(E) = \frac{\omega_n}2 -  \mathcal{H}^{n-2}(\tilde S)  \int_{0}^k \frac{dt}{(1+t^2)^{\frac{n}2}}.}
Let 
\[ \mathfrak{P}_+
:=\{ (x',x_n) \; \big| \; x_n>0\}, \quad \quad \mathfrak{P}_-:=\{ (x',x_n) \; \big| \; x_n<0\}\] and we consider
the subgraph
\[ F:= \big\{ (x',x_n)\; \big| \; 0<x_n< u(x')\big \}.\]
Then \[ E \cup F=\mathfrak P_+, \quad \quad \mathfrak P_-\cup F = \C E.\] 
Using the additivity property in Proposition \ref{subsetssmin}, we see that
\eqlab{\label{ah123}  \overline \alpha(E)\geq \frac{\omega_n}2-\overline \alpha(F), \quad \quad  \omega_n - \underline \alpha(E) = \overline \alpha(\C E)\leq \frac{\omega_n}2 +\overline \alpha(F).} 
Let $R>0$ be arbitrary. We get that
\[\alpha_s(x_0,R,F) \leq  \int_{\left(B'_R(x'_0)\times \R\right)\cap \C B_R(x_0)}\frac{\chi_{F}(y)}{|y-x_0|^{n+s}}\, dy + \int_{\C \left( B'_R(x'_0)\times \R\right)}\frac{\chi_{F}(y)}{|y-x_0|^{n+s}}\, dy\]
so
\eqlab{\label{i123}  \alpha_s(x_0,R,F) \leq &\;  \int_{B'_R(x'_0)}\frac{dy'}{|y'-x_0'|^{n-1+s}}\int_{\frac{\sqrt{R^2-|y'-x_0'|^2}}{|y'-x_0'|}}^{\infty} \frac{dt}{(1+t^2)^{\frac{n+s}2}}
   \\ &\; + 
   \int_{\C B'_R(x'_0)\cap X}\frac{dy'}{|y'-x_0'|^{n-1+s}}\int_{0}^{k} \frac{dt}{(1+t^2)^{\frac{n+s}2}}\\&\;=I_1+I_2.} 
   Using that $1+t^2\geq \max\{1,t^2\}$ and passing to polar coordinates, we obtain that
   \bgs{ I_1= &\; \int_{B'_R(x_0')}\frac{dy'}{|y'-x_0'|^{n-1+s}}\bigg(\int_{\frac{\sqrt{R^2-|y'-x_0'|^2}}{|y'-x_0'|}}^{\frac{R}{|y'-x_0'|}} \frac{dt}{(1+t^2)^{\frac{n+s}2}} + \int_{\frac{R}{|y'-x_0'|}}^{\infty}\frac{dt}{(1+t^2)^{\frac{n+s}2}}\bigg)\\
   \leq&\; \omega_{n-1} \bigg( \int_0^R \tau^{-s-2}   \left(R- \sqrt{R^2-\rho^2}\right)\, d \rho  +  \frac{R^{-n-s+1}}{n+s-1} \int_0^R \rho^{n-2}\, d \rho    \bigg)\\
   =&\;   \omega_{n-1} \bigg( R^{-s} \int_0^1 \tau^{-s-2}   \left(1- \sqrt{1-\tau^2}\right)\, d \tau  +  \frac{R^{-s}}{(n+s-1)(n-1)}  \bigg)  .}
   Also, for any $\tau \in (0,1)$ we have that
\[1 -  \sqrt{1-\tau^2 }\leq c \tau^2,\] for some positive constant $c$, independent on $n,s$.
Therefore 
\bgs{ I_1\leq  \frac{ c \omega_{n-1} R^{-s} }{1-s} + \frac{ \omega_{n-1} R^{-s} }{(n-1)(n+s-1)} .}
Moreover, 
\[ I_2 =   \mathcal{H}^{n-2}(\tilde S) \frac{R^{-s}}s \int_{0}^k \frac{dt}{(1+t^2)^{\frac{n+s}2}}.\]
So
passing to limsup and liminf as $s\to 0^+$ in \eqref{i123} and using Fatou's lemma we obtain that
\[\overline \alpha(F)\leq   \mathcal{H}^{n-2}(\tilde S)\int_{0}^k \frac{dt}{(1+t^2)^{\frac{n}2}},\quad \quad \underline \alpha(F) \geq  \mathcal{H}^{n-2}(\tilde S)\int_{0}^k \frac{dt}{(1+t^2)^{\frac{n}2}}.\] In particular $\alpha(F)$ exists, and from \eqref{ah123} we get that 
\[ \frac{\omega_n}2 -\alpha(F) \leq \underline \alpha(E)\leq \overline \alpha(E)\leq\frac{\omega_n}2 -\alpha(F).\] 
Therefore, $\alpha(E)$ exists and 
\[  \alpha(E) =  \frac{\omega_n}2 -  \mathcal{H}^{n-2}(\tilde S)  \int_{0}^k \frac{dt}{(1+t^2)^{\frac{n}2}} .\] 
 \end{example}

\subsection{Continuity of the fractional mean curvature and a sign changing
property \\of the nonlocal mean curvature}\label{cont}

We use a formula proved in \cite{regularity} to show that the $s$-fractional mean curvature is continuous
with respect to $C^{1,\alpha}$ convergence of sets, for any $s<\alpha$ and with respect to $C^2$ convergence of sets, for $s$ close to 1.

By $C^{1,\alpha}$ convergence of sets we mean that our sets locally converge in measure and can locally be described as the supergraphs
of functions which converge in $C^{1,\alpha}$. 

\begin{defn}\label{convofsets}
Let $E\subset\R^n$ and let $q\in\partial E$ such that $\partial E$ is $C^{1,\alpha}$ near $q$, for some $\alpha\in(0,1]$. We say that the sequence $E_k\subset\R^n$ converges to $E$ in a $C^{1,\alpha}$ sense (and write $E_k\xrightarrow{C^{1,\alpha}}E$) in a neighborhood of $q$ if:\\
(i)  the sets $E_k$ locally converge in measure to $E$, i.e.
\[|(E_k\Delta E)\cap B_r|\xrightarrow{k\to\infty}0 \quad \mbox{ for any } r>0\]
and \\
(ii) the boundaries $\partial E_k$ converge to $\partial E$ in $C^{1,\alpha}$ sense in a neighborhood of $q$.\\
We define in a similar way the $C^2$ convergence of sets.
\end{defn}

More precisely, 
we denote  
\[Q_{r,h}(x):=B'_r(x')\times(x_n-h,x_n+h),\]
for $x\in\R^n$, $r,h>0$. If $x=0$, we drop it in formulas and simply write $Q_{r,h}:=Q_{r,h}(0)$. Notice that up to a translation and a rotation, we can suppose that $q=0$ and
\eqlab{\label{opossum1}E\cap Q_{2r,2h}=\{(x',x_n)\in\R^n\,|\,x'\in B'_{2r},\,u(x')<x_n<2h\},}
for some $r,h>0$ small enough and $u\in C^{1,\alpha}(\overline B'_{2r})$ such that $u(0)=0$. 
%We also consider a sequence of sets $E_k\subset\R^n$ which locally converge in measure to $E$, i.e.
%\[|(E_k\Delta E)\cap B_r|\xrightarrow{k\to\infty}0 \quad \mbox{ for any } r>0\]
%and whose boundaries $\partial E_k$ converge to $\partial E$ in $C^{1,\alpha}$ sense in a neighborhood of $q=0$.
Then, point $(ii)$ means that we can write
\eqlab{\label{opossum}E_k\cap Q_{2r,2h}=\{(x',x_n)\in\R^n\,|\,x'\in B'_{2r},\,u_k(x')<x_n<2h\},}
for some functions $u_k\in C^{1,\alpha}(\overline B_{2r}')$ such that
\eqlab{\label{norm_graph_conv}\lim_{k\to\infty}\|u_k-u\|_{C^{1,\alpha}(\overline B'_{2r})}=0.}
\smallskip
We remark that, by the continuity of $u$,
up to considering a smaller $r$, we can suppose that
\eqlab{\label{bded_graph_hp}|u(x')|<\frac{h}{2},\qquad\forall\,x'\in B'_{2r}.}

We have the following result.

\begin{theorem}\label{everything_converges}
Let $E_k\xrightarrow{C^{1,\alpha}}E$ in a neighborhood of $q\in \partial E$. Let $q_k\in\partial E_k$ be such that $
q_k\longrightarrow q$ and let $s,s_k\in(0,\alpha)$ be such that $s_k\xrightarrow{k\to\infty} s$.
Then
\[\lim_{k\to\infty}\I_{s_k}[E_k](q_k)=\I_s[E](q).\]

Let $E_k\xrightarrow{C^2}E$ in a neighborhood of $q\in \partial E$. Let $q_k\in\partial E_k$ be such that $q_k\longrightarrow q$ and let $s_k\in(0,1)$ be such that $s_k\xrightarrow{k\to\infty} 1$. Then
\[\lim_{k\to\infty}(1-s_k)\I_{s_k}[E_k](q_k)=\omega_{n-1}H[E](q).\]
%
%Let $E_k\xrightarrow{C^{1,\alpha}}E$ in a neighborhood of $q\in \partial E$ and $|E_k\Delta E|\to0$.
%Let $q_k\in\partial E_k$ be such that $q_k\longrightarrow q$ and let $s_k\in(0,\alpha)$ be such that $s_k\xrightarrow{k\to\infty} 0$. Then
%\[\lim_{k\to\infty}s_k\I_{s_k}[E_k](q_k)=\omega_n-2\alpha(E).\]
\end{theorem}
\medskip 

A similar problem is studied also in \cite{mattheorem}, where the author estimates the
difference between the fractional mean curvature of a set $E$ with $C^{1,\alpha}$
boundary and that of the set
$\Phi(E)$, where $\Phi$ is a $C^{1,\alpha}$ diffeomorphism of $\R^n$,
%The estimates obtained there are much more precise than ours.
%However, as far as only convergence is
%considered, our result is more general in the sense that the sets involved need not be diffeomorphic.
%
%Moreover, as remarked above, our convergence is somewhat local, while the setting in \cite{mattheorem} is global. Indeed, the author estimates the difference between the curvatures
in terms of the $C^{0,\alpha}$ norm of the Jacobian of the diffeomorphism $\Phi$.

%Thus, even if we want the convergence of the curvatures only in a neighborhood $0$,
%to use \cite{mattheorem}
%we still need to ask $C^{1,\alpha}$ regularity for the whole boundary. 

%On the other hand, w
\smallskip

When $s\to 0^+$ we do not need the $C^{1,\alpha}$ convergence of sets, but only the uniform boundedness  of the $C^{1,\alpha}$ norms of the functions defining the boundary of $E_k$ in a neighborhood of the boundary points. However, we have to require that the measure of the symmetric difference is uniformly bounded. More precisely:

\begin{prop}\label{propsto0}
Let $ E\subset \Rn$
be such that $\alpha(E)$ exists.  Let 
$q \in \partial E$ be such that 
\bgs{ E\cap Q_{r,h}(q)=\{(x',x_n)\in\R^n\,|\, x'\in B'_{r}(q'),\,u(x')<x_n<h+q_n\},}
for some $r,h>0$ small enough and $u\in C^{1,\alpha}(\overline B'_{r}(q'))$ such that $u(q')=q_n$. 
Let $E_k\subset \Rn$ be such that
\[ |E_k\Delta E|<C_1  \] 
for some $C_1>0$. Let $q_k\in \partial E_k \cap B_d$, for some $d>0$, such that 
 \bgs{E_k\cap Q_{r,h}(q_k)=\{(x',x_n)\in\R^n\,|\,x'\in B'_{r}(q_k'),\,u_k(x')<x_n<h+q_{k,n}\} }
for some functions $u_k\in C^{1,\alpha}(\overline B_{r}'(q_k'))$ such that $u_k(q_k')=q_{k,n}$ and
\[ \|u_k\|_{C^{1,\alpha}(\overline B'_{r}(q'_k))} <C_2 \] 
for some $C_2>0$. Let $s_k\in(0,\alpha)$ be such that $s_k\xrightarrow{k\to\infty} 0$. Then
\bgs{\lim_{k\to\infty}s_k\I_{s_k}[E_k](q_k)=\omega_n-2\alpha(E).}
\end{prop}

In particular, fixing $E_k=E$ in Theorem \ref{everything_converges} and Proposition \ref{propsto0} we obtain Proposition \ref{rsdfyish} stated in the Introduction. 

To prove
Theorem \ref{everything_converges}
%(and so Corollary \ref{rsdfyish}),
we prove at first the following preliminary result. 

\begin{lemma}\label{supergraph_hp_for_proof}
Let  $E_k\xrightarrow{C^{1,\alpha}}E$ in a neighborhood of $0\in \partial E$. Let $q_k\in\partial E_k$ be such that $
q_k\longrightarrow 0$.  Then \[ E_k-q_k\xrightarrow{C^{1,\beta}}E  \quad \mbox{ in a neighborhood of $0$},\] 
for every $\beta \in(0,\alpha)$.\\
Moreover, if $E_k\xrightarrow{C^{2}}E$ in a neighborhood of $0\in \partial E$, $q_k\in\partial E_k$ are such that $
q_k\longrightarrow 0$ and $\mathcal{R}_k\in  SO(n)$ are such that \[\lim_{k\to \infty} |\mathcal R_k -\mbox{Id}|=0,\] then
\[ \mathcal R_k (E_k-q_k) \xrightarrow{C^{2}}E \quad \mbox{ in a neighborhood of $0$ }.\]
%Then
%there exists $\tilde{k}\in\N$ big enough such that
%
%$(i\mbox{  })\quad$ for every $k\geq\tilde{k}$ we have
%\eqlab{\label{graphs_for_the_proof_eq}
%E_k\cap Q_{r,h}(q_k)=\{(x',x_n)\in\R^n\,|\,x'\in B'_r(q'_k),\,u_k(x')<x_n<q_{k,n}+h\}.}
%
%$(ii)\mbox{ }\quad$ if we denote
%\[\tilde{u}_k(x'):=u_k(x'+q'_k),\qquad x'\in \overline{B}'_r,\]
%then $\tilde{u}_k\in C^{1,\alpha}(\overline{B}'_r)$ and
%\eqlab{\label{bded_norm_of_graphs}
%\|\tilde{u}_k\|_{C^{1,\alpha}(\overline{B}'_r)}\leq M,\qquad\forall\,k\geq\tilde{k}.}
%
%$(iii)\quad$ we have
%\eqlab{\label{conv_transl_graph}
%\lim_{k\to\infty}\|\tilde{u}_k-u\|_{C^1(\overline{B}'_r)}=0.}
\end{lemma}

\begin{proof}
First of all, notice that since $q_k\longrightarrow0$, for $k$ big enough we have
\[|q'_k|<\frac{1}{2}r\qquad\textrm{and}\qquad|q_{k,n}|=|u_k(q'_k)|<\frac{1}{8}h.\]
 By \eqref{bded_graph_hp} and \eqref{norm_graph_conv}, we
see that for $k$ big enough
\[|u_k(x')|\leq\frac{3}{4}h,\qquad\forall\,x'\in B_{2r}'.\]
Therefore
\[|u_k(x')-q_{k,n}|<\frac{7}{8}h<h,\qquad\forall\,x'\in B_{2r}'.\]
%In particular, we have established \eqref{graphs_for_the_proof_eq}, as desired.
If we define
\[\tilde{u}_k(x'):=u_k(x'+q'_k),\qquad x'\in \overline{B}'_r,\]
% then $\tilde{u}_k\in C^{1,\alpha}(\overline{B}'_r)$ and
%\eqlab{\label{bded_norm_of_graphs}
%\|\tilde{u}_k\|_{C^{1,\alpha}(\overline{B}'_r)}\leq M,\qquad\forall\,k\geq\tilde{k}.}
for every $k$ big enough we have 
\eqlab{\label{graphs_for_the_proof_eq}
(E_k-q_k) \cap Q_{r,h} =\{(x',x_n)\in\R^n\,|\,x'\in B'_r,\, \tilde u_k(x')<x_n<h\}.}
It is easy to check that the sequence $E_k-q_k$ locally converges in measure to $E$. We claim that
\eqlab{\label{conv_transl_graph}
\lim_{k\to\infty}\|\tilde{u}_k-u\|_{C^{1,\beta}(\overline{B}'_r)}=0.}
Indeed, let 
\[ \tau_k u(x'):= u(x'+q_k').\] 
We have that
\[ \|\tilde u_k-\tau_k u\|_{C^{1}(\overline B_r')}  \leq \|u_k-u\|_{C^{1}\big(\overline B'_{\frac{3}2r}\big)}  \]
and that
\[ \|\tau_k u-u\|_{C^{1}(\overline B_r')} \leq  \|\nabla u\|_{C^0\big(\overline B'_{\frac{3}2r}\big)}  |q_k'| + \|u\|_{C^{1,\alpha}\big( \overline B'_{\frac{3r}2}\big)}|q'_k|^{\alpha} .\] Thus by the triangular inequality 
\[  \lim_{k \to \infty} \|\tilde u_k -u\|_{C^{1}(\overline B_r')} =0,\]
thanks to \eqref{norm_graph_conv} and the fact that $q_k\to 0$.

Now, notice that $\nabla (\tilde u_k) =\tau_k (\nabla u_k)$, so 
\[ [\nabla \tilde u_k -\nabla u]_{C^{0,\beta}(\overline B'_r)} \leq  [\tau_k (\nabla  u_k -\nabla u)]_{C^{0,\beta}(\overline B'_r)}+ [\tau_k(\nabla  u) -\nabla u)]_{C^{0,\beta}(\overline B'_r)}.\] 
Therefore
\[[\tau_k (\nabla  u_k -\nabla u)]_{C^{0,\beta}(\overline B'_r)} \leq  [\nabla  u_k -\nabla u]_{C^{0,\beta}\big(\overline B'_{\frac{3r}2}\big)}\]
and for every $\delta >0$ we obtain
\[  [\tau_k( \nabla  u) -\nabla u]_{C^{0,\beta}(\overline B'_r)} \leq \frac{2}{\delta^\beta} \|\tau_k (\nabla u) -\nabla u\|_{C^{0}\big(\overline B'_{\frac{3r}2}\big)} + 2[\nabla u]_{C^{0,\alpha}(\overline B'_r)} \delta^{\alpha-\beta}.\]
Sending $k\to \infty$ we find that
\[ \limsup_{k\to \infty}  [\tau_k (\nabla  u) -\nabla u)]_{C^{0,\beta}(\overline B'_r)}  \leq 2[\nabla u]_{C^{0,\alpha}(\overline B'_r)} \delta^{\alpha-\beta}\]
for every $\delta>0$, hence
\[ \lim_{k \to \infty}  [\nabla \tilde u_k -\nabla u]_{C^{0,\beta}(\overline B'_r)} =0.\]
This concludes the proof of the first part of the Lemma.\\
As for the second part, the $C^2$ convergence of sets in a neighborhood of $0$ can be proved similarly. Some care must be taken when considering rotations, since one needs to use the implicit function theorem.
\end{proof}

\begin{proof}[Proof of Theorem \ref{everything_converges}]

Up to a translation and a rotation, we can suppose that $q=0$ and $\nu_E(0)=0$. Then we can find $r,h>0$ small enough
and $u\in C^{1,\alpha}(\overline B'_r)$ such that we can write $E\cap Q_{2r,2h}$ as in \eqref{opossum1}.

Since $s_k\to s\in(0,\alpha)$ for $k$ large enough we can suppose that $s_k,s \in[\sigma_0,\sigma_1]$ for $0<\sigma_0<\sigma_1<\beta<\alpha$.
%and since $q_k\longrightarrow0$, we can suppose that
%\eqlab{\label{continuity_eq1}Q_{\frac{r}{2},\frac{h}{2}}\subset Q_{r,h}(q_k)\qquad\textrm{for every }k,}
Notice that there exists $\delta>0$ such that
\eqlab{\label{continuity_eq2}
%B_\delta(q_k),\,
B_\delta\subset\subset Q_{r,h}.}
We take an arbitrary $R>1$ as large as we want and define the sets
\[F_k:= (E_k\cap B_R) -q_k.\]
From Lemma \ref{supergraph_hp_for_proof} we have that in a neighborhood of $0$
\[ F_k\xrightarrow{C^{1,\beta}} E\cap B_R.\] 
In other words,
\eqlab{\label{convvvvv1}\lim_{k\to \infty} |F_k \Delta (E\cap B_R)| =0.}
Moreover, if $u_k$ is a function defining $E_k$ as a supergraph in a neighborhood of $0$ as in \eqref{opossum},
denoting $\tilde u_k(x')=u_k(x'+q_k')$ we have that
\[F_k\cap Q_{r,h}=\{(x',x_n)\in\Rn\,|\,x'\in B'_r,\,\tilde{u}_k(x')<x_n<h\}\]
and that
\eqlab{\label{convvvvv} \lim_{k\to \infty} \|\tilde u_k -u\|_{C^{1,\beta}(\overline B'_r)} =0, \quad  \quad  \|\tilde u_k\|_{C^{1,\beta}(\overline B'_r)}\leq M \; \mbox{ for some } \; M>0.}
We also remark that, by \eqref{bded_graph_hp} we can write
\[E\cap Q_{r,h}=\{(x',x_n)\in\R^n\,|\,x'\in B'_r,\,u(x')<x_n<h\}.\]

Exploiting \eqref{graphs_for_the_proof_eq} we can write the fractional mean curvature of $F_k$ in $0$
by using formula \eqref{complete_curv_formula}, that is
\begin{equation}\label{5rs6ydbfd}
\begin{split}
\I_{s_k}[F_k](0)&=2\int_{B'_r}\Big\{G_{s_k}\Big(\frac{\tilde{u}_k(y')-\tilde{u}_k(0)}{|y'|}\Big)
-G_{s_k}\Big(\nabla \tilde{u}_k(0)\cdot\frac{y'}{|y'|}\Big)\Big\}\frac{dy'}{|y'|^{n-1+s_k}}\\
&
\qquad\qquad+\int_{\Rn}\frac{\chi_{\C F_k}(y)-\chi_{F_k}(y)}{|y|^{n+s_k}}\chi_{\C Q_{r,h}}(y)\,dy.\end{split}\end{equation}
Now, we denote as in \eqref{mathcalg}
\[\mathcal G(s_k,\tilde{u}_k,y'):=\mathcal G(s_k,\tilde{u}_k,0,y')= G_{s_k}\Big(\frac{\tilde{u}_k(y')-\tilde{u}_k(0)}{|y'|}\Big)
-G_{s_k}\Big(\nabla \tilde{u}_k(0)\cdot\frac{y'}{|y'|}\Big)\] and we rewrite the
identity in \eqref{5rs6ydbfd} as
\bgs{\I_{s_k}[F_k](0)&=2\int_{B'_r}\mathcal G(s_k,\tilde{u}_k,y') \frac{dy'}{|y'|^{n-1+s_k}}+\int_{\R^n}\frac{\chi_{\C F_k}(y)-\chi_{F_k}(y)}{|y|^{n+s_k}}\chi_{\C Q_{r,h}}(y)\,dy.} 
Also, with this notation and by formula \eqref{complete_curv_formula} we have for $E$
\[\I_s[E\cap B_R](0)=2\int_{B'_r}\mathcal G(s,u,y')\frac{dy'}{|y'|^{n-1+s}}
+\int_{\R^n}\frac{\chi_{\C (E\cap B_R)}(y)-\chi_{E\cap B_R}(y)}{|y|^{n+s}}\chi_{\C Q_{r,h}}(y)\,dy.\]
%
%Notice that we can find $\sigma_0<\sigma_1\in(0,\alpha)$ such that $s_k,s\in[\sigma_0,\sigma_1]$. 
We can  suppose that $r<1$. We begin by showing that for every $y'\in B_r'\setminus\{0\}$ we have
\eqlab{\label{pwise_conv}\lim_{k\to\infty}\mathcal G(s_k,\tilde{u}_k,y')=\mathcal G(s,u,y').}
First of all, we observe that
\[|\mathcal G(s_k,\tilde{u}_k,y')-\mathcal G(s,u,y')|
\leq|\mathcal G(s_k,\tilde{u}_k,y')-\mathcal G(s,\tilde{u}_k,y')|+|\mathcal G(s,\tilde{u}_k,y')-\mathcal G(s,u,y')|.\]
Then
\bgs{|\mathcal G(s_k,\tilde{u}_k,y')-\mathcal G(s,\tilde{u}_k,y')|&
=\Big|\int_{\nabla \tilde{u}_k(0)\cdot\frac{y'}{|y'|}}^{\frac{\tilde{u}_k(y')-\tilde{u}_k(0)}{|y'|}}(g_{s_k}(t)-g_s(t))\,dt\Big|\\
&
\leq2\int_0^{+\infty}|g_{s_k}(t)-g_s(t)|\,dt.}
Notice that for every $t\in\R$
\[\lim_{k\to\infty}|g_{s_k}(t)-g_s(t)|=0,\qquad\textrm{and}\qquad |g_{s_k}(t)-g_s(t)|\leq2 g_{\sigma_0}(t),\quad\forall\,k\in\mathbb N.\]
Since $g_{\sigma_0}\in L^1(\R)$, by the Dominated Convergence Theorem we obtain that
\[\lim_{k\to\infty}|\mathcal G(s_k,\tilde{u}_k,y')-\mathcal G(s,\tilde{u}_k,y')|=0.\]
We estimate
\bgs{|\mathcal G&(s,\tilde{u}_k,y')-\mathcal G(s,u,y')|\leq\Big|G_s\Big(\frac{\tilde{u}_k(y')-\tilde{u}_k(0)}{|y'|}\Big)
-G_s\Big(\frac{u(y')-u(0)}{|y'|}\Big)\Big|\\
&
\qquad\qquad\qquad\qquad
+\Big|G_s\Big(\nabla\tilde{u}_k(0)\cdot\frac{y'}{|y'|}\Big)-G_s\Big(\nabla u(0)\cdot\frac{y'}{|y'|}\Big)\Big|\\
&
\leq\Big|\frac{\tilde{u}_k(y')-\tilde{u}_k(0)}{|y'|}-\frac{u(y')-u(0)}{|y'|}\Big|
+|\nabla\tilde{u}_k(0)-\nabla u(0)|\\
&
=\Big|\nabla(\tilde{u}_k-u)(\xi)\cdot\frac{y'}{|y'|}\Big|+|\nabla\tilde{u}_k(0)-\nabla u(0)|\\
&
\leq2\|\nabla\tilde{u}_k-\nabla u\|_{C^0(B'_r)},
}
which, by \eqref{conv_transl_graph}, tends to 0 as $k\to\infty$. This proves the pointwise convergence claimed in \eqref{pwise_conv}.\\
Therefore, for every $y'\in B'_r\setminus\{0\}$,
\[\lim_{k\to\infty}\frac{\mathcal G(s_k,\tilde{u}_k,y')}{|y'|^{n-1+s_k}}=
\frac{\mathcal G(s,u,y')}{|y'|^{n-1+s}}.\]
Thus, by \eqref{Holder_useful} we obtain that
\[\Big|\frac{\mathcal G(s_k,\tilde{u}_k,y')}{|y'|^{n-1+s_k}}\Big|
\leq\|\tilde{u}_k\|_{C^{1,\beta}(B'_r)}\frac{1}{|y'|^{n-1-(\beta-s_k)}}
\leq \frac{M}{|y'|^{n-1-(\beta-\sigma_1)}}\in L^1_{loc}(\R^{n-1}),
\]
given \eqref{convvvvv}. 
The Dominated Convergence Theorem then implies that
\eqlab{\label{first_piece_conv}
\lim_{k\to\infty}\int_{B'_r}\mathcal G(s_k,\tilde{u}_k,y')\frac{dy'}{|y'|^{n-1+s_k}}=
\int_{B'_r}\mathcal G(s,u,y')\frac{dy'}{|y'|^{n-1+s}}.
}

Now, we show that 
\eqlab{\label{second_piece} \lim_{k \to \infty}\int_{\Rn}\frac{\chi_{\C F_k}(y)-\chi_{F_k}(y)}{|y|^{n+s_k}} \chi_{\C Q_{r,h}}(y)\, dy = \int_{\Rn}\frac{\chi_{\C (E\cap B_R)}(y)-\chi_{E\cap B_R}(y)}{|y|^{n+s}}\chi_{\C Q_{r,h}}(y) \, dy.}  
For this, we observe that
\bgs{\Big|\int_{\C Q_{r,h}}&(\chi_{\C (E\cap B_R)}(y)-\chi_{E\cap B_R}(y))\Big(\frac{1}{|y|^{n+s_k}}-\frac{1}{|y|^{n+s}}\Big)
dy\Big|
%\leq\int_{\C Q_{r,h}(q_k)}\Big|\frac{1}{|y-q_k|^{n+s_k}}-\frac{1}{|y-q_k|^{n+s}}\Big|dy\\
%&
%=\int_{\C Q_{r,h}}\Big|\frac{1}{|y|^{n+s_k}}-\frac{1}{|y|^{n+s}}\Big|dy
\leq\int_{\C B_\delta}\Big|\frac{1}{|y|^{n+s_k}}-\frac{1}{|y|^{n+s}}\Big|dy,
}where we have used
%\eqref{continuity_eq1} and
\eqref{continuity_eq2}  in the last inequality. 
For $y\in \C B_1$ 
\bgs{ \Big|\frac{1}{|y|^{n+s_k}}-\frac{1}{|y|^{n+s}}\Big| \leq \frac{2}{|y|^{n+\sigma_0} }\in L^1(\C B_1) }
and  for  $y\in  B_1\setminus B_\delta$ 
\bgs{ \Big|\frac{1}{|y|^{n+s_k}}-\frac{1}{|y|^{n+s}}\Big| \leq \frac{2}{|y|^{n+\sigma_1}} \in L^1(B_1\setminus B_\delta).} 
We use then the Dominated Convergence Theorem and get that
\bgs{ \lim_{k \to \infty} \int_{\C Q_{r,h}}(\chi_{\C (E\cap B_R)}(y)-\chi_{E\cap B_R}(y))\Big(\frac{1}{|y|^{n+s_k}}-\frac{1}{|y|^{n+s}}\Big)
dy =0.} 
Now
\bgs{ \bigg|\int_{\C Q_{r,h}} & \;\frac{ \chi_{\C F_k}(y)-\chi_F{_k}(y) -\left(\chi_{\C (E\cap B_R)(y) }- \chi_{E\cap B_R} (y)\right) }{|y|^{n+s_k}}\, dy\bigg| =2\int_{\C Q_{r,h}} \frac{\chi_{F_k \Delta (E\cap B_R)} (y)}{|y|^{n+s_k}}\, dy \\ \leq& \;2 \frac{ |F_k \Delta (E\cap B_R)|}{\delta^{n+\sigma_1}} \xrightarrow{k\to \infty} 0,}
according to \eqref{convvvvv1}. The last two limits prove \eqref{second_piece}. Recalling \eqref{first_piece_conv}, we obtain that
\[ \lim_{k\to \infty} \I_{s_k}[F_k](0) = \I_s[E\cap B_R](0).\] 
We have that $\I_{s_k} [F_k](0)= \I_{s_k} [E_k\cap B_R](q_k)$, so
\bgs{|\I_{s_k}[E_k](q_k)-\I_s[E](0)|\leq &\;|\I_{s_k}[E_k](q_k) - \I_{s_k}[E_k\cap B_R](q_k)|+ |\I_{s_k}[F_k](0)- \I_s[E\cap B_R](0)| \\ &\;+ |\I_s[E\cap B_R](0)- \I_s[E](0)|.}
Since
\bgs{\label{ekbounded} | \I_{s_k}[E_k](q_k) - \I_{s_k} [E_k\cap B_R] (q_k) | +| \I_{s}[E](0) - \I_{s} [E\cap B_R] (0) | \leq  \frac{4 \omega_n }{\sigma_0} R^{-\sigma_0},    }
sending $R\to \infty$
\[ \lim_{k \to \infty} \I_{s_k}[E_k](q_k)=\I_s[E](0).\]
This concludes the proof of the first part of the Theorem.

\bigskip

In order to prove the second part of Theorem \ref{everything_converges}, we fix $R>1$ and we denote
\[F_k:=\mathcal R_k\big((E_k\cap B_R)-q_k\big),\]
where $\mathcal R_k\in SO(n)$ is a rotation such that
\[\mathcal R_k:\nu_{E_k}(0)\longmapsto\nu_E(0)=-e_n\quad\mbox{ and }\quad
\lim_{k\to\infty}|\mathcal R_k-\mbox{Id}|=0.\]
Thus, by Lemma \ref{supergraph_hp_for_proof} we know that $F_k\xrightarrow{C^2}E$ in a neighborhood of $0$.\\
To be more precise,
\eqlab{\label{convvvvv2}\lim_{k\to \infty} |F_k \Delta (E\cap B_R)| =0.}
Moreover, there exist $r,h>0$ small enough and $v_k,u\in C^2(\overline B'_r)$ such that
\bgs{&F_k\cap Q_{r,h}=\{(x',x_n)\in\Rn\,|\,x'\in B'_r,\,v_k(x')<x_n<h\},\\
&
E\cap Q_{r,h}=\{(x',x_n)\in\Rn\,|\,x'\in B'_r,\,u(x')<x_n<h\}}
and that
\eqlab{\label{convvvvv3} \lim_{k\to \infty} \|v_k -u\|_{C^2(\overline B'_r)} =0.}
Notice that $0\in\partial F_k$ and $\nu_{F_k}(0)=e_n$ for every $k$, that is,
\eqlab{\label{opossum4}v_k(0)=u(0)=0,\quad\nabla v_k(0)=\nabla u(0)=0.}

We claim that
\eqlab{\label{opossum3}\lim_{k\to\infty}(1-s_k)\big|\I_{s_k}[F_k](0)-\I_{s_k}[E\cap B_R](0)\big|=0.}
By \eqref{opossum4} 
%%%%%%%%%%%%%%%%%
%%%%%%%%%%%%%%CURVATURE DIREZIONALI DOWN%%%%
%
%we can use the nonlocal directional curvatures, in particular \cite[ Definition 6, Theorem 8]{Abaty}, to write
%\bgs{\label{opossum5}
%\I_{s_k}[F_k](0)&=2\int_{\mathbb S^{n-2}}\bigg[\int_0^r\rho^{n-2}\bigg(\int_0^{v_k(\rho e)}\frac{dt}{(\rho^2+t^2)^\frac{n+s_k}{2}}\bigg)d\rho\bigg]d\mathcal H^{n-2}_e
%+\int_{\C Q_{r,h}}\frac{\chi_{\C F_k}(y)-\chi_{F_k}(y)}{|y|^{n+s_k}}\,dy\\
%&
%=:2\int_{\mathbb S^{n-2}}\overline K_{s_k,e}[F_k]d\mathcal H^{n-2}_e
%+\int_{\C Q_{r,h}}\frac{\chi_{\C F_k}(y)-\chi_{F_k}(y)}{|y|^{n+s_k}}\,dy
%}
%%%%%%%%%%PER RECUPERARE CURVATURE DIREZIONALI UP%%%%
and formula \eqref{complete_curv_formula} we have that
\bgs{\label{opossum5}
\I_{s_k}[F_k](0)&=2\int_{B'_r}\frac{dy'}{|y'|^{n+s_k-1}} \int_0^{\frac{v_k(y')}{|y'|} }\frac{dt}{(1+t^2)^{\frac{n+s_k}2}}
+\int_{\C Q_{r,h}}\frac{\chi_{\C F_k}(y)-\chi_{F_k}(y)}{|y|^{n+s_k}}\,dy\\
&= \I^{loc}_{s_k}[F_k](0) +\int_{\C Q_{r,h}}\frac{\chi_{\C F_k}(y)-\chi_{F_k}(y)}{|y|^{n+s_k}}\,dy.}
We use the same formula for $E\cap B_R$ and prove at first that
\bgs{\bigg|\int_{\C Q_{r,h}}\frac{\chi_{\C F_k}(y)-\chi_{F_k}(y)-\chi_{\C(E\cap B_R)}(y)+\chi_{E\cap B_R}(y)}{|y|^{n+s_k}}\,dy\bigg|\le\frac{|F_k\Delta(E\cap B_R)|}{\delta^{n+s_k}}\le\frac{|F_k\Delta(E\cap B_R)|}{\delta^{n+1}},
}
(where we have used \eqref{continuity_eq2}), which tends to 0 as $k\to\infty$, by \eqref{convvvvv2}.

Moreover, notice that by the Mean Value Theorem and \eqref{opossum4} we have
\[|(v_k-u)(y')|\le\frac{1}{2}|D^2(v_k-u)(\xi')||y'|^2\le\frac{\|v_k-u\|_{C^2(\overline B'_r)}}{2}|y'|^2.\]
Thus
\bgs{ &\big| \I^{loc}_{s_k}[F_k](0)  - \I^{loc}_{s_k}[E\cap B_R](0) |
\le2\int_{B'_r} \frac{dy'}{|y'|^{n+s_k-1}} \bigg|\int_{\frac{u(y')}{|y'|}}^{\frac{v_k(y')}{|y'|}}\frac{dt}{(1+t^2)^\frac{n+s_k}{2}}\bigg|\\
&
\le2\int_{B'_r}|y'|^{-n-s_k}|(v_k-u)(y')|\,d y'
\le\frac{\omega_{n-1}\,\|v_k-u\|_{C^2(\overline B'_r)}}{1-s_k}r^{1-s_k},
}
hence by \eqref{convvvvv3} we obtain
\eqlab{\label{opossum7}
\lim_{k\to\infty}(1-s_k)\big|\I^{loc}_{s_k}[F_k](0)  - \I^{loc}_{s_k}[E\cap B_R](0)|=0.
}
This concludes the proof of claim \eqref{opossum3}.

Now we use the triangle inequality and have that
\bgs{
\big|(1-s_k)\I_{s_k}&[E_k](q_k)-H[E](0)\big|\le(1-s_k)\big|\I_{s_k}[E_k](q_k)-\I_{s_k}[F_k](0)\big|\\
&
+(1-s_k)\big|\I_{s_k}[F_k](0)-\I_{s_k}[E\cap B_R](0)\big|
+\big|(1-s_k)\I_{s_k}[E\cap B_R](0)-H[E](0)\big|.
}
The last term in the right hand side converges by Theorem 12 in \cite{Abaty}. As for the first term,
notice that
\[\I_{s_k}[F_k](0)=\I_{s_k}[E_k\cap B_R](q_k),\]
hence
\[\lim_{k\to\infty}(1-s_k)\big|\I_{s_k}[E_k\cap B_R](q_k)-\I_{s_k}[E_k](q_k)\big|\le\limsup_{k\to\infty}(1-s_k)\frac{2\omega_n}{s_k}R^{-s_k}=0.\]
Sending $k\to\infty$ in the triangle inequality above, we conclude the proof of the second part of Theorem \ref{everything_converges}.
\end{proof}

\begin{remark}
In relation to the second part of the proof, we point out that using the directional fractional mean curvature defined in \cite[ Definition 6, Theorem 8]{Abaty}, we can write
\bgs{ \I^{loc}_{s_k}[F_k](0)= &\;2\int_{\mathbb S^{n-2}}\bigg[\int_0^r\rho^{n-2}\bigg(\int_0^{v_k(\rho e)}\frac{dt}{(\rho^2+t^2)^\frac{n+s_k}{2}}\bigg)d\rho\bigg]d\mathcal H^{n-2}_e \\
=&\;2\int_{\mathbb S^{n-2}} \overline K_{s_k,e} d\mathcal H^{n-2}_e.}
One is then actually able to prove that
\bgs{\lim_{k\to\infty}(1-s_k)\overline K_{s_k,e}[E_k-q_k](0)=H_e[E](0),}
uniformly in $e\in\mathbb S^{n-2}$, by using formula \eqref{opossum7} and the first claim of Theorem 12 in \cite{Abaty}.
\end{remark}

\bigskip
We prove now the continuity of the fractional mean curvature as $s\to 0$.

\begin{proof}[Proof of Proposition \ref{propsto0}]
Up to a translation, we can take $q=0$ and $u(0)=0$. \\
%We can find $r,h>0$ small enough 
%and $u\in C^{1,\alpha}(\overline B'_{2r})$ such that we can write $E\cap Q_{2r,2h}$ as in \eqref{opossum1}. 
%Suppose (for $k$ large enough) that 
%\[ q'_k\in B'_{r/2},	 \qquad  q_{k,n} =|u(q'_k)|< \frac{h}8.\]
%Let also 
%\[ \tilde u_k(x'):= u_k(x'+q_k'),\]
%% \quad  \mbox{and} \quad  \tau_k u(x'):= u(x'+q_k'), \quad  x'\in \overline B_r'\] 
% and for $k$ large enough we have \eqref{graphs_for_the_proof_eq}.
%% Moreover we can write
%%\[ \left(E-q_k\right)\cap  Q_{r,h} = \{ (x',x_n) \in \Rn \; \big| \; x'\in B'_r, 	\; \tau_k u (x')<x_n<h\} .\] 
% Notice that 
%\[ \|\tilde u_k\|_{C^{1,\alpha}(\overline B'_r) }\leq \|u_k\|_{C^{1,\alpha}(\overline B'_{2r})}<C_2 .\]
%%\quad \mbox{ and that } \quad
%%% \|\tau_k u\|_{C^{1,\alpha}(\overline B'_r) }\leq \|u\|_{C^{1,\alpha}(\overline B'_{2r})}.\]
%At first,  by Theorem \ref{asympts} we have that
%\eqlab{\label{sto01}\lim_{k\to \infty} |s_k \I_{s_k}[E](0) -\left(\omega_n-2\alpha(E)\right)|=0.}
For $R>2\max\{r,h\}$, we write
\bgs{  \I_{s_k}[E_k](q_k) =&\;
P.V.\int_{ Q_{r,h}(q_k)} \frac{\chi_{\C E_k}(y) -\chi_{E_k}(y)}{|y-q_k|^{n+s_k}}\, dy
+\int_{\C Q_{r,h}(q_k)} \frac{\chi_{\C E_k}(y)-\chi_{E_k}(y)}{|y-q_k|^{n+s_k}}\, dy 
\\
=&\; P.V.\int_{ Q_{r,h}(q_k)} \frac{\chi_{\C E_k}(y)-\chi_{E_k}(y)}{|y-q_k|^{n+s_k}}\, dy  + \int_{ B_R(q_k)\setminus Q_{r,h}(q_k)}\frac{\chi_{\C E_k}(y)-\chi_{E_k}(y)}{|y-q_k|^{n+s_k}}\, dy
\\
&\; +\int_{\C B_R(q_k)} \frac{\chi_{\C E_k}(y)-\chi_{E_k}(y)}{|y-q_k|^{n+s_k}}\, dy  
 \\
=&\;I_1(k)+I_2(k)+I_3(k).}
Now  using  \eqref{complete_curv_formula}, \eqref{mathcalg} and \eqref{Holder_useful}  we have that 
\bgs{|I_1(k)|\leq &\;  2 \int_{B'_r(q_k')} \frac{| \mathcal G(s_k,u_k,q_k', y')|}{|y'-q_k'|^{n+s_k-1}} \,dy'
 \leq 2  \| u_k \|_{C^{1,\alpha}(\overline{B}'_r(q_k'))}\int_{B'_r(q_k')} \frac{  |y'-q_k'|^{\alpha}}{|y'-q_k'|^{n+s_k-1}} \,dy'
 \\
 \leq  &\; 2 C_2 \omega_{n-1} \frac{r^{\alpha-s_k}}{\alpha-s_k}.
}
Using \eqref{continuity_eq2} we also have that
\[|I_2(k)| \leq  \int_{ B_R(q_k)\setminus B_\delta(q_k)}\frac{dy}{|y-q_k|^{n+s_k}} 
= \omega_n\frac{\delta^{-s_k}-R^{-s_k}}{s_k}. \]
Thus
\eqlab{\label{kangaroo}
\lim_{k\to\infty}s_k\big(|I_1(k)|+|I_2(k)|\big)=0.}
Furthermore
\bgs{ \big | s_k I_3(k)- &\big(\omega_n-2 s_k \alpha_{s_k}(0,R,E) \big)\big|
\\
\leq &\; \bigg|s_k \int_{\C B_R(q_k)} \frac{dy}{|y-q_k|^{n+s_k}}  -2 s_k \int_{\C B_R(q_k)} \frac{\chi_{E_k}(y)}{|y-q_k|^{n+s_k}}\, dy   - \omega_n+2s_k \alpha_{s_k}(q_k,R,E))\bigg|
\\
&\;+2s_k |\alpha_{s_k}(q_k,R,E)- \alpha_{s_k}(0,R,E)|
\\
\leq& \; |\omega_n R^{-s_k}-\omega_n| + 2 s_k\bigg| \int_{\C B_R(q_k)}\frac{\chi_{E_k}(y)}{|y-q_k|^{n+s_k}}\, dy  - \int_{\C B_R(q_k)} \frac{\chi_{E}(y)}{|y-q_k|^{n+s_k}}\, dy \bigg| 
\\
&\;+ 2s_k |\alpha_{s_k}(q_k,R,E)- \alpha_{s_k}(0,R,E)|
\\
\leq &\; |\omega_n R^{-s_k}-\omega_n| + 2 s_k \int_{\C B_R(q_k)}\frac{\chi_{E_k\Delta E}(y)}{|y-q_k|^{n+s_k}}\, dy  
+ 2s_k| \alpha_{s_k}(q_k,R,E)- \alpha_{s_k}(0,R,E)|
\\
\leq &\; |\omega_n R^{-s_k}-\omega_n| + 2 C_1 s_k R^{-n-s_k} +2s_k | \alpha_{s_k}(q_k,R,E)- \alpha_{s_k}(0,R,E)|,
 }
 where we have used that $|E_k\Delta E|<C_1$.
 
 Therefore, since $q_k\in B_d$ for every $k$, as a consequence of Proposition \ref{unifrq} it follows that
\eqlab{\label{kangaroo1}\lim_{k\to \infty}\big | s_k I_3(k)- &\big(\omega_n-2 s_k \alpha_{s_k}(0,R,E) \big)\big|=0.}

%Multiplying $I_1(k), I_2(k)$ by $s_k$ and summing up, using the triangle inequality ans sending $k\ to \infty$ w
Hence, by \eqref{kangaroo} and \eqref{kangaroo1}, we get that
\bgs{ \lim_{k \to \infty} s_k \I_{s_k}[E_k](q_k)=\omega_n-2\lim_{k\to\infty}s_k\alpha_{s_k}(0,R,E)=\omega_n-2\alpha(E),}
concluding the proof.
\end{proof} 

\begin{proof}[Proof of Theorem \ref{asympts}]
Arguing as in the proof of Proposition \ref{propsto0}, by keeping fixed $E_k=E$ and $q_k=p$, we obtain
\bgs{ \liminf_{s\to0} s\, \I_s[E](p)=\omega_n-2\limsup_{s\to0}s\,\alpha_s(0,R,E)=\omega_n-2\overline{\alpha}(E),}
and similarly for the limsup.
\end{proof}

As a corollary of Theorem \ref{everything_converges} and Theorem \ref{asympts}, we have the following result.
\begin{theorem}\label{changeyoursign}
Let $E\subset\R^n$ and let $p\in\partial E$ be such that $\partial E\cap B_r(p)$ is $C^2$ for some $r>0$.
Suppose that the classical mean curvature of $E$ in $p$ is $H(p)<0$. Also assume that 
\[\overline \alpha(E) < \frac{\omega_n}2.\] Then there exist $\sigma_0<\tilde{s}<\sigma_1$ in $(0,1)$ such that

$(i)\quad\I_s[E](p)>0$ for every $s\in(0,\sigma_0]$, and actually
\[\liminf_{s\to0^+}s \;\I_s[E](p)=\omega_n- 2\overline \alpha(E),\]

$(ii)\quad\I_{\tilde{s}}[E](p)=0,$

$(iii)\quad\I_s[E](p)<0$ for every $s\in[\sigma_1,1)$,
and actually
\[ \lim_{s\to 1} (1-s)\;\I_s[E](p)= \omega_{n-1}H(p).\]

\end{theorem}

\subsection{Some useful results}\label{appendicite} \quad { } \quad

\medskip
\noindent \textbf{Sliding the balls.}
For the convenience of the reader, we collect here some auxiliary and elementary
results of geometric nature, that are used in the proofs of the main results.

\begin{lemma}\label{slidetheballs}
%Let $\mathcal O\subset\Rn$ be an open set and let $p,\,q\in\mathcal O$.
%Suppose that there exists a continuous curve $c:[0,1]\longrightarrow\Rn$
%connecting $p$ to $q$ inside $\mathcal O$, that is
%\[c\big([0,1])\subset\mathcal O,\qquad c(0)=p\qquad\textrm{and}\qquad c(1)=q.\]
%Further assume that $\delta>0$ is such that
%\[\overline{\bigcup_{t\in[0,1]}B_\delta\big(c(t)\big)}\subset\overline{\mathcal O}.\]
Let $F\subset\Rn$ be such that
\[B_\delta(p)\subset F_{ext}\quad\textrm{for some }\delta>0\qquad\textrm{and}\qquad q\in\overline{F},\]
and let $c:[0,1]\longrightarrow\Rn$ be a continuous curve connecting $p$ to $q$, that is
\[c(0)=p\qquad\textrm{and}\qquad c(1)=q.\]
Then there exists $t_0\in[0,1)$ such that $B_\delta\big(c(t_0)\big)$ is an exterior tangent ball to $F$,
that is
\eqlab{\label{slide_ext_tg}
B_\delta\big(c(t_0)\big)\subset F_{ext}\qquad\textrm{and}\qquad\partial B_\delta\big(c(t_0)\big)\cap\partial F\not=\emptyset.}
\end{lemma}

\begin{proof}
Define
\eqlab{\label{slide_t}t_0:=\sup\Big\{\tau\in[0,1]\,\big|\,\bigcup_{t\in[0,\tau]}B_\delta\big(c(t)\big)\subset F_{ext}\Big\}.}
Notice that
\[q\in\overline{F}=F_{int}\cup\partial F\quad\Longrightarrow\quad B_\delta(q)\cap F_{int}\not =\emptyset,\]
hence we have that $t_0<1$.

Now we prove that $t_0$ as defined in \eqref{slide_t} satisfies \eqref{slide_ext_tg}.

Notice that by definition of $t_0$
\[B_\delta\big(c(t_0)\big)\subset F_{ext},\]
hence
\eqlab{\label{slide_pf}\overline{B_\delta\big(c(t_0)\big)}\subset \overline{F_{ext}}=F_{ext}\cup\partial F.}

Now, suppose that
\[\partial B_\delta\big(c(t_0)\big)\cap\partial F=\emptyset.\]
Then \eqref{slide_pf} implies that
\[\overline{B_\delta\big(c(t_0)\big)}\subset F_{ext}.\]
Since $F_{ext}$ is an open set, we can find $\tilde\delta>\delta$ such that
\[B_{\tilde\delta}\big(c(t_0)\big)\subset F_{ext}.\]
By continuity of $c$ we can find $\eps\in(0,1-t_0)$ small enough such that
\[|c(t)-c(t_0)|<\tilde\delta-\delta,\qquad\forall\,t\in[t_0,t_0+\eps].\]
Therefore
\[B_\delta\big(c(t)\big)\subset B_{\tilde\delta}\big(c(t_0)\big)\subset F_{ext},\qquad\forall\,t\in[t_0,t_0+\eps],\]
and hence
\[\bigcup_{t\in[0,t_0+\eps]}B_\delta\big(c(t)\big)\subset F_{ext},\]
which is in contradiction with
the definition of $t_0$. Thus
\[\partial B_\delta\big(c(t_0)\big)\cap\partial F\not=\emptyset,\]
which concludes the proof.
\end{proof}

\noindent \textbf{Smooth domains.} Given a set $F\subset\R^n$, the signed distance function $\bar{d}_F$ from $\partial F$, negative inside $F$, is defined as
\begin{equation*}
\bar{d}_F(x)=d(x,F)-d(x,\C F)\qquad\mbox{for every }x\in\R^n,
\end{equation*}
where
\[d(x,A):=\inf_{y\in A}|x-y|,\]
denotes the usual distance from a set $A$. Given an open set $\Omega\subset\R^n$, we denote by
\begin{equation*}
N_\rho(\partial\Omega):=\{|\bar{d}_\Omega|<\rho\}=\{x\in\R^n\,|\,d(x,\partial\Omega)<\rho\}
\end{equation*}
the tubular $\rho$-neighborhood of $\partial\Omega$.
For the details about the properties of the
signed distance function, we refer to \cite{trudy, Ambrosio} and the references cited therein.

Now we recall the notion of (uniform) interior ball condition.
\begin{defn}
We say that an open set $\mathcal O$ satisfies an interior ball condition at $x\in\partial\mathcal O$ if
there exists a ball $B_r(y)$ s.t.
\begin{equation*}
B_r(y)\subset\mathcal O\qquad\textrm{and}\qquad x\in\partial B_r(y).
\end{equation*}
We say that the condition is ``strict'' if $x$ is the only tangency point, i.e.
\[\partial B_r(y)\cap\partial\mathcal O=\{x\}.\]
The open set $\mathcal O$ satisfies a uniform (strict) interior ball condition of radius $r$ if
it satisfies the (strict) interior ball condition at every point of $\partial\mathcal O$,
with an interior tangent ball of radius at least $r$.\\
In a similar way one defines exterior ball conditions.
\end{defn}
We remark that
if $\mathcal O$ satisfies an interior ball condition of radius $r$ at $x\in\partial\mathcal O$,
then the condition is strict for every radius $r'<r$.

\begin{remark}\label{ext_unif_omega}
Let $\Omega\subset\R^n$ be a bounded open set with $C^2$ boundary. It is well
known that $\Omega$ satisfies a uniform interior and exterior ball condition. We fix $r_0=r_0(\Omega)>0$
such that $\Omega$ satisfies a strict interior and a strict exterior ball contition of radius $2r_0$
at every point $x\in\partial\Omega$.
Then
\begin{equation}\label{r01}
\bar{d}_\Omega\in C^2(N_{2r_0}(\partial\Omega)),
\end{equation}
(see e.g. Lemma 14.16 in \cite{trudy}).
\end{remark}

We remark that the distance function $d(-,E)$ is differentiable at $x\in\R^n\setminus\overline E$ if
and only if there is a unique point $y\in\partial E$ of minimum distance, i.e.
\[d(x,E)=|x-y|.\]
In this case, the two points $x$ and $y$ are related by the formula
\[y=x-d(x,E)\nabla d(x,E).\]%\quad\textrm{ and }\quad x=y+d(x,E)\nu_\Omega(y),\]
%where $\nu_\Omega$ denotes the exterior unit normal.

This generalizes to the signed distance function. In particular, if $\Omega$ is bounded and has $C^2$ boundary, then we can
define a $C^1$ projection function from the tubular $2r_0$-neighborhood $N_{2r_0}(\partial\Omega)$ onto $\partial\Omega$ by
assigning to a point $x$ its unique nearest point $\pi(x)$, that is
\[\pi:N_{2r_0}(\partial\Omega)\longrightarrow\partial\Omega,\qquad\pi(x):=x-\bar{d}_\Omega(x)\nabla\bar{d}_\Omega(x).\]
We also remark that
on $\partial\Omega$ we have that $\nabla\bar{d}_\Omega=\nu_\Omega$ and that
\[\nabla\bar{d}_\Omega(x)=\nabla\bar{d}_\Omega(\pi(x))=\nu_\Omega(\pi(x)),\qquad\forall x\in N_{2r_0}(\partial\Omega).\]
Thus $\nabla\bar{d}_\Omega$ is a vector field which extends
the outer unit normal to a tubular neighborhood of $\partial\Omega$, in a $C^2$ way.

Notice that given a point $y\in\partial\Omega$, for every $|\delta|<2r_0$ the point $x:=y+\delta\nu_\Omega(y)$ is such that $\bar{d}_\Omega(x)=\delta$ (and $y$ is its unique nearest point).
%That is, roughly speaking, if $\delta$ is positive (negative) we are moving in normal direction, towards the exterior (interior) of $\Omega$, starting from $y$.
Indeed, we consider for example $\delta\in(0,2r_0)$. Then we can find an exterior tangent ball
\[B_{2r_0}(z)\subset\C\Omega,\qquad\partial B_{2r_0}(z)\cap\partial\Omega=\{y\}.\]
Notice that the center of the ball must be
\[z=y+2r_0\nu_\Omega(y).\]
Then, for every $\delta\in(0,2r_0)$ we have
\[B_\delta(y+\delta\nu_\Omega(y))\subset B_{2r_0}(y+2r_0\nu_\Omega(y))\subset
\C\Omega,\qquad\partial B_\delta(y+\delta\nu_\Omega(y))\cap\partial\Omega=\{y\}.\]
This proves that
\[|\bar{d}_\Omega(y+\delta\nu_\Omega(y))|=d(x,\partial\Omega)=\delta.\]
Finally, since the point $x$ lies outside $\Omega$, its signed distance function is positive.

\begin{remark}\label{c21}
Since $|\nabla\bar{d}_\Omega|=1$, the bounded open sets
\begin{equation*}
\Omega_\delta:=\{\bar{d}_\Omega<\delta\}
\end{equation*}
have $C^2$ boundary
\begin{equation*}
\partial\Omega_\delta=\{\bar{d}_\Omega=\delta\},
\end{equation*}
for every $\delta\in(-2r_0,2r_0)$.
\end{remark}

As a consequence, we know that for every $|\delta|<2r_0$ the set $\Omega_\delta$
satisfies a uniform interior and exterior ball condition of radius $r(\delta)>0$.
Moreover, we have that $r(\delta)\geq r_0$ for every $|\delta|\leq r_0$
(see also Appendix A in \cite{MR3436398}
for related results).

\begin{lemma}\label{geomlem}
Let $\Omega\subset\R^n$ be a bounded open set with $C^2$ boundary.
Then for every $\delta\in[-r_0,r_0]$ the set $\Omega_\delta$ 
satisfies a uniform interior and exterior ball condition of radius at least $r_0$, i.e.
\begin{equation*}
r(\delta)\geq r_0\qquad\textrm{for every }|\delta|\leq r_0.
\end{equation*}
\end{lemma}
\begin{proof}
%It is enough to notice that for every $|\delta|\leq r_0$
%\begin{equation*}
%\Omega_\delta=\big(\Omega_{\delta-\frac{r_0}{2}}\big)_\frac{r_0}{2}.
%\end{equation*}
%Then, if $x\in\partial\Omega_\delta$ and $y\in\partial\Omega_{\delta-\frac{r_0}{2}}$
%is the closest point, i.e. $|x-y|=\frac{r_0}{2}$, we have
%\begin{equation*}
%B_\frac{r_0}{2}(y)\subset\Omega_\delta\qquad\textrm{and}\qquad x\in\partial B_\frac{r_0}{2}(y).
%\qedhere\end{equation*}
Take for example $\delta\in[-r_0,0)$ and let $x\in\partial\Omega_\delta=\{\bar{d}_\Omega=\delta\}$.
We show that $\Omega_\delta$ has an interior tangent ball of radius $r_0$ at $x$. The other cases are proven in a similar way.

Consider the projection $\pi(x)\in\partial\Omega$ and the point
\[x_0:=x-r_0\nabla\bar{d}_\Omega(x)=\pi(x)-(r_0+|\delta|)\nu_\Omega(\pi(x)).\]
Then
\[B_{r_0}(x_0)\subset\Omega_\delta\quad\textrm{ and }\quad x\in\partial B_{r_0}(x_0)\cap\partial\Omega_\delta.\]
Indeed, notice that, as remarked above,
\[d(x_0,\partial\Omega)=|x_0-\pi(x)|=(r_0+|\delta|).\]
Thus, by the triangle inequality we have that
\[d(z,\partial\Omega)\ge d(x_0,\partial\Omega)-|z-x_0|>|\delta|,\qquad\textrm{ for every }z\in B_{r_0}(x_0),\]
so $B_{r_0}\subset\Omega_\delta$. Moreover, by definition of $x_0$ we have
\[x\in\partial B_{r_0}(x_0)\cap\partial\Omega_\delta\]
and the desired result follows.
\end{proof}

To conclude, we remark that the sets $\overline{\Omega_{-\delta}}$ are retracts of $\Omega$, for every $\delta\in(0,r_0]$.
Indeed, roughly speaking, each set $\overline{\Omega_{-\delta}}$ is obtained by deforming $\Omega$ in normal direction,
towards the interior.
An important consequence is that if $\Omega$ is connected then $\overline{\Omega_{-\delta}}$ is path connected.

To be more precise, we have the following:

\begin{prop}\label{retract}
Let $\Omega\subset\Rn$ be a bounded open set with $C^2$ boundary.
Let $\delta\in(0,r_0]$ and define
\[\mathcal D:\Omega\longrightarrow\overline{\Omega_{-\delta}},\qquad\mathcal D(x):=
\left\{\begin{split}&x,& &x\in\Omega_{-\delta},\\
&x-\big(\delta+\bar{d}_\Omega(x)\big)\nabla\bar{d}_\Omega(x),& & x\in\Omega\setminus\Omega_{-\delta}.\end{split}\right.\]
Then $\mathcal D$ is a retraction of $\Omega$ onto $\overline{\Omega_{-\delta}}$, i.e. it is continuous and
$\mathcal D(x)=x$ for every $x\in\overline{\Omega_{-\delta}}$.
In particular, if $\Omega$ is connected, then $\overline{\Omega_{-\delta}}$ is path connected.
\end{prop}

\begin{proof}
Notice that the function
\[\Phi(x):=x-\big(\delta+\bar{d}_\Omega(x)\big)\nabla\bar{d}_\Omega(x)\]
is continuous in $\Omega\setminus\Omega_{-\delta}$ and $\Phi(x)=x$ for every $x\in\partial\Omega_{-\delta}$.
Therefore the function $\mathcal D$ is continuous.

We are left to show that
\[\mathcal D(\Omega\setminus\Omega_{-\delta})\subset\partial\Omega_{-\delta}.\]
For this, it is enough to notice that
\[\mathcal D(x)=\pi(x)-\delta\nu_\Omega(\pi(x))\qquad\textrm{for every }x\in\Omega\setminus\Omega_{-\delta}.\]
To conclude, suppose that $\Omega$ is connected and recall that if an open set $\Omega\subset\R^n$ is connected, then it is also path connected.
Thus $\overline{\Omega_{-\delta}}$,
being the continuous image of a path connected space, is itself
path connected.
\end{proof}   
%
%\noindent \textbf{Collection of other useful results on nonlocal minimal surfaces}\label{appendicite2}

\noindent \textbf{Explicit formulas for the fractional mean curvature of a graph.}
Now, we collect some auxiliary results on nonlocal minimal surfaces.
In particular, we recall the representation of
the fractional mean curvature when the set is a graph and
a useful and general version of the maximum principle.\\
We denote 
\[Q_{r,h}(x):=B'_r(x')\times(x_n-h,x_n+h),\] for $x\in\R^n,$ $r,h>0$. If $x=0$, we write $Q_{r,h}:=Q_{r,h}(0)$. Let also 
\[g_s(t):=\frac{1}{(1+t^2)^\frac{n+s}{2}}\qquad\textrm{and}\qquad G_s(t):=\int_0^tg_s(\tau)\,d\tau.\]
Notice that
\[0<g_s(t)\leq1,\quad\forall\,t\in\R\qquad\textrm{and}\qquad\int_{-\infty}^{+\infty}g_s(t)\,dt<\infty,\]
for every $s\in(0,1)$.

In this notation, we can write the fractional mean curvature of a graph as follows:

\begin{prop}
Let $F\subset\R^n$ and $p\in\partial F$ such that
\[F\cap Q _{r,h}(p)=\{(x',x_n)\in\R^n\,|\,x'\in B'_r(p'),\,v(x')<x_n<p_n+h\},\]
for some $v\in C^{1,\alpha}(B'_r(p'))$. Then for every $s\in(0,\alpha)$
\eqlab{\label{complete_curv_formula}\I_s[F](p)&%=P.V.\int_{\R^n}\frac{\chi_{\C F}(y)-\chi_F(y)}{|y-p|^{n+s}}\,dy
=2\int_{B'_r(p')}\Big\{G_s\Big(\frac{v(y')-v(p')}{|y'-p'|}\Big)
-G_s\Big(\nabla v(p')\cdot\frac{y'-p'}{|y'-p'|}\Big)\Big\}\frac{dy'}{|y'-p'|^{n-1+s}}\\
&
\qquad\qquad+\int_{\R^n\setminus Q_{r,h}(p)}\frac{\chi_{\C F}(y)-\chi_F(y)}{|y-p|^{n+s}}\,dy.}
\end{prop}

This explicit formula was introduced in \cite{regularity} (see also \cite{Abaty,lukes}) when $\nabla v(p)=0$. In \cite{bootstrap}, the reader can find the formula for the case of non-zero gradient. 

\begin{remark}
In the right hand side of \eqref{complete_curv_formula}
there is no need to consider the principal value, since the integrals are summable.
Indeed,
\bgs{\label{Holder_useful}
\Big|G_s\Big(&\frac{v(y')-v(p')}{|y'-p'|}\Big)
-G_s\Big(\nabla v(p')\cdot\frac{y'-p'}{|y'-p'|}\Big)\Big|
=\Big|\int_{\nabla v(p')\cdot\frac{y'-p'}{|y'-p'|}}^{\frac{v(y')-v(p')}{|y'-p'|}}g_s(t)\,dt\Big|\\
&
\leq\Big|\frac{v(y')-v(p')-\nabla v(p')\cdot(y'-p')}{|y'-p'|}\Big|\leq \|v\|_{C^{1,\alpha}(B'_r(p'))}|y'-p'|^\alpha,
}
for every $y'\in B'_r(p')$.
As for the last inequality, notice that by the Mean value Theorem we have
\[v(y')-v(p')=\nabla v(\xi)\cdot(y'-p'),\]
for some $\xi\in B'_r(p')$ on the segment with end points $y'$ and $p'$. Thus
\bgs{|v(y')-v(p')&-\nabla v(p')\cdot(y'-p')|=|(\nabla v(\xi)-\nabla v(p'))\cdot(y'-p')|\\
&
\leq|\nabla v(\xi)-\nabla v(p')||y'-p'|\leq\|\nabla v\|_{C^{0,\alpha}(B'_r(p'))}|\xi-p'|^\alpha|y'-p'|\\
&
\leq\|v\|_{C^{1,\alpha}(B'_r(p'))}|y'-p'|^{1+\alpha}.
}
We denote for simplicity
\eqlab{ \label{mathcalg} \mathcal G(s,v,y',p'):= G_s\Big(&\frac{v(y')-v(p')}{|y'-p'|}\Big)
-G_s\Big(\nabla v(p')\cdot\frac{y'-p'}{|y'-p'|}\Big).}
With this notation, we have
\eqlab{\label{Holder_useful} |\mathcal G(s,v,y',p')|  \leq \|v\|_{C^{1,\alpha}(B'_r(p'))}|y'-p'|^\alpha.}
\end{remark}

    \noindent \textbf{A maximum principle.}
By exploiting the Euler-Lagrange equation, we can compare an $s$-minimal set with half spaces.
We show that if $E$ is $s$-minimal in $\Omega$ and the exterior data $E_0:=E\setminus\Omega$ lies above a half-space, then
also $E\cap\Omega$ must lie above that same half-space. This is indeed
a very general principle, that we now discuss in full detail.
To this aim, it is convenient to point out that
if $E\subset F$ and the boundaries of the two sets touch at a common point $x_0$ where the $s$-fractional mean curvatures coincide, then the two sets must be equal.
The precise result goes as follows:

\begin{lemma}\label{curv_rigidity}
Let $E,F\subset\R^n$ be such that $E\subset F$ and $x_0\in\partial E\cap\partial F$. Then
\begin{equation}\label{confront_curv_ineq}
\I_s^\rho[E](x_0)\geq\I_s^\rho[F](x_0)\qquad\textrm{for every }\rho>0.
\end{equation}
Furthermore, if
\begin{equation}\label{ineq_for_curvs}
\liminf_{\rho\to0^+}\I_s^\rho[F](x_0)\geq a\quad\textrm{and}\quad\limsup_{\rho\to0^+}\I_s^\rho[E](x_0)\leq a,
%\I_s[E](x_0)=\I_s[F](x_0)
\end{equation}
then $E=F$, the fractional mean curvature is well defined in $x_0$ and $\I_s[E](x_0)=a$.

\begin{proof}
To get $(\ref{confront_curv_ineq})$ it is enough to notice that
\[
E\subset F\quad\Longrightarrow\quad\big(\chi_{\C E}(y)-\chi_E(y)\big)
\geq\big(\chi_{\C F}(y)-\chi_F(y)\big)\qquad\forall\,y\in\R^n.
\]
Now suppose that $(\ref{ineq_for_curvs})$ holds true. Then by $(\ref{confront_curv_ineq})$ we find that
\[
\exists\,\lim_{\rho\to0^+}\I_s[E](x_0)=\lim_{\rho\to0^+}\I_s[F](x_0)=a.
\]

To conclude, notice that if the two curvatures are well defined (in the principal value sense) in $x_0$ and are equal,
then
\begin{equation*}\begin{split}
0\leq\int_{\C B_\rho(x_0)}&\frac{\big(\chi_{\C E}(y)-\chi_E(y)\big)-\big(\chi_{\C F}(y)-\chi_F(y)\big)}{|x_0-y|^{n+s}}dy\\
&
=\I_s^\rho[E](x_0)-\I_s^\rho[F](x_0)\xrightarrow{\rho\to0^+}0,
\end{split}\end{equation*}
which implies that $\chi_E(y)=\chi_F(y)$ for a.e. $y\in\R^n$, i.e. $E=F$.
\end{proof}\end{lemma}

\begin{prop}\label{maximum_principle}[Maximum Principle]
Let $\Omega\subset\R^n$ be a bounded open set with $C^{1,1}$ boundary. Let $s\in(0,1)$ and let $E$ be 
$s$-minimal in $\Omega$. If
\begin{equation}\label{ext_data_incl}
\{x\cdot\nu\leq a\}\setminus\Omega\subset\C E,\end{equation}
for some $\nu\in\mathbb S^{n-1}$ and $a\in\R$, then
\[\{x\cdot\nu\leq a\}\subset \C E.\]

\begin{proof}
First of all, we remark that up to a rotation and translation, we can suppose that $\nu=e_n$ and $a=0$.
Furthermore we can assume that
\[
\inf_{x\in\overline{\Omega}}x_n<0,
\]
otherwise there is nothing to prove.

If $E\cap\Omega=\emptyset$, i.e. $\Omega\subset\C E$, we are done.
Thus we can suppose that $E\cap\Omega\not =\emptyset$.\\
Since $\overline{E}\cap\overline{\Omega}$ is compact, we 
have
\[
b:=\min_{x\in\overline{E}\cap\overline{\Omega}}x_n\in\R.
\]
Now we consider the set of points which realize the minimum
above, namely we set
$$\mathcal P:=\{p\in\overline{E}\cap\overline{\Omega}\,|\,p_n=b\}.$$
Notice that
\begin{equation}\label{confronto_first_incl}
\big\{x_n\leq\min\{b,0\}\big\}\subset\C E,
\end{equation}
so we are reduced to prove that $b\geq0$.

We argue by contradiction and suppose that $b<0$. We will prove that $\mathcal P=\emptyset$.
We remark that $\mathcal P\subset\partial E\cap\overline{\Omega}$.

Indeed, if $p\in\mathcal P$, then by $(\ref{confronto_first_incl})$ we have that
$B_\delta(p)\cap\{x_n\leq b\}\subset\C E$
for every $\delta >0$, so $|B_\delta(p)\cap\C E|\geq\frac{\omega_n}{2}\delta^n$
and $p\not\in E_1$. Therefore, since $\overline{E}=E_1\cup\partial E$, we 
find that $p\in\partial E$.

Roughly speaking, we are sliding upwards the half-space $\{x_n\leq t\}$ until we first touch the set $\overline{E}$. Then the contact points must belong to the boundary of $E$.

Notice that the points of $\mathcal P$ can be either inside $\Omega$ or on $\partial\Omega$.
In both cases we can use the Euler-Lagrange equation to get a contradiction. The precise argument goes as follows.

First, if $p=(p',b)\in\partial E\cap\Omega$, then since $H:=\{x_n\leq b\}\subset\C E$,
we can find an exterior tangent ball to $E$ at $p$ (contained in $\Omega$), so $\I_s[E](p)=0$.

On the other hand, if $p\in\partial E\cap\partial\Omega$, then
$B_{|b|}(p)\setminus\Omega\subset\C E$ and hence (by Theorem 5.1 of \cite{graph})
$\partial E\cap B_r(p)$ is $C^{1,\frac{s+1}{2}}$ for some $r\in(0,|b|)$, and $\I_s[E](p)\leq0$ by Theorem 1.1 of \cite{elsulbordo} (we remark that
sign convention here
is different than the one in \cite{elsulbordo}).

In both cases, we have that
\[
p\in\partial H\cap\partial E,\quad H\subset \C E\quad\textrm{and}\quad\I_s[\C E](p)=-\I_s[E](p)\geq0=\I_s[H](p),
\]
and hence Lemma \ref{curv_rigidity} implies $\C E=H$. However, since $b<0$, this contradicts
$(\ref{ext_data_incl})$.

This proves that $b\geq0$, thus concluding the proof.
\end{proof}
\end{prop}

{F}rom this, we obtain a strong comparison principle with planes,
as follows:

\begin{corollary}
Let $\Omega\subset\R^n$ be a bounded open set with $C^{1,1}$ boundary.
Let $E\subset\R^n$ be $s$-minimal in $\Omega$, with
$\{x_n\leq0\}\setminus\Omega\subset\C E$.
Then

$(i)\quad$ if $|(\C E\setminus\Omega)\cap\{x_n>0\})|=0$, then $E=\{x_n>0\}$;

$(ii)\quad$ if $|(\C E\setminus\Omega)\cap\{x_n>0\}|>0$,
then for every $x=(x',0)\in\Omega\cap\{x_n=0\}$ there exists $\delta_x\in(0,d(x,\partial\Omega))$ s.t.
$B_{\delta_x}(x)\subset\C E$. Thus
\begin{equation}
\{x_n\leq0\}\cup\bigcup_{(x',0)\in\Omega}B_{\delta_x}(x)\subset\C E.
\end{equation}

\begin{proof}
First of all, Proposition \ref{maximum_principle} guarantees that
\begin{equation*}
\{x_n\leq0\}\subset\C E.
\end{equation*}

$(i)\quad$ Notice that since $E$ is $s$-minimal in $\Omega$, also $\C E$ is $s$-minimal in $\Omega$.\\
Thus, since $\{x_n>0\}\setminus\Omega\subset E=\C(\C E)$, we can use again Proposition \ref{maximum_principle}
(notice that $\{x_n=0\}$ is a set of measure zero) to
get $\{x_n>0\}\subset E$, proving the claim.

$(ii)\quad$ Let $x\in\{x_n=0\}\cap \Omega$.

We argue by contradiction. Suppose that $|B_\delta(x)\cap E|>0$ for every $\delta>0$.
Notice that,
since $B_\delta(x)\cap\{x_n\leq0\}\subset\C E$ for every $\delta>0$, this implies that $x\in\partial E\cap \Omega$.
Moreover, we can find an exterior tangent ball to $E$ in $x$, namely
\begin{equation*}
B_\eps(x-\eps\,e_n)\subset\{x_n\leq0\}\cap\Omega\subset\C E\cap\Omega.
\end{equation*}
Thus the Euler-Lagrange equation gives $\I_s[E](x)=0$.

Let $H:=\{x_n\leq0\}$.
Since $x\in\partial H$, $H\subset\C E$ and also $\I_s[H](x)=0$,
Lemma \ref{curv_rigidity} implies $\C E=H$.
However this contradicts the hypothesis
\begin{equation*}
|(\C E\setminus\Omega)\cap\{x_n>0\}|>0,
\end{equation*}
which completes the proof.
\end{proof}\end{corollary}
\appendix
\chapter{Appendix A}
\section{The Fourier transform}\label{Four}
\noindent We consider the  Schwartz space of rapidly decaying functions defined as 
	\eqlab{\label{schsp}\mathcal{S} ({\Rn}): = \left \{  f \in C^\infty(\Rn)\;  \big|   \; \forall \alpha, \, \beta \in \mathbf{N}^n_0, \,  \sup_{x\in {\Rn}} |x^{\alpha} \partial_{\beta} f(x)|<  \infty \right \}.  }
	In other words, the Schwartz space consists of smooth functions whose derivatives (including the function itself) decay at infinity faster than any power of $x$. Endowed with the family of seminorms
	\begin{equation} \label{seminormss} [f]_{\Sa(\Rn)}^{\alpha,N} =\sup_{x\in {\Rn}} (1+|x|)^N  \sum_{|\alpha| \leq N}| D^{\alpha} f(x)|, \end{equation}
	the Schwartz space is a locally convex topological space. We denote by $\Sa'(\Rn)$ the space of tempered distributions, the topological dual of $\Sa(\Rn)$.
	
Using $x \in \Rn$ as the  space variable and $\xi \in \Rn$ as the frequency variable, the Fourier transform
and the inverse Fourier transform of  $f \in L^1(\Rn)$, are defined respectively, as
		\begin{equation*} 
		\widehat f(\xi) := \mathcal{F} f(\xi):= 
\int_{{\Rn}} f(x) e^{-  2\pi i x \cdot \xi} \, dx
		\label{transF} 
	\end{equation*}
and
	\begin{equation*}	\label{invF}
		\wck f(x)  =\mathcal{F}^{-1} f(x)  = \int_{\Rn} f (\xi) e^{2\pi i x \cdot \xi} \, d\xi.
	\end{equation*} 
	We recall that the pointwise product is taken into the convolution product and vice versa, namely for all $ f,\,g \in L^1(\Rn)$
	\eqlab{\label{Fconv}  \mathcal{F} (f*g)=  \mathcal{F}(f) \, \mathcal{F}(g) . }
	We have that  $f(x)= \F \big(\F^{-1} f)(x) =  \F^{-1} \big(\F f)(x)$ holds almost everywhere if both $f$ and $\widehat f\in L^1(\Rn)$, and pointwise if $f$ is also continuous.
Also for all $ f,\,g \in L^1(\Rn)$
	\[  \int_{\Rn} \widehat f(\xi)  g(\xi) \, d\xi = \int_{\Rn}f(\xi) \widehat g (\xi) \, d\xi  . \]
	On the Schwartz space, the Fourier transform gives a continuous bijection between $\mathcal{S}(\Rn)$  and $\mathcal{S}(\Rn)$.

\section{Special functions}\label{special}
We recall here a few notions on the special functions Gamma, Beta and hypergeometric (see \cite{ABRAMOWITZ}, Chapters 6 and 15 for details).\\

\noindent\textbf{Gamma function.} The Gamma function is defined  for $x>0$ as (see \cite{ABRAMOWITZ}, Chapter 6):
\begin{equation} \Gamma(x):=\int_0^{\infty}t^{x-1}e^{-t}\, dt.\label{ABRAMOWITZ}\end{equation}
This function has an unique continuation to the whole $\mathbb{R}$ except at the negative integers, by means of Euler's infinite product.
We have that $\Gamma(1)=\Gamma(2)=1$ and $\Gamma({1}/{2})= \sqrt{\pi}$. We  also recall the next useful identities:
 \begin{align}
			&\Gamma(n+1)= n! &   \mbox{ for any } &n \in \N,\label{gammafac}\\
			&\Gamma(x+1)=x\Gamma(x) &   \mbox{ for any } &x >0,   \label{gamxx1}\\
			&\frac{\Gamma(1/2+x)}{\Gamma(2x)} =  \frac{\sqrt{\pi} 2^{1-2x} }{\Gamma(x)}&   \mbox{ for any } &x >0 , \label{gam2}\\
			 	&\Gamma(s)\Gamma(1-s) \;\;=  \frac{\pi}{\sin (\pi s)}&   \mbox{ for } &s\in (0,1),   \label{gam3}\\
			& \Gamma(1/2-s) \Gamma(1/2+s) =   \frac{\pi}{\cos(\pi s)}&   \mbox{ for } &s\in (0,1),   \label{gam1}\\
	 	 &  \Gamma(1-s) =  (-s)\Gamma(-s) &   \mbox{ for } &s\in (0,1)\label{gam4}.
	 \end{align}

\noindent\textbf{Beta function.}
The Beta function can be represented as an integral (see \cite{ABRAMOWITZ}, Section 6.2), namely for $x, y>0$
	 \begin{equation}
			 \beta(x,y) =  \int_0^{\infty} \frac {t^{x-1}}{(1+t)^{x+y}} \, dt \label{beta}
	\end{equation}
and equivalently
	\begin{equation}
 	\beta(x,y)=  \int_0^1 t^{x-1}(1-t)^{y-1} \, dt \label{betazerouno}.
	\end{equation}
Furthermore, we have the identity
	\begin{equation}
\beta(x,y)= \frac {\Gamma(x)\Gamma(y)} {\Gamma(x+y)} \label{betagamma}.
	\end{equation}
	In particular, using this and \eqref{gam3}, we get the useful formula
	\eqlab{\label{betas} \beta(s,1-s)= \frac{\pi}{\sin(\pi s)}.}

\noindent\textbf{Hypergeometric functions.} There are several representations for the hypergeometric function (see \cite{ABRAMOWITZ}, Chapter 15, or page~211 in~\cite{Oberhettinger}). We recall the ones useful for our own purposes. \\

\noindent \emph{(1)  Gauss series}
		\begin{equation}  F(a,b,c,w) = \sum_{k=0}^{\infty} \frac{ (a)_k (b)_k }{(c)_k} \frac{w^k} { k!},\label{gausshyp} \end{equation}
	where $(q)_k$ is the Pochhammer symbol defined by:
		\begin{equation}\label{Pochh}
		(q)_k = \begin{cases}
			1   &\mbox{ for } k = 0 ,\\
  		q(q+1) \cdots (q+k-1) &\mbox{ for } k > 0.
		 \end{cases}
		\end{equation} \\
The interval of convergence of the series is $|w|\leq 1$. The Gauss series, on its interval of convergence, diverges when $c-a-b\leq -1$, is absolutely convergent when $c-a-b>0$ and is conditionally convergent when $|w|<1$ and $-1<c-a-b\leq 0$.	
Also, the series is not defined when $c$ is a negative integer $-m$, provided $a$ or $b$ is a positive integer $n$ and $n<m$.

Some useful elementary computations are
	\begin{subequations}	
		 \begin{align}
		&F(a,b,b,w)=(1-w)^{-a}.\label{hypelc1}\\
		&F\Big(a, \frac{1}{2}+a, \frac{1}{2},w^2\Big) = \frac{ (1+w)^{-2a} +(1-w)^{-2a}}{2}.\label{hypelc2}
		\end{align}
	\end{subequations}

\noindent\emph{(2)  Integral representation}
	\begin{equation}
	 F(a,b,c,w) :=  \frac{\Gamma(c)}{\Gamma(b) \Gamma(c-b)} \int_0^1 t^{b -1} (1-t)^{c-b-1} (1- w t)^{-a}  \, dt .\label{inthyp}
	\end{equation}
The integral is convergent (thus $F$ is defined as an integral) when $c>b>0$ and $|w|<1$. \\

\noindent\emph{ (3) Linear transformation formulas}

From the integral representation \eqref{inthyp}, the following transformations can be deduced.
	\begin{subequations}
		 \begin{align}	
		F(a,b,c,w) =&\;  (1-w)^{c-a-b} F(c-a,c-b,c,w), \label{hyp1}\\
			 =&\;(1-w)^{-a} F\Big(a,c-b,c, \frac{w}{w-1}\Big),\label{hyp2} \\
				 =&\;(1-w)^{-b} F\Big(b,c-a,c,\frac{w}{w-1}\Big),\label{hyp3}\\
				= &\;\frac{\Gamma(c)\Gamma(c-a-b)}{\Gamma(c-a)\Gamma(c-b)} F(a,b,a+b-c+1,1-w) \notag \\&\;+ (1-w)^{c-a-b} \frac{\Gamma(c)\Gamma(a+b-c)}{\Gamma(a)\Gamma(b)}F(c-a,c-b,c-a-b+1, 1-w), \notag\\ &\text{ when } 0<w<1\label{hyp4} .
		\end{align}
	\end{subequations}

\bibliography{biblio}
\bibliographystyle{plain}
\section*{Thank you note}

\begin{flushright}
\textit{Not all those who wander are lost\footnote{This verse is from a poem written by J. R. R. Tolkien for his fantasy novel \textit{The Lord of the Rings}.  I didn't know it at the time I first read it. I was in Edinburgh when
%, I had drank a sip of the best scotch I have ever had, and 
I saw these white words  on the side of a building,  in a dark Scottish winter night, and I thought exactly of this thesis and of this  \textit{Thank you} note. }.}
\end{flushright}
\bigskip
\bigskip
\noindent I have the pleasure to thank in this page some of the people that during these three years of PhD have helped me, supported me, and shared with me some really great moments.
\bigskip

\noindent An immense thank you and all my gratitude go to my supervisor, Enrico! None of this work, and none of these three years would have happened without you. Thank for your encouragement, for pushing me to do more and to try to do better, to put in more and more time, for your constant help and you being available at every hour of the day and the night (time zone's fault). For having given me the possibility to travel literally around the world,  visiting beautiful cities and meeting and working with wonderful people and  great mathematicians. For being so friendly and helpful in every occasion, was that for maths or anything else. \\
A huge thank you goes to my family / Multumesc din inima familiei mele: mami, tati, Alin, Boby, Dani, Mamae, Oaoa.  My mother is  prouder  than I am of this PhD thesis! \\
Another huge thank you to Vicky, for all your patience, for staying by my side in my darkest moments, for loving me and supporting me day by day. I would not imagine these years without you.\\
A more than a special thank you to Luca, a great mathematician and a great friend, who (also) forced me to write this \textit{Thank you} note. \\
A great thank you to my office colleagues pasts and presents, especially to Gugu, Pit, Tommy, Teo, Mono, Davide and Alessandra.\\
A huge hug to Maria, thank you for all the great time we had in Berlin and Madrid.\\
 A great thank you to Fausto Ferrari and Aram Karkhanyan, with whom I had the pleasure of working. \\ 
  A thank you to all the people I had the pleasure of meeting in my mathematical journeys, especially to Serena, Pietro and Eleonora.\\
  A big hi to all the people I met in Cetraro (that was by far the best summer school I attended), and  in Cortona.\\
 Thank you, theater, for keeping me sane and always in a parallel world.\\
 Thank you for your hospitality Berlin, Madrid, Edinburgh, Melbourne, Oaxaca, Cetraro, Cortona, Reggio Calabria, Pisa, Roma, Bologna, Magdeburg and of course Milano and Severin. \\
%@article{ALBA,
%  title={L{\'e}vy flight search patterns of wandering albatrosses},
%  author={Viswanathan, G.M. and Afanasyev, V. and Buldyrev, S.V. and Murphy, E.J. and Prince, P.A. and Stanley, H. Eugene},
%  journal={Nature},
%  volume={381},
%  number={6581},
%  pages={413--415},
%  year={1996},
%}
\end{document}